\documentclass[12pt]{article}

\usepackage{amsmath,amssymb,amsthm}
\usepackage{amscd}
\usepackage{subfig}
\usepackage{tikz}
\usepackage{tikz-cd}
\usepackage{graphicx}
\usepackage[mathcal]{eucal}
\usepackage{xcolor}
\definecolor{allrefcolors}{rgb}{0,0.5,0.4}
\usepackage[pagebackref,linktocpage=true,colorlinks=true,allcolors=allrefcolors]{hyperref}

\newtheorem{theorem}{Theorem}[subsection]
\newtheorem{lemma}[theorem]{Lemma}
\newtheorem{corollary}[theorem]{Corollary}
\newtheorem{proposition}[theorem]{Proposition}

\theoremstyle{definition}
\newtheorem{definition}[theorem]{Definition}
\theoremstyle{remark}
\newtheorem{remark}[theorem]{Remark}
\newtheorem{convention}[theorem]{Convention}

\newtheorem{example}[theorem]{Example}

\numberwithin{equation}{subsection}

\makeatletter
\@addtoreset{equation}{section}
\@addtoreset{theorem}{section}
\makeatother

\newcommand{\F}{\mathcal F}
\newcommand{\G}{\mathcal G}
\newcommand{\T}{\mathcal T}
\newcommand{\E}{\mathcal E}
\newcommand{\K}{\mathcal K}
\newcommand{\C}{\mathcal C}
\newcommand{\W}{\mathcal W}
\newcommand{\s}{\mathfrak s}
\newcommand{\ttt}{\mathfrak t}
\newcommand{\rr}{\mathfrak r}
\newcommand{\oo}{\mathfrak o}

\newcommand{\X}{\mathcal X}
\newcommand{\Y}{\mathcal Y}
\newcommand{\Z}{\mathcal Z}
\newcommand{\CC}{\mathbb C}

\newcommand{\ZZ}{\mathbb Z}
\newcommand{\RR}{\mathbb R}
\newcommand{\QQ}{\mathbb Q}
\newcommand{\HHH}{\mathcal H}
\newcommand{\HH}{\mathbb H}
\newcommand{\PPP}{\mathcal P}
\newcommand{\SSS}{\mathcal S}
\newcommand{\PT}{\mathsf{PT}}
\newcommand{\ttop}{\mathrm{top}}
\newcommand{\A}{\mathcal A}
\newcommand{\B}{\mathcal B}
\newcommand{\Abar}{\bar\A}
\newcommand{\Bbar}{\bar\B}
\newcommand{\cH}{\check H}
\newcommand{\sC}{\overline C}
\newcommand{\sH}{\overline H}
\newcommand{\cC}{\check C}
\newcommand{\Fix}{\operatorname{Fix}}
\newcommand{\ind}{\operatorname{ind}}
\newcommand{\proj}{\operatorname{proj}}
\newcommand{\std}{\mathrm{std}}
\newcommand{\Mbar}{\overline{\mathcal M}}
\newcommand{\Cbar}{\overline{\mathcal C}}
\newcommand{\delbar}{\bar\partial}
\newcommand{\Hom}{\operatorname{Hom}}
\newcommand{\Ext}{\operatorname{Ext}}

\newcommand{\Tor}{\operatorname{Tor}}
\newcommand{\coker}{\operatorname{coker}}
\newcommand{\colim}{\operatornamewithlimits{colim}}
\newcommand{\hocolim}{\operatornamewithlimits{hocolim}}
\newcommand{\codim}{\operatorname{codim}}
\newcommand{\pre}{\mathrm{pre}}
\newcommand{\tors}{\mathrm{tors}}
\newcommand{\vdim}{\operatorname{vdim}}
\newcommand{\sF}{\mathsf{F}}
\newcommand{\Ch}{\mathsf{Ch}}
\newcommand{\Set}{\mathsf{Set}}
\newcommand{\FlowCat}{\mathsf{FlowCat}}
\newcommand{\FlowCatIA}{\mathsf{FlowCat}^{\mathsf{IA}}}

\newcommand{\Ndg}{\mathsf{N}_{\mathsf{dg}}}
\newcommand{\supp}{\operatorname{supp}}
\newcommand{\im}{\operatorname{im}}
\newcommand{\op}{\mathrm{op}}
\newcommand{\inj}{\mathrm{inj}}
\newcommand{\rk}{\operatorname{rk}}
\newcommand{\vir}{\mathrm{vir}}
\newcommand{\ess}{\mathrm{ess}}

\newcommand{\pt}{\mathrm{pt}}
\newcommand{\crit}{\operatorname{crit}}
\newcommand{\st}{\mathrm{st}}
\newcommand{\id}{\operatorname{id}}
\newcommand{\gr}{\operatorname{gr}}
\newcommand{\GW}{\mathrm{GW}}
\newcommand{\HF}{\mathrm{HF}}
\newcommand{\JH}{J\!H}

\newcommand{\morse}{\mathrm{morse}}
\newcommand{\reg}{\mathrm{reg}}
\newcommand{\rel}{\operatorname{rel}}
\newcommand{\Prshv}{\operatorname{Prshv}}
\newcommand{\Shv}{\operatorname{Shv}}
\newcommand{\hShv}{\operatorname{hShv}}
\newcommand{\Rlim}{\mathop{R\varprojlim}}
\newcommand{\comment}[1]{}

\newcounter{rcounti}

\newenvironment{rlist}{\begin{list}{\roman{rcounti}.}{\usecounter{rcounti}
\setlength\itemsep{0pt}
\setlength\labelsep{4pt}
\setlength\labelwidth{16pt}
\setlength\leftmargin{32pt}
\setlength\listparindent{0pt}
\setlength\parsep{0pt}
\setlength\parskip{0pt}
\setlength\partopsep{0pt}
\setlength\rightmargin{0pt}
\setlength\topsep{0pt}
\setlength\itemindent{0pt}
}}{\end{list}}
\newcounter{rcountsave}
\newcommand{\rlistsave}{\setcounter{rcountsave}{\value{rcounti}}}
\newcommand{\rlistresume}{\setcounter{rcounti}{\value{rcountsave}}}

\begin{document}

\title{An algebraic approach to virtual fundamental cycles on moduli spaces of pseudo-holomorphic curves}

\author{John Pardon}

\date{8 June 2015}

\maketitle

\nocite{luriehttpub}

\begin{abstract}
We develop techniques for defining and working with virtual fundamental cycles on moduli spaces of pseudo-holomorphic curves which are not necessarily cut out transversally.  Such techniques have the potential for applications as foundations for invariants in symplectic topology arising from ``counting'' pseudo-holomorphic curves.

We introduce the notion of an \emph{implicit atlas} on a moduli space, which is (roughly) a convenient system of local finite-dimensional reductions.  We present a general intrinsic strategy for constructing a canonical implicit atlas on any moduli space of pseudo-holomorphic curves.  The main technical step in applying this strategy in any particular setting is to prove appropriate gluing theorems.  We require only topological gluing theorems, that is, smoothness of the transition maps between gluing charts need not be addressed.  Our approach to virtual fundamental cycles is algebraic rather than geometric (in particular, we do not use perturbation).  Sheaf-theoretic tools play an important role in setting up our functorial algebraic ``VFC package''.

We illustrate the methods we introduce by giving definitions of Gromov--Witten invariants and Hamiltonian Floer homology over $\QQ$ for general symplectic manifolds.  Our framework generalizes to the $S^1$-equivariant setting, and we use $S^1$-localization to calculate Hamiltonian Floer homology.  The Arnold conjecture (as treated by Floer, Hofer--Salamon, Ono, Liu--Tian, Ruan, and Fukaya--Ono) is a well-known corollary of this calculation.

MSC 2010 Primary: 37J10, 53D35, 53D40, 53D45, 57R17

MSC 2010 Secondary: 53D37, 53D42, 54B40

\end{abstract}

\setcounter{tocdepth}{2}
\tableofcontents

\section{Introduction}\label{introsec}

In this paper, we develop a collection of tools and techniques for defining and working with virtual fundamental cycles on compact moduli spaces of pseudo-holomorphic curves (in the sense of Gromov \cite{gromov}) which are not necessarily cut out transversally.  Such techniques have a myriad of potential applications in symplectic geometry by providing foundations for invariants obtained by ``counting'' pseudo-holomorphic curves:
\begin{equation}\label{paradigm}
\begin{matrix}\text{Symplectic}\cr\text{manifold}\end{matrix}\;\Longrightarrow\;\begin{matrix}\text{Moduli space(s) of}\cr\text{pseudo-holomorphic curves}\end{matrix}\;\overset{\text{VFC}}\Longrightarrow\;\begin{matrix}\text{Desired}\cr\text{invariant}\end{matrix}
\end{equation}
In this paper, we build a general framework which can potentially be applied to give rigorous foundations for the wide variety of invariants defined using \eqref{paradigm}.  We hope that this framework may also be applicable to moduli spaces of solutions to other nonlinear elliptic PDEs which give rise to interesting invariants.

When a moduli space is not cut out transversally, its topological structure does not determine its virtual fundamental cycle; rather it must be endowed (canonically) with some additional extra structure.  We introduce the notion of an \emph{implicit atlas} on a compact Hausdorff space, which serves as this extra structure on moduli spaces of pseudo-holomorphic curves.  We use implicit atlases as a layer of abstraction between the two steps in \eqref{paradigm}, making them logically independent.

Our notion of an implicit atlas and our constructions of implicit atlases on moduli spaces of pseudo-holomorphic curves constitute a reworking of existing ideas, with convenient canonicity and functoriality properties.  Our construction of virtual fundamental cycles from implicit atlases is more novel (using algebraic rather than geometric methods), and also has good functoriality properties which are useful in applications.  It is noteworthy that this algebraic VFC setup requires only topological gluing theorems as input.

The basic idea of using an atlas of charts of the form \eqref{localchart} on a moduli space to construct its virtual fundamental cycle has existed since the inception of this problem, see for example Li--Tian \cite{litianII}, Liu--Tian \cite{liutian}, Fukaya--Ono \cite{fukayaono}, Ruan \cite{ruan}, Lu--Tian \cite{lutian}, Fukaya--Oh--Ohta--Ono \cite{FOOOI,FOOOII,foootechnicaldetails,fooonewI} (the theory of Kuranishi structures), and McDuff--Wehrheim \cite{mcduffwehrheim,mcduffnotes}.  The polyfolds project of Hofer--Wysocki--Zehnder \cite{polyfoldI,polyfoldII,polyfoldIII,polyfoldint,polyfoldscnew,polyfoldbookI,polyfoldGW} gives another method for defining virtual fundamental cycles by describing moduli spaces via a generalized infinite-dimensional Fredholm setup.

\subsection{Implicit atlases}

An implicit atlas organizes together a collection of local charts for a compact moduli space $X$.  A local chart for $X$ is a diagram:
\begin{equation}\label{localchart}
\begin{tikzcd}
X\ar[hookleftarrow]{r}{\text{open}}&s_\alpha^{-1}(0)\ar[hook]{r}{\text{closed}}&X_\alpha\ar{r}{s_\alpha}&E_\alpha
\end{tikzcd}
\end{equation}
where $E_\alpha$ is a finite-dimensional vector space (called the \emph{obstruction space}), $X_\alpha$ is an auxiliary moduli space (called the \emph{$\alpha$-thickened\footnote{Perhaps a better name would be ``$\alpha$-stabilized moduli space'', though we have decided not to risk confusing this notion of stabilization (i.e.\ product with a vector space) and the notion of stabilizing a Riemann surface (i.e.\ adding marked points).} moduli space}), and $s_\alpha$ is called the \emph{Kuranishi map}.  For the purpose of constructing the virtual fundamental cycle of $X$, such a local chart \eqref{localchart} is useful over $X\cap X_\alpha^\reg$, where $X_\alpha^\reg\subseteq X_\alpha$ (called the \emph{regular locus}) is the locus where $X_\alpha$ is cut out transversally (and thus, in particular, $X_\alpha^\reg$ is a finite-dimensional manifold).

An implicit atlas is an index set $A$ (whose elements are called \emph{thickening datums}) along with obstruction spaces $E_\alpha$ (for all $\alpha\in A$) and $I$-thickened moduli spaces $X_I$ (for all finite subsets $I\subseteq A$) fitting together globally in a natural generalization of \eqref{localchart}, where \emph{the $\varnothing$-thickened moduli space $X_\varnothing$ is identified with the original moduli space $X$}.  An implicit atlas also includes the data of open subsets $X_I^\reg\subseteq X_I$ which are manifolds and are required to cover all of $X$.  In particular, an implicit atlas carries a parameter $d\in\ZZ$, the \emph{virtual dimension}, and we require that $\dim X_I^\reg=d+\dim E_I$ for all $I\subseteq A$ (where $E_I:=\bigoplus_{\alpha\in I}E_\alpha$).  Implicit atlases also allow charts \eqref{localchart} which incorporate the action of a finite group $\Gamma_\alpha$ (so that such charts exist on spaces $X$ with nontrivial isotropy), though we will introduce the necessary notation later in the paper.

There is a natural notion of an ``implicit atlas with boundary'' (or corners) and of the ``product implicit atlas'' on a product of spaces equipped with implicit atlases (with boundary/corners).  These notions enable us to treat Floer-type homology theories via implicit atlases.

We use only ``topological'' implicit atlases in this paper (i.e.\ we only require that the $X_I^\reg$ are topological manifolds), since the topological structure is sufficient to construct virtual fundamental cycles.  There is, of course, a parallel notion of a smooth implicit atlas, which we will not need here.

\begin{remark}
From a theoretical standpoint, it would be desirable to endow the moduli space $X$ with the canonical structure of a ``derived manifold'' (a notion which should be more intrinsic than the notion of an implicit atlas on $X$).  A good notion of a ``derived smooth manifold'' exists (due to Spivak \cite{spivak}, Borisov--Noel \cite{borisovnoel}, and Joyce \cite{joyceI,joyceII}), and it is reasonable to expect that a parallel topological theory exists as well.  By nature, the theory of derived manifolds uses the language of higher category theory.  An implicit atlas on $X$ may be thought of as giving a ``presentation'' of $X$ as a derived manifold (just as a collection of open sets of $\RR^n$ and gluing data can be used to present an ordinary manifold).

The ``atlas'' approach which we follow here, while being less intrinsic, has the advantage of being more concrete and more elementary.  We believe that a more intrinsic approach is unlikely to lead to any simplification in the construction of this extra structure (implicit atlas or derived manifold structure) on moduli spaces of pseudo-holomorphic curves.  However, it would likely make it easier to work with (and, in particular, calculate) virtual fundamental cycles on such spaces.
\end{remark}

\subsection{Construction of implicit atlases}\label{constructionIntro}

Implicit atlases are designed to encode a system of charts which can be constructed naturally and intrisically on moduli spaces of pseudo-holomorphic curves in wide generality.  Moreover, the basic ingredients which go into the construction of implicit atlases are all familiar in the field.  We think of our specific examples of constructions of implicit atlases as special cases of a general strategy which produces on any moduli space of pseudo-holomorphic curves a canonical implicit atlas.  For this general strategy to succeed, there are (essentially) two steps which require setting-specific arguments.

The first step requiring setting-specific arguments is \emph{domain stabilization}.  For any pseudo-holomorphic curve $u:C\to M$ in $X$, we must show that there is a smooth codimension two submanifold (possibly with boundary) $D\subseteq M$ which intersects it transversally such that adding $u^{-1}(D)$ as added marked points on $C$ makes $C$ stable.\footnote{One can sometimes get away with a little less if this specific form of domain stabilization does not hold.  We will encounter such a situation when constructing $S^1$-equivariant implicit atlases on moduli spaces of stable Floer trajectories for $S^1$-invariant Hamiltonians.}  This is an important ingredient in verifying the \emph{covering axiom} of an implicit atlas.

The second step requiring setting-specific arguments is \emph{formal regularity implies topological regularity}.  For the $I$-thickened moduli spaces $X_I$, we denote by $X_I^\reg\subseteq X_I$ the locus where the relevant linearized operator is surjective.  We must show that $X_I^\reg$ is open and is a topological manifold of the correct dimension (we also need a certain topological submersion condition on how $X_I^\reg$ is cut out inside $X_J^\reg$ for $I\subseteq J$).  These are the \emph{openess and submersion axioms} of an implicit atlas.  This is the step where we must appeal to serious analytic results (in particular, this is where gluing of pseudo-holomorphic curves takes place).  Note, though, that we only require topological gluing results (i.e.\ smoothness of transition maps between gluing charts need not be addressed).  We should also point out that these axioms are local statements about the spaces $X_I^\reg$, and hence are independent of any auxiliary group action.

Our general strategy constructs an implicit atlas on $X$ which is \emph{canonical} (in the sense that we do not need to make any choices during its construction).  This is achieved simply by defining $A$ to be the set of \emph{all} possible thickening datums (of which there are uncountably many), where a thickening datum is a choice of divisor $D$, an obstruction space $E$, plus some additional data.  Note that it is nontrivial to formulate a notion of ``atlas'' which allows this type of ``universal'' construction.  Having such a canonical procedure is very useful in a number of places, and this aspect of our approach appears to be new.

Our general strategy works well for the construction of $S^1$-equivariant implicit atlases on the moduli spaces of stable Floer trajectories for $S^1$-invariant Hamiltonians (a key step towards the Arnold conjecture).  The domain stabilization step is harder in the equivariant setting (we must use a divisor inside $M$ instead of inside $M\times S^1$, and we must be satisfied with not stabilizing Morse-type components of trajectories), but is not difficult.  The formal regularity implies topological regularity step is the \emph{same} as in the non-equivariant case: the analysis is \emph{independent of the $S^1$-equivariance}, and thus is (essentially) identical to that in the non-equivariant case.  In particular, we never need to check that $S^1$ acts smoothly on anything.

\subsection{Construction of virtual fundamental cycles}

We develop a ``VFC package'' for any space $X$ equipped with an implicit atlas $A$.  The tools we develop are primarily \emph{algebraic} (chain complexes and sheaves) rather than geometric or topological.  Furthermore, these tools have nice functorial properties which allow them to be applied essentially independently of how their internals are constructed.  As mentioned earlier, this set of tools does not require a smooth structure on the atlas $A$.

Let us now discuss the components of the VFC package.  The main object we develop is the \emph{virtual cochain complex} $C^\bullet_\vir(X;A)$ defined whenever $A$ is \emph{finite}.\footnote{If $A\subseteq A'$, then $C^\bullet_\vir(X;A)$ and $C^\bullet_\vir(X;A')$ are equivalent for the purposes of defining virtual fundamental cycles, so given an infinite implicit atlas we can just work with any choice of finite subatlas.}  It comes with a canonical isomorphism:\footnote{In the present discussion, we ignore issues about orientations.}
\begin{equation}\label{fundisointro}
\cH^\bullet(X)\xrightarrow\sim H^\bullet_\vir(X;A)
\end{equation}
($H^\bullet_\vir$ is the homology of $C^\bullet_\vir$) and with a canonical map:
\begin{equation}\label{spushforwardintro}
C^{d+\bullet}_\vir(X;A)\xrightarrow{s_\ast}C_{\dim E_A-\bullet}(E_A,E_A\setminus 0)
\end{equation}
(where $E_A:=\bigoplus_{\alpha\in A}E_\alpha$ and $d$ is the virtual dimension of $A$).  Since $H_{\dim E_A}(E_A,E_A\setminus 0)=\ZZ$, combining \eqref{fundisointro} and \eqref{spushforwardintro} yields a map $\cH^d(X)\to\ZZ$.  We define the virtual fundamental class to be this element $[X]^\vir_A\in\cH^d(X)^\vee$; if $X$ is cut out transversally (that is, $X=X^\reg$), then it agrees with the usual fundamental class of $X$ as a closed manifold.  This complex $C_\vir^\bullet$ generalizes naturally to ``implicit atlases with boundary'' as well.

We use $C_\vir^\bullet$ for much more than just defining the virtual fundamental class.  Since it is a complex (rather than just a sequence of homology groups), it is sufficiently rich to provide a useful notion of virtual fundamental \emph{cycle} (more precisely, the map $s_\ast$ \eqref{spushforwardintro} can be thought of as the chain level virtual fundamental cycle).  This enables us to use the VFC package to treat Floer-type homology theories, which requires something like a ``coherent system of virtual fundamental cycles'' over a large system of spaces (a ``flow category'') equipped with implicit atlases.

Let us now explain and motivate the definition of $C^\bullet_\vir$.  First, let us imagine we have a single chart \eqref{localchart} which is global (i.e.\ $X=s_\alpha^{-1}(0)$) and $s_\alpha^{-1}(0)\subseteq X_\alpha^\reg$.  Then we consider the following diagram:
\begin{equation}
\cH^{\dim X_\alpha^\reg-\dim E_\alpha}(X)\xrightarrow\sim H_{\dim E_\alpha}(X_\alpha^\reg,X_\alpha^\reg\setminus X)\xrightarrow{(s_\alpha)_\ast}H_{\dim E_\alpha}(E_\alpha,E_\alpha\setminus 0)=\ZZ
\end{equation}
The first isomorphism is Poincar\'e--Lefschetz duality.  The virtual fundamental class is the resulting element $[X]^\vir\in\cH^d(X)^\vee$ (recall $d=\dim X_\alpha^\reg-\dim E_\alpha$) which is easily seen to agree with the class defined via perturbation.  Thus in this case, the complex $C_{\dim X_\alpha^\reg-\bullet}(X_\alpha^\reg,X_\alpha^\reg\setminus X)$ plays the role of $C^\bullet_\vir(X;A)$.

In general, $C^\bullet_\vir(X;A)$ (which may be thought of as a ``global finite-dimensional reduction up to homotopy'') is built out of the singular chain complexes of the $X_I^\reg$ (or, more accurately, of some auxiliary spaces $X_{I,J,A}$ defined from the $X_I^\reg$ using a \emph{deformation to the normal cone}).  We glue together these singular chain complexes using an appropriate homotopy colimit (a generalized sort of mapping cone); this is technically more convenient than gluing together the spaces themselves.  The latter could probably be made to work as well, as long as one is careful to glue using homotopy colimits\footnote{Here is a baby example of how one may use mapping cones to bypass point-set topological issues.  The long exact sequence of the pair $\cdots\to H_\bullet(A)\to H_\bullet(X)\to H_\bullet(X,A)\to\cdots$ is valid for an arbitrary pair $(X,A)$ of topological spaces (meaning $A\subseteq X$ has the subspace topology), and relative homology $H_\bullet(X,A)$ is always naturally isomorphic to the reduced homology of the mapping cone $(cA\sqcup X)/\sim$.  On the other hand, understanding $H_\bullet(X/A)$ usually requires some niceness assumptions on $(X,A)$ (to ensure that the natural map $H_\bullet(X,A)\to H_\bullet(X/A,\pt)$ is an isomorphism).  On the other hand, if one is content working with $H_\bullet(X,A)$ or with the mapping cone, then such niceness assumptions are unnecessary.} (one can run into point-set topological issues with certain natural topological gluings).\footnote{McDuff--Wehrheim \cite[Examples 3.1.14 and 3.1.15]{mcduffwehrheim} give examples of natural topological quotients which fail to be Hausdorff, fail to be locally compact, or fail to be locally metrizable.}

We use the language of sheaves to give an especially efficient construction of the key isomorphism \eqref{fundisointro}.  In particular, we reduce \eqref{fundisointro} to a statement which can be checked \emph{locally} on $X$.  In a word, we define a ``homotopy $\K$-sheaf'' $K\mapsto C^\bullet_\vir(K;A)$ on $X$, and we show that the stalk cohomology $H^\bullet_\vir(\{p\};A)$ is isomorphic to $\ZZ$ concentrated in degree zero.  The isomorphism \eqref{fundisointro} then follows from rather general sheaf-theoretic arguments (and in fact there is a corresponding map of complexes).  It should not be surprising that sheaves can be used effectively in this setting, since the problem we are facing is precisely to patch together local homological information (from charts \eqref{localchart}) into global homological information.

Moreover, we find in this paper that this sheaf-theoretic formalism continues to play a key role in the study and application of the complexes $C^\bullet_\vir$, beyond simply constructing the fundamental isomorphism \eqref{fundisointro}.  Hence we believe that the sheaf-theoretic formalism is of more importance than the precise manner of definition of $C^\bullet_\vir$.  In particular, checking commutativity of certain diagrams of homology groups can often be reduced to checking that a certain corresponding diagram of sheaves commutes (which is then just a \emph{local} calculation).  This is a key proof technique in many places where we use the VFC package.  This is perhaps surprising since the virtual fundamental class itself has no local homological characterization (though see Remark \ref{chainsishomotopysheaf}).

\begin{remark}
Though our definition of the virtual fundamental cycle does not involve perturbation, we do in some sense show that perturbation is a valid way to compute the virtual fundamental class (in fact, this is an easy corollary of some of its formal properties).
\end{remark}

\begin{remark}\label{smoothcandomaybe}
In this paper, we only construct virtual fundamental cycles in ordinary homology, which is thus inadequate for applications of moduli spaces of pseudo-holomorphic curves involving their fundamental class in smooth (or framed) bordism (e.g.\ as in Abouzaid \cite{abouzaid} or Ekholm--Smith \cite{ekholmsmith}).  However, in discussions with Abouzaid, we realized that by working on the space level, one can probably upgrade our construction to yield a virtual fundamental cycle in the (Steenrod) framed bordism group of $X$, twisted by a certain ``virtual spherical normal bundle''.
\end{remark}

\begin{remark}\label{chainsishomotopysheaf}
In broad philosophical outline, the strategy we follow to define the virtual fundamental cycle is the following.  The functor $U\mapsto C_\bullet(U\rel\partial U)$ (or Borel--Moore chains $C_\bullet^\mathrm{BM}(U)$) is homotopy sheaf on $X$, and the virtual fundamental cycle is a global section thereof.  Hence we may define the virtual fundamental cycle by \emph{specifying it on a convenient finite open cover and giving patching data on all higher overlaps} (in our case, the particular cover being $X=\bigcup_{I\subseteq A}\psi_{\varnothing I}((s_I|X_I^\reg)^{-1}(0))$ for finite $A$).  This philosophy is very natural and could apply quite generally.
\end{remark}

\begin{remark}\label{inftycategoriesbetter}
The language of $\infty$-categories as developed by Lurie \cite{luriehtt} seems to be a natural setting for the VFC package and, more generally, for a good theory of derived manifolds and their virtual fundamental cycles (indeed, the existing theory of derived smooth manifolds is by necessity written in the language of higher category theory).  We have avoided $\infty$-categories in this paper for sake of concreteness (though at the cost of needing to use lots of explicit homotopy colimits).  However, we expect that in this more abstract framework, one could ultimately develop the most flexible calculational tools (computing virtual fundamental cycles directly from our defintion seems rather difficult, due to the inexplicit nature of the isomorphism \eqref{fundisointro}).
\end{remark}

\subsection{Example applications}\label{applsec}

We use the framework developed in this paper to give new VFC foundations for classical results which rely on virtual moduli cycle techniques.  We define Gromov--Witten invariants for general symplectic manifolds (originally due in this generality to Li--Tian \cite{litianII}, Fukaya--Ono \cite{fukayaono}, and Ruan \cite{ruan}).  We define Hamiltonian Floer homology over $\QQ$ for general closed symplectic manifolds, and we use $S^1$-localization methods to calculate Hamiltonian Floer homology (originally due in this generality to Liu--Tian \cite{liutian}, Fukaya--Ono \cite{fukayaono}, and Ruan \cite{ruan}).  The Arnold conjecture on Hamiltonian fixed points is a standard corollary of this calculation.

We hope that the examples we treat here may persuade the reader that it is reasonable to expect to be able to construct implicit atlases on moduli spaces of pseudo-holomorphic curves in considerable generality, and hence that our VFC package is applicable to other curve counting invariants in symplectic topology.

\subsection{How to read this paper}

The logical dependence of the sections is roughly as follows:
\begin{equation*}
\begin{tikzcd}[column sep = tiny]
&&\text{\S\ref{gluingappendix}}\begin{matrix}\text{Gluing}\cr\text{for GW}\end{matrix}\ar[dotted]{d}
&&&\text{\S\ref{gluingHF}}\begin{matrix}\text{Gluing}\cr\text{for HF}\end{matrix}\ar[dotted]{d}\\
\text{\S\ref{implicitatlasdefsec}}\begin{matrix}\text{Implicit}\cr\text{atlases}\end{matrix}\ar{rr}\ar{rd}&&
\text{\S\ref{gromovwittensection}}\begin{matrix}\text{Gromov--}\cr\text{Witten}\end{matrix}\ar{rrr}&&&
\text{\S\ref{hamiltonianfloersec}}\begin{matrix}\text{Hamilt.}\cr\text{Floer}\end{matrix}\\
\text{\S\ref{homologicalalgebrasection}}\begin{matrix}\text{Homol.}\cr\text{algebra}\end{matrix}\ar{r}&
\text{\S\ref{Fsheafsection}}\begin{matrix}\text{VFC}\cr\text{package}\end{matrix}\ar{r}\ar[bend right]{rr}&
\text{\S\ref{fundamentalclasssection}}\begin{matrix}\text{Fund.}\cr\text{class}\end{matrix}\ar[dotted]{u}\ar[bend left, dotted]{rr}&
\text{\S\ref{auxiliarysection}}\text{ Stratifications}\ar{r}&
\text{\S\ref{homologygroupssection}}\begin{matrix}\text{Floer-type}\cr\text{homology}\end{matrix}\ar{r}\ar[dotted]{ur}&
\text{\S\ref{Slocalizationsection}}\begin{matrix}S^1\cr\text{localiz.}\end{matrix}\ar[dotted]{u}
\end{tikzcd}
\end{equation*}
Solid arrows indicate a strong dependency (logical and conceptual); dashed arrows indicate isolated logical dependence (quoting results as black boxes).  Obviously, the choice of solid/dashed is rather subjective.  In the bottom row, we develop the VFC package and apply it abstractly to spaces equipped with implicit atlases.  In the middle row, we construct implicit atlases on various moduli spaces of pseudo-holomorphic curves and define the desired invariants by applying the abstract results developed in the bottom row.  The top row contains gluing results necessary to prove certain key axioms for the implicit atlases constructed in the middle row.

We now summarize the contents of each section, from which the reader may decide which sections to read in detail.

\S\ref{technicalintrosection} is of an introductory flavor (it does not contain any definitions or results to be used elsewhere); it aims to develop technical intuition for our approach without getting bogged down in details.  We give a simplified definition of an implicit atlas, and we give some simple examples.  We give some prototypical constructions of implicit atlases in simplified settings.  We define the virtual fundamental class algebraically from some simple implicit atlases.  We also give a simplified outline of how we apply the VFC package to construct Floer-type homology theories from moduli spaces equipped with implicit atlases.

In \S\ref{implicitatlasdefsec}, we give the definition of an implicit atlas (and its variant with boundary).

In \S\ref{Fsheafsection}, we construct the VFC package.  For a space $X$ equipped with a finite implicit atlas $A$ with boundary, we define and study the \emph{virtual cochain complexes} $C^\bullet_\vir(X;A)$ and $C^\bullet_\vir(X\rel\partial;A)$.  The algebraic and sheaf-theoretic foundations from Appendix \ref{homologicalalgebrasection} play a key role.

In \S\ref{fundamentalclasssection}, we use the VFC package to define the virtual fundamental class.  We also derive some of its basic properties and provide some calculation tools.

In \S\ref{auxiliarysection}, we introduce and study implicit atlases with stratification.  We use the VFC package to obtain an inductive ``stratum by stratum'' understanding of virtual fundamental cycles.

In \S\ref{homologygroupssection}, we use the VFC package to define homology groups from a suitably compatible system of implicit atlases on a ``flow category''.  We also use the results concerning stratifications from \S\ref{auxiliarysection}.

In \S\ref{Slocalizationsection}, we construct an $S^1$-equivariant VFC package and use it prove some $S^1$-localization results for the virtual fundamental classes of \S\ref{fundamentalclasssection} and the homology groups of \S\ref{homologygroupssection}.

In \S\ref{gromovwittensection}, we construct Gromov--Witten invariants by constructing an implicit atlas on the moduli space of stable pseudo-holomorphic maps and using the results of \S\ref{fundamentalclasssection}.

In \S\ref{hamiltonianfloersec}, we construct Hamiltonian Floer homology by constructing implicit atlases on the moduli spaces of stable Floer trajectories and using the results of \S\ref{homologygroupssection}.  We also construct $S^1$-equivariant implicit atlases on the moduli spaces for time-independent Hamiltonians.  Applying the results of \S\ref{Slocalizationsection}, we calculate Hamiltonian Floer homology and thus deduce the Arnold conjecture.

In Appendix \ref{homologicalalgebrasection}, we review and develop foundational results about sheaves, homotopy sheaves, \v Cech cohomology, Poincar\'e duality, and homotopy colimits.

In Appendix \ref{gluingappendix}, we prove the gluing theorem needed in the construction of implicit atlases on Gromov--Witten moduli spaces in \S\ref{gromovwittensection}.

In Appendix \ref{gluingHF}, we prove the gluing theorem needed in the construction of implicit atlases on Hamiltonian Floer moduli spaces in \S\ref{hamiltonianfloersec}.

\subsection{Acknowledgements}

I owe much thanks to my advisor Yasha Eliashberg for suggesting this problem and for many useful conversations.  I also thank Mohammed Abouzaid, Dusa McDuff, and the anonymous referees in particular for many specific and detailed comments on this work.  I am grateful for the useful exchanges I had with Tobias Ekholm, S\o ren Galatius, Helmut Hofer, Michael Hutchings, Eleny Ionel, Dominic Joyce, Daniel Litt, Patrick Massot, Rafe Mazzeo, Paul Seidel, Dennis Sullivan, Arnav Tripathy, Ravi Vakil, and Katrin Wehrheim.  I thank the Simons Center for Geometry and Physics for its hospitality during visits where I presented and discussed this work.  This paper is part of the author's Ph.\,D.\ thesis at Stanford University.

The author was partially supported by a National Science Foundation Graduate Research Fellowship under grant number DGE--1147470.

\section{Technical introduction}\label{technicalintrosection}

We now give a more technical introduction to the main ideas of this paper.  We feel free to make simplifying assumptions for the sake of clarity of exposition, though we do aim to highlight some important technical points.  A full treatment is deferred to the body of the paper, where everything is properly defined.

In \S\ref{protoIAallsection}, we familiarize the reader with the notion of an implicit atlas.  In \S\ref{constructintro}, we give some prototypical constructions of implicit atlases.  In \S\ref{vfcintrosec}, we explain how to construct the virtual fundamental cycle from an implicit atlas in some simple cases.  In \S\ref{floerintrosection}, we show how our methods can be applied to construct Floer-type homology theories.  In \S\ref{Slocalintrosec}, we explain a rudimentary $S^1$-localization result for virtual fundamental cycles.

\subsection{Implicit atlases}\label{protoIAallsection}

\subsubsection{Implicit atlases (proto version)}\label{protoimplicitatlassection}

We now introduce (a simplified version of) implicit atlases and give some intuition for what they mean geometrically.  The impatient reader may also wish to refer to the true definition an implicit atlas in \S\ref{implicitatlasdefsec} (Definition \ref{implicitatlasdefinition}).  Here we have simplified things by assuming $\Gamma_\alpha=1$ (``trivial covering groups''), $X_I^\reg=X_I$ for $\#I\geq 1$, and $X_\varnothing^\reg=\varnothing$.

After the giving the definition (which appears rather complicated at first), we explain it further with examples.

\begin{definition}[Implicit atlas (proto version)]\label{IAproto}
Let $X$ be a compact Hausdorff space.  A \emph{(proto) implicit atlas} of dimension $d$ on $X$ is an index set $A$ along with the following data:
\begin{rlist}
\item(Obstruction spaces) A finite-dimensional $\RR$ vector space $E_\alpha$ for all $\alpha\in A$ (let $E_I:=\bigoplus_{\alpha\in I}E_\alpha$).
\item(Thickenings) A Hausdorff space $X_I$ for all finite $I\subseteq A$, and a homeomorphism $X\to X_\varnothing$.
\item(Kuranishi maps) A function $s_\alpha:X_I\to E_\alpha$ for all $\alpha\in I\subseteq A$ (for $I\subseteq J$, let $s_I:X_J\to E_I$ denote $\bigoplus_{\alpha\in I}s_\alpha$).
\item(Footprints) An open set $U_{IJ}\subseteq X_I$ for all $I\subseteq J\subseteq A$.
\item(Footprint maps) A function $\psi_{IJ}:(s_{J\setminus I}|X_J)^{-1}(0)\to U_{IJ}$ for all $I\subseteq J\subseteq A$.
\end{rlist}
which must satisfy the following ``compatibility axioms'':
\begin{rlist}
\item$\psi_{IJ}\psi_{JK}=\psi_{IK}$ and $\psi_{II}=\id$.
\item$s_I\psi_{IJ}=s_I$.
\item$U_{IJ_1}\cap U_{IJ_2}=U_{I,J_1\cup J_2}$ and $U_{II}=X_I$.
\item$\psi_{IJ}^{-1}(U_{IK})=U_{JK}\cap(s_{J\setminus I}|X_J)^{-1}(0)$.
\item(Homeomorphism axiom) $\psi_{IJ}$ is a homeomorphism.
\rlistsave
\end{rlist}
and the following ``transversality axioms'':
\begin{rlist}
\rlistresume
\item(Submersion axiom) $s_{J\setminus I}:X_J\to E_{J\setminus I}$ is locally modeled on the projection $\RR^{d+\dim E_J}\to\RR^{\dim E_{J\setminus I}}$ over $0\in E_{J\setminus I}$ for $\#I\geq 1$ (in particular, taking $I=J$ implies that $X_I$ is a topological manifold dimension $d+\dim E_I$ for $\#I\geq 1$).
\item(Covering axiom) $X_\varnothing=\bigcup_{\varnothing\ne I\subseteq A}\psi_{\varnothing I}((s_I|X_I)^{-1}(0))$.
\end{rlist}
\end{definition}

Let us unpack this definition a bit.  We have manifolds $X_\alpha$ (meaning $X_I$ for $I=\{\alpha\}$) indexed by $\alpha\in A$.  Each is equipped with a function $s_\alpha:X_\alpha\to E_\alpha$ and a homeomorphism $\psi_\alpha:s_\alpha^{-1}(0)\to U_\alpha$ (meaning $\psi_{\varnothing\{\alpha\}}:(s_\alpha|X_\alpha)^{-1}(0)\to U_{\varnothing\{\alpha\}}$) where $X=\bigcup_{\alpha\in A}U_\alpha$ is an open cover.  Thus so far this is nothing more than a set of charts (the ``basic charts'') of a particular form \eqref{localchart} covering $X$ and indexed by $\alpha\in A$.  If $A=\{\alpha\}$, then the implicit atlas is simply a single global chart \eqref{localchart}, and this is illustrated in Figure \ref{implicitatlaspicture}\subref{pica}.

\begin{figure}
\centering
\subfloat[$A=\{\alpha\}$, $d=1$, $E_\alpha=\RR$.  Illustrated are $X$ (black) and $X_\alpha$ (red; light/dark according to the sign of $s_\alpha$).]{\label{pica}\includegraphics[width=2.8in]{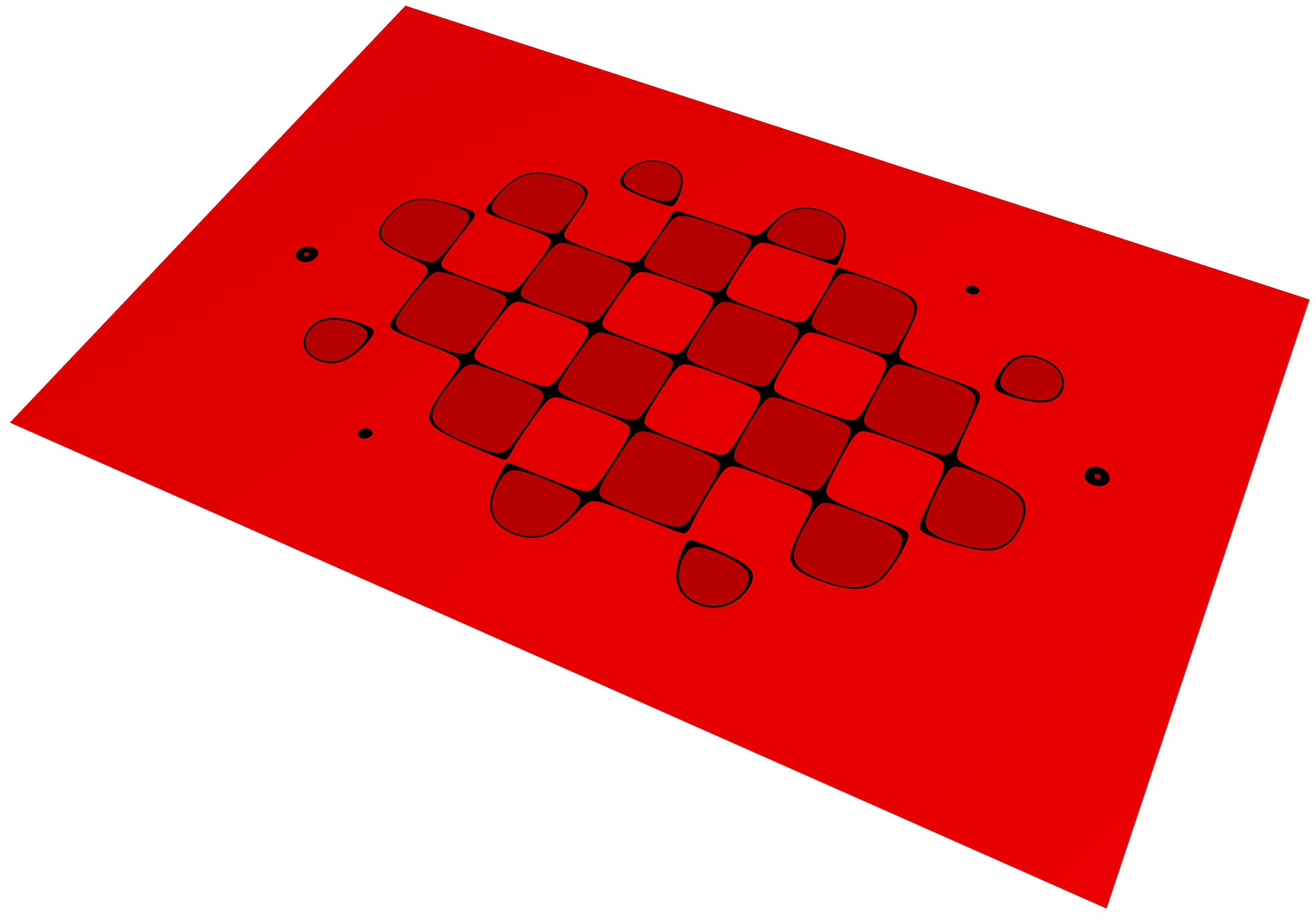}}
\qquad
\subfloat[$A=\{\alpha,\beta\}$, $d=1$, $E_\alpha=E_\beta=\RR$.  Illustrated are $X$ (black), $X_\alpha$ (red; light/dark according to the sign of $s_\alpha$), $X_\beta$ (blue; light/dark according to the sign of $s_\beta$), and $X_{\alpha\beta}$ (green boundary).]{\label{picb}\includegraphics[width=2.8in]{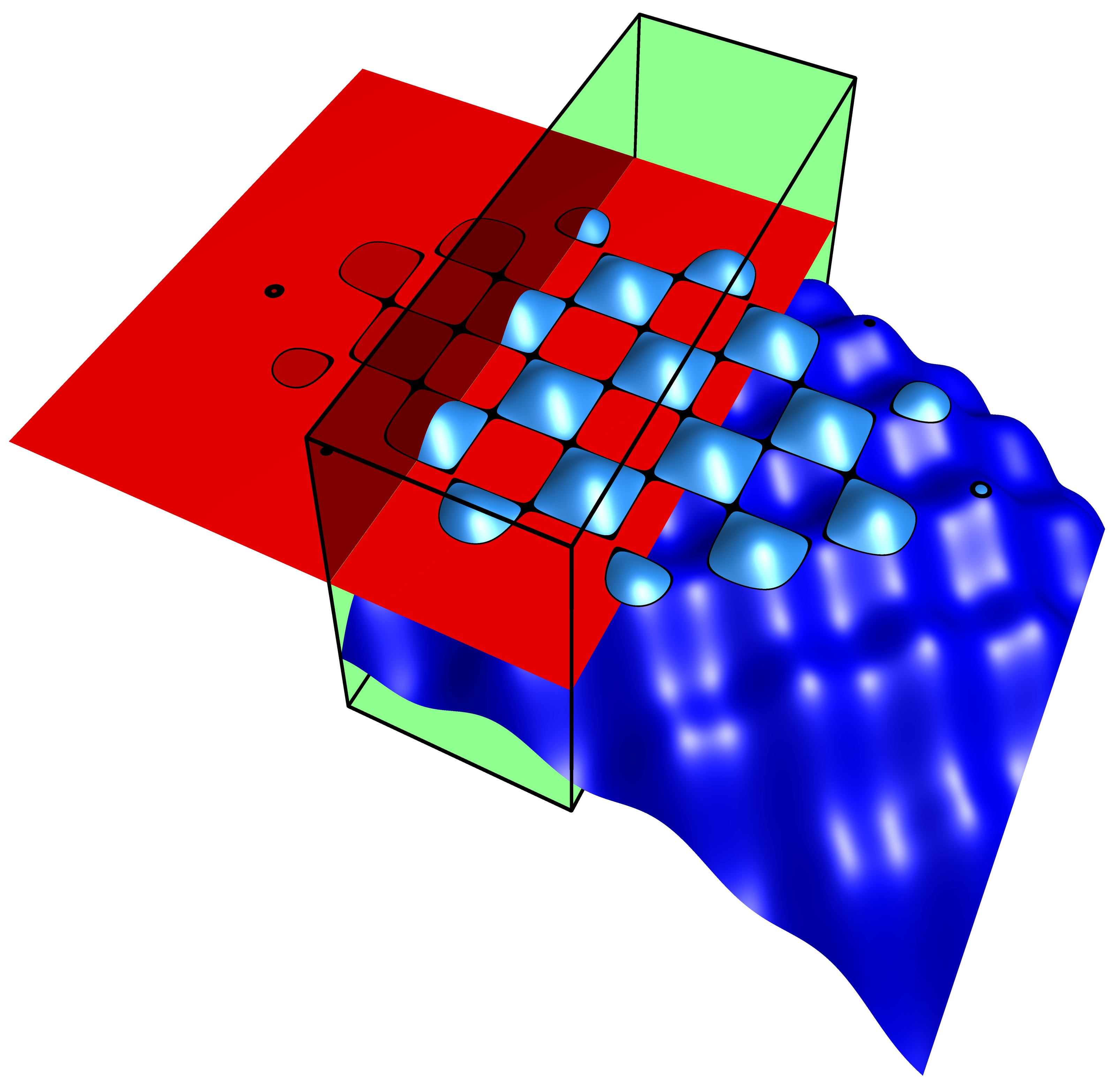}}
\caption{Illustrations of implicit atlases.}\label{implicitatlaspicture}
\end{figure}

Now for every pair of basic charts $\alpha,\beta\in A$, there is a ``overlap chart'' $X_{\alpha\beta}$ with footprint $U_\alpha\cap U_\beta$ and obstruction space $E_\alpha\oplus E_\beta$.  Furthermore, (open subsets $U_{\alpha,\alpha\beta}$ and $U_{\beta,\alpha\beta}$ of) the original charts $X_\alpha$ and $X_\beta$ are identified (via inclusion maps $\psi_{\alpha,\alpha\beta}^{-1}$ and $\psi_{\beta,\alpha\beta}^{-1}$) as the zero sets of $s_\beta$ and $s_\alpha$ respectively.  Note that they are cut out transversally by the submersion axiom, though they may not intersect each other transversally (they do so at precisely those points where $X$ is cut out transversally).  Such a system of charts in the case $A=\{\alpha,\beta\}$ is illustrated in Figure \ref{implicitatlaspicture}\subref{picb}.

More generally, we have charts indexed by the lattice of finite subsets of $A$.  The compatibility axioms relating $U_{IJ}$, $\psi_{IJ}$, and $s_\alpha$ are all just various aspects of the charts being suitably compatible with each other.  The submersion axiom is the precise property (which in practice would follow from every $X_I$ being cut out transversally for $\#I\geq 1$) we need in order to glue together the ``local virtual fundamental cycles''.  The covering axiom simply says the charts cover all of $X$, so we have enough information to recover its global virtual fundamental cycle.

\begin{remark}
The basic charts of a Kuranishi structure are indexed by the points of $X$.  Hence to define a Kuranishi structure on a space $X$, one must make a choice for each $p\in X$.  The charts of an implicit atlas are indexed by an abstract set $A$, and hence we can define implicit atlases without making any choices (see \S\ref{constructintro}).
\end{remark}

\begin{remark}\label{Iemptynotspecial}
Most of the axioms of an implicit atlas are stated without reference to whether $I$ is empty or nonempty.  This contrasts with other approaches, where there is an axiomatic and notational distinction between the $X_I$, $\psi_{IJ}$, or $U_{IJ}$ depending on whether $I$ is empty or nonempty.  We believe that our uniform treatment makes implicit atlases simpler notationally and conceptually, and this is a novel aspect of our approach.
\end{remark}

\begin{remark}\label{bettertouseXSreg}
The requirement that $X_I$ be a manifold \emph{whenever $I$ is nonempty} is rather unnatural (c.f.\ Remark \ref{Iemptynotspecial}) and is too strong of an assumption for two different reasons.  First, we do not know how to construct implicit atlases with this property on moduli spaces of pseudo-holomorphic curves (it is a subtle question of choosing good neighborhoods to ensure transversality over all $X_I$ if $\#I\geq 1$).  Second, with this axiom we cannot form the ``product implicit atlas'' (Definition \ref{productimplicitatlas}) which is crucial for understanding coherence of virtual fundamental cycles between moduli spaces when treating Floer-type homology theories.  Thus in the real definition of an implicit atlas (Definition \ref{implicitatlasdefinition}), there is no such requirement on $X_I$.  Rather, we specify open subsets $X_I^\reg\subseteq X_I$ for all $I\subseteq A$ (morally, this is the locus where $X_I$ is cut out transversally), and we modify the transversality axioms appropriately.
\end{remark}

\begin{remark}\label{equivalenceofIA}
The most natural notion of ``equivalence'' between two implicit atlases $A$ and $B$ on a space $X$ seems to be the existence of a chain of inclusions of implicit atlases $A\subseteq C_1\supseteq\cdots\subseteq C_n\supseteq B$.  It is common to speak of ``cobordisms'' between Kuranishi structures or Kuranishi atlases; the analogous notion of cobordism between implicit atlases (say on the same index set $A$) on a space $X$ is an implicit atlas with boundary on $X\times[0,1]$ whose restriction to the boundary coincides with the first (resp.\ second) implicit atlas on $X\times\{0\}$ (resp.\ $X\times\{1\}$).

It is an easy consequence of the VFC machinery that the virtual fundamental class is invariant under both notions of equivalence.
\end{remark}

\subsubsection{Implicit atlases on spaces with nontrivial isotropy}\label{isotropyintrosec}

We now give a simple construction of a convenient system of charts on any smooth orbifold.  This system of charts is a special case of (and motivates the definition of) an implicit atlas (with nontrivial covering groups, so as to apply to spaces with nontrivial isotropy).

Fix a smooth orbifold $X$ and let $\{X_\alpha/\Gamma_\alpha=V_\alpha\subseteq X\}_{\alpha\in A}$ be an open cover, where each $X_\alpha$ is a smooth manifold with a smooth action by a finite group $\Gamma_\alpha$ (let us call this the ``covering group'').  Then for any finite subset $I=\{\alpha_1,\ldots,\alpha_n\}\subseteq A$, there is an ``overlap chart'':
\begin{align}
\label{fpoverorbifold}X_I&:=X_{\alpha_1}\underset X\times\cdots\underset X\times X_{\alpha_n}\\
\Gamma_I&:=\Gamma_{\alpha_1}\times\cdots\times\Gamma_{\alpha_n}
\end{align}
(where \eqref{fpoverorbifold} is the ``orbifold fiber product''; see Remark \ref{orbifoldfiberproductrmk} below).  It is easy to check that $X_I$ is a smooth manifold with a smooth action by $\Gamma_I$ and that:
\begin{equation}
X_I/\Gamma_I=V_{\alpha_1}\cap\cdots\cap V_{\alpha_n}\subseteq X
\end{equation}
As an exercise, the reader may check that this system of charts described above gives an implicit atlas in the sense of Definition \ref{implicitatlasdefinition} where every $E_\alpha=0$.

\begin{remark}[Orbifold fiber product]\label{orbifoldfiberproductrmk}
Let $\X,\Y,\Z$ be orbifolds, and fix maps of orbifolds $\X,\Y\to\Z$.  Then the \emph{orbifold fiber product} $\X\times_\Z\Y$ is simply the categorical $2$-fiber product in the weak $2$-category of orbifolds.  It exists whenever $\X\times\Y\to\Z\times\Z$ is transverse to the diagonal $\Z\to\Z\times\Z$ (in particular, it exists if at least one of the maps $\X,\Y\to\Z$ is \emph{\'etale} as in the above example of $X_\alpha\to X$).

Thurston \cite[Proof of Proposition 13.2.4]{thurston} gives an explicit hands-on definition of $\X\times_\Z\Y$ in the case $\X,\Y\to\Z$ are both \'etale by working locally on $\Z$ and then gluing.  In a more modern perspective (defining orbifolds as certain stacks on the site of smooth manifolds), we may define the orbifold fiber product by the following universal property:
\begin{equation}
\Hom(S,\X\times_\Z\Y):=\Hom(S,\X)\times_{\Hom(S,\Z)}\Hom(S,\Y)
\end{equation}
for any smooth manifold $S$.  The right hand side denotes the $2$-fiber product in the weak $2$-category of groupoids, which admits the explicit description:
\begin{equation}
G_1\times_HG_2:=\Bigl\{(g_1,g_2,\theta)\Bigm|g_1\in G_1,\,g_2\in G_2,\,\theta\!:\!f_1(g_1)\to f_2(g_2)\Bigr\}
\end{equation}
(with the obvious notion of isomorphism between triples $(g_1,g_2,\theta)$) for groupoids $G_1,G_2,H$ and functors $f_i:G_i\to H$.

The orbifold fiber product usually does not coincide with the fiber product of the underlying topological spaces $X\times_YZ$, though there is at least always a map $\X\times_\Z\Y\to X\times_ZY$.  The construction \eqref{fpoverorbifold} would not work if we used the fiber product of topological spaces.
\end{remark}

\subsection{Constructions of implicit atlases}\label{constructintro}

\subsubsection{Zero set of a smooth section}\label{constructFredholmimplicitatlasintro}

We now give a simple example of the construction of an implicit atlas.  This example is also universal in the sense that all constructions of implicit atlases in this paper are to be thought of as generalizations of this construction.

Suppose we have a smooth manifold $B$, a smooth vector bundle $p:E\to B$, and a smooth section $s:B\to E$ with $s^{-1}(0)$ compact.  Let us now construct an implicit atlas $A$ of dimension $d:=\dim B-\dim E$ on $X:=s^{-1}(0)$.  We will revisit this example in \S\ref{vfcobsbundspace}, so the reader may also refer there for more details.

We set $A$ to be the set of all \emph{thickening datums} where a thickening datum $\alpha$ is a triple $(V_\alpha,E_\alpha,\lambda_\alpha)$ consisting of:
\begin{rlist}
\item An open set $V_\alpha\subseteq B$.
\item A finite-dimensional vector space $E_\alpha$.
\item A smooth homomorphism of vector bundles $\lambda_\alpha:V_\alpha\times E_\alpha\to p^{-1}(V_\alpha)$.
\end{rlist}
Now our thickenings are:
\begin{equation}\label{fredholmsectionchart}
X_I:=\Bigl\{(x,\{e_\alpha\}_{\alpha\in I})\in\bigcap_{\alpha\in I}V_\alpha\times\bigoplus_{\alpha\in I}E_\alpha\Bigm|s(x)+\sum_{\alpha\in I}\lambda_\alpha(x,e_\alpha)=0\Bigr\}
\end{equation}
The function $s_\alpha:X_I\to E_\alpha$ is simply projection to the $E_\alpha$ component.  The set $U_{IJ}\subseteq X_I$ is the locus where $x\in\bigcap_{\alpha\in J}V_\alpha$, and the footprint map $\psi_{IJ}$ is simply the natural map forgetting $e_\alpha$ for $\alpha\in J\setminus I$.  It may be a good exercise for the reader to verify the compatibility axioms in this particular case.

The transversality axioms as we have stated them in Definition \ref{IAproto} do not hold because $X_I$ might not be cut out transversally for $\#I\geq 1$.  The best way to fix this is to use the real definition of an implicit atlas (Definition \ref{implicitatlasdefinition}) where we keep track of the locus $X_I^\reg\subseteq X_I$ where it is cut out transversally and modify the transversality axioms appropriately (c.f.\ Remark \ref{bettertouseXSreg}).

\begin{remark}\label{notsetremark}
The reader may rightfully object that the index set $A$ defined above is not a set but rather a groupoid (just as there is no ``set of all finite sets'' or ``set of all compact smooth manifolds'').  There are two ways of resolving this issue.  The simplest solution is to add appropriate ``rigidifying data'' to turn $A$ into a set (e.g.\ we could add the data of an isomorphism $E_\alpha\xrightarrow\sim\RR^{\dim E_\alpha}$ to the definition of a thickening datum).  Another (more cumbersome, and ultimately unnecessary) solution would be to just check that the notion of an implicit atlas and the accompanying theory of virtual fundamental cycles remains valid for $A$ a groupoid instead of a set (see also Remark \ref{bettertouseovercategory} for more details).
\end{remark}

\begin{remark}
The index set $A$ consists of \emph{all} possible thickening datums.  It is thus canonical in the sense that it does not depend on any auxiliary choice.  For the purposes of extracting the virtual fundamental cycle from an implicit atlas, the only choice we will need to make is that of a finite subatlas; the independence of this choice is part of the VFC machinery.
\end{remark}

\begin{remark}\label{directsumisbetter}
To construct the overlap charts $X_I$, it is crucial that we are able to take the abstract direct sum of the obstruction spaces $E_\alpha$.  Hence, it is better to think of them as abstract vector spaces equipped with maps to $E$ over $V_\alpha$, rather than as trivialized subbundles of $E$ over $V_\alpha$ (which also imposes the unnecessary restriction that $\lambda_\alpha$ be injective), since the latter category is not closed under direct sum.

To illustrate this distinction further, let us mention a problematic alternative version of \eqref{fredholmsectionchart}:
\begin{equation}\label{fredholmsectionchartwrong}
X_I:=\Bigl\{x\in\bigcap_{\alpha\in I}V_\alpha\Bigm|s(x)\in\Bigl(\bigoplus_{\alpha\in I}\lambda_\alpha\Bigr)(E_I)\Bigr\}
\end{equation}
This definition agrees with \eqref{fredholmsectionchart} if the following map is injective:
\begin{equation}\label{lambdatotal}
\bigoplus_{\alpha\in I}\lambda_\alpha:\bigcap_{\alpha\in I}V_\alpha\times\bigoplus_{\alpha\in I}E_\alpha\to p^{-1}\Bigl(\bigcap_{\alpha\in I}V_\alpha\Bigr)
\end{equation}
but in general it can be different, and indeed, if \eqref{lambdatotal} fails to be injective, the definition \eqref{fredholmsectionchartwrong} isn't particularly useful.  Note that this can occur even if we add the requirement that each $\lambda_\alpha$ be injective.

The importance of using the direct sum was independently observed by McDuff--Wehrheim \cite{mcduffwehrheim}, and it is implicit in Fukaya--Oh--Ohta--Ono \cite{foootechnicaldetails}.
\end{remark}

\subsubsection{Moduli space of pseudo-holomorphic curves}\label{constructJholimplicitatlasintro}

We now give an example of the construction of an implicit atlas on a moduli space of pseudo-holomorphic curves.  This construction can be fruitfully interpreted as a generalization\footnote{Moduli spaces of pseudo-holomorphic curves do not fit literally into the setting of \S\ref{constructFredholmimplicitatlasintro} (generalized appropriately to Banach manifolds/bundles) because of three main issues: gluing (nodal domain curves), orbifold structure (nontrivial isotropy groups), and varying complex structures on the domain (nondifferentiability of the reparameterization action).  The polyfolds project of Hofer--Wysocki--Zehnder aims to setup an infinite-dimensional Fredholm framework in which moduli spaces of pseudo-holomorphic curves may be described.} of the construction from \S\ref{constructFredholmimplicitatlasintro}.

Specifically, we construct an implicit atlas $A$ on the moduli space of stable maps $\Mbar_{0,0}(X,B)$ (we fix a symplectic manifold $X$, a smooth $\omega$-tame almost complex structure $J$, and a homology class $B\in H_2(X;\ZZ)$).  The reader impatient for the full details may also wish to refer to \S\ref{gromovwittensection} where we give a full treatment.  Here we have simplified things by assuming that $(g,n)=(0,0)$, $\Gamma_\alpha=S_{r_\alpha}$ and $\Mbar_\alpha=\Mbar_{0,r_\alpha}$ (which is a smooth manifold!).

We define $A$ to be the set of all \emph{thickening datums} where a thickening datum is a $4$-tuple $(r_\alpha,D_\alpha,E_\alpha,\lambda_\alpha)$ consisting of:
\begin{rlist}
\item An integer $r_\alpha>2$; let $\Gamma_\alpha:=S_{r_\alpha}$.
\item A smooth compact codimension two submanifold $D_\alpha\subseteq X$ with boundary.
\item A finite-dimensional $\RR[S_{r_\alpha}]$-module $E_\alpha$.
\item A $S_{r_\alpha}$-equivariant map $\lambda_\alpha:E_\alpha\to C^\infty(\Cbar_{0,r_\alpha}\times X,\Omega^{0,1}_{\Cbar_{0,r_\alpha}/\Mbar_{0,r_\alpha}}\otimes_\CC TX)$ (where $\Cbar_{0,r_\alpha}\to\Mbar_{0,r_\alpha}$ is the universal family over the Deligne--Mumford moduli space) supported away from the nodes and marked points of the fibers.
\end{rlist}
Let us remark that the analogue of the open set $V_\alpha$ from \S\ref{constructFredholmimplicitatlasintro} is the set of maps $u:C\to X$ satisfying the conditions appearing in (\ref{firstcondition}) below.

Now our thickening $\Mbar_{0,0}(X,B)_I$ (any finite $I\subseteq A$) is defined as the set of:
\begin{rlist}
\item\label{firstcondition}A smooth map $u:C\to X$ where $C$ is a nodal curve of genus $0$ so that for all $\alpha\in I$, we have $u\pitchfork D_\alpha$ (meaning $u^{-1}(\partial D_\alpha)=\varnothing$ and for all $p\in u^{-1}(D_\alpha)$, the differential $(du)_p:T_pC\to T_{u(p)}X/T_{u(p)}D_\alpha$ is surjective and $p$ is not a node of $C$) with $\#u^{-1}(D_\alpha)=r_\alpha$ such that adding these $r_\alpha$ intersections as marked points makes $C$ stable.
\item Elements $e_\alpha\in E_\alpha$ for all $\alpha\in I$
\item Labellings of $u^{-1}(D_\alpha)$ by $\{1,\ldots,r_\alpha\}$ for all $\alpha\in I$.
\end{rlist}
such that:
\begin{equation}\label{thickenedintro}
\delbar u+\sum_{\alpha\in I}\lambda_\alpha(e_\alpha)(\phi_\alpha,u)=0
\end{equation}
where $\phi_\alpha:C\to\Cbar_{0,r_\alpha}$ is the unique isomorphism onto a fiber respecting the labeling of $u^{-1}(D_\alpha)$.  There is an action of $\Gamma_\alpha=S_{r_\alpha}$ on $\Mbar_{0,0}(X,B)_I$ (for $\alpha\in I$) given by changing the labelling of $u^{-1}(D_\alpha)$ and by its given action on $e_\alpha\in E_\alpha$.

There are obvious projection maps $s_\alpha:\Mbar_{0,0}(X,B)_I\to E_\alpha$ and forgetful maps $\psi_{IJ}:(s_{J\setminus I}|X_J)^{-1}(0)\to U_{IJ}$, where $U_{IJ}\subseteq\Mbar_{0,0}(X,B)_I$ is the locus of curves satisfying the conidition in (\ref{firstcondition}) above for all $\alpha\in J$.  Thus we have specified the atlas data for $A$.  The compatibility axioms are rather trivial (as in \S\ref{constructFredholmimplicitatlasintro}), though for the homeomorphism axiom requires a bit of thought.

Now let us discuss the transversality axioms, which have much nontrivial content.  To verify the covering axiom, we need to show \emph{domain stabilization}, namely that for any $J$-holomorphic $u:C\to X$, there is a divisor $D_\alpha$ with $u\pitchfork D_\alpha$ and so that adding $u^{-1}(D_\alpha)$ to $C$ as marked points makes $C$ stable.  Given domain stabilization, the rest of the proof of the covering axiom is rather standard (choose $(E_\alpha,\lambda_\alpha)$ big enough to cover the cokernel of the linearized operator $D\delbar$ at $u$).  The submersion axiom asserts (in particular) that the regular locus $\Mbar_{0,0}(X,B)_I^\reg\subseteq\Mbar_{0,0}(X,B)_I$ is a topological manifold.  Proving the submersion axiom is not too difficult over the locus where the domain curve $C$ is smooth (it follows immediately from the implicit function theorem for Banach manifolds), but to show it near a nodal domain curve amounts to proving a gluing theorem (which we do in Appendix \ref{gluingappendix}).

\begin{remark}
The thickened moduli spaces $\Mbar_{0,0}(X,B)_I$ are defined as moduli spaces of solutions to the ``$I$-thickened $\delbar$-equation'' \eqref{thickenedintro}.  With this intrinsic definition, the convenient overlap properties of the charts (the compatibility axioms) follow rather trivially.  The atlas also clearly does not depend on any choice of Sobolev norms or ``gluing profile''.

On this point, it is useful to compare with other approaches, which often take the perspective of defining thickened moduli spaces as subsets of some particular Banach manifold of maps.  In this context, achieving good overlap properties seems to be more difficult and less conceptual.  Another inconvenience of this setting is the lack of differentiability of the reparameterization action (and large reparameterization groups for bubbles).
\end{remark}

\begin{remark}
In our approach, the most serious analytic questions are encountered in verifying the openness and submersion axioms (that is, in proving the necessary gluing theorems).  We remark that these only concern the local properties of the thickened moduli spaces, and, in particular, are separated from the other technical aspects of the construction of the implicit atlas (e.g.\ the compatibility axioms, the action of the groups $\Gamma_I$, or the action of a larger symmetry group on the entire atlas).
\end{remark}

\begin{remark}
Standard gluing techniques suffice to verify the openness and submersion axioms for the implicit atlases we construct on moduli spaces of pseudo-holomorphic curves.  In fact, the transition maps between gluing charts are clearly smooth when restricted to each stratum (i.e.\ for a fixed topological type of the domain), and this yields a canonical ``stratified smooth structure'' on each $X_I^\reg$.  If one wanted to obtain a smooth structure on the $X_I^\reg$, one would need to construct gluing charts whose transition maps are truly smooth.  This would require a choice of ``gluing profile'' (on which the resulting smooth structure would depend) and is slightly more delicate (see Fukaya--Oh--Ohta--Ono \cite{FOOOII,foootechnicaldetails} or Hofer--Wysocki--Zehnder \cite{polyfoldGW}).  Such methods might yield smooth implicit atlases (see Definition \ref{smoothIAdef}).
\end{remark}

\subsection{Construction of virtual fundamental cycles}\label{vfcintrosec}

Let us now describe concretely some simple cases of our algebraic definition of the virtual fundamental class of a space $X$ with implicit atlas $A$.  While the cases we treat (one chart and two charts) are admittedly rather basic, they nevertheless illustrate the main ideas necessary to deal with arbitrary implicit atlases.  We will see that certain chain complexes play a key role; they will turn out to be the \emph{virtual cochain complexes} $C_\vir^\bullet$, which are the central objects of the ``VFC package''.

The reader interested in the details of our treatment in full generality should refer to \S\ref{Fsheafsection} (where we construct the VFC package) and \S\ref{fundamentalclasssection} (where we define the virtual fundamental class).

For the purposes of this section, we use implicit atlases in the sense of Definition \ref{IAproto}.  We will ignore issues about orientations (as they can be dealt with rather trivially by introducing the relevant orientation sheaves/groups).

Here we work over $\ZZ$; in the main body of the paper we consider any ground ring in which the orders of all relevant ``covering groups'' are invertible.  It seems plausible that, with some more work, this could be weakened to assuming only invertibility of $\#\Gamma_x$ for all $x\in X$, where $\Gamma_x$ denotes the isotropy group of $x\in X$ (i.e.\ the stabilizer of any inverse image of $x$ under $\psi_{\varnothing,I}$ lying in $X_I^\reg$).

\subsubsection{What ``homology group'' does the virtual fundamental class live in?}

Since $X$ is just a compact Hausdorff space,\footnote{The existence of an implicit atlas on $X$ does impose some additional restriction on the topology on $X$.  For example, if $X$ admits an implicit atlas then it is locally metrizable and hence metrizable by Smirnov's theorem (a paracompact Hausdorff locally metrizable space is metrizable).} we must be careful about what homology theory we use to house the virtual fundamental class (Example \ref{singularhomologyisbad} below shows that the singular homology $H_\bullet(X)$ of $X$ is insufficient for this purpose).

The \emph{dual of \v Cech cohomology} $\cH^\bullet(X)^\vee:=\Hom(\cH^\bullet(X),\ZZ)$ is a good candidate (and it is what we choose to use, though see Remark \ref{alsogetsteenrod} for further discussion).  For example, a map of spaces $f:X\to Y$ induces a pushforward $f_\ast:\cH^\bullet(X)^\vee\to\cH^\bullet(Y)^\vee$ (defined as the dual of pullback $f^\ast:\cH^\bullet(Y)\to\cH^\bullet(X)$).  Moreover, for a finite CW-complex $Z$, we have $\cH^\bullet(Z)=H^\bullet(Z)$, so $\cH^\bullet(Z)^\vee=H_\bullet(Z)/\tors$.  It follows that (for many purposes) a virtual fundamental class $[X]_A^\vir\in\cH^d(X)^\vee$ can be used in the same way one would use the fundamental class $[X]\in H_d(X)$ if $X$ were a closed manifold of dimension $d$.

\begin{figure}
\centering
\includegraphics[width=2.5in]{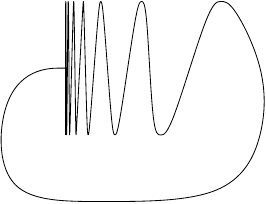}
\caption{The Warsaw circle $W\subseteq\RR^2$, namely the union of $\overline{\{(x,\sin\frac\pi x):0<x<1\}}$ and a path from $(0,0)$ to $(1,0)$.}\label{warsawfigure}
\end{figure}

\begin{example}[Insufficiency of singular homology]\label{singularhomologyisbad}
Consider the Warsaw circle $W\subseteq\RR^2$ as illustrated in Figure \ref{warsawfigure}; note that singular $H_1(W)=0$.  Now $\RR^2\setminus W$ has two connected components; let $s:\RR^2\to\RR$ be positive on one component and negative on the other; this gives the data of an implicit atlas on $W=s^{-1}(0)$.  Using any reasonable definition, we certainly want $[W]^\vir\in\cH^1(W)^\vee\cong\ZZ$ to be a generator, however this is clearly not in the image of singular homology under the natural map $H_\bullet(W)\to H^\bullet(W)^\vee\to\cH^\bullet(W)^\vee$.  Alternatively, the pushforward of $[W]^\vir$ to a small annular neighborhood $A\subseteq\RR^2$ of $W$ should be a generator of $H_1(A)\cong\ZZ$ (as one can see by perturbing $s$).
\end{example}

\subsubsection{Virtual fundamental class from a single chart}\label{vfcsinglechart}

We have a space $X$, and the implicit atlas $A=\{\alpha\}$ consists of the following data.  We have a topological manifold $X_\alpha$ (not necessarily compact), a vector space $E_\alpha$, and a continuous function $s_\alpha:X_\alpha\to E_\alpha$.  We also have an identification $X=s_\alpha^{-1}(0)$ (see Figure \ref{implicitatlaspicture}\subref{pica}).

We define the virtual fundamental class via the following diagram, which we explain below:
\begin{equation}\label{singlechartFC}
\cH^d(X)=H_{\dim E_\alpha}(X_\alpha,X_\alpha\setminus X)\xrightarrow{(s_\alpha)_\ast}H_{\dim E_\alpha}(E_\alpha,E_\alpha\setminus 0)\xrightarrow{[E_\alpha]\mapsto 1}\ZZ
\end{equation}
(recall that $\dim X_\alpha=d+\dim E_\alpha$).  Poincar\'e duality gives a canonical\footnote{For the moment, we ignore the necessary twist by the orientation sheaf.} identification $\cH^\bullet(X)=H_{\dim X_\alpha-\bullet}(X_\alpha,X_\alpha\setminus X)$, which gives the first equality in \eqref{singlechartFC} (this is a rather strong version of Poincar\'e duality, valid for any compact subset $X$ of a manifold $X_\alpha$; we will say more about how to prove it in \S\ref{homotopysheavesintrosec}).  We think of the composite of the maps in \eqref{singlechartFC} as an element $[X]_A^\vir\in\cH^d(X)^\vee$, which we call the virtual fundamental class.

In this setting, one can also define the virtual fundamental class using perturbation (under the additional assumption that $X_\alpha$ carries a smooth structure).  Specifically, one can perturb $s_\alpha$ to $\tilde s_\alpha$ so that it is a submersion over $0\in E_\alpha$; we consider $\tilde s_\alpha^{-1}(0)$ as a ``perturbed moduli space'' near $X$.  Using the continuity property of \v Cech cohomology, one can then make sense of $\lim_{\tilde s_\alpha\to s_\alpha}[\tilde s_\alpha^{-1}(0)]$ as an element of $\cH^d(X)^\vee$.

The algebraic approach is more general (it does not require a smooth structure on $X_\alpha$), and it is easy to see that it gives the same answer as the perturbation approach when $X_\alpha$ has a smooth structure.

\begin{example}
Let $X=\{0\}$, $X_\alpha=E_\alpha=\RR$, and $s_\alpha(x)=x^n$ ($n\geq 1$).  The reader may easily check that with our definition, the virtual fundamental class is $1$ if $n$ is odd and $0$ if $n$ is even (as is consistent with the perturbation picture).

Let $X=\{0\}$, $X_\alpha=E_\alpha=\CC$ (considered as an $\RR$-vector space), and $s_\alpha(z)=z^n$ ($n\geq 1$).  The reader may check that in this case, the virtual fundamental class is $n$.
\end{example}

\subsubsection{Virtual fundamental class from a single chart (with covering group)}\label{vfcsinglechartisotropy}

We now describe how the construction from \S\ref{vfcsinglechart} must be modified in the presence of nontrivial covering groups (as in \S\ref{isotropyintrosec}).  We have not yet introduced implicit atlases with nontrivial covering groups, so we will simply say explicitly what this means in the present situation of a single chart (the reader may also refer to Definition \ref{implicitatlasdefinition}).

We have a space $X$, and the implicit atlas $A=\{\alpha\}$ consists of the following data.  We have a topological manifold $X_\alpha$ (not necessarily compact), a vector space $E_\alpha$, and a continuous function $s_\alpha:X_\alpha\to E_\alpha$.  We have a finite group $\Gamma_\alpha$ acting on $X_\alpha$ and (linearly) on $E_\alpha$ so that $s_\alpha$ is $\Gamma_\alpha$-equivariant.  Finally, we have an identification $X=s_\alpha^{-1}(0)/\Gamma_\alpha$.  We must now work over the coefficient ring $R=\ZZ[\frac 1{\#\Gamma_\alpha}]$.

To define the virtual fundamental class we consider:
\begin{equation}\label{singlechartisotropyFC}
\cH^d(s_\alpha^{-1}(0);R)^{\Gamma_\alpha}=H_{\dim E_\alpha}(X_\alpha,X_\alpha\setminus s_\alpha^{-1}(0);R)^{\Gamma_\alpha}\xrightarrow{(s_\alpha)_\ast}H_{\dim E_\alpha}(E_\alpha,E_\alpha\setminus 0;R)^{\Gamma_\alpha}\xrightarrow{[E_\alpha]\mapsto 1}R
\end{equation}
Now it is a general property of \v Cech cohomology that $\cH^\bullet(Y;R)^\Gamma=\cH^\bullet(Y/\Gamma;R)$ for a compact Hausdorff space $Y$ acted on by a finite group $\Gamma$ as long as $R$ is a module over $\ZZ[\frac 1{\#\Gamma}]$.  We precompose \eqref{singlechartisotropyFC} with the canonical isomorphism $\cH^d(X;R)\to\cH^d(s_\alpha^{-1}(0);R)^{\Gamma_\alpha}$ given as $\frac 1{\#\Gamma_\alpha}$ times the pullback.  This gives an element $[X]_A^\vir\in\cH^d(X;R)^\vee$, which we call the virtual fundamental class.

\subsubsection{Virtual fundamental class from two charts}\label{vfcmultiplecharts}

Generalizing the approach in \S\ref{vfcsinglechart} to multiple charts leads immediately to the heart of the problem of defining virtual fundamental cycles.  Namely, we must figure out how to glue together the information contained in each local chart to define the global virtual fundamental class.  We explain our solution to this problem in the simple case of $A=\{\alpha,\beta\}$, which nevertheless illustrates most of the ideas necessary to deal with the general case (which in addition just requires good organization of the combinatorics).  Since the problem we are facing is precisely to glue together local homological information into global homological information, it should not be surprising that sheaf-theoretic tools and homological algebra are useful.

Our space is $X$, and the implicit atlas $A=\{\alpha,\beta\}$ amounts to the following data:
\begin{align}
s_\alpha:X_\alpha&\to E_\alpha&s_\alpha^{-1}(0)=U_\alpha\subseteq X\text{ (open subset)}\\
s_\beta:X_\beta&\to E_\beta&s_\beta^{-1}(0)=U_\beta\subseteq X\text{ (open subset)}\\
s_\alpha\oplus s_\beta:X_{\alpha\beta}&\to E_\alpha\oplus E_\beta&(s_\alpha\oplus s_\beta)^{-1}(0)=U_{\alpha\beta}=U_\alpha\cap U_\beta\subseteq X
\end{align}
which fit together as outlined in \S\ref{protoimplicitatlassection} (and as illustrated in Figure \ref{implicitatlaspicture}\subref{picb}).

We would like to generalize the approach in \S\ref{vfcsinglechart}, specifically equation \eqref{singlechartFC}.  For this, we need a replacement for the group $H_\bullet(X_\alpha,X_\alpha\setminus X)$.  To construct such a replacement, we would like to ``glue together'' $C_\bullet(X_\alpha,X_\alpha\setminus U_\alpha)$ and $C_\bullet(X_\beta,X_\beta\setminus U_\beta)$ along $C_\bullet(X_{\alpha\beta},X_{\alpha\beta}\setminus U_{\alpha\beta})$ (as remarked earlier, it is easier to glue together these \emph{complexes} rather than glue together the corresponding \emph{spaces}).  The resulting complex should calculate the \v Cech cohomology of $X$ (by some version of Poincar\'e duality) and also have a natural map to $C_\bullet(E_\alpha\oplus E_\beta,E_\alpha\oplus E_\beta\setminus 0)$.  If we can construct a complex with these two properties, then we can define the virtual fundamental class just as in \eqref{singlechartFC}.  This complex we construct will be called the \emph{virtual cochain complex}.

\begin{remark}
The complex $C_{\dim X_\alpha-\bullet}(X_\alpha,X_\alpha\setminus U_\alpha)$ calculates $\cH^\bullet_c(U_\alpha)$, and the map $s_\ast:C_\bullet(X_\alpha,X_\alpha\setminus U_\alpha)\to C_\bullet(E_\alpha,E_\alpha\setminus 0)$ should be thought of as the chain level ``local virtual fundamental cycle'' $[X]^\vir\in\cH^{\dim X_\alpha-\dim E_\alpha}_c(U_\alpha)^\vee$, which we would like to glue together into the global virtual fundamental cycle.
\end{remark}

\begin{remark}
Technically speaking, it is very important to have a uniform functorial definition of the virtual cochain complexes (one which does not require making any extra choices).
\end{remark}

As a first try towards gluing the desired complexes together, let us consider using the mapping cone of the following chain map:
\begin{equation}\label{mappingconefirsttry}
C_{\dim X_{\alpha\beta}-\bullet}(X_{\alpha\beta},X_{\alpha\beta}\setminus U_{\alpha\beta})\xrightarrow{\cap[s_\beta^{-1}(0)]\oplus\cap[s_\alpha^{-1}(0)]}\begin{matrix}C_{\dim X_\alpha-\bullet}(X_\alpha,X_\alpha\setminus U_\alpha)\cr\oplus\cr C_{\dim X_\beta-\bullet}(X_\beta,X_\beta\setminus U_\beta)\end{matrix}
\end{equation}
where the maps are intersection of chains with the (transversely cut out!) submanifolds $s_\beta^{-1}(0)$ and $s_\alpha^{-1}(0)$ of $X_{\alpha\beta}$.  There is the question, though, of how these maps are to be defined on the chain level.  There are various direct ways to define these maps\footnote{One could use very fine chains, do a (chain level) cap product with a choice of cochain level Poincar\'e dual of the relevant submanifold and then project ``orthogonally'' onto the submanifold.  Alternatively, one could just use ``generic'' chains (or perturb the chains) so they are transverse to the submanifold, and then triangulate the intersection in a suitable way.}, however (at least when attempted by the author) the multitude of ``choices'' one has to make invariably leads to a big mess.

As a second try, let us try to glue together the complexes:
\begin{align}
C_{\dim X_{\alpha\beta}-\bullet}&(X_\alpha\times E_\beta,[X_\alpha\setminus U_\alpha]\times E_\beta)\\
C_{\dim X_{\alpha\beta}-\bullet}&(E_\alpha\times X_\beta,E_\alpha\times[X_\beta\setminus U_\beta])\\
C_{\dim X_{\alpha\beta}-\bullet}&(X_{\alpha\beta},X_{\alpha\beta}\setminus U_{\alpha\beta})
\end{align}
This should let us let us avoid the cap product maps (since there is no ``dimension shift'' between these complexes).  Each of these complexes has a canonical map to $C_{\dim X_{\alpha\beta}-\bullet}(E_\alpha\oplus E_\beta,E_\alpha\oplus E_\beta\setminus 0)$, but it is not clear how to glue them together in a manner compatible with this map.  In particular, there is no reason for there to exist commuting diagrams:
\begin{equation}\label{strongtubularneighborhood}
\begin{tikzcd}[column sep = tiny]
U_{\alpha,\alpha\beta}\times E_\beta\arrow{rr}\arrow{dr}[swap]{s_\alpha\oplus\id}&&X_{\alpha\beta}\arrow{dl}{s_{\alpha\beta}}\\
&E_\alpha\oplus E_\beta
\end{tikzcd}
\qquad
\begin{tikzcd}[column sep = tiny]
E_\alpha\times U_{\beta,\alpha\beta}\arrow{rr}\arrow{dr}[swap]{\id\oplus s_\beta}&&X_{\alpha\beta}\arrow{dl}{s_{\alpha\beta}}\\
&E_\alpha\oplus E_\beta
\end{tikzcd}
\end{equation}
Moreover, it is technically very inconvenient (functoriality of the construction is a mess) to have a complex which depends on a choice of maps \eqref{strongtubularneighborhood} (even if this turns out to be a contractible choice, and thus morally irrelevant).

As a third try (which ends up working nicely), let us consider the following ``deformation to the normal cone'':
\begin{equation}
Y_{\alpha\beta}:=\left\{(e_\alpha,e_\beta,t,x)\in E_\alpha\times E_\beta\times[0,1]\times X_{\alpha\beta}\;\middle|\;\begin{matrix}s_\alpha(x)=t\cdot e_\alpha\hfill\cr s_\beta(x)=(1-t)\cdot e_\beta\hfill\end{matrix}\right\}
\end{equation}
We think of $Y_{\alpha\beta}$ as a family of spaces parameterized by $[0,1]$ (via the projection $Y_{\alpha\beta}\to[0,1]$).  Observe that if $t\in(0,1)$, then $e_\alpha,e_\beta$ are determined uniquely by $x$.  Therefore, over the open interval $(0,1)$, we see that $Y_{\alpha\beta}$ is a trivial product space $X_{\alpha\beta}\times(0,1)$.  Over the point $t=0$, though, the fiber is $E_\alpha\times U_{\beta,\alpha\beta}$, which is the ``normal cone'' of the submanifold $U_{\beta,\alpha\beta}\subseteq X_{\alpha\beta}$ cut out (transversally!) by the equation $s_\alpha=0$.  Similarly, over the point $t=1$, the fiber is $U_{\alpha,\alpha\beta}\times E_\beta$.  Also observe that $Y_{\alpha\beta}$ is a manifold by the submersion axiom.

Now we consider the mapping cone of the following:
\begin{equation}\label{multichartcomplex}
\begin{matrix}
C_{\dim X_{\alpha\beta}-\bullet}(U_{\alpha,\alpha\beta}\times E_\beta,U_{\alpha,\alpha\beta}\times E_\beta\setminus U_\alpha\times 0)\cr\oplus\cr
C_{\dim X_{\alpha\beta}-\bullet}(E_\alpha\times U_{\beta,\alpha\beta},E_\alpha\times U_{\beta,\alpha\beta}\setminus 0\times U_\beta)
\end{matrix}
\to
\begin{matrix}
C_{\dim X_{\alpha\beta}-\bullet}(X_\alpha\times E_\beta,X_\alpha\times E_\beta\setminus U_\alpha\times 0)\cr\oplus\cr
C_{\dim X_{\alpha\beta}-\bullet}(Y_{\alpha\beta},Y_{\alpha\beta}\setminus 0\times[0,1]\times U_{\alpha\beta})\cr\oplus\cr
C_{\dim X_{\alpha\beta}-\bullet}(E_\alpha\times X_\beta,E_\alpha\times X_\beta\setminus 0\times U_\beta)
\end{matrix}
\end{equation}
The maps are simply pushforward along the maps of spaces (with appropriate signs):
\begin{align}
U_{\alpha,\alpha\beta}\times E_\beta&\hookrightarrow X_\alpha\times E_\beta\\
U_{\alpha,\alpha\beta}\times E_\beta&\hookrightarrow Y_{\alpha\beta}\quad\text{(isomorphism onto $t=1$ fiber)}\\
E_\alpha\times U_{\beta,\alpha\beta}&\hookrightarrow Y_{\alpha\beta}\quad\text{(isomorphism onto $t=0$ fiber)}\\
E_\alpha\times U_{\beta,\alpha\beta}&\hookrightarrow E_\alpha\times X_\beta
\end{align}
There is also an evident map from (the mapping cone of) \eqref{multichartcomplex} to $C_\bullet(E_\alpha\oplus E_\beta,E_\alpha\oplus E_\beta\setminus 0)$ (namely pushforward on the three factors on the right hand side and zero on the left hand side).

Now to complete the definition of the virtual fundamental cycle of $X$ using the mapping cone of \eqref{multichartcomplex}, we need an argument to show that its homology is canonically isomorphic to the \v Cech cohomology of $X$ (a sort of Poincar\'e duality isomorphism).  We discuss this next in \S\ref{homotopysheavesintrosec}.

\subsubsection{Homotopy \texorpdfstring{$\K$}{K}-sheaves in the theory of virtual fundamental cycles}\label{homotopysheavesintrosec}

As we have already mentioned, the central objects we use to understand virtual fundamental cycles are the \emph{virtual cochain complexes} $C^\bullet_\vir(X;A)$ (for example, the mapping cone of \eqref{multichartcomplex} plays the role of the virtual cochain complex in \S\ref{vfcmultiplecharts}).  A crucial ingredient in this approach is an isomorphism:
\begin{equation}\label{homotopyisogoal}
H^\bullet_\vir(X;A)=\cH^\bullet(X)
\end{equation}
In \S\ref{vfcsinglechart}, where we used $C^\bullet_\vir(X;A):=C_{\dim X_\alpha-\bullet}(X_\alpha,X_\alpha\setminus X)$, this isomorphism was simply (a strong form of) Poincar\'e duality.

Let us now discuss our general approach to the isomorphism \eqref{homotopyisogoal}, which we think of as a generalized form of Poincar\'e duality.  An efficient approach (in fact, the only approach known to the author) to constructing this isomorphism is through the language of homotopy $\K$-sheaves, and so this is the way we present it.  We develop the necessary sheaf-theoretic foundations in Appendix \ref{homologicalalgebrasection}, so the reader may also wish to refer to that section for more details.

As an introduction to the language of sheaves and homotopy sheaves, let us first use it to give a proof of ordinary Poincar\'e duality (in fact, a strong version for arbitrary compact subsets of a manifold).\footnote{An easier proof (using the fact that \v Cech cohomology satisfies the ``continuity axiom'') is available if $M$ has a smooth structure (along the lines of \cite[pp887--888 Lemmas 3.1 and 3.3]{pardonHS}).  This approach does not seem to apply to the more general setting we need to treat here.}  The following argument appears in full detail in Lemma \ref{poincareduality}.

Fix a closed manifold $M$ of dimension $n$.  For any compact subset $K\subseteq M$, let $\F^\bullet(K)$ denote the complex $C_{n-\bullet}(M,M\setminus K)$.  This object $\F^\bullet$ is a \emph{$\K$-presheaf\footnote{The prefix ``$\K$-'' indicates sections are given over compact sets instead of open sets.  For technical reasons, it is more convenient to work with $\K$-presheaves rather than presheaves, though at the conceptual level, the reader may safely ignore the difference.} of complexes} (or \emph{complex of $\K$-presheaves}) on $M$, which just means that we have natural maps $\F^\bullet(K)\to\F^\bullet(K')$ for $K\supseteq K'$, which are suitably compatible with each other.  Now $\F^\bullet$ satisfies the following key properties:
\begin{rlist}
\item(``$\F^\bullet$ is a homotopy $\K$-sheaf'') The total complex of the following double complex is acyclic:
\begin{equation}
\F^\bullet(K_1\cup K_2)\to\F^\bullet(K_1)\oplus\F^\bullet(K_2)\to\F^\bullet(K_1\cap K_2)
\end{equation}
This is essentially just a restatement of the Mayer--Vietoris exact sequence.
\item(``$\F^\bullet$ is pure and $H^0\F^\bullet=\oo_M$'') The homology of $\F^\bullet(\{p\})$ (namely $H_{n-\bullet}(M,M\setminus p)$) is concentrated in degree zero, where it can be canonically identified with the fiber of $\oo_M$ (the orientation sheaf of $M$) at $p\in M$.
\end{rlist}
(the precise notions of a homotopy $\K$-sheaf and of purity are given in Definitions \ref{homotopyKsheaf} and \ref{purehomotopyKsheaf}; in the above statements they have been simplified for sake of exposition).  A formal consequence (Proposition \ref{cechcohomologypurehomotopysheaf}) of the fact that $\F^\bullet$ is a pure homotopy $\K$-sheaf with $H^0\F^\bullet=\oo_M$ is that there is a canonical isomorphism:
\begin{equation}\label{PDintropurehomotopysheafex}
H_{n-i}(M,M\setminus K)=H^i\F^\bullet(K)=\cH^i(K;\oo_M)
\end{equation}
Specializing to $K=M$, we have derived the Poincar\'e duality isomorphism $H_{n-i}(M)=\cH^i(M;\oo_M)$.

This argument generalizes as follows to prove the isomorphism \eqref{homotopyisogoal}.  For sake of concreteness, let us take $C^\bullet_\vir(X;A)$ to be the mapping cone of \eqref{multichartcomplex}, though the general case is not much different.  First of all, we observe that there is a natural complex of $\K$-presheaves $\F^\bullet$ on $X$ whose complex of global sections is $C^\bullet_\vir(X;A)$.  Namely, to get $\F^\bullet(K)$ we simply replace every occurence of $U_\alpha$, $U_\beta$, or $U_{\alpha\beta}$ in \eqref{multichartcomplex} by its intersection with $K$.  Now $\F^\bullet$ is a homotopy $\K$-sheaf (extensions of homotopy $\K$-sheaves are homotopy $\K$-sheaves by Lemma \ref{associatedgradedhomotopysheaf}, and each of the individual complexes appearing in \eqref{multichartcomplex} gives a homotopy $\K$-sheaf by Mayer--Vietoris).  To prove that $\F^\bullet$ is pure and to identify its $H^0$, we can calculate $H^i\F^\bullet(\{p\})$ using a spectral sequence which degenerates at the $E_2$ term (this is the argument in Lemma \ref{gluingcechsimple}).  It is this second step where it matters that we glued the complexes together ``correctly''.  Since $\F^\bullet$ is a pure homotopy $\K$-sheaf and we have identified its $H^0$, this gives the desired isomorphism.

Let us also mention that the fact that the virtual cochain complex $C^\bullet_\vir(X;A)$ is naturally the complex of global sections of a homotopy $\K$-sheaf on $X$ plays a key role in proving other crucial properties in addition to \eqref{homotopyisogoal}.

\subsection{Floer-type homology theories}\label{floerintrosection}

In this section, we introduce the basic ideas needed to apply our methods to construct Hamiltonian Floer homology in the setting of non-degenerate periodic orbits and non-transverse moduli spaces of Floer trajectories.  A toy example of the same flavor (which we mention only for sake of analogy) is the problem of defining Morse homology from a Morse function with gradient-like vector field which is not necessarily Morse--Smale.

The methods developed thus far (the VFC package and the framework for constructing implicit atlases) are robust in that they apply to moduli spaces of Floer trajectories without much modification.  The main task is to add a layer of (rather intricate) combinatorics and algebra to properly organize together the information they yield.  Essentially what we must do is execute the key diagram \eqref{fundisointro}--\eqref{spushforwardintro} on the chain level.

We approach the problem in two logically separate steps.  In \S\ref{floerintroIA}, we describe the implicit atlases we put on the moduli spaces of Floer trajectories.  In \S\ref{floerintroVFC}, we describe how to use the VFC package to define Floer-type homology groups from an appropriate abstract collection of ``flow spaces'' equipped with implicit atlases.

\subsubsection{The system of implicit atlases}\label{floerintroIA}

We describe the ``compatible system of implicit atlases'' we put on the moduli spaces of Floer trajectories relevant for defining Hamiltonian Floer homology.  For sake of exposition, we will imagine we are in an artificially simplified setting (the reader may refer to \S\ref{hamiltonianfloersec} for the full details).

We assume there are just three periodic orbits $p\prec q\prec r$ (ordered by action).  Hence there are just three moduli spaces we have to deal with: $\Mbar(p,q)$, $\Mbar(q,r)$, and $\Mbar(p,r)$ (which are all compact).  There is a single concatenation map:
\begin{equation}\label{concatintro}
\Mbar(p,q)\times\Mbar(q,r)\to\Mbar(p,r)
\end{equation}
and it is a homeomorphism onto its image.  We describe the construction of implicit atlases ((\ref{IAPQ})--(\ref{IAPQR}) below) which are enough to provide a robust notion of ``coherent system of virtual fundamental cycles'' with which we can define Floer homology.

The moduli spaces $\Mbar(p,q)$ and $\Mbar(q,r)$ do not contain any ``broken trajectories''.  There is thus a straightforward generalization of the construction in \S\ref{constructJholimplicitatlasintro} by which we may define:
\begin{rlist}
\item\label{IAPQ}An implicit atlas $A(p,q)$ on $\Mbar(p,q)$.
\item\label{IAQR}An implicit atlas $A(q,r)$ on $\Mbar(q,r)$.
\rlistsave
\end{rlist}
Now since $\Mbar(p,r)$ contains the codimension one boundary stratum $\Mbar(p,q)\times\Mbar(q,r)$, we cannot expect to equip it with an implicit atlas in the sense of Definition \ref{IAproto} or \ref{implicitatlasdefinition}.  Rather, we would like to equip it with an \emph{implicit atlas with boundary}.  Such an atlas on a space $X$ is given by the same data as an implicit atlas, except that in addition we specify closed subsets $\partial X_I\subseteq X_I$ for all $I\subseteq A$ which are compatible with the $\psi_{IJ}$, and we modify the transversality axioms to assert that $X_I^\reg$ is a manifold with boundary $\partial X_I\cap X_I^\reg$.  In particular, there should already be a natural choice of $\partial X\subseteq X$, which in the present case is simply the definition $\partial\Mbar(p,r):=\Mbar(p,q)\times\Mbar(q,r)$.  The notion of an implicit atlas with boundary is formulated so that the natural generalization of the construction in \S\ref{constructJholimplicitatlasintro} defines:
\begin{rlist}
\rlistresume
\item\label{IAPR}An implicit atlas with boundary $A(p,r)$ on $\Mbar(p,r)$.
\rlistsave
\end{rlist}
For this atlas, the closed subsets $\partial X_I\subseteq X_I$ are simply the loci where the thickened trajectory has a break at $q$.

Of course, to have any reasonable notion of ``coherent virtual fundamental cycles'' for the spaces $\Mbar(\cdot,\cdot)$, we need to ``relate'' the atlases (\ref{IAPQ}), (\ref{IAQR}), (\ref{IAPR}).  It turns out that there is a natural way to do this ``over $\Mbar(p,q)\times\Mbar(q,r)$'' which we now describe.

First, let us observe that (\ref{IAPQ}) and (\ref{IAQR}) naturally give rise to an implicit atlas on $\Mbar(p,q)\times\Mbar(q,r)$.  This is a special case of a general observation: implicit atlases $A$ on $X$ and $A'$ on $X'$ induce an implicit atlas $A\sqcup A'$ (disjoint union of index sets) on $X\times X'$, simply by defining $(X\times X')_{I\sqcup I'}:=X_I\times X'_{I'}$.  Hence there is a ``product implicit atlas'' $A(p,q)\sqcup A(q,r)$ on $\Mbar(p,q)\times\Mbar(q,r)$.

Second, let us observe that (\ref{IAPR}) naturally gives rise to an implicit atlas on $\partial\Mbar(p,r)=\Mbar(p,q)\times\Mbar(q,r)$.  This is also a special case of a general observation: an implicit atlas with boundary on $X$ induces an implicit atlas (with the same index set) on $\partial X$, simply by defining $(\partial X)_I:=\partial X_I$.  Hence there is a ``restriction to the boundary'' implicit atlas $A(p,r)$ on $\Mbar(p,q)\times\Mbar(q,r)$.

Now, a good notion of ``compatibility'' between two implicit atlases $A$ and $B$ on a space $X$ is the existence of an implicit atlas on $X$ with index set $A\sqcup B$, whose subatlases $A$ and $B$ coincide with the given atlases.  More importantly, this seems to be the notion of compatibility which arises most naturally in practice (in particular, there is usually a canonical choice for the atlas $A\sqcup B$ which does not depend on any ``extra choices'').  Hence the final implicit atlas we need is:
\begin{rlist}
\rlistresume
\item\label{IAPQR}An implicit atlas $A(p,q)\sqcup A(q,r)\sqcup A(p,r)$ on $\Mbar(p,q)\times\Mbar(q,r)$ whose subatlas $A(p,q)\sqcup A(q,r)$ coincides with the ``product implicit atlas'' above and whose subatlas $A(p,r)$ coincides with the ``restriction to the boundary'' above.
\end{rlist}
This atlas (\ref{IAPQR}) is constructed essentially using the same ideas from \S\ref{constructJholimplicitatlasintro}, however a few remarks are in order about its definition.

The thickened moduli spaces for the atlas (\ref{IAPQR}) are defined in the usual way, as moduli spaces of broken Floer trajectories $u:C\to M\times S^1$ with some intersection conditions with divisors $D_\alpha$, satisfying a ``thickened $\delbar$-equation''.  Let us discuss these conditions more precisely.  The broken Floer trajectories in question are trajectories from $p$ to $r$ broken at $q$, or, equivalently, a pair of trajectories, one from $p$ to $q$ and one from $q$ to $r$.  Let us denote the entire broken trajectory as $u:C_{p,r}\to M\times S^1$, and let us use $C_{p,q},C_{q,r}\subseteq C_{p,r}$ to denote the closed subcurves representing the portion of the trajectory from $p$ to $q$ and from $q$ to $r$ respectively (so $C_{p,r}=(C_{p,q}\sqcup C_{q,r})/\sim$ where $\sim$ identifies a point on $C_{p,q}$ with a point on $C_{q,r}$).  Then the important points are:
\begin{rlist}
\item For thickening datums $\alpha\in A(p,q)$ we require that $(u|C_{p,q})\pitchfork D_\alpha$, and we label the intersections with $\{1,\ldots,r_\alpha\}$, inducing a unique map $\phi_\alpha:C_{p,q}\to\Cbar_{0,2+r_\alpha}$.
\item For thickening datums $\alpha\in A(q,r)$ we require that $(u|C_{q,r})\pitchfork D_\alpha$, and we label the intersections with $\{1,\ldots,r_\alpha\}$, inducing a unique map $\phi_\alpha:C_{q,r}\to\Cbar_{0,2+r_\alpha}$.
\item For thickening datums $\alpha\in A(p,r)$ we require that $(u|C_{p,r})\pitchfork D_\alpha$, and we label the intersections with $\{1,\ldots,r_\alpha\}$, inducing a unique map $\phi_\alpha:C_{p,r}\to\Cbar_{0,2+r_\alpha}$.
\item The thickened $\delbar$-equation we impose is still written in the form \eqref{thickenedintro}, though the term $\lambda_\alpha(e_\alpha)(\phi_\alpha,u)$ is defined to be zero outside the domain of $\phi_\alpha$ (namely $C_{p,q}$, $C_{q,r}$, or $C_{p,r}$, depending on whether $\alpha$ comes from $A(p,q)$, $A(q,r)$ or $A(p,r)$).
\end{rlist}
A good exercise in understanding this definition is to check that the subatlases $A(p,q)\sqcup A(q,r)$ and $A(p,r)$ are the ``product implicit atlas'' and ``restriction to the boundary'' respectively (this is just a matter of matching up definitions).

\begin{remark}
In theory, there is no way to go from a pair of implicit atlases $A$ and $B$ on a space $X$ to an implicit atlas $A\sqcup B$ on $X$ whose subatlases $A$ and $B$ coincide with the given atlases (and there is no uniqueness for atlases $A\sqcup B$ with this property).  However in practice, there is often a natural choice of such an atlas, which moreover is essentially the only reasonable choice.  The atlas (\ref{IAPQR}) is a good example of this.
\end{remark}

The main ideas necessary to define the compatible system of implicit atlases in the general case are all present above; the only real difference is that there are more moduli spaces and more atlases to keep track of.

\subsubsection{Applying the VFC package}\label{floerintroVFC}

We now explain how to use the VFC package to define homology groups from a system of moduli spaces (equipped with implicit atlases) as which appear in a Morse-type setup.  The extra complicating factor in this construction (compared with constructing virtual fundamental classes) is that we must make ``coherent'' choices for each moduli space (i.e.\ choices with certain compatibility properties with respect to the maps between moduli spaces).  Such choices give rise to ``counts'' for the $0$-dimensional moduli spaces and thus to a differential.  Thus we must show that such choices always exist and that the resulting homology groups are independent of this choice.

In our presentation here, we make a number of simplifying assumptions (which we try to point out when relevant) for sake of exposition.  The full details appear in \S\ref{homologygroupssection}.

Let us being by fixing some notation.  We fix a finite set $\PPP$ (the ``set of generators'') equipped with partial order $\prec$ and a grading $\gr:\PPP\to\ZZ$.  For all pairs $p\prec q$ in $\PPP$, we fix a compact Hausdorff space $\X(p,q)$ (the ``space of broken trajectories from $p$ to $q$'').  These spaces $\X(p,q)$ are also equipped with ``concatenation maps'':
\begin{equation}\label{productmapexample}
\X(p,q)\times\X(q,r)\to\X(p,r)
\end{equation}
which satisfy some natural properties, in particular associativity.

We call such a pair $(\PPP,\X)$ (when defined precisely) a \emph{flow category}.  This terminology is due to Cohen--Jones--Segal \cite{cohenjonessegal}, who used it to mean something closer to what we prefer to call a \emph{Morse--Smale flow category}, namely a flow category in which every $\X(p,q)$ is a manifold with corners of dimension $\gr(q)-\gr(p)-1$ in a manner compatible with the concatenation maps \eqref{productmapexample}.  The basic example of a flow category is the flow category of a Morse function: $\PPP$ is the set of critical points, and $\X(p,q)$ is the space of (broken) Morse trajectories from $p$ to $q$; if the Morse function is Morse--Smale, then this flow category is Morse--Smale.

Given a Morse--Smale flow category $(\PPP,\X)$, one can construct a map $d:\ZZ[\PPP]\to\ZZ[\PPP]$ by counting those $\X$ of dimension $0$, and one can prove that $d^2=0$ by considering those $\X$ of dimension $1$.  Our goal is to generalize this construction to flow categories equipped with implicit atlases (meaning the spaces $\X(p,q)$ carry suitably compatible implicit atlases with boundary, with dimension $\gr(q)-\gr(p)-1$).  Specifically, let us assume that we have fixed implicit atlases $\A(p,q)$ on $\X(p,q)$ and that $\A(p,q)\sqcup\A(q,r)=\A(p,r)$ for all $p\prec q\prec r\in\PPP$.

\begin{remark}
There are many different choices for what one could mean by a ``compatible system of implicit atlases'' $\A(p,q)$ on the spaces $\X(p,q)$.  For sake of exposition, we have chosen here the structure for which it is \emph{easiest to apply the VFC package}.  In the actual construction in \S\ref{homologygroupssection}, we use the structure which is the \emph{easiest to construct in practice} (via the natural generalization of \S\ref{floerintroIA}).  As a result, various complexes that here \emph{coincide} are only \emph{canonically quasi-isomorphic} in \S\ref{homologygroupssection}.  Basically, this means that in \S\ref{homologygroupssection}, we will have to take lots of homotopy colimits to make certain maps well-defined on the chain level.  A systematic use of $\infty$-categories of complexes would tame the resulting explosion of notation, at the cost of relying on that more abstract language/machinery.
\end{remark}

We now review what the VFC package gives to us for a flow category $(\PPP,\X)$ equipped with a compatible system of implicit atlases.  For a space $X$ equipped with an implicit atlas with boundary, the VFC package provides \emph{virtual cochain complexes} $C^\bullet_\vir(X\rel\partial)$ and $C^\bullet_\vir(X)$ (defined using the ideas from \S\ref{vfcintrosec}).  Moreover, giving $\partial X$ the ``restriction to the boundary'' implicit atlas, there is a natural map:
\begin{equation}\label{introcoboundary}
C_\vir^{\bullet-1}(\partial X)\to C_\vir^\bullet(X\rel\partial)
\end{equation}
whose mapping cone is $C^\bullet_\vir(X)$ (by definition).  There are canonical isomorphisms:
\begin{align}
\label{HXisointro}H^\bullet_\vir(X)&\xrightarrow\sim\cH^\bullet(X;\oo_X)\\
\label{HXrelisointro}H^\bullet_\vir(X\rel\partial)&\xrightarrow\sim\cH^\bullet(X;\oo_{X\rel\partial})
\end{align}
where $\oo_X$ is the ``virtual orientation sheaf'' of $X$.  There is also a map:
\begin{equation}\label{thomintro}
C^\bullet_\vir(X\rel\partial)\xrightarrow{s_\ast}C_{\dim E_A-\bullet}(E_A,E_A\setminus 0)
\end{equation}
(which, when combined with \eqref{HXrelisointro}, can be thought of as pairing against the virtual fundamental cycle of $X$).  Now, there are also product maps:
\begin{equation}
C_\vir^\bullet(\X(p,q)\rel\partial)\otimes C_\vir^\bullet(\X(q,r)\rel\partial)\to C_\vir^\bullet([\X(p,q)\times\X(q,r)]\rel\partial)
\end{equation}
which, when combined with the concatenation maps \eqref{productmapexample} induce maps:
\begin{equation}\label{productmapfromIA}
C_\vir^\bullet(\X(p,q)\rel\partial)\otimes C_\vir^\bullet(\X(q,r)\rel\partial)\to C_\vir^\bullet(\partial\X(p,r))
\end{equation}
From construction, it is clear that the following diagram commutes:
\begin{equation}\label{productofVCandonEcompare}
\begin{tikzcd}
C_\vir^\bullet(\X(p,q)\rel\partial)\otimes C_\vir^\bullet(\X(q,r)\rel\partial)\ar{r}{\eqref{productmapfromIA}}\ar{d}{s_\ast\otimes s_\ast}&C_\vir^\bullet(\partial\X(p,r))\ar{d}{s_\ast}\\
C_\bullet(E_{\A(p,q)},E_{\A(p,q)}\setminus 0)\otimes C_\bullet(E_{\A(q,r)},E_{\A(q,r)}\setminus 0)\ar{r}&C_\bullet(E_{\A(p,r)},E_{\A(p,r)}\setminus 0)
\end{tikzcd}
\end{equation}
where the bottom map is simply the cartesian product on chains (recall that $\A(p,r)=\A(p,q)\sqcup\A(q,r)$).

Now, let us describe the construction of a boundary operator $d:\ZZ[\PPP]\to\ZZ[\PPP]$ given the flow category $(\PPP,\X)$ and its system of implicit atlases.  To define $d$, we need to choose:\footnote{That these should be contractible choices is suggested by the Dold--Thom--Almgren theorem \cite{almgren} \cite[p430]{gromovI}.}
\begin{rlist}
\item(``Chain level coherent orientations'') Cochains $\lambda(p,q)\in C^0_\vir(\X(p,q)\rel\partial)$ satisfying the following property.  Define:
\begin{equation}\label{boundarychainlevelcoherent}
\mu(p,r):=\sum_{p\prec q\prec r}\lambda(p,q)\cdot\lambda(q,r)\in C_\vir^0(\partial\X(p,r))
\end{equation}
(where we implicitly use \eqref{productmapfromIA} on the right hand side).  We require that $d\lambda(p,r)$ equal (the image under \eqref{introcoboundary}) of $\mu(p,r)$.  Thus $(\mu(p,r),\lambda(p,r))$ defines a cycle in the mapping cone of \eqref{introcoboundary}, and thus a homology class in $H^0_\vir(\X(p,r))=\cH^0(\X(p,r),\oo_{\X(p,r)})$ (by \eqref{HXisointro}).  We require that this homology class coincide with the (given) orientation on $\X(p,r)$.
\item(``Thom cocycles'') Cocycles $[[E_{\A(p,q)}]]\in C^{\dim E_{\A(p,q)}}(E_{\A(p,q)},E_{\A(p,q)}\setminus 0)$ whose pairing with $[E_{\A(p,q)}]$ is $1$ and which are compatible in the sense that the following diagram commutes:
\begin{equation}\label{coherentcofundamentalcocycles}
\hspace{-1.5cm}
\begin{tikzcd}
C_\bullet(E_{\A(p,q)},E_{\A(p,q)}\setminus 0)\otimes C_\bullet(E_{\A(q,r)},E_{\A(q,r)}\setminus 0)\ar{r}\ar{d}[swap]{[[E_{\A(p,q)}]]\otimes[[E_{\A(q,r)}]]}&C_\bullet(E_{\A(p,r)},E_{\A(p,r)}\setminus 0)\ar{ld}{[[E_{\A(p,r)}]]}\\
\ZZ
\end{tikzcd}
\hspace{-1.5cm}
\end{equation}
\end{rlist}
Given such choices, we define the boundary operator $d:\ZZ[\PPP]\to\ZZ[\PPP]$ to have ``matrix coefficients'' $[[E_{\A(p,q)}]](s_\ast\lambda(p,q))$.  One can show that $d^2=0$ by using $d\lambda(p,q)=\mu(p,q)$ and the compatibilities \eqref{productofVCandonEcompare} and \eqref{coherentcofundamentalcocycles}.

Let us now sketch the proof of the existence of $\lambda$ and $[[E]]$ as above.  We will work by induction on $(p,q)$; that is, we show that valid $\lambda(p,q)$ and $[[E_{\A(p,q)}]]$ exist given that we have fixed valid $\lambda(p',q')$ and $[[E_{\A(p',q')}]]$ for all $p\preceq p'\prec q'\preceq q$ (other than $(p',q')=(p,q)$).

To construct $\lambda(p,q)$, argue as follows.  Notice that $\mu(p,q)$ is automatically a cycle (apply \eqref{boundarychainlevelcoherent} to expand its boundary and everything cancels).  Now from a general statement about cycles in mapping cones, the existence of $\lambda(p,q)$ inducing the correct homology class in $H_\vir^0(\X(p,q))=\cH^0(\X(p,q);\oo_{\X(p,q)})$ reduces to showing that the homology class $[\mu(p,q)]\in H^0_\vir(\partial\X(p,q))=\cH^0(\partial\X(p,q),\oo_{\partial\X(p,q)})$ is correct (namely, that it coincides with the image of the desired class under the coboundary map $\cH^0(\X(p,q);\oo_{\X(p,q)})\to\cH^0(\partial\X(p,q),\oo_{\partial\X(p,q)})$).  The key observation is that \emph{this can be checked locally} since $\oo_{\partial\X(p,q)}$ is a sheaf.  Over the top strata of $\partial\X(p,q)$ (that is, those trajectories that split only once), the agreement is clear by the induction hypothesis on $\lambda$ and the compatibility of the given coherent orientations.  Unfortunately, the top strata may not be dense in $\partial\X(p,q)$, so we need to work harder (see Proposition \ref{stratifiedispurehomotopysheaf}).  In the end, we must use the induction hypothesis for \emph{all} $\lambda(p',q')$.

To construct $[[E_{\A(p,r)}]]$, argue as follows.  First, observe that the ``homology diagram'' trivially commutes:
\begin{equation}\label{coherentcofundamentalcocycleshomology}
\begin{tikzcd}
H_\bullet(E_{\A(p,q)},E_{\A(p,q)}\setminus 0)\otimes H_\bullet(E_{\A(q,r)},E_{\A(q,r)}\setminus 0)\ar{r}\ar{d}[swap]{[E_{\A(p,q)}]\otimes[E_{\A(q,r)}]\mapsto 1}&H_\bullet(E_{\A(p,r)},E_{\A(p,r)}\setminus 0)\ar{ld}{[E_{\A(p,r)}]\mapsto 1}\\
\ZZ
\end{tikzcd}
\end{equation}
If the horizontal map in \eqref{coherentcofundamentalcocycles} were a cofibration (think ``injective'') in a suitable sense, then the commutativity of \eqref{coherentcofundamentalcocycleshomology} would be sufficient to imply the existence of $[[E_{\A(p,r)}]]$.  Unfortunately, it is far from clear that this map is a cofibration; moreover, its failure to be a cofibration is a genuine obstruction to defining $[[E_{\A(p,r)}]]$ inductively.  For the correct inductive construction, we must use ``cofibrant replacements'' for the system of complexes $C_\bullet(E_{\A(p,q)},E_{\A(p,q)}\setminus 0)$ (the details of which we leave for \S\ref{homologygroupssection}).

Finally, we should argue that the homology groups defined via a choice of $(\lambda,[[E]])$ are in fact independent of that choice.  For this, we use the usual strategy of constructing a chain map between $(\mathbb Z[\PPP],d)$ and $(\mathbb Z[\PPP],d')$ for any $d$ and $d'$ arising from $(\lambda,[[E]])$ and $(\lambda',[[E]]')$ respectively (plus appropriate chain homotopies).  To construct such a chain map, we use a similar inductive procedure, starting from the base choices $(\lambda,[[E]])$ and $(\lambda',[[E]]')$.  Note that for this argument for independence of choice to work, the inductive nature of our construction is crucial.

\subsection{\texorpdfstring{$S^1$}{S{\textasciicircum}1}-localization}\label{Slocalintrosec}

We now explain our strategy for proving $S^1$-localization results for virtual fundamental cycles.  The full details of our treatment appear in \S\ref{Slocalizationsection}.

The most basic setting in which our results apply (and are interesting) is that of a free $S^1$-space $X$ (i.e.\ a space with a free action of $S^1$) equipped with an $S^1$-equivariant implicit atlas $A$.  Such an atlas is simply an implicit atlas where all the thickenings $X_I$ are equipped with an action of $S^1$ so that the functions $\psi_{IJ}$ are $S^1$-equivariant and the functions $s_I$ are $S^1$-invariant.  This last point bears repeating: $S^1$ does not act on the obstruction spaces $E_I$ (or, alternatively, it acts trivially on them).

In the above setup, our $S^1$-localization result states that $\pi_\ast[X]^\vir=0$, where $\pi:X\to X/S^1$ is the quotient map and $\pi_\ast:\cH^\bullet(X)^\vee\to\cH^\bullet(X/S^1)^\vee$ is the dual of the pullback $\pi^\ast:\cH^\bullet(X/S^1)\to\cH^\bullet(X)$.  Observe that $\pi_\ast$ is an isomorphism for $\bullet=0$ (since $S^1$ is connected), hence this implies that if the virtual dimension $d$ of $X$ is zero, then $[X]^\vir=0$.

One should expect this result to be true if one believes that one can choose a ``chain representative'' of $[X]^\vir$ which is $S^1$-invariant (as the pushforward of such a chain representative is clearly null-homologous).  For instance, in the perturbation approach, this result would follow if one could construct $S^1$-invariant transverse perturbations.  Conversely, one can interpret our vanishing result $\pi_\ast[X]^\vir=0$ as a sense in which our $[X]^\vir$ is $S^1$-invariant at the chain level.

\begin{remark}\label{cannotquotientbySS}
A natural strategy for proving $S^1$-localization results is to consider $X_I/S^1$ as forming an atlas on $X/S^1$ of virtual dimension one less, and then ``pulling back'' to $X$ the virtual fundamental cycle on $X/S^1$ (this is the approach taken by Fukaya--Oh--Ohta--Ono \cite{fukayaono,foootechnicaldetails}).  In the general setting where the $S^1$-action is merely continuous, the implicit atlas on $X$ does not induce an implicit atlas on $X/S^1$ because of the existence of free $S^1$-actions on topological manifolds whose quotients are not manifolds\footnote{Such an action may be constructed out of any non-manifold $X$ for which $X\times\RR$ is a manfold (namely the obvious action on $X\times S^1$).  Many examples of such spaces are known, the first being due to Bing \cite{bingstablemanifoldexample}; see also Cannon \cite{cannon}.}, though we still consider this ``quotient implicit atlas'' in spirit.  This extra generality is convenient, since it means we do not need to construct an $S^1$-equivariant gluing map (providing the local slice necessary to show that the ``quotient implicit atlas'' exists).
\end{remark}

Let us now prove our assertion $\pi_\ast[X]^\vir=0$ in the simple case of a single chart (for which we defined $[X]^\vir$ in \S\ref{vfcsinglechart}).  In this case, the implicit atlas $A=\{\alpha\}$ consists of a topological manifold $X_\alpha$, a function $s_\alpha:X_\alpha\to E_\alpha$, and an identification $X=s_\alpha^{-1}(0)$.  This atlas being $S^1$-equivariant means that $X_\alpha$ is equipped with an $S^1$-action for which $s_\alpha$ is $S^1$-invariant and which induces the given action on $X$.  The desired statement follows from the following diagram, as we explain below:
\begin{equation*}
\begin{CD}
\cH^d(X/S^1)@=H_{\dim E_\alpha-1}^{S^1}(X_\alpha,X_\alpha\setminus X)@>(s_\alpha)_\ast>>H_{\dim E_\alpha-1}^{S^1}(E_\alpha,E_\alpha\setminus 0)\cr
@V\pi^\ast VV@VV\pi^!V@VV\pi^!V\cr
\cH^d(X)@=H_{\dim E_\alpha}(X_\alpha,X_\alpha\setminus X)@>(s_\alpha)_\ast>>H_{\dim E_\alpha}(E_\alpha,E_\alpha\setminus 0)@>[E_\alpha]\mapsto 1>>\ZZ
\end{CD}
\end{equation*}
The map $\pi^\ast$ is pullback under $\pi:X\to X/S^1$.  The maps $\pi^!$ are the maps from the Gysin long exact sequence for an $S^1$-space:
\begin{equation}
\cdots\xrightarrow{\cap e}H_{\bullet-1}^{S^1}(Z)\xrightarrow{\pi^!}H_\bullet(Z)\xrightarrow{\pi_\ast}H_\bullet^{S^1}(Z)\xrightarrow{\cap e}H_{\bullet-2}^{S^1}(Z)\xrightarrow{\pi^!}\cdots
\end{equation}
The two leftmost horizontal identifications are a form of Poincar\'e duality (which can be proven, along with the commutativity of the square, with homotopy $\K$-sheaves as in \S\ref{homotopysheavesintrosec}); it is at this step where we use the fact that $S^1$ acts freely on $X$.

Now $[X]^\vir$ is by definition the composition of the bottom row.  Hence $\pi_\ast[X]^\vir$ is the map from the upper left corner to the bottom right.  On the other hand, the rightmost vertical map $\pi^!:H_{\dim E_\alpha-1}^{S^1}(E_\alpha,E_\alpha\setminus 0)\to H_{\dim E_\alpha}(E_\alpha,E_\alpha\setminus 0)$ vanishes since $S^1$ acts trivially on $E_\alpha$.  It follows that $\pi_\ast[X]^\vir=0$ as desired.

To generalize this approach to arbitrary implicit atlases, we introduce ``$S^1$-equivariant virtual cochain complexes'' which play the role of $H_{\dim E_\alpha-1}^{S^1}(X_\alpha,X_\alpha\setminus X)$ above.  Morally speaking, these $S^1$-equivariant virtual cochain complexes play the role of the virtual cochain complexes of the (non-existent; c.f.\ Remark \ref{cannotquotientbySS}) induced implicit atlas on $X/S^1$.  To define the $S^1$-equivariant virtual cochain complexes, we use the same definition as for the ordinary virtual cochain complexes, except using (shifted) $S^1$-equivariant chains $C_{\bullet-1}^{S^1}$ in place of chains $C_\bullet$.

\begin{remark}
We do not construct an $S^1$-equivariant virtual fundamental class, nor do we address $S^1$-localization for actions which are not free or almost free (having finite stabilizer at every point).  However, the machinery we develop could potentially be used for these purposes, see Remark \ref{SSequivariantvfcmaybe}.
\end{remark}

\section{Implicit atlases}\label{implicitatlasdefsec}

\subsection{Implicit atlases}

\begin{definition}[Implicit atlas]\label{implicitatlasdefinition}
Let $X$ be a compact Hausdorff space.  An \emph{implicit atlas of dimension $d=\vdim_AX$} on $X$ is an index set $A$ along with the following data:
\begin{rlist}
\item(Covering groups) A finite group $\Gamma_\alpha$ for all $\alpha\in A$ (let $\Gamma_I:=\prod_{\alpha\in I}\Gamma_\alpha$).
\item(Obstruction spaces) A finitely generated $\RR[\Gamma_\alpha]$-module $E_\alpha$ for all $\alpha\in A$ (let $E_I:=\bigoplus_{\alpha\in I}E_\alpha$).
\item(Thickenings) A Hausdorff $\Gamma_I$-space $X_I$ for all finite $I\subseteq A$, and a homeomorphism $X\to X_\varnothing$.
\item(Kuranishi maps) A $\Gamma_\alpha$-equivariant function $s_\alpha:X_I\to E_\alpha$ for all $\alpha\in I\subseteq A$ (for $I\subseteq J$, let $s_I:X_J\to E_I$ denote $\bigoplus_{\alpha\in I}s_\alpha$).
\item(Footprints) A $\Gamma_I$-invariant open set $U_{IJ}\subseteq X_I$ for all $I\subseteq J\subseteq A$.
\item(Footprint maps) A $\Gamma_J$-equivariant function $\psi_{IJ}:(s_{J\setminus I}|X_J)^{-1}(0)\to U_{IJ}$ for all $I\subseteq J\subseteq A$.
\item(Regular locus) A $\Gamma_I$-invariant subset $X_I^\reg\subseteq X_I$ for all $I\subseteq A$.
\end{rlist}
which must satisfy the following ``compatibility axioms'':
\begin{rlist}
\item$\psi_{IJ}\psi_{JK}=\psi_{IK}$ and $\psi_{II}=\id$.
\item$s_I\psi_{IJ}=s_I$.
\item$U_{IJ_1}\cap U_{IJ_2}=U_{I,J_1\cup J_2}$ and $U_{II}=X_I$.
\item$\psi_{IJ}^{-1}(U_{IK})=U_{JK}\cap(s_{J\setminus I}|X_J)^{-1}(0)$.\footnote{Added in proof: Abouzaid recently pointed out that this axiom is superfluous; it is implied by the other axioms.}
\item(Homeomorphism axiom) $\psi_{IJ}$ induces a homeomorphism $(s_{J\setminus I}|X_J)^{-1}(0)/\Gamma_{J\setminus I}\to U_{IJ}$.
\rlistsave
\end{rlist}
and the following ``transversality axioms'':
\begin{rlist}
\rlistresume
\item$\psi_{IJ}^{-1}(X_I^\reg)\subseteq X_J^\reg$.
\item$\Gamma_{J\setminus I}$ acts freely on $\psi_{IJ}^{-1}(X_I^\reg)$.
\item(Openness axiom) $X_I^\reg\subseteq X_I$ is open.
\item(Submersion axiom) The map $s_{J\setminus I}:X_J\to E_{J\setminus I}$ is locally modeled on the projection $\RR^{d+\dim E_I}\times\RR^{\dim E_{J\setminus I}}\to\RR^{\dim E_{J\setminus I}}$ over $\psi_{IJ}^{-1}(X_I^\reg)\subseteq X_J$.
\item(Covering axiom) $X_\varnothing=\bigcup_{I\subseteq A}\psi_{\varnothing I}((s_I|X_I^\reg)^{-1}(0))$.
\end{rlist}
\end{definition}

\begin{remark}
The VFC machinery in this paper would go through admitting a slight weakening of the axioms of an implicit atlas.  For example, we only ever use the fact that the openness and submersion axioms hold in a neighborhood of $(s_I|X_I)^{-1}(0)$.  We will not make this precise here, however, since the constructions of implicit atlases we know of would not be made any easier by such a weakening of the axioms.
\end{remark}

\begin{definition}[Smooth implicit atlas]\label{smoothIAdef}
A \emph{smooth structure} on an implicit atlas consists of a smooth structure on each $X_I^\reg$ such that:
\begin{rlist}
\item$\Gamma_I$ acts smoothly on $X_I^\reg$.
\item$s_I$ is smooth over $X_I^\reg$.
\item$s_{J\setminus I}:X_J\to E_{J\setminus I}$ is a smooth submersion over $\psi_{IJ}^{-1}(X_I^\reg)$.
\item$\psi_{IJ}$ is a local diffeomorphism over $\psi_{IJ}^{-1}(X_I^\reg)$.
\end{rlist}
\end{definition}

\begin{remark}
The VFC machinery in this paper applies to implicit atlases (without a smooth structure), though the notion of a smooth implicit atlas may be useful for other applications of implicit atlases.
\end{remark}

\begin{remark}[Using finite $I\to A$ instead of finite $I\subseteq A$]\label{bettertouseovercategory}
An implicit atlas consists of data parameterized by the \emph{category of finite subsets of $A$} (objects: finite subsets, morphisms: inclusions).  A direct modification of the definition allows one to instead use the \emph{category of finite sets over $A$} (objects: finite sets $I\to A$, morphisms: injective maps $I\hookrightarrow J$ compatible with the maps to $A$).  In fact, all constructions of implicit atlases we know of yield implicit atlases in this generalized sense.  We won't need this generalization in this paper, but let us point out some reasons why it may be useful to keep in mind.

With the definition as it stands now (using finite subsets of $A$), we can ``pull back'' an implicit atlas along any \emph{injection} $B\hookrightarrow A$.  If we instead use finite sets over $A$, then we can pull back an implicit atlas along any \emph{map} $B\to A$ (in fact, we can pull back along any \emph{coproduct preserving functor} from finite sets over $B$ to finite sets over $A$, which amounts to the specification of a finite set $I_\beta\to A$ for all $\beta\in B$).  Also, in the category of finite subsets of $A$ we can only take the disjoint union of subsets which are already disjoint; however in the category of finite sets over $A$ there exist arbitrary abstract finite disjoint unions.  This allows some extra (though currently unneeded) flexibility in certain constructions, since we do not need to ensure certain sets are disjoint or that certain maps are injective.

Using the category of sets over $A$ is also the natural perspective to take if we wanted to allow $A$ to be a groupoid instead of a set (then an object is a finite set $I$ along with a collection of objects $\{\alpha_i\}_{i\in I}$ of $A$, and a morphism $(I,\{\alpha_i\}_{i\in I})\to(J,\{\alpha_j\}_{j\in J})$ is an injection $j:I\hookrightarrow J$ along with isomorphisms $\alpha_i\xrightarrow\sim\alpha_{j(i)}$ in $A$).
\end{remark}

\subsection{Implicit atlases with boundary}

\begin{definition}[Implicit atlas with boundary]\label{IAwboundary}
Let $X$ be a compact Hausdorff space together with a closed subset denoted $\partial X\subseteq X$.  An \emph{implicit atlas of dimension $d$ with boundary} on $X$ consists of the same data as an implicit atlas, except that in addition we specify a $\Gamma_I$-invariant closed subset $\partial X_I\subseteq X_I$ for all $I\subseteq A$, such that $\partial X_\varnothing=\partial X$.  We add the following ``compatibility axiom'':
\begin{rlist}
\item$\psi_{IJ}^{-1}(\partial X_I)=(s_{J\setminus I}|\partial X_J)^{-1}(0)$.
\rlistsave
\end{rlist}
and we modify one ``transversality axiom'':
\begin{rlist}
\rlistresume
\item(Submersion axiom) We allow an additional local model $\RR_{\geq 0}\times\RR^{d+\dim E_I-1}\times\RR^{\dim E_{J\setminus I}}\to\RR^{\dim E_{J\setminus I}}$, and $\partial X_J^\reg\subseteq X_J^\reg$ must correspond to the boundary of the local model.
\end{rlist}
\end{definition}

\begin{remark}
Just as a manifold is a special case of a manifold with boundary, an implicit atlas is a special case of an implicit atlas with boundary (namely where $\partial X_I=\varnothing$ for all $I$).
\end{remark}

\begin{definition}[Restriction of implicit atlas to boundary]
Let $X$ be a space with implicit atlas $A$ of dimension $d$ with boundary.  Then this induces an implicit atlas $A$ (the same index set) of dimension $d-1$ on $\partial X$, simply by setting $(\partial X)_I:=\partial X_I$ and restricting the rest of the data to these subspaces.
\end{definition}

\section{The VFC package}\label{Fsheafsection}

In this section, we develop the \emph{VFC package}, which is the algebraic machinery we will apply in later sections to work with virtual fundamental cycles.  The reader should be comfortable with the material from Appendix \ref{homologicalalgebrasection}, where we recall and develop the necessary foundational language of sheaves and homological algebra.

\begin{convention}
In this section, we work over a fixed ground ring $R$, and everything takes place in the category of $R$-modules.  We restrict to implicit atlases $A$ for which $\#\Gamma_\alpha$ is invertible in $R$ for all $\alpha\in A$.
\end{convention}

Let us now introduce the formalism of our VFC package.

For any space $X$ equipped with a \emph{finite} \emph{locally orientable} implicit atlas $A$ of dimension $d$ with boundary, we define \emph{virtual cochain complexes} $C^\bullet_\vir(X;A)$ and $C^\bullet_\vir(X\rel\partial;A)$.  We construct canonical isomorphisms:
\begin{align}
\label{HXiso}H^\bullet_\vir(X;A)&\xrightarrow\sim\cH^\bullet(X;\oo_X)\\
\label{HXreliso}H^\bullet_\vir(X\rel\partial;A)&\xrightarrow\sim\cH^\bullet(X;\oo_{X\rel\partial})
\end{align}
($H^\bullet_\vir$ denotes the cohomology of $C^\bullet_\vir$) for certain \emph{virtual orientation sheaves} $\oo_X$ and $\oo_{X\rel\partial}$ on $X$.  There is a canonical map:
\begin{equation}\label{sastintrotvfcsec}
C^{d+\bullet}_\vir(X\rel\partial;A)\xrightarrow{s_\ast}C_{\dim E_A-\bullet}(E_A,E_A\setminus 0;\oo_{E_A}^\vee)^{\Gamma_A}
\end{equation}
which can be thought of as the (chain level) \emph{virtual fundamental cycle}.

To study the virtual cochain complexes (in particular, to construct the isomorphisms \eqref{HXiso}--\eqref{HXreliso}), we define complexes of $\K$-presheaves on $X$:
\begin{align}
K&\mapsto C^\bullet_\vir(K;A)\\
K&\mapsto C^\bullet_\vir(K\rel\partial;A)
\end{align}
whose global sections are the virtual cochain complexes.\footnote{Note that this is a certain abuse of notation, as $A$ is not an implicit atlas on $K\subsetneqq X$.}\footnote{Note that the map $s_\ast$ \eqref{sastintrotvfcsec} is global; it does not exist on $C^\bullet_\vir(K\rel\partial;A)$ for $K\subsetneqq X$ or on any $C^\bullet_\vir(K;A)$.}  We show that they are pure homotopy $\K$-sheaves, and that there are canonical isomorphisms of sheaves on $X$:
\begin{align}
H^0_\vir(-;A)&=\oo_X\\
H^0_\vir(-\rel\partial;A)&=\oo_{X\rel\partial}
\end{align}
The isomorphisms \eqref{HXiso}--\eqref{HXreliso} then follow from Proposition \ref{cechcohomologypurehomotopysheaf}.

The fact that the virtual cochain complexes are the global sections of pure homotopy $\K$-sheaves with known $H^0$ will also play a key role in the applications of the VFC package.

Another useful fact we prove here is that the isomorphisms \eqref{HXiso}--\eqref{HXreliso} are compatible with the long exact sequence of the pair $(X,\partial X)$ in $\cH^\bullet$ and a corresponding long exact sequence of $H^\bullet_\vir$.

\subsection{Orientations}\label{virtorsheaf}

Recall the notion of the orientation sheaf of a topological manifold (resp.\ with boundary) given in Definition \ref{manordef}.

\begin{definition}[Orientation module of a vector space]
Let $E$ be a finite-dimensional vector space over $\RR$.  We let $\oo_E$ denote the orientation module of $E$, namely $H_{\dim E}(E,E\setminus 0)$ (a free $R$-module of rank $1$).\footnote{According to the previous definition, we should really call this $(\oo_E)_0$ (the stalk at $0\in E$ of the orientation sheaf of $E$ considered as a manifold), though we do not anticipate this abuse causing any particular confusion.}
\end{definition}

\begin{definition}[Locally orientable implicit atlas]\label{IAlocalorient}
Let $X$ be a space with implicit atlas $A$ with boundary.  We say that $A$ is \emph{locally orientable}\footnote{This is analogous to the notion of an orbifold being locally orientable.} iff for every $I\subseteq A$ and every $x\in(s_I|X_I^\reg)^{-1}(0)$, the stabilizer $(\Gamma_I)_x$ acts trivially on $(\oo_{X_I^\reg})_x\otimes\oo_{E_I}^\vee$ (this action is always by a sign $(\Gamma_I)_x\to\{\pm 1\}$).  This notion is independent of the ring $R$ (due to our restriction that $\#\Gamma_\alpha$ be invertible in $R$).
\end{definition}

\begin{definition}[Virtual orientation sheaf $\oo_X$ of a space with implicit atlas]\label{IAorientationsheaf}
Let $X$ be a space with locally orientable implicit atlas $A$ with boundary.  Then there exists a sheaf $\oo_{X,A}$ on $X$ equipped with $\Gamma_I$-equivariant isomorphisms $\psi_{\varnothing I}^\ast\oo_{X,A}\xrightarrow\sim\oo_{X_I^\reg}\otimes\oo_{E_I}^\vee$ over $(s_I|X_I^\reg)^{-1}(0)$ for every $I\subseteq A$, which are compatible with the maps $\psi_{IJ}$.  We call $\oo_{X,A}$ the \emph{virtual orientation sheaf} of $X$ (it is unique up to unique isomorphism); it is locally isomorphic to the constant sheaf $\underline R$.  We write $\oo_X$ for $\oo_{X,A}$ when the atlas is clear from context.  We let $\oo_{X\rel\partial}:=j_!j^\ast\oo_X$ where $j:X\setminus\partial X\hookrightarrow X$.
\end{definition}

\subsection{Virtual cochain complexes \texorpdfstring{$C^\bullet_\vir(X;A)$}{C\_vir(X;A)} and \texorpdfstring{$C^\bullet_\vir(X\rel\partial;A)$}{C\_vir(X rel d;A)}}\label{Fsheafdefsec}

\begin{definition}[Deformation to the normal cone]
Let $X$ be a space with finite implicit atlas $A$ with boundary.  For $I\subseteq J\subseteq A$, we define:
\begin{equation*}
X_{I,J,A}:=\left\{(e,t,x)\in E_A\times\RR_{\geq 0}^A\times X_J^\reg\;\middle|\;\begin{matrix}t_\alpha=0\text{ for }\alpha\in A\setminus I\hfill\cr s_\alpha(x)=t_\alpha e_\alpha\text{ for }\alpha\in J\hfill\cr\psi_{\{\alpha\in A:t_\alpha>0\},J}(x)\in X_{\{\alpha\in A:t_\alpha>0\}}^\reg\hfill\end{matrix}\right\}
\end{equation*}
The condition $\psi_{\{\alpha\in A:t_\alpha>0\},J}(x)\in X_{\{\alpha\in A:t_\alpha>0\}}^\reg$ ensures that we only deform to the normal cone of those zero sets $s_{J\setminus I'}^{-1}(0)$ which are cut out transversally ($I'\subseteq I$).  Clearly $(\partial X)_{I,J,A}\subseteq X_{I,J,A}$ is the subset where $x\in\partial X_J^\reg$ (the former being with respect to the restriction of $A$ to $\partial X$).  There are compatible maps:
\begin{equation}\label{pushforward}
X_{I,J,A}\times E_{A'\setminus A}\to X_{I',J',A'}\text{ for }I\subseteq I'\subseteq J'\subseteq J\subseteq A\subseteq A'
\end{equation}
given by $t_{A'\setminus A}=0$ and $x\mapsto\psi_{J',J}(x)$.  For $K\subseteq X$, we let $X_{I,J,A}^K\subseteq X_{I,J,A}$ denote the subspace where $e=0$ and $x\in\psi_{\varnothing J}^{-1}(K)$.  Note that $\Gamma_A$ acts on $X_{I,J,A}$ (acting on $X_J^\reg$ via the projection $\Gamma_A\to\Gamma_J$ and on $E_A$).
\end{definition}

\begin{remark}[Chains and cochains]\label{chainfixing}
We use $C_\bullet(X)$ (resp.\ $C^\bullet(X)$) to denote singular simplicial chains (resp.\ cochains) on a space $X$, and we use $C_\bullet(X,Y)$ to denote the cokernel of $C_\bullet(Y)\hookrightarrow C_\bullet(X)$ for $Y\subseteq X$ (``relative chains'').

Let us also recall the ``Eilenberg--Zilber map'' $C_\bullet(X)\otimes C_\bullet(Y)\to C_\bullet(X\times Y)$ for spaces $X$ and $Y$, corresponding to the standard subdivision of $\Delta^n\times\Delta^m$ into $\binom{n+m}n$ copies of $\Delta^{n+m}$.  The Eilenberg--Zilber map is associative (in the sense that it gives rise to a unique map $C_\bullet(X)\otimes C_\bullet(Y)\otimes C_\bullet(Z)\to C_\bullet(X\times Y\times Z)$) and commutative (in the sense that the following diagram commutes:
\begin{equation}\label{chainproductcommutes}
\begin{CD}
C_\bullet(X)\otimes C_\bullet(Y)@>>>C_\bullet(X\times Y)\cr
@VVV@VVV\cr
C_\bullet(Y)\otimes C_\bullet(X)@>>>C_\bullet(Y\times X)
\end{CD}
\end{equation}
for all $X$ and $Y$).  We should point out that this commutativity fails for some other common models of singular chains, for example singular cubical chains (modulo degeneracies).
\end{remark}

\begin{remark}[Independence of chain model]\label{independenceofchainmodel}
The particular choice of \emph{singular simplicial} chains is not particularly important.  The virtual fundamental classes, etc.\ resulting from our theory should be unchanged by using any other model of singular chains.  This independence would follow immediately (and all issues about finding chain models with good chain level functoriality properties would go away) if we setup the VFC package using $\infty$-categories.
\end{remark}

\begin{definition}[Fundamental cycles of vector spaces]
Let $I$ be a finite subset of an implicit atlas $A$.  We define:
\begin{equation}
C_\bullet(E;I):=C_{\dim E_I+\bullet}(E_I,E_I\setminus 0;\oo_{E_I}^\vee)^{\Gamma_I}
\end{equation}
There is a canonical isomorphism $H_\bullet(E;I)=R$ (concentrated in degree zero), and we denote the canonical generator by $[E_I]\in H_\bullet(E;I)$.
\end{definition}

\begin{definition}[Partial virtual cochain complexes $C^\bullet_\vir(-;A)_{IJ}$ and $C^\bullet_\vir(-\rel\partial;A)_{IJ}$]\label{CSTcomplexdef}
Let $X$ be a space with finite implicit atlas $A$ of dimension $d$ with boundary.  For any compact $K\subseteq X$, we define:
\begin{align}
\label{CcomplexcriticaldefI}C^\bullet_\vir(K\rel\partial;A)_{IJ}&:=C_{d+\dim E_A-\bullet}(X_{I,J,A},X_{I,J,A}\setminus X_{I,J,A}^K;\oo_{E_A}^\vee)^{\Gamma_A}\\
\label{CcomplexcriticaldefII}C^\bullet_\vir(K;A)_{IJ}&:=\left[\begin{matrix}C_{d+\dim E_A-1-\bullet}((\partial X)_{I,J,A},(\partial X)_{I,J,A}\setminus(\partial X)_{I,J,A}^{K\cap\partial X};\oo_{E_A}^\vee)^{\Gamma_A}\hfill\cr\downarrow\cr C_{d+\dim E_A-\bullet}(X_{I,J,A},X_{I,J,A}\setminus X_{I,J,A}^K;\oo_{E_A}^\vee)^{\Gamma_A}\hfill\end{matrix}\right]
\end{align}
It is clear that $K\mapsto C^\bullet_\vir(K;A)_{IJ}$ and $K\mapsto C^\bullet_\vir(K\rel\partial;A)_{IJ}$ are both complexes of $\K$-presheaves on $X$.

There are canonical maps:
\begin{align}
\label{Erelpartialpullback}C^\bullet_\vir(K\rel\partial;A)_{IJ}&\to C^\bullet_\vir(K;A)_{IJ}\\
\label{Ethompushforward}C^{d+\bullet}_\vir(X\rel\partial;A)_{IJ}&\xrightarrow{s_\ast}C_{-\bullet}(E;A)\\
\label{Epushforward1}\hspace{1.5in}C^\bullet_\vir(-;A)_{IJ}&\to C^\bullet_\vir(-;A)_{I',J'}\text{ for }I\subseteq I'\subseteq J'\subseteq J
\end{align}
(\eqref{Ethompushforward} is induced by the projection $X_{I,J,A}\to E_A$, and \eqref{Epushforward1} is induced by \eqref{pushforward} with $A=A'$).  These are compatible with each other in that certain obvious diagrams commute.
\end{definition}

\begin{definition}[Virtual cochain complexes $C^\bullet_\vir(-;A)$ and $C^\bullet_\vir(-\rel\partial;A)$]\label{Ccomplexdef}
Let $X$ be a space with finite implicit atlas $A$ of dimension $d$ with boundary.  For any compact $K\subseteq X$, we define:
\begin{align}
\label{CvirdefI}C^\bullet_\vir(K;A)&:=\hocolim_{I\subseteq J\subseteq A}C^\bullet_\vir(K;A)_{IJ}\\
\label{CvirdefII}C^\bullet_\vir(K\rel\partial;A)&:=\hocolim_{I\subseteq J\subseteq A}C^\bullet_\vir(K\rel\partial;A)_{IJ}
\end{align}
where $\hocolim_{I\subseteq J\subseteq A}$ is the homotopy colimit (Definition \ref{homotopycolimit}) over $2^A$ with structure maps given by $(\#\Gamma_{J\setminus J'})^{-1}$ times \eqref{Epushforward1}.  It is clear that $K\mapsto C^\bullet_\vir(K;A)$ and $K\mapsto C^\bullet_\vir(K\rel\partial;A)$ are both complexes of $\K$-presheaves on $X$.

There are canonical maps:
\begin{align}
\label{Frelpartialpullback}C^\bullet_\vir(K\rel\partial;A)&\to C^\bullet_\vir(K;A)\\
\label{Fpushforward}C^{d+\bullet}_\vir(X\rel\partial;A)&\xrightarrow{s_\ast}C_{-\bullet}(E;A)
\end{align}
(induced by \eqref{Erelpartialpullback}--\eqref{Ethompushforward}).  More precisely, \eqref{Fpushforward} is given by \eqref{Ethompushforward} on the $p=0$ part of the $\hocolim$ and is zero on the rest.
\end{definition}

\begin{definition}[Maps $C^\bullet_\vir(-;A)\to C^\bullet_\vir(-;A')$]
Let $X$ be a space with finite implicit atlases $A\subseteq A'$ with boundary.  There are canonical maps:
\begin{equation}
\label{Epushforward1more}C^\bullet_\vir(-;A)_{IJ}\otimes C_{-\bullet}(E;A'\setminus A)\to C^\bullet_\vir(-;A')_{I',J'}\text{ for }I\subseteq I'\subseteq J'\subseteq J\subseteq A\subseteq A'
\end{equation}
induced by \eqref{pushforward} (of which \eqref{Epushforward1} is a special case).  These are compatible with each other, and thus extend to the homotopy colimit over $I\subseteq J\subseteq A$.  Hence we get canonical maps:
\begin{equation}
\label{Fpushforward1}C^\bullet_\vir(-;A)\otimes C_{-\bullet}(E;A'\setminus A)\to C^\bullet_\vir(-;A')
\end{equation}
which are compatible with \eqref{Frelpartialpullback}--\eqref{Fpushforward}.
\end{definition}

Note that all of the complexes defined here are free $R$-modules (this follows from our assumption that $\#\Gamma_\alpha$ be invertible in $R$) and that everything in sight is compatible with base change $\otimes_RS$ for ring homomorphisms $R\to S$.

\subsection{Isomorphisms \texorpdfstring{$H^\bullet_\vir(X;A)=\cH^\bullet(X;\oo_X)$}{H\_vir(X;A)=H(X;o\_X)} (also \texorpdfstring{$\rel\partial$}{rel d})}\label{Fsheafcechsimplesec}

\begin{lemma}[$C^\bullet_\vir(-;A)_{IJ}$ are pure homotopy $\K$-sheaves]\label{CSTpure}
Let $X$ be a space with finite locally orientable implicit atlas $A$ with boundary.  Then $C^\bullet_\vir(-;A)_{IJ}$ and $C^\bullet_\vir(-\rel\partial;A)_{IJ}$ are pure homotopy $\K$-sheaves on $X$.  Furthermore, there are canonical isomorphisms of sheaves on $X$:
\begin{align}
\label{EisoI}H^0_\vir(-;A)_{IJ}&=j_!j^\ast\oo_X\\
\label{EisoII}H^0_\vir(-\rel\partial;A)_{IJ}&=j_!j^\ast\oo_{X\rel\partial}
\end{align}
where $j:V_I\cap V_J\hookrightarrow X$ for $V_I:=\psi_{\varnothing I}((s_I|X_I^\reg)^{-1}(0))\subseteq X$ (an open subset).
\end{lemma}

\begin{proof}
Applying Lemmas \ref{singularchainshomotopyKsheaf} and \ref{associatedgradedhomotopysheaf}, we see that $C^\bullet_\vir(-;A)_{IJ}$ and $C^\bullet_\vir(-\rel\partial;A)_{IJ}$ are both homotopy $\K$-sheaves.

Now let us calculate $H^\bullet_\vir(K;A)_{IJ}$ and $H^\bullet_\vir(K\rel\partial;A)_{IJ}$.  It follows from the submersion axiom that $X_{I,J,A}$ is a topological manifold of dimension $\vdim_AX+\dim E_A+\#I$ with boundary (not necessarily second countable or paracompact).\footnote{In fact, it carries a natural structure of a manifold with corners, which for the present purpose is mostly irrelevant.}  The boundary $\partial(X_{I,J,A})$ is a (not necessarily disjoint) union of two pieces, namely the locus where $x\in\partial X_J^\reg$ (which is precisely $(\partial X)_{I,J,A}$) and the locus where $t_\alpha=0$ for some $\alpha\in I$.  It is easy to see that the first piece $(\partial X)_{I,J,A}\subseteq\partial(X_{I,J,A})$ is a closed tamely embedded codimension zero submanifold with boundary.  Hence we may apply Poincar\'e--Lefschetz duality in the form of Lemma \ref{poincareduality} to see that:
\begin{align}
H^\bullet_\vir(K\rel\partial;A)_{IJ}&=\cH_c^{\bullet+\#I}(\psi_{\varnothing J}^{-1}(K\cap V_I\cap V_J)\times\RR_{>0}^I,\oo_\RR^{\otimes I}\otimes\oo_{E_J}^\vee\otimes\oo_{X_J^\reg\rel\partial})^{\Gamma_A}\\
H^\bullet_\vir(K;A)_{IJ}&=\cH_c^{\bullet+\#I}(\psi_{\varnothing J}^{-1}(K\cap V_I\cap V_J)\times\RR_{>0}^I,\oo_\RR^{\otimes I}\otimes\oo_{E_J}^\vee\otimes\oo_{X_J^\reg})^{\Gamma_A}
\end{align}
(the orientation sheaf of $X_{I,J,A}$ is given by $\oo_{X_J^\reg}\otimes\oo_{E_A}\otimes\oo_{E_J}^\vee\otimes\oo_\RR^{\otimes I}$).  Now we apply the K\"unneth formula to conclude that:
\begin{align}
H^\bullet_\vir(K\rel\partial;A)_{IJ}&=\cH^\bullet_c(\psi_{\varnothing J}^{-1}(K\cap V_I\cap V_J),\oo_{E_J}^\vee\otimes\oo_{X_J^\reg\rel\partial})^{\Gamma_A}\\
H^\bullet_\vir(K;A)_{IJ}&=\cH^\bullet_c(\psi_{\varnothing J}^{-1}(K\cap V_I\cap V_J),\oo_{E_J}^\vee\otimes\oo_{X_J^\reg})^{\Gamma_A}
\end{align}
Since $X$ is locally orientable, we have $\oo_{X_J^\reg}\otimes\oo_{E_J}^\vee=\psi_{\varnothing J}^\ast\oo_X$, so we can write:
\begin{align}
H^\bullet_\vir(K\rel\partial;A)_{IJ}&=\cH^\bullet_c(\psi_{\varnothing J}^{-1}(K\cap V_I\cap V_J),\psi_{\varnothing J}^\ast\oo_{X\rel\partial})^{\Gamma_A}\\
H^\bullet_\vir(K;A)_{IJ}&=\cH^\bullet_c(\psi_{\varnothing J}^{-1}(K\cap V_I\cap V_J),\psi_{\varnothing J}^\ast\oo_X)^{\Gamma_A}
\end{align}
Now $\Gamma_A$ acts through the projection $\Gamma_A\to\Gamma_J$, and $\psi_{\varnothing J}:\psi_{\varnothing J}^{-1}(K\cap V_I\cap V_J)\to K\cap V_I\cap V_J$ is exactly the quotient by $\Gamma_J$.  Hence by Lemma \ref{cechcohomologyfinitequotient} we get isomorphisms:
\begin{align}
H^\bullet_\vir(K\rel\partial;A)_{IJ}&=\cH^\bullet_c(K\cap V_I\cap V_J,\oo_{X\rel\partial})\\
H^\bullet_\vir(K;A)_{IJ}&=\cH^\bullet_c(K\cap V_I\cap V_J,\oo_X)
\end{align}
Now use Lemma \ref{cptcechembeddingproperty} to write this as:
\begin{align}
H^\bullet_\vir(K\rel\partial;A)_{IJ}&=\cH^\bullet(K,j_!j^\ast\oo_{X\rel\partial})\\
H^\bullet_\vir(K;A)_{IJ}&=\cH^\bullet(K,j_!j^\ast\oo_X)
\end{align}
Thus $C^\bullet_\vir(-;A)_{IJ}$ and $C^\bullet_\vir(-\rel\partial;A)_{IJ}$ are both pure, and we manifestly have the desired isomorphisms \eqref{EisoI}--\eqref{EisoII}.
\end{proof}

\begin{lemma}[Local description of isomorphisms $H^0_\vir(-;A)_{II}=j_!j^\ast\oo_X$]\label{localdescofISO}
Let $X$ be a space with finite locally orientable implicit atlas $A$ of dimension $d$ with boundary.  Consider the following maps of complexes of $\K$-presheaves on $X$:
\begin{align}
C^\bullet_\vir(K;A)_{II}&\leftarrow\left[\begin{matrix}C_{d+\dim E_A-1-\bullet}(E_{A\setminus I}\times\partial X_I^\reg,E_{A\setminus I}\times(\partial X_I^\reg\setminus\psi_{\varnothing I}^{-1}(K));\oo_{E_A}^\vee)^{\Gamma_A}\hfill\cr\downarrow\cr C_{d+\dim E_A-\bullet}(E_{A\setminus I}\times X_I^\reg,E_{A\setminus I}\times(X_I^\reg\setminus\psi_{\varnothing I}^{-1}(K));\oo_{E_A}^\vee)^{\Gamma_A}\hfill\end{matrix}\right]\\
C^\bullet_\vir(K\rel\partial;A)_{II}&\leftarrow C_{d+\dim E_A-\bullet}(E_{A\setminus I}\times X_I^\reg,E_{A\setminus I}\times(X_I^\reg\setminus\psi_{\varnothing I}^{-1}(K));\oo_{E_A}^\vee)^{\Gamma_A}
\end{align}
(induced by the corresponding maps on spaces).  These maps are quasi-isomorphisms, and the isomorphisms \eqref{EisoI}--\eqref{EisoII} from Lemma \ref{CSTpure} coincide with the Poincar\'e duality isomorphisms for the complexes on the right above.
\end{lemma}

\begin{proof}
Clear from the proof of Lemma \ref{CSTpure}.
\end{proof}

\begin{proposition}[$C^\bullet_\vir(-;A)$ are pure homotopy $\K$-sheaves]\label{Cpure}
Let $X$ be a space with finite locally orientable implicit atlas $A$ with boundary.  Then $C^\bullet_\vir(-;A)$ and $C^\bullet_\vir(-\rel\partial;A)$ are pure homotopy $\K$-sheaves on $X$.  Furthermore, there are canonical isomorphisms of sheaves on $X$:
\begin{align}
\label{FisoI}H^0_\vir(-;A)&=\oo_X\\
\label{FisoII}H^0_\vir(-\rel\partial;A)&=\oo_{X\rel\partial}
\end{align}
\end{proposition}

\begin{proof}
This is a special case of Lemma \ref{gluingcechsimple}; we just need to make sure all of the hypotheses are satisfied.  The open cover in question is $V_I:=\psi_{\varnothing I}((s_I|X_I^\reg)^{-1}(0))$; it follows from the axioms of an implicit atlas that $V_I\cap V_K\subseteq V_J$ for $I\subseteq J\subseteq K$ and $V_I\cap V_{I'}\subseteq V_{I\cup I'}$ for all $I,I'$.  Now we just need to check that the system of isomorphisms from Lemma \ref{CSTpure} are compatible with the maps of the homotopy diagram.  These being maps of sheaves, it suffices to check compatibility locally (on stalks) and thus is a straightforward calculation (this is where the extra normalization factor of $(\#\Gamma_{J\setminus J'})^{-1}$ in Definition \ref{Ccomplexdef} is important).
\end{proof}

\begin{theorem}[Calculation of $H^\bullet_\vir$]\label{Fsheaffundamentaliso}
Let $X$ be a space with finite locally orientable implicit atlas $A$ with boundary.  Then there are canonical isomorphisms:
\begin{align}
H^\bullet_\vir(X;A)&=\cH^\bullet(X;\oo_X)\\
H^\bullet_\vir(X\rel\partial;A)&=\cH^\bullet(X;\oo_{X\rel\partial})
\end{align}
\end{theorem}

\begin{proof}
By Proposition \ref{cechcohomologypurehomotopysheaf}, this is a consequence of Proposition \ref{Cpure}.
\end{proof}

\subsection{Long exact sequence for the pair \texorpdfstring{$(X,\partial X)$}{(X,dX)}}\label{Fsheafproperties}

In this subsection, we compare two long exact sequences for the pair $(X,\partial X)$, namely the one in \v Cech cohomology and one coming from the virtual cochain complexes.

It follows from the definition that $C^\bullet_\vir(K;A)_{IJ}$ is the mapping cone of the obvious map $C^{\bullet-1}_\vir(K\cap\partial X;A)_{IJ}\to C^\bullet_\vir(K\rel\partial;A)_{IJ}$.  Since homotopy colimit commutes with the formation of mapping cones, we see that the same is true dropping the $_{IJ}$ subscript.  Hence there are natural maps:
\begin{equation}\label{sesofF}
\cdots\to C^{\bullet-1}_\vir(K\cap\partial X;A)\to C^\bullet_\vir(K\rel\partial;A)\to C^\bullet_\vir(K;A)\to\cdots
\end{equation}
and they induce a long exact sequence on cohomology.

Similarly, there is a sequence of sheaves on $X$ which is exact on stalks:
\begin{equation}\label{sesofO}
0\to\oo_{X\rel\partial}\to\oo_X\to i_\ast\oo_{\partial X}\to 0
\end{equation}
(where $i:\partial X\hookrightarrow X$).  This induces a long exact sequence on \v Cech cohomology (Lemma \ref{cechles}).  Note that $\cH^\bullet(X;i_\ast\oo_{\partial X})=\cH^\bullet(\partial X;\oo_{\partial X})$ (Lemma \ref{finitemap}).

\begin{remark}\label{bdryorrmk}
To be slightly pedantic about orientations, it would be more precise to say that $C^\bullet_\vir(-;A)=[C^\bullet_\vir(-\cap\partial X;A)\otimes\oo_\RR\to C^\bullet_\vir(-\rel\partial;A)]$, where we identify $\oo_\RR$ canonically with the orientation line of the normal bundle of $\partial X_I^\reg\subseteq X_I^\reg$.  Also, note that we should either say that the last map in \eqref{sesofO} is odd or that it is really $\oo_X\to i_\ast\oo_{\partial X}\otimes\oo_\RR$ (and is even).
\end{remark}

\begin{proposition}[Compatibility of long exact sequences of the pair $(X,\partial X)$]\label{compatibilityoflongexactsequence}
The following diagram commutes:
\begin{equation}\label{compatibilityoflesdiagram}
\begin{tikzcd}
H^\bullet_\vir(X\rel\partial;A)\ar{r}{\eqref{sesofF}}\ar[equal]{d}{\text{Thm \ref{Fsheaffundamentaliso}}}&H^\bullet_\vir(X;A)\ar{r}{\eqref{sesofF}}\ar[equal]{d}{\text{Thm \ref{Fsheaffundamentaliso}}}&H^\bullet_\vir(\partial X;A)\ar{r}{\eqref{sesofF}}\ar[equal]{d}{\text{Thm \ref{Fsheaffundamentaliso}}}&H^{\bullet+1}_\vir(X\rel\partial;A)\ar[equal]{d}{\text{Thm \ref{Fsheaffundamentaliso}}}\\
\cH^\bullet(X;\oo_{X\rel\partial})\ar{r}{\eqref{sesofO}}&\cH^\bullet(X;\oo_X)\ar{r}{\eqref{sesofO}}&\cH^\bullet(\partial X;\oo_{\partial X})\ar{r}{\eqref{sesofO}}&\cH^{\bullet+1}(X;\oo_{X\rel\partial})
\end{tikzcd}
\end{equation}
\end{proposition}

\begin{proof}
By the definition of the vertical identifications, the commutativity of the first two squares of \eqref{compatibilityoflesdiagram} reduces to the commutativity of the diagram of sheaves:
\begin{equation}\label{sheafcommutativity}
\begin{tikzcd}
H^0_\vir(-\rel\partial;A)\ar{r}{\eqref{sesofF}}\ar[equal]{d}{\text{Prop \ref{Cpure}}}&H^0_\vir(-;A)\ar{r}{\eqref{sesofF}}\ar[equal]{d}{\text{Prop \ref{Cpure}}}&H^0_\vir(-;A)\ar[equal]{d}{\text{Prop \ref{Cpure}}}\\
\oo_{X\rel\partial}\ar{r}{\eqref{sesofO}}&\oo_X\ar{r}{\eqref{sesofO}}&i_\ast\oo_{\partial X}
\end{tikzcd}
\end{equation}
It suffices to check commutativity on stalks, and to do this, we can work with the complexes on the right in Lemma \ref{localdescofISO}, for which the commutativity is clear.

In fact, the commutativity of the last square of \eqref{compatibilityoflesdiagram} is also a consequence of the commutativity of \eqref{sheafcommutativity}.  To see this, consider the following diagram:
\begin{equation*}
\begin{tikzcd}[column sep = small]
H^\bullet_\vir(\partial X;A)\ar[equal]{d}{\text{Prop \ref{cechcohomologypurehomotopysheaf}}}&\ar{l}[swap]{\sim}H^\bullet[C^{\bullet+1}_\vir(X\rel\partial;A)\to C^\bullet_\vir(X;A)]\ar{r}\ar[equal]{d}{\text{Prop \ref{cechcohomologypurehomotopysheaf}}}&H^{\bullet+1}_\vir(X\rel\partial;A)\ar[equal]{d}{\text{Prop \ref{cechcohomologypurehomotopysheaf}}}\\
\cH^\bullet(\partial X;H^0_\vir(-;A))\ar{d}{\sim}&\ar{l}[swap]{\sim}\cH^\bullet(X;[H^0_\vir(-\rel\partial;A)[1]\to H^0_\vir(-;A)])\ar{r}\ar{d}{\sim}&\cH^{\bullet+1}(X,H^0_\vir(-\rel\partial;A))\ar{d}{\sim}\\
\cH^\bullet(\partial X;\oo_{\partial X})&\ar{l}[swap]{\sim}\cH^\bullet(X;[\oo_{X\rel\partial}[1]\to\oo_X])\ar{r}&\cH^{\bullet+1}(X,\oo_{X\rel\partial})
\end{tikzcd}
\end{equation*}
The middle vertical arrow is defined because of the commutativity of the left square in \eqref{sheafcommutativity}.  The top left horizontal arrow is a quasi-isomorphism, and thus every left horizontal arrow is a quasi-isomorphism.  The bottom left square is commutative because of the commutativity of the right square in \eqref{sheafcommutativity}; the other squares are trivially commutative.  Reversing the left horizontal arrows, we see that the outermost square commutes, and this is exactly the last square in \eqref{compatibilityoflesdiagram}.
\end{proof}

\section{Virtual fundamental classes}\label{fundamentalclasssection}

In this section, we use the technical results of \S\ref{Fsheafsection} to define the virtual fundamental class of a space with implicit atlas and derive some of its properties.  We also show how these properties can be used to calculate the virtual fundamental class in some special situations (calculation directly from the definition seems prohibitively complicated in all but the simplest of cases).  The properties we prove here are sufficient for some rudimentary purposes, and we think they at least demonstrate that the virtual fundamental class we have defined is the ``right'' one.

For more sophisticated applications than those considered in this paper, one would certainly like to have more properties than those proven here.  For example, one would very much like to prove the expected formula for $[X\times_MY]^\vir$ in terms of $[X]^\vir$ and $[Y]^\vir$ (given some natural ``fiber product implicit atlas'' on $X\times_MY$, where $X,Y$ are spaces with implicit atlases and $M$ is a manifold).

\begin{convention}
In this section, we work over a fixed ground ring $R$, and everything takes place in the category of $R$-modules.  We restrict to implicit atlases $A$ for which $\#\Gamma_\alpha$ is invertible in $R$ for all $\alpha\in A$.
\end{convention}

\begin{remark}[Comparison of homology theories]\label{alsogetsteenrod}
There are (at least) three natural ``homology groups'' which one can assign to a compact Hausdorff space $X$:
\begin{rlist}
\item\emph{Dual of \v Cech cohomology} $\cH^\bullet(X)^\vee$ ($^\vee$ denotes dual, i.e.\ $\Hom(-,R)$).
\item\emph{\v Cech homology} $\cH_\bullet(X)$ (the inverse limit of the homology of nerves of finite covers).
\item\emph{Steenrod--Sitnikov homology} $\sH_\bullet(X)$ (the homology of the homotopy/derived inverse limit of nerves of finite covers; see \S\ref{steenrodhomologysec}).
\end{rlist}
These are successively more refined, in the sense that there are natural maps:
\begin{equation}
\sH_\bullet(X)\to\cH_\bullet(X)\to\cH^\bullet(X)^\vee
\end{equation}
If $X$ is homeomorphic to a finite CW-complex, then there are natural isomorphisms:
\begin{equation}
\sH_\bullet(X)=H_\bullet(X)\qquad\cH_\bullet(X)=H_\bullet(X)\qquad\cH^\bullet(X)^\vee=H^\bullet(X)^\vee
\end{equation}
It follows that a virtual fundamental class in any of these groups can be used for all the applications we are aware of.  On the other hand, there are some potential advantages to working with the more refined homology groups:
\begin{rlist}
\item$\sH_\bullet$ and $\cH_\bullet$ retain torsion information.
\item$\sH_\bullet$ and $\cH_\bullet$ have a natural ``extension of scalars'' map for any map of rings $R\to S$.
\item$\sH_\bullet$ has the expected long exact sequence for pairs of spaces.
\end{rlist}
\end{remark}

We first define the virtual fundamental class as an element of $\cH_\bullet(X)^\vee$ (with appropriate twisted coefficients) and derive some properties.  At the end, we indicate how to define a canonical lift to $\sH_\bullet(X)$ by working at the level of the derived category.

\subsection{Definition}\label{fclassdefsec}

\begin{definition}[{Virtual fundamental class $[X]^\vir$}]\label{fclassdefinition}
Let $X$ be a space with locally orientable implicit atlas $A$ of dimension $d$ with boundary.  Let $B\subseteq A$ be a finite subatlas (which exists by compactness).  We consider the composite:
\begin{equation}\label{Fpushforwardfundclassdef}
\cH^{d+\bullet}(X;\oo_{X\rel\partial})\overset{\text{Thm \ref{Fsheaffundamentaliso}}}=H^{d+\bullet}_\vir(X\rel\partial;B)\xrightarrow{\eqref{Fpushforward}}H_{-\bullet}(E;B)\xrightarrow{[E_B]\mapsto 1}R
\end{equation}
We thus get a map $\cH^d(X;\oo_{X\rel\partial})\to R$.  Now suppose $B\subseteq B'$ are two finite subatlases.  Then the following diagram commutes (it suffices to check commutativity on stalks, and to do this, we can work with the complexes on the right in Lemma \ref{localdescofISO}, for which the commutativity is clear):
\begin{equation}\label{sheafcommforFC}
\begin{tikzcd}
\oo_{X\rel\partial}\ar[equal]{r}{\text{Thm \ref{Fsheaffundamentaliso}}}\ar[equal]{d}&H^0_\vir(-\rel\partial;B)\ar{d}{\times[E_{B'\setminus B}]}[swap]{\eqref{Fpushforward1}}\\
\oo_{X\rel\partial}\ar[equal]{r}{\text{Thm \ref{Fsheaffundamentaliso}}}&H^0_\vir(-\rel\partial;B')
\end{tikzcd}
\end{equation}
Hence the following diagram commutes (the first square following from \eqref{sheafcommforFC}, the rest being clear):
\begin{equation}\label{Fpushforwardfundclassdefcommdiagram}
\begin{tikzcd}
\cH^{d+\bullet}(X;\oo_{X\rel\partial})\ar[equal]{r}{\text{Thm \ref{Fsheaffundamentaliso}}}\ar[equal]{d}&H^{d+\bullet}_\vir(X\rel\partial;B)\ar{r}{\eqref{Fpushforward}}\ar{d}{\times[E_{B'\setminus B}]}[swap]{\eqref{Fpushforward1}}&H_{-\bullet}(E;B)\ar{r}{[E_B]\mapsto 1}\ar{d}{\times[E_{B'\setminus B}]}&R\ar[equal]{d}\\
\cH^{d+\bullet}(X;\oo_{X\rel\partial})\ar[equal]{r}{\text{Thm \ref{Fsheaffundamentaliso}}}&H^{d+\bullet}_\vir(X\rel\partial;B')\ar{r}{\eqref{Fpushforward}}&H_{-\bullet}(E;B')\ar{r}{[E_{B'}]\mapsto 1}&R
\end{tikzcd}
\end{equation}
This shows that the maps $\cH^d(X;\oo_{X\rel\partial})\to R$ induced by $B$ and $B'$ coincide.  Since any two $B_1,B_2\subseteq A$ are contained in a third $B_1\cup B_2\subseteq A$, we see that the resulting element:
\begin{equation}
[X]^\vir_A\in\cH^d(X;\oo_{X\rel\partial})^\vee
\end{equation}
is independent of $B$.  We write $[X]^\vir$ for $[X]_A^\vir$ when the atlas is clear from context.
\end{definition}

\subsection{Properties}\label{fclassproperties}

\begin{lemma}[{Passing to a subatlas preserves $[X]^\vir$}]\label{sameassubatlas}
Let $X$ be a space with locally orientable implicit atlas $A$ with boundary.  If $B\subseteq A$ is any subatlas, then $[X]^\vir_A=[X]^\vir_B$.
\end{lemma}

\begin{proof}
This follows immediately from the definition.
\end{proof}

\begin{lemma}[{Shrinking the charts preserves $[X]^\vir$}]\label{openreduction}
Let $X$ be a space with locally orientable implicit atlas $A$ with boundary.  Let $A'$ be obtained from $A$ by using instead some open subsets $U_{IJ}'\subseteq U_{IJ}$, $X_I'\subseteq X_I$, $X_I^{\reg\prime}\subseteq X_I^\reg$, and restricting $\psi_{IJ}$, $s_I$ to these subsets, so that $A'$ is also an implicit atlas.  Then $[X]^\vir_A=[X]^\vir_{A'}$.
\end{lemma}

\begin{proof}
We may assume that $A$ is finite.  Certainly there is a map $C^\bullet_\vir(X\rel\partial;A')\to C^\bullet_\vir(X\rel\partial;A)$ which respects the map \eqref{Fpushforward}.  Hence it suffices to show that the following diagram commutes:
\begin{equation}
\begin{tikzcd}
H^\bullet_\vir(X\rel\partial;A')\ar{d}\ar[equal]{r}{\text{Thm \ref{Fsheaffundamentaliso}}}&\cH^\bullet(X;\oo_{X\rel\partial})\ar[equal]{d}\\
H^\bullet_\vir(X\rel\partial;A)\ar[equal]{r}{\text{Thm \ref{Fsheaffundamentaliso}}}&\cH^\bullet(X;\oo_{X\rel\partial})
\end{tikzcd}
\end{equation}
where the left vertical arrow is induced by the obvious pushforward on chains.  By definition of the isomorphisms above, it suffices to check commutativity of the corresponding diagram of sheaves with $\bullet=0$.  This can be checked locally, where it is clear.
\end{proof}

\begin{lemma}[{$[X\sqcup Y]^\vir=[X]^\vir\oplus[Y]^\vir$}]\label{disjointunionfc}
Let $X$ and $Y$ be spaces equipped with locally orientable implicit atlases with boundary, on the same index set $A$ and of the same virtual dimension.  Let us also denote by $A$ the resulting implicit atlas on $X\sqcup Y$ (let $(X\sqcup Y)_I:=X_I\sqcup Y_I$).  Then $A$ on $X\sqcup Y$ is locally orientable with $\oo_{X\sqcup Y}=(i_X)_\ast\oo_X\oplus(i_Y)_\ast\oo_Y$, and $[X\sqcup Y]^\vir=[X]^\vir\oplus[Y]^\vir\in\cH^\bullet(X\sqcup Y;\oo_{X\sqcup Y\rel\partial})^\vee=\cH^\bullet(X;\oo_{X\rel\partial})^\vee\oplus\cH^\bullet(Y;\oo_{Y\rel\partial})^\vee$.
\end{lemma}

\begin{proof}
We may assume that $A$ is finite.  There is a natural isomorphism:
\begin{equation}
C^\bullet_\vir(X\sqcup Y\rel\partial;A)=C^\bullet_\vir(X\rel\partial;A)\oplus C^\bullet_\vir(Y\rel\partial;A)
\end{equation}
compatible with the map to $C_{-\bullet}(E;A)$.  Thus it suffices to show that the following square of isomorphisms commutes:
\begin{equation}
\begin{tikzcd}
H^\bullet_\vir(X\sqcup Y\rel\partial;A)\ar[equal]{d}\ar[equal]{r}{\text{Thm \ref{Fsheaffundamentaliso}}}&\cH^\bullet(X\sqcup Y;\oo_{X\sqcup Y\rel\partial})\ar[equal]{d}\\
H^\bullet_\vir(X\rel\partial;A)\oplus H^\bullet_\vir(Y\rel\partial;A)\ar[equal]{r}{\text{Thm \ref{Fsheaffundamentaliso}}}&\cH^\bullet(X;\oo_{X\rel\partial})\oplus\cH^\bullet(Y;\oo_{Y\rel\partial})
\end{tikzcd}
\end{equation}
By definition of the horizontal maps, it suffices to check the commutativity of the corresponding diagram of sheaves on $X\sqcup Y$, which is clear.
\end{proof}

\begin{lemma}[{$\partial[X]^\vir=[\partial X]^\vir$}]\label{boundaryfundamentalclass}
Let $X$ be a space with locally orientable implicit atlas $A$ with boundary.  Then the dual connecting homomorphism $\delta^\vee:\cH^\bullet(X;\oo_{X\rel\partial})^\vee\to\cH^{\bullet-1}(\partial X;\oo_{\partial X})^\vee$ sends $[X]^\vir$ to $[\partial X]^\vir$.
\end{lemma}

\begin{proof}
We may assume that $A$ is finite.  The map $C^\bullet_\vir(\partial X;A)\to C^{\bullet+1}_\vir(X\rel\partial;A)$ in \eqref{sesofF} commutes with the maps from both of these groups to $C_{d-\bullet}(E;A)$.  Hence the result follows from the commutativity of the last square in Proposition \ref{compatibilityoflongexactsequence}.
\end{proof}

\begin{lemma}[{$(\partial X\to X)_\ast[\partial X]^\vir=0$}]\label{cobordismworks}
Let $X$ be a space with locally orientable implicit atlas with boundary.  Then the pushforward map $\cH^\bullet_\vir(\partial X;\oo_{\partial X})^\vee\to\cH^\bullet(X;\oo_X)^\vee$ annihilates $[\partial X]^\vir$.
\end{lemma}

\begin{proof}
The composition $\cH^{\bullet+1}(X;\oo_{X\rel\partial})^\vee\to\cH^\bullet(\partial X;\oo_{\partial X})^\vee\to\cH^\bullet(X;\oo_X)^\vee$ is zero, and $[\partial X]^\vir$ is in the image of the first map by Lemma \ref{boundaryfundamentalclass}.
\end{proof}

\begin{lemma}[{If $X=X^\reg$ then $[X]^\vir=[X]$}]\label{fclasstransverse}
Let $X$ be a space with locally orientable implicit atlas $A$ with boundary.  If $X^\reg=X$ (so in particular, $X$ is a compact topological manifold with boundary of dimension $d=\vdim_AX$), then $[X]^\vir_A$ is the usual fundamental class of $X$.
\end{lemma}

\begin{proof}
We may replace $A$ with the finite subatlas $\varnothing\subseteq A$.  Now in this case, we have:
\begin{equation}
C^\bullet_\vir(K\rel\partial;\varnothing)=C^\bullet_\vir(K\rel\partial;\varnothing)_{\varnothing\varnothing}=C_{d-\bullet}(X_{\varnothing,\varnothing,\varnothing},X_{\varnothing,\varnothing,\varnothing}\setminus X_{\varnothing,\varnothing,\varnothing}^K)=C_{d-\bullet}(X,X\setminus K)
\end{equation}
It follows that the identification $\cH^\bullet(X;\oo_{X\rel\partial})=H^\bullet_\vir(X\rel\partial;\varnothing)=H_{d-\bullet}(X)$ is simply the usual Poincar\'e duality isomorphism, and the map $H_{-\bullet}(X)\to H_{-\bullet}(E;\varnothing)\to R$ is the usual augmentation on $H_0$.  Hence the composition $\cH^\bullet(X;\oo_{X\rel\partial})\to R$ is pairing against the usual fundamental class of $X$.
\end{proof}

\begin{remark}[{Calculating $[X]^\vir$ using perturbation}]\label{perturbationcalc}
The properties from this section can be used to calculate the virtual fundamental class in the following sense.  Let us suppose we are given an ``explicit'' implicit atlas $A$ on a space $X$.  We may always (for convenience) replace $A$ by a subatlas and/or shrink the charts of $A$ and the virtual fundamental class is preserved (Lemmas \ref{sameassubatlas} and \ref{openreduction}).  Now suppose we can extend $A$ to an implicit atlas with boundary on a space $Y$ with $\partial Y=X\cup X'$ where $X'=(X')^\reg$ (if $A$ is a smooth implicit atlas, then such an extension is obtained if one can ``coherently perturb'' the Kuranishi maps $s_\alpha$ so they become transverse to zero).  Then $X'$ is a closed smooth manifold, and $[X']^\vir=[X']$ is the naive fundamental class (Lemma \ref{fclasstransverse}).  Now $Y$ is a cobordism between $X$ and $X'$, so we have $[X]^\vir=[X']$ as homology classes in $Y$ (Lemmas \ref{disjointunionfc} and \ref{cobordismworks}).  This allows us to understand the pushforward of $[X]^\vir$ under any map $f:X\to Z$ which extends continuously to $Y$.  We do not claim such perturbations always exist; an affirmative answer is provided in very similar, though not identical, contexts by Fukaya--Ono \cite{fukayaono}, Fukaya--Oh--Ohta--Ono \cite{FOOOII,foootechnicaldetails} and McDuff--Wehrheim \cite{mcduffwehrheim}.
\end{remark}

\subsection{Manifold with obstruction bundle}\label{vfcobsbundspace}

A natural ``test case'' (beyond Lemma \ref{fclasstransverse} and Remark \ref{perturbationcalc}) for our definition of the virtual fundamental class is that of the natural implicit atlas on a ``manifold with obstruction bundle'' (the expected answer being the Poincar\'e dual of the Euler class).  More generally, there is a natural implicit atlas on the zero set of a section of a vector bundle over a manifold (the ``manifold with obstruction bundle'' case is when the section is identically zero), and again the expected answer is a type of Euler class.

In this section, we show that our definition of the virtual fundamental class indeed agrees with this expected answer.  To prove this, we use only the properties of the virtual fundamental class from \S\ref{fclassproperties}.

\begin{definition}[Implicit atlas on the zero set of a smooth section]\label{IAonvectorbundlezeroset}
Let $B$ be a smooth manifold with boundary, let $p:E\to B$ be a smooth vector bundle, and let $s:B\to E$ be a smooth section with $s^{-1}(0)$ compact.  We define an implicit atlas of dimension $\dim B-\dim E$ with boundary on $X:=s^{-1}(0)$ as follows.  The index set $A$ consists of all triples $(V_\alpha,E_\alpha,\lambda_\alpha)$ (called \emph{thickening datums}) where:
\begin{rlist}
\item$V_\alpha\subseteq B$ is an open subset.
\item$E_\alpha$ is a finite-dimensional vector space.
\item$\lambda_\alpha:E_\alpha\times V_\alpha\to p^{-1}(V_\alpha)$ is a smooth homomorphism of vector bundles.
\end{rlist}
The thickened spaces are defined as follows:
\begin{equation}
X_I:=\Bigl\{(x,\{e_\alpha\}_{\alpha\in I})\in\bigcap_{\alpha\in I}V_\alpha\times\bigoplus_{\alpha\in I}E_\alpha\Bigm|s(x)+\sum_{\alpha\in I}\lambda_\alpha(x,e_\alpha)=0\Bigr\}
\end{equation}
The Kuranishi map $s_\alpha:X_I\to E_\alpha$ is the obvious projection map, the footprint $U_{IJ}$ is the locus where $x\in V_\alpha$ for all $\alpha\in J$, and the footprint maps $\psi_{IJ}:(s_{J\setminus I}|X_J)^{-1}(0)\to U_{IJ}$ are the obvious forgetful maps.  The compatibility axioms are an easy exercise.

The regular locus $X_I^\reg\subseteq X_I$ is the locus where $X_I$ is ``cut out transversally''; more precisely, $(x,\{e_\alpha\}_{\alpha\in I})\in X_I$ is in $X_I^\reg$ iff the following map is surjective:
\begin{equation}
\Bigl(ds+\sum_{\alpha\in I}d\lambda_\alpha(\cdot,e_\alpha)\Bigr)\oplus\bigoplus_{\alpha\in I}\lambda_\alpha(x,\cdot):T_xB\oplus\bigoplus_{\alpha\in I}E_\alpha\to E_x
\end{equation}
The transversality axioms are an easy exercise.  Thus $A$ is indeed an implicit atlas on $X$.

Since everything here is in the smooth category, $A$ is in fact a smooth implicit atlas.
\end{definition}

\begin{definition}[Euler class]
Let $B$ be a space equipped with a vector bundle $p:E\to B$ of rank $k$.  This induces a canonical locally constant sheaf $\oo_E$ on $B$ whose stalk at $b\in B$ is $\oo_{E_b}$.  Let $\tau_E\in H^k(E,E\setminus 0;p^\ast\oo_E)$ denote the Thom class of $E$ (characterized uniquely by the property that its restriction to any local trivialization $\RR^n\times U\to U$ is the pullback of the tautological class in $H^n(\RR^n,\RR^n\setminus 0;\oo_{\RR^n})$).

The \emph{Euler class} $e(E)\in H^k(B;\oo_E)$ is $s^\ast\tau_E$ where $s:B\to E$ is any section.  Any two sections are homotopic, so $e(E)$ is well-defined.  For any section $s:B\to E$ with $s^{-1}(0)$ compact, the \emph{Euler class with compact support} $e_c(E,s)\in H^k_c(B;\oo_E)$ is $s^\ast\tau_E$.  If $s_0$ and $s_1$ are homotopic through a section $\tilde s:B\times[0,1]\to E$ with compact support, then $e_c(E,s_0)=e_c(E,s_1)$.  The natural map $H^k_c(B;\oo_E)\to H^k(B;\oo_E)$ sends the Euler class with compact support to the Euler class.
\end{definition}

\begin{remark}
The Euler class with compact support can be nonzero even for the trivial vector bundle.
\end{remark}

\begin{proposition}\label{actuallygiveseulerclass}
Let $B$ be a smooth manifold, let $p:E\to B$ be a smooth vector bundle, and let $s:B\to E$ be a smooth section with $s^{-1}(0)$ compact.  Consider the implicit atlas $A$ on $X:=s^{-1}(0)$ from Definition \ref{IAonvectorbundlezeroset}.  Then there is a canonical isomorphism of sheaves $\oo_X=\oo_B\otimes\oo_E^\vee$ on $X$ (in particular, $A$ is locally orientable), and the image of $[X]^\vir$ in $H^{\dim B-\dim E}(B;\oo_B\otimes\oo_E^\vee)^\vee$ equals $e_c(E,s)\cap[B]$.
\end{proposition}

Note that $\cH^\bullet(X;\oo_X)=\varinjlim_{X\subseteq U}H^\bullet(U;\oo_B\otimes\oo_E^\vee)$ (direct limit over open neighborhoods $U\subseteq B$ of $X$), by the continuity axiom of \v Cech cohomology.  The proposition applies equally well with $U$ in place of $B$, and thus determines $[X]^\vir\in\cH^\bullet(X;\oo_X)^\vee$ uniquely (note the use of either Lemma \ref{sameassubatlas} or \ref{openreduction}).

\begin{proof}
The statement about orientation sheaves is a straightforward calculation which we omit.

Choose some smooth family of sections $\tilde s:B\times[0,1]\to E$ with $\tilde s^{-1}(0)$ compact, so that $\tilde s(\cdot,0)=s$ and $s_1:=\tilde s(\cdot,1)$ is transverse to the zero section.

Now $B\times[0,1]$ is a smooth manifold with boundary, so let $\tilde A$ be the implicit atlas on $\tilde X:=\tilde s^{-1}(0)$ from Definition \ref{IAonvectorbundlezeroset}.  Similarly, let $A'$ be the implicit atlas on $X':=(s')^{-1}(0)$.

Now we have two implicit atlases on $X$, namely $A$ and (the restriction to the boundary of) $\tilde A$.  On the other hand, we can exhibit $A$ as a subatlas of $\tilde A$ by the map on thickening datums sending $V_\alpha$ to $V_\alpha\times[0,1]$ and $\lambda_\alpha$ to its obvious extension.  Hence they induce the same virtual fundamental class (Lemma \ref{sameassubatlas}), so we may just write $[X]^\vir$ (dropping the subscript indicating which atlas we use).  The same reasoning applies to $X'$.

Now $\partial\tilde X=X\sqcup X'$, so we have $[X]^\vir=[X']^\vir$ as homology classes on $\tilde X$ (Lemmas \ref{disjointunionfc} and \ref{cobordismworks}).  Hence we have:
\begin{equation}
[X]^\vir=[X']^\vir
\end{equation}
as elements of $H^{\dim B-\dim E}(B;\oo_B\otimes\oo_E^\vee)^\vee$.

We know $X'=(X')^\reg$ (it is cut out transversally), so $X'$ is a closed smooth manifold and $[X']^\vir=[X']$ (Lemma \ref{fclasstransverse}).  But now $[X']=e_c(E,s)\cap[B]$ by elementary Poincar\'e duality.
\end{proof}

\begin{remark}
Proposition \ref{actuallygiveseulerclass} in its present form is useless in practice, because the implicit atlases which arise in interesting examples are not literally isomorphic to the one given in Definition \ref{IAonvectorbundlezeroset}.  One could, though, hope to show that they are equivalent (c.f.\ Remark \ref{equivalenceofIA}) and then apply Proposition \ref{actuallygiveseulerclass}.
\end{remark}

\subsection{Lift to Steenrod homology}

We now define a virtual fundamental class in Steenrod homology $\sH_\bullet(X,\partial X;\oo_X^\vee)$ refining the class in $\cH^\bullet(X;\oo_{X\rel\partial})^\vee$ constructed in \S\ref{fclassdefsec}.  To do this, it suffices to lift our earlier reasoning at the level of homology groups to the level of objects in the derived category.  We will be using the notation and results from \S\ref{steenrodhomologysec}.

\begin{remark}
The natural map $\sH_\bullet(X,\partial X;\oo_X^\vee)\to\cH^\bullet(X;\oo_{X\rel\partial})^\vee$ is an isomorphism if the base ring $R$ is a field (this follows from Lemma \ref{steenrodisderiveddualofcech}), but for $R=\ZZ$ the kernel can be substantial (in particular, it contains all torsion).
\end{remark}

\begin{definition}[Virtual fundamental class in Steenrod homology]
Let $X$ be a space with localy orientable implicit atlas $A$ of dimension $d$ with boundary.  Let $B\subseteq A$ be a finite subatlas.  By Propositions \ref{Cpure} and \ref{cechcohomologypurehomotopysheaf}, there is a canonical isomorphism $\cC^\bullet(X;\oo_{X\rel\partial})=C_\vir^\bullet(X\rel\partial;B)$ in $D(R)$ (observe that the isomorphisms in the proof of Proposition \ref{cechcohomologypurehomotopysheaf} all come from canonical quasi-isomorphisms on the chain level).  There is also a canonical isomorphism $C_{-\bullet}(E;B)=R$ in $D(R)$.    Since $C_\vir^\bullet(X\rel\partial;B)$ is free and bounded above, we have:
\begin{equation}
\Hom_R(C_\vir^\bullet(X\rel\partial;B),C_{-\bullet}(E;B))=R\Hom_{D(R)}(\cC^\bullet(X;\oo_{X\rel\partial}),R)
\end{equation}
Thus the map $s_\ast:C_\vir^{d+\bullet}(X\rel\partial;B)\to C_{-\bullet}(E;B)$ gives rise to an element:
\begin{equation}\label{steenrodfclasseqndef}
[X]^\vir_A:=[s_\ast]\in\Hom_{D(R)}(\cC^\bullet(X;\oo_{X\rel\partial}),R[-d])\overset{\text{Lem \ref{steenrodisderiveddualofcech}}}=\sH_d(X,\partial X;\oo_X^\vee)
\end{equation}
The commutativity of \eqref{sheafcommforFC} implies that this element is independent of the chosen finite subatlas $B\subseteq A$.

By definition, this fundamental class \eqref{steenrodfclasseqndef} projects to the fundamental class of Definition \ref{fclassdefinition} under the map $\sH_\bullet(X,\partial X;\oo_X^\vee)\to\cH^\bullet(X;\oo_{X\rel\partial})^\vee$.  By chasing the various isomorphisms involved in its definition, it can be checked that $[X]^\vir\in\sH_d(X,\partial X;\oo_X^\vee)$ is preserved under extension of scalars.
\end{definition}

The statements and proofs of Lemmas \ref{sameassubatlas}--\ref{fclasstransverse} and Proposition \ref{actuallygiveseulerclass} generalize readily to the case of $[X]^\vir\in\sH_d(X,\partial X;\oo_X^\vee)$.  Note that (the generalization of) Proposition \ref{actuallygiveseulerclass} does not determine $[X]^\vir\in\sH_d(X,\partial X;\oo_X^\vee)$ uniquely since the map from the Steenrod homology of $X$ to the inverse limit of the homology of its neighborhoods is not necessarily an isomorphism (see \cite[p87, Theorem 4]{milnor}).  However, it does at least determine the image of $[X]^\vir$ under any map from $X$ to a finite CW-complex, since any such map extends to some neighborhood of $X$.

\section{Stratifications}\label{auxiliarysection}

In this section, we introduce \emph{implicit atlases with cell-like stratification}.  Roughly speaking, an implicit atlas with cell-like stratification is an implicit atlas on a stratified space, along with suitably compatible stratifications on each of the thickenings.  We show how to apply the VFC package in this setting to obtain a ``stratum by stratum'' understanding of virtual fundamental cycles.  We also define the \emph{product implicit atlas}, a natural implicit atlas on $X\times Y$ induced from implicit atlases on $X$ and $Y$.

\begin{convention}
In this section, we work over a fixed ground ring $R$, and everything takes place in the category of $R$-modules.  We restrict to implicit atlases $A$ for which $\#\Gamma_\alpha$ is invertible in $R$ for all $\alpha\in A$.
\end{convention}

\subsection{Implicit atlas with cell-like stratification}\label{stratifiedIAsec}

\begin{definition}[Stratification]\label{stratificationdefinition}
Let $X$ be a topological space and let $\SSS$ be a poset.  A \emph{stratification} of $X$ by $\SSS$ is a lower semicontinuous function $X\to\SSS$.  We let $X^{\leq\s}$ (resp. $X^\s$, $X^{\geq\s}$) denote the inverse image of $\SSS^{\leq\s}$ (resp.\ $\{\s\}$, $\SSS^{\geq\s}$); lower semicontinuity of $X\to\SSS$ means by definition that every $X^{\geq\s}$ is open.

For a pair of spaces $(X,\partial X)$, a stratification $(X,\partial X)\to(\SSS,\partial\SSS)$ shall mean a stratification $X\to\SSS$ along with a downward closed (i.e.\ closed under taking smaller elements) subset $\partial\SSS\subseteq\SSS$ such that $X^{\partial\SSS}=\partial X$ (note that this implies automatically that $\partial X\subseteq X$ is closed).
\end{definition}

\begin{definition}[Cell-like stratification]\label{celllikedef}
Let $(M,\partial M)$ be a topological manifold with boundary and stratification by $(\SSS,\partial\SSS)$, and fix a map $\dim:\SSS\to\ZZ$.  This stratification is called \emph{cell-like} iff each pair $(M^{\leq\s},M^{<\s})$ is a topological manifold with boundary of dimension $\dim\s$.

If $(M,\partial M)\to(\SSS,\partial\SSS)$ is cell-like, then so is $\partial M\to\partial\SSS$ (with empty boundary) and $(M^{\leq\s},M^{<\s})\to(\SSS^{\leq\s},\SSS^{<\s})$.
\end{definition}

\begin{example}\label{celllikeexample}
Let $T$ be a simplicial complex, and suppose that its geometric realization $M:=\left|T\right|$ is a topological manifold with boundary (then $\partial M$ necessarily corresponds to a subcomplex $\partial T\subseteq T$).  Let $(\SSS,\partial\SSS):=(\F(T),\F(\partial T))$ be the face poset of $T$ (resp.\ $\partial T$).  Then the stratification $(M,\partial M)\to(\SSS,\partial\SSS)$ is cell-like (note that $T$ need not be a PL manifold).
\end{example}

\begin{example}\label{celllikecornerexample}
The natural stratification $\RR_{\geq 0}^k\times\RR^{n-k}\to\{(0,\infty)>\{0\}\}^k$ is cell-like.  The natural stratification $\RR_{\geq 0}^k\times\RR^{n-k}\to\ZZ$ given by ($n$ minus) the number of zeros in the first $k$ coordinates is \emph{not} cell-like for $k\geq 2$.
\end{example}

\begin{lemma}[Some local properties of cell-like stratifications]\label{celllikelocalstructure}
Let $(M,\partial M)\to(\SSS,\partial\SSS)$ be cell-like.  Then:
\begin{rlist}
\item\label{inbdry}If $\s\prec\ttt$ then $M^\s\subseteq\partial M^{\leq\ttt}$.
\item\label{toplessdim}If $M^\s\ne\varnothing$, then $\dim\s\leq\dim M$.
\item\label{topopen}If $\dim\s=\dim M$, then $M^\s\subseteq M$ is open.
\item\label{bdrystratumstruct}If $\dim\s=\dim M-1$, $\s\in\partial\SSS$, and $M^\s\ne\varnothing$, then $\#\SSS^{>\s}=1$.
\item\label{intstratumstruct}If $\dim\s=\dim M-1$, $\s\notin\partial\SSS$, and $M^\s\ne\varnothing$, then $\#\SSS^{>\s}=2$, and these $M^{\leq\ttt}$ give collars on either side of $M^\s\subseteq M$ (so in particular $M^\s\subseteq M$ is locally flat).
\end{rlist}
\end{lemma}

\begin{proof}
Since $\s\prec\ttt$, we have $M^\s\subseteq M^{<\ttt}=\partial M^{\leq\ttt}$, giving (\ref{inbdry}).

Recall Brouwer's ``invariance of domain'', which implies that if a subset $X\subseteq\RR^n$ is (in the subspace topology) locally homeomorphic to $\RR^m$, then $m\leq n$, with equality iff $X$ is open.  This immediately gives (\ref{toplessdim}), (\ref{topopen}).

We prove (\ref{bdrystratumstruct}).  Fix $p\in M^\s$.  A neighborhood of $p$ is covered by strata $\succ\s$, so since $M^\s\subseteq M$ is not open, there exists a stratum $\ttt\succ\s$.  Now $M^\s\subseteq\partial M^{\leq\ttt}$ is open by (\ref{topopen}) (applied to $M^{<\ttt}\to\SSS^{<\ttt}$).  A doubling argument (and invariance of domain) near $p\in\partial M$ shows that $M^s\cup M^\ttt$ contains a neighborhood of $p$.  Since we have exhausted a neighborhood of $p\in M^\s$, it follows using (\ref{inbdry}) that there is no other stratum $\succ\s$.

We prove (\ref{intstratumstruct}).  Fix $p\in M^\s$.  A neighborhood of $p$ is covered by strata $\succ\s$, so since $M^\s\subseteq M$ is not open, there exists a stratum $\ttt\succ\s$.  Now $M^\s\subseteq\partial M^{\leq\ttt}$ is open by (\ref{topopen}), as is $M^\ttt\subseteq M$.  This gives a collar on one side of $M^\s\subseteq M$.  This still does not fill out a neighborhood of $p$, so there exists $\ttt'\ne\ttt$ with $\ttt'\succ\s$, giving a collar on the other side of $M^\s\subseteq M$.  Now by invariance of domain, we have thus exhausted a neighborhood of $p\in M^\s$ inside $M$, and hence it follows using (\ref{inbdry}) that there are no more strata $\succ\s$ other than $\ttt,\ttt'$.
\end{proof}

\begin{definition}[Implicit atlas with cell-like stratification]\label{IAwithstratification}
Let $(X,\partial X)$ be a pair of compact Hausdorff spaces equipped with a stratification $(X,\partial X)\to(\SSS,\partial\SSS)$ ($\SSS$ finite) along with a map $\dim:\SSS\to\ZZ$.  Also fix ``orientation data'':
\begin{rlist}
\item For every $\s\in\SSS$, an orientation line $\oo_\s$ (i.e.\ a free $\ZZ/2$-graded $\ZZ$-module of rank one).
\item For $\codim(\s\preceq\ttt)=1$, an odd ``coboundary'' map $\oo_\s\to\oo_\ttt$.
\end{rlist}
An \emph{implicit atlas of dimension $d$ with boundary and cell-like stratification} on $(X,\partial X)\to(\SSS,\partial\SSS)$ consists of the same data as an implicit atlas with boundary, except that in addition we specify a $\Gamma_I$-invariant stratification $(X_I,\partial X_I)\to(\SSS,\partial\SSS)$ for all $I\subseteq A$, whose restriction to $X_\varnothing$ is the given stratification.  We add the following ``compatibiltiy axiom'':
\begin{rlist}
\item The restriction of the stratification on $X_J$ to $(s_{J\setminus I}|X_J)^{-1}(0)$ coincides with the pullback of the stratification on $X_I$ via $\psi_{IJ}$.
\rlistsave
\end{rlist}
and we modify one ``transversality axiom'':
\begin{rlist}
\rlistresume
\item\label{stratiscelllike}(Submersion axiom) In the local model $\RR^{d+\dim E_I}\times\RR^{\dim E_{J\setminus I}}\to\RR^{\dim E_{J\setminus I}}$ (or $(\RR_{\geq 0}\times\RR^{d+\dim E_I-1})\times\RR^{\dim E_{J\setminus I}}\to\RR^{\dim E_{J\setminus I}}$), we require the stratification on the domain to be pulled back from a cell-like stratification on the first factor (for $\s\mapsto\dim\s+\dim E_I$).
\end{rlist}
We also specify isomorphisms of sheaves $\oo_{(X_I^\reg)^{\leq\s}}\otimes\oo_\s\xrightarrow\sim\oo_{X_I^\reg}$ over $(X_I^\reg)^{\leq\s}$ which are compatible with $\psi_{IJ}$ and such that the following diagram commutes:
\begin{equation}
\begin{tikzcd}
\oo_{(X_I^\reg)^{\leq\s}}\otimes\oo_\s\ar{r}\ar{d}&\oo_{X_I^\reg}\ar[equals]{d}\\
\oo_{(X_I^\reg)^{\leq\ttt}}\otimes\oo_\ttt\ar{r}&\oo_{X_I^\reg}
\end{tikzcd}
\end{equation}
for $\codim(\s\preceq\ttt)=1$ (the left vertical map is the tensor product of the coboundary map $\oo_\s\to\oo_\ttt$ and the inverse of the boundary map $\oo_{(X_I^\reg)^{\leq\ttt}}\to\oo_{(X_I^\reg)^{\leq\s}}$).  In other words, we identify $\oo_\s$ with the orientation line of the normal bundle of $(X_I^\reg)^{\leq\s}\subseteq X_I^\reg$, so that the coboundary maps $\oo_\s\to\oo_\ttt$ coincide with the geometric coboundary maps on normal bundles.

Given an implicit atlas with boundary and cell-like stratification $A$ on $(X,\partial X)\to(\SSS,\partial\SSS)$, we may obtain by restriction to the corresponding strata an implicit atlas with cell-like stratification on $\partial X\to\partial\SSS$ (empty boundary, tensor every orientation line with $\oo_\RR$) and an implicit atlas with boundary and cell-like stratification on $(X^{\leq\s},X^{<\s})\to(\SSS^{\leq\s},\SSS^{<\s})$ (tensor every orientation line with $\oo_\s^\vee$).
\end{definition}

\begin{remark}
There should be a slightly more general setting for the results of this section (and their proofs), which takes as input a weakened version of Definition \ref{celllikedef}.  For example, it is probably enough to require that in a neighborhood of any $p\in M$, the closure of any (local) component of $M^\s$ is a manifold with boundary whose interior is this local component (this is satisfied by both stratifications in Example \ref{celllikecornerexample}); basically this allows ``non-embedded faces''.
\end{remark}

\subsection{Stratified virtual cochain complexes}\label{stratvirtcochainssec}

In this section, we apply the VFC package to obtain a ``stratum by stratum'' understanding of virtual fundamental cycles on a space with implicit atlas with cell-like stratification.  To do this, we build a complex out of the virtual cochain complexes associated to each stratum and then study the properties of this larger complex.  This construction can be viewed as a generalization of the definition of $C_\vir^\bullet(X;A)$ as the mapping cone $[C^\bullet_\vir(\partial X;A)\to C^\bullet_\vir(X\rel\partial;A)]$, and the main result here can be viewed as a generalization of Lemma \ref{compatibilityoflongexactsequence}.

\begin{example}\label{stratifiedchainsremark}
As a first step towards understanding the main construction of this section, let us first describe a similar construction in a more familiar setting.  Let $M$ be a topological manifold with cell-like stratification by $\SSS$.  Define:
\begin{equation}
C_\bullet(M;\SSS):=\bigoplus_{\s\in\SSS}C_{\bullet-(\dim M-\dim\s)}(M^{\leq\s};\oo_M\otimes\oo_{M^{\leq\s}}^\vee)
\end{equation}
equipped with the differential given by the sum over $\codim(\s\preceq\ttt)=1$ of the pushforwards $C_\bullet(M^{\leq\s})\to C_\bullet(M^{\leq\ttt})$ (covered by the dual of the boundary map $\oo_{M^{\leq\ttt}}\to\oo_{M^{\leq\s}}$) plus the internal differential.  This differential squares to zero by Lemma \ref{celllikelocalstructure} (the square of the differential is a sum over $\codim(\s\preceq\ttt)=2$ of maps $C_\bullet(M^{\leq\s})\to C_\bullet(M^{\leq\ttt})$, each of which can be seen to vanish by applying Lemma \ref{celllikelocalstructure}(\ref{intstratumstruct}) to $M^{<\ttt}\to\SSS^{<\ttt}$).  Now there is a natural map:
\begin{align}\label{stratchainstousual}
C_\bullet(M;\SSS)&\to C_\bullet(M)\\
\bigoplus_{\s\in\SSS}\gamma_\s&\mapsto\sum_{\begin{smallmatrix}\s\in\SSS\cr\dim\s=\dim M\end{smallmatrix}}\gamma_\s
\end{align}
(it follows from Lemma \ref{celllikelocalstructure} that this is a chain map) which we claim is a quasi-isomorphism.

To see that \eqref{stratchainstousual} is a quasi-isomorphism, observe that it is the map on global sections of a corresponding map of complexes of $\K$-presheaves $C_\bullet(M,M\setminus K;\SSS)\to C_\bullet(M,M\setminus K)$ on the one-point compactification $M_+$ of $M$.  Both are homotopy $\K$-sheaves by Lemma \ref{singularchainshomotopyKsheaf}.  Thus by Corollary \ref{homotopysheafcheckqionstalks}, it suffices to show that the map $C_\bullet(M,M\setminus p;\SSS)\to C_\bullet(M,M\setminus p)$ is a quasi-isomorphism for every $p\in M$.  This holds by the following local argument.  Note that $H_\bullet(M,M\setminus p;\SSS)$ is isomorphic to $\ZZ$ (to see this, consider the filtration by $\dim\s$), so it suffices to construct a cycle in $C_\bullet(M,M\setminus p;\SSS)$ representing a generator of $\ZZ$ and show that its image in $H_\bullet(M,M\setminus p)=\ZZ$ is a generator.  Such a cycle may be constructed by induction on $\SSS$ (starting from the stratum containing $p$ and going up), and its image in $H_\bullet(M,M\setminus p)$ generates since its degree at $p$ coincides with its degree at any nearby point in an open stratum, which is one by construction.  A similar argument appears in Barraud--Cornea \cite[p670, Lemma 2.2]{barraudcornea}.

Thus we can think of $C_\bullet(M;\SSS)$ as a model for chains on $M$ (the reader should make sure they understand what this means geometrically).
\end{example}

\begin{definition}[Stratified virtual cochain complexes $C^\bullet_\vir(-,\SSS;A)$]\label{stratifiedvirtualdef}
Let $(X,\partial X)\to(\SSS,\partial\SSS)$ be equipped with a finite implicit atlas with boundary and cell-like stratification $A$.  For any compact $K\subseteq X$, we define:
\begin{equation}\label{virtualchainswithcornersdef}
C^\bullet_\vir(K,\SSS;A):=\bigoplus_{\s\in\SSS}C^\bullet_\vir(K\cap X^{\leq\s}\rel\partial;A)\otimes\oo_\s
\end{equation}
(on the right, $A$ refers to the restriction of the atlas to $X^{\leq\s}$) equipped with the differential given by the sum over $\codim(\s\preceq\ttt)=1$ of the pushforwards:
\begin{equation}\label{stratumpushforward}
C^\bullet_\vir(K\cap X^{\leq\s}\rel\partial;A)\to C^{\bullet+1}_\vir(K\cap X^{\leq\ttt}\rel\partial;A)
\end{equation}
tensored with the specified coboundary map $\oo_\s\to\oo_\ttt$ (plus the internal differential).  The fact that this differential squares to zero follows from Lemma \ref{celllikelocalstructure} and the compatibility of the coboundary maps $\oo_\s\to\oo_\ttt$ with the geometric boundary maps $\oo_{(X_I^\reg)^{\leq\ttt}}\to\oo_{(X_I^\reg)^{\leq\s}}$.

There is a natural map:
\begin{align}\label{stratifiedtounstratifiedpushforward}
C^\bullet_\vir(K,\SSS;A)&\xrightarrow\sim C^\bullet(K;A)\\
\label{stratifiedtounformula}\{\gamma_\s\}_{\s\in\SSS}&\mapsto\biggl(\sum_{\begin{smallmatrix}\s\in\partial\SSS\cr\dim\s=d-1\end{smallmatrix}}\gamma_\s,\sum_{\begin{smallmatrix}\s\in\SSS\cr\dim\s=d\end{smallmatrix}}\gamma_\s\biggr)
\end{align}
Note that for $\dim s=d$, the identification $\oo_{(X_I^\reg)^{\leq\s}}\otimes\oo_\s=\oo_{X_I^\reg}$ gives a (locally constant, but possibly not constant) isomorphism $\oo_\s=\ZZ$ over $(X_I^\reg)^{\leq\s}$ (and similarly for $\s\in\partial\SSS$ with $\dim\s=d-1$, see also Remark \ref{bdryorrmk}) which is used implicitly in \eqref{stratifiedtounformula}.  This is a chain map by Lemma \ref{celllikelocalstructure}.
\end{definition}

\begin{proposition}[$C^\bullet_\vir(-,\SSS;A)\to C^\bullet_\vir(-;A)$ is a quasi-isomorphism]\label{stratifiedispurehomotopysheaf}
Let $(X,\partial X)\to(\SSS,\partial\SSS)$ be equipped with a finite locally orientable implicit atlas with boundary and cell-like stratification $A$.  Then \eqref{stratifiedtounstratifiedpushforward} is a quasi-isomorphism, and the following diagram of sheaves on $X$ commutes:
\begin{equation*}
\begin{tikzcd}
H^0_\vir(-,\SSS;A)\ar{d}\ar{r}{\eqref{stratifiedtounstratifiedpushforward}}&H^0_\vir(-;A)\ar{r}{\text{Prop \ref{Cpure}}}&\oo_X\ar{d}\\
H^0_\vir(-\cap X^{\leq\s},\SSS^{\leq\s};A)\otimes\oo_\s\ar{r}{\eqref{stratifiedtounstratifiedpushforward}}&H^0_\vir(-\cap X^{\leq\s};A)\otimes\oo_\s\ar{r}{\text{Prop \ref{Cpure}}}&\oo_{X^{\leq\s}}\otimes\oo_\s\\
H^0_\vir(-\cap X^{\leq\s}\rel\partial;A)\otimes\oo_\s\ar{u}\ar{rr}{\text{Prop \ref{Cpure}}}&&\oo_{X^{\leq\s}\rel\partial}\otimes\oo_\s\ar{u}
\end{tikzcd}
\end{equation*}
\end{proposition}

Note that $C^\bullet_\vir(K\cap X^{\leq\s}\rel\partial;A)\to C^\bullet_\vir(K\cap X^{\leq\s},\SSS^{\leq\s};A)$ is an isomorphism for $K\subseteq X^\s$.

\begin{proof}
Filter $C^\bullet_\vir(-,\SSS;A)$ by $\dim\s$; the associated graded of this filtration is the direct sum of $C^\bullet_\vir(-\cap X^{\leq\s}\rel\partial;A)$.  Each these is a homotopy $\K$-sheaf on $X$ (pushforward from $X^{\leq\s}$ to $X$ preserves homotopy $\K$-sheaves by Definition \ref{pushforwardofHsheavesworks}), and hence $C^\bullet_\vir(-,\SSS;A)$ is also a homotopy $\K$-sheaf (Lemma \ref{associatedgradedhomotopysheaf}).

Now the map \eqref{stratifiedtounstratifiedpushforward} is a map of homotopy $\K$-sheaves, so to check that it is a quasi-isomorphism, it suffices to check it is a quasi-isomorphism on stalks (Corollary \ref{homotopysheafcheckqionstalks}).  For this, we may use the argument from Example \ref{stratifiedchainsremark} (adapted to the case with boundary) along with Lemma \ref{localdescofISO}.  Moreover, this local construction also gives us the desired commutativity of the diagram of sheaves on $X$.
\end{proof}

\subsection{Product implicit atlas}\label{productatlasdef}

\begin{definition}[Product implicit atlas]\label{productimplicitatlas}
Let $X_1$ and $X_2$ be spaces with implicit atlases $A_1$ and $A_2$ respectively.  The \emph{product implicit atlas} $A_1\sqcup A_2$ on $X_1\times X_2$ is defined by setting $(X_1\times X_2)_{I_1\sqcup I_2}:=(X_1)_{I_1}\times(X_2)_{I_2}$ and $(X_1\times Y_2)^\reg_{I_1\sqcup I_2}:=(X_1)_{I_1}^\reg\times(X_2)_{I_2}^\reg$, with the rest of the data extended in the obvious manner.  Of course, this extends naturally to the setting of implicit atlases with boundary and cell-like stratification (given $X_1$ and $X_2$ stratified by $\SSS_1$ and $\SSS_2$ respectively, their product $X_1\times X_2$ is stratified by $\SSS_1\times\SSS_2$, with $\partial(\SSS_1\times\SSS_2):=(\partial\SSS_1\times\SSS_2)\cup(\SSS_1\times\partial\SSS_2)$).
\end{definition}

\begin{definition}\label{productatlasF}
Let $X_1$ and $X_2$ be spaces with finite implicit atlases $A_1$ and $A_2$; equip $X_1\times X_2$ with the product implicit atlas $A_1\sqcup A_2$.  Let us define a canonical map:
\begin{equation}\label{productatlaspushforward}
C^\bullet_\vir(X_1\rel\partial;A_1)\otimes C^\bullet_\vir(X_2\rel\partial;A_2)\to C^\bullet_\vir(X_1\times X_2\rel\partial;A_1\sqcup A_2)
\end{equation}
which is compatible with the maps \eqref{Fpushforward} and \eqref{Fpushforward1} and is associative.

There are isomorphisms $(X_1)_{I_1,J_1,A_1}\times(X_2)_{I_2,J_2,A_2}\to(X_1\times X_2)_{I_1\sqcup I_2,J_1\sqcup J_2,A_1\sqcup A_2}$ which are compatible with the maps \eqref{pushforward}.  This induces maps:
\begin{equation}
C^\bullet_\vir(X_1\rel\partial;A_1)_{I_1,J_1}\otimes C^\bullet_\vir(X_2\rel\partial;A_2)_{I_2,J_2}\to C^\bullet_\vir(X_1\times X_2\rel\partial;A_1\sqcup A_2)_{I_1\sqcup I_2,J_1\sqcup J_2}
\end{equation}
These maps are compatible with the maps \eqref{Epushforward1more} (note that this compatibility uses the commutativity of \eqref{chainproductcommutes}).  It follows using Definition \ref{simplicialproduct} that they induce the desired map \eqref{productatlaspushforward}.

Note also that \eqref{productatlaspushforward} extends to a collection of maps:
\begin{equation}\label{productatlaspushforwardonK}
C^\bullet_\vir(K_1\rel\partial;A_1)\otimes C^\bullet_\vir(K_2\rel\partial;A_2)\to C^\bullet_\vir(K_1\times K_2\rel\partial;A_1\sqcup A_2)
\end{equation}
for compact $K_1\subseteq X_1$ and $K_2\subseteq X_2$, which are compatible with restriction, thus inducing a map of sheaves $p_1^\ast\oo_{X_1}\otimes p_2^\ast\oo_{X_2}\to\oo_{X_1\times X_2}$.  It can be checked that this is the tautological such map by checking locally (i.e.\ for $K_1$ and $K_2$ single points) using Lemma \ref{localdescofISO}.
\end{definition}

\section{Floer-type homology theories}\label{homologygroupssection}

In this section, we define Floer-type homology groups from a collection of ``flow spaces'' (equipped with appropriately compatible implicit atlases) as which arise in a Morse-type setup.  The necessary VFC machinery has already been setup in \S\ref{Fsheafsection} and \S\ref{auxiliarysection}.  The main task in this section is to correctly organize everything together algebraically.

\begin{convention}
In this section, we work over a fixed ground ring $R$, and everything takes place in the category of $R$-modules unless stated otherwise.  We restrict to implicit atlases $A$ for which $\#\Gamma_\alpha$ is invertible in $R$ for all $\alpha\in A$.
\end{convention}

The main object of study is a \emph{flow category diagram} $\X/Z_\bullet$ where $Z_\bullet$ is a semisimplicial set.  Roughly speaking, this consists of a set of generators $\PPP_z$ for every vertex $z\in Z_0$, along with a collection of spaces $\X(\sigma,p,q)$, which are to be thought of as the spaces of flows from $p\in\PPP_{z_0}$ to $q\in\PPP_{z_n}$ over $\sigma\in Z_n$ spanning vertices $z_0,\ldots,z_n$.  Given a flow category diagram $\X/Z_\bullet$, our goal is to construct:
\begin{rlist}
\item For every $\sigma\in Z_0$, a boundary map $R[\PPP_{z_0}]\to R[\PPP_{z_0}]$.
\item For every $\sigma\in Z_1$, a chain map $R[\PPP_{z_0}]\to R[\PPP_{z_1}]$.
\item For every $\sigma\in Z_2$, a chain homotopy between the two maps $R[\PPP_{z_0}]\to R[\PPP_{z_2}]$.
\item For every $\sigma\in Z_3$, \ldots
\end{rlist}
We will refer to such data as a \emph{diagram} $\HH:Z_\bullet\to\Ndg(\Ch_R)$ (see Definition \ref{NdgdefII} for a precise formulation).

Indeed, when the flow spaces $\X(\sigma,p,q)$ are compact oriented manifolds with corners of dimension $\gr(q)-\gr(p)+\dim\sigma-1$ in a compatible manner (that is, $\X/Z_\bullet$ is a \emph{Morse--Smale flow category diagram}), one may obtain such a diagram $\HH$ by counting the $0$-dimensional flow spaces (one may see that the maps satisfy the required identities by considering the boundary of the $1$-dimensional flow spaces).

Our goal is to generalize this construction to the setting where the spaces $\X$ are equipped with appropriately compatible implicit atlases.  In this generalization, the diagram $\HH:Z_\bullet\to\Ndg(\Ch_R)$ is not determined uniquely.  Rather, its construction depends on making a certain set of ``coherent choices'' (of virtual fundamental cycles) for the spaces $\X$.  Hence, the main steps we must take are: formulating precisely what ``coherent choices'' mean, proving such choices always exist, and proving that the diagram is (in a suitable sense) independent of the choices.  Let us now comment briefly on the latter two steps.

Given a flow category diagram $\X/Z_\bullet$ with an implicit atlas, we encode the ``space'' of coherent choices via a map $\pi:\tilde Z_\bullet\to Z_\bullet$.  Namely, $\pi:\tilde Z_\bullet\to Z_\bullet$ is defined by the property that \emph{giving a section $s:Z_\bullet\to\tilde Z_\bullet$ of $\pi$ is the same as making coherent choices over all of $Z_\bullet$}.  Now, the statement that coherent choices give rise to a diagram $\HH:Z_\bullet\to\Ndg(\Ch_R)$ translates into a canonical diagram $\widetilde\HH:\tilde Z_\bullet\to\Ndg(\Ch_R)$ (defined essentially by the property that the set of coherent choices over $Z_\bullet$ corresponding to a section $s:Z_\bullet\to\tilde Z_\bullet$ gives rise to the diagram $\HH:=\widetilde\HH\circ s$).  Thus, we have constructed:
\begin{equation}
\begin{CD}
\tilde Z_\bullet@>\widetilde\HH>>\Ndg(\Ch_R)\cr
@V\pi VV\cr
Z_\bullet
\end{CD}
\end{equation}
Now, the key result we prove is that \emph{$\pi:\tilde Z_\bullet\to Z_\bullet$ is a trivial Kan fibration} (think: ``is a bundle with contractible fibers'').  From this, we obtain (mostly formally) that coherent choices exist and that the resulting diagram is (up to quasi-isomorphism) independent of the choice (both are incarnations of the fact that ``the space of sections of a trivial Kan fibration is contractible'').

\begin{remark}[Restricting to the $2$-skeleton of $Z_\bullet$]\label{restricttotwoskeleton}
If one is satisfied with working in the homotopy category (i.e.\ constructing a diagram $Z_\bullet\to H^0(\Ch_R)$), then one needs only the $2$-skeleton of $Z_\bullet$.  On the other hand, there is little simplification to be gained by using $2$-truncated semisimplicial sets instead of semisimplicial sets.  Moreover, the ``higher homotopies'' which are kept track of in $\Ndg(\Ch_R)$ are known to contain interesting information in certain settings (for example, they can be used to obstruct isotopies between symplectic embeddings, as in Floer--Hofer--Wysocki \cite{floerhoferwysockiSHI}).
\end{remark}

\subsection{Sets of generators, triples \texorpdfstring{$(\sigma,p,q)$}{(sigma,p,q)}, and \texorpdfstring{$\sF$}{F}-modules}

\begin{definition}[Simplicial set and semisimplicial set]
Let $\Delta$ be the category of finite nonempty totally ordered sets with morphisms weakly order-preserving maps.  Let $\Delta_\inj$ be the subcategory of injective morphisms.  A \emph{simplicial set} $Z_\bullet$ is a functor $Z:\Delta^\op\to\Set$, and a \emph{semisimplicial set} $Z_\bullet$ is a functor $Z:\Delta_\inj^\op\to\Set$.  In both cases, we write $Z_n$ for $Z(\{0,\ldots,n\})$.  For $\sigma\in Z_n$ and $0\leq j_0<\cdots<j_m\leq n$, we denote by $\sigma|[j_0\ldots j_m]$ the image of $\sigma$ under the map $Z_n\to Z_m$ induced by the map $\{0,\ldots,m\}\to\{0,\ldots,n\}$ given by $i\mapsto j_i$ ($\sigma|[j_0\ldots j_m]$ is called a facet of $\sigma$); if $m=0$ we also write $\sigma(j)$ for $\sigma|[j]$.
\end{definition}

\begin{definition}[Set of generators]
Let $Z$ be a set.  A \emph{set of generators} $\PPP/Z$ is a collection of sets $\{\PPP_z\}_{z\in Z}$, each equipped with a grading $\gr:\PPP_z\to\ZZ$ and an ``action'' $a:\PPP_z\to\RR$.  Given a set of generators $\PPP/Z$ and a map $f:Y\to Z$, we can form the pullback set of generators $f^\ast\PPP/Y$ defined by $(f^\ast\PPP)_y:=\PPP_{f(y)}$.
\end{definition}

\begin{definition}[Triples $(\sigma,p,q)$]
Let $Z_\bullet$ be a semisimplicial set, and let $\PPP/Z_0$ be a set of generators.  The notation $(\sigma,p,q)$ \emph{always} means a triple where $\sigma\in Z_n$ is an $n$-simplex and $(p,q)\in\PPP_{\sigma(0)}\times\PPP_{\sigma(n)}$, where either $\dim\sigma>0$ or $a(p)<a(q)$.  We say that $(\sigma',p',q')\prec(\sigma,p,q)$ (``strictly precedes'') iff one of the two conditions holds:
\begin{rlist}
\item$\sigma'\subsetneqq\sigma$ (i.e.\ $\sigma'$ is a facet of positive codimension of $\sigma$).
\item$\sigma'=\sigma$ and $a(p)\leq a(p')$ and $a(q')\leq a(q)$ with at least one inequality being strict.
\end{rlist}
It is easy to see that $\prec$ is a partial order.
\end{definition}

\begin{definition}[$\sF$-module]\label{flowobjectdef}
Let $Z_\bullet$ be a semisimplicial set, and let $\PPP/Z_0$ be a set of generators.  Let $\C^\otimes$ be a monoidal category with an initial object $0\in\C$ such that $X\otimes 0=0=0\otimes X$ for all $X\in\C$.

An \emph{$\sF(\PPP/Z_\bullet)$-module} $\W$ (often abbreviated ``$\sF$-module'') in $\C$ is a collection of objects $\W=\{\W(\sigma,p,q)\in\C\}_{(\sigma,p,q)}$ equipped with:
\begin{align}
\label{Wcomposition}\text{\emph{Product maps}:}&&\W(\sigma|[0\ldots k],p,q)\otimes\W(\sigma|[k\ldots n],q,r)&\to\W(\sigma,p,r)\quad\text{for }0\leq k\leq n\\
\label{Wsimplexextension}\text{\emph{Face maps}:}&&\W(\sigma|[0\ldots\hat k\ldots n],p,q)&\to\W(\sigma,p,q)\quad\text{for }0<k<n
\end{align}
which are compatible in a sense we will now describe.  Note that for both \eqref{Wcomposition} and \eqref{Wsimplexextension}, the triples indexing the domain strictly precede the triple indexing the target (because we always restrict to triples $(\sigma,p,q)$ with $\dim\sigma>0$ or $a(p)<a(q)$).

Now given any $\sigma\in Z_\bullet$ spanning vertices $0,\ldots,n$ and a choice of:
\begin{align}
\label{SdataA}0&=j_0<\cdots<j_\ell=n\\
0&=a_0\leq\cdots\leq a_m=\ell\\
\label{SdataB}p_i&\in\PPP_{\sigma(j_{a_i})}\quad\text{for }0\leq i\leq m
\end{align}
we can apply the product map $m-1$ times and the face map $n-\ell$ times in some order to obtain a map:
\begin{equation}\label{Wtotalcomposition}
\W(\sigma|[j_{a_0}\ldots j_{a_1}],p_0,p_1)\otimes\cdots\otimes\W(\sigma|[j_{a_{m-1}}\ldots j_{a_m}],p_{m-1},p_m)\to\W(\sigma,p_0,p_m)
\end{equation}
We say that the product/face maps are \emph{compatible} iff this map is independent of the order in which they are applied (this reduces to three basic commutation identities).

Given an $\sF(\PPP/Z_\bullet)$-module $\W$ and a map $f:Y_\bullet\to Z_\bullet$, we can form the pullback $f^\ast\W$ which is an $\sF(f^\ast\PPP/Y_\bullet)$-module defined by $(f^\ast\W)(\sigma,p,q):=\W(f(\sigma),p,q)$.

The categories $\C$ relevant for this paper are:
\begin{rlist}
\item The category of spaces with the product monoidal structure (``$\sF$-module space'').
\item The category of posets with the product monoidal structure (``$\sF$-module poset'').
\item The category of chain complexes with the tensor product monoidal structure (``$\sF$-module complex'').
\end{rlist}
All are in fact symmetric monoidal (noting that the relevant symmetric monoidal structure on complexes is the super tensor product).
\end{definition}

\begin{example}\label{Fmoduleexample}
Let $Z_\bullet$ be any semisimplicial set, and let us take as set of generators $\PPP:=Z_0$ (i.e.\ a single generator over every vertex of $Z_\bullet$).  We define an $\sF(\PPP/Z_\bullet)$-module space $\W$ by $\W(\sigma,p,q):=\F(\sigma)$, where $\F$ is the space of broken Morse trajectories from Definition \ref{simplexflows}.  The reader may easily verify that this forms an $\sF$-module space, with product/face maps given by \eqref{Fproduct}--\eqref{Fface}.
\end{example}

\begin{definition}[Support of an $\sF$-module]
Let $\W$ be an $\sF$-module.  We define the \emph{support} of $\W$, denoted $\supp\W$, as the smallest collection of triples containing those for which $\W(\sigma,p,q)\ne 0$ that is closed under product/face operations (meaning that if the triples on the left side of \eqref{Wcomposition} or \eqref{Wsimplexextension} are in the set, then so is the triple on the right).  Equivalently, $(\sigma,p,q)\in\supp\W$ iff there is some choice of \eqref{SdataA}--\eqref{SdataB} for which $p_0=p$ and $p_m=q$ and for which every factor on the left hand side of \eqref{Wtotalcomposition} is $\ne 0$.
\end{definition}

\begin{definition}[Strata of an $\sF$-module]
Let $\W$ be an $\sF$-module.  We let $\SSS_\W(\sigma,p,q)$ denote the set of choices of \eqref{SdataA}--\eqref{SdataB} for which $p_0=p$ and $p_m=q$ and for which every factor on the left hand side of \eqref{Wtotalcomposition} is in $\supp\W$.  We equip $\SSS_\W(\sigma,p,q)$ with the partial order induced by formally applying product/face maps.  The reader may easily convince themselves that $\SSS_\W$ is itself an $\sF$-module poset.

There is an order-reversing map $\codim:\SSS_\W(\sigma,p,q)\to\ZZ_{\geq 0}$ defined by $\codim\s:=(m-1)+(n-\ell)$ for $\s\in\SSS_\W(\sigma,p,q)$.  We let $\s^\ttop\in\SSS_\W(\sigma,p,q)$ denote the unique maximal element (the only element $\s$ with $\codim\s=0$; it is given by $\ell=n$ and $m=1$).

For $\s\in\SSS_\W(\sigma,p,q)$, we let $\W(\sigma,p,q,\s)$ denote the left hand side of \eqref{Wtotalcomposition}.  The following \emph{boundary inclusion map}:
\begin{equation}\label{boundaryinclusion}
\colim_{\s\in\partial\SSS_\W(\sigma,p,q)}\W(\sigma,p,q,\s)\to\W(\sigma,p,q)
\end{equation}
(where $\partial\SSS_\W(\sigma,p,q)$ denotes $\SSS_\W(\sigma,p,q)\setminus\s^\ttop$) will play an important role.
\end{definition}

\subsection{Flow category diagrams and their implicit atlases}

\begin{definition}[Flow category diagram]\label{semisimplicialfcdef}
Let $Z_\bullet$ be a semisimplicial set.  A \emph{flow category diagram} $\X/Z_\bullet$ (read ``$\X$ over $Z_\bullet$'') is:
\begin{rlist}
\item A set of generators $\PPP/Z_0$.
\item An $\sF$-module space $\X$ where each $\X(\sigma,p,q)$ is compact Hausdorff and each $\SSS_\X(\sigma,p,q)$ is finite.
\item A stratification of each $\X(\sigma,p,q)$ by $\SSS_\X(\sigma,p,q)$ which is compatible with the product/face maps and so that $\X(\sigma,p,q,\s)\to\X(\sigma,p,q)$ is a homeomorphism onto $\X(\sigma,p,q)^{\leq\s}$.
\rlistsave
\end{rlist}
with the following finiteness properties:
\begin{rlist}
\rlistresume
\item\label{novikovfiniteness}For all $\sigma$, $p$, and $M<\infty$, we have $\#\{q:\X(\sigma,p,q)\ne\varnothing\text{ and }a(q)<a(p)+M\}<\infty$.
\item\label{novikovfinitenessII}For all $\sigma$, we have $\inf\{a(q)-a(p):\X(\sigma,p,q)\ne\varnothing\}>-\infty$.
\rlistsave
\end{rlist}
Let $H$ be a group.  An \emph{$H$-equivariant flow category diagram} is a flow category diagram along with:
\begin{rlist}
\rlistresume
\item A free action of $H$ on $\PPP$.
\item An action of $H$ on $\X$ (meaning compatible maps $h:\X(\sigma,p,q)\to\X(\sigma,hp,hq)$).
\item Homomorphisms $\gr:H\to\ZZ$ and $a:H\to\RR$ such that $\gr(hp)=\gr(h)+\gr(p)$ and $a(hp)=a(h)+a(p)$ for all $h\in H$ and $p\in\PPP$.
\end{rlist}
Given an $H$-equivariant flow category diagram $\X/Z_\bullet$ and a map $f:Y_\bullet\to Z_\bullet$, we can form the pullback $H$-equivariant flow category diagram $f^\ast\X/Y_\bullet$.
\end{definition}

\begin{remark}[Morse--Smale flow category diagram]\label{MSflowcatediag}
A \emph{Morse--Smale} flow category diagram is one in which each $\X(\sigma,p,q)$ is a (compact) topological manifold with corners of dimension $\gr(q)-\gr(p)+\dim\sigma-1$ (the corner structure being induced by the stratification by $\SSS_\X(\sigma,p,q)$).
\end{remark}

\begin{remark}[$\infty$-category $\FlowCat$]
Let $\FlowCat$ be the semisimplicial set which represents the functor $Z_\bullet\mapsto\{$Flow category diagrams over $Z_\bullet\}$.  The reader familiar with $\infty$-categories may wish to think of $\FlowCat$ as an $\infty$-category of flow categories (though only in a vague sense, since we have not given it the structure of a simplicial set, nor have we verified the weak Kan condition).  All of the constructions in this section involving flow category diagrams over a semisimplicial set $Z_\bullet$ are compatible with pullback, and thus can be equivalently thought of as ``universal'' constructions over $\FlowCatIA$ (which represents the functor of flow category diagrams equipped with implicit atlases).
\end{remark}

\begin{definition}[Implicit atlas on flow category diagram]\label{IAonflowcategorydiagramdef}
Let $\X/Z_\bullet$ be a flow category diagram.  An implicit atlas $\A$ on $\X/Z_\bullet$ consists of the following data.  We give index sets $\Abar(\sigma,p,q)$, and we define:\footnote{Warning: a particular set $\Abar(\sigma',p',q')$ may appear many times on the right hand side.}
\begin{equation}\label{coproductforA}
\A(\sigma,p,q)^{\geq\s}:=\coprod_{\begin{smallmatrix}0\leq i_0<\cdots<i_m\leq n\cr(p',q')\in\PPP_{\sigma(i_0)}\times\PPP_{\sigma(i_m)}\cr\exists\ttt\in\SSS_\X(\sigma,p,q)^{\geq\s}\text{ containing }([i_0\ldots i_m],p',q')\end{smallmatrix}}\Abar(\sigma|[i_0\ldots i_m],p',q')
\end{equation}
We explain the notation: recall that $\SSS_\X(\sigma,p,q)$ parameterizes the ``possible left hand sides'' of \eqref{Wtotalcomposition}; the coproduct is over all $([i_0\ldots i_m],p',q')$ which appear as a factor in some $\ttt\in\SSS_\X(\sigma,p,q)$ with $\s\preceq\ttt$.

For all $(\sigma,p,q)$ and $\s\in\SSS_\X(\sigma,p,q)$, we give an implicit atlas with boundary with cell-like stratification $\A(\sigma,p,q)^{\geq\s}$ on $\X(\sigma,p,q)^{\leq\s}$ (stratified by $\SSS_\X(\sigma,p,q)^{\leq\s}$, of virtual dimension $\gr(q)-\gr(p)+\dim\sigma-1-\codim\s$), for which the stratification conforms to the following local model.  Given $\s'\preceq\s$, let $G=G(\s',\s)$ denote the set of possible product/face operations which may be applied to $\s'$ for which the result is still $\preceq\s$.  There is a tautological isomorphism of posets $2^G\to\SSS_\X(\sigma,p,q)^{\s'\leq\cdot\leq\s}$ sending a given set of product/face operations to the result of applying them to $\s'$.  Now the local model for the stratification on (regular thickened moduli spaces of) $\X(\sigma,p,q)^{\leq\s}$ near a point of type $\s'$ is given by $\RR_{\geq 0}^G\times\RR^N$, stratified in the obvious way by $\SSS_\X(\sigma,p,q)^{\s'\leq\cdot\leq\s}$.  Clearly this stratification is cell-like.  Moreover, the normal bundle to the $\s'$ stratum is canonically identified with $\oo_\RR^{\otimes G(\s',\s)}$, and the implicit atlas should use this as the orientation data.

In addition, we give compatible identifications between these atlases as follows:
\begin{rlist}
\item\label{substratumC}Let $\s\preceq\ttt\in\SSS_\X(\sigma,p,q)$.  Then by definition:
\begin{equation*}
\A(\sigma,p,q)^{\geq\ttt}\subseteq\A(\sigma,p,q)^{\geq\s}
\end{equation*}
Both are implicit atlases on $\X(\sigma,p,q)^{\leq\s}$ (the former by restriction to this substratum of $\X(\sigma,p,q)^{\leq\ttt}$), and we identify the former with the subatlas of the latter corresponding to this tautological inclusion of index sets.
\item\label{faceC}Let $\s\in\SSS_\X(\sigma|[0\ldots\hat k\ldots n],p,q)\subseteq\SSS_\X(\sigma,p,q)$.  Then by definition:
\begin{equation*}
\A(\sigma|[0\ldots\hat k\ldots n],p,q)^{\geq\s}\subseteq\A(\sigma,p,q)^{\geq\s}
\end{equation*}
Both are implicit atlases on:
\begin{equation*}
\X(\sigma|[0\ldots\hat k\ldots n],p,q)^{\leq\s}=\X(\sigma,p,q)^{\leq\s}
\end{equation*}
and we identify the former with the subatlas of the latter corresponding to this tautological inclusion of index sets.
\item\label{productC}Let $\s_1\times\s_2\in\SSS_\X(\sigma|[0\ldots k],p,q)\times\SSS_\X(\sigma|[k\ldots n],q,r)\subseteq\SSS_\X(\sigma,p,r)$.  Then by definition:
\begin{equation*}
\A(\sigma|[0\ldots k],p,q)^{\geq\s_1}\sqcup\A(\sigma|[k\ldots n],q,r)^{\geq\s_2}\subseteq\A(\sigma,p,r)^{\geq(\s_1\times\s_2)}
\end{equation*}
Both are implicit atlases on:
\begin{equation*}
\X(\sigma|[0\ldots k],p,q)^{\leq\s_1}\times\X(\sigma|[k\ldots n],q,r)^{\leq\s_2}=\X(\sigma,p,r)^{\leq(\s_1\times\s_2)}
\end{equation*}
and we identify the former with the subatlas of the latter corresponding to this tautological inclusion of index sets.
\end{rlist}
An implicit atlas on an $H$-equivariant flow category $\X/Z_\bullet$ is an implicit atlas along with a lift of the action of $H$ to the implicit atlas structure.

Given an implicit atlas $\A$ on $\X/Z_\bullet$ and a map $f:Y_\bullet\to Z_\bullet$, we can form the pullback implicit atlas $f^\ast\A$ on $f^\ast\X/Y_\bullet$.
\end{definition}

\begin{remark}
The above definition has been formulated to reflect the collection of atlases which is the simplest to construct, yet still sufficient to define Floer-type homology groups.
\end{remark}

\begin{definition}[Coherent orientations]\label{cohordefabstract}
Let $\X/Z_\bullet$ be a flow category diagram with locally orientable implicit atlas $\A$.  A set of \emph{coherent orientations} $\omega$ is a choice of global sections $\omega(\sigma,p,q)\in\cH^0(\X(\sigma,p,q);\oo_{\X(\sigma,p,q)})$ with the following property.  Note that covering each of the product/face maps:
\begin{align}
\X(\sigma|[0\ldots k],p,q)\times\X(\sigma|[k\ldots n],q,r)&\to\partial\X(\sigma,p,r)\\
\X(\sigma|[0\ldots\hat k\ldots n],p,q)&\to\partial\X(\sigma,p,q)
\end{align}
is an isomorphism of orientation sheaves.  We require that $\omega$ transform in the following way under this isomorphism:
\begin{align}
\label{coherentOproduct}\omega(\sigma|[0\ldots k],p,q)\times\omega(\sigma|[k\ldots n],q,r)&=(-1)^{k+\gr(p)}d\omega(\sigma,p,r)\\
\label{coherentOface}-\omega(\sigma|[0\ldots\hat k\ldots n],p,q)&=(-1)^{k+\gr(p)}d\omega(\sigma,p,q)
\end{align}
where $d\omega\in\cH^0(\partial\X;\oo_{\partial\X})$ is the boundary orientation induced by $\omega$.  Each of the following:
\begin{gather}
\label{multiorientationI}\omega(\sigma|[0\ldots k],p,p')\times\omega(\sigma|[k\ldots\ell],p',p'')\times\omega(\sigma|[\ell\ldots n],p'',q)\\
\label{multiorientationII}\omega(\sigma|[0\ldots\hat k\ldots\ell],p,p')\times\omega(\sigma|[\ell\ldots n],p',q)\\
\label{multiorientationIII}\omega(\sigma|[0\ldots k],p,p')\times\omega(\sigma|[k\ldots\hat\ell\ldots n],p',q)\\
\label{multiorientationIV}\omega(\sigma|[0\ldots\hat k\ldots\hat\ell\ldots n],p,q)
\end{gather}
can be expressed in terms of $\omega(\sigma,p,q)$ in two different ways using \eqref{coherentOproduct}--\eqref{coherentOface}.  One can easily check that with the choice of signs in \eqref{coherentOproduct}--\eqref{coherentOface}, these two expressions coincide for each of \eqref{multiorientationI}--\eqref{multiorientationIV} (the Koszul rule of signs applies to calculating the boundary of a product of orientations, which means that this coherence condition involves $\vdim\X(\sigma|[0\ldots k],p,p')$ for both \eqref{multiorientationI} and \eqref{multiorientationIII}).

Coherent orientations on an $H$-equivariant flow category diagram with implicit atlas are coherent orientations which are invariant under the action of $H$.

Given coherent orientations $\omega$ on $\X/Z_\bullet$ and a map $f:Y_\bullet\to Z_\bullet$, we can form the pullback coherent orientations $f^\ast\omega$ on $f^\ast\X/Y_\bullet$.
\end{definition}

\begin{remark}
One can obtain alternative sign conventions in \eqref{coherentOproduct}--\eqref{coherentOface} by ``twisting'' $\omega(\sigma,p,q)$.  For example, multiplying $\omega(\sigma,p,q)$ by $(-1)$ flips the sign of \eqref{coherentOproduct}, multiplying by $(-1)^{\dim\sigma}$ flips the sign of \eqref{coherentOface}, and multiplying by $(-1)^{\gr(p)}$ or $(-1)^{\gr(q)}$ multiplies \eqref{coherentOproduct} by $(-1)^{\gr(q)}$.
\end{remark}

\subsection{Augmented virtual cochain complexes}

We would like to endow $C^\bullet_\vir(\X(\sigma,p,q)\rel\partial)$ with the structure of an $\sF$-module complex.  More precisely, we would like to construct product/face maps:
\begin{align}
\label{Ccompositionnaive}C^\bullet_\vir(\X(\sigma|[0\ldots k],p,q)\rel\partial)\otimes C^\bullet_\vir(\X(\sigma|[k\ldots n],q,r)\rel\partial)&\to C^\bullet_\vir(\partial\X(\sigma,p,r))\\
\label{Csimplexextensionnaive}C^\bullet_\vir(\X(\sigma|[0\ldots\hat k\ldots n],p,q)\rel\partial)&\to C^\bullet_\vir(\partial\X(\sigma,p,q))
\end{align}
(induced by the corresponding product/face maps of the spaces $\X(\sigma,p,q)$).  However, to obtain maps \eqref{Ccompositionnaive}--\eqref{Csimplexextensionnaive} defined on the chain level, we must replace the virtual cochain complexes with certain \emph{augmented virtual cochain complexes} (which are canonically quasi-isomorphic to their ``non-augmented'' counterparts).

In this subsection, we build the augmented virtual cochain complexes (using a homotopy colimit construction) and then we define the $\sF$-module structure on them.  We also define the analogue of the map $s_\ast$ for the augmented virtual cochain complexes.  We remark that the proliferation of homotopy colimits could probably be abated (and, indeed, this entire subsection eliminated) at the expense of using more abstract language (specifically, working in a symmetric monoidal $\infty$-category of complexes).

\begin{definition}[Augmented virtual cochain complexes $C^\bullet_\vir(-;\A)^+$]\label{Gsheavesdef}
Let $\X/Z_\bullet$ be a flow category diagram with implicit atlas $\A$, where every $\Abar(\sigma,p,q)$ is finite and we have fixed fundamental cycles $[E_\alpha]\in C_\bullet(E;\alpha)$ for all $\alpha\in\Abar(\sigma,p,q)$.  We define the following complexes:
\begin{align}
\label{GmasterI}C^\bullet_\vir(\X\rel\partial;\A)^+(\sigma,p,q)&:=\hocolim_{\mathfrak s\preceq\mathfrak t\in\SSS_\X(\sigma,p,q)}\hphantom{\Biggl\{}C^{\bullet-\codim\mathfrak s}_\vir(\X(\sigma,p,q)^{\leq\s}\rel\partial;\A(\sigma,p,q)^{\geq\ttt})\\
\label{GmasterII}C^{\bullet-1}_\vir(\partial\X;\A)^+(\sigma,p,q)&:=\hocolim_{\mathfrak s\preceq\mathfrak t\in\SSS_\X(\sigma,p,q)}\begin{cases}\hfill C^{\bullet-1}_\vir(\partial\X(\sigma,p,q);\A(\sigma,p,q)^{\geq\ttt})&\mathfrak s=\mathfrak t=\mathfrak s^\ttop\cr C^{\bullet-\codim\mathfrak s}_\vir(\X(\sigma,p,q)^{\leq\s}\rel\partial;\A(\sigma,p,q)^{\geq\ttt})&\text{otherwise}\end{cases}
\end{align}
The structure maps of the homotopy diagrams come from the obvious pushforward maps (for increasing $\s$) and the maps \eqref{Fpushforward1} using the fixed fundamental cycles $[E_\alpha]$ (for decreasing $\ttt$).  These are compatible because of the commutativity of \eqref{chainproductcommutes}.

Next, let us observe that we have a natural commutative diagram:
\begin{equation}\label{topiscofinal}
\begin{tikzcd}
C^{\bullet-1}_\vir(\partial\X(\sigma,p,q);\A(\sigma,p,q)^{\geq\s^\ttop})\ar{d}\ar[hook]{r}{\sim}&C^{\bullet-1}_\vir(\partial\X;\A)^+(\sigma,p,q)\ar{d}\\
C^\bullet_\vir(\X(\sigma,p,q)\rel\partial;\A(\sigma,p,q)^{\geq\s^\ttop})\ar[hook]{r}{\sim}&C^\bullet_\vir(\X\rel\partial;\A)^+(\sigma,p,q)
\end{tikzcd}
\end{equation}
The horizontal maps (inclusions of the $\mathfrak s=\mathfrak t=\mathfrak s^\ttop$ subcomplexes) are quasi-isomorphisms by Lemma \ref{comboretraction} (which morally says that $\mathfrak s^\ttop\in\SSS_\X(\sigma,p,q)$ acts as a final object in the homotopy colimits \eqref{GmasterI}--\eqref{GmasterII}) which applies because each of the structure maps from $(\mathfrak s,\mathfrak t)$ to $(\mathfrak s,\mathfrak t')$ is a quasi-isomorphism.

Note that by definition, the support of $C^\bullet_\vir(\X\rel\partial;\A)^+$ and of $C^\bullet_\vir(\partial\X;\A)^+$ are contained in the support of $\X$.
\end{definition}

\begin{definition}[Product/face maps for $C^\bullet_\vir(-;\A)^+$]\label{Cproductface}
Let $\X/Z_\bullet$ be a flow category diagram with implicit atlas $\A$, where every $\Abar(\sigma,p,q)$ is finite and we have fixed fundamental cycles $[E_\alpha]$.  Let us now define face and product maps for $C^\bullet_\vir(\X\rel\partial;\A)^+$.  These will be of degree $1$ and will be equipped with a canonical factorization through $C^{\bullet-1}_\vir(\partial\X;\A)^+\to C^\bullet_\vir(\X\rel\partial;\A)^+$.  In other words, we really are going to construct maps:
\begin{align}
\label{Ccomposition}C^\bullet_\vir(\X\rel\partial;\A)^+(\sigma|[0\ldots k],p,q)\otimes C^\bullet_\vir(\X\rel\partial;\A)^+(\sigma|[k\ldots n],q,r)&\to C^\bullet_\vir(\partial\X;\A)^+(\sigma,p,r)\\
\label{Csimplexextension}C^\bullet_\vir(\X\rel\partial;\A)^+(\sigma|[0\ldots\hat k\ldots n],p,q)&\to C^\bullet_\vir(\partial\X;\A)^+(\sigma,p,q)
\end{align}

We construct \eqref{Csimplexextension}.  The corresponding face map for $\SSS_\X$ is covered by a corresponding morphism of the homotopy diagrams \eqref{GmasterI}--\eqref{GmasterII} (namely \eqref{Fpushforward1} using the fixed fundamental cycle $[E_{\A(\sigma,p,q)^{\geq\ttt}\setminus\A(\sigma|[0,\ldots,\hat k,\ldots,n],p,q)^{\geq\ttt}}]$).  This gives rise to a corresponding map \eqref{Csimplexextension} on homotopy colimits.

We construct \eqref{Ccomposition}.  We construct a morphism of homotopy diagrams \eqref{GmasterI}--\eqref{GmasterII} covering the corresponding product map for $\SSS_\X$.  Using the product operation on homotopy diagrams (Definition \ref{simplicialproduct}), it suffices to construct compatible maps:
\begin{multline}
C^\bullet_\vir(\X(\sigma|[0,\ldots,k],p,q)^{\leq\s_1}\rel\partial;\A(\sigma|[0,\ldots,k],p,q)^{\geq\ttt_1})\otimes\\C^\bullet_\vir(\X(\sigma|[k,\ldots,n],q,r)^{\leq\s_2}\rel\partial;\A(\sigma|[k,\ldots,n],q,r)^{\geq\ttt_2})\\\to C^\bullet_\vir(\X(\sigma,p,q)^{\leq\s_1\times\s_2}\rel\partial;\A(\sigma,p,q)^{\geq\ttt_1\times\ttt_2})
\end{multline}
for $\s_1\preceq\ttt_1\in\SSS(\sigma|[0,\ldots,k],p,q)$ and $\s_2\preceq\ttt_2\in\SSS(\sigma|[k,\ldots,n],q,r)$.  By definition, the subatlas $\A(\sigma|[0,\ldots,k],p,q)^{\geq\ttt_1}\sqcup\A(\sigma|[k,\ldots,n],q,r)^{\geq\ttt_2}\subseteq\A(\sigma,p,q)^{\geq\ttt_1\times\ttt_2}$ is the product implicit atlas on $\X(\sigma,p,q)^{\leq\s_1\times\s_2}=\X(\sigma|[0,\ldots,k],p,q)^{\leq\s_1}\times\X(\sigma|[k,\ldots,n],q,r)^{\leq\s_2}$.  Thus the desired map is constructed in Definition \ref{productatlasF}.
\end{definition}

\begin{definition}[Complexes $C_\bullet(E;\A)^+$]\label{CEcomplexaugmented}
Let $\X/Z_\bullet$ be a flow category diagram with implicit atlas $\A$, where every $\Abar(\sigma,p,q)$ is finite and we have fixed fundamental cycles $[E_\alpha]$.  We define:
\begin{equation}\label{Cmaster}
C_\bullet(E;\A)^+(\sigma,p,q):=\hocolim_{\mathfrak s\preceq\mathfrak t\in\SSS_\X(\sigma,p,q)}C_\bullet(E;\A^{\geq\ttt}(\sigma,p,q))
\end{equation}
where the maps in the homotopy diagram are $\times[E_{A(\sigma,p,q,\mathfrak t')\setminus A(\sigma,p,q,\mathfrak t)}]$.  We equip $C_\bullet(E;\A)^+$ with product/face maps just as in Definition \ref{Cproductface}.  Now there are natural maps:
\begin{equation}\label{GtoCpushforward}
C^{\vdim\X(\sigma,p,q)+\bullet}_\vir(\X\rel\partial;\A)^+(\sigma,p,q)\xrightarrow{s_\ast}C_{-\bullet}(E;\A)^+(\sigma,p,q)
\end{equation}
(induced by \eqref{Fpushforward}), which are maps of $\sF$-modules (that is, they respect the product/face maps).

Note that by definition, the support of $C_\bullet(E;\A)^+$ is contained in the support of $\X$.
\end{definition}

\subsection{Cofibrant \texorpdfstring{$\sF$}{F}-module complexes}

We introduce the notion of an $\sF$-module complex being \emph{cofibrant}, and we introduce a cofibrant replacement functor $Q$ for $\sF$-module complexes.  The machinery we set up is used only for technical reasons in Definition \ref{Lambdadef} so that the proof of Proposition \ref{trivialKanfibration} works correctly.

Recall that an injection of modules with projective cokernel automatically splits.

\begin{definition}[Cofibrations of complexes]
We say a complex is \emph{cofibrant} iff it is projective (as a module).  We say a map of complexes is a \emph{cofibration} iff it is injective and its cokernel is cofibrant; we use the arrow $\rightarrowtail$ to indicate that a map is a cofibration.  It is easy to check that a composition of cofibrations is again a cofibration.
\end{definition}

\begin{definition}[Cofibrant $\sF$-module complex]\label{Fmodulecofibrantdef}
Let $\W_\bullet$ be an $\sF$-module complex.  We say that $\W_\bullet$ is \emph{cofibrant} iff for all $(\sigma,p,q)$:
\begin{rlist}
\item$\W_\bullet(\sigma,p,q)$ is cofibrant.
\item The map:
\begin{equation}\label{boundaryinjective}
\colim_{\s\in\partial\SSS_\W(\sigma,p,q)}\W_\bullet(\sigma,p,q,\s)\rightarrowtail\W_\bullet(\sigma,p,q)
\end{equation}
is a cofibration.
\end{rlist}
\end{definition}

\begin{lemma}\label{cofibrantinjectivelemma}
Let $\W_\bullet$ be a cofibrant $\sF$-module complex.  Then the map:
\begin{equation}\label{generalizedboundaryinjective}
\colim_{\begin{smallmatrix}\s\in\SSS_\W(\sigma,p,q)\cr\s\prec\s_0\end{smallmatrix}}\W(\sigma,p,q,\s)\rightarrowtail\W(\sigma,p,q,\s_0)
\end{equation}
a cofibration for all $\s_0\in\SSS(\sigma,p,q)$.
\end{lemma}

\begin{proof}
Write $\s_0$ in the form \eqref{SdataA}--\eqref{SdataB}.  Then we have:
\begin{equation}
\SSS_\W(\sigma,p,q)^{\leq\s_0}=\SSS_\W(\sigma|[j_{a_0}\ldots j_{a_1}],p_0,p_1)\times\cdots\times\SSS_\W(\sigma|[j_{a_{m-1}}\ldots j_{a_m}],p_{m-1},p_m)
\end{equation}
Now we consider the $m$-cubical diagram:
\begin{equation}
\bigotimes_{i=1}^m\left[\colim_{\s\in\partial\SSS_\W(\sigma|[j_{a_{i-1}}\ldots j_{a_i}],p_{i-1},p_i)}\W_\bullet(\sigma|[j_{a_{i-1}}\ldots j_{a_i}],p_{i-1},p_i,\s)\rightarrowtail\W_\bullet(\sigma|[j_{a_{i-1}}\ldots j_{a_i}],p_{i-1},p_i)\right]
\end{equation}
Now \eqref{generalizedboundaryinjective} is simply the map to the maximal vertex of the $m$-cube from the colimit over the $m$-cube minus the maximal vertex.  This is a cofibration for any cubical diagram of the form $\bigotimes_{i=1}^m[A_i\rightarrowtail B_i]$ where each $A_i\rightarrowtail B_i$ is a cofibration.  To see this, write $B_i=A_i\oplus P_i$ where $P_i$ is projective, and then the map is obviously injective with cokernel $\bigotimes_{i=1}^mP_i$.  A tensor product of projective modules is projective, as can be seen using either the tensor-hom adjunction or the fact that a module is projective iff it is a direct summand of a free module.
\end{proof}

\begin{lemma}\label{cofibranthaslotsofcofibrations}
Let $\W_\bullet$ be a cofibrant $\sF$-module complex.  Let $\SSS\subseteq\T\subseteq\SSS_\W(\sigma,p,q)$ be finite downward closed subsets.  Then the map:
\begin{equation}
\colim_{\s\in\SSS}\W_\bullet(\sigma,p,q,\s)\rightarrowtail\colim_{\s\in\T}\W_\bullet(\sigma,p,q,\s)
\end{equation}
is a cofibration.  In particular, $\colim_{\s\in\T}\W_\bullet(\sigma,p,q,\s)$ is projective (take $\SSS=\varnothing$ above).
\end{lemma}

\begin{proof}
Let us abbreviate $A_\s:=\W_\bullet(\sigma,p,q,\s)$.

We proceed by induction on the cardinality of $\T$, the case $\T=\varnothing$ being clear.  Using the fact that a composition of cofibrations is again a cofibration, it suffices to consider the case $(\SSS,\T)=(\SSS\setminus\s_0,\SSS)$ where $\s_0\in\SSS$ is a maximal element.  Now $\colim_{\s\in\SSS}A_\s$ is the colimit of the following diagram:
\begin{equation}\label{colimitdiagrambuildup}
\begin{tikzcd}[row sep = tiny]
&A_{\s_0}\\
\smash{\displaystyle\colim_{\begin{smallmatrix}\s\in\SSS\cr\s\prec\s_0\end{smallmatrix}}}A_\s\ar[tail]{ur}\ar[tail]{rd}\\
&\displaystyle\colim_{\s\in\SSS\setminus\s_0}A_\s
\end{tikzcd}
\end{equation}
where the top arrow is a cofibration by Lemma \ref{cofibrantinjectivelemma} and the bottom arrow is a cofibration by the induction hypothesis.  It follows that $\colim_{\s\in\SSS\setminus\s_0}A_\s\to\colim_{\s\in\SSS}A_\s$ is injective, with cokernel isomorphic to the cokernel of the top map above.  Hence it is a cofibration as needed.
\end{proof}

\begin{definition}[Cofibrant replacement functor $Q$]\label{cofibrantreplacements}
Let $\W_\bullet$ be an $\sF$-module complex.  Suppose that $\W_\bullet$ satisfies the following properties:
\begin{rlist}
\item Each $\W_\bullet(\sigma,p,q)$ is cofibrant.
\item Each $\SSS_\W(\sigma,p,q)$ is finite.
\end{rlist}
In this case, we define (functorially) an $\sF$-module complex $Q\W_\bullet$ (called the \emph{cofibrant replacement}) with the following properties:
\begin{rlist}
\item$Q\W_\bullet$ is cofibrant.
\item$\supp Q\W_\bullet=\supp\W_\bullet$.
\item There is a (functorial) surjective quasi-isomorphism $Q\W_\bullet\overset\sim\twoheadrightarrow\W_\bullet$ (compatible with product/face maps).
\end{rlist}
We define $Q\W_\bullet(\sigma,p,q)$ and the product/face maps with target $(\sigma,p,q)$ by induction on the set of triples $(\sigma,p,q)$, equipped with the partial order $\preceq_\W$ in which $(\sigma',p',q')\preceq_\W(\sigma,p,q)$ iff $(\sigma',p',q')$ appears as a ``factor'' of some element of $\SSS_\W(\sigma,p,q)$.  This partial order is \emph{well-founded} (i.e.\ there is no strictly decreasing infinite sequence) since each $\SSS_\W(\sigma,p,q)$ is finite, and thus it is valid for induction.  The inductive step for $(\sigma,p,q)$ works as follows.  By the induction hypothesis, we have defined $Q\W_\bullet(\sigma,p,q,\s)$ for all $\s\in\partial\SSS_\W(\sigma,p,q)$.  Hence it suffices to construct (functorially) $Q\W_\bullet(\sigma,p,q)$ fitting into the following commutative diagram:
\begin{equation}\label{cofibrantdefdiagram}
\begin{tikzcd}[column sep = huge]
\smash{\displaystyle\colim_{\s\in\partial\SSS_\W(\sigma,p,q)}}Q\W_\bullet(\sigma,p,q,\mathfrak s)\ar{d}\ar[dashed, tail]{r}{\text{product/face}}&?\exists\,Q\W_\bullet(\sigma,p,q)\ar[dashed, two heads]{d}{\sim}\\
\displaystyle\colim_{\s\in\partial\SSS_\W(\sigma,p,q)}\W_\bullet(\sigma,p,q,\mathfrak s)\ar{r}{\text{product/face}}&\W_\bullet(\sigma,p,q)
\end{tikzcd}
\end{equation}
Now we define $Q\W_\bullet(\sigma,p,q)$ to be the mapping cylinder of the diagonal composition.  Since the domain and codomain of this map are both cofibrant (for the domain, use Lemma \ref{cofibranthaslotsofcofibrations}, which applies since $\SSS_\W(\sigma,p,q)$ has been assumed to be finite), it follows that the top map is a cofibration and that $Q\W_\bullet(\sigma,p,q)$ is cofibrant.  The product/face maps with target $Q\W_\bullet(\sigma,p,q)$ are defined using the top horizontal map; they are compatible since by construction they factor through the colimit in the upper left corner.

Certainly $\supp Q\W_\bullet\supseteq\supp\W_\bullet$.  Conversely, suppose $Q\W_\bullet(\sigma,p,q)$ is nonzero.  Then either $\W_\bullet(\sigma,p,q)\ne 0$ (so $(\sigma,p,q)\in\supp\W_\bullet$), or $Q\W_\bullet(\sigma,p,q,\s)\ne 0$ for some $\s\in\partial\SSS(\sigma,p,q)$ (so by induction, the triples comprising $\s$ are in $\supp\W_\bullet$, and hence so is $(\sigma,p,q)$).
\end{definition}

\begin{lemma}\label{tensorproductvanishing}
Let $\{C_\bullet^k\}_{k=1}^n$ be cofibrant complexes over $\ZZ$ such that $H_iC_\bullet^k=0$ for $i<0$.  Then $H_i(\bigotimes_{k=1}^nC_\bullet^k)=0$ for $i<0$.
\end{lemma}

\begin{proof}
Since the tensor product of projective modules is projective, we may use induction to reduce to the case $k=2$.

Thus, we have two cofibrant complexes $A_\bullet$ and $B_\bullet$ with $H_iA_\bullet=H_iB_\bullet=0$ for $i<0$ and we would like to conclude that $H_i(A_\bullet\otimes B_\bullet)=0$ for $i<0$.  This follows from the K\"unneth theorem, specifically in the form of \cite[p301, Theorem 9.16]{osborne} or \cite[p679, Theorem 10.81]{rotman} (both of which apply because $\Tor^i_\ZZ(\cdot,\cdot)=0$ for $i>1$).\footnote{If we were assuming $A_\bullet$ and $B_\bullet$ to be bounded below, then we could apply the K\"unneth spectral sequence $\bigoplus_{i+j=q}\Tor_p(H_iA_\bullet,H_jB_\bullet)\Rightarrow H_{p+q}(A_\bullet\otimes B_\bullet)$ (see \cite[p686, Theorem 10.90]{rotman}) to reach the desired conclusion without any assumptions on the ground ring.}
\end{proof}

\begin{lemma}\label{Hzeroofcofibrantcolimit}
Let $\W_\bullet$ be a cofibrant $\sF$-module complex over $\ZZ$ such that $H_i\W_\bullet=0$ for $i<0$.  Then for any finite downward closed subset $\SSS\subseteq\SSS_\W(\sigma,p,q)$, the map:
\begin{equation}\label{tocolimitsurjective}
H_i\bigoplus_{\mathfrak s\in\SSS}\W_\bullet(\sigma,p,q,\mathfrak s)\twoheadrightarrow H_i\colim_{\mathfrak s\in\SSS}\W_\bullet(\sigma,p,q,\mathfrak s)
\end{equation}
is surjective for $i\leq 0$.  Moreover, both sides vanish for $i<0$.
\end{lemma}

\begin{proof}
Let us abbreviate $A_\s:=\W_\bullet(\sigma,p,q,\s)$.  Since $A_\s$ is a tensor product of various $\W_\bullet(\sigma',p',q')$, it follows (using Lemma \ref{tensorproductvanishing}) that $H_iA_\s$ for $i<0$, so the left hand side of \eqref{tocolimitsurjective} vanishes for $i<0$.  Hence it remains just to show that \eqref{tocolimitsurjective} is surjective for $i\leq 0$.

We proceed by induction on the cardinality of $\SSS$, the case $\SSS=\varnothing$ being clear.  Let $\s_0\in\SSS$ be any maximal element.  Now $\colim_{\s\in\SSS}A_\s$ is the colimit of \eqref{colimitdiagrambuildup}; it follows that we have the following exact sequence:
\begin{equation}\label{colimitSEScofibrant}
0\to\colim_{\begin{smallmatrix}\s\in\SSS\cr\mathfrak s\prec\s_0\end{smallmatrix}}A_\mathfrak s\to A_{\s_0}\oplus\colim_{\s\in\SSS\setminus\s_0}A_\s\to\colim_{\s\in\SSS}A_\s\to 0
\end{equation}
By the induction hypothesis, it suffices to show that the second map above is surjective on $H_i$ for $i\leq 0$.  By the long exact sequence induced by \eqref{colimitSEScofibrant}, it suffices to show that $H_i\colim_{\begin{smallmatrix}\s\in\SSS\cr\mathfrak s\prec\s_0\end{smallmatrix}}A_\mathfrak s=0$ for $i<0$.  Now this follows from the induction hypothesis.
\end{proof}

\begin{lemma}[Lifting cycles along a fibration]\label{cycleliftinglemma}
Let $\tilde A^\bullet\twoheadrightarrow A^\bullet$ be surjective.  Fix a cycle $a\in A^0$ and a homology class $\tilde{\mathfrak a}\in H^0\tilde A^\bullet$, whose images in $H^0A^\bullet$ coincide.  Then there exists a cycle $\tilde a\in\tilde A^0$ which maps to $a$ and which represents $\tilde{\mathfrak a}$.
\end{lemma}

\begin{proof}
Pick any cycle $\tilde a'\in\tilde A^0$ representing $\tilde{\mathfrak a}$.  Then $\tilde a'-a\in A^0$ is a boundary $db$, and lifting $b\in A^{-1}$ to $\tilde b\in\tilde A^{-1}$, we let $\tilde a:=\tilde a'-d\tilde b$.
\end{proof}

\begin{lemma}[Representing homology classes in mapping cones]\label{mappingconechoiceboundary}
Fix $f:A^\bullet\to B^\bullet$ and $\omega\in H^0[A^\bullet\to B^{\bullet-1}]$.  Let $a\in A^0$ be a cycle representing $\delta\omega\in H^0(A^\bullet)$.  Then there exists $b\in B^{-1}$ with $db=f(a)$ such that $a\oplus b$ represents $\omega$.
\end{lemma}

\begin{proof}
This is a special case of Lemma \ref{cycleliftinglemma} for the surjection $[A^\bullet\to B^{\bullet-1}]\twoheadrightarrow A^\bullet$.
\end{proof}

\begin{lemma}[Universal coefficient theorem]\label{universalcoefficient}
Let $A_\bullet$ be a cofibrant complex over $\ZZ$.  Then there is a natural short exact sequence:
\begin{equation}
0\to\Ext^1(H_{i-1}A_\bullet,\ZZ)\to H^i\Hom(A_\bullet,\ZZ)\to\Hom(H_iA_\bullet,\ZZ)\to 0
\end{equation}
\end{lemma}

\begin{proof}
Well-known.
\end{proof}

\begin{lemma}[Extending cocycles along a cofibration]\label{cocycleextendinglemma}
Suppose we have the following commuting diagrams (solid arrows) of complexes over $\ZZ$ and their homology (where $\ZZ$ is concentrated in degree zero):
\begin{equation}
\begin{tikzcd}
A_\bullet\ar[tail]{r}\ar{d}&\tilde A_\bullet\ar[dashed]{dl}\\
\ZZ
\end{tikzcd}
\overset{H_\bullet}\Longrightarrow
\begin{tikzcd}
H_\bullet A_\bullet\ar{r}\ar{d}&H_\bullet\tilde A_\bullet\ar{dl}\\
\ZZ
\end{tikzcd}
\end{equation}
where $A_\bullet$ and $\tilde A_\bullet$ are cofibrant and $H_{-1}A_\bullet=0$.  Then there exists a dashed arrow compatible with the rest of the diagram.
\end{lemma}

\begin{proof}
By Lemma \ref{universalcoefficient}, we have:
\begin{align}
\label{notildeiso}H^0\Hom(A_\bullet,\ZZ)&\xrightarrow\sim\Hom(H_0A_\bullet,\ZZ)\\
\label{tildesurj}H^0\Hom(\tilde A_\bullet,\ZZ)&\twoheadrightarrow\Hom(H_0\tilde A_\bullet,\ZZ)
\end{align}
Let $f:A_0\to\ZZ$ denote the vertical map in the first diagram.  Using the surjectivity of \eqref{tildesurj}, there exists a cocycle $\tilde f:\tilde A_0\to\ZZ$ giving the desired map on homology.  Denote by $i:A_\bullet\rightarrowtail\tilde A_\bullet$ the inclusion.  Now by the commutativity of the second diagram, the difference $f-\tilde fi$ acts as zero on $H_0A_\bullet$.  Since \eqref{notildeiso} is an isomorphism, this difference is thus a coboundary $\delta g$ for some $g:A_{-1}\to\ZZ$.  Now extend $g$ to $\tilde g:\tilde A_{-1}\to\ZZ$ (using the fact that $A_{-1}\rightarrowtail\tilde A_{-1}$ splits) and let the dotted arrow be $\tilde f+\delta\tilde g$.
\end{proof}

\subsection{Resolution \texorpdfstring{$\tilde Z_\bullet\to Z_\bullet$}{tilde Z to Z}}

We now introduce the space $\pi:\tilde Z_\bullet\to Z_\bullet$ (depending on a flow category diagram $\X/Z_\bullet$ equipped with an implicit atlas and coherent orientations) which one may think of as parameterizing coherent choices of virtual fundamental cycles over all of the flow spaces $\X(\sigma,p,q)$.

The main result is that $\pi:\tilde Z_\bullet\to Z_\bullet$ is a trivial Kan fibration.

\begin{definition}[System of chains]\label{systemofcycles}
Let $\W^\bullet$ be an $\sF$-module complex with each $\SSS_\W(\sigma,p,q)$ finite.  A \emph{system of chains} $\lambda\in\W^\bullet$ (of degree $d(\sigma,p,q)$) is a collection of elements $\lambda_{\sigma,p,q}\in\W^{d(\sigma,p,q)}(\sigma,p,q)$ satisfying $d\lambda_{\sigma,p,q}=\mu_{\sigma,p,q}$ where:
\begin{equation}\label{lambdamaster}
\mu_{\sigma,p,r}:=\sum_{k=0}^n\sum_{q\in\PPP_{\sigma(k)}}(-1)^{k+\gr(p)}\lambda_{\sigma|[0\ldots k],p,q}\cdot\lambda_{\sigma|[k\ldots n],q,r}-\sum_{k=1}^{n-1}(-1)^{k+\gr(p)}\lambda_{\sigma|[0\ldots\hat k\ldots n],p,r}\in\W^\bullet(\sigma,p,r)
\end{equation}
(using the product/face maps on the right hand side).  Note that this sum is finite since $\SSS_\W(\sigma,p,q)$ is finite.  Also, note that the triples on the right hand side all strictly precede the triple on the left hand side.  We also require that the parity (in the sense of Convention \ref{signconventions}) of $\lambda_{\sigma,p,q}$ equals $\gr(q)-\gr(p)+\dim\sigma-1\in\ZZ/2$.  This ensures (via the Koszul rule of signs) that expanding $d\mu_{\sigma,p,q}$ using the identity $d\lambda_{\sigma,p,q}=\mu_{\sigma,p,q}$ yields zero (the signs work out correctly as a consequence of the discussion surrounding \eqref{multiorientationI}--\eqref{multiorientationIV}).  In practice the degrees $d(\sigma,p,q)$ are chosen so that $\mu_{\sigma,p,q}$ (as defined by \eqref{lambdamaster}) is formally homogeneous of the same degree as $d\lambda_{\sigma,p,q}$.
\end{definition}

\begin{remark}
There is a natural bijection between systems of chains $\lambda\in\W^\bullet$ and maps of $\sF$-modules $R[\SSS_\W]\to\W^\bullet$, where $R[\SSS_\W](\sigma,p,q)$ is the free $R$-module on $\SSS_\W(\sigma,p,q)$ with differential $d\s:=\sum_{\codim(\s'\preceq\s)=1}\s'$ (with appropriate signs) and equipped with the obvious product/face maps coming from the $\sF$-module structure on $\SSS_\W$ (more intrinsically, $R[\SSS_\W](\sigma,p,q)$ is the direct sum over $\s\in\SSS_\W(\sigma,p,q)$ of the orientation lines from Definition \ref{IAonflowcategorydiagramdef}).  This perspective is relevant for the key step of the proof of Proposition \ref{trivialKanfibration} below.
\end{remark}

\begin{definition}[Resolution $\pi:\tilde Z_\bullet\to Z_\bullet$]\label{Lambdadef}
Let $\X/Z_\bullet$ be an $H$-equivariant flow category diagram with implicit atlas $\A$ and coherent orientations $\omega$.  We construct a \emph{resolution} $\pi:\tilde Z_\bullet\to Z_\bullet$ (which depends on $\X$, $\A$, $\omega$) as follows.  A simplex $\Delta^n\to\tilde Z_\bullet$ consists of the following data:
\begin{rlist}
\item A map $f:\Delta^n\to Z_\bullet$ (where $\Delta^n$ is the semisimplicial $n$-simplex).
\item\label{subatlaschoice} An $H$-invariant subatlas $\B\subseteq f^\ast\A$ (meaning choices of $\Bbar(\sigma,p,q)\subseteq\Abar(f(\sigma),p,q)$) where each $\Bbar(\sigma,p,q)$ is finite.
\item An $H$-invariant collection of fundamental cycles $[E_\alpha]\in C_\bullet(E;\alpha)$ for all $\alpha\in\Bbar$.
\item\label{lambdachoice}An $H$-invariant system of chains $\lambda\in C^\bullet_\vir(f^\ast\X\rel\partial;\B)^+$ (degree $0$ and supported inside $\supp\X$) with the following property.  Note that $(\mu,\lambda)$ is a cycle in the mapping cone $[C^\bullet_\vir(\partial f^\ast\X;\B)^+\to C^\bullet_\vir(f^\ast\X\rel\partial;\B)^+]$, whose homology is identified with $\cH^\bullet(f^\ast\X;\oo_{f^\ast\X})$.  We require that the homology class of $(\mu,\lambda)$ equal $f^\ast\omega\in\cH^0(f^\ast\X;\oo_{f^\ast\X})$.
\item\label{liftschoice}An $H$-invariant system of chains $\tilde\lambda$ for $QC_\bullet(E;\B)^+$ (degree $\gr(q)-\gr(p)+\dim\sigma-1$ and supported inside $\supp\X$) whose image in $C_\bullet(E;\B)^+$ coincides with $s_\ast\lambda$ (recall \eqref{GtoCpushforward}).
\item\label{augmentationschoice}An $H$-invariant map of $\sF$-modules $[[E]]:QC_\bullet(E;\B)^+_\ZZ\to\ZZ$ which sends the fundamental class to $1$.  The left hand side $QC_\bullet(E;\B)^+_\ZZ$ is defined over $\ZZ$, and the right hand side is the $\sF$-module which for every $(\sigma,p,q)$ is $\ZZ$ concentrated in degree zero (with the product/face maps being multiplication/identity respectively).
\end{rlist}
\end{definition}

\begin{remark}[Resolution commutes with pullback]\label{Lambdaandpullback}
Let $\X/Z_\bullet$ be an $H$-equivariant flow category diagram with implicit atlas $\A$ and coherent orientations $\omega$.  Let $f:Y_\bullet\to Z_\bullet$ be a map, and consider $f^\ast\X/Y_\bullet$ with implicit atlas $f^\ast\A$ and coherent orientations $f^\ast\omega$.  Then there is a canonical fiber diagram relating the resolutions:
\begin{equation}
\begin{CD}
\tilde Y_\bullet@>>>\tilde Z_\bullet\cr
@VVV@VVV\cr
Y_\bullet@>>>Z_\bullet
\end{CD}
\end{equation}
\end{remark}

\begin{proposition}[$\pi:\tilde Z_\bullet\to Z_\bullet$ is a trivial Kan fibration]\label{trivialKanfibration}
The map $\pi:\tilde Z_\bullet\to Z_\bullet$ has the right lifting property with respect to the boundary inclusions $\partial\Delta^n\hookrightarrow\Delta^n$ for all $n\geq 0$ (where $\Delta^n$ is the semisimplicial $n$-simplex).  In other words, given any commuting diagram of solid arrows below:
\begin{equation}
\begin{tikzcd}
\partial\Delta^n\arrow{r}\arrow{d}&\tilde Z_\bullet\arrow{d}\\
\Delta^n\arrow{r}\arrow[dashrightarrow]{ur}&Z_\bullet
\end{tikzcd}
\end{equation}
there exists a dashed arrow making the diagram commute.
\end{proposition}

\begin{proof}
It is equivalent to show that given any $f:\Delta^n\to Z_\bullet$ along with data (\ref{subatlaschoice})--(\ref{augmentationschoice}) over $\partial\Delta^n$, the data can be extended over $\Delta^n$.  For ease of notation, let us rename $(f^\ast\X,f^\ast\A,f^\ast\omega)$ as $(\X,\A,\omega)$.

We construct (\ref{subatlaschoice})--(\ref{augmentationschoice}) via $H$-equivariant induction on the set of triples $(\sigma,p,q)$, equipped with the partial order $\preceq_\X$ in which $(\sigma',p',q')\preceq_\X(\sigma,p,q)$ iff $(\sigma',p',q')$ appears as a ``factor'' of some element of $\SSS_\X(\sigma,p,q)$.  This partial order is \emph{well-founded} (i.e.\ there is no strictly decreasing infinite sequence) since each $\SSS_\X(\sigma,p,q)$ is finite, and thus it is valid for induction.  The induction works $H$-equivariantly since the action of $H$ on $\PPP$ is free, and $(\sigma,p,q)$ and $(\sigma,hp,hq)$ are always incomparable for $h\in H$ (so we may henceforth ignore the action of $H$).  The inductive step is as follows.

\textbf{Choosing $\B(\sigma,p,q)$.}\  We must choose a finite $\Bbar(\sigma,p,q)\subseteq\Abar(\sigma,p,q)$ so that the corresponding subatlas of $\A(\sigma,p,q)^{\geq\s^\ttop}=\Abar(\sigma,p,q)$ on $\X(\sigma,p,q)$ satisfies the covering axiom (and thus is an implicit atlas).  This is possible by compactness.

\textbf{Choosing $[E_\alpha]$.}\  Trivial.

\textbf{Choosing $[[E]]_{\sigma,p,q}:QC_\bullet(E;\B)^+_\ZZ(\sigma,p,q)\to\ZZ$.}\  We need to construct $[[E]]_{\sigma,p,q}$ fitting into the following diagram:
\begin{equation}
\label{choosingEEdiagram}\begin{tikzcd}[column sep = huge]
\smash{\displaystyle\colim_{\s\in\partial\SSS_\X(\sigma,p,q)}}QC_\bullet(E;\B)^+_\ZZ(\sigma,p,q,\mathfrak s)\ar[tail]{r}{\text{product/face}}\ar{d}[swap]{[[E]]}&QC_\bullet(E;\B)^+_\ZZ(\sigma,p,q)\ar[dashed]{dl}{?\exists\,[[E]]_{\sigma,p,q}}\\
\ZZ\\
H_\bullet\smash{\displaystyle\colim_{\s\in\partial\SSS_\X(\sigma,p,q)}}QC_\bullet(E;\B)^+_\ZZ(\sigma,p,q,\mathfrak s)\ar{r}{\text{product/face}}\ar{d}[swap]{[[E]]}&H_\bullet QC_\bullet(E;\B)^+_\ZZ(\sigma,p,q)\ar{dl}{[E]\mapsto 1}\\
\ZZ
\end{tikzcd}
\end{equation}
Since $QC_\bullet(E;\B)^+_\ZZ$ is cofibrant, it follows that the top horizontal map is a cofibration between cofibrant complexes (this uses Lemma \ref{cofibranthaslotsofcofibrations}).  Now Lemma \ref{Hzeroofcofibrantcolimit} tells us that the direct sum of $H_0QC_\bullet(E;\B)^+_\ZZ(\sigma,p,q,\mathfrak s)$ surjects onto $H_0$ on the left, and this makes it easy to check that the second diagram commutes.  Lemma \ref{Hzeroofcofibrantcolimit} also tells us that $H_{-1}$ on the left vanishes.  Hence we may apply Lemma \ref{cocycleextendinglemma} to conclude the existence of a suitable $[[E]]_{\sigma,p,q}$.

\textbf{Choosing $\lambda_{\sigma,p,q}$ (the key step).}\  According to Lemma \ref{mappingconechoiceboundary}, we may choose a $\lambda_{\sigma,p,q}$ with the required property iff the homology class of $\mu_{\sigma,p,q}\in H^\bullet_\vir(\partial\X;\B)^+(\sigma,p,q)$ coincides with $d\omega(\sigma,p,q)\in\cH^0(\partial\X(\sigma,p,q);\oo_{\partial\X(\sigma,p,q)})$.  Let us now prove this desired statement $[\mu_{\sigma,p,q}]=d\omega(\sigma,p,q)$.  Let us fix $\rr_0\in\partial\SSS_\X(\sigma,p,q)$ and verify equality over $\X(\sigma,p,q)^{\rr_0}$ (this is clearly enough).

We consider the following diagram:
\begin{align}
\label{mucareabout}\mu_{\sigma,p,q}\in&&&C_\vir^\bullet(\partial\X;\B)^+(\sigma,p,q)\\
&&&\qquad\uparrow\cr
\label{muallstrat}\bigoplus_{\rr\in\partial\SSS_\X(\sigma,p,q)}\mu_{\sigma,p,q,\rr}\in&&\bigoplus_{\rr\in\partial\SSS_\X(\sigma,p,q)}\hocolim_{\s\prec\ttt\in\SSS_\X(\sigma,p,q)^{\leq\rr}}&C_\vir^\bullet(\X(\sigma,p,q)^{\leq\s}\rel\partial;\B(\sigma,p,q)^{\geq\ttt})\\
&&&\qquad\downarrow\cr
\label{musomestrat}\bigoplus_{\rr\in\partial\SSS_\X(\sigma,p,q)^{\leq\rr_0}}\mu_{\sigma,p,q,\rr}\in&&\bigoplus_{\rr\in\partial\SSS_\X(\sigma,p,q)^{\leq\rr_0}}\hocolim_{\s\prec\ttt\in\SSS_\X(\sigma,p,q)^{\leq\rr}}&C_\vir^\bullet(\X(\sigma,p,q)^{\leq\s}\rel\partial;\B(\sigma,p,q)^{\geq\ttt})\\
&&&\qquad\uparrow\cr
\label{muonestratum}\mu_{\sigma,p,q,\rr_0}\in&&\hocolim_{\s\prec\ttt\in\SSS_\X(\sigma,p,q)^{\leq\rr_0}}&C_\vir^\bullet(\X(\sigma,p,q)^{\leq\s}\rel\partial;\B(\sigma,p,q)^{\geq\ttt})
\end{align}
The differentials in \eqref{muallstrat}--\eqref{musomestrat} are given by the sum over $\codim(\rr\preceq\rr')=1$ of the obvious pushforward maps (plus the internal differentials of each $\hocolim$).  The first vertical map is the sum over $\codim\rr=0$ (i.e.\ over maximal elements of $\partial\SSS_\X(\sigma,p,q)$); clearly this is a chain map.  The remaining vertical maps are clear.  The easiest way to keep track of signs in the present discussion is to use the orientation lines associated to elements of $\SSS_\X(\sigma,p,q)$, as in Definition \ref{stratifiedvirtualdef}, though we will supress them from the notation.  Note that there is a natural diagram of complexes of $\K$-presheaves on $\partial\X(\sigma,p,q)$ whose diagram of global sections is \eqref{mucareabout}--\eqref{muonestratum}.  We will see below that these complexes of $\K$-presheaves are in fact pure homotopy $\K$-sheaves and that the induced diagram of sheaves (obtained by taking $H^0$) is given by the tautological maps:
\begin{equation}\label{orientationdiagramforHFbdry}
\oo_{\partial\X(\sigma,p,q)}\leftarrow\oo_{\partial\X(\sigma,p,q)}\to\oo_{\X(\sigma,p,q)^{\leq\rr_0}}\leftarrow\oo_{\X(\sigma,p,q)^{\leq\rr_0}\rel\partial}
\end{equation}
First, though, let us argue that this claim implies the desired result.  The cycle $\mu_{\sigma,p,q}\in\eqref{mucareabout}$ lifts to a cycle $\bigoplus_{\rr\in\partial\SSS_\X(\sigma,p,q)}\mu_{\sigma,p,q,\rr}\in\eqref{muallstrat}$, where $\mu_{\sigma,p,q,\rr}$ is obtained by applying the product/face maps to the tensor product of the various $\lambda$'s corresponding to the factors on the left hand side of \eqref{Wtotalcomposition} corresponding to $\rr$ (this is a cycle by \eqref{lambdamaster}).  Obviously $\bigoplus_{\rr\in\partial\SSS_\X(\sigma,p,q)}\mu_{\sigma,p,q,\rr}\in\eqref{muallstrat}$ maps to $\bigoplus_{\rr\in\partial\SSS_\X(\sigma,p,q)^{\leq\rr_0}}\mu_{\sigma,p,q,\rr}\in\eqref{musomestrat}$.  On sections over $K\subseteq\X(\sigma,p,q)^{\rr_0}$, the last vertical map is an isomorphism, and (the restriction to $K$ of) $\bigoplus_{\rr\in\partial\SSS_\X(\sigma,p,q)^{\leq\rr_0}}\mu_{\sigma,p,q,\rr}\in\eqref{musomestrat}$ lifts to (the restriction to $K$ of) $\mu_{\sigma,p,q,\rr_0}\in\eqref{muonestratum}$.  Finally, observe that over any $K\subseteq\X(\sigma,p,q)^{\rr_0}$, the homology class of $\mu_{\sigma,p,q,\rr_0}\in\eqref{muonestratum}$ coincides with the restriction of $d\omega(\sigma,p,q)$ to the $\rr_0$ stratum (this follows from the induction hypothesis, using the fact that the orientations are coherent and the observation at the end of Definition \ref{productatlasF}).  Clearly this implies that $\mu_{\sigma,p,q}$ coincides with $d\omega(\sigma,p,q)$ in homology over $\X(\sigma,p,q)^{\rr_0}$.

Thus it remains only to show that \eqref{mucareabout}--\eqref{muonestratum} are pure homotopy $\K$-sheaves and that the induced maps on $H^0$ coincide with the tautological maps on orientation sheaves \eqref{orientationdiagramforHFbdry}.  Note that by Proposition \ref{stratifiedispurehomotopysheaf}, the following closely related diagram satisfies all of the desired properties (i.e.\ is a diagram of pure homotopy $\K$-sheaves inducing \eqref{orientationdiagramforHFbdry} on $H^0$):
\begin{align}
&C_\vir^\bullet(\partial\X(\sigma,p,q);\B(\sigma,p,q)^{\geq\s^\ttop})\\
&\qquad\uparrow\cr
\bigoplus_{\rr\in\partial\SSS_\X(\sigma,p,q)}&C_\vir^\bullet(\X(\sigma,p,q)^{\leq\rr}\rel\partial;\B(\sigma,p,q)^{\geq\s^\ttop})\\
&\qquad\downarrow\cr
\bigoplus_{\rr\in\partial\SSS_\X(\sigma,p,q)^{\leq\rr_0}}&C_\vir^\bullet(\X(\sigma,p,q)^{\leq\rr}\rel\partial;\B(\sigma,p,q)^{\geq\s^\ttop})\\
&\qquad\uparrow\cr
&C_\vir^\bullet(\X(\sigma,p,q)^{\leq\rr_0}\rel\partial;\B(\sigma,p,q)^{\geq\s^\ttop})
\end{align}
Hence it is enough to relate this diagram to \eqref{mucareabout}--\eqref{muonestratum} via quasi-isomorphisms.  The diagram obtained from \eqref{mucareabout}--\eqref{muonestratum} by replacing every occurence of $\B(\sigma,p,q)^{\geq\ttt}$ with $\B(\sigma,p,q)^{\geq\s^\ttop}$ maps quasi-isomorphically to \eqref{mucareabout}--\eqref{muonestratum} (by pairing with the fixed fundamental cycles), and it also maps quasi-isomorphically to the diagram above (use the obvious pushforwards on the $p=0$ level of the homotopy colimits, and zero for $p>0$), thus giving the desired result.

\textbf{Choosing $\tilde\lambda_{\sigma,p,q}$.}\  We must find $\tilde\lambda_{\sigma,p,q}$ lifting $s_\ast\lambda_{\sigma,p,q}$ and satisfying $d\tilde\lambda_{\sigma,p,q}=\tilde\mu_{\sigma,p,q}$.  Now we know that $\tilde\mu_{\sigma,p,q}$ is a cycle (see Definition \ref{systemofcycles}); furthermore its image in $C_\bullet(E;\B)^+$ is $s_\ast\mu_{\sigma,p,q}=d(s_\ast\lambda_{\sigma,p,q})$ which is null-homologous.  Since $Q\W_\bullet\overset\sim\twoheadrightarrow\W_\bullet$ is always a quasi-isomorphism, it follows that $\tilde\mu_{\sigma,p,q}$ is null-homologous.  Hence there exists $\tilde\lambda_{\sigma,p,q}'$ with $d\tilde\lambda_{\sigma,p,q}'=\tilde\mu_{\sigma,p,q}$; now we must modify $\tilde\lambda_{\sigma,p,q}'$ so that it lifts $s_\ast\lambda_{\sigma,p,q}$.  The difference of $s_\ast\lambda_{\sigma,p,q}$ and the image of $\tilde\lambda_{\sigma,p,q}'$ is a cycle in $C_\bullet(E;\B)^+$.  It can be lifted to a cycle in $QC_\bullet(E;\B)^+$ using Lemma \ref{cycleliftinglemma} (again using that $Q\W_\bullet\overset\sim\twoheadrightarrow\W_\bullet$ is a quasi-isomorphism), which is enough.
\end{proof}

\subsection{Categories of complexes}

We review various categories of complexes over graded rings.  We use these categories as targets for the Floer-type homology groups we construct.

\begin{definition}[Complexes over graded rings]
Let $S$ be a graded ring.  A \emph{differential graded $S$-module} is a graded module $A^\bullet$ over $S$ along with a map $d:A^\bullet\to A^{\bullet+1}$ satisfying $d^2=0$.
\end{definition}

\begin{definition}[Category $\Ch_S$]\label{ChRdef}
Let $S$ be a graded ring.  We let $\Ch_S$ denote the category whose objects are differential graded $S$-modules and whose morphisms are chain maps.
\end{definition}

\begin{definition}[Category $H^0(\Ch_S)$]\label{HChRdef}
Let $S$ be a graded ring.  We let $H^0(\Ch_S)$ denote the category whose objects are objects of $\Ch_S$ and whose morphisms are chain maps modulo chain homotopy.
\end{definition}

\begin{definition}[$\infty$-category $\Ndg(\Ch_S)$]\label{Ndgdef}
Let $S$ be a graded ring.  We let $\Ndg(\Ch_S)$ denote the differential graded nerve of $\Ch_S$ (see Lurie \cite[Construction 1.3.1.6]{lurieha} or Definition \ref{NdgdefII} below).
\end{definition}

Let us now discuss the relationship between these three categories and in particular explain the definition of $\Ndg(\Ch_S)$ in concrete terms.  Despite appearances, the reader need not be familiar with $\infty$-categories to understand $\Ndg(\Ch_S)$.

Recall that if $\C$ is a category and $X_\bullet$ is a simplicial set, a \emph{diagram} $X_\bullet\to\C$ is a map of simplicial sets $X_\bullet\to N_\bullet\C$ where $N_\bullet\C$ denotes the \emph{nerve}\footnote{An $n$-simplex $\Delta^n\to N_\bullet\C$ is a diagram $A_0\xrightarrow{f_{01}}A_1\xrightarrow{f_{12}}\cdots\xrightarrow{f_{n-1,n}}A_n$ in $\C$.  Strictly speaking, $N_\bullet\C$ is not a simplicial \emph{set} unless $\C$ is small.} of $\C$.  Concretely, a diagram $F:X_\bullet\to\C$ is the data of:
\begin{rlist}
\item For every vertex $v\in X_0$, an object $A_v\in\C$.
\item For every edge $e\in X_1$ from $v_0$ to $v_1$, a morphism $f_e:A_{v_0}\to A_{v_1}$ in $\C$.
\rlistsave
\end{rlist}
such that:
\begin{rlist}
\rlistresume
\item\label{degenerateisIDdiagram}For every degenerate edge $e\in X_1$ over vertex a $v$, we have $f_e=\id_{A_v}$.
\item\label{compositionagrees} For every face in $X_2$ spanning edges $e_{01},e_{12},e_{02}$, we have $f_{e_{12}}\circ f_{e_{01}}=f_{e_{02}}$.
\end{rlist}
We often speak of a diagram $X_\bullet\to\C$ where $X_\bullet$ is only a \emph{semi}simplicial set, in which case we ignore condition (\ref{degenerateisIDdiagram}).

It should now be clear what we mean by a diagram $X_\bullet\to\Ch_S$ or $X_\bullet\to H^0(\Ch_S)$ if $X_\bullet$ is a (semi)simplicial set.  Let us now say what we mean by a diagram $X_\bullet\to\Ndg(\Ch_S)$.  Such a diagram is similar to a diagram $X_\bullet\to\Ch_S$, except that we only require condition (\ref{compositionagrees}) to hold ``up to coherent higher homotopy''.

\begin{definition}[{Diagram $X_\bullet\to\Ndg(\Ch_S)$ \cite[Construction 1.3.1.6]{lurieha}}]\label{NdgdefII}
Let $S$ be a graded ring and let $X_\bullet$ be a simplicial set.  A diagram $X_\bullet\to\Ndg(\Ch_S)$ consists of:\footnote{For the reader familiar with $\infty$-categories: the notion of a diagram $X_\bullet\to\Ndg(\Ch_S)$ suffices to define $\Ndg(\Ch_S)$ by Yoneda's Lemma.}
\begin{rlist}
\item For every $v\in X_0$, a graded $S$-module $A_v^\bullet$.
\item For every $\sigma\in X_n$ spanning $v_0,\ldots,v_n\in X_0$, a map $f_\sigma:A_{v_0}\to A_{v_n}$ of degree $1-n$, such that:
\begin{equation}\label{Ndgmaster}
\sum_{k=0}^n(-1)^kf_{\sigma|[k\ldots n]}\circ f_{\sigma|[0\ldots k]}=\sum_{k=1}^{n-1}(-1)^kf_{\sigma|[0\ldots\hat k\ldots n]}
\end{equation}
and if $\sigma$ is degenerate, then:
\begin{equation}\label{Ndgdegenerate}
f_\sigma=\begin{cases}\id&\dim\sigma=1\cr 0&\text{otherwise}\end{cases}
\end{equation}
\end{rlist}
If $X_\bullet$ is only a \emph{semi}simplicial set, then we ignore \eqref{Ndgdegenerate}.
\end{definition}

Let us now explain this definition by examining what \eqref{Ndgmaster} says for low-dimensional simplices $\sigma:\Delta^n\to X_\bullet$ (following \cite[Example 1.3.1.8]{lurieha}).
\begin{rlist}
\item Let $n=0$.  Then $f_0:A_0^\bullet\to A_0^\bullet$ has degree $1$, and \eqref{Ndgmaster} asserts that $f_0\circ f_0=0$.  Thus $(A_0,f_0)$ is a chain complex.
\item Let $n=1$.  Then $f_{01}:A_0^\bullet\to A_1^\bullet$ has degree $0$, and \eqref{Ndgmaster} asserts that $f_{01}\circ f_0-f_1\circ f_{01}=0$.  Thus $f_{01}:(A^\bullet_0,f_0)\to(A^\bullet_1,f_1)$ is a chain map.  For a degenerate $1$-simplex, \eqref{Ndgdegenerate} asserts that $f_{01}:(A^\bullet_0,f_0)\to(A^\bullet_1,f_1)$ is the identity map.
\item Let $n=2$.  Then $f_{012}:A^\bullet_0\to A^\bullet_2$ has degree $-1$, and \eqref{Ndgmaster} asserts that $f_{012}f_0-f_{12}f_{01}+f_2f_{012}=-f_{02}$.  Thus $f_{012}$ is a chain homotopy between $f_{02},f_{12}f_{01}:A^\bullet_0\to A^\bullet_2$.
\end{rlist}
We can now relate $\Ndg(\Ch_S)$ to the more familiar categories $\Ch_S$ and $H^0(\Ch_S)$ by introducing natural forgetful functors:
\begin{equation}\label{forgetfulonchaincomplexes}
\Ch_S\to\Ndg(\Ch_S)\to H^0(\Ch_S)
\end{equation}
which should be clear given the above interpretation of diagrams to $\Ndg(\Ch_S)$.  More precisely, a diagram $X_\bullet\to\Ch_S$ gives rise to a diagram $X_\bullet\to\Ndg(\Ch_S)$ where the higher homotopies $\{f_\sigma\}_{\dim\sigma\geq 2}$ are all zero.  A diagram $X_\bullet\to\Ndg(\Ch_S)$ gives rise to a diagram $X_\bullet\to H^0(\Ch_S)$ which forgets about the choice of higher homotopies $\{f_\sigma\}_{\dim\sigma\geq 2}$, remembering only the existence of $\{f_\sigma\}_{\dim\sigma=2}$ satisfying \eqref{Ndgmaster}.

\subsection{Definition}\label{homologydefsubsec}

We now define the Floer-type homology groups of an $H$-equivariant flow category diagram $\X/Z_\bullet$ with implicit atlas and coherent orientations.  More precisely, given such data we construct a diagram $\HH(\X):Z_\bullet\to H^0(\Ch_{R[[H]]})$.

We first define a diagram $\widetilde\HH(\X):\tilde Z_\bullet\to\Ndg(\Ch_{R[[H]]})$ (this is straightforward from the definition of $\tilde Z_\bullet$).  Schematically:
\begin{equation}
\begin{CD}
\tilde Z_\bullet@>\widetilde\HH(\X)>>\Ndg(\Ch_{R[[H]]})\cr
@V\pi VV\cr
Z_\bullet
\end{CD}
\end{equation}
We then use Proposition \ref{trivialKanfibration} (that $\pi:\tilde Z_\bullet\to Z_\bullet$ is a trivial Kan fibration) to show that $\widetilde\HH(\X)$ descends uniquely to a diagram $\HH(\X):Z_\bullet\to H^0(\Ch_{R[[H]]})$.

For the descent argument, we need to assume that $Z_\bullet$ is a \emph{simplicial} set.  The main nontrivial step is to show that $\tilde\HH(\X)$ sends certain lifts of degenerate edges in $Z_\bullet$ to $\tilde Z_\bullet$ to the identity map in $H^0(\Ch_{R[[H]]})$.

\begin{remark}
There should also be a (more refined) descent with target $\Ndg(\Ch_{R[[H]]})$ (c.f.\ Remark \ref{restricttotwoskeleton}), though we have decided to omit this for sake of brevity (the correct uniqueness statement is more complicated to state).
\end{remark}

\begin{definition}[Novikov rings]\label{gradedcompletion}
Let $T$ be a set equipped with a ``grading'' $\gr:T\to\ZZ$ and an ``action'' $a:T\to\RR$.  We let $R[[T]]$ denote the graded $R$-module consisting of formal sums $\sum_{t\in T}c_t\cdot t$ with $c_t\in R$, satisfying the following two finiteness conditions:
\begin{rlist}
\item$\#\{n\in\ZZ:\exists t$ such that $c_t\ne 0$ and $\gr(t)=n\}<\infty$.
\item$\#\{t\in T:c_t\ne 0$ and $a(t)<M\}<\infty$ for all $M<\infty$.
\end{rlist}
In other words, $R[[T]]$ is the graded completion of $R[T]$ with respect to the non-archimedean $a$-adic norm $\left|t\right|_a:=e^{-a(t)}$.  If $T$ is a group and $\gr$, $a$ are group homomorphisms, then $R[[T]]$ is a graded ring.
\end{definition}

\begin{definition}[(Co)homology groups $\widetilde\HH(\X):\tilde Z_\bullet\to\Ndg(\Ch_{R[[H]]})$]\label{diagramoverZtilde}
Let $\X/Z_\bullet$ be an $H$-equivariant flow category diagram with implicit atlas $\A$ and coherent orientations $\omega$.  We define a diagram:
\begin{equation}
\widetilde\HH(\X)_{A,\omega}:\tilde Z_\bullet\to\Ndg(\Ch_{R[[H]]})
\end{equation}
We will write $\widetilde\HH(\X)$ for $\widetilde\HH(\X)_{A,\omega}$ when the atlas and orientations are clear from context.

To a vertex of $\tilde Z_\bullet$, we associate $R[[\PPP_z]]$ where $z$ is the corresponding vertex in $Z_\bullet$.  This is clearly a graded $R[[H]]$-module.

Now for a simplex $\sigma\in\tilde Z_n$, we aim to define the map $f_\sigma:R[[\PPP_{\sigma(0)}]]\to R[[\PPP_{\sigma(n)}]]$ by the formula:
\begin{equation}\label{diagramndgformula}
f_\sigma(p):=\sum_{q\in\PPP_{\sigma(n)}}([[E]]_{\sigma,p,q}\otimes\id_R)(\tilde\lambda_{\sigma,p,q})\cdot q
\end{equation}
Let us now argue that \eqref{diagramndgformula} makes sense and that the resulting maps $f_\sigma$ satisfy \eqref{Ndgmaster}.

First, observe that since $\tilde\lambda_{\sigma,p,q}$ is of degree $\gr(q)-\gr(p)+\dim\sigma-1$ and $[[E]]_{\sigma,p,q}$ has degree $0$, all terms on the right hand side of \eqref{diagramndgformula} are of degree $\gr(p)+1-\dim\sigma$.  Now, we have $\tilde\lambda_{\sigma,p,q}=0$ if $\X(\sigma,p,q)=\varnothing$.  Hence finiteness condition Definition \ref{semisimplicialfcdef}(\ref{novikovfiniteness}) implies that \eqref{diagramndgformula} converges in the $a$-adic topology, and hence defines an $R$-linear map $f_\sigma:R[\PPP_{\sigma(0)}]\to R[[\PPP_{\sigma(0)}]]$.  Finiteness condition (\ref{novikovfinitenessII}) implies that there exists $M<\infty$ such that $\left|f_\sigma(p)\right|_a\leq M\left|p\right|_a$.  Thus $f_\sigma$ extends uniquely to a continuous (in fact, bounded) $R$-linear map $f_\sigma:R[[\PPP_{\sigma(0)}]]\to R[[\PPP_{\sigma(0)}]]$.  Clearly $f_\sigma$ is $R[[H]]$-linear since $[[E]]$ and $\tilde\lambda$ are both $H$-invariant.

Now it remains to verify that the $f_\sigma$ defined by \eqref{diagramndgformula} satisfy \eqref{Ndgmaster}.  To see this, write:
\begin{align*}
0&=(-1)^{\gr(p)}\cdot[[E]]_{\sigma,p,r}(d\tilde\lambda_{\sigma,p,r})\cr
&=(-1)^{\gr(p)}\cdot[[E]]_{\sigma,p,r}(\tilde\mu_{\sigma,p,r})\cr
&=[[E]]_{\sigma,p,r}\left(\sum_{k=0}^n\sum_{q\in\PPP_{\sigma(k)}}(-1)^k\tilde\lambda_{\sigma|[0\ldots k],p,q}\cdot\tilde\lambda_{\sigma|[k\ldots n],q,r}-\sum_{k=1}^{n-1}(-1)^k\tilde\lambda_{\sigma|[0\ldots\hat k\ldots n],p,r}\right)\cr
&=\sum_{k=0}^n\sum_{q\in\PPP_{\sigma(k)}}(-1)^k[[E]]_{\sigma|[0\ldots k],p,q}(\tilde\lambda_{\sigma|[0\ldots k],p,q})\cdot[[E]]_{\sigma|[k\ldots n],q,r}(\tilde\lambda_{\sigma|[k\ldots n],q,r})\cr&\qquad-\sum_{k=1}^{n-1}(-1)^k[[E]]_{\sigma|[0\ldots\hat k\ldots n],p,r}(\tilde\lambda_{\sigma|[0\ldots\hat k\ldots n],p,r})\cr
&=\text{coefficient of }r\text{ in }\sum_{k=0}^n(-1)^kf_{\sigma|[k\ldots n]}(f_{\sigma|[0\ldots k]}(p))-\sum_{k=1}^{n-1}(-1)^kf_{\sigma|[0\ldots\hat k\ldots n]}(p)
\end{align*}
The first equality follows because $[[E]]_{\sigma,p,r}$ is a chain map, the second equality follows from the definition of a system of chains, the third equality is \eqref{lambdamaster}, the fourth equality follows from the fact that $[[E]]$ is compatible with the product/face maps, and the fifth equality follows from the definition of $f_\sigma$ \eqref{diagramndgformula}.
\end{definition}

\begin{proposition}[Degenerate edges in $\tilde Z_\bullet$ give the identity map up to homotopy]\label{degenerateID}
Let $Z_\bullet=\ast$ be the semisimplicial set with a single simplex $\sigma^i$ in dimension $i$ for all $i\geq 0$ (i.e.\ the simplicial $0$-simplex).  Let $\X/Z_\bullet$ be an $H$-equivariant flow category diagram with implicit atlas and coherent orientations.  Suppose that:
\begin{rlist}
\item\label{identitytransverse}$\X(\sigma^1,p,p)=\X(\sigma^1,p,p)^\reg$ is a single point and $\omega(\sigma^1,p,p)=1$, for all $p$.
\item\label{highernonemptyissameaszero}$\X(\sigma^0,p,q)=\varnothing\implies\X(\sigma^i,p,q)=\varnothing$ for all $i$ and all $p\ne q$.
\end{rlist}
Then for any $\sigma\in\tilde Z_1$ whose two vertices coincide, the associated map in $\widetilde\HH(\X)$ is homotopic to the identity map.
\end{proposition}

\begin{proof}
Let $d:R[[\PPP]]\to R[[\PPP]]$ denote the boundary operator associated to the vertex of $\sigma$, and let $1-\epsilon:R[[\PPP]]\to R[[\PPP]]$ denote the chain map associated to $\sigma$.  We must show that $\epsilon$ is chain homotopic to the zero map.  Since $\tilde Z_\bullet\to Z_\bullet$ is a trivial Kan fibration (Proposition \ref{trivialKanfibration}), there exists a $2$-simplex in $\tilde Z_\bullet$ all of whose edges are $\sigma$.  Associated to this $2$-simplex is a chain homotopy $h:R[[\PPP]]\to R[[\PPP]]$ between $(1-\epsilon)$ and $(1-\epsilon)\circ(1-\epsilon)$; in other words:
\begin{equation}
\epsilon=dh+hd+\epsilon^2
\end{equation}

Now we claim that the only nonzero ``matrix coefficients'' $c_{p,q}$ of $d,\epsilon,h$ are those for which $\X(\sigma^0,p,q)\ne\varnothing$.  By definition, $f_\sigma$ can have nonzero matrix coefficients $c_{p,q}$ only for $\X(\sigma,p,q)\ne\varnothing$.  Hence hypothesis (\ref{highernonemptyissameaszero}) gives the desired claim as long as $p\ne q$.  For the diagonal matrix coefficients $c_{p,p}$, we argue separately as follows.  By degree considerations, only $\epsilon$ can have nonzero $c_{p,p}$.  It thus suffices to show that the matrix coefficient $c_{p,p}=[[E]]_{\sigma^1,p,p}(\tilde\lambda_{\sigma^1,p,p})$ of $f_{\sigma^1}=1-\epsilon$ equals $1$.  Note that no product/face maps have target $(\sigma^1,p,p)$, so $\lambda_{\sigma^1,p,p}$ is a cycle representing $\omega(\sigma^1,p,p)=1$.  Hence $\tilde\lambda_{\sigma^1,p,p}$ is a cycle representing $[E]$, and so $[[E]]_{\sigma^1,p,p}(\tilde\lambda_{\sigma,p,p})=1$ as needed.  Hence the claim is valid.  From the claim, we make the following two observations:
\begin{rlist}
\item\label{firstO}We have $\left|d(p)\right|_a,\left|\epsilon(p)\right|_a,\left|h(p)\right|_a\leq\left|p\right|_a$.
\item\label{secondO}For all $M<\infty$ and $p\in\PPP$, there exists $N<\infty$ such that any length $N$ composition of $d,\epsilon,h$ applied to $p$ has $a$-adic norm $\leq e^{-M}$.
\end{rlist}
(the second observation also uses Definition \ref{semisimplicialfcdef}(\ref{novikovfiniteness})).

Now iterating the identity $\epsilon=dh+hd+\epsilon^2$, we are led to the following infinite series:\footnote{The coefficients $\frac 1n\binom{2n-2}{n-1}$ are integers (Catalan numbers $C_{n-1}$), so writing this expression does not make any implicit assumptions on the ring $R$.}
\begin{equation}
\epsilon=\sum_{n=1}^\infty\frac 1n\binom{2n-2}{n-1}(hd+dh)^n
\end{equation}
which by observations (\ref{firstO})--(\ref{secondO}) converges in the $a$-adic topology when applied to any element of $R[[\PPP]]$.  Now we have $(hd+dh)^n=\sum_{k=0}^n(dh)^k(hd)^{n-k}=H_nd+dH_n$ where $H_n=\sum_{k=1}^nh(dh)^{k-1}(hd)^{n-k}$.  This gives the desired chain homotopy between $\epsilon$ and zero (again using observations (\ref{firstO})--(\ref{secondO}) to justify convergence of infinite sums).
\end{proof}

\begin{lemma}[Criterion for descent along a trivial Kan fibration]\label{descentofdiagram}
Let $\pi:\tilde Z_\bullet\to Z_\bullet$ be a trivial Kan fibration of semisimplicial sets, and let $\widetilde\HH:\tilde Z_\bullet\to\C$ be a diagram in some category $\C$.  A \emph{descent} of $\widetilde\HH$ to $Z_\bullet$ is a diagram $\HH:Z_\bullet\to\C$ along with an isomorphism $\widetilde\HH\to\HH\circ\pi$:
\begin{equation}
\begin{tikzcd}
\tilde Z_\bullet\arrow{r}{\widetilde\HH}\arrow{d}[swap]{\pi}&\C\\
Z_\bullet\arrow[dashed]{ur}[swap]{\HH}
\end{tikzcd}
\end{equation}
Suppose that:
\begin{rlist}
\item$Z_\bullet$ is a simplicial set.
\item\label{degenerateIDindescent}For any edge $\sigma^1$ in $\tilde Z_\bullet$ whose endpoints coincide and which projects to a degenerate edge in $Z_\bullet$, the associated map in $\C$ is the identity map.
\end{rlist}
Then a descent exists and is unique up to unique isomorphism.
\end{lemma}

\begin{proof}
The proof below shows that we may define $\HH:=\widetilde\HH\circ s$ for any section $s:Z_\bullet\to\tilde Z_\bullet$.

\textbf{Defining $\HH$ over $0$-simplices.}\  Fix some $0$-simplex $v$ of $Z_\bullet$.  For every lift $\tilde v$, we have an object $\widetilde\HH(\tilde v)$.  Furthermore, for every lift $\tilde e$ of the degenerate $1$-simplex $e$ over $v$, we get a map $\widetilde\HH(\tilde v_1)\xrightarrow{\widetilde\HH(\tilde e)}\widetilde\HH(\tilde v_2)$, which is the identity map if $\tilde v_1=\tilde v_2$.  Finally, for every lift $\tilde f$ of the degenerate $2$-simplex $f$ over $v$, the following diagram commutes:
\begin{equation}\label{liftedF}
\begin{tikzcd}[row sep = small]
&\widetilde\HH(\tilde v_2)\ar{rd}{\widetilde\HH(\tilde e_{23})}&{}\cr
\widetilde\HH(\tilde v_1)\ar{ur}{\widetilde\HH(\tilde e_{12})}\ar{rr}[swap]{\widetilde\HH(\tilde e_{13})}&&{\widetilde\HH(\tilde v_3)}
\end{tikzcd}
\end{equation}
where $\tilde e_{12},\tilde e_{23},\tilde e_{13}$ are the edges of $\tilde f$.  Using that $\tilde Z_\bullet\to Z_\bullet$ is a trivial Kan fibration to conclude that we can always find lifts with specified boundary conditions, it follows that all $\widetilde\HH(\tilde v)$ are canonically isomorphic.  Thus there exists a unique choice for $\HH(v)$.

\textbf{Defining $\HH$ over $1$-simplices.}\  Fix some $1$-simplex $e$ of $Z_\bullet$ with vertices $v_1,v_2$.  If $\tilde e$ is any lift of $e$, then we get a map $\HH(v_1)=\widetilde\HH(\tilde v_1)\xrightarrow{\widetilde\HH(\tilde e)}\widetilde\HH(\tilde v_2)=\HH(v_2)$.  Furthermore, this map $\HH(v_1)\to\HH(v_2)$ is seen to be independent of the choice of lift $\tilde e$ by lifting degenerate $2$-simplices over $e$.  Thus there exists a unique choice for $\HH(e)$.

\textbf{Defining $\HH$ over $n$-simplices for $n\geq 2$.}\  We just need to check that the following diagram commutes for all $2$-simplices $f$ in $Z_\bullet$:
\begin{equation}
\begin{tikzcd}[row sep = small]
&\HH(v_2)\ar{rd}{\HH(e_{23})}&{}\cr
\HH(v_1)\ar{ur}{\HH(e_{12})}\ar{rr}[swap]{\HH(e_{13})}&&{\HH(v_3)}
\end{tikzcd}
\end{equation}
where $e_{12},e_{23},e_{13}$ are the edges of $f$.  This follows from lifting $f$ and the commutativity of \eqref{liftedF}.
\end{proof}

\begin{definition}[(Co)homology groups $\HH(\X):Z_\bullet\to H^0(\Ch_{R[[H]]})$]\label{defofhomologygroups}
Let $Z_\bullet$ be a simplicial set.  Let $\X/Z_\bullet$ be an $H$-equivariant flow category diagram with implicit atlas $\A$ and coherent orientations $\omega$.  Suppose that for any vertex $\sigma^0$ of $Z_\bullet$, we have:
\begin{rlist}
\item$\X(\sigma^1,p,p)=\X(\sigma^1,p,p)^\reg$ is a single point and $\omega(\sigma^1,p,p)=1$, for all $p$.
\item$\X(\sigma^0,p,q)=\varnothing\implies\X(\sigma^i,p,q)=\varnothing$ for all $i$ and $p\ne q$
\end{rlist}
(where $\sigma^i$ denotes the completely degenerate $i$-simplex over $\sigma^0$).  We have:
\begin{equation}
\begin{tikzcd}
\tilde Z_\bullet\arrow{r}{\widetilde\HH(\X)}\arrow{d}[swap]{\pi}&H^0(\Ch_{R[[H]]})\\
Z_\bullet
\end{tikzcd}
\end{equation}
(abusing notation and use $\widetilde\HH(\X):\tilde Z_\bullet\to H^0(\Ch_{R[[H]]})$ to denote the composition of $\widetilde\HH(\X)$ with the forgetful functor $\Ndg\to H^0$).  The hypotheses of Lemma \ref{descentofdiagram} are satisfied by Propositions \ref{trivialKanfibration} and \ref{degenerateID}, and hence we get a canonical descent:
\begin{equation}
\HH(\X)_{A,\omega}:Z_\bullet\to H^0(\Ch_{R[[H]]})
\end{equation}
We will write $\HH(\X)$ for $\HH(\X)_{A,\omega}$ when the atlas and orientations are clear from context.
\end{definition}

\subsection{Properties}

\begin{lemma}[Passing to a subatlas preserves $\HH(\X)$]\label{homologysameassubatlas}
Let $Z_\bullet$, $\X/Z_\bullet$, $\A$, and $\omega$ be as in Definition \ref{defofhomologygroups}.  If $\B\subseteq\A$ is any subatlas, then there is a canonical isomorphism $\HH(\X)_\A=\HH(\X)_\B$.
\end{lemma}

\begin{proof}
Indeed, we have $\tilde Z_\bullet^\B\subseteq\tilde Z_\bullet^\A$ compatibly with $\widetilde\HH(\X)$.
\end{proof}

\begin{lemma}[Shrinking the charts preserves $\HH(\X)$]\label{homologyopenreduction}
Let $Z_\bullet$, $\X/Z_\bullet$, $\A$, and $\omega$ be as in Definition \ref{defofhomologygroups}.  Let $\A'$ be obtained from $\A$ by using instead some open subsets $U_{IJ}'\subseteq U_{IJ}$, $X_I'\subseteq X_I$, $X_I^{\reg\prime}\subseteq X_I^\reg$, and restricting $\psi_{IJ}$, $s_I$ to these subsets, so that $\A'$ is also an implicit atlas.  Then there is a canonical isomorphism $\HH(\X)_\A=\HH(\X)_{\A'}$.
\end{lemma}

\begin{proof}
There is a natural map $\tilde Z_\bullet^{\A'}\to\tilde Z_\bullet^\A$ for which the pullback of $\widetilde\HH(\X)_\A$ is canonically isomorphic to $\widetilde\HH(\X)_{\A'}$.  This is enough.
\end{proof}

\begin{lemma}[Universal coefficients]
Let $Z_\bullet$, $\X/Z_\bullet$, $\A$, and $\omega$ be as in Definition \ref{defofhomologygroups}.  Let $R\to S$ a homomorphism of base rings.  Then there is a canonical isomorphism $\HH(\X)^S=\HH(\X)^R\otimes_{R[[H]]}S[[H]]$ (denoting $\otimes_{R[[H]]}S[[H]]:H^0(\Ch_{R[[H]]})\to H^0(\Ch_{S[[H]]})$).
\end{lemma}

\begin{proof}
There is a natural base change map $b:\tilde Z_\bullet^R\to\tilde Z_\bullet^S$ and a canonical isomorphism $\widetilde\HH(\X)^S\circ b=\widetilde\HH(\X)^R\otimes_{R[[H]]}S[[H]]$.  Now the result follows since $\widetilde\HH(\X)^S\circ b$ and $\widetilde\HH(\X)^S$ have the same descent to $Z_\bullet$.
\end{proof}

\begin{proposition}[If $\X=\X^\reg$, then $\HH(\X)$ is given by counting $0$-dimensional flow spaces]\label{iftransversejustcountzero}
Let $Z_\bullet$, $\X/Z_\bullet$, $\A$, and $\omega$ be as in Definition \ref{defofhomologygroups}.  Suppose that $\X(\sigma,p,q)=\X(\sigma,p,q)^\reg$ for all $(\sigma,p,q)$ (so, in particular, $\X/Z_\bullet$ is Morse--Smale in the sense of Remark \ref{MSflowcatediag}).

Define a diagram $\HH'(\X):Z_\bullet\to\Ndg(\Ch_{R[[H]]})$ as follows.  The graded $R[[H]]$-module associated to a vertex $z$ of $\tilde Z_\bullet$ is $R[[\PPP_z]]$.  For a simplex $\sigma\in\tilde Z_n$, we define the map $f_\sigma:R[[\PPP_{\sigma(0)}]]\to R[[\PPP_{\sigma(n)}]]$ by the formula:
\begin{equation}\label{diagramndgformulaMS}
f_\sigma(p):=\sum_{\begin{smallmatrix}q\in\PPP_{\sigma(n)}\cr\gr(q)-\gr(p)=1-\dim\sigma\end{smallmatrix}}\langle\omega(\sigma,p,q),[\X(\sigma,p,q)]\rangle\cdot q
\end{equation}
It is easy to verify that the maps $f_\sigma$ are well-defined and satisfy \eqref{Ndgmaster}.

Now there is a canonical isomorphism $\HH(\X)=\HH'(\X)$ (where on the right hand side we implicitly compose with the forgetful functor $\Ndg\to H^0$).
\end{proposition}

\begin{proof}
We will show an equality of ``matrix coefficients'':
\begin{equation}\label{matrixcoeffsame}
([[E]]_{\sigma,p,q}\otimes\id_R)(\tilde\lambda_{\sigma,p,q})=\langle\omega(\sigma,p,q),[\X(\sigma,p,q)]\rangle
\end{equation}
for $(\sigma,p,q)\in\tilde Z_\bullet$ with $\vdim\X(\sigma,p,q)=0$.  Thus comparing \eqref{diagramndgformula} and \eqref{diagramndgformulaMS}, it follows that the two diagrams $\tilde Z_\bullet\to\Ndg(\Ch_{R[[H]]})$ in question, namely $\widetilde\HH(\X)$ and $\HH'(\X)\circ\pi$, coincide.  The desired isomorphism thus follows from the definition of $\HH(\X)$ as the descent of $\widetilde\HH(\X)$.

To prove \eqref{matrixcoeffsame}, argue as follows.  Since $\X(\sigma,p,q)=\X(\sigma,p,q)^\reg$, we know that if $\vdim\X(\sigma,p,q)<0$ then $\X(\sigma,p,q)=\varnothing$ and hence $\lambda_{\sigma,p,q}=0$ and $\tilde\lambda_{\sigma,p,q}=0$.  Now when $\vdim\X(\sigma,p,q)=0$, we see that $\mu_{\sigma,p,q}=0$ and $\tilde\mu_{\sigma,p,q}=0$ (because all of the terms defining them involve a triple with negative dimension), so $\lambda_{\sigma,p,q}$ and $\tilde\lambda_{\sigma,p,q}$ are cycles.  Since the homology class of $\lambda_{\sigma,p,q}$ coincides with $\omega(\sigma,p,q)$, the left hand side of \eqref{matrixcoeffsame} coincides with the evaluation of $\omega(\sigma,p,q)$ on $[\X(\sigma,p,q)]^\vir\in\cH^0(\X(\sigma,p,q);\oo_{\X(\sigma,p,q)})^\vee$.  Now we are done since $[\X(\sigma,p,q)]^\vir=[\X(\sigma,p,q)]$ by Lemma \ref{fclasstransverse}.
\end{proof}

\section{\texorpdfstring{$S^1$}{S{\textasciicircum}1}-localization}\label{Slocalizationsection}

In this section, we prove vanishing results for virtual fundamental cycles on almost free\footnote{An ``almost free'' action is one for which the stabilizer group of every point is finite.} $S^1$-spaces equipped with an \emph{$S^1$-equivariant implicit atlas}.

\begin{convention}
In this section, we work over a fixed ground ring $R$, and everything takes place in the category of $R$-modules.  We restrict to implicit atlases $A$ for which $\#\Gamma_\alpha$ is invertible in $R$ for all $\alpha\in A$.  We restrict to $S^1$-equivariant implicit atlases for which $\#(S^1)_p$ is invertible in $R$ for all $p\in X_\varnothing$.
\end{convention}

For this purpose, we introduce the \emph{$S^1$-equivariant virtual cochain complexes} $C^\bullet_{S^1,\vir}$, which enjoys properties similar to those of $C^\bullet_\vir$.  There are canonical ``comparison maps'':
\begin{align}
\label{CSScompare}C^\bullet_{S^1,\vir}(X;A)&\to C^\bullet_\vir(X;A)\\
\label{CSScomparerel}C^\bullet_{S^1,\vir}(X\rel\partial;A)&\to C^\bullet_\vir(X\rel\partial;A)
\end{align}
and a canonical commutative diagram:
\begin{equation}\label{SSpushforwardcompatiblity}
\begin{CD}
C^{d+\bullet}_{S^1,\vir}(X\rel\partial;A)@>s_\ast>>C_{\dim E_A-\bullet-1}^{S^1}(E_A,E_A\setminus 0;\oo_{E_A}^\vee)^{\Gamma_A}\cr
@VVV@VV\pi^!V\cr
C^{d+\bullet}_\vir(X\rel\partial;A)@>s_\ast>>C_{\dim E_A-\bullet}(E_A,E_A\setminus 0;\oo_{E_A}^\vee)^{\Gamma_A}
\end{CD}
\end{equation}
($S^1$ acting trivially on $E_A$).  Furthermore, if $X$ is an almost free $S^1$-space and $A$ is \emph{locally $S^1$-orientable}, then there are canonical isomorphisms:
\begin{align}
\label{HXSSiso}H^\bullet_{S^1,\vir}(X;A)&=\cH^\bullet(X/S^1;\pi_\ast\oo_X)\\
\label{HXSSreliso}H^\bullet_{S^1,\vir}(X\rel\partial;A)&=\cH^\bullet(X/S^1;\pi_\ast\oo_{X\rel\partial})
\end{align}
The construction of $C^\bullet_{S^1,\vir}$ and the proof of these properties are the main technical ingredients for the $S^1$-localization statements we prove.

To prove the desired vanishing results, we consider using $C^\bullet_{S^1,\vir}$ in place of $C^\bullet_\vir$ to define virtual fundamental cycles.  The properties and compatibilities above then show that the desired statements follow essentially from the vanishing (on homology) of the right vertical map in \eqref{SSpushforwardcompatiblity}.  This basic strategy works easily to give the desired vanishing results for the virtual fundamental classes from \S\ref{fundamentalclasssection}.  We also apply this strategy to prove results for the Floer-type homology groups from \S\ref{homologygroupssection} in the presence of an $S^1$-action on the flow spaces (to the effect that flow spaces on which the action is almost free may be ignored).

We do not construct an $S^1$-equivariant virtual fundamental cycle, though the machinery we set up is a step in this direction (see Remark \ref{SSequivariantvfcmaybe}).

\subsection{Background on \texorpdfstring{$S^1$}{S{\textasciicircum}1}-equivariant homology}

\begin{definition}[Gysin sequence]
Let $\pi:E\to B$ be a principal $S^1$-bundle.  Analysis of the associated Serre spectral sequence gives the following Gysin long exact sequence:
\begin{equation}\label{gysinserregeneral}
\cdots\xrightarrow{\cap e}H_{\bullet-1}(B)\xrightarrow{\pi^!}H_\bullet(E)\xrightarrow{\pi_\ast}H_\bullet(B)\xrightarrow{\cap e}H_{\bullet-2}(B)\xrightarrow{\pi^!}\cdots
\end{equation}
The $\cap e$ map is cap product with the Euler class $e(E)\in H^2(B)$.  To see this, observe that the sequence \eqref{gysinserregeneral} coincides with the long exact sequence of the pair for (the mapping cone of) $\pi:E\to B$, and that there is a natural isomorphism $H_{\bullet+2}(B,E)\xrightarrow{\cap\tau}H_\bullet(B)$ where $\tau=\tau(E)\in H^2(B,E)$ is the Thom class (e.g.\ as argued in \cite[p444]{hatcher} for the corresponding sequence in cohomology).

For any $S^1$-space $X$. there is a principal $S^1$-bundle $\pi:X\times ES^1\to(X\times ES^1)/S^1$, and thus a long exact sequence:
\begin{equation}\label{gysinserre}
\cdots\xrightarrow{\cap e}H_{\bullet-1}^{S^1}(X)\xrightarrow{\pi^!}H_\bullet(X)\xrightarrow{\pi_\ast}H_\bullet^{S^1}(X)\xrightarrow{\cap e}H_{\bullet-2}^{S^1}(X)\xrightarrow{\pi^!}\cdots
\end{equation}
for $e(X)\in H^2_{S^1}(X)$.  The same reasoning applies for pairs of spaces, so the same exact sequence exists for relative homology as well.
\end{definition}

\begin{lemma}\label{trivialactiongysin}
Let $X$ be a trivial $S^1$-space.  Then the $\pi^!$ map in the Gysin sequence vanishes, turning it into a short exact sequence:
\begin{equation}
0\to H_\bullet(X)\xrightarrow{\pi_\ast}H_\bullet^{S^1}(X)\xrightarrow{\cap e}H_{\bullet-2}^{S^1}(X)\to 0
\end{equation}
The same statement applies to relative homology of trivial $S^1$-spaces.
\end{lemma}

\begin{proof}
The composition $H_\bullet(X)\to H_\bullet(X\times BS^1)\to H_\bullet(X)$ is the identity map, and so the first map is injective.  On the other hand, this map is precisely $\pi_\ast:H_\bullet(X)\to H_\bullet^{S^1}(X)$ since the $S^1$-action is trivial, so $\pi_\ast$ is injective which is sufficient.  The same argument applies in the relative setting as well.
\end{proof}

\begin{lemma}\label{freeactionislocallytrivial}
Let $X$ be a locally compact Hausdorff $S^1$-space which is almost free at $p\in X$, and suppose that the order of the stabilizer $\#(S^1)_p$ is invertible in the ground ring $R$.  Then there exists an $S^1$-invariant neighborhood $S^1p\subseteq K\subseteq X$ so that the Euler class $e(K)\in H^2_{S^1}(K)$ vanishes.
\end{lemma}

\begin{proof}
Apply the Tietze extension theorem to the identity map $S^1p\to S^1p$ to obtain an $S^1$-invariant neighborhood $K$ of $S^1p$ and a retraction $r:K\to S^1p$.  By averaging and passing to a smaller neighborhood, we may assume without loss of generality that $r$ is $S^1$-equivariant.  Now by the naturality of the Euler class, we have $e(K)=r^\ast e(S^1p)$.  Thus it suffices to show that $e(S^1p)\in H^2_{S^1}(S^1p)$ vanishes.  Now we have $H^2_{S^1}(S^1p)=H^2_{(S^1)_p}(p)=H^2((S^1)_p)$, where the latter is the group cohomology of the finite stabilizer group $(S^1)_p$.  It is a standard fact that the group cohomology of a finite group is annihilated by the order of the group.  Hence our assumption that $\#(S^1)_p$ is invertible in $R$ guarantees that this cohomology group vanishes.
\end{proof}

\begin{lemma}\label{SSmanifoldPDbase}
Let $M$ be a topological $S^1$-manifold of dimension $d$ which is almost free near $p\in M$.  Suppose that $\#(S^1)_p$ is invertible in the ground ring $R$.  Then there is a canonical isomorphism:
\begin{equation}
H_\bullet^{S^1}(M,M\setminus S^1p)=\begin{cases}H^1(S^1p;\oo_M)&\bullet=d-1\cr 0&\bullet\ne d-1\end{cases}
\end{equation}
(note also that $H^1(S^1p;\oo_M)=H^0(S^1p;\oo_M\otimes\oo_{S^1p}^\vee)$).
\end{lemma}

\begin{proof}
By Poincar\'e duality, we have canonical isomorphisms:
\begin{equation}\label{PDforSSPD}
H_\bullet(M,M\setminus S^1p)=\begin{cases}H^0(S^1p;\oo_M)&\bullet=d\cr H^1(S^1p;\oo_M)&\bullet=d-1\cr 0&\text{otherwise}\end{cases}
\end{equation}
By excision and Lemma \ref{freeactionislocallytrivial}, the Gysin sequence reduces to a short exact sequence:
\begin{equation}\label{reducedgysinserreforSSPD}
0\to H_{\bullet-1}^{S^1}(M,M\setminus S^1p)\xrightarrow{\pi^!}H_\bullet(M,M\setminus S^1p)\xrightarrow{\pi_\ast}H_\bullet^{S^1}(M,M\setminus S^1p)\to 0
\end{equation}
Combining this with \eqref{PDforSSPD}, we see that $H_\bullet^{S^1}(M,M\setminus S^1p)=0$ for $\bullet\ne d-1$, and that the following are both isomorphisms:
\begin{equation}\label{gysinweirdcomposition}
H_{d-1}(M,M\setminus S^1p)\xrightarrow{\pi_\ast}H_{d-1}^{S^1}(M,M\setminus S^1p)\xrightarrow{\pi^!}H_d(M,M\setminus S^1p)
\end{equation}
which yields the desired result.
\end{proof}

\subsection{\texorpdfstring{$S^1$}{S{\textasciicircum}1}-equivariant implicit atlases}

\begin{definition}[$S^1$-equivariant implicit atlas]\label{SSimplicitatlasdef}
Let $X$ be an $S^1$-space.  An \emph{$S^1$-equivariant implicit atlas} $A$ on $X$ is an implicit atlas $A$ along with an action of $S^1$ on each thickening $X_I$ (commuting with the $\Gamma_I$-action) such that each map $\psi_{IJ}$ is $S^1$-equivariant, each function $s_\alpha$ is $S^1$-invariant, and each subset $X_I^\reg$ is $S^1$-invariant.

Similarly, we define an $S^1$-equivariant implicit atlas with boundary and/or stratification by in addition requiring that the boundary loci $\partial X_I$ and/or stratifications $X_I\to\SSS$ be $S^1$-invariant.
\end{definition}

Note that in the above definition, $S^1$ does not act on any of the obstruction spaces $E_\alpha$ (or, alternatively, it acts trivially on them).

\subsection{\texorpdfstring{$S^1$}{S{\textasciicircum}1}-equivariant orientations}

We begin with the trivial observation that if $X$ is equipped with a locally orientable $S^1$-equivariant implicit atlas $A$, then the $S^1$-action on $X$ lifts canonically to the orientation sheaf $\oo_X$.

\begin{definition}[Locally $S^1$-orientable implicit atlas]\label{virtualorSS}
Let $X$ be an $S^1$-space with $S^1$-equivariant implicit atlas with boundary.  We say that $A$ is \emph{locally $S^1$-orientable} iff it is locally orientable and for all $p\in X$, the stabilizer $(S^1)_p$ acts trivially on $(\oo_X)_p$ (this action is always by a sign $(S^1)_p\to\{\pm 1\}$).  This notion is independent of the ring $R$ (due to our restriction that $\#(S^1)_p$ be invertible in $R$).
\end{definition}

It is easy to see that if $A$ is locally $S^1$-orientable, then $\pi_\ast\oo_X$ is locally isomorphic to the constant sheaf $\underline R$ (where $\pi:X\to X/S^1$).

\begin{remark}\label{orientedthenSSoriented}
If $A$ is locally orientable and $\oo_X$ has a global section, then $A$ is automatically locally $S^1$-orientable ($S^1$ is connected, so the section must be $S^1$-invariant, hence the claim).
\end{remark}

\subsection{\texorpdfstring{$S^1$}{S{\textasciicircum}1}-equivariant virtual cochain complexes \texorpdfstring{$C^\bullet_{S^1,\vir}(X;A)$}{C\_S{\textasciicircum}1,vir(X;A)} and \texorpdfstring{$C^\bullet_{S^1,\vir}(X\rel\partial;A)$}{C\_S{\textasciicircum}1,vir(X rel d;A)}}\label{SSCconstruct}

To define the $S^1$-equivariant virtual cochain complexes, we must first fix a model $C_\bullet^{S^1}$ of $S^1$-equivariant chains to work with.

\begin{remark}
If we used the language of $\infty$-categories, there would be no need to construct models of chains with good (chain level) functoriality properties (c.f.\ Remark \ref{independenceofchainmodel}).
\end{remark}

We begin by stating the (chain level) properties we would like our model $C_\bullet^{S^1}$ to satisfy.  We want a functor $C_\bullet^{S^1}$ from spaces to chain complexes of free $\ZZ$-modules (and then we can tensor up to any base ring $R$); we also demand that $C_\bullet^{S^1}(A)\to C_\bullet^{S^1}(X)$ be injective for $A\subseteq X$ (and then we define relative chains $C_\bullet^{S^1}(X,A)$ as the cokernel).  Now we need there to be functorial maps:
\begin{align}
\label{pishriek}C_\bullet^{S^1}(X)&\xrightarrow{\pi^!}C_{\bullet+1}(X)\\
\label{chainproducteq2}C_\bullet^{S^1}(X)\otimes C_\bullet(Y)&\to C_\bullet^{S^1}(X\times Y)
\end{align}
realizing (respectively) the Gysin map and the obvious product map.  The map \eqref{chainproducteq2} must be compatible with the Eilenberg--Zilber map on $C_\bullet$ in that the two ways of building up the following map coincide:
\begin{equation}
C_\bullet^{S^1}(X)\otimes C_\bullet(Y)\otimes C_\bullet(Z)\to C_\bullet^{S^1}(X\times Y\times Z)
\end{equation}
Moreover, \eqref{pishriek} and \eqref{chainproducteq2} must be compatible in that the following diagram commutes:
\begin{equation}
\begin{CD}
C_\bullet^{S^1}(X)\otimes C_\bullet(Y)@>>>C_\bullet^{S^1}(X\times Y)\cr
@VVV@VVV\cr
C_{\bullet+1}(X)\otimes C_\bullet(Y)@>>>C_{\bullet+1}(X\times Y)\cr
\end{CD}
\end{equation}

To define $C_\bullet^{S^1}$ with the aforementioned properties, let us recall the construction of the Serre spectral sequence due to Dress \cite{dress}.  For a Serre fibration $\pi:E\to B$, we consider diagrams of the form:
\begin{equation}
\begin{CD}
\Delta^p\times\Delta^q@>>>E\cr
@VVV@VV\pi V\cr
\Delta^p@>>>B
\end{CD}
\end{equation}
Let $C_{p,q}(\pi:E\to B)$ denote the free abelian group generated by such diagrams.  Then the direct sum of all of these $C_{\bullet,\bullet}(\pi:E\to B)$ is a double complex (differentials corresponding to the two pieces of boundary $\partial\Delta^p\times\Delta^q$ and $\Delta^p\times\partial\Delta^q$).  We let $C_\bullet(\pi:E\to B)$ denote the corresponding total complex.  There is a natural map $C_\bullet(\pi:E\to B)\to C_\bullet(E)$, where we subdivide $\Delta^p\times\Delta^q$ in the usual way.  By considering the spectral sequence associated to the filtration by $q$, Dress showed that this map is a quasi-isomorphism.  Dress also showed that the spectral sequence associated to the filtration by $p$ is the Serre spectral sequence; in particular, the $E^2_{p,q}$ term is $H_p(B,H_q(F))$.

To define $C_\bullet^{S^1}$, fix once and for all an $ES^1$.  Then for any space $X$, the map $\pi:X\times ES^1\to(X\times ES^1)/S^1$ is a principal $S^1$-bundle, so \emph{a fortiori} it is a Serre fibration.  We define:
\begin{multline}\label{qdegreedef}
C_{\bullet-1}^{S^1}(X):=\Bigl\{\gamma\in C_\bullet(\pi:X\times ES^1\to(X\times ES^1)/S^1)\Bigm|\\
\gamma,d\gamma\text{ have no component with $q$-degree}<1\Bigr\}
\end{multline}
The inclusion $C_{\bullet-1}^{S^1}(X)\hookrightarrow C_\bullet(\pi:X\times ES^1\to(X\times ES^1)/S^1)$ is compatible with the $p$-grading on each.  Let us consider the associated morphism of spectral sequences (induced by the $p$-filtration).  The latter has $E^2_{p,q}$ term $H_p^{S^1}(X,H_q(S^1))$ by Dress.  It follows from the definition \eqref{qdegreedef} that the $E^2_{p,q}$ term of the former is the same in degrees $q\geq 1$ and zero otherwise.  Since $H_q(S^1)$ is nonzero only for $q\leq 1$, it follows that the former spectral sequence has no further differentials, and we conclude that the homology of $C_\bullet^{S^1}(X)$ is indeed $H_\bullet^{S^1}(X)$ as needed.

Now let us define the maps \eqref{pishriek} and \eqref{chainproducteq2} and verify the required properties.  The map \eqref{pishriek} is obtained by the standard subdivision of $\Delta^p\times\Delta^q$ into simplices along with the projection map $X\times ES^1\to X$.  The map \eqref{chainproducteq2} is defined as follows.  Given maps $\Delta^p\times\Delta^q\to X\times ES^1$ and $\Delta^{p'}\to Y$, we obtain a map $\Delta^{p'}\times\Delta^p\times\Delta^q\to Y\times X\times ES^1$ and we subdivide $\Delta^{p'}\times\Delta^p$.  The required properties are then straightforward to verify.

\begin{definition}[$S^1$-equivariant virtual cochain complexes $C_{S^1,\vir}^\bullet(-;A)$ (and $_{IJ}$)]
Let $X$ be an $S^1$-space with finite $S^1$-equivariant implicit atlas with boundary $A$.  For any $S^1$-invariant compact $K\subseteq X$, we define:
\begin{align}
\label{bigSSpartial}C^\bullet_{S^1,\vir}(K;A)_{IJ}\quad&C^\bullet_{S^1,\vir}(K\rel\partial;A)_{IJ}\\
\label{bigSS}C^\bullet_{S^1,\vir}(K;A)\phantom{_{IJ}}\quad&C^\bullet_{S^1,\vir}(K\rel\partial;A)
\end{align}
as in Definitions \ref{CSTcomplexdef}--\ref{Ccomplexdef}, except using $C_{\bullet-1}^{S^1}$ in place of $C_\bullet$ in \eqref{CcomplexcriticaldefI}--\eqref{CcomplexcriticaldefII}.  It is clear that \eqref{bigSSpartial}--\eqref{bigSS} are complexes of $\K$-presheaves on $X/S^1$ (replace $K$ with $\pi^{-1}(K)$, where $\pi:X\to X/S^1$).  The Gysin map \eqref{pishriek} induces ``comparison maps'':
\begin{equation}\label{CSStimesSS}
C^\bullet_{S^1,\vir}(-;A)_{(IJ)}\to C^\bullet_\vir(-;A)_{(IJ)}
\end{equation}
for all flavors \eqref{bigSSpartial}--\eqref{bigSS}

Analogously with \eqref{Erelpartialpullback}--\eqref{Epushforward1} and \eqref{Frelpartialpullback}--\eqref{Fpushforward}, there are natural maps (compatible with \eqref{CSStimesSS}):
\begin{align}
C^\bullet_{S^1,\vir}(K\rel\partial;A)_{IJ}&\to C^\bullet_{S^1,\vir}(K;A)_{IJ}\\
C^{d+\bullet}_{S^1,\vir}(X\rel\partial;A)_{IJ}&\xrightarrow{s_\ast}C_{-\bullet-1}^{S^1}(E;A)\\
C^\bullet_{S^1,\vir}(-;A)_{IJ}&\to C^\bullet_{S^1,\vir}(-;A)_{I',J'}\\
\label{CSSrelpartialpullback}C^\bullet_{S^1,\vir}(K\rel\partial;A)&\to C^\bullet_{S^1,\vir}(K;A)\\
\label{CSSpushforward}C^{d+\bullet}_{S^1,\vir}(X\rel\partial;A)&\xrightarrow{s_\ast}C_{-\bullet-1}^{S^1}(E;A)
\end{align}
Analogously with \eqref{Epushforward1more}--\eqref{Fpushforward1}, there are natural maps (compatible with \eqref{CSStimesSS}):
\begin{align}
C^\bullet_{S^1,\vir}(-;A)_{IJ}\otimes C_{-\bullet}(E;A'\setminus A)&\to C^\bullet_{S^1,\vir}(-;A')_{I',J'}\\
C^\bullet_{S^1,\vir}(-;A)\otimes C_{-\bullet}(E;A'\setminus A)&\to C^\bullet_{S^1,\vir}(-;A')
\end{align}
\end{definition}

\subsection{Isomorphisms \texorpdfstring{$H^\bullet_{S^1,\vir}(X;A)=\cH^\bullet(X/S^1;\pi_\ast\oo_X)$}{H\_S{\textasciicircum}1,vir(X;A)=H(X/S{\textasciicircum}1,pi o\_X)} (also \texorpdfstring{$\rel\partial$}{rel d})}

\begin{lemma}[$C^\bullet_{S^1,\vir}(-;A)_{IJ}$ are pure homotopy $\K$-sheaves]\label{CSTpureSS}
Let $X$ be an almost free $S^1$-space with finite locally $S^1$-orientable $S^1$-equivariant implicit atlas with boundary $A$.  Then $C^\bullet_{S^1,\vir}(-;A)_{IJ}$ and $C^\bullet_{S^1,\vir}(-\rel\partial;A)_{IJ}$ are pure homotopy $\K$-sheaves on $X/S^1$.  Furthermore, there are canonical isomorphisms of sheaves on $X/S^1$:
\begin{align}
H^0_{S^1,\vir}(-;A)_{IJ}&=\pi_\ast j_!j^\ast\oo_X\\
H^0_{S^1,\vir}(-\rel\partial;A)_{IJ}&=\pi_\ast j_!j^\ast\oo_{X\rel\partial}
\end{align}
where $j:V_I\cap V_J\hookrightarrow X$ for $V_I:=\psi_{\varnothing I}((s_I|X_I^\reg)^{-1}(0))\subseteq X$ and $\pi:X\to X/S^1$.
\end{lemma}

\begin{proof}
As in the proof of Lemma \ref{CSTpure}, we use Lemmas \ref{singularchainshomotopyKsheaf} and \ref{associatedgradedhomotopysheaf} to see that $C^\bullet_{S^1,\vir}(-;A)_{IJ}$ and $C^\bullet_{S^1,\vir}(-\rel\partial;A)_{IJ}$ are homotopy $\K$-sheaves.  More precisely, in the present context we need the statement of Lemma \ref{singularchainshomotopyKsheaf} for $C_\bullet^{S^1}$ in place of $C_\bullet$.  The proof of this lemma applies equally well to $C_\bullet^{S^1}$, noting that since $C_\bullet^{S^1}$ is a model of $S^1$-equivariant chains, it in particular satisfies Mayer--Vietoris in that following (total complex) is acyclic:
\begin{equation}
C_\bullet^{S^1}(U\cap U')\to C_\bullet^{S^1}(U)\oplus C_\bullet^{S^1}(U')\to C_\bullet^{S^1}(U\cup U')
\end{equation}

Now let us show purity.  By Lemma \ref{openclosedpurity}, it suffices to show that the restrictions of $C^\bullet_{S^1,\vir}(-;A)_{IJ}$ and $C^\bullet_{S^1,\vir}(-\rel\partial;A)_{IJ}$ to $(V_I\cap V_J)/S^1$ and to $X/S^1\setminus(V_I\cap V_J)/S^1$ are pure.  The latter restriction is trivially pure, since both complexes are simply zero for $K\cap V_I\cap V_J=\varnothing$.  Hence it suffices to show that the restrictions of both complexes to $(V_I\cap V_J)/S^1$ are pure.

We consider the comparison maps:
\begin{align}
\label{HSTinjective}H^\bullet_{S^1,\vir}(K;A)_{IJ}&\to H^\bullet_\vir(K;A)_{IJ}\\
\label{HSTinjectiverel}H^\bullet_{S^1,\vir}(K\rel\partial;A)_{IJ}&\to H^\bullet_\vir(K\rel\partial;A)_{IJ}
\end{align}
These are Gysin maps $\pi^!$, and hence by the Gysin sequence, their kernels coincide (respectively) with the images of:
\begin{align}
\label{HSTzero}H^{\bullet-2}_{S^1,\vir}(K;A)_{IJ}&\xrightarrow{\cap e}H^\bullet_{S^1,\vir}(K;A)_{IJ}\\
\label{HSTzerorel}H^{\bullet-2}_{S^1,\vir}(K\rel\partial;A)_{IJ}&\xrightarrow{\cap e}H^\bullet_{S^1,\vir}(K\rel\partial;A)_{IJ}
\end{align}
where $e$ is (the pullback of) $e\in H^2_{S^1}(X_J^\reg/\Gamma_J)$.  Pick any $p\in V_J\subseteq X$.  By Lemma \ref{freeactionislocallytrivial}, this $e$ vanishes when restricted to small $S^1$-invariant compact neighborhoods of $S^1p\subseteq V_J=(s_J|X_J^\reg)^{-1}(0)/\Gamma_J\subseteq X_J^\reg/\Gamma_J$.  It follows that \eqref{HSTzero}--\eqref{HSTzerorel} vanish for small $S^1$-invariant compact neighborhoods $K$ of $S^1p\subseteq V_J$.  Hence \eqref{HSTinjective}--\eqref{HSTinjectiverel} are injective for such $K$.  Now, using this injectivity and the fact that $C^\bullet_\vir(-;A)_{IJ}$ and $C^\bullet_\vir(-\rel\partial;A)_{IJ}$ are pure homotopy $\K$-sheaves (Lemma \ref{CSTpure}), it follows that $C^\bullet_{S^1,\vir}(-;A)_{IJ}$ and $C^\bullet_{S^1,\vir}(-\rel\partial;A)_{IJ}$ are pure on $V_J/S^1$, and hence on all of $X/S^1$.

It remains to identify $H^0_{S^1,\vir}(-;A)_{IJ}$ and $H^0_{S^1,\vir}(-\rel\partial;A)_{IJ}$.  Consider the comparison maps:
\begin{align}
H^0_{S^1,\vir}(-;A)_{IJ}&\to\pi_\ast H^0_\vir(-;A)_{IJ}\\
H^0_{S^1,\vir}(-\rel\partial;A)_{IJ}&\to\pi_\ast H^0_\vir(-\rel\partial;A)_{IJ}
\end{align}
which are maps of sheaves on $X/S^1$.  It suffices (by Lemma \ref{CSTpure}) to show that these are isomorphisms (which we will check on stalks, i.e.\ for $K=S^1p$).  Now this is just a calculation, similar to that in the proof of Lemma \ref{CSTpure}, except using an $S^1$-equivariant version of Poincar\'e--Lefschetz duality based on Lemma \ref{SSmanifoldPDbase} in place of Lemma \ref{poincareduality} (and it is in this calculation where we use the local $S^1$-orientability of $A$).
\end{proof}

\begin{proposition}[$C^\bullet_{S^1,\vir}(-;A)$ are pure homotopy $\K$-sheaves]\label{CpureSS}
Let $X$ be an almost free $S^1$-space with finite locally $S^1$-orientable $S^1$-equivariant implicit atlas with boundary $A$.  Then $C^\bullet_{S^1,\vir}(-;A)$ and $C^\bullet_{S^1,\vir}(-\rel\partial;A)$ are pure homotopy $\K$-sheaves on $X/S^1$.  Furthermore, there are canonical isomorphisms of sheaves on $X/S^1$:
\begin{align}
H^0_{S^1,\vir}(-;A)&=\pi_\ast\oo_X\\
H^0_{S^1,\vir}(-\rel\partial;A)&=\pi_\ast\oo_{X\rel\partial}
\end{align}
where $\pi:X\to X/S^1$.
\end{proposition}

\begin{proof}
This is exactly analogous to Proposition \ref{Cpure} and has the same proof (using Lemma \ref{CSTpureSS} in place of Lemma \ref{CSTpure}).
\end{proof}

\begin{theorem}[Calculation of $H^\bullet_{S^1,\vir}$]\label{Slocalfundamentaliso}
Let $X$ be an almost free $S^1$-space with finite locally $S^1$-orientable $S^1$-equivariant implicit atlas with boundary $A$.  Then there are canonical isomorphisms fitting (as the top horizontal maps) into commutative diagrams:
\begin{equation*}
\begin{CD}
H^\bullet_{S^1,\vir}(X;A)@>\sim>>\cH^\bullet(X/S^1,\pi_\ast\oo_X)\cr
@V\eqref{CSStimesSS}VV@V\pi^\ast VV\cr
H^\bullet_\vir(X;A)@>\text{Thm \ref{Fsheaffundamentaliso}}>>\cH^\bullet(X,\oo_X)
\end{CD}
\qquad
\begin{CD}
H^\bullet_{S^1,\vir}(X\rel\partial;A)@>\sim>>\cH^\bullet(X/S^1,\pi_\ast\oo_{X\rel\partial})\cr
@V\eqref{CSStimesSS}VV@V\pi^\ast VV\cr
H^\bullet_\vir(X\rel\partial;A)@>\text{Thm \ref{Fsheaffundamentaliso}}>>\cH^\bullet(X,\oo_{X\rel\partial})
\end{CD}
\end{equation*}
\end{theorem}

\begin{proof}
Consider the following diagram:
\begin{equation*}
\begin{tikzcd}
H^\bullet_{S^1,\vir}(X;A)\ar{d}{\sim}\ar{rr}{\eqref{CSStimesSS}}&&H^\bullet_\vir(X;A)\ar{d}{\sim}\\
\cH^\bullet(X/S^1;C^\bullet_{S^1,\vir}(-;A))\ar{r}{\eqref{CSStimesSS}}\ar{d}{\sim}&\cH^\bullet(X/S^1,\pi_\ast C^\bullet_\vir(-;A))\ar{r}{\pi^\ast}\ar{d}{\sim}&\cH^\bullet(X;C^\bullet_\vir(-;A))\ar{d}{\sim}\\
\cH^\bullet(X/S^1;\tau_{\geq 0}C^\bullet_{S^1,\vir}(-;A))\ar{r}{\eqref{CSStimesSS}}&\cH^\bullet(X/S^1,\pi_\ast\tau_{\geq 0}C^\bullet_\vir(-;A))\ar{r}{\pi^\ast}&\cH^\bullet(X;\tau_{\geq 0}C^\bullet_\vir(-;A))\\
\cH^\bullet(X/S^1;H^0_{S^1,\vir}(-;A))\ar{u}[swap]{\sim}\ar{r}{\eqref{CSStimesSS}}\ar{d}{\sim}&\cH^\bullet(X/S^1,\pi_\ast H^0_\vir(-;A))\ar{r}{\pi^\ast}\ar{u}\ar{d}{\sim}&\cH^\bullet(X;H^0_\vir(-;A))\ar{u}[swap]{\sim}\ar{d}{\sim}\\
\cH^\bullet(X/S^1;\pi_\ast\oo_X)\ar[equal]{r}&\cH^\bullet(X/S^1,\pi_\ast\oo_X)\ar{r}{\pi^\ast}&\cH^\bullet(X;\oo_X)
\end{tikzcd}
\end{equation*}
where the vertical isomorphisms are by Propositions \ref{CpureSS} and \ref{Cpure} (see also \eqref{isosforpurehtpysheafcech}).

Now each small square of this diagram above commutes (only the bottom left requires an argument; one shows that the corresponding diagram of sheaves commutes by checking it locally, where it is just a calculation).  Now the first square in Theorem \ref{Slocalfundamentaliso} is the same as the big square above, which commutes by a diagram chase (for which one should be careful of the fact that one of the vertical arrows is not an isomorphism).

An identical argument applies to the second square in Theorem \ref{Slocalfundamentaliso}.
\end{proof}

\subsection{Localization for virtual fundamental classes}\label{fclocalizationsection}

\begin{theorem}[{$S^1$-localization for $[X]^\vir$}]\label{fclocalization}
Let $X$ be an almost free $S^1$-space with locally $S^1$-orientable $S^1$-equivariant implicit atlas with boundary $A$.  Then $\pi_\ast[X]^\vir=0$, where $\pi_\ast=(\pi^\ast)^\vee:\cH^\bullet(X;\oo_{X\rel\partial})^\vee\to\cH^\bullet(X/S^1;\pi_\ast\oo_{X\rel\partial})^\vee$.
\end{theorem}

\begin{proof}
We may assume that $A$ is finite.  Now by Theorem \ref{Slocalfundamentaliso}, we have a commutative diagram:
\begin{equation*}
\begin{tikzcd}
\cH^{d+\bullet}(X/S^1;\pi_\ast\oo_{X\rel\partial})\ar{d}{\pi^\ast}\ar[equal]{r}{\text{Thm \ref{Slocalfundamentaliso}}}&H^{d+\bullet}_{S^1,\vir}(X\rel\partial;A)\ar{d}{\eqref{CSStimesSS}}\ar{r}{\eqref{CSSpushforward}}&H_{-\bullet-1}^{S^1}(E;A)\ar{d}{\pi^!}\\
\cH^{d+\bullet}(X;\oo_{X\rel\partial})\ar[equal]{r}{\text{Thm \ref{Fsheaffundamentaliso}}}&H^{d+\bullet}_\vir(X\rel\partial;A)\ar{r}{\eqref{Fpushforward}}&H_{-\bullet}(E;A)\ar{r}{[E_A]\mapsto 1}&R
\end{tikzcd}
\end{equation*}
By definition, $[X]^\vir$ is the total composition of the bottom row for $\bullet=0$.  Now the desired statement follows since the rightmost vertical map is zero (by Lemma \ref{trivialactiongysin}, because $S^1$ acts trivially on $E_A$).
\end{proof}

\begin{corollary}
Let $X$ be an almost free $S^1$-space with locally $S^1$-orientable $S^1$-equivariant implicit atlas with boundary $A$ of dimension $0$.  Then $[X]^\vir=0$.
\end{corollary}

\begin{proof}
This follows from Theorem \ref{fclocalization} since $\pi^\ast:\cH^0(X/S^1;\pi_\ast\oo_{X\rel\partial})\to\cH^0(X;\oo_{X\rel\partial})$ is an isomorphism (since $S^1$ is connected and $A$ is locally $S^1$-orientable).
\end{proof}

\begin{remark}\label{SSequivariantvfcmaybe}
We expect that the machinery of this section can be used to define an $S^1$-equivariant virtual fundamental cycle $[X]^{S^1,\vir}\in\cH^{d-1}(X/S^1;\pi_\ast\oo_{X\rel\partial})^\vee\xrightarrow{(\pi_!)^\vee}\cH^d(X;\oo_{X\rel\partial})^\vee$ lifting $[X]^\vir$, via the diagram:
\begin{equation*}
\begin{tikzcd}
\cH^{d+\bullet}(X;\oo_{X\rel\partial})\ar{d}{\pi_!}\ar[equal]{r}{\text{Thm \ref{Fsheaffundamentaliso}}}&H^{d+\bullet}_\vir(X\rel\partial;A)\ar{d}\ar{r}{\eqref{Fpushforward}}&H_{-\bullet}(E;A)\ar{d}{\pi_\ast}\ar{r}{[E_A]\mapsto 1}&R\ar[equal]{d}\\
\cH^{d+\bullet-1}(X/S^1;\pi_\ast\oo_{X\rel\partial})\ar[equal]{r}{\text{Thm \ref{Slocalfundamentaliso}}}&H^{d+\bullet-1}_{S^1,\vir}(X\rel\partial;A)\ar{r}{\eqref{CSSpushforward}}&H_{-\bullet}^{S^1}(E;A)\ar{r}{[E_A]\mapsto 1}&R
\end{tikzcd}
\end{equation*}
where the second vertical map is induced by pushforward $\pi_\ast$ on chains.  Note that we would need to show that the first square commutes.  This would provide another proof of Theorem \ref{fclocalization}.  This should also be applicable without any restriction on the $S^1$-action on $X$ (provided we use $S^1$-equivariant (co)homology in the appropriate places).
\end{remark}

\subsection{Localization for homology}\label{homologylocalizationsection}

To prove the desired localization result for Floer-type homology groups, the chain models used in \S\ref{SSCconstruct} are inadequate.  Specifically, the fact that we need to consider product maps $\X(p,q)\times\X(q,r)\to\X(p,r)$ between flow spaces with an $S^1$-action means that we need a corresponding map $C_\bullet^{S^1}(X)\times C_\bullet^{S^1}(Y)\to C_{\bullet+1}^{S^1}(X\times Y)$ on $S^1$-equivariant chains.  To obtain such a map, we will end up using a different models of chains for every $(\sigma,p,q)$.  We first introduce the models of chains we will use, we then describe how to use these chain models for the constructions of \S\ref{homologygroupssection}, and finally we prove the localization result for homology groups.

\begin{remark}
If we used the language of $\infty$-categories, there would be no need to construct models of chains with good (chain level) functoriality properties (c.f.\ Remark \ref{independenceofchainmodel}).
\end{remark}

\subsubsection{MF-sets and $S^1$-MF-sets}

\begin{definition}
A \emph{PL manifold with cells} $(M,\SSS)$ is a compact connected nonempty orientable PL-manifold with boundary $M$ with stratification by a finite poset $\SSS$ (see Definition \ref{stratificationdefinition}) such that:
\begin{rlist}
\item$M^{\leq\s}$ is a connected nonempty orientable PL-submanifold with boundary.
\item$(M^{\leq\s})^\circ=M^\s$.
\item$\SSS$ has a unique maximal element.
\end{rlist}
For every $\s\in\SSS$, there corresponds a \emph{face} $(M^{\leq\s},\SSS^{\leq\s})$, which is also a PL manifold with cells.
\end{definition}

\begin{definition}
An \emph{MF-set} is a set $\Y$ along with:
\begin{rlist}
\item For every $i\in\Y$, a PL manifold with cells $(M_i,\SSS_i)$.
\item(Face identifications). For every $i\in\Y$ and every $\s\in\SSS_i$, an index $j\in\Y$ and an isomorphism $(M_i^{\leq\s},\SSS_i^{\leq\s})\xrightarrow\sim(M_j,\SSS_j)$.  These indices and isomorphisms must be (strictly) transitive in the following obvious sense.  If $\s\in\SSS$ is maximal, then $j=i$ and $(M_i,\SSS_i)=(M_i^{\leq\s},\SSS_i^{\leq\s})\xrightarrow\sim(M_j,\SSS_j)=(M_i,\SSS_i)$ is the identity map.  If $\s'\preceq\s$ is a nested pair of faces with identifications $(M_i^{\leq\s},\SSS_i^{\leq\s})\xrightarrow\sim(M_j,\SSS_j)$ and $(M_i^{\leq\s'},\SSS_i^{\leq\s'})\xrightarrow\sim(M_{j'},\SSS_{j'})$, and the $\s'$ face of $(M_j,\SSS_j)$ is identified $(M_j^{\leq\s'},\SSS_j^{\leq\s'})\xrightarrow\sim(M_k,\SSS_k)$, then $k=j'$ and the two identifications of $(M_i^{\leq\s'},\SSS_i^{\leq\s'})$ with $(M_k,\SSS_k)=(M_{j'},\SSS_{j'})$ are the same.
\end{rlist}
A morphism of MF-sets $f:\Y\to\Y'$ is a map of sets covered by isomorphisms $(M_i,\SSS_i)\to(M_{f(i)},\SSS_{f(i)})$, compatible with the face identifications for $\Y$ and $\Y'$ in the obvious way.

The category of MF-sets has a natural symmetric monoidal structure: given MF-sets $\{M_i\}_{i\in\Y}$ and $\{M_j'\}_{j\in\Y'}$, their product is defined to be $\{M_i\times M_j'\}_{(i,j)\in\Y\times\Y'}$, which is again an MF-set.

For an MF-set $\Y$, let $C^\Y_\bullet(X)$ denote the complex freely generated\footnote{A given map $M_i\to X$ contributes a copy of the orientation module of $M_i$, which is isomorphic to $\ZZ$ but not canonically so.} by maps $M_i\to X$ (for $i\in\Y$), with differential given by the obvious sum over all codimension one faces.  Note that there is a natural map:
\begin{equation}
C_\bullet^\Y(X)\otimes\C_\bullet^{\Y'}(X')\to C_\bullet^{\Y\times\Y'}(X\times X')
\end{equation}
\end{definition}

\begin{remark}
It would perhaps be more natural to work with MF-sets in the DIFF category, although in that case it is not clear precisely what sort of stratifications and corner structure one should allow so that the proof of Lemma \ref{saturatedcalculatesH} goes through.
\end{remark}

\begin{example}\label{standardsimplicialMFset}
The collection of standard simplices $\{\Delta^n\}_{n\geq 0}$ equipped with their standard simplicial stratifications and the standard identifications of the faces of $\Delta^n$ with the various $\Delta^i$, forms an MF-set.
\end{example}

\begin{definition}\label{saturateddef}
An MF-set $\Y$ is called \emph{saturated} iff for every PL manifold with cells $(M,\SSS)$ along with, for every $\s\in\SSS$ of positive codimension, an index $j\in\Y$ and an isomorphism $(M_j,\SSS_j)\xrightarrow\sim(M^{\leq\s},\SSS^{\leq\s})$, such that these indices and isomorphisms are (strictly) transitive in the obvious sense, there exists $i\in\Y$ and an isomorphism $(M_i,\SSS_i)\xrightarrow\sim(M,\SSS)$ respecting these given face identifications.
\end{definition}

Every MF-set $\Y$ embeds into a saturated MF-set $\Y_\infty$, which may be constructed (non-canonically) as follows.  We define a sequence of inclusions $\Y=\Y_{-1}\hookrightarrow\Y_0\hookrightarrow\Y_1\hookrightarrow\cdots$, and we let $\Y_\infty:=\varinjlim\Y_n$.  To define $\Y_n$, we consider all PL manifolds with cells of dimension $n$ with face identifications to elements of $\Y_{n-1}$ as in Definition \ref{saturateddef}.  The collection of isomorphism classes of such data (manifold along with face identifications) forms a set, so we may choose (non-canonically) a set $\Z_n$ parameterizing all of them.  We then set $\Y_n:=\Y_{n-1}\sqcup\Z_n$.  Now it is easy to check that $\Y_\infty$ is saturated.

\begin{lemma}\label{saturatedcalculatesH}
Let $\Y$ be a saturated MF-set.  Then there is a canonical isomorphism between singular homology $H_\bullet(X)$ and $H^\Y_\bullet(X)$ (the homology of $C^\Y_\bullet(X)$).
\end{lemma}

\begin{proof}
Since $\Y$ is saturated, there exists a morphism of MF-sets $\{\Delta^n\}_{n\geq 0}\to\{M_i\}_{i\in\Y}$ (where $\{\Delta^n\}_{n\geq 0}$ is as in Example \ref{standardsimplicialMFset}).  Since $C_\bullet(X)=C^{\{\Delta^n\}_{n\geq 0}}_\bullet(X)$ by definition, we obtain a chain map $C_\bullet(X)\to C_\bullet^\Y(X)$.  The resulting map:
\begin{equation}\label{homologyYcomparison}
H_\bullet(X)\to H_\bullet^\Y(X)
\end{equation}
is independent of the choice of morphism $\{\Delta^n\}_{n\geq 0}\to\{M_i\}_{i\in\Y}$, as can be seen as follows.  There is an MF-set $\{\Delta^n\}_{n\geq 0}\sqcup\{\Delta^n\times[0,1]\}_{n\geq 0}\sqcup\{\Delta^n\}_{n\geq 0}$ where each $\{\Delta^n\}_{n\geq 0}$ is as in Example \ref{standardsimplicialMFset}, where $\Delta^n\times[0,1]$ is given the product stratification ($[0,1]$ stratified by $\{\{0\},\{1\},(0,1)\}$), and $\Delta^n\times\{0\}$ (resp.\ $\Delta^n\times\{1\}$) is identified with the first (resp.\ second) copy of $\Delta^n$.  Since $\Y$ is saturated, it follows that for any pair of morphisms $\{\Delta^n\}_{n\geq 0}\to\{M_i\}_{i\in\Y}$ there is a morphism $\{\Delta^n\}_{n\geq 0}\sqcup\{\Delta^n\times[0,1]\}_{n\geq 0}\sqcup\{\Delta^n\}_{n\geq 0}\to\{M_i\}_{i\in\Y}$ whose restriction to the two copies of $\{\Delta^n\}_{n\geq 0}$ are the two given morphisms.  From this data one easily constructs a chain homotopy between the two maps $C_\bullet(X)\to C_\bullet^\Y(X)$.  Hence the map \eqref{homologyYcomparison} is canonical.

Now let us show that \eqref{homologyYcomparison} is an isomorphism.  Fix a map $\{\Delta^n\}_{n\geq 0}\to\{M_i\}_{i\in\Y}$.  Fix triangulations $T_i$ of $M_i$ for which each $M_i^{\leq\s}$ is a union of simplices, which are compatible with the face identifications, and which restrict to the tautological triangulation of $\{\Delta^n\}_{n\geq 0}$.  Such triangulations may be constructed by induction.  By triangulation, we mean a triangulation in which each simplex is equipped with a total order on its set of vertices, compatible with its faces (i.e.\ a semisimplicial set rather than a simplicial complex).  Such triangluations induce a map of complexes $C_\bullet^\Y(X)\to C_\bullet(X)$, and the composition $C_\bullet(X)\to C_\bullet^\Y(X)\to C_\bullet(X)$ is clearly the identity map.  It suffices to show that the other composition is chain homotopic to the identity map.

Now let us define a new MF-set in terms of the map $\{\Delta^n\}_{n\geq 0}\to\{M_i\}_{i\in\Y}$ and the triangulations $T_i$.  Note that the stratification of $M_i$ by the face poset $\F(T_i)$ of $T_i$ refines the stratification by $\SSS_i$; in other words we have maps $M_i\to\F(T_i)\to\SSS_i$.  The objects of the new MF-set are $\{M_i\}_{i\in\Y}\sqcup\{M_i\times[0,1]\}_{i\in\Y}$, where $M_i\times[0,1]$ is stratified by the following strata:
\begin{align}
M_i^\s\times\{0\}&\quad\s\in\SSS_i\\
M_i^\s\times(0,1)&\quad\s\in\SSS_i\\
M_i^\ttt\times\{1\}&\quad\ttt\in\F(T_i)
\end{align}
Thus the poset of strata of $M_i\times[0,1]$ is $\SSS_i\sqcup\SSS_i\sqcup\F(T_i)$.  For the face identifications, $M_i\times\{0\}$ is identified tautologically with $M_i$, and each of the closed strata $M_i^{\leq\ttt}\times\{1\}$ ($\ttt\in\F(T_i)$) is identified with the corresponding $\Delta^n$ (considered as an object of $\{M_i\}_{i\in\Y}$ via the inclusion $\{\Delta^n\}_{n\geq 0}\to\{M_i\}_{i\in\Y}$).

Now we have a chain of inclusions $\{\Delta^n\}_{n\geq 0}\hookrightarrow\{M_i\}_{i\in\Y}\hookrightarrow\{M_i\}_{i\in\Y}\sqcup\{M_i\times[0,1]\}_{i\in\Y}$.  Since $\Y$ is saturated, there in fact exists a map backwards $\{M_i\}_{i\in\Y}\sqcup\{M_i\times[0,1]\}_{i\in\Y}\to\{M_i\}_{i\in\Y}$ (acting identically on $\{M_i\}_{i\in\Y}$).  Using this map, we obtain a chain map $C_\bullet^\Y(X)\to C_{\bullet+1}^\Y(X)$ (precompose chains with the projection $M_i\times[0,1]\to M_i$).  This map is a chain homotopy between the identity map and the composition $C_\bullet^\Y(X)\to C_\bullet(X)\to C_\bullet^\Y(X)$, so we are done.
\end{proof}

Equip $S^1$ with its standard PL structure, for which the group operations are PL.

\begin{definition}\label{SSMFdef}
An \emph{$S^1$-MF-set} is an MF-set $\Y$ along with a principal $S^1$-bundle $(ES^1)_\Y\to(BS^1)_\Y$ where $(BS^1)_\Y$ (and hence also $(ES^1)_\Y$) is an increasing union of compact polyhedra (we do not require $(ES^1)_\Y$ to be contractible), and for every $i\in\Y$, a pullback diagram:
\begin{equation}\label{SSMFstructuremaps}
\begin{tikzcd}
M_i\ar{r}\ar{d}&(ES^1)_\Y\ar{d}\\
N_i\ar{r}&(BS^1)_\Y
\end{tikzcd}
\end{equation}
with PL maps where $N_i$ (and hence also $M_i$) is a PL manifold, and where the stratification on $M_i$ is pulled back from $N_i$, such that these diagrams are compatible with the face identifications.  A map of $S^1$-MF-sets $\Y\to\Y'$ is said to be injective iff it is injective as a map of sets and the map $(ES^1)_\Y\to(ES^1)_{\Y'}$ is injective.

There is a forgetful functor from $S^1$-MF-sets to MF-sets, where we remember $M_i$ (of course, there is another natural forgetful functor remembering $N_i$, though we will never use it).  When speaking of a morphism $\Y\to\Y'$ where $\Y$ is an $S^1$-MF-set and $\Y'$ is an MF-set, we implicitly apply the forgetful functor to $\Y$.

The category of $S^1$-MF-sets has a natural symmetric monoidal structure: given $S^1$-MF-sets $\{M_i\}_{i\in\Y}$ and $\{M_j'\}_{j\in\Y'}$, we may define their product $\{M_i\times M_j'\}_{(i,j)\in\Y\times\Y'}$, which is again an $S^1$-MF-set, via the diagonal $S^1$-action on $M_i\times M_j'$ and $(ES^1)_{\Y\times\Y'}:=(ES^1)_\Y\times(ES^1)_{\Y'}$ with the diagonal action.  The forgetful functor from $S^1$-MF-sets to MF-sets is clearly a symmetric monoidal functor.  It also makes sense to take the product of an $S^1$-MF-set and an MF-set.

For an $S^1$-MF-set $\Y$, let $C^{S^1,\Y}_\bullet(X)$ (for $X$ an $S^1$-space) denote the complex generated by commuting diagrams of $S^1$-equivariant maps:
\begin{equation}\label{SSchaindiagram}
\begin{tikzcd}
M_i\ar{r}\ar{d}&X\times(ES^1)_\Y\ar{r}\ar{d}&(ES^1)_\Y\ar{d}\\
N_i\ar{r}&(X\times(ES^1)_\Y)/S^1\ar{r}&(BS^1)_\Y
\end{tikzcd}
\end{equation}
($i\in\Y$), where the outer square coincides with \eqref{SSMFstructuremaps}, with differential given by the obvious sum over all codimension one faces.  A generator \eqref{SSchaindiagram} resides in degree $\dim N_i$.  Note that there are natural compatible maps:
\begin{align}
\C_\bullet^{S^1,\Y}(X')&\to C_{\bullet+1}^\Y(X)\\
C_\bullet^\Y(X)\otimes\C_\bullet^{S^1,\Y'}(X')&\to C_\bullet^{S^1,\Y\times\Y'}(X\times X')\\
C_\bullet^{S^1,\Y}(X)\otimes\C_\bullet^{S^1,\Y'}(X')&\to C_{\bullet+1}^{S^1,\Y\times\Y'}(X\times X')
\end{align}
\end{definition}

\begin{definition}
An $S^1$-MF-set $\Y$ is called \emph{saturated} iff $(ES^1)_\Y$ is contractible and every pullback diagram:
\begin{equation}
\begin{tikzcd}
M\ar{r}\ar{d}&(ES^1)_\Y\ar{d}\\
N\ar{r}&(BS^1)_\Y
\end{tikzcd}
\end{equation}
where $N$ (and thus $M$) is a PL manifold with cells, along with strictly transitive identifications of the faces of positive codimension with elements of $\Y$ (as in Definition \ref{SSMFdef}), is isomorphic to some $i\in\Y$.

Note that the notions of saturation for MF-sets and $S^1$-MF-sets are different: an $S^1$-MF-set is \emph{never} saturated as an MF-set.
\end{definition}

Every $S^1$-MF-set $\Y$ embeds into a saturated $S^1$-MF-set $\Y_\infty$, which may be constructed (again, non-canonically) by first embedding $(ES^1)_\Y$ into something contractible, and then proceeding as in the case of MF-sets.

\begin{lemma}\label{saturatedcalculatesHSS}
Let $\Y$ be a saturated $S^1$-MF-set.  Then there is a canonical isomorphism between $S^1$-equivariant singular homology $H^{S^1}_\bullet(X)$ and $H^{S^1,\Y}_\bullet(X)$ (the homology of $C^{S^1,\Y}_\bullet(X)$).  Furthermore, the natural map $C_\bullet^{S^1,\Y}(X)\to C_{\bullet+1}^{\Y'}(X)$ induces the Gysin map on homology, for $\Y'$ a saturated MF-set.
\end{lemma}

\begin{proof}
All $S^1$-MF-sets in this proof will share the same $(ES^1)_\Y\to(BS^1)_\Y$, so we will omit the subscript $_\Y$ from the notation.

Let $\{\Delta^n\}_{n\geq 0,f:\Delta^n\to BS^1}$ denote the $S^1$-MF-set indexed by pairs $(n,f)$ consisting of an integer $n\geq 0$ and a PL map $f:\Delta^n\to BS^1$, where the stratifications and face identifications are as in Example \ref{standardsimplicialMFset}.  Now the complex $C^{S^1,\{\Delta^n\}_{n\geq 0,f:\Delta^n\to BS^1}}_\bullet(X)$ is freely generated by maps $\Delta^n\to(X\times ES^1)/S^1$ whose composition with the projection $(X\times ES^1)/S^1\to BS^1$ is PL.  A straightforward approximation argument shows that the inclusion of this complex into the complex generated by all maps $\Delta^n\to(X\times ES^1)/S^1$ is a quasi-isomorphism.  Hence there is a canonical isomorphism:
\begin{equation}
H^{S^1,\{\Delta^n\}_{n\geq 0,f:\Delta^n\to BS^1}}_\bullet(X)\xrightarrow\sim H_\bullet((X\times ES^1)/S^1)=H_\bullet^{S^1}(X)
\end{equation}

As in the proof of Lemma \ref{saturatedcalculatesH} there exists a morphism of $S^1$-MF-sets $\{\Delta^n\}_{n\geq 0,f:\Delta^n\to BS^1}\to\Y$ since $\Y$ is saturated.  This induces a chain map $C^{S^1,\{\Delta^n\}_{n\geq 0,f:\Delta^n\to BS^1}}_\bullet(X)\to C^{S^1,\Y}_\bullet(X)$, which as before induces a canonical map on homology:
\begin{equation}
H^{S^1,\{\Delta^n\}_{n\geq 0,f:\Delta^n\to BS^1}}_\bullet(X)\to H^{S^1,\Y}_\bullet(X)
\end{equation}
which is independent of the choice of morphism $\{\Delta^n\}_{n\geq 0,f:\Delta^n\to BS^1}\to\Y$.  It suffices to show that this map is an isomorphism.

Fix a map $\{\Delta^n\}_{n\geq 0,f:\Delta^n\to BS^1}\to\Y$, which we observe is necessarily injective.  Fix triangulations $T_i$ of $M_i$ as in the proof of Lemma \ref{saturatedcalculatesH}.  Such triangulations induce a map of complexes $C^{S^1,\Y}_\bullet(X)\to C^{S^1,\{\Delta^n\}_{n\geq 0,f:\Delta^n\to BS^1}}_\bullet(X)$.  The composition $C^{S^1,\{\Delta^n\}_{n\geq 0,f:\Delta^n\to BS^1}}_\bullet(X)\to C^{S^1,\Y}_\bullet(X)\to C^{S^1,\{\Delta^n\}_{n\geq 0,f:\Delta^n\to BS^1}}_\bullet(X)$ is clearly the identity map.  It thus suffices to show that the reverse composition is chain homotopic to the identity.

As in the proof of Lemma \ref{saturatedcalculatesH}, we construct a new $S^1$-MF-set, namely $\{M_i\}_{i\in\Y}\sqcup\{M_i\times[0,1]\}_{i\in\Y}$, from the inclusion $\{\Delta^n\}_{n\geq 0,f:\Delta^n\to BS^1}\hookrightarrow\{M_i\}_{i\in\Y}$ and the triangulations $T_i$.  There are inclusions $\{\Delta^n\}_{n\geq 0,f:\Delta^n\to BS^1}\hookrightarrow\{M_i\}_{i\in\Y}\hookrightarrow\{M_i\}_{i\in\Y}\sqcup\{M_i\times[0,1]\}_{i\in\Y}$.  The desired chain homotopy may then be constructed as in the proof of Lemma \ref{saturatedcalculatesH}.

The fact that the induced map $H^{S^1}_\bullet(X)\to H_{\bullet+1}(X)$ is the Gysin map is left to the reader.
\end{proof}

\subsubsection{$\sF$-modules valued in MF-sets and $S^1$-MF-sets}

Plugging the category of MF-sets into Definition \ref{flowobjectdef}, we may talk about $\sF$-modules valued in MF-sets.  In other words, an $\sF$-module MF-set is a collection of MF-sets $\Y(\sigma,p,q)$ along with product/face maps:
\begin{align}
\Y(\sigma|[0\ldots\hat k\ldots n],p,q)&\to\Y(\sigma,p,q)\\
\Y(\sigma|[0\ldots k],p,q)\times\Y(\sigma|[k\ldots n],q,r)&\to\Y(\sigma,p,r)
\end{align}
which are compatible in the sense of Definition \ref{flowobjectdef}.  Similarly, we may talk about $\sF$-modules of $S^1$-MF-sets.  We may also talk about morphisms from $\sF$-module $S^1$-MF-sets to $\sF$-module MF-sets using the forgetful functor described earlier.

\begin{proposition}\label{inductionformodels}
Let $\X/Z_\bullet$ be a flow category diagram.  There exists an $\sF$-module of MF-sets $\Y$ and an $\sF$-module of $S^1$-MF-sets $\Y_{S^1}$ (both supported inside $\supp\X$) along with a morphism $\Y_{S^1}\to\Y$, satisfying the following property.  For all $(\sigma,p,q)\in\supp\X$, both $\Y(\sigma,p,q)$ and $\Y_{S^1}(\sigma,p,q)$ are saturated, and the tautologous maps:
\begin{align}
\label{colimitYone}\colim_{\s\in\partial\SSS_\X(\sigma,p,q)}\Y(\sigma,p,q,\s)&\hookrightarrow\Y(\sigma,p,q)\\
\label{colimitYtwo}\colim_{\s\in\partial\SSS_\X(\sigma,p,q)}\Y_{S^1}(\sigma,p,q,\s)&\hookrightarrow\Y_{S^1}(\sigma,p,q)
\end{align}
are injective (and the colimits on the left exist).  Moreover, such $\Y_{S^1}\to\Y$ may be constructed by induction on $(\sigma,p,q)$, partially ordered by $\preceq_\X$ ($\preceq_\X$ is well-founded, see Definition \ref{cofibrantreplacements} or the proof of Proposition \ref{trivialKanfibration}).
\end{proposition}

Note the similarity of \eqref{colimitYone}--\eqref{colimitYtwo} with Definition \ref{Fmodulecofibrantdef}.  One should think of these conditions as being: ``$\Y$ and $\Y_{S^1}$ are both cofibrant''.

\begin{proof}
It suffices to perform the inductive step for a given $(\sigma,p,q)$.  First, we show that the colimits \eqref{colimitYone}--\eqref{colimitYtwo} exist, essentially by rewriting the proofs of Lemmas \ref{cofibrantinjectivelemma} and \ref{cofibranthaslotsofcofibrations} in the present context.  We actually show the more general statement that the colimits:
\begin{equation}\label{Tcolimit}
\colim_{\ttt\in\T}\Y(\sigma,p,q,\ttt)\qquad\colim_{\ttt\in\T}\Y_{S^1}(\sigma,p,q,\ttt)
\end{equation}
exist for any $\T\subseteq\partial\SSS_\X(\sigma,p,q)$ which is downward closed.  We show this by induction on the cardinality of $\T$.  If $\T=\varnothing$, existence is trivial.  For $\T$ nonempty, let $\ttt_0\in\T$ be any maximal element, and observe that \eqref{Tcolimit} can be written as the colimit of the diagram:
\begin{equation}\label{colimitdiagramnice}
\begin{tikzcd}[row sep = tiny]
&\Y(\sigma,p,q,\ttt_0)\\
\smash{\displaystyle\colim_{\begin{smallmatrix}\ttt\in\T\cr\ttt\prec\ttt_0\end{smallmatrix}}}\Y(\sigma,p,q,\ttt)\ar[hook]{ur}\ar{rd}\\
&\displaystyle\colim_{\ttt\in\T\setminus\ttt_0}\Y(\sigma,p,q,\ttt)
\end{tikzcd}
\end{equation}
Let us show that the top arrow is injective.  Write $\ttt_0$ in the form \eqref{SdataA}--\eqref{SdataB}.  Then we have:
\begin{equation}
\SSS_\X(\sigma,p,q)^{\leq\ttt_0}=\SSS_\X(\sigma|[j_{a_0}\ldots j_{a_1}],p_0,p_1)\times\cdots\times\SSS_\X(\sigma|[j_{a_{m-1}}\ldots j_{a_m}],p_{m-1},p_m)
\end{equation}
Now we consider the $m$-cubical diagram:
\begin{equation}
\prod_{i=1}^m\left[\colim_{\s\in\partial\SSS_\X(\sigma|[j_{a_{i-1}}\ldots j_{a_i}],p_{i-1},p_i)}\Y(\sigma|[j_{a_{i-1}}\ldots j_{a_i}],p_{i-1},p_i,\s)\hookrightarrow\Y(\sigma|[j_{a_{i-1}}\ldots j_{a_i}],p_{i-1},p_i)\right]
\end{equation}
Now the top arrow of \eqref{colimitdiagramnice} is simply the map to the maximal vertex of this $m$-cube from the colimit over the $m$-cube minus the maximal vertex (this is where we use the fact that $\T$ is downward closed).  This is injective for any cubical diagram of the form $\prod_{i=1}^m[A_i\hookrightarrow B_i]$ where each $A_i\hookrightarrow B_i$ is injective.  Now that the top arrow in \eqref{colimitdiagramnice} is injective, the existence of the colimit is clear (first take the colimit of underlying sets and then the rest of the structure extends in an obvious way), and hence the colimit in \eqref{Tcolimit} exists.  This reasoning applies equally well to $\Y_{S^1}$ (including $(ES^1)_\Y\to(BS^1)_\Y$).

Now that we have shown that the colimits on the left hand side of \eqref{colimitYone}--\eqref{colimitYtwo} exist, it remains to define $\Y(\sigma,p,q)$ and $\Y_{S^1}(\sigma,p,q)$.

We observed earlier that any $S^1$-MF-set injects into a saturated $S^1$-MF-set.  Let $\Y_{S^1}(\sigma,p,q)$ be any saturated $S^1$-MF-set with an injection \eqref{colimitYtwo}.  Note that with this definition, the product/face maps with target $\Y_{S^1}(\sigma,p,q)$ are obvious, as is their compatibility.

To define $\Y(\sigma,p,q)$, consider the following diagram:
\begin{equation}\label{biggluingcolimit}
\begin{tikzcd}
\displaystyle\colim_{\s\in\partial\SSS_\X(\sigma,p,q)}\Y(\sigma,p,q,\s)\\
\displaystyle\colim_{\s\in\partial\SSS_\X(\sigma,p,q)}\Y_{S^1}(\sigma,p,q,\s)\ar{u}\ar[hook]{r}&\Y_{S^1}(\sigma,p,q)
\end{tikzcd}
\end{equation}
The lower colimit is in the category of $S^1$-MF-sets.  It is clear from the inductive proof of its existence that this particular colimit commutes with the forgetful functor to MF-sets.  Hence we may equally well think of this colimit as taking place in the category of MF-sets, and with this perspective the definition of the vertical map is clear.  Now since the horizontal map is injective, the colimit of this diagram (in the category of MF-sets) clearly exists (first take the colimit in sets and then the rest of the structure is obvious).  We pick any saturated MF-set into which the colimit of \eqref{biggluingcolimit} embeds, and we call this saturated MF-set $\Y(\sigma,p,q)$.  Via the embedding of the upper colimit into $\Y(\sigma,p,q)$, it is tautological that the face/product maps with target $\Y(\sigma,p,q)$ exist and are compatible.  The map $\Y_{S^1}(\sigma,p,q)\to\Y(\sigma,p,q)$ is similarly tautological, as is the fact that it is compatible with the product/face maps.
\end{proof}

\subsubsection{Augmented virtual cochain complexes from $\sF$-module MF-sets and $S^1$-MF-sets}\label{SSaugmentedsec}

We now describe a modified version of the complexes from Definitions \ref{Gsheavesdef} and \ref{CEcomplexaugmented} and their $S^1$-equivariant versions using a choice of $\Y_{S^1}\to\Y$ as in Proposition \ref{inductionformodels}.

Let us give alternative definitions the complexes:
\begin{equation}\label{keycomplexes}
C^\bullet_\vir(\X\rel\partial;\A)^+(\sigma,p,q)\quad C^\bullet_\vir(\partial\X;\A)^+(\sigma,p,q)\quad C_\bullet(E;\A)^+(\sigma,p,q)
\end{equation}
The ``fixed fundamental cycles'' $[E_\alpha]\in C_\bullet(E;\alpha)$ still live in ordinary singular simplicial chains.  Now \eqref{keycomplexes} are defined in terms of the (relative) singular chains on certain spaces; the modification we make is just to use a different model of singular chains (depending on $(\sigma,p,q)$) which we now describe.  Our model of chains on a space $X$ is generated by the set of maps:
\begin{equation}\label{bigcomplex}
M_i\times\prod_{\alpha\in\coprod\Abar(\sigma',p',q')}\Delta^{i_\alpha}\to X
\end{equation}
where $i\in\Y(\sigma,p,q)$, $i_\alpha\in\ZZ_{\geq 0}$ (all but finitely many must be zero), and $\coprod\bar \A(\sigma',p',q')$ stands for:
\begin{equation}
\coprod_{\begin{smallmatrix}0\leq i_0<\cdots<i_m\leq n\cr(p',q')\in\PPP_{\sigma(i_0)}\times\PPP_{\sigma(i_m)}\cr\exists\ttt\in\SSS_\X(\sigma,p,q)\text{ containing }([i_0\ldots i_m],p',q')\end{smallmatrix}}\Abar(\sigma|[i_0\ldots i_m],p',q')
\end{equation}
(this is very similar to \eqref{coproductforA}). We define relative chains as usual: $C_\bullet(X;Y):=C_\bullet(X)/C_\bullet(Y)$ for $Y\subseteq X$.  It is a straightforward (though tedious) exercise to verify that with this definition of singular chains, the structure maps for the homotopy colimits used to construct \eqref{keycomplexes} are all well-defined and appropriately compatible.  We define the product $\times[E_\alpha]$ via the Eilenberg--Zilber subdivision of $\Delta^p\times\Delta^q$ (as in Remark \ref{chainfixing}) at the index $\alpha$.  It may also be verified in a similar manner that the product/face maps for \eqref{keycomplexes} as defined in Definitions \ref{Cproductface} and \ref{CEcomplexaugmented} make sense and are compatible given this model of chains.  For these maps, we use the product/face maps of $\Y$ along with Eilenberg--Zilber (separately at each index $\alpha$).

We also define $S^1$-equivariant versions of \eqref{keycomplexes}:
\begin{equation}\label{keycomplexesSS}
C^\bullet_{S^1,\vir}(\X\rel\partial;\A)^+(\sigma,p,q)\quad C^\bullet_{S^1,\vir}(\partial\X;\A)^+(\sigma,p,q)\quad C^{S^1}_\bullet(E;\A)^+(\sigma,p,q)
\end{equation}
using $C_{S^1,\vir}^\bullet$ in place of $C_\vir^\bullet$ (and $C_\bullet^{S^1}$ in place of $C_\bullet$).  The fixed fundamental cycles $[E_\alpha]$ again live in ordinary singular simplicial chains.  We use the following as our model for (relative) $S^1$-equivariant chains (for $C_{S^1,\vir}^\bullet$).  Our model of $S^1$-equivariant chains on an $S^1$-space $X$ is generated by the set of commuting diagrams of $S^1$-equivariant maps:
\begin{equation}\label{bigcomplexSS}
\begin{tikzcd}
M_i\times\prod_{\alpha\in\coprod\Abar(\sigma',p',q')}\Delta^{i_\alpha}\ar{d}\ar{r}&X\times(ES^1)_{\Y_{S^1}(\sigma,p,q)}\ar{d}\\
M_i\ar{r}&(ES^1)_{\Y_{S^1}(\sigma,p,q)}
\end{tikzcd}
\end{equation}
where $i\in\Y_{S^1}(\sigma,p,q)$, the vertical maps are the projections, the bottom horizontal map is the given structure map, and $\coprod\Abar(\sigma',p',q')$ is as before.  It is then a straightforward (though tedious) exercise to verify that with this definition of $S^1$-equivariant chains, the complexes \eqref{keycomplexesSS} are well-defined, have well-defined product/face maps, and that the obvious forgetful maps to their non-equivariant versions \eqref{keycomplexes} are well-defined and compatible with the product/face maps.

Now it remains only to show that our models of chains and $S^1$-equivariant chains described above do actually calculate singular homology and $S^1$-equivariant singular homology.\footnote{Technically speaking, we must also show that the pushforward maps, product maps, and Gysin maps have the expected action on homology, though this verification is safely left to the reader.}  Note that once we do this, we can use our alternative definitions of \eqref{keycomplexes} and \eqref{keycomplexesSS} in place of the originals in all the arguments of \S\ref{homologygroupssection}.

We show that our model of chains (generated by diagrams \eqref{bigcomplex}) calculates singular homology as follows.  Let us abbreviate $\Y=\Y(\sigma,p,q)$.  Our complex generated by maps \eqref{bigcomplex} is in fact a double complex with bigrading $p=\sum_\alpha i_\alpha$ and $q=\dim M_i$.  Now let us consider the associated spectral sequence.  We calculate the $E^1_{p,q}$ term as follows.  By definition $E^1_{p,q}$ is simply the homology of the complex with only the differentials decreasing $q$.  Clearly this breaks up as a direct sum over tuples $\{i_\alpha\in\ZZ_{\geq 0}\}$.  For a given tuple $\{i_\alpha\in\ZZ_{\geq 0}\}$, we must calculate the homology of the complex:
\begin{equation}\label{withDelta}
\bigoplus_{\begin{smallmatrix}i\in\Y\cr M_i\times\prod_\alpha\Delta^{i_\alpha}\to X\end{smallmatrix}}\ZZ
\end{equation}
We claim that this is the same as the homology of the complex:
\begin{equation}\label{withoutDelta}
\bigoplus_{\begin{smallmatrix}i\in\Y\cr M_i\to X\end{smallmatrix}}\ZZ
\end{equation}
There is a clearly a map $\eqref{withoutDelta}\to\eqref{withDelta}$ (precompose with the projection $M_i\times\prod_\alpha\Delta^{i_\alpha}\to M_i$), and picking a point $x\in\prod_\alpha\Delta^{i_\alpha}$ gives a map in the opposite direction.  The composition $\eqref{withoutDelta}\to\eqref{withDelta}\to\eqref{withoutDelta}$ is clearly the identity.  The other composition is chain homotopic to the identity by the following argument.  Since $\Y$ is saturated, we can choose a map $I:\Y\to\Y$ and isomorphisms $(M_{I(i)},\SSS_{I(i)})\xrightarrow\sim(M_i\times[0,1],\SSS_i\times\SSS)$ compatible with the face identifications, where $([0,1],\SSS)$ is the PL manifold $[0,1]$ with strata $\{\{0\},\{1\},(0,1)\}$, and $M_i\times\{0\}$, $M_i\times\{1\}$ have the obvious face identifications with $M_i$ (construct $I$ and the isomorphisms by induction on dimension).  Using this coherent choice of ``cylinder objects'', it is easy to define a chain homotopy between the identity map and the composition $\eqref{withDelta}\to\eqref{withoutDelta}\to\eqref{withDelta}$ (using the fact that $\prod_\alpha\Delta^{i_\alpha}$ is contractible).  It follows that the canonical map $\eqref{withoutDelta}\to\eqref{withDelta}$ is a quasi-isomorphism.  By Lemma \ref{saturatedcalculatesH}, the complex $\eqref{withoutDelta}$ calculates singular homology.  Hence we conclude that the $E^1$ term is:
\begin{equation}
E^1_{p,q}=\bigoplus_{\sum_\alpha i_\alpha=p}H_q(X)
\end{equation}
The differentials on the $E^1$ page are easy to understand (essentially we have $H_q(X)$ tensored with the restricted tensor product $\bigotimes_\alpha C_\bullet(\pt)$).  We conclude that:
\begin{equation}
E^2_{p,q}=\begin{cases}H_q(X)&p=0\cr 0&p\ne 0\end{cases}
\end{equation}
Hence there are no further differentials and we are done.

We now use a similar argument to show that our model of $S^1$-equivariant chains calculates $S^1$-equivariant singular homology.  Let us abbreviate $\Y_{S^1}=\Y_{S^1}(\sigma,p,q)$.  Our complex generated by maps \eqref{bigcomplexSS} is a double complex as before, and we consider the associated spectral sequence.  As before, $E^1_{p,q}$ splits up as a direct sum over tuples $\{i_\alpha\in\ZZ_{\geq 0}\}$.  For a given tuple $\{i_\alpha\in\ZZ_{\geq 0}\}$, the corresponding direct summand is the homology of the complex:
\begin{equation}
\bigoplus_{\begin{smallmatrix}&i\in\Y_{S^1}\cr M_i\times\prod_\alpha\Delta^{i_\alpha}&\to&X\times(ES^1)_{\Y_{S^1}}\cr\downarrow&&\downarrow\cr M_i&\to&(ES^1)_{\Y_{S^1}}\end{smallmatrix}}\ZZ
\end{equation}
As in the non-equivariant case, this is canonically quasi-isomorphic to:
\begin{equation}
\bigoplus_{\begin{smallmatrix}&i\in\Y_{S^1}\cr M_i&\to&X\times(ES^1)_{\Y_{S^1}}\cr\downarrow&&\downarrow\cr M_i&\to&(ES^1)_{\Y_{S^1}}\end{smallmatrix}}\ZZ
\end{equation}
Thus by Lemma \ref{saturatedcalculatesHSS} the $E^1$ term is:
\begin{equation}
E^1_{p,q}=\bigoplus_{\sum_\alpha i_\alpha=p}H_q^{S^1}(X)
\end{equation}
The differentials on the $E^1$ page are as before, and thus:
\begin{equation}
E^2_{p,q}=\begin{cases}H_q^{S^1}(X)&p=0\cr 0&p\ne 0\end{cases}
\end{equation}
so we are done.

\subsubsection{Floer-type homology groups from $\sF$-module MF-sets and $S^1$-MF-sets}\label{floerhomologyfromMFsets}

We now give a new definition of Floer-type homology groups identical to Definition \ref{defofhomologygroups} except using the chain models from \S\ref{SSaugmentedsec}.  Precisely, we define a new resolution $\tilde Z_\bullet\to Z_\bullet$ by including the data\footnote{For the reader concerned with set-theoretic issues: so that the collection of all possible choices of this data forms a set, one may add ``rigidifying data'' to the definition of an ($S^1$-)MF-set, e.g.\ require $\Y$ to be countable, equip $\Y$ with an injection into $\omega_1$ (the first uncountable ordinal), and equip each $M_i$ and $(ES^1)_\Y$ with an embedding into $\RR^\infty$ (note that this rigidifying data is not required to be compatible in any way with the rest of the structure, and thus it does not affect any of our previous arguments).} of $\sF$-module ($S^1$-)MF-sets $\Y_{S^1}\to\Y$ over $\Delta^n$ for $f^\ast\X$ satisfying the conclusion of Proposition \ref{inductionformodels}.  Now Proposition \ref{inductionformodels} implies that this new resolution is still a trivial Kan fibration, and thus the rest of the construction of Floer-type homology groups in \S\ref{homologygroupssection} applies as written.

\begin{remark}
These Floer-type homology groups are \emph{a priori} different from the ones constructed in \S\ref{homologygroupssection}.  However, they share all the same properties, and we certainly expect them to be canonically isomorphic (c.f.\ Remark \ref{independenceofchainmodel}).  We won't pursue this here, though, since it is not necessary for our intended application.
\end{remark}

\subsubsection{Localization}

The following localization result concerns the Floer-type homology groups from \S\ref{floerhomologyfromMFsets}.

\begin{theorem}[$S^1$-localization for $\HH(\X)$]\label{circlelocalizationhomology}
Let $\X/Z_\bullet$ be an $H$-equivariant flow category diagram with implicit atlas $\A$ and coherent orientations $\omega$, satisfying the hypotheses of Definition \ref{defofhomologygroups}.  Suppose that $S^1$ acts compatibly on this entire structure ($S^1$-actions on $\X(\sigma,p,q)$ and $S^1$-equivariant implicit atlases, so that all the relevant structure maps are $S^1$-equivariant).

Fix partitions (into disjoint closed subsets):
\begin{equation}\label{partitionsfreefixedII}
\X(\sigma,p,q)=\X(\sigma,p,q)^0\sqcup\X(\sigma,p,q)^1
\end{equation}
where $S^1$ acts almost freely on $\X(\sigma,p,q)^1$, and suppose that the product/face maps specialize to maps:
\begin{align}
\label{PFfirst}\X(\sigma|[0\ldots k],p,q)^0\times\X(\sigma|[k\ldots n],q,r)^0&\to\X(\sigma,p,r)^0\\
\X(\sigma|[0\ldots k],p,q)^0\times\X(\sigma|[k\ldots n],q,r)^1&\to\X(\sigma,p,r)^1\\
\X(\sigma|[0\ldots k],p,q)^1\times\X(\sigma|[k\ldots n],q,r)^0&\to\X(\sigma,p,r)^1\\
\X(\sigma|[0\ldots k],p,q)^1\times\X(\sigma|[k\ldots n],q,r)^1&\to\X(\sigma,p,r)^1\\
\label{PFface}\X(\sigma|[0\ldots\hat k\ldots n],p,q)^0&\to\X(\sigma,p,q)^0\\
\label{PFlast}\X(\sigma|[0\ldots\hat k\ldots n],p,q)^1&\to\X(\sigma,p,q)^1
\end{align}
Assume in addition that for any vertex $\sigma^0\in Z_\bullet$, we have $\X(\sigma^0,p,q)^0=\varnothing\implies\X(\sigma^i,p,q)^0=\varnothing$ for all $i$ and $p\ne q$.

Then there is a canonical isomorphism $\HH(\X)_\A=\HH(\X^0)_\A$, where $\X^0$ is the $H$-equivariant flow category diagram with product/face maps given by \eqref{PFfirst} and \eqref{PFface}, equipped with the implicit atlas obtained by removing $\X^1$ from every thickening.
\end{theorem}

\begin{proof}
\textbf{Part I.}\  We begin by reducing to a situation in which the partition \eqref{partitionsfreefixedII} is extended to every thickened moduli space, so that the product/face maps on thickened moduli spaces also take the form \eqref{PFfirst}--\eqref{PFlast}.  Precisely, consider a new implicit atlas $\tilde\A$ (on the same index set $\A$) in which the thickened moduli spaces are:
\begin{equation}\label{newatlastrickSS}
\X(\sigma,p,q)_{\tilde I}^{\leq\s}:=[\X(\sigma,p,q)_I^{\leq\s}\setminus\X(\sigma,p,q)^1]\sqcup[\X(\sigma,p,q)_I^{\leq\s}\setminus\X(\sigma,p,q)^0]
\end{equation}
Now these thickened moduli spaces are by definition equipped with a partition $\X(\sigma,p,q)_{\tilde I}^{\leq\s}=(\X(\sigma,p,q)_{\tilde I}^{\leq\s})^0\sqcup(\X(\sigma,p,q)_{\tilde I}^{\leq\s})^1$ (namely the partition \eqref{newatlastrickSS}), and the natural product/face maps for the implicit atlas $\tilde\A$ take the form \eqref{PFfirst}--\eqref{PFlast}.

The thickened moduli spaces of $\tilde\A$ are equipped with natural maps to those of $\A$ (compatibile with the product/face maps); these maps are not open embeddings, but nevertheless the proof of Lemma \ref{homologyopenreduction} applies to give a canonical isomorphism $\HH(\X)_\A=\HH(\X)_{\tilde\A}$.  Thus it suffices to produce an isomorphism $\HH(\X)_{\tilde\A}=\HH(\X^0)_{\tilde\A^0}$, where $\tilde\A^0$ is the implicit atlas on $\X^0$ with thickened moduli spaces $(\X(\sigma,p,q)_{\tilde I}^{\leq\s})^0$.  From now on, we will work exclusively with the atlas $\tilde\A$, which we now rather abusively rename as $\A$.

\textbf{Part II.}\  We will define an isomorphism $\HH(\X)_\A=\HH(\X^0)_{\A^0}$, where $\A^0$ is the implicit atlas on $\X^0$ with thickened moduli spaces $(\X(\sigma,p,q)_I^{\leq\s})^0$.

Note that since we have coherent orientations $\omega$ and $S^1$ is connected, it follows that all the flow spaces are locally $S^1$-orientable (see Remark \ref{orientedthenSSoriented}), so we may use the $S^1$-localization machinery freely.

Let $\tilde Z_\bullet\to Z_\bullet$ denote the resolution (as in \S\ref{floerhomologyfromMFsets}) associated to $\X/Z_\bullet$ with the implicit atlas $\A$, and let $\tilde Z_\bullet^0\to Z_\bullet$ denote the resolution associated to $\X^0/Z_\bullet$ with the implicit atlas $\A^0$.

We will construct a diagram of the following shape:
\begin{equation}\label{differentresolutions}
\begin{tikzcd}[row sep = large, column sep = tiny]
\tilde Z_\bullet\ar{drr}&\tilde Z_\bullet^a\ar{l}\ar{r}\ar{dr}&\tilde Z_\bullet^b\ar{d}&\tilde Z_\bullet^c\ar{l}\ar{r}\ar{dl}&\tilde Z_\bullet^0\ar{dll}\\
&&Z_\bullet
\end{tikzcd}
\end{equation}
and natural diagrams $\widetilde\HH,\widetilde\HH^a,\widetilde\HH^b,\widetilde\HH^c,\widetilde\HH^0$ from $\tilde Z_\bullet,\tilde Z_\bullet^a,\tilde Z_\bullet^b,\tilde Z_\bullet^c,\tilde Z_\bullet^0$ to $\Ndg(\Ch_{R[[H]]})$ (respectively).  We will also construct natural isomorphisms relating these and their pullbacks under the horizontal maps in \eqref{differentresolutions}.  Finally, we will show that each of the vertical maps \eqref{differentresolutions} satisfies the conclusions of Propositions \ref{trivialKanfibration} and \ref{degenerateID}.  The desired result follows easily from these statements; let us now start on the proof.

We may consider $\A(\sigma,p,q)$ as a system of implicit atlases on $\X(\sigma,p,q)^0$ and on $\X(\sigma,p,q)^1$ separately.  Now consider the following maps (where the complexes on the left are defined analogously to those on the right as in \S\ref{SSaugmentedsec}):
\begin{align}
\label{partialSSpushforwardbigcomplex}C^\bullet_\vir(\partial\X^0,\A)^+&\oplus C^\bullet_{S^1,\vir}(\partial\X^1,\A)^+&&\to C^\bullet_\vir(\partial\X;\A)^+\\
\label{relpartialSSpushforwardbigcomplex}C^\bullet_\vir(\X^0\rel\partial,\A)^+&\oplus C^\bullet_{S^1,\vir}(\X^1\rel\partial,\A)^+&&\to C^\bullet_\vir(\X\rel\partial;\A)^+
\end{align}
We give $C^\bullet_\vir(\X^0\rel\partial,\A)^+\oplus C^\bullet_{S^1,\vir}(\X^1\rel\partial,\A)^+$ the structure of an $\sF$-module with product/face maps:
\begin{align}
\label{CSSproductfacesplitfirst}C^\bullet_\vir(\X^0\rel\partial;\A)^+(\cdots)\otimes C^\bullet_\vir(\X^0\rel\partial;\A)^+(\cdots)&\to C^\bullet_\vir(\partial\X^0;\A)^+(\cdots)\\
C^\bullet_\vir(\X^0\rel\partial;\A)^+(\cdots)\otimes C^\bullet_{S^1,\vir}(\X^1\rel\partial;\A)^+(\cdots)&\to C^\bullet_{S^1,\vir}(\partial\X^1;\A)^+(\cdots)\\
C^\bullet_{S^1,\vir}(\X^1\rel\partial;\A)^+(\cdots)\otimes C^\bullet_\vir(\X^0\rel\partial;\A)^+(\cdots)&\to C^\bullet_{S^1,\vir}(\partial\X^1;\A)^+(\cdots)\\
C^\bullet_{S^1,\vir}(\X^1\rel\partial;\A)^+(\cdots)\otimes C^\bullet_{S^1,\vir}(\X^1\rel\partial;\A)^+(\cdots)&\to C^\bullet_{S^1,\vir}(\partial\X^1;\A)^+(\cdots)\\
C^\bullet_\vir(\X^0\rel\partial;\A)^+(\cdots)&\to C^\bullet_\vir(\partial\X^0;\A)^+(\cdots)\\
\label{CSSproductfacesplitlast}C^\bullet_{S^1,\vir}(\X^1\rel\partial;\A)^+(\cdots)&\to C^\bullet_{S^1,\vir}(\partial\X^1;\A)^+(\cdots)
\end{align}
just as in Definition \ref{Cproductface} (and using the chain models from \S\ref{SSaugmentedsec}).  These are compatible with \eqref{partialSSpushforwardbigcomplex}--\eqref{relpartialSSpushforwardbigcomplex} and \eqref{Ccomposition}--\eqref{Csimplexextension}.

Next, we consider $C_\bullet(E;\A)^+\oplus C_{\bullet-1}^{S^1}(E;\A)^+$ (as in \S\ref{SSaugmentedsec}).  Now, there are \emph{two} natural maps:
\begin{equation}\label{CESSpushforward}
\begin{tikzcd}
C_\bullet(E;\A)^+\oplus C_{\bullet-1}^{S^1}(E;\A)^+\ar[yshift=0.5ex]{r}{\id\oplus{\pi^!}}\ar[yshift=-0.5ex]{r}[swap]{\id\oplus 0}&C_\bullet(E;\A)^+
\end{tikzcd}
\end{equation}
We give $C_\bullet(E;\A)^+\oplus C_{\bullet-1}^{S^1}(E;\A)^+$ the structure of an $\sF$-module with product/face maps of the shape \eqref{CSSproductfacesplitfirst}--\eqref{CSSproductfacesplitlast} just as in Definition \ref{CEcomplexaugmented}.  Then both maps \eqref{CESSpushforward} are maps of $\sF$-modules.  There is also a pushforward map:
\begin{equation*}
C^{\vdim\X+\bullet}_\vir(\X^0\rel\partial,\A)^+\oplus C^{\vdim\X+\bullet}_{S^1,\vir}(\X^1\rel\partial,\A)^+\to C_{-\bullet}(E;\A)^+\oplus C_{-\bullet-1}^{S^1}(E;\A)^+
\end{equation*}
which is a map of $\sF$-modules.

Now let us define $\tilde Z_\bullet^a$, $\tilde Z_\bullet^b$, and $\tilde Z_\bullet^c$.  We modify the definition of $\tilde Z_\bullet\to Z_\bullet$ (i.e.\ Definition \ref{Lambdadef} as amended in \S\ref{floerhomologyfromMFsets}) as follows.  For $\tilde Z_\bullet^a=\tilde Z_\bullet^c$ we replace Definition \ref{Lambdadef}(\ref{lambdachoice},\ref{liftschoice}) with (\ref{lambdachoiceSS},\ref{liftschoiceSS}) below, and for $\tilde Z_\bullet^b$ we replace Definition \ref{Lambdadef}(\ref{lambdachoice},\ref{liftschoice},\ref{augmentationschoice}) with (\ref{lambdachoiceSS},\ref{liftschoiceSS},\ref{augmentationschoiceSS}) below.
\begin{rlist}
\item\label{lambdachoiceSS}An $H$-invariant system of chains:
\begin{equation*}
\lambda=\lambda^0\oplus\lambda^1\in C^\bullet_\vir(\X^0\rel\partial,\B)^+\oplus C^\bullet_{S^1,\vir}(\X^1\rel\partial,\B)^+
\end{equation*}
(degree $0$ and supported inside $\supp\X$) with the following property.  Note that $(\mu,\lambda)$ is a cycle in the mapping cone:
\begin{equation*}
\begin{matrix}
C^\bullet_\vir(\partial f^\ast\X^0;\B)^+\cr\oplus\cr C^\bullet_{S^1,\vir}(\partial f^\ast\X^1;\B)^+
\end{matrix}
\longrightarrow
\begin{matrix}
C^\bullet_{\vir}(f^\ast\X^0\rel\partial;\B)^+\cr\oplus\cr C^\bullet_{S^1,\vir}(f^\ast\X^1\rel\partial;\B)^+
\end{matrix}
\end{equation*}
whose homology is identified with $\cH^\bullet(f^\ast\X^0;\oo_{f^\ast\X^0})\oplus\cH^\bullet((f^\ast\X^1)/S^1;\pi_\ast\oo_{f^\ast\X^1})$, which in degree zero simply equals $\cH^0(f^\ast\X;\oo_{f^\ast\X})$.  We require that the homology class of $(\mu,\lambda)$ equal $f^\ast\omega\in\cH^0(f^\ast\X;\oo_{f^\ast\X})$.
\item\label{liftschoiceSS}An $H$-invariant system of chains $\tilde\lambda$ for $Q[C_\bullet(E;\A)^+\oplus C_{\bullet-1}^{S^1}(E;\A)^+]$ (degree $\gr(q)-\gr(p)+\dim\sigma-1$ and supported inside $\supp\X$) whose image in $C_\bullet(E;\A)^+\oplus C_{\bullet-1}^{S^1}(E;\A)^+$ coincides with the image of $\lambda$.
\item\label{augmentationschoiceSS}An $H$-invariant map of $\sF$-modules $[[E]]:Q[C_\bullet(E;\A)^+_\ZZ\oplus C_{\bullet-1}^{S^1}(E;\A)^+_\ZZ]\to\ZZ$ which sends the fundamental class in $H_\bullet(E;\A)^+_\ZZ$ to $1$ and is zero on $H_{\bullet-1}^{S^1}(E;\A)^+_\ZZ$.
\end{rlist}
The maps $\tilde Z_\bullet\leftarrow\tilde Z_\bullet^a\to\tilde Z_\bullet^b\leftarrow\tilde Z_\bullet^c\to\tilde Z_\bullet^0$ are defined as follows:
\begin{equation}\label{defofZabcmaps}
\begin{tikzcd}
\tilde Z_\bullet&\ar{l}\tilde Z_\bullet^a\ar{r}&\tilde Z_\bullet^b&\ar{l}\tilde Z_\bullet^c\ar{r}&\tilde Z_\bullet^0\cr
\lambda&\ar[mapsto]{l}[swap]{\eqref{relpartialSSpushforwardbigcomplex}}\lambda&&\lambda\ar[mapsto]{r}{\lambda=\lambda^0\oplus\lambda^1}&\lambda^0\cr
\tilde\lambda&\ar[mapsto]{l}[swap]{Q[\id\oplus\pi^!]}\tilde\lambda&&\tilde\lambda\ar[mapsto]{r}{Q[\id\oplus 0]}&\tilde\lambda^0\cr
&{}[[E]]\ar[mapsto]{r}{\circ Q[\id\oplus\pi^!]}&{}[[E]]&\ar[mapsto]{l}[swap]{\circ Q[\id\oplus 0]}\null[[E]]
\end{tikzcd}
\end{equation}
Note that Lemma \ref{trivialactiongysin} shows that the map $\tilde Z^a\to\tilde Z^b$ above produces an $[[E]]$ satisfying (\ref{augmentationschoiceSS}).

Now $\widetilde\HH,\widetilde\HH^a,\widetilde\HH^b,\widetilde\HH^c,\widetilde\HH^0$ are defined via the matrix coefficients $c_{\sigma,p,q}$ defined as follows:
\begin{align*}
\text{for $\widetilde\HH$:}\quad c_{\sigma,p,q}&:=([[E]]_{\sigma,p,q}\otimes\id_R)(\tilde\lambda_{\sigma,p,q})\\
\text{for $\widetilde\HH^a$:}\quad c_{\sigma,p,q}&:=([[E]]_{\sigma,p,q}\otimes\id_R)(\id\oplus\pi^!)(\tilde\lambda_{\sigma,p,q})\\
\text{for $\widetilde\HH^b$:}\quad c_{\sigma,p,q}&:=([[E]]_{\sigma,p,q}\otimes\id_R)(\tilde\lambda_{\sigma,p,q})\\
\text{for $\widetilde\HH^c$:}\quad c_{\sigma,p,q}&:=([[E]]_{\sigma,p,q}\otimes\id_R)(\id\oplus 0)(\tilde\lambda_{\sigma,p,q})\\
\text{for $\widetilde\HH^0$:}\quad c_{\sigma,p,q}&:=([[E]]_{\sigma,p,q}\otimes\id_R)(\tilde\lambda_{\sigma,p,q})
\end{align*}
These matrix coefficients give rise to diagrams in $\Ndg(\Ch_{R[[H]]})$ since $[[E]]$ is a map of $\sF$-modules.  The isomorphisms between these diagrams and their pullbacks under the maps \eqref{defofZabcmaps} are evident.

It now remains to show that each of the maps $\tilde Z_\bullet^a,\tilde Z_\bullet^b,\tilde Z_\bullet^c\to\tilde Z_\bullet$ satisfies the conclusions of Propositions \ref{trivialKanfibration} and \ref{degenerateID}.  The proof of Proposition \ref{trivialKanfibration} applies to $\tilde Z_\bullet^a,\tilde Z_\bullet^b,\tilde Z_\bullet^c\to\tilde Z_\bullet$ as written (for $\tilde Z_\bullet^b$, the extension of $[[E]]$ step uses the fact that $H_\bullet(E;\A)\oplus H_{\bullet-1}^{S^1}(E;\A)$ is concentrated in degrees $\geq 0$ so that Lemma \ref{Hzeroofcofibrantcolimit} still applies).  The proof of Proposition \ref{degenerateID} also applies to $\tilde Z_\bullet^a,\tilde Z_\bullet^b,\tilde Z_\bullet^c\to\tilde Z_\bullet$ as written.  Now the uniqueness of descent from Lemma \ref{descentofdiagram} shows that there are canonical isomorphisms of the descents $\HH=\HH^a=\HH^b=\HH^c=\HH^0$.
\end{proof}

\section{Gromov--Witten invariants}\label{gromovwittensection}

In this section, we define Gromov--Witten invariants for a general smooth closed symplectic manifold $(X,\omega)$ (which we now fix).  This has been treated in the literature by Li--Tian \cite{litianII}, Fukaya--Ono \cite{fukayaono}, and Ruan \cite{ruan}.

More specifically, the main subject of this section is the construction of an implicit atlas on the moduli space of stable $J$-holomorphic maps $\Mbar_{g,n}^\beta(X)$; the same method also yields an implicit atlas on $\Mbar_{g,n}^\beta(X,J_{[0,1]})$ (the moduli space associated to a family of $J$ parameterized by $[0,1]$).  Hence the virtual fundamental class $[\Mbar_{g,n}^\beta(X)]^\vir$ (Definition \ref{fclassdefinition}) is defined, and the properties derived in \S\ref{fundamentalclasssection} show that its image in $H_\bullet(\Mbar_{g,n}\times X^n)$ is independent of $J$.  The necessary gluing results (which we isolate in Proposition \ref{GWgluingneeded}) are proved in Appendix \ref{gluingappendix}.

It would be interesting to use implicit atlases to prove that the invariants defined here satisfy the Kontsevich--Manin axioms \cite{kontsevichmanin}, as proved in Fukaya--Ono \cite{fukayaono} for their definition (to do this, one would need to show additional properties of the virtual fundamental class, e.g.\ as suggested in \S\ref{fundamentalclasssection}).

\subsection{Moduli space \texorpdfstring{$\Mbar_{g,n}^\beta(X)$}{Mbar\_g,n\textasciicircum beta(X)}}

Let us now fix a smooth almost complex structure $J$ on $X$ which is tamed by $\omega$.

\begin{definition}[Nodal curve of type $(g,n)$]
A \emph{nodal curve of type $(g,n)$} is a compact nodal Riemann surface $C$ of arithmetic genus $g$ along with an injective function $l:\{1,\ldots,n\}\to C$ (the ``$n$ marked points'') whose image is disjoint from the nodes of $C$.  An \emph{isomorphism $(C,l)\to(C',l')$ of curves of type $(g,n)$} is an isomorphism of Riemann surfaces $\iota:C\to C'$ such that $l'=\iota\circ l$.  We usually omit $l$ from the notation.  Such a curve is called \emph{stable} iff its automorphism group is finite.
\end{definition}

We denote by $\Mbar_{g,n}$ the Deligne--Mumford moduli space of stable nodal curves of type $(g,n)$, and we denote by $\Cbar_{g,n}\to\Mbar_{g,n}$ the universal family (which coincides with the ``forget the last marked point and stabilize'' map $\Mbar_{g,n+1}\to\Mbar_{g,n}$).  The moduli space $\Mbar_{g,n}$ is a compact complex analytic orbifold.

\begin{definition}[$J$-holomorphic map]
A \emph{$J$-holomorphic map of type $(g,n)$} is a pair $(C,u)$ where $C$ is a nodal curve of type $(g,n)$ and $u:C\to X$ is smooth and satisfies $\delbar u=0$.  An \emph{isomorphism $(C,u)\to(C',u')$ of $J$-holomorphic maps of type $(g,n)$} is an isomorphism $\iota:C\to C'$ of curves of type $(g,n)$ such that $u=u'\circ\iota$.  We say a $J$-holomorphic map is \emph{stable} iff its automorphism group (i.e.\ group of self-isomorphisms) is finite.
\end{definition}

\begin{definition}[Moduli space of stable maps; introduced by Kontsevich \cite{kontsevich}]
Let $\beta\in H_2(X,\ZZ)$.  We define $\Mbar_{g,n}^\beta(X)$ as the set of stable $J$-holomorphic maps of type $(g,n)$ for which $u_\ast[C]=\beta$.  We equip $\Mbar_{g,n}^\beta(X)$ with the Gromov topology, which is well-known to be compact Hausdorff.
\end{definition}

For completeness, let us recall the definition of the Gromov topology on $\Mbar_{g,n}^\beta(X)$.  A neighborhood base\footnote{A \emph{neighborhood} of a point $x$ in a topological space $X$ is a subset $N\subseteq X$ such that $x\in N^\circ$ (the interior).  A \emph{neighborhood base} at a point $x\in X$ is a collection of neighborhoods $\{N_\alpha\}$ of $x$ such that for every open $U\subseteq X$ containing $x$, there exists some $N_\alpha\subseteq U$.  A neighborhood base is necessarily nonempty and \emph{filtered}, i.e.\ for all $\alpha,\beta$, there exists $\gamma$ with $N_\gamma\subseteq N_\alpha\cap N_\beta$.  Conversely, given a set $X$ and for every $x\in X$ a nonempty filtered collection of subsets $\{N_\alpha^x\}$ each containing $x$, there is a unique topology on $X$ such that $\{N^x_\alpha\}$ is a neighborhood base at $x$ for all $x\in X$.} at a pair $(C,u)$ may be obtained as follows.  We choose some additional $\ell$ marked points on $C$ so that it has no automorphisms fixing these points, and we consider the graph $\Gamma_u\subseteq X\times\Cbar_{g,n+\ell}$ where $\Cbar_{g,n+\ell}\to\Mbar_{g,n+\ell}$ denotes the universal curve.  A neighborhood base at $(C,u)$ is obtained taking all $J$-holomorphic maps from curves in $\Mbar_{g,n+\ell}$ whose graph is close to $\Gamma_u$ in the Hausdorff topology (and forgetting the $\ell$ extra marked points).  Choosing different $\ell'$ marked points yields an equivalent neighborhood base.

\subsection{Implicit atlas on \texorpdfstring{$\Mbar_{g,n}^\beta(X)$}{Mbar\_g,n\textasciicircum beta(X)}}\label{gromovwittenimplicitatlassubsection}

We define an implicit atlas $A^\GW$ on $\Mbar_{g,n}^\beta(X)$, proceeding in several steps.  Note that the space on which we will define an implicit atlas is no longer denoted $X$, and this leads to a few (evident) notational differences from the earlier sections where we considered implicit atlases abstractly.

\begin{definition}[Index set $A^\GW$]\label{divisorthicken}
A \emph{(Gromov--Witten) thickening datum} $\alpha$ is a 6-tuple $(D_\alpha,r_\alpha,\Gamma_\alpha,\Mbar_\alpha,E_\alpha,\lambda_\alpha)$ where:
\begin{rlist}
\item$D_\alpha\to X$ is a compact smooth embedded codimension two submanifold with boundary.
\item$r_\alpha\geq 0$ is an integer such that $2g+n+r_\alpha>2$.
\item$\Gamma_\alpha$ is a finite group.
\item$\Mbar_\alpha$ is a smooth $\Gamma_\alpha$-manifold\footnote{Note that we do not assume $\Gamma_\alpha$ acts effectively on $\Mbar_\alpha$.  Indeed, $\Mbar_{1,1}$ and $\Mbar_{2,0}$ are ineffective orbifolds, so we must allow ineffective actions if we want $\Mbar_\alpha/\Gamma_\alpha$ to be isomorphic to an open subset of one of these spaces.} equipped with an isomorphism of orbifolds between $\Mbar_\alpha/\Gamma_\alpha$ and an open subset of $\Mbar_{g,n+r_\alpha}/S_{r_\alpha}$ (let $\Cbar_\alpha\to\Mbar_\alpha$ be the pullback\footnote{Pullback here means the orbifold fiber product (see Remark \ref{orbifoldfiberproductrmk}).  Note that the fibers of $\Cbar_\alpha\to\Mbar_\alpha$ are nodal curves of type $(g,n+r_\alpha)$ (not quotiented by their automorphism group).} of the universal family; the action of $\Gamma_\alpha$ on $\Mbar_\alpha$ lifts canonically to $\Cbar_\alpha$ by the universal property of the pullback).
\item$E_\alpha$ is a finitely generated $\RR[\Gamma_\alpha]$-module.
\item$\lambda_\alpha:E_\alpha\to C^\infty(\Cbar_\alpha\times X,\Omega^{0,1}_{\Cbar_\alpha/\Mbar_\alpha}\otimes_\CC TX)$ is a $\Gamma_\alpha$-equivariant linear map supported away from the nodes and marked points of the fibers of $\Cbar_\alpha\to\Mbar_\alpha$.
\end{rlist}
Let $A^{\GW}$ denote the set\footnote{The reader concerned with set theoretic issues may wish to add ``rigidifying data'' to the definition of a thickening datum as in Remark \ref{notsetremark}.} of all thickening datums.
\end{definition}

\begin{definition}[Transversality of smooth maps]
Let $u:C\to X$ be a smooth map from a nodal curve of type $(g,n)$, and let $D\subseteq X$ be a smooth codimension two submanifold with boundary.  We say \emph{$u$ is transverse to $D$} (written $u\pitchfork D$) iff $u^{-1}(\partial D)=0$ and $\forall\,p\in u^{-1}(D)$, the derivative $du:T_pC\to T_{u(p)}X/T_{u(p)}D$ is surjective and $p$ is neither a node nor a marked point of $C$.
\end{definition}

\begin{definition}[$I$-thickened $J$-holomorphic map]\label{thickendefs}
Let $I\subseteq A^\GW$ be a finite subset.  An \emph{$I$-thickened $J$-holomorphic map of type $(g,n)$} is a quadruple $(C,u,\{\phi_\alpha\}_{\alpha\in I},\{e_\alpha\}_{\alpha\in I})$ where:
\begin{rlist}
\item$C$ is a nodal curve of type $(g,n)$.
\item$u:C\to X$ is a smooth map such that $u\pitchfork D_\alpha$ with exactly $r_\alpha$ intersections for all $\alpha\in I$.
\item$\phi_\alpha:C\to\Cbar_\alpha$ an isomorphism between $C$ (with $r_\alpha$ extra marked points $u^{-1}(D_\alpha)$) and some fiber of $\Cbar_\alpha\to\Mbar_\alpha$.
\item$e_\alpha\in E_\alpha$.
\item The following \emph{$I$-thickened $\delbar$-equation} is satisfied:
\begin{equation}\label{complementeddelbar}
\delbar u+\sum_{\alpha\in I}\lambda_\alpha(e_\alpha)(\phi_\alpha,u)=0
\end{equation}
This equation takes place in $C^\infty(\tilde C,\Omega^{0,1}_{\tilde C}\otimes_\CC u^\ast TX)$.  Certainly $\delbar u$ is a smooth section of $\Omega^{0,1}_{\tilde C}\otimes_\CC u^\ast TX$ over $\tilde C$.  To interpret $\lambda_\alpha(e_\alpha)(\phi_\alpha,u)$ as such a section, consider the following composition:
\begin{equation}
C\xrightarrow{(\phi_\alpha,u)}\Cbar_\alpha\times X\xrightarrow{\lambda(e_\alpha)}\Omega^{0,1}_{\Cbar_\alpha/\Mbar_\alpha}\otimes_\CC TX
\end{equation}
This sends a point $p\in C$ to an element of $(\Omega^{0,1}_{\Cbar_\alpha/\Mbar_\alpha})_{\phi_\alpha(p)}\otimes T_{u(p)}X$, which we identify via $\phi_\alpha$ with $(\Omega^{0,1}_C\otimes_\CC u^\ast TX)_p$.  This is what we mean by $\lambda_\alpha(e_\alpha)(\phi_\alpha,u)$.
\end{rlist}
An \emph{isomorphism} $\iota:(C,u,\{\phi_\alpha\}_{\alpha\in I},\{e_\alpha\}_{\alpha\in I})\to(C',u',\{\phi_\alpha'\}_{\alpha\in I},\{e_\alpha'\}_{\alpha\in I})$ between two $I$-thickened $J$-holomorphic maps of type $(g,n)$ is an isomorphism $\iota:C\to C'$ of curves of $(g,n)$-type such that $u=u'\circ\iota$, $\phi_\alpha=\phi_\alpha'\circ\iota$, and $e_\alpha=e_\alpha'$.  We say an $I$-thickened $J$-holomorphic map is \emph{stable} iff its automorphism group (i.e.\ group of self-isomorphisms) is finite.
\end{definition}

\begin{definition}[Atlas data for $A^\GW$ on $\Mbar_{g,n}^\beta(X)$]\label{GWatlasdatadef}
We define $\Mbar_{g,n}^\beta(X)_I$ as the set of isomorphism classes of stable $I$-thickened $J$-holomorphic maps of type $(g,n)$ such that $u_\ast[C]=\beta$.  Equip $\Mbar_{g,n}^\beta(X)_I$ with the Gromov topology\footnote{When $I\ne\varnothing$, this topology is easy to define: it is simply given by using the Hausdorff distance on the image $C\subseteq X\times\prod_{\alpha\in I}\Cbar_\alpha$.} for $(u,\phi_\alpha)$ and with the obvious topology for $e_\alpha$.  It is clear by definition that $\Mbar_{g,n}^\beta(X)=\Mbar_{g,n}^\beta(X)_\varnothing$.

There is an evident action $\Gamma_I$ on $\Mbar_{g,n}^\beta(X)_I$, namely $\{g_\alpha\}_{\alpha\in I}\cdot(u,\{\phi_\alpha\}_{\alpha\in I},\{e_\alpha\}_{\alpha\in I})=(u,\{g_\alpha\cdot\phi_\alpha\}_{\alpha\in I},\{g_\alpha\cdot e_\alpha\}_{\alpha\in I})$ which works since $\lambda_\alpha(e_\alpha)(\phi_\alpha,u)=\lambda_\alpha(g_\alpha\cdot e_\alpha)(g_\alpha\cdot\phi_\alpha,u)$ by $\Gamma_\alpha$-equivariance.

There are evident maps $s_\alpha:\Mbar_{g,n}^\beta(X)_I\to E_\alpha$ simply picking out $e_\alpha$.  

For $I\subseteq J\subseteq A^\GW$, there is an obvious forgetful map $\psi_{IJ}:(s_{J\setminus I}|\Mbar_{g,n}^\beta(X)_J)^{-1}(0)\to\Mbar_{g,n}^\beta(X)_I$.  Let $U_{IJ}\subseteq\Mbar_{g,n}^\beta(X)_I$ consist of those elements such that $u\pitchfork D_\alpha$ with exactly $r_\alpha$ intersections for all $\alpha\in J\setminus I$ and such that adding these extra marked points makes $C$ isomorphic to a fiber of $\Cbar_\alpha\to\Mbar_\alpha$.  It's not too hard to see that this is an open set (using elliptic regularity).
\end{definition}

The compatibility axioms are all immediate, though let us justify the homeomorphism axiom.  First, we claim that the map $(s_{J\setminus I}|\Mbar_{g,n}^\beta(X)_J)^{-1}(0)/\Gamma_{J\setminus I}\to U_{IJ}\subseteq\Mbar_{g,n}^\beta(X)_I$ induced by $\psi_{IJ}$ is a bijection.  This is more or less clear: fixing an element of $U_{IJ}$, an inverse image must have $e_\alpha=0$ for $\alpha\in J\setminus I$, and by definition of $U_{IJ}$, there exists a suitable $\phi_\alpha$ which is clearly unique up to the action of $\Gamma_\alpha$ for $\alpha\in J\setminus I$.  That the topologies coincide can be checked directly from their definition.

\begin{definition}[Regular locus for $A^\GW$ on $\Mbar_{g,n}^\beta(X)$]\label{regularsubsetdefinition}
We define $\Mbar_{g,n}^\beta(X)_I^\reg\subseteq\Mbar_{g,n}^\beta(X)_I$.

Let $(C,u_0,\{\phi_\alpha\}_{\alpha\in I},\{e_\alpha\}_{\alpha\in I})\in\Mbar_{g,n}^\beta(X)_I$.  Roughly speaking, this point is contained in $\Mbar_{g,n}^\beta(X)_I^\reg$ iff it has trivial automorphism group and the ``vertical'' (i.e.\ from a fixed domain curve) linearized version of the $I$-thickened $\delbar$-equation \eqref{complementeddelbar} is surjective.  Let us now make this precise.

Consider the smooth Banach manifold $W^{k,p}(C,X)$ for some large integer $k$ and some $p\in(1,\infty)$.  Note that the condition $u\pitchfork D_\alpha$ with exactly $r_\alpha$ intersections making $C$ isomorphic to a fiber of $\Cbar_\alpha/\Mbar_\alpha$ (for all $\alpha\in I$) forms an open subset of $W^{k,p}(C,X)$ containing $u_0$.  Furthermore, there is a unique continuous choice of:
\begin{equation}\label{phiparametric}
C\times W^{k,p}(C,X)\to\Cbar_\alpha
\end{equation}
over $C$ times a neighborhood of $u_0$ which extends $\phi_\alpha$ on $C\times\{u_0\}$ (this holds because $\Mbar_\alpha\to\Mbar_{g,n+r_\alpha}/S_{r_\alpha}$ is an \'etale map of orbifolds).  One can also check that this map \eqref{phiparametric} is highly differentiable (this depends on $k$ being large).

Now we consider the smooth Banach bundle whose fiber over $(u,\{e_\alpha\}_{\alpha\in I})\in W^{k,p}(C,X)\times E_I$ is $W^{k-1,p}(\tilde C,\Omega^{0,1}_{\tilde C}\otimes_\CC u^\ast TX)$ (where $\tilde C$ is the normalization of $C$).  Now the left hand side of \eqref{complementeddelbar} (using \eqref{phiparametric} in place of $\phi_\alpha$) is a highly differentiable section of this bundle.  We say that $(C,u_0,\{\phi_\alpha\}_{\alpha\in I},\{e_\alpha\}_{\alpha\in I})\in\Mbar_{g,n}^\beta(X)_I^\reg$ iff the following two conditions hold:
\begin{rlist}
\item This section is transverse to the zero section at $(u_0,e_\alpha)$.
\item The automorphism group of $(C,u_0,\{\phi_\alpha\}_{\alpha\in I},\{e_\alpha\}_{\alpha\in I})$ is trivial.
\end{rlist}
It is an easy exercise in elliptic regularity to show that the first condition is independent of the choice of $(k,p)$ (as long as $k$ is sufficiently large so that the condition makes sense).
\end{definition}

Let:
\begin{equation}
\vdim\Mbar_{g,n}^\beta(X):=\dim\Mbar_{g,n}+(1-g)\dim X+2\langle c_1(X),\beta\rangle
\end{equation}

Let us now verify the transversality axioms.  Freeness of the action of $\Gamma_{J\setminus I}$ on $\psi_{IJ}^{-1}(X_I^\reg)$ follows from the fact that points in $X_I^\reg$ have trivial automorphism group.  The openness and submersion axioms follow from the following result, whose proof is given in Appendix \ref{gluingappendix}.

\begin{proposition}[Formal regularity implies topological regularity]\label{GWgluingneeded}
For all $I\subseteq J\subseteq A^\GW$, we have:
\begin{rlist}
\item\label{GWgluingopen}$\Mbar_{g,n}^\beta(X)^\reg_I\subseteq\Mbar_{g,n}^\beta(X)_I$ is an open subset.
\item\label{GWgluingsubmersion}The map $s_{J\setminus I}:\Mbar_{g,n}^\beta(X)_J\to E_{J\setminus I}$ is locally modeled on the projection:
\begin{equation}
\RR^{\vdim\Mbar_{g,n}^\beta(X)+\dim E_I}\times\RR^{\dim E_{J\setminus I}}\to\RR^{\dim E_{J\setminus I}}
\end{equation}
over $\psi_{IJ}^{-1}(\Mbar_{g,n}^\beta(X)_I^\reg))\subseteq\Mbar_{g,n}^\beta(X)_J$.
\item\label{GWgluingor}There is a canonical identification of the orientation local system $\oo_{\Mbar_{g,n}^\beta(X)_I^\reg}$ with $\oo_{E_I}$ (by the standard reduction to the canonical orientation of a complex linear Fredholm operator as in McDuff--Salamon \cite{mcduffsalamonJholquant,mcduffsalamonJholsymp}).
\end{rlist}
\end{proposition}

Near points of $\Mbar_{g,n}^\beta(X)^\reg_I$ with smooth domain curve, the proposition follows from standard techniques (implicit function theorem, elliptic regularity, and an index theorem).  The real content of the proposition is that it holds near points with nodal domain curve.

Finally, let us verify the covering axiom.

\begin{lemma}\label{sortofsard}
Let $f:N^n\to M^m$ be a smooth map of smooth manifolds with $n\leq m$.  Then for every $p\in N$ with $df_p$ injective and every neighborhood $U$ of $f(p)$, there exists a codimension $n$ smooth submanifold with boundary $D\subseteq M$ contained in $U$ such that $f\pitchfork D$ and $f^{-1}(D)$ contains a point arbitrarily close to $p$.
\end{lemma}

\begin{proof}
After possibly shrinking $U$, choose a local projection $\pi:U\to T_pN$.  Then let $D$ be $\pi^{-1}$ of a regular value of $\pi\circ f$ (which exists by Sard's theorem).
\end{proof}

\begin{lemma}[Domain stabilization for stable $J$-holomorphic maps]\label{stabilizationlemma}
Let $u:C\to X$ be a stable $J$-holomorphic map.  Then there exists a smooth codimension two submanifold with boundary $D\subseteq X$ such that $u\pitchfork D$ and adding $u^{-1}(D)$ to $C$ as extra marked points makes $C$ stable.
\end{lemma}

\begin{proof}
Let us observe that if $C_0\subseteq C$ is any unstable irreducible component, we must have $u:C_0\to X$ is nonconstant.  Hence since $u$ is $J$-holomorphic, it follows that there is a point on $C_0$ (which is not a node or marked point) where $du$ is injective.  It follows using Lemma \ref{sortofsard} that there exists $D\subseteq X$ such that $u\pitchfork D$ and adding $u^{-1}(D)$ to $C$ as marked points makes $C$ stable.
\end{proof}

\begin{lemma}[Covering axiom for $A^\GW$ on $\Mbar_{g,n}^\beta(X)$]\label{everythingcanbethickened}
We have:
\begin{equation}
\Mbar_{g,n}^\beta(X)=\bigcup_{I\subseteq A^\GW}\psi_{\varnothing I}((s_I|\Mbar_{g,n}^\beta(X)_I^\reg)^{-1}(0))
\end{equation}
\end{lemma}

\begin{proof}
Fix a point in $\Mbar_{g,n}^\beta(X)$ (that is, a curve $u:C\to X$).  We will construct $\alpha\in A^\GW$ so that this point is contained in $\psi_{\varnothing\{\alpha\}}((s_\alpha|\Mbar_{g,n}^\beta(X)_{\{\alpha\}}^\reg)^{-1}(0))$.

First, pick $D_\alpha\subseteq X$ satisfying the conclusion of Lemma \ref{stabilizationlemma}.

Now let $r_\alpha=\#u^{-1}(D_\alpha)$.  Adding $u^{-1}(D_\alpha)$ as extra marked points to $C$ gives a point in $\Mbar_{g,n+r_\alpha}/S_{r_\alpha}$, and we pick some local orbifold chart $\Mbar_\alpha/\Gamma_\alpha\to\Mbar_{g,n+r_\alpha}/S_{r_\alpha}$ covering this point.

Now let us consider the linearized $\delbar$ operator:
\begin{equation}\label{linearizeddelbarforthickening}
D\delbar(u,\cdot):C^\infty(C,u^\ast TX)\to C^\infty(\tilde C,\Omega^{0,1}_{\tilde C}\otimes_\CC u^\ast TX)
\end{equation}
($\tilde C$ being the normalization of $C$) where $C^\infty(C,u^\ast TX)\subseteq C^\infty(\tilde C,u^\ast TX)$ is the subspace of functions descending continuously to $C$.  If $(k,p)\in\ZZ_{\geq 1}\times(1,\infty)$ with $kp>2$, then we get a corresponding $D\delbar(u,\cdot)$ map $W^{k,p}\to W^{k-1,p}$ (the restriction on $(k,p)$ comes from the need to define $W^{k,p}(C,u^\ast TX)\subseteq W^{k,p}(\tilde C,u^\ast TX)$).  This operator is Fredholm; in particular its cokernel is finite-dimensional.

Suppose we have a finite-dimensional vector space $E_0$ equipped with a linear map:
\begin{equation}\label{lambdazero}
\tilde\lambda_0:E_0\to C^\infty(\tilde C,\Omega^{0,1}_{\tilde C}\otimes_\CC u^\ast TX)
\end{equation}
supported away from (the inverse image in $\tilde C$ of) the nodes and marked points of $C$.  Then the following conditions are equivalent (exercise using elliptic regularity):
\begin{rlist}
\item\eqref{lambdazero} is surjective onto the cokernel of $D\delbar(u,\cdot):C^\infty\to C^\infty$.
\item\eqref{lambdazero} is surjective onto the cokernel of $D\delbar(u,\cdot):W^{k,p}\to W^{k-1,p}$ for some $(k,p)$.
\item\eqref{lambdazero} is surjective onto the cokernel of $D\delbar(u,\cdot):W^{k,p}\to W^{k-1,p}$ for all $(k,p)$.
\end{rlist}
There exists such a pair $E_0$ and $\tilde\lambda_0$ satisfying these conditions since we may choose $k=1$ and $p>2$ and then remember that $C^\infty$ functions supported away from the nodes and marked points are dense in $W^{0,p}=L^p$.

Now pick any isomorphism to a fiber $\phi_\alpha:C\to\Cbar_\alpha$ and extend $\tilde\lambda_0$ to a map $\lambda_0:E_0\to C^\infty(\Cbar_\alpha\times X,\Omega^{0,1}_{\Cbar_\alpha/\Mbar_\alpha}\otimes_\CC TX)$ supported away from the nodes and marked points of the fibers.  Define $E_\alpha=E_0[\Gamma]$ and define $\lambda_\alpha:E_\alpha\to C^\infty(\Cbar_\alpha\times X,\Omega^{0,1}_{\Cbar_\alpha/\Mbar_\alpha}\otimes_\CC TX)$ by $\Gamma_\alpha$-linear extension from $\lambda_0$.  It is now clear by definition that $u:C\to X$ is covered by $\Mbar_{g,n}^\beta(X)_{\{\alpha\}}^\reg$, where $\alpha\in A^\GW$ is the element just constructed.
\end{proof}

We have now shown the following.

\begin{theorem}
$A^\GW$ is an implicit atlas on $\Mbar_{g,n}^\beta(X)$.
\end{theorem}

\subsection{Definition of Gromov--Witten invariants}

\begin{definition}[Gromov--Witten invariants]
Fix nonnegative integers $g$ and $n$ with $2g+n>2$, and fix $\beta\in H_2(X;\ZZ)$.  By Proposition \ref{GWgluingneeded}, the virtual orientation sheaf induced by $A^\GW$ is canonically trivialized.  Thus the implicit atlas $A^\GW$ induces a virtual fundamental class $[\Mbar_{g,n}^\beta(X)]^\vir\in\cH^\bullet(\Mbar_{g,n}^\beta(X);\QQ)^\vee$.  We define the \emph{Gromov--Witten invariant}:
\begin{equation}
\GW_{g,n}^\beta(X)\in H_\bullet(\Mbar_{g,n}\times X^n;\QQ)
\end{equation}
as the pushforward of $[\Mbar_{g,n}^\beta(X)]^\vir$ under the tautological map $\Mbar_{g,n}^\beta(X)\to\Mbar_{g,n}\times X^n$.  This is well-defined by the next lemma.
\end{definition}

\begin{lemma}
$\GW_{g,n}^\beta(X)\in H_\bullet(\Mbar_{g,n}\times X^n)$ is independent of the choice of $J$.
\end{lemma}

\begin{proof}
Let $J_0$ and $J_1$ be any two smooth $\omega$-tame almost complex structures on $X$.  We denote by $\Mbar_{g,n}^\beta(X,J_0)$ and $\Mbar_{g,n}^\beta(X,J_1)$ the corresponding moduli spaces of stable maps (denoted earlier by simply $\Mbar_{g,n}^\beta(X)$ when we considered just a fixed $J$).

There exists a smooth path of $\omega$-tame almost complex structures $J_{[0,1]}=\{J_t\}_{t\in[0,1]}$ connecting $J_0$ and $J_1$.  Let us consider the corresponding ``parameterized'' moduli space of stable maps $\Mbar_{g,n}^\beta(X,J_{[0,1]})$.  The construction from \S\ref{gromovwittenimplicitatlassubsection} gives an implicit atlas with boundary $A^\GW$ on $\Mbar_{g,n}^\beta(X,J_{[0,1]})$ whose restriction to $\partial\Mbar_{g,n}^\beta(X,J_{[0,1]}):=\Mbar_{g,n}^\beta(X,J_0)\sqcup\Mbar_{g,n}^\beta(X,J_1)$ agrees with $A^\GW$ on these spaces.  Hence it follows using Lemmas \ref{disjointunionfc} and \ref{cobordismworks} that $[\Mbar_{g,n}^\beta(X,J_0)]^\vir=[\Mbar_{g,n}^\beta(X,J_1)]^\vir$ in $\cH^\bullet(\Mbar_{g,n}^\beta(X,J_{[0,1]}))^\vee$, which is enough.
\end{proof}

\section{Hamiltonian Floer homology}\label{hamiltonianfloersec}

In this section, we define Hamiltonian Floer homology for a general closed symplectic manifold $M$ (which we now fix).  We also calculate Hamiltonian Floer homology using the $S^1$-localization idea of Floer, and we derive the Arnold conjecture from this calculation.  These results (in this generality) are originally due to Liu--Tian \cite{liutian}, Fukaya--Ono \cite{fukayaono}, and Ruan \cite{ruan}.  For a general introduction to Hamiltonian Floer homology, the reader may consult Salamon \cite{salamonnotes} (we assume some familiarity with the basic theory).

The main content of this section is the construction of implicit atlases on the relevant spaces of stable pseudo-holomorphic cylinders.  Once we do this, the definition from \S\ref{homologygroupssection} gives the desired homology groups.  We also construct $S^1$-equivariant implicit atlases on the moduli spaces for time-independent Hamiltonians.  This allows us to use the $S^1$-localization results of \S\ref{Slocalizationsection} to show that Hamiltonian Floer homology coincides with Morse homology (a standard corollary of which is the Arnold conjecture).  The necessary gluing results are stated in Propositions \ref{HFgluingneeded} and \ref{HFgluingSS} and are proved in Appendix \ref{gluingHF}.

It would be interesting to define this isomorphism as in Piunikhin--Salamon--Schwarz \cite{pss} using their moduli spaces of ``spiked disks'' (this route avoids the use of $S^1$-localization).

\subsection{Preliminaries}

\begin{definition}[Abelian cover of free loop space]
Let $L_0M$ denote the space of null-homotopic smooth maps $S^1\to M$, and let $\widetilde{L_0M}$ denote the space of such loops together with a homology class of bounding $2$-disk.  Then $\widetilde{L_0M}\to L_0M$ is a $\pi$-cover where $\pi=\im(\pi_2(M)\to H_2(M;\ZZ))$.\footnote{We could just as easily work on the smaller cover corresponding to the image of $\omega\oplus c_1(M):\pi_2(M)\to\RR\oplus\ZZ$.  The corresponding equivalence relation is that $f_1,f_2:D^2\to M$ with $f_1|_{S^1}=f_2|_{S^1}$ are equivalent iff $\omega$ and $c_1(M)$ both vanish on ``$f_1-f_2$''$\in\pi_2(M)$.}
\end{definition}

\begin{definition}[Hamiltonian flows]
For a smooth function $H:M\times S^1\to\RR$, let $X_H:M\times S^1\to TM$ denote the Hamiltonian vector field induced by $H$, and let $\phi_H:M\to M$ denote the time $1$ flow map of $X_H$.  A \emph{periodic orbit} of $H$ is a smooth function $\gamma:S^1\to M$ satisfying $\gamma'(t)=X_{H(t)}(\gamma(t))$.  Let $C^\infty(M\times S^1)^\reg\subseteq C^\infty(M\times S^1)$ denote those functions $H$ for which $\phi_H$ has non-degenerate fixed points.
\end{definition}

\begin{definition}[Simplicial sets of $H$ and $J$]\label{JHdef}
Define the simplicial set $H_\bullet(M)$ where $H_n(M)$ is the set of smooth functions $H:\Delta^n\to C^\infty(M\times S^1)$ which are constant near the vertices and send the vertices to $C^\infty(M\times S^1)^\reg$.  Define the simplicial set $J_\bullet(M)$ where $J_n(M)$ is the set of smooth functions $J:\Delta^n\to J(M)$ which are constant near the vertices ($J(M)$ is the space of smooth almost complex structures tamed by $\omega$) and which send the vertices to almost complex structures which are $\omega$-compatible.

It is easy to see that $H_\bullet(M)$ and $J_\bullet(M)$ are both \emph{contractible Kan complexes}.  A semisimplicial set $Z_\bullet$ is a contractible Kan complex iff every map $\partial\Delta^n\to Z_\bullet$ can be extended to a map $\Delta^n\to Z_\bullet$ for all $n\geq 0$ (where $\Delta^n$ is the semisimplicial $n$-simplex).

Let $\JH_\bullet(M)=J_\bullet(M)\times H_\bullet(M)$, which of course is also a contractible Kan complex.
\end{definition}

\begin{definition}[Standard Morse function on $\Delta^n$]\label{simplexflows}
For this definition, let us view the $n$-simplex $\Delta^n$ as:
\begin{equation}
\Delta^n=\{\underline x\in[0,1]^{n+1}:0=x_0\leq\cdots\leq x_n\leq 1\}
\end{equation}
The $i$th vertex of $\Delta^n$ is given by $x_{n-i}=0$ and $x_{n-i+1}=1$.  We now consider the Morse function on $\Delta^n$ given by $f(x):=\sum_{i=1}^n\cos\pi x_i$.  Its gradient:
\begin{equation}\label{simplexmorsefunctiongradient}
\nabla f(x)=\sum_{i=1}^n\pi\sin(\pi x_i)\frac\partial{\partial x_i}
\end{equation}
is tangent to the boundary of $\Delta^n$, and its critical points are precisely the vertices of $\Delta^n$, the index at vertex $i$ being $n-i$.  Note also that for any facet inclusion $\Delta^k\hookrightarrow\Delta^n$, the pushforward of $\nabla f$ is again $\nabla f$.

Let us consider the space $\F(\Delta^n)$ of broken Morse flow lines from vertex $0$ (index $n$) to vertex $n$ (index $0$).  This space is homeomorphic to a cube $[0,1]^{n-1}$, the factors $[0,1]$ being naturally indexed by the ``middle vertices'' $1,\ldots,n-1$ of $\Delta^n$.  A flow line is broken at a vertex $i\in\{1,\ldots,n-1\}$ iff the corresponding coordinate in $[0,1]^{n-1}$ equals $1$.  Note that for any simplex $\sigma$ there are canonical compatible product/face maps:
\begin{align}
\label{Fproduct}\F(\sigma|[0\ldots k])\times\F(\sigma|[k\ldots n])&\to\F(\sigma)\\
\label{Fface}\F(\sigma|[0\ldots\hat k\ldots n])&\to\F(\sigma)
\end{align}
See also \S\ref{simplexflowsorsec} for more details.  Adams \cite{adamscobar} considered spaces of paths on $\Delta^n$ with the same key properties (though it is important that our flow lines are smooth, whereas Adams' are not).
\end{definition}

\subsection{Moduli space of Floer trajectories}

Let us now define the flow category diagram (Definition \ref{semisimplicialfcdef}) which gives rise to Hamiltonian Floer homology.

\begin{definition}
For $H\in C^\infty(M\times S^1)^\reg$, let $\PPP_H\subseteq\widetilde{L_0M}$ denote the set of null-homotopic periodic orbits equipped with a homology class of bounding $2$-disk.
\end{definition}

In the following definition, the reader may prefer to focus on the cases $n=0$ (Floer trajectories relevant for the differential), $n=1$ (Floer trajectories relevant for the continuation maps), and $n=2$ (Floer trajectories relevant for the homotopies between continuation maps).

\begin{definition}[Floer trajectory]\label{floertrajdef}
Let $\sigma\in J_n(M)\times H_n(M)$ be an $n$-simplex; we denote by $H^\sigma:\Delta^n\times M\times S^1\to\RR$ and $J^\sigma:\Delta^n\to J(M)$ the corresponding smooth families.  Let $p\in\PPP_{H_0}$ and $q\in\PPP_{H_n}$ be periodic orbits, where $H_i=H^\sigma(i\in\Delta^n,\cdot,\cdot)$ is the Hamiltonian associated to the $i$th vertex of $\Delta^n$.  A \emph{Floer trajectory of type $(\sigma,p,q)$} is a triple $(C,\ell,u)$ where:
\begin{rlist}
\item$C$ is a nodal curve of type $(0,2)$.  Let us call the two marked points $x^-,x^+\in C$, and let $k=k(C)$ be the number of vertices (irreducible components of $C$) on the unique path from $x^-$ to $x^+$ in the dual graph of $C$.
\item$\ell:\coprod_{i=1}^k\RR\to\Delta^n$ is a broken Morse flow line from vertex $0$ to vertex $n$ (for the Morse function in Definition \ref{simplexflows}).  Let $0=v_0\leq\cdots\leq v_k=n$ be the corresponding sequence of vertices.  We allow $\ell$ to contain constant flow lines, i.e.\ we allow $v_i=v_{i+1}$.
\item\label{buildingdef}$u:C\to M\times S^1\times\coprod_{i=1}^k\RR$ is a smooth \emph{building} of type $(\sigma,p,q)$, by which we mean the following.  Let $C^\circ$ be $C$ punctured at $\{x^-,x^+\}$ and at the nodes corresponding to the edges in the unique path in the dual graph of $C$ from $x^-$ to $x^+$.  The connected components $\{C^\circ_1,\ldots,C^\circ_k\}$ are naturally ordered ($x^-$ on the $1$st component and $x^+$ on the $k$th component).  There must be periodic orbits $\{\gamma_0,\ldots,\gamma_k\}$ where $\gamma_i\in\PPP_{H_{v_i}}$ with $\gamma_0=p$ and $\gamma_k=q$.  Then the negative (resp.\ positive) end of $C^\circ_i$ must be asymptotic to $(\gamma_{i-1}(t),t)$ (resp.\ $(\gamma_i(t),t)$) (with multiplicity one).  We also require that $u$ have ``finite energy''.  In addition, $u|C_i^\circ$ must be in the correct homology class: the element of $\pi_2(M)$ obtained by gluing together $u|C_i^\circ$ (resolving any nodes of $C_i^\circ$) with the given disks bounding $\gamma_i$ and $\gamma_{i-1}$ must vanish in homology.
\item$u$ is pseudo-holomorphic with respect to the almost complex structure on $M\times S^1\times\coprod_{i=1}^k\RR$ defined as follows.  Use coordinates $(t,s)\in S^1\times\coprod_{i=1}^k\RR$.  Fix the standard almost complex structure on $S^1\times\coprod_{i=1}^k\RR$, namely $J_{S^1\times\coprod_{i=1}^k\RR}(\frac\partial{\partial s})=\frac\partial{\partial t}$.  Also fix the ($s$-dependent) almost complex structure $J^\sigma(\ell(s))$ on $M$.  We let $A:T[S^1\times\coprod_{i=1}^k\RR]\to TM$ be defined by $A(\frac\partial{\partial s})=X_{H^\sigma(\ell(s),t,\cdot)}$ and extended anti-holomorphically.  Now we use the following almost complex structure on $M\times S^1\times\coprod_{i=1}^k\RR$:
\begin{equation}\label{compositeJoncylinder}
J=\left(\begin{matrix}
J^\sigma(\ell(s))&A\cr
0&J_{S^1\times\coprod_{i=1}^k\RR}
\end{matrix}\right)
\end{equation}
Note that the projection $M\times S^1\times\coprod_{i=1}^k\RR\to S^1\times\coprod_{i=1}^k\RR$ is holomorphic.
\end{rlist}
An \emph{isomorphism} $\iota:(C,\ell,u)\to(C',\ell',u')$ of Floer trajectories is an isomorphism $\iota_1:C\to C'$ of curves of type $(0,2)$ and an isomorphism $\iota_2:\coprod_{i=1}^k\RR\to\coprod_{i=1}^k\RR$ (acting by translation on each factor and respecting the ordering of the terms; note that the existence of $\iota_1$ implies that $k(C)=k(C')$) such that $u=(\id_{M\times S^1}\times\iota_2^{-1})\circ u'\circ\iota_1$ and $\ell=\ell'\circ\iota_2$.  We say a Floer trajectory is \emph{stable} iff its automorphism group (i.e.\ group of self-isomorphisms) is finite.
\end{definition}

\begin{definition}
We define a $\pi$-equivariant flow category diagram $\Mbar/\JH_\bullet(M)$ as follows.
\begin{rlist}
\item For a vertex $(J,H)\in J_0(M)\times H_0(M)$, we let $\PPP_{(J,H)}=\PPP_H$.
\item The grading $\gr:\PPP\to\ZZ$ is the usual Conley--Zehnder index, and $\gr:\pi\to\ZZ$ is given by $\gr(h)=2\langle c_1(TM),h\rangle$.
\item The action $a:\PPP\to\RR$ is the usual symplectic action, and $a:\pi\to\RR$ is given by $a(h)=\langle\omega,h\rangle$.
\item We let $\Mbar(\sigma,p,q)$ be the set of stable Floer trajectories of type $(\sigma,p,q)$, equipped with the Gromov topology.  It is well-known that $\Mbar(\sigma,p,q)$ is compact Hausdorff.  The finiteness conditions required on $\Mbar(\sigma,p,q)$ also follow from Gromov compactness.  The product/face maps on $\Mbar(\sigma,p,q)$ are evident.
\item The action of $\pi$ on everything is clear.
\end{rlist}
\end{definition}

\subsection{Implicit atlas}\label{HFIAsubsec}

Let us now define an implicit atlas on the flow category diagram $\Mbar/\JH_\bullet(M)$ (recall Definition \ref{IAonflowcategorydiagramdef}).  This construction follows the same outline as the construction of an implicit atlas on the moduli space of stable maps in \S\ref{gromovwittenimplicitatlassubsection} (the main difference being that here there is more notation to keep track of).  Note that the flow category diagram on which we will define an implicit atlas is no longer denoted $\X/Z_\bullet$, and this leads to a few (evident) notational differences from \S\ref{homologygroupssection} where we considered implicit atlases on flow category diagrams abstractly.

\begin{definition}[Index set $A^\HF(\Delta^n)$]\label{implicitatlasonFloer}
A \emph{(Hamiltonian Floer) thickening datum} $\alpha$ on the simplex $\Delta^n$ is a quadruple $(D_\alpha,r_\alpha,E_\alpha,\lambda_\alpha)$ where:
\begin{rlist}
\item\label{implicitatlasFloerD}$D_\alpha\subseteq M\times S^1\times\Delta^n$ is a compact smooth submanifold with corners locally modeled on $\RR_{\geq 0}^N\times\RR^{N'}\subseteq\RR_{\geq 0}^N\times\RR^{N'+2}$ or $\RR_{\geq 0}^{N+1}\times\RR^{N'}\subseteq\RR_{\geq 0}^N\times\RR^{N'+3}$.  Let us denote by $\partial^\ess D_\alpha\subseteq\partial D_\alpha$ the closure of $\partial D_\alpha\setminus[M\times S^1\times\partial\Delta^n]$ (which is precisely the set of points with local model of the second type).
\item$r_\alpha\geq 1$ is an integer; let $\Gamma_\alpha=S_{r_\alpha}$.
\item$E_\alpha$ is a finitely generated $\RR[S_{r_\alpha}]$-module.
\item$\lambda_\alpha:E_\alpha\to C^\infty(\Cbar_{0,2+r_\alpha}\times M,\Omega^{0,1}_{\Cbar_{0,2+r_\alpha}/\Mbar_{0,2+r_\alpha}}\otimes_\RR TM)$ is an $S_{r_\alpha}$-equivariant linear map supported away from the nodes and marked points of the fibers of $\Cbar_{0,2+r_\alpha}\to\Mbar_{0,2+r_\alpha}$ (the universal family).
\end{rlist}
Let $A^\HF(\Delta^n)$ denote the set of all thickening datums on $\Delta^n$.
\end{definition}

We define a $\pi$-equivariant implicit atlas $\A$ on the flow category diagram $\Mbar/\JH_\bullet(M)$ as follows.  We set $\Abar(\sigma,p,q):=A^\HF(\sigma)$.  Let us now define the implicit atlas on $\Mbar(\sigma,p,q)^{\leq\s}$ with index set $\A(\sigma,p,q)^{\geq\s}$ (built from $\Abar(\sigma,p,q)$ as in \eqref{coproductforA}).

\begin{definition}[Atlas data for $\A(\sigma,p,q)^{\geq\s}$ on $\Mbar(\sigma,p,q)^{\leq\s}$]\label{IAonfloerthickeningsdef}
For $I\subseteq\A(\sigma,p,q)^{\geq\s}$, an \emph{$I$-thickened Floer trajectory of type $(\sigma,p,q)^{\leq\s}$} is a $5$-tuple $(C,\ell,u,\{\phi_\alpha\}_{\alpha\in I},\{e_\alpha\}_{\alpha\in I})$ where:
\begin{rlist}
\item$C$ is a nodal curve of type $(0,2)$.
\item$\ell:\coprod_{i=1}^k\RR\to\Delta^n$ is a broken flow line from vertex $0$ to vertex $n$.
\item\label{hfiacalphapart}$u:C\to M\times S^1\times\coprod_{i=1}^k\RR$ is a smooth building of type $(\sigma,p,q)$ (in the sense of Definition \ref{floertrajdef}(\ref{buildingdef})), with combinatorial type of $u$ belonging to $\SSS_{\Mbar}(\sigma,p,q)^{\leq\s}$.  Recall that by definition (see \eqref{coproductforA}), $\A(\sigma,p,q)^{\geq\s}$ is a disjoint union of various $\Abar(\sigma|[i_0\ldots i_n],p',q'):=A^\HF(\sigma|[i_0\ldots i_n])$.  Hence any given $\alpha\in I$ comes from one of these, say $\Abar(\sigma|[i_0^\alpha\ldots i_n^\alpha],p'_\alpha,q'_\alpha)$.  Let $C_\alpha\subseteq C$ denote the union of irreducible components corresponding to this triple, which exists because the combinatorial type of $u$ belongs to $\SSS_{\Mbar}(\sigma,p,q)^{\leq\s}$ (this $C_\alpha$ is a key notion for the present construction of an implicit atlas).
\item\label{transversalitypartIAfloer} For all $\alpha\in I$, we must have $u|_{C_\alpha}\pitchfork D_\alpha$ with exactly $r_\alpha$ intersections.  By $u|_{C_\alpha}\pitchfork D_\alpha$, we mean that under the map $(\id_{M\times S^1}\times\ell)\circ u:C_\alpha\to M\times S^1\times\Delta^{\sigma|[i_0\ldots i_m]}$, we have\footnote{The closure of the image $\overline{C_\alpha}$ is precisely the image of $C_{\alpha}$ union the asymptotic periodic orbits.} $\overline{C_\alpha}\cap\partial^\ess D_\alpha=\varnothing$ and for every point $p\in C_\alpha$ mapping to $D_\alpha$, the derivative $d((\id_{M\times S^1}\times\ell)\circ u):T_pC\to T_{u(p)}[M\times S^1\times\Delta^{\sigma|[i_0\ldots i_m]}]/T_{u(p)}D_\alpha$ is surjective and $p$ is not a node or marked point of $C$.
\item\label{hfiaxpxmpart}$\phi_\alpha:C_\alpha\to\Cbar_{0,2+r_\alpha}$ is an isomorphism with a fiber (where $C_\alpha$ is considered to have two marked points $x^-,x^+$ corresponding to $p_\alpha',q_\alpha'$ plus the $r_\alpha$ marked points $(u|C_\alpha)^{-1}(D_\alpha)$).
\item$e_\alpha\in E_\alpha$.
\item The following \emph{$I$-thickened $\delbar$-equation} is satisfied:
\begin{equation}\label{approxJholforHF}
\delbar u+\sum_{\alpha\in I}\lambda_\alpha(e_\alpha)(\phi_\alpha,u)=0
\end{equation}
where we use the almost complex structure on $M\times S^1\times\coprod_{i=1}^k\RR$ defined in \eqref{compositeJoncylinder}.  The term $\lambda_\alpha(e_\alpha)(\phi_\alpha,u)$ only makes sense over $C_\alpha$; we interpret it as zero over the rest of $C$.  Note that for \eqref{approxJholforHF}, we project $\lambda_\alpha$ onto $\Omega^{0,1}_{\Cbar_{0,2+r_\beta}/\Mbar_{0,2+r_\beta}}\otimes_\CC TM_{J^\sigma(\ell(s))}$.
\end{rlist}
An \emph{isomorphism} between two $I$-thickened Floer trajectories $(C,\ell,u,\{e_\alpha\},\{\phi_\alpha\})$ and $(C',\ell',u',\{e_\alpha'\},\{\phi_\alpha'\})$ is an isomorphism $\iota_1:C\to C'$ of curves of type $(0,2)$ and an isomorphism $\iota_2:\coprod_{i=1}^k\RR\to\coprod_{i=1}^k\RR$ (acting by translation on each factor and respecting the ordering of the terms; note that the existence of $\iota_1$ implies that $k(C)=k(C')$) such that $u=(\id_{M\times S^1}\times\iota_2^{-1})\circ u'\circ\iota_1$, $\ell=\ell'\circ\iota_2$, $\phi_\alpha=\phi_\alpha'\circ\iota_1$, and $e_\alpha=e_\alpha'$ for all $\alpha\in I$.  We say an $I$-thickened Floer trajectory is \emph{stable} iff its automorphism group (i.e.\ group of self-isomorphisms) is finite.

Let $\Mbar(\sigma,p,q)^{\leq\s}_I$ denote the set of stable $I$-thickened Floer trajectories of type $(\sigma,p,q)^{\leq\s}$, and equip it with the Gromov topology.  The actions of $\Gamma_I$ on the thickened moduli spaces, the functions $s_I$, the projections $\psi_{IJ}$, and the sets $U_{IJ}$ are all defined as in Definition \ref{GWatlasdatadef} from the Gromov--Witten setting.

The stratification of $\Mbar(\sigma,p,q)^{\leq\s}_I$ by $\SSS_{\Mbar}(\sigma,p,q)^{\leq\s}$ is evident.
\end{definition}

The compatibility axioms for $\A(\sigma,p,q)^{\geq\s}$ on $\Mbar(\sigma,p,q)^{\leq\s}$ are all immediate; the homeomorphism axiom can again be verified directly as in the Gromov--Witten setting.

The regular loci for $\A(\sigma,p,q)^{\geq\s}$ on $\Mbar(\sigma,p,q)^{\leq\s}$ are defined following Definition \ref{regularsubsetdefinition}, meaning that a $5$-tuple $(C,\ell,u,\{\phi_\alpha\}_{\alpha\in I},\{e_\alpha\}_{\alpha\in I})$ is regular iff it has trival automorphism group and the linearized operator (fixing $C$ and varying $u$, $\ell$, and $\{e_\alpha\}_{\alpha\in I}$) is surjective (see \S\ref{regularlocusdef} for more details).  Let $\vdim\Mbar(\sigma,p,q)^{\leq\s}:=\gr(q)-\gr(p)+\dim\sigma-1-\codim\s$.

Let us now discuss the (nontrivial) transversality axioms for $\A(\sigma,p,q)^{\geq\s}$ on $\Mbar(\sigma,p,q)^{\leq\s}$.  The openness and submersion axioms follow from the following result, whose proof is given in Appendix \ref{gluingHF}.

\begin{proposition}[Formal regularity implies topological regularity]\label{HFgluingneeded}
For all $I\subseteq J\subseteq\A(\sigma,p,q)^{\geq\s}$, we have:
\begin{rlist}
\item$(\Mbar(\sigma,p,q)^{\leq\s}_I)^\reg\subseteq\Mbar(\sigma,p,q)^{\leq\s}_I$ is an open subset.
\item The map $s_{J\setminus I}:\Mbar(\sigma,p,q)^{\leq\s}_J\to E_{J\setminus I}$ over the locus $\psi_{IJ}^{-1}((\Mbar(\sigma,p,q)_I^{\leq\s})^\reg)\subseteq\Mbar(\sigma,p,q)^{\leq\s}_J$ is locally modeled on the projection:
\begin{equation}
\RR^{\vdim\Mbar(\sigma,p,q)^{\leq\s}+\dim E_I}\times\RR^{\dim E_{J\setminus I}}\to\RR^{\dim E_{J\setminus I}}
\end{equation}
over the top stratum $\s\in\SSS_{\Mbar}(\sigma,p,q)^{\leq\s}$.  More generally, the local model (compatible with stratifications) is given by replacing the first factor on the left by $\RR_{\geq 0}^n\times\RR^{n'}$ stratified appropriately by $\SSS_{\Mbar}(\sigma,p,q)$.
\item There exist $\pi$-invariant coherent trivializations of the local systems $\oo_{(\Mbar(\sigma,p,q)^{\leq\s}_I)^\reg}\otimes\oo_{E_I}^\vee$ (in the sense of Definition \ref{cohordefabstract}).
\end{rlist}
\end{proposition}

Finally, to verify the covering axiom, we use the general strategy from Lemma \ref{everythingcanbethickened} in the Gromov--Witten case.  To apply this in the present setting, we just need the following stabilization lemma to take the place of Lemma \ref{stabilizationlemma}.

\begin{lemma}[Domain stabilization for stable Floer trajectories]\label{stabilizationlemmaforHF}
Let $\ell:\coprod_{i=1}^k\RR\to\Delta^n$ and $u:C\to M\times S^1\times\coprod_{i=1}^k\RR$ be a point in $\Mbar(\sigma,p,q)$.  Then there exists $D\subseteq M\times S^1\times\Delta^n$ as in Definition \ref{implicitatlasonFloer}(\ref{implicitatlasFloerD}) such that $C\pitchfork D$ in the sense of Definition \ref{IAonfloerthickeningsdef}(\ref{transversalitypartIAfloer}) and so that adding the intersections to $C$ as extra marked points makes $C$ stable.
\end{lemma}

\begin{proof}
As in the proof of Lemma \ref{stabilizationlemma}, we use (an appropriate variant for manifolds with corners of) Lemma \ref{sortofsard} (which we may also use to avoid the periodic orbits in question).  It thus suffices to show that for any unstable component $C_0$ of $C$, there exists a point $p\in C_0$ where $d((\id_{M\times S^1}\times\ell)\circ u):TC_0\to TM\times TS^1\times T\Delta^n$ injective.  To find such a point, we split into two cases.

First, suppose $u:C_0\to S^1\times\RR$ is constant.  Then $u:C_0\to M$ is a (nonconstant!) $J$-holomorphic sphere, and thus has a point of injectivity of $du$.

Second, suppose $u:C_0\to S^1\times\RR$ is not constant.  Then $u:C_0\to S^1\times\RR$ is an isomorphism; let us use $(t,s)\in S^1\times\RR$ as coordinates on $C_0$.  Now $\frac{\partial u}{\partial t}\in TM\times TS^1$ has nonzero coordinate in the $TS^1$ component (everywhere), and $\frac{\partial u}{\partial s}\in TM\times TS^1$ has zero coordinate in the $TS^1$ coordinate (everywhere).  Hence if $du$ is everywhere noninjective, we find that $\frac{\partial u}{\partial s}=0\in TM\times TS^1$ everywhere.  It follows that $u:C_0\to M\times S^1$ is independent of the $\RR$ coordinate, and thus is simply a trivial cylinder mapping onto a periodic orbit.  Now the stability condition (the automorphism group being finite) implies that the corresponding (piece of a) flow line $\ell:\RR\to\Delta^n$ is nontrivial.  It follows that we have the desired injectivity.
\end{proof}

\begin{theorem}\label{IAonFloerworks}
$\A$ is a $\pi$-equivariant implicit atlas on $\Mbar/\JH_\bullet(M)$.
\end{theorem}

\begin{proof}
We have shown above that each individual $\A(\sigma,p,q)^{\geq\s}$ is an implicit atlas on $\Mbar(\sigma,p,q)^{\leq\s}$.  The required compatibility isomorphisms between these implicit atlases follow directly from the definition.
\end{proof}

\subsection{Definition of Hamiltonian Floer homology}

\begin{definition}[Hamiltonian Floer homology]
We have a $\pi$-equivariant flow category diagram $\Mbar/\JH_\bullet(M)$ equipped with an implicit atlas $\A$.  Moreover, Proposition \ref{HFgluingneeded} gives coherent orientations $\omega$.  Hence according to Definition \ref{defofhomologygroups}, we get a diagram $F\HH:\JH_\bullet(M)\to H^0(\Ch_{\QQ[[\pi]]})$ (the hypotheses of Definition \ref{defofhomologygroups} can be easily verified).  Since $\JH_\bullet(M)$ is a contractible Kan complex, this is really just a single object $F\HH^\bullet(M)\in H^0(\Ch_{\QQ[[\pi]]})$ which we call the \emph{Hamiltonian Floer homology of $M$}.
\end{definition}

\begin{remark}
The ring $\QQ[[\pi]]$ is the \emph{graded} completion of $\QQ[\pi]$, see Definition \ref{gradedcompletion}.
\end{remark}

\subsection{\texorpdfstring{$S^1$}{S{\textasciicircum}1}-invariant Hamiltonians}\label{SSham}

To calculate the Hamiltonian Floer homology $F\HH^\bullet(M)$ as defined above, we consider the case when $H$ is a (time-indepedent) Morse function on $M$.

Fix a smooth almost complex structure $J$ on $M$ compatible with $\omega$.  This induces a metric on $M$, so there is a notion of gradient flow line for smooth functions on $M$.

Let $H:M\to\RR$ be a Morse function for which the time $1$ Hamiltonian flow map of $H$ has non-degerate fixed points, all of which are critical points of $H$ (for example, $H=\epsilon\cdot H_0$ is such a function for any Morse function $H_0:M\to\RR$ and sufficiently small $\epsilon>0$).  Consider the inclusion $\ast\hookrightarrow\JH_\bullet(M)$ (where $\ast$ is the simplicial $0$-simplex, i.e.\ the simplicial set with a single $n$-simplex for all $n$) defined by mapping everything to the constant families of almost complex structures and Hamiltonians given by $J$ and $H$.  We will restrict attention to the (pullback) flow category diagram $\Mbar/\ast$.  With our assumptions on $H$, the set of generators is $\PPP=\crit(H)\times\pi$ canonically, and the grading on $\PPP$ is given by the Morse index on $\crit(H)$ plus $\gr:\pi\to\ZZ$.

Now there is a canonical $S^1$-action on the spaces of stable Floer trajectories in $\Mbar/\ast$ (postcompose $u$ with a rotation of $S^1$) which is compatible with the product/face maps (this action exists since $H$ is independent of the $S^1$-coordinate).  It follows that $\Mbar^{S^1}/\ast$ is also a flow category diagram (defined using the fixed locus $\Mbar(\sigma,p,q)^{S^1}\subseteq\Mbar(\sigma,p,q)$).

\begin{definition}[Morse trajectory]\label{morseTdef}
Let $\sigma\in\ast$ be the $n$-simplex, and let $p,q\in\PPP=\crit(H)\times\pi$.  A \emph{Morse trajectory of type $(\sigma,p,q)$} is a triple $(k,\ell,u)$ where:
\begin{rlist}
\item$k\geq 1$ is a positive integer.
\item$\ell:\coprod_{i=1}^k\RR\to\Delta^n$ is a broken Morse flow line from vertex $0$ to vertex $n$ (for the Morse function in Definition \ref{simplexflows}).  Let $0=v_0\leq\cdots\leq v_k=n$ be the corresponding sequence of vertices.  We allow $\ell$ to contain constant flow lines, i.e.\ we allow $v_i=v_{i+1}$.
\item$u:\coprod_{i=1}^k\RR\to M$ is a broken Morse flow line from $p$ to $q$ for the function $H$.  We allow $u$ to contain constant flow lines.
\item The $\pi$-components of $p$ and $q$ agree.
\end{rlist}
An \emph{isomorphism} $\iota:(k,\ell,u)\to(k',\ell',u')$ of Morse trajectories is an isomorphism $\iota:\coprod_{i=1}^k\RR\to\coprod_{i=1}^{k'}\RR$ (acting by translation on each factor and respecting the ordering of the terms; we require $k=k'$) such that $u=u'\circ\iota$ and $\ell=\ell'\circ\iota$.  We say a Morse trajectory is \emph{stable} iff its automorphism group (i.e.\ group of self-isomorphisms) is finite.

Let $\Mbar_\morse(\sigma,p,q)$ denote the space of stable Morse trajectories.
\end{definition}

\begin{proposition}[Formal regularity implies topological regularity]\label{morseassumptionMS}
Suppose $H$ is Morse--Smale.  Then with the stratification by $k$, the spaces $\Mbar_\morse(\sigma,p,q)$ are compact topological manifolds with corners.
\end{proposition}

For the case of the zero simplex $\sigma=\sigma^0$, this is proved by Wehrheim \cite{wehrheimMS}.  In fact, the general case also follows from \cite{wehrheimMS} since $\Mbar_\morse(\sigma,p,q)$ is the space of broken Morse flow lines on $M\times\Delta^n$ from $p\times 0$ to $q\times n$ for the Morse function $H+f$ (and $f$ is defined and smooth on a neighborhood of $\Delta^n\subseteq\RR^n$).  Alternatively, one may restrict to $H$ which have a particular normal form near each critical point in which case this result holds by more elementary arguments (see also \cite{wehrheimMS}).

\begin{lemma}\label{fixedismorse}
There is canonical homeomorphism $\Mbar(\sigma,p,q)^{S^1}=\Mbar_\morse(\sigma,p,q)$.
\end{lemma}

\begin{proof}
Suppose a stable Floer trajectory $(C,\ell,u)$ is $S^1$-invariant.  Then each of the components $\{C_1^\circ,\ldots,C_k^\circ\}$ must be smooth (no nodes) and hence isomorphic to $S^1\times\RR$.  Using the holomorphic projection $M\times S^1\times\coprod_{i=1}^k\RR\to S^1\times\coprod_{i=1}^k\RR$, we get holomorphic identifications $C_i^\circ=S^1\times\RR$.  Since $(C,\ell,u)$ is $S^1$-invariant, the function $u:C_i^\circ\to M$ must be independent of the $S^1$-coordinate, and hence (examining the $\delbar$-equation) is simply a Morse flow line of $H$.  Running this argument in reverse, we also see that every stable Morse trajectory gives rise to a stable Floer trajectory which is $S^1$-invariant.
\end{proof}

We will also need the following deeper fact (and henceforth we assume that $H$ is defined as in Lemma \ref{morseistransverse}):

\begin{lemma}\label{morseistransverse}
Fix a Morse function $H_0:M\to\RR$ whose gradient flow is Morse--Smale and suppose $H=\epsilon\cdot H_0$ with $\epsilon>0$ sufficiently small.  Then $\Mbar(\sigma,p,q)^{S^1}\subseteq\Mbar(\sigma,p,q)$ is open and cut out transversally (meaning $\Mbar(\sigma,p,q)^{S^1}\subseteq\Mbar(\sigma,p,q)^\reg$).  Hence we have a partition into disjoint closed subsets:
\begin{equation}
\Mbar(\sigma,p,q)=\Mbar(\sigma,p,q)^{S^1}\sqcup[\Mbar(\sigma,p,q)\setminus\Mbar(\sigma,p,q)^{S^1}]
\end{equation}
\end{lemma}

\begin{proof}
Salamon--Zehnder \cite[p1342, Theorem 7.3(1)]{salamonzehnder} show that $\Mbar(\sigma^0,p,q)^{S^1}\subseteq\Mbar(\sigma^0,p,q)^\reg$ for the $0$-simplex $\sigma^0$ and sufficiently small $\epsilon>0$.  Since we only consider constant families of $J$ and $H$, the linearized operators for the higher $\sigma^i$ can be written in terms of the linearized operators for $\sigma^0$, and it follows that $\Mbar(\sigma,p,q)^{S^1}\subseteq\Mbar(\sigma,p,q)^\reg$ for all $\sigma$.

Now we have $\Mbar_\morse(\sigma,p,q)=\Mbar(\sigma,p,q)^{S^1}\subseteq\Mbar(\sigma,p,q)^\reg$.  It remains to show that this is an open inclusion.  This can likely be seen by following closely the gluing argument used to prove Propositions \ref{HFgluingneeded} and \ref{morseassumptionMS}, however we can argue directly as follows.  Both $\Mbar_\morse(\sigma,p,q)$ and $\Mbar(\sigma,p,q)^\reg$ are topological manifolds with corners, and the stratification for the former is the pullback of the stratification for the latter.  Their dimensions are dictated by the Morse index and the Conley--Zehnder index respectively, which in this case coincide.  Hence we are done by Lemma \ref{cornersinvarianceofdomain} below.
\end{proof}

\begin{lemma}\label{cornersinvarianceofdomain}
Let $M$ be a topological manifold with corners and $K\subseteq M$ a closed subset.  Suppose that the restriction of the corner stratification on $M$ induces a topological manifold with corners structure on $K$ of the same dimension.  Then $K\subseteq M$ is open.
\end{lemma}

\begin{proof}
The question is local on $M$, so we may assume that $M=\RR^n\times\RR_{\geq 0}^m$.  Let $\tilde K\subseteq\RR^n\times\RR^m$ be obtained by reflecting $K$ across the last $m$ coordinate hyperplanes.  The hypotheses then imply that $\tilde K$ is a manifold of dimension $n+m$.  Now we have $\tilde K\subseteq\RR^{n+m}$ is open by Brouwer's ``invariance of domain''.  This is enough.
\end{proof}

\subsection{\texorpdfstring{$S^1$}{S{\textasciicircum}1}-equivariant implicit atlas}\label{SSequivariantFloeratlassec}

We constructed in \S\ref{HFIAsubsec} an implicit atlas $\A$ on $\Mbar/\JH_\bullet(M)$, and thus in particular on $\Mbar/\ast$.  In this subsection, we modify this construction to define another implicit atlas $\B^{S^1}$ on $\Mbar/\ast$, one which is $S^1$-equivariant in the sense that the $S^1$-action on $\Mbar(\sigma,p,q)$ extends canonically to all the thickenings $\Mbar(\sigma,p,q)^{\leq\s}_I$ in $\B^{S^1}$.  The key technical step is to perform domain stabilization with $S^1$-invariant divisors (Lemma \ref{stabilizationlemmaforHFSS}).

Let us first motivate the definition of $\B^{S^1}$ by describing a ``first attempt'' at defining an $S^1$-equivariant implicit atlas on $\Mbar/\ast$.  We consider the subatlas $\A^{S^1}\subseteq\A$ consisting of those thickening datums $\alpha$ for which $D_\alpha$ is $S^1$-invariant.  Now there is clearly a canonical $S^1$-action on the thickenings $\Mbar(\sigma,p,q)^{\leq\s}_I$ (postcomposition of $u$ with a rotation of $S^1$) for $I\subseteq\A^{S^1}(\sigma,p,q)^{\geq\s}$.\footnote{Note that we do not need to put any restrictions on $\lambda_\alpha$.}  Now $\A^{S^1}\subseteq\A$ forms an implicit atlas if and only if it satisfies the covering axiom.  However, the covering axiom for $\A^{S^1}$ fails: we cannot stabilize the domains of Morse flow lines (points of $\Mbar^{S^1}$) using $S^1$-invariant divisors $D_\alpha$ (more generally, we cannot stabilize the domain of any broken trajectory containing a Morse flow line).

To fix this issue, we first modify the definition of $\A$ to allow Morse components of Floer trajectories which do not get stabilized.

\begin{definition}[Implicit atlas $\B$ on $\Mbar/\ast$]\label{modifiedatlasforMorse}
We define an implicit atlas $\B$ on $\Mbar/\ast$ as follows.  On the level of index sets, we define $\B:=\A$.  However, we modify the definition of an $I$-thickened Floer trajectory as follows.  We require that when $C_\alpha$ is considered with the $r_\alpha$ extra marked points $(u|C_\alpha)^{-1}(D_\alpha)$, the only unstable components are mapped by $u$ to Morse flow lines; let $C_\alpha\to C_\alpha^\st$ be the map contracting all such unstable components.  Now instead of $\phi_\alpha:C_\alpha\to\Cbar_{0,2+r_\alpha}$, we use $\phi_\alpha:C_\alpha^\st\to\Cbar_{0,2+r_\alpha}$.

The rest of the atlas data is defined analogously with that of $\A$ without any serious difference.  Note, however, that to verify that the locus $U_{IJ}\subseteq X_I$ is open, we must appeal to Lemma \ref{morseistransverse}.
\end{definition}

Note that the thickened moduli spaces for $\A$ are open subsets of those for $\B$, so the covering axiom for $\A$ implies the covering axiom for $\B$.  Now to verify that $\B$ is an implicit atlas, everything is the same as for $\A$ except for the openness and submersion axioms, which follow from the following result (identical to Proposition \ref{HFgluingneeded}), whose proof is given in Appendix \ref{gluingHF}.

\begin{proposition}[Formal regularity implies topological regularity]\label{HFgluingSS}
For all $I\subseteq J\subseteq\B(\sigma,p,q)^{\geq\s}$, we have:
\begin{rlist}
\item$(\Mbar(\sigma,p,q)^{\leq\s}_I)^\reg\subseteq\Mbar(\sigma,p,q)^{\leq\s}_I$ is an open subset.
\item The map $s_{J\setminus I}:\Mbar(\sigma,p,q)^{\leq\s}_J\to E_{J\setminus I}$ over the locus $\psi_{IJ}^{-1}((\Mbar(\sigma,p,q)_I^{\leq\s})^\reg)\subseteq\Mbar(\sigma,p,q)^{\leq\s}_J$ is locally modeled on the projection:
\begin{equation}
\RR^{\vdim\Mbar(\sigma,p,q)^{\leq\s}+\dim E_I}\times\RR^{\dim E_{J\setminus I}}\to\RR^{\dim E_{J\setminus I}}
\end{equation}
over the top stratum $\s\in\SSS_{\Mbar}(\sigma,p,q)^{\leq\s}$.  More generally, the local model (compatible with stratifications) is given by replacing the first factor on the left by $\RR_{\geq 0}^n\times\RR^{n'}$ stratified appropriately by $\SSS_{\Mbar}(\sigma,p,q)$.
\item There exist $\pi$-invariant coherent trivializations of the local systems $\oo_{(\Mbar(\sigma,p,q)^{\leq\s}_I)^\reg}\otimes\oo_{E_I}^\vee$ (in the sense of Definition \ref{cohordefabstract}), agreeing by restriction with those for $\A$, and coinciding with the usual orientations from Morse theory on $\Mbar_\morse(\sigma,p,q)=\Mbar(\sigma,p,q)^{S^1}\subseteq\Mbar(\sigma,p,q)^\reg$.
\end{rlist}
\end{proposition}

Thus $\B$ is an implicit atlas on $\Mbar/\ast$.

\begin{definition}[$S^1$-equivariant implicit atlas $\B^{S^1}$ on $\Mbar/\ast$]
Let $\B^{S^1}\subseteq\B$ consist of those thickening datums $\alpha$ for which $D_\alpha$ is $S^1$-invariant.  There is a canonical $S^1$-action on the thickenings $\Mbar(\sigma,p,q)^{\leq\s}_I$ (postcomposition of $u$ with a rotation of $S^1$) for $I\subseteq\B^{S^1}(\sigma,p,q)^{\geq\s}$, and this $S^1$-action is compatible with the rest of the structure.
\end{definition}

To verify that $\B^{S^1}\subseteq\B$ is an implicit atlas, we just need to verify the covering axiom.  We follow the usual proof of the covering axiom as in Lemma \ref{everythingcanbethickened} and use the fact that Morse components are already cut out transversally (Lemma \ref{morseistransverse}).  To complete the proof, we just need Lemma \ref{stabilizationlemmaforHFSS} below, which says that for any stable Floer trajectory, we can stabilize the domain using $S^1$-invariant divisors (except, of course, for any irreducible components mapping to Morse flow lines).

\begin{lemma}[$S^1$-equivariant  domain stabilization for stable Floer trajectories of $\Mbar/\ast$]\label{stabilizationlemmaforHFSS}
Let $\ell:\coprod_{i=1}^k\RR\to\Delta^n$ and $u:C\to M\times S^1\times\coprod_{i=1}^k\RR$ be a point in $\Mbar(\sigma,p,q)$.  Then there exists $D\subseteq M\times S^1\times\Delta^n$ as in Definition \ref{implicitatlasonFloer}(\ref{implicitatlasFloerD}) which is $S^1$-invariant with $C\pitchfork D$ in the sense of Definition \ref{IAonfloerthickeningsdef}(\ref{transversalitypartIAfloer}) and so that adding these intersections to $C$ as extra marked points makes $C$ stable, except for irreducible components $S^1\times\RR\subseteq C$ on which $u$ is independent of the $S^1$-coordinate (``Morse flow lines'').
\end{lemma}

\begin{proof}
Instead of finding an $S^1$-invariant $D\subseteq M\times S^1\times\Delta^n$, we find a $D\subseteq M\times\Delta^n$ (which is clearly equivalent) and we ignore the $S^1$ factor in the codomain of $u$.  As in Lemma \ref{stabilizationlemmaforHF}, it suffices to show that for any unstable component $C_0\subseteq C$, either $C_0$ is a Morse flow line or there exists a point (other than a node) where $d((\id_M\times\ell)\circ u):TC_0\to TM\times T\Delta^n$ is injective.

If the projection $C_0\to S^1\times\RR$ is constant, then $u:C_0\to M$ is a (nonconstant!) $J$-holomorphic sphere and we are done as in Lemma \ref{stabilizationlemmaforHF}.  Hence it suffices to treat the case when $u:C_0\to S^1\times\RR$ is not constant.  Thus $u:C_0\to S^1\times\RR$ is an isomorphism, so we can use $(t,s)\in S^1\times\RR$ as coordinates on $C_0$.  Now we split into two cases.

First, suppose $\ell:\RR\to\Delta^n$ is not constant.  If $u:C_0\to M$ is independent of the $S^1$ coordinate, then it is a Morse flow line, and we do not need to stabilize.  Otherwise, there is a point where $\frac{\partial u}{\partial t}\ne 0$, and since $\ell:\RR\to\Delta^n$ has nonvanishing derivative everywhere, it follows that $du$ is injective at this point.

Second, suppose $\ell:\RR\to\Delta^n$ is constant.  Then our map $u:C_0\to M$ satisfies:
\begin{equation}\label{pseudoholoncylinderformorseHamiltonian}
\frac{\partial u}{\partial t}+J\circ\frac{\partial u}{\partial s}=\nabla H
\end{equation}
Certainly $u:C_0\to M$ is not constant; otherwise it would be unstable (infinite automorphism group).  Thus $du$ is nonzero somewhere.  If $du$ has rank two somewhere, then we are done.  Thus let us suppose that this is not the case and show that $u|C_0$ is independent of the $t$ coordinate (and thus is a Morse flow line).  Thus there exists some open set $U\subseteq C_0=S^1\times\RR$ where $du$ has rank $1$.  Inside $U$, we have $\ker du\subseteq TC_0=T(S^1\times\RR)$ is an (integrable!) $1$-dimensional distribution, so $U$ is equipped with a $1$-dimensional foliation and $u$ is constant on the leaves.  Thus we have (locally) a factorization $u:S^1\times\RR\xrightarrow r(-\epsilon,\epsilon)\xrightarrow wM$, the leaves of the foliation being given by $r^{-1}(\delta)$ for $\delta\in(-\epsilon,\epsilon)$.  Now \eqref{pseudoholoncylinderformorseHamiltonian} becomes:
\begin{equation}
\Bigl(\frac{\partial r}{\partial t}+\frac{\partial r}{\partial s}\cdot J\Bigr)\cdot w'(r(s,t))=(\nabla H)(w(r(s,t)))
\end{equation}
Since $du$ has rank $1$, we know that $w'(r(s,t))\ne 0$.  Hence the value of $r(s,t)$ determines the value of its derivative uniquely, i.e.\ $dr$ is constant along the leaves $r^{-1}(\delta)$ of the foliation.  It follows that the foliation is (locally) linear(!) and that we can follow any leaf infinitely in both directions and it never exits $U$ (since $dr=0$ outside $U$).  Now if any leaf had nonzero slope, it would force $u:C_0\to M$ to be constant, a contradiction.  Thus all leaves have slope zero; in other words $u$ is independent of the $S^1$-coordinate over $U=S^1\times U'$.  But now we see that $U'=\RR$, since if $U'$ had boundary, it would imply that we have a Morse flow line reaching a critical point in finite time.  Thus $u$ is (globally) a Morse flow line.
\end{proof}

\subsection{Calculation of Hamiltonian Floer homology and the Arnold conjecture}

Arnold conjectured that the minimal number of fixed points of a non-degenerate Hamiltonian symplectomorphism $M\to M$ enjoys a lower bound similar to the minimal number of critical points of a Morse function on $M$ (known as the Morse number of $M$).  It remains an open problem to obtain a sharp bound on the minimal number of symplectic fixed points, though much progress has been made.

Arnold's conjecture was proved for surfaces by Eliashberg \cite{eliashbergsurfaces} and for tori by Conley--Zehnder \cite{conleyzehnder}.  The existence of at least one fixed point was shown by Gromov \cite[p331, 2.3$B_4'$]{gromov} under the assumption $\omega|_{\pi_2(M)}=0$.

A major breakthrough was made by Floer \cite{floersfhs}, who introduced Hamiltonian Floer homology and showed (under some assumptions) that it is isomorphic to singular homology.  Floer's work provides a lower bound on the number of symplectic fixed points of the type predicted by Arnold.  Indeed, if Hamiltonian Floer homology can be defined and shown to be isomorphic to singular homology, then we get a lower bound towards the Arnold conjecture of the form:
\begin{equation}\label{mingenerators}
\min\left\{\rk D_\bullet\,\middle|\,(D_\bullet,d)\;\begin{matrix}\text{a free differential graded $\Lambda$-module}\hfill\cr\text{homotopy equivalent to }C_\bullet(M;\Lambda)\hfill\end{matrix}\right\}
\end{equation}
where $\Lambda$ is the Novikov ring:
\begin{equation}
\Lambda:=\ZZ[\im(\pi_2(M)\xrightarrow{\omega\oplus 2c_1(M)}\RR\oplus\ZZ)]^\wedge
\end{equation}
completed with respect to $\omega$ and graded by $2c_1$ (note that if the grading on $\Lambda$ is nontrivial, a ``differential graded $\Lambda$-module'' is not the same as a ``complex of $\Lambda$-modules'').  One can also adjoin $\pi_1(M)$ to the coefficient ring (as in Fukaya \cite{fukayagroupring} or Abouzaid \cite{abouzaidnearby}) to obtain a sharper lower bound in \eqref{mingenerators}, and furthermore the methods of Sullivan \cite{sullivangroupring} allow one (at least in many cases) to replace ``homotopy equivalent'' in \eqref{mingenerators} with ``simple homotopy equivalent''.  Note that for $\Lambda=\ZZ[\pi_1(M)]$ and ``simple homotopy equivalent'' in place of ``homotopy equivalent'', the lower bound \eqref{mingenerators} is precisely the stable Morse number of $M$ (see Damian \cite[p424, Corollary 2.6]{damian}).

Floer's original work \cite{floersfhs} covered the case of monotone symplectic manifolds (i.e.\ $\omega=\lambda c_1$ on $\pi_2(M)$ for some $\lambda>0$), and the work of Hofer--Salamon \cite{hofersalamon} and Ono \cite{ono} extended this to semi-positive symplectic manifolds (i.e.\ there do not exist classes $A\in\pi_2(M)$ with $\omega(A)>0$ and $3-n\leq c_1(A)<0$).  The case of general symplectic manifolds is due to Liu--Tian \cite{liutian}, Fukaya--Ono \cite{fukayaono}, and Ruan \cite{ruan}, using virtual techniques (which require rational coefficients) to resolve lack of transversality.  We reprove their results below using the VFC machinery developed in this paper.

In the following result, we use the definition of Floer-type homology groups from \S\ref{floerhomologyfromMFsets}.

\begin{theorem}\label{HFcalculation}
$F\HH^\bullet(M)$ is isomorphic to $H^\bullet(M;\ZZ)\otimes\QQ[[\pi]]$ as modules over $\QQ[[\pi]]$.
\end{theorem}

\begin{proof}
We use the setup of \S\ref{SSham}--\ref{SSequivariantFloeratlassec}.

The homology groups associated to the flow category diagram $\Mbar/\ast$ and the implicit atlas $\A$ are by definition $F\HH^\bullet(M)$.  Now, as we observed previously, the thickened moduli spaces of $\A$ are open subsets of those of $\B$, so by Lemma \ref{homologyopenreduction}, $F\HH^\bullet(M)$ may also be defined using the implicit atlas $\B$ on $\Mbar/\ast$.  Now $\B^{S^1}\subseteq\B$ is a subatlas, so by Lemma \ref{homologysameassubatlas} it may also be used to define $F\HH^\bullet(M)$.

Thus let us restrict attention to the atlas $\B^{S^1}$ on $\Mbar/\ast$.  Recall that by Lemma \ref{morseistransverse}, there is a partition into closed subsets:
\begin{equation}
\Mbar(\sigma,p,q)=\Mbar(\sigma,p,q)^{S^1}\sqcup[\Mbar(\sigma,p,q)\setminus\Mbar(\sigma,p,q)^{S^1}]
\end{equation}
Now we apply $S^1$-localization to $\B^{S^1}$ on $\Mbar/\ast$ in the form of Theorem \ref{circlelocalizationhomology}, which applies since $S^1$ acts with finite stabilizers on $\Mbar(\sigma,p,q)\setminus\Mbar(\sigma,p,q)^{S^1}$ and our coefficient group is $\QQ$.  It follows that $F\HH^\bullet(M)$ may be defined using the flow category diagram $\Mbar^{S^1}/\ast$ with the implicit atlas obtained from $\B^{S^1}$ by removing $\Mbar\setminus\Mbar^{S^1}$ from every thickening.

All the flow spaces of $\Mbar^{S^1}/\ast$ are cut out transversally by Lemma \ref{morseistransverse}, so by Proposition \ref{iftransversejustcountzero}, the homology groups of $\Mbar^{S^1}/\ast$ can be defined by simply counting the $0$-dimensional flow spaces according to the orientations $\omega$.  Now $\Mbar^{S^1}$ coincides with the Morse flow category diagram $\Mbar_\morse/\ast$ of $H$ by Lemma \ref{fixedismorse} (with the same orientations by Proposition \ref{HFgluingSS}), and this gives the desired isomorphism.
\end{proof}

\begin{remark}
One expects to be able show that the isomorphism in Theorem \ref{HFcalculation} is canonical by considering continuation maps and chain homotopies associated to families of Morse functions.
\end{remark}

\begin{corollary}[Arnold conjecture]
Let $H:M\times S^1\to\RR$ be a smooth function whose time $1$ Hamiltonian flow $\phi_H:M\to M$ has non-degenerate fixed points.  Then $\#\Fix\phi_H\geq\dim H_\bullet(M;\QQ)$ (in fact, we may replace $\Fix\phi_H$ with those fixed points whose associated periodic orbit is null-homotopic in $M$).
\end{corollary}

\begin{proof}
Pick any $\omega$-compatible almost complex structure $J$ and consider the vertex $(J,H)\in\JH_\bullet(M)$.  Over this vertex, pick any complex $F\CC^\bullet(M)$ which calculates $F\HH^\bullet(M)$.  Then $F\CC^\bullet(M)$ is a free $\QQ[[\pi]]$-module of rank:
\begin{equation}
\rk_{\QQ[[\pi]]}F\CC^\bullet(M)\leq\#\Fix\phi_H
\end{equation}
(its rank equals the number of null-homotopic periodic orbits).  By Theorem \ref{HFcalculation}, the homology $F\HH^\bullet(M)$ is free of rank:
\begin{equation}
\rk_{\QQ[[\pi]]}F\HH^\bullet(M)=\dim H_\bullet(M;\QQ)
\end{equation}
Now by definition, there is a $\QQ[[\pi]]$-linear boundary map $d:F\CC^\bullet(M)\to F\CC^{\bullet+1}(M)$ and by definition $\rk_{\QQ[[\pi]]}F\HH^\bullet(M)=\ker d/\im d$.  Now apply Lemma \ref{homologyhassmallerrank} to conclude that:
\begin{equation}
\rk_{\QQ[[\pi]]}F\HH^\bullet(M)\leq\rk_{\QQ[[\pi]]}F\CC^\bullet(M)
\end{equation}
Thus we are done.
\end{proof}

\subsection{A little commutative algebra}

\begin{lemma}\label{homologyhassmallerrank}
Let $M$ be a free module over a commutative ring $R$, and let $d:M\to M$ satisfy $d^2=0$.  If $H=\ker d/\im d$ is free, then $\rk H\leq\rk M$.
\end{lemma}

\begin{proof}
Since $H$ is free, it is projective, so the surjection $\ker d\twoheadrightarrow H$ has a section $H\hookrightarrow\ker d\subseteq M$.  Hence there is an injection $H\hookrightarrow M$.  Now use Lemma \ref{freesubhassmallerrank}.
\end{proof}

\begin{lemma}\label{freesubhassmallerrank}
Let $\phi:R^{\oplus A}\hookrightarrow R^{\oplus B}$ be an inclusion of free modules over a commutative ring $R$.  Then $A\leq B$.
\end{lemma}

\begin{proof}
This is a standard yet tricky exercise.  We recall one of the (many) standard proofs.

The map $\phi$ is described by some matrix of $A\times B$ elements of $R$.  Certainly the kernel of this matrix remains zero over the subring of $R$ generated by its entries.  Thus we may assume without loss of generality that $R$ is a finitely generated $\ZZ$-algebra.  Now localize at a prime ideal $\mathfrak p\subset R$ of height zero.  Localization is exact, so again $\phi$ remains injective.  Thus we may assume without loss of generality that $R$ is a local Noetherian ring of dimension zero, and thus $R$ is an Artin local ring \cite[p90, Theorem 8.5]{atiyahmacdonald}.  Since $R$ is Artinian, all finitely generated modules have finite length \cite[pp76--77, Propositions 6.5 and 6.8]{atiyahmacdonald}, and hence there is a length homomorphism $K_0(R)\to\ZZ$ \cite[p77, Proposition 6.9]{atiyahmacdonald} (which is clearly an isomorphism since $R$ is local).  It thus follows from the short exact sequence $0\to R^{\oplus A}\xrightarrow\phi R^{\oplus B}\to\coker\phi\to 0$ that $A\leq B$.
\end{proof}

\appendix

\section{Homological algebra}\label{homologicalalgebrasection}

In this appendix, we collect some useful facts concerning sheaves, homotopy sheaves, and their \v Cech cohomology.  We assume the reader is familiar with most elementary aspects of sheaves.  Many of the results in this appendix are also elementary, though for completeness we give most proofs as we do not know of a good reference.

In \S\ref{sheafintro}, we recall presheaves and sheaves.  In \S\ref{homotopysheafintro}, we introduce homotopy sheaves.  In \S\ref{pullbackintro}, we list standard pushforward and pullback operations on (homotopy) sheaves.  In \S\ref{cechsection}, we introduce and prove basic properties of \v Cech cohomology.  In \S\ref{cechsimplesection}, we introduce the central notion of a pure homotopy sheaf.  In \S\ref{poincarelefschetzsection}, we prove a version of Poincar\'e--Lefschetz duality using pure homotopy sheaves.  In \S\ref{hocolimintro}, we introduce a certain relevant type of homotopy colimit.  In \S\ref{cechsimplegluingsection}, we prove an easy lemma about homotopy colimits of pure homotopy sheaves.  In \S\ref{steenrodhomologysec}, we review the definition of Steenrod homology.

\begin{convention}
In this appendix, by \emph{space} we mean locally compact Hausdorff space.
\end{convention}

\begin{convention}\label{complexconventions}
By a \emph{complex} $C^\bullet$ we mean a $\ZZ$-graded object $\bigoplus_{i\in\ZZ}C^i$ (in some abelian category) along with a degree $1$ endomorphism $d$ with $d^2=0$.  The \emph{homology} of a complex $C^\bullet$ is denoted $H^\bullet C^\bullet$ (defined by $H^iC^\bullet:=\ker(C^i\xrightarrow dC^{i+1})/\im(C^{i-1}\xrightarrow dC^i)$).  A complex is called \emph{acyclic} iff its homology vanishes.  A map of complexes $f:A^\bullet\to B^\bullet$ is a called a \emph{quasi-isomorphism} iff it induces an isomorphism on homology.  We will often use the fact that a map of complexes is a quasi-isomorphism iff its mapping cone is acyclic.

The \emph{shift} of a complex $C^\bullet[n]$ is defined by $(C^\bullet[n])^i:=C^{i+n}$.  We use the \emph{truncation functors} defined by:
\begin{equation*}
(\tau_{\geq i}C^\bullet)^j:=\begin{cases}C^j&j>i\cr\coker(C^{i-1}\xrightarrow dC^i)&j=i\cr 0&j<i\end{cases}\qquad
(\tau_{\leq i}C^\bullet)^j:=\begin{cases}0&j>i\cr\ker(C^i\xrightarrow dC^{i+1})&j=i\cr C^j&j<i\end{cases}
\end{equation*}
Given a sequence of maps of complexes $A_0^\bullet\to\cdots\to A_n^\bullet$ such that adjacent maps compose to zero, we denote by $[A_0^\bullet\to\cdots\to A_n^{\bullet-n}]$ the associated total complex of this double complex.  For example, $f:A^\bullet\to B^\bullet$ denotes a map, and $[A^\bullet\xrightarrow fB^{\bullet-1}]$ denotes its mapping cone.
\end{convention}

\begin{convention}\label{signconventions}
We fix sign conventions by making everything $\ZZ/2$-graded and always using the super tensor product $\otimes$ (namely, where the isomorphism $A\otimes B\xrightarrow\sim B\otimes A$ is given by $a\otimes b\mapsto(-1)^{\left|a\right|\left|b\right|}b\otimes a$ and where $(f\otimes g)(a\otimes b):=(-1)^{\left|g\right|\left|a\right|}f(a)\otimes g(b)$).  We fix $\Hom(A,B)\otimes A\to B$ as given by $f\otimes a\mapsto f(a)$.  Complexes are $(\ZZ,\ZZ/2)$-bigraded; differentials are always odd and chain maps are always even.  Note that the $\ZZ/2$-grading of a complex is often, but not always, the reduction modulo $2$ of the $\ZZ$-grading.
\end{convention}

\begin{convention}
Direct limits and inverse limits always take place over directed sets.
\end{convention}

\subsection{Presheaves and sheaves}\label{sheafintro}

\begin{definition}[Presheaf and $\K$-presheaf\footnote{Terminology ``$\K$-'' following Lurie \cite[Definition 7.3.4.1]{luriehtt}.}]
Let $X$ be a space.  A \emph{presheaf} (resp.\ \emph{$\K$-presheaf}) on $X$ is a contravariant functor from the category of open (resp.\ compact) sets of $X$ to the category of abelian groups.  A morphism of presheaves is simply a natural transformation of functors.  The category of presheaves (resp.\ $\K$-presheaves) is denoted $\Prshv X$ (resp.\ $\Prshv_\K X$).
\end{definition}

\begin{definition}[Stalk]
For a presheaf $\F$, let $\F_p:=\varinjlim_{p\in U}\F(U)$, and for a $\K$-presheaf $\F$ let $\F_p:=\F(\{p\})$.  In both cases we say $\F_p$ is the \emph{stalk} of $\F$ at $p$.
\end{definition}

\begin{definition}[Sheaf]
A \emph{sheaf} is a presheaf $\F$ satisfying the following condition:
\begin{equation*}
\tag{Sh}\label{opensheafproperty}0\to\F\Bigl(\bigcup_{\alpha\in A}U_\alpha\Bigr)\to\prod_{\alpha\in A}\F(U_\alpha)\to\prod_{\alpha,\beta\in A}\F(U_\alpha\cap U_\beta)\quad\text{ is exact }\forall\text{ }\{U_\alpha\subseteq X\}_{\alpha\in A}
\end{equation*}
The category of sheaves on a space $X$ is denoted $\Shv X$ (a full subcategory of $\Prshv X$, meaning a morphism between sheaves is the same as a morphism of the corresponding presheaves).
\end{definition}

\begin{definition}[$\K$-sheaf]
A \emph{$\K$-sheaf} is a $\K$-presheaf $\F$ satisfying the following three conditions:\footnote{A similar set of axioms appears in Lurie \cite[Definition 7.3.4.1]{luriehtt}.}
\begin{align*}
\tag{Sh$_\K$1}\label{sheafvanishing}\F(\varnothing)&=0\\
\tag{Sh$_\K$2}\label{sheafproperty}0\to\F(K_1\cup K_2)\to\F(K_1)\oplus\F(K_2)\to\F(K_1\cap K_2)&\text{ is exact }\forall\text{ }K_1,K_2\subseteq X\\
\tag{Sh$_\K$3}\label{sheafcontinuity}\varinjlim_{\begin{smallmatrix}K\subseteq U\cr U\textrm{ open}\end{smallmatrix}}\F(\overline U)\to\F(K)&\text{ is an isomorphism }\forall\text{ }K\subseteq X
\end{align*}
The category of $\K$-sheaves on a space $X$ is denoted $\Shv_\K X$ (a full subcategory of $\Prshv_\K X$).
\end{definition}

\begin{remark}\label{nbhdintersectiongood}
It is always the case that $K=\bigcap_{\begin{smallmatrix}K\subseteq U\cr U\textrm{ open}\end{smallmatrix}}\overline U$ for compact $K\subseteq X$.
\end{remark}

\begin{definition}\label{opencptrelationfunctors}
We define functors:
\begin{equation}
\begin{tikzcd}\Prshv_\K X\ar[yshift=0.5ex]{r}{\alpha_\ast}&\ar[yshift=-0.5ex]{l}{\alpha^\ast}\Prshv X\end{tikzcd}
\end{equation}
by the formulas:
\begin{align}
\label{alphaast}(\alpha^\ast\F)(K)&:=\varinjlim_{\begin{smallmatrix}K\subseteq U\cr U\textrm{ open}\end{smallmatrix}}\F(U)\\
\label{alphaupperast}(\alpha_\ast\F)(U)&:=\varprojlim_{\begin{smallmatrix}K\subseteq U\cr K\textrm{ compact}\end{smallmatrix}}\F(K)
\end{align}
It is easy to see that there is an adjunction $\Hom_{\Prshv_\K X}(\alpha^\ast\F,\G)=\Hom_{\Prshv X}(\F,\alpha_\ast\G)$ (giving an element of either $\Hom$-set is the same as giving a compatible system of maps $\F(U)\to\G(K)$ for pairs $K\subseteq U$).  
\end{definition}

\begin{lemma}\label{intersectioncofinal}
Let $K_1,\ldots,K_n\subseteq X$ be compact.  Then $\{U_1\cap\cdots\cap U_n\}_{\text{open }U_i\supseteq K_i}$ forms a cofinal system of neighborhoods of $K_1\cap\cdots\cap K_n$.
\end{lemma}

\begin{proof}
We may assume without loss of generality that $X$ is compact.  By induction, it suffices to treat the case $n=2$.  The rest is an exercise (use the fact that a compact Hausdorff space is normal).
\end{proof}

\begin{lemma}\label{opencptshvequivalence}
We have:
\begin{equation}
\begin{tikzcd}\Shv_\K X\ar[yshift=0.5ex]{r}{\alpha_\ast}&\ar[yshift=-0.5ex]{l}{\alpha^\ast}\Shv X\end{tikzcd}
\end{equation}
and this is an equivalence of categories.\footnote{A similar result appears in Lurie \cite[Corollary 7.3.4.10]{luriehtt}.}
\end{lemma}

\begin{proof}
Suppose $\F$ is a sheaf, and let us verify $\alpha^\ast\F$ is a $\K$-sheaf.  Axiom \eqref{sheafvanishing} is clear (take $A=\varnothing$ in \eqref{opensheafproperty}).  Axiom \eqref{sheafproperty} follows from \eqref{opensheafproperty} and Lemma \ref{intersectioncofinal} since direct limits are exact.  Axiom \eqref{sheafcontinuity} is clear from the definition also.

Suppose $\F$ is a $\K$-sheaf, and let us verify that $\alpha_\ast\F$ is a sheaf.  Let us first observe that (by induction using \eqref{sheafvanishing} and \eqref{sheafproperty}) if $\{K_\alpha\subseteq X\}_{\alpha\in A}$ is any collection of compact sets, all but finitely many of which are empty, then the following is exact:
\begin{equation}\label{compactfiniteexact}
0\to\F\Bigl(\bigcup_{\alpha\in A}K_\alpha\Bigr)\to\prod_{\alpha\in A}\F(K_\alpha)\to\prod_{\alpha,\beta\in A}\F(K_\alpha\cap K_\beta)
\end{equation}
Let us now verify axiom \eqref{opensheafproperty} for $\alpha_\ast\F$ for open sets $\{U_\alpha\subseteq X\}_{\alpha\in A}$.  Certainly \eqref{opensheafproperty} is the inverse limit of \eqref{compactfiniteexact} over all collections of compact subsets $\{K_\alpha\subseteq U_\alpha\}_{\alpha\in A}$ for which $K_\alpha=\varnothing$ except for finitely many $\alpha\in A$.  This is sufficient since inverse limit is left exact.

Now to see that the adjoint pair $\alpha^\ast\dashv\alpha_\ast$ is actually an equivalence of categories, it suffices to show that the natural morphisms $\F\to\alpha_\ast\alpha^\ast\F$ and $\alpha^\ast\alpha_\ast\G\to\G$ are isomorphisms.  In other words, we must show that following natural maps are isomorphisms:
\begin{align}
\label{Funit}\F(U)&\to\varprojlim_{\begin{smallmatrix}K\subseteq U\cr K\textrm{ compact}\end{smallmatrix}}\varinjlim_{\begin{smallmatrix}K\subseteq U'\cr U'\textrm{ open}\end{smallmatrix}}\F(U')\\
\label{Gunit}\varinjlim_{\begin{smallmatrix}K\subseteq U\cr U\textrm{ open}\end{smallmatrix}}\varprojlim_{\begin{smallmatrix}K'\subseteq U\cr K'\textrm{ compact}\end{smallmatrix}}\G(K')&\to\G(K)
\end{align}
The right hand side of \eqref{Funit} can be thought of as the inverse limit of $\F(U')$ over all open $U'\subseteq U$ with $\overline{U'}\subseteq U$ and $\overline{U'}$ compact.  It is easy to see that the map from $\F(U)$ to this inverse limit is an isomorphism if $\F$ is a sheaf.  The left hand side of \eqref{Gunit} can be identified with the left hand side of \eqref{sheafcontinuity}, and thus the map to $\G(K)$ is an isomorphism if $\F$ is a $\K$-sheaf.
\end{proof}

\begin{convention}
In view of the canonical equivalence of categories $\Shv X=\Shv_\K X$ from Lemma \ref{opencptshvequivalence}, we use the single word ``sheaf'' for an object of either category.
\end{convention}

\begin{lemma}\label{anycompactintersectionworksSH}
Let $\{K_\alpha\}_{\alpha\in A}$ be a filtered directed system of compact subsets of $X$ (i.e.\ for all $\alpha,\beta\in A$ there exists $\gamma\in A$ with $K_\gamma\subseteq K_\alpha\cap K_\beta$).  Then for any $\F$ satisfying \eqref{sheafcontinuity}, the following is an isomorphism:
\begin{equation}
\varinjlim_{\alpha\in A}\F(K_\alpha)\to\F\Bigl(\bigcap_{\alpha\in A}K_\alpha\Bigr)
\end{equation}
\end{lemma}

\begin{proof}
Write $K:=\bigcap_{\alpha\in A}K_\alpha$.  Now consider the diagram:
\begin{equation}
\begin{tikzcd}
\varinjlim\limits_{\alpha\in A}\varinjlim\limits_{\begin{smallmatrix}K_\alpha\subseteq U\cr U\textrm{ open}\end{smallmatrix}}\F(\overline U)\ar{r}\ar{d}&
\varinjlim\limits_{\begin{smallmatrix}\exists\alpha\in A:K_\alpha\subseteq U\cr U\textrm{ open}\end{smallmatrix}}\F(\overline U)\ar{r}&
\varinjlim\limits_{\begin{smallmatrix}K\subseteq U\cr U\textrm{ open}\end{smallmatrix}}\F(\overline U)\ar{d}\\
\varinjlim\limits_{\alpha\in A}\F(K_\alpha)\ar{rr}&&\F(K)
\end{tikzcd}
\end{equation}
The vertical maps are both isomorphisms by \eqref{sheafcontinuity}.  The first horizontal map is an isomorphism since $\{K_\alpha\}_{\alpha\in A}$ is filtered.  It thus remains to show that if $K\subseteq U$ and $U$ is open, then there exists $\alpha\in A$ such that $K_\alpha\subseteq U$.  Since $K\subseteq U$, we have $(X\setminus U)\cap K=\varnothing$, so $\bigcap_{\alpha\in A}(X\setminus U)\cap K_\alpha=\varnothing$.  This is a filtered directed system of compact sets, and so the intersection being empty implies that one of the terms $(X\setminus U)\cap K_\alpha$ is empty, so $K_\alpha\subseteq U$ as desired.
\end{proof}

\subsection{Homotopy sheaves}\label{homotopysheafintro}

\begin{convention}
Let $\F^\bullet$ be a complex of ($\K$-)presheaves.  The notions from Convention \ref{complexconventions} are applied ``objectwise'', that is, to each complex $\F^\bullet(U)$ (resp.\ $F^\bullet(K)$) individually.  For example, the homology $H^i\F^\bullet$ is again a ($\K$-)presheaf, and a map of ($\K$-)presheaves $f:\F^\bullet\to\G^\bullet$ is a quasi-isomorphism iff it induces an isomorphism of ($\K$-)presheaves $H^i\F^\bullet\to H^i\G^\bullet$ for all $i$.  The homology ($\K$-)presheaves $H^i\F^\bullet$ should not be confused with the various flavors of \v Cech (hyper)cohomology $\cH^i(X;\F^\bullet)$ we introduce in \S\ref{cechsection}.
\end{convention}

\begin{definition}[Homotopy sheaf]
A \emph{homotopy sheaf} is a complex of presheaves $\F^\bullet$ satisfying the following condition:
\begin{equation*}
\tag{hSh}\label{openhomotopysheafproperty}\hspace{-20pt}\biggl[\F^\bullet\Bigl(\bigcup_{\alpha\in A}U_\alpha\Bigr)\to\prod_{\alpha\in A}\F^{\bullet-1}(U_\alpha)\to\prod_{\alpha,\beta\in A}\F^{\bullet-2}(U_\alpha\cap U_\beta)\to\cdots\biggr]\text{ is acyclic }\forall\text{ }\{U_\alpha\subseteq X\}_{\alpha\in A}
\end{equation*}
The category of homotopy sheaves on a space $X$ is denoted $\hShv X$ (morphisms are morphisms of complexes of presheaves).
\end{definition}

\begin{remark}
The definition above is given for illustrative purposes only.  It is probably only a good definition for $\F^\bullet$ which is bounded below, and it would perhaps be better to impose \eqref{openhomotopysheafproperty} for all \emph{hypercovers} of open subsets $U\subseteq X$.
\end{remark}

\begin{example}\label{flasquehomotopysheafresolution}
A sheaf $\F$ is \emph{flasque} iff all restriction maps $\F(U)\to\F(U')$ are surjective.  It is easy to see that any bounded below complex of flasque sheaves $\F^\bullet$ is a homotopy sheaf.  For example, let $\F^\bullet(U):=C^\bullet(U)$ be the presheaf of singular cochain complexes.  Then the sheafification of $\F^\bullet$ is a complex of flasque sheaves, and thus is a homotopy sheaf.  One can show (using barycentric subdivision) that the map from $\F^\bullet$ to its sheafification is a quasi-isomorphism, and hence $\F^\bullet$ is a homotopy sheaf as well.
\end{example}

\begin{definition}[Homotopy $\K$-sheaf]\label{homotopyKsheaf}
A \emph{homotopy $\K$-sheaf} is a complex of $\K$-presheaves $\F^\bullet$ satisfying the following three conditions:
\begin{align*}
\tag{hSh$_\K$1}\label{homotopysheafacyclic}\F^\bullet(\varnothing)&\text{ is acyclic}\\
\tag{hSh$_\K$2}\label{homotopysheafproperty}\hspace{-10pt}\bigl[\F^\bullet(K_1\cup K_2)\to\F^{\bullet-1}(K_1)\oplus\F^{\bullet-1}(K_2)\to\F^{\bullet-2}(K_1\cap K_2)\bigr]&\text{ is acyclic }\forall\text{ }K_1,K_2\subseteq X\\
\tag{hSh$_\K$3}\label{homotopysheafcontinuity}\smash{\varinjlim_{\begin{smallmatrix}K\subseteq U\cr U\textrm{ open}\end{smallmatrix}}}\F^\bullet(\overline U)\to\F^\bullet(K)&\text{ is a quasi-isomorphism }\forall\text{ }K\subseteq X
\end{align*}
The category of homotopy $\K$-sheaves on a space $X$ is denoted $\hShv_\K X$ (morphisms are morphisms of complexes of $\K$-presheaves).
\end{definition}

\begin{remark}
Note that \eqref{homotopysheafproperty} gives rise to a ``Mayer--Vietoris'' long exact sequence in cohomology.
\end{remark}

\begin{example}\label{softhomotopysheafresolution}
A $\K$-sheaf $\F$ is \emph{soft} iff all restriction maps $\F(K)\to\F(K')$ are surjective.  It is easy to see that any complex of soft $\K$-sheaves $\F^\bullet$ is a homotopy $\K$-sheaf.
\end{example}

\begin{remark}\label{homotopyKnotsheavesequivalencermk}
In analogy with Definition \ref{opencptrelationfunctors}, we expect there are functors:
\begin{equation}
\begin{tikzcd}\hShv_\K X\ar[yshift=0.5ex]{r}{R\alpha_\ast}&\ar[yshift=-0.5ex]{l}{\alpha^\ast}\hShv X\end{tikzcd}
\end{equation}
where $\alpha^\ast$ is the direct limit \eqref{alphaast} and $R\alpha_\ast$ is the \emph{homotopy} (or \emph{derived}) version $R\varprojlim$ of the inverse limit \eqref{alphaupperast} (the naive inverse limit functor $\alpha_\ast$ is the ``wrong'' functor since inverse limit is not exact).

In analogy with Lemma \ref{opencptshvequivalence}, it seems likely\footnote{A similar result appears in Lurie \cite[Corollary 7.3.4.10]{luriehtt}.} that there is an adjunction $\alpha^\ast\dashv R\alpha_\ast$ which is an equivalence (in the sense of model categories or $\infty$-categories), though perhaps only after restricting to the (possibly better-behaved) subcategory of homotopy ($\K$-)sheaves which are bounded below.
\end{remark}

\begin{remark}
Guided by the needs of the rest of the paper, we proceed to focus on homotopy $\K$-sheaves rather than on homotopy sheaves.
\end{remark}

\begin{lemma}\label{qihomotopysheaf}
Properties \eqref{homotopysheafacyclic}--\eqref{homotopysheafcontinuity} are preserved by quasi-isomorphisms.
\end{lemma}

\begin{proof}
For \eqref{homotopysheafacyclic} and \eqref{homotopysheafcontinuity} this is trivial.  For \eqref{homotopysheafproperty}, suppose $\F^\bullet\xrightarrow\sim\G^\bullet$ is a quasi-isomorphism.  Now the sequence \eqref{homotopysheafproperty} applied to the mapping cone $[\F^\bullet\to\G^{\bullet-1}]$ is certainly acyclic (since it has a finite filtration whose associated graded is acyclic).  Thus \eqref{homotopysheafproperty} holds for $\F^\bullet$ iff it holds for $\G^\bullet$.
\end{proof}

\begin{lemma}[Extensions of homotopy $\K$-sheaves are homotopy $\K$-sheaves]\label{associatedgradedhomotopysheaf}
Let $\F^\bullet$ be a complex of $\K$-presheaves.  If $\F^\bullet$ has a finite filtration whose associated graded is a homotopy $\K$-sheaf, then $\F^\bullet$ is a homotopy $\K$-sheaf.
\end{lemma}

\begin{proof}
This follows from the fact that if a complex $C^\bullet$ has a finite filtration whose associated graded is acyclic then $C^\bullet$ is itself acyclic.
\end{proof}

\begin{lemma}[Lowest nonzero homology $\K$-presheaf of a homotopy $\K$-sheaf is a $\K$-sheaf]\label{lowervanishingimpliessheaf}
If $\F^\bullet$ is a homotopy $\K$-sheaf and $H^{-1}\F^\bullet=0$, then $H^0\F^\bullet$ is a $\K$-sheaf.
\end{lemma}

\begin{proof}
Properties \eqref{sheafvanishing} and \eqref{sheafcontinuity} follow directly from \eqref{homotopysheafacyclic} and \eqref{homotopysheafcontinuity}.  To show \eqref{sheafproperty} for $H^0\F^\bullet$, use the long exact sequence induced by \eqref{homotopysheafproperty} and the vanishing of $H^{-1}\F^\bullet(K_1\cap K_2)$.
\end{proof}

\begin{lemma}\label{anycompactintersectionworks}
Let $\{K_\alpha\}_{\alpha\in A}$ be a filtered directed system of compact subsets of $X$ (i.e.\ for all $\alpha,\beta\in A$ there exists $\gamma\in A$ with $K_\gamma\subseteq K_\alpha\cap K_\beta$).  Then for any $\F^\bullet$ satisfying \eqref{homotopysheafcontinuity}, the following is a quasi-isomorphism:
\begin{equation}
\varinjlim_{\alpha\in A}\F^\bullet(K_\alpha)\to\F^\bullet\Bigl(\bigcap_{\alpha\in A}K_\alpha\Bigr)
\end{equation}
\end{lemma}

\begin{proof}
Same as for Lemma \ref{anycompactintersectionworksSH}.
\end{proof}

\subsection{Pushforward, exceptional pushforward, and pullback}\label{pullbackintro}

\begin{definition}[Pushforward of ($\K$-)presheaves]\label{presheafpushforward}
Let $f:X\to Y$ be a map of spaces.  We define functors:
\begin{rlist}
\item$f_\ast:\Prshv X\to\Prshv Y$ by $(f_\ast\F)(U):=\F(f^{-1}(U))$.
\item$f_\ast:\Prshv_\K X\to\Prshv_\K Y$ by $(f_\ast\F)(K):=\F(f^{-1}(K))$ (if $f$ is proper).
\end{rlist}
(the action of $f_\ast$ on morphism spaces is obvious).
\end{definition}

\begin{lemma}\label{properpullbackfinalneighborhood}
Let $f:X\to Y$ be proper.  Then $\{f^{-1}(U)\}_{K\subseteq U}$ is a cofinal system of neighborhoods of $f^{-1}(K)$ for any compact $K\subseteq X$.
\end{lemma}

\begin{proof}
Exercise (use the fact that a compact Hausdorff space is normal).
\end{proof}

\begin{lemma}\label{pushforwardandsheafproperties}
$f_\ast$ preserves \eqref{opensheafproperty}, \eqref{openhomotopysheafproperty}, \eqref{sheafvanishing}--\eqref{sheafcontinuity}, and \eqref{homotopysheafacyclic}--\eqref{homotopysheafcontinuity}.
\end{lemma}

\begin{proof}
These are trivial except for \eqref{sheafcontinuity} and \eqref{homotopysheafcontinuity}, which use Lemma \ref{properpullbackfinalneighborhood}.
\end{proof}

\begin{definition}[Pushforward of ($\K$-)sheaves and homotopy $\K$-sheaves]\label{pushforwardofHsheavesworks}
Let $f:X\to Y$ be a map of spaces.  By Lemma \ref{pushforwardandsheafproperties}, Definition \ref{presheafpushforward} gives rise to functors:
\begin{rlist}
\item$f_\ast:\Shv X\to\Shv Y$.
\item$f_\ast:\hShv X\to\hShv Y$.
\item$f_\ast:\Shv_\K X\to\Shv_\K Y$ (if $f$ is proper).
\item$f_\ast:\hShv_\K X\to\hShv_\K Y$ (if $f$ is proper).
\end{rlist}
(the action on morphism spaces is induced from the $f_\ast$ at the level of (complexes of) \mbox{($\K$-)presheaves}).

It is easy to see that $f_\ast$ commutes with the equivalence $\Shv X=\Shv_\K X$ (if $f$ is proper).
\end{definition}

\begin{definition}[Pullback and exceptional pushforward of sheaves]
Let $f:X\to Y$ be a map of spaces.  We define:
\begin{rlist}
\item$f^\ast:\Shv Y\to\Shv X$ the standard sheaf pullback (namely $f^\ast\F$ is the sheafification of the presheaf $U\mapsto\varinjlim_{f(U)\subseteq V}\F(V)$).
\item$f_!:\Shv X\to\Shv Y$ by $(f_!\F)(U)\subseteq(f_\ast\F)(U)$ being the subspace of sections which vanish in a neighborhood of $Y\setminus X$ (if $f$ is the inclusion of an open set).
\end{rlist}
\end{definition}

\begin{definition}[Pullback of homotopy $\K$-sheaves]
Let $f:X\to Y$ be an injective map of spaces.  We define:
\begin{rlist}
\item$f^\ast:\hShv_\K Y\to\hShv_\K X$ by $(f^\ast\F^\bullet)(K):=\F^\bullet(f(K))$.
\end{rlist}
We check the properties: \eqref{homotopysheafacyclic} is clear, and \eqref{homotopysheafproperty} follows since $f$ is injective.  For \eqref{homotopysheafcontinuity}, use Remark \ref{nbhdintersectiongood}, the injectivity of $f$, and Lemma \ref{anycompactintersectionworks}.
\end{definition}

\subsection{\texorpdfstring{\v C}{C}ech cohomology}\label{cechsection}

We introduce various flavors of \v Cech (hyper)cohomology relevant to our situation.

\begin{remark}\label{coveringscategoryrmk}
A \emph{refinement} of a cover $\{U_\alpha\}_{\alpha\in A}$ is a cover $\{U_\beta\}_{\beta\in B}$ along with a map $f:B\to A$ such that $U_\beta\subseteq U_{f(\alpha)}$.  Refinements are the morphisms used in the directed systems used to define (all flavors of) \v Cech cohomology.  Different refinements $\{U_\alpha\}_{\alpha\in A}\to\{U_\beta\}_{\beta\in B}$ induce different maps on \v Cech complexes, but they all agree after passing to cohomology (more precisely, the ``space'' of such refinements is contractible or empty).  In particular, it follows that the directed systems used in defining \v Cech cohomology are \emph{filtered}.
\end{remark}

\begin{remark}
The empty covering is a final object in the category of coverings of $\varnothing$, so we always have $\cH^\bullet(\varnothing;-)=0$.
\end{remark}

\subsubsection{\ldots of sheaves}

\begin{definition}[$\cH^\bullet$ and $\cH^\bullet_c$ of sheaves]\label{cechdefnI}
Let $\F$ be a sheaf on a space $X$.  We define the \v Cech cohomology:
\begin{equation}\label{cechcomplexI}
\cH^\bullet(X;\F):=\varinjlim_{\begin{smallmatrix}X=\bigcup_{\alpha\in A}U_\alpha\cr\text{open cover}\end{smallmatrix}}H^\bullet\biggl[\bigoplus_{p\geq 0}\prod_{\begin{smallmatrix}S\subseteq A\cr\left|S\right|=p+1\end{smallmatrix}}\F\Bigl(\bigcap_{\alpha\in S}U_\alpha\Bigr)[-p]\biggr]
\end{equation}
with the standard \v Cech differential.\footnote{Technically speaking, so that the signs in the differential can be defined canonically, we should really tensor each term of the direct product with $(\mathbb Zo_1\oplus\mathbb Zo_2)/(o_1+o_2)\cong\mathbb Z$ where $o_1,o_2$ denote the two orientations of the $p$-simplex on vertex set $S$.}  For any compact $K\subseteq X$, define $\cH^\bullet_K(X;\F)$ (\v Cech cohomology with supports in $K$) via \eqref{cechcomplexI} except replacing every instance of $\F(U)$ with $\ker[\F(U)\to\F(U\setminus K)]$.  We let $\cH^\bullet_c(X;\F):=\varinjlim_{K\subseteq X}\cH^\bullet_K(X;\F)$ (\v Cech cohomology with compact supports).
\end{definition}

\begin{lemma}[$\cH^\bullet$ on a compact space needs only finite open covers]
If $X$ is compact, then the natural map $\cH^\bullet_\mathrm{fin}(X;\F)\to\cH^\bullet(X;\F)$ is an isomorphism, where the left hand side is defined as in \eqref{cechcomplexI} except using only finite open covers.
\end{lemma}

\begin{proof}
Since $X$ is compact, it follows that finite open covers are cofinal (every open cover has a finite refinement).
\end{proof}

\begin{lemma}[$\cH_c^\bullet=\cH^\bullet$ on compact space]
If $X$ is compact, then the natural map $\cH_c^\bullet(X;\F)\to\cH^\bullet(X;\F)$ is an isomorphism.
\end{lemma}

\begin{proof}
Trivial.
\end{proof}

\begin{definition}[Pullback on $\cH^\bullet$ and $\cH^\bullet_c$]\label{pullbackdefn}
Let $f:X\to Y$ be a map of spaces.  An open cover of $Y$ pulls back to give an open cover of $X$, and this gives an identification of the \v Cech complex for the cover of $Y$ with coefficients in $f_\ast\F$ with the \v Cech complex for the cover of $X$ with coefficients in $\F$.  Hence we get natural maps:
\begin{rlist}
\item$f^\ast:\cH^\bullet(Y;f_\ast\F)\to\cH^\bullet(X;\F)$ for $\F\in\Shv X$.
\item$f^\ast:\cH^\bullet_c(Y;f_\ast\F)\to\cH^\bullet_c(X;\F)$ for $\F\in\Shv X$ (if $f$ is proper).
\end{rlist}
\end{definition}

\begin{lemma}[$\cH_c^\bullet$ commutes with $f_!$]\label{cptcechembeddingproperty}
Let $f:X\hookrightarrow Y$ be the inclusion of an open subset.  Then there is a natural isomorphism $f_!:\cH^\bullet_c(X;\F)\to\cH^\bullet_c(Y;f_!\F)$.
\end{lemma}

\begin{proof}
For $K\subseteq X$, there are natural maps:
\begin{equation}\label{trivialKequivalence}
\begin{tikzcd}
\cH^\bullet_K(Y;f_!\F)\ar[yshift=0.5ex]{r}{f^\ast}&\ar[yshift=-0.5ex]{l}{f_!}\cH^\bullet_K(X;\F)
\end{tikzcd}
\end{equation}
(for $f^\ast$: pull back the open cover) (for $f_!$: add $Y\setminus K$ to the open cover and extend by zero).  It is easy to see that $f_!$ and $f^\ast$ are inverses.  The desired map $f_!:\cH^\bullet_c(X;\F)\to\cH^\bullet_c(Y;f_!\F)$ is defined as the composition:
\begin{equation}
\varinjlim_{K\subseteq X}\cH^\bullet_K(X;\F)\overset{\eqref{trivialKequivalence}}=\varinjlim_{K\subseteq X}\cH^\bullet_K(Y;f_!\F)\to\varinjlim_{K\subseteq Y}\cH^\bullet_K(Y;f_!\F)
\end{equation}
We must show that the second map is an isomorphism; to see this, it suffices to show that the following is an isomorphism for all $K\subseteq Y$:
\begin{equation}\label{supportcompactinXorY}
\varinjlim_{K'\subseteq X\cap K}\cH^\bullet_{K'}(Y;f_!\F)\to\cH^\bullet_K(Y;f_!\F)
\end{equation}
We claim that for any \v Cech cochain $\beta$ for $f_!\F$ with supports in $K$ subordinate to an open cover of $Y$, there is a refinement on which the restriction of $\beta$ has supports in some $K'\subseteq X\cap K$.  It follows from the claim (using a cofinality argument) that \eqref{supportcompactinXorY} is an isomorphism.  To prove the claim, argue as follows.

First, choose a refinement for which only finitely many open sets $U_1,\ldots,U_n$ intersect $K$ and for which the remaining open sets cover $Y\setminus K$.  Pick open sets $V_i\subseteq U_i$ which cover $K$ and for which $\overline{V_i}$ is compact and $\overline{V_i}\subseteq U_i$ (this is always possible).

Now $\beta$ is a finite collection $\{\beta_I\in(f_!\F)(\bigcap_{i\in I}U_i)\}_{\varnothing\ne I\subseteq\{1,\ldots,n\}}$.  We have:
\begin{equation}\label{restrictionsupport}
\supp\Bigl(\beta_I|\bigcap_{i\in I}V_i\Bigr)\subseteq\supp\beta_I\cap\bigcap_{i\in I}\overline{V_i}
\end{equation}
Since $\supp\beta_I\subseteq\bigcap_{i\in I}U_i$ is (relatively) closed and $\bigcap_{i\in I}\overline{V_i}\subseteq\bigcap_{i\in I}U_i$ is compact, it follows that the right hand side is compact.  Also, we have $\supp\beta_I\subseteq X$ (by definition of $f_!\F$).  Hence the right hand side of \eqref{restrictionsupport} is a compact subset of $X$.  It follows that the restriction of $\beta$ to the refinement obtained by replacing $U_i$ with $V_i$ for $i=1,\ldots,n$ has support in a compact subset of $X$.
\end{proof}

\begin{lemma}[$f^\ast$ is an isomorphism if $f$ has finite fibers]\label{finitemap}
Let $f:X\to Y$ be proper with finite fibers.  Then $f^\ast:\cH^\bullet(Y;f_\ast\F)\to\cH^\bullet(X;\F)$ is an isomorphism, as is $f^\ast:\cH^\bullet_c(Y;f_\ast\F)\to\cH^\bullet_c(X;\F)$.
\end{lemma}

\begin{proof}
Given an open cover containing $U=U_1\sqcup\cdots\sqcup U_n$ (finite disjoint union), we get a refinement by replacing $U$ with $\{U_1,\ldots,U_n\}$.  A \emph{partition} of an open cover is a refinement obtained by doing such a replacement on some (possibly infinitely many) open sets of the cover.  Note that a partition induces an isomorphism on \v Cech cochains since $\F$ is a sheaf.

We claim that partitions of pullbacks of open covers of $Y$ form a cofinal system of open covers of $X$.  This is clear using Lemma \ref{properpullbackfinalneighborhood} and the fact that $f$ has finite fibers.  It follows from this cofinality that $f^\ast$ is an isomorphism.
\end{proof}

\begin{lemma}[$\cH^\bullet$ and $\cH^\bullet_c$ commute with finite quotients]\label{cechcohomologyfinitequotient}
Let $X$ be a space and let $\pi:X\to X/\Gamma$ be the quotient map under a finite group action.  Let $\F$ be any sheaf of $\ZZ[\frac 1{\#\Gamma}]$-modules on $X/\Gamma$.  Then the following maps are all isomorphisms:
\begin{align*}
\cH^\bullet(X/\Gamma;\F)\to\cH^\bullet(X/\Gamma;(\pi_\ast\pi^\ast\F)^\Gamma)\to\cH^\bullet(X/\Gamma;\pi_\ast\pi^\ast\F)^\Gamma\to\cH^\bullet(X;\pi^\ast\F)^\Gamma\\
\cH^\bullet_c(X/\Gamma;\F)\to\cH^\bullet_c(X/\Gamma;(\pi_\ast\pi^\ast\F)^\Gamma)\to\cH^\bullet_c(X/\Gamma;\pi_\ast\pi^\ast\F)^\Gamma\to\cH^\bullet_c(X;\pi^\ast\F)^\Gamma
\end{align*}
(here we note that by functoriality, $\pi^\ast\F$ is $\Gamma$-equivariant, and thus $\Gamma$ acts on $\pi_\ast\pi^\ast\F$).
\end{lemma}

\begin{proof}
Isomorphism one: the natural map $\F\to(\pi_\ast\pi^\ast)^\Gamma$ is in fact an isomorphism of sheaves (check on stalks).  Isomorphism two: obvious since taking $\Gamma$-invariants is exact on $\ZZ[\frac 1{\#\Gamma}]$-modules.  Isomorphism three: use Lemma \ref{finitemap}, which applies since $\pi$ is automatically proper.
\end{proof}

\subsubsection{\ldots of complexes of \texorpdfstring{$\K$}{K}-presheaves}

\begin{definition}[$\cH^\bullet$ of complexes of $\K$-presheaves]\label{cechdefnII}
Let $\F^\bullet$ be a complex of $\K$-presheaves on a compact space $X$.  We define:
\begin{equation}\label{cechcomplexII}
\cH^\bullet(X;\F^\bullet):=\varinjlim_{\begin{smallmatrix}X=\bigcup_{i=1}^nK_i\cr\text{finite compact cover}\end{smallmatrix}}H^\bullet\biggl[\bigoplus_{p\geq 0}\bigoplus_{1\leq i_0<\cdots<i_p\leq n}\F^{\bullet-p}\Bigl(\bigcap_{j=0}^pK_{i_j}\Bigr)\biggr]
\end{equation}
with the standard \v Cech differential (plus the internal differential of $\F^\bullet$).  We also define $\cH^\bullet(X;\F)$ for any $\K$-presheaf $\F$ by viewing it as a complex concentrated in degree zero.
\end{definition}

\begin{lemma}[Two definitions of $\cH^\bullet$ on $\Shv X=\Shv_\K X$ agree]\label{cHsameforopencompactcovers}
Let $\F$ be a sheaf on a compact space $X$.  Then there is a natural isomorphism $\cH^\bullet(X;\F)\to\cH^\bullet(X;\alpha^\ast\F)$ ($\alpha^\ast\F$ is a $\K$-sheaf; c.f.\ Definition \ref{opencptrelationfunctors}).
\end{lemma}

\begin{proof}
Let us denote by $\cH^\bullet(X;\F;\{U_\alpha\}_{\alpha\in A})$ (resp.\ $\cH^\bullet(X;\alpha^\ast\F;\{K_i\}_{i=1}^n)$) the argument of the direct limit \eqref{cechcomplexI} (resp.\ \eqref{cechcomplexII}).

Since $X$ is compact, every open cover has a finite compact refinement.  This gives a map $\cH^\bullet(X;\F)\to\cH^\bullet(X;\alpha^\ast\F)$.  To show that this map is an isomorphism, it suffices to show that for any fixed finite compact cover $\{K_i\}_{i=1}^n$, the following is an isomorphism:
\begin{equation}
\varinjlim_{\begin{smallmatrix}X=\bigcup_{\alpha\in A}U_\alpha\cr\text{open cover refined by }\{K_i\}_{i=1}^n\end{smallmatrix}}\cH^\bullet(X;\F;\{U_\alpha\}_{\alpha\in A})\to\cH^\bullet(X;\alpha^\ast\F;\{K_i\}_{i=1}^n)
\end{equation}
By a cofinality argument, we can change the directed system on the left side to be open covers $\{U_i\}_{i=1}^n$ with $U_i\supseteq K_i$.  Hence it suffices to show that the following is an isomorphism:
\begin{equation}
\varinjlim_{\{U_i\supseteq K_i\}_{i=1}^n}\cH^\bullet(X;\F;\{U_i\}_{i=1}^n)\to\cH^\bullet(X;\alpha^\ast\F;\{K_i\}_{i=1}^n)
\end{equation}
This is clear from the definition of $\alpha^\ast$ and from Lemma \ref{intersectioncofinal}.
\end{proof}

\begin{lemma}[$\cH^\bullet$ preserves quasi-isomorphisms]\label{qicech}
If $\F^\bullet\to\G^\bullet$ is a quasi-isomorphism of complexes of $\K$-presheaves, then the induced map $\cH^\bullet(X;\F^\bullet)\to\cH^\bullet(X,\G^\bullet)$ is an isomorphism.
\end{lemma}

\begin{proof}
There is clearly a long exact sequence:
\begin{equation}\label{cechmappingconeles}
\cdots\to\cH^{\bullet-1}(X;\G^\bullet)\to\cH^\bullet(X;[\F^\bullet\to\G^{\bullet-1}])\to\cH^\bullet(X;\F^\bullet)\to\cH^\bullet(X;\G^\bullet)\to\cdots
\end{equation}
Hence it suffices to show that if $\F^\bullet$ is acyclic then $\cH^\bullet(X;\F^\bullet)=0$.  This is true because then each \v Cech complex has a finite filtration whose associated graded is acyclic.
\end{proof}

\begin{lemma}[Hypercohomology spectral sequence]\label{hypercohomologyss}
Let $\F^\bullet$ be a bounded below complex of $\K$-presheaves.  Then there is a convergent spectral sequence $E_1^{p,q}=\cH^q(X;\F^p)\Rightarrow\cH^{p+q}(X;\F^\bullet)$.
\end{lemma}

\begin{proof}
This is just the spectral sequence of the \v Cech double complex.
\end{proof}

\begin{proposition}[A homotopy $\K$-sheaf calculates its own $\cH^\bullet$]\label{homotopysheafcechquasiisomorphism}
If $\F^\bullet$ satisfies \eqref{homotopysheafacyclic} and \eqref{homotopysheafproperty}, then the canonical map $H^\bullet\F^\bullet(X)\to\cH^\bullet(X;\F^\bullet)$ is an isomorphism.
\end{proposition}

\begin{proof}
We prove that $\F^\bullet(X)\to\cC^\bullet(X;\F^\bullet;K_1,\ldots,K_n)$ (the right hand side denotes \v Cech complex for the finite compact cover $X=K_1\cup\cdots\cup K_n$) is a quasi-isomorphism by induction on $n$.  The base case $n=1$ is obvious since $\cC^\bullet(X;\F^\bullet;X)=\F^\bullet(X)$ by definition.  For the inductive step, it suffices to show that the natural map $\cC^\bullet(X;\F^\bullet;K_1\cup K_2,K_3,\ldots,K_n)\to\cC^\bullet(X;\F^\bullet;K_1,K_2,K_3,\ldots,K_n)$ is a quasi-isomorphism.  We will show that the mapping cone is acyclic; to see this, let us filter it according to how many of the $K_3,\ldots,K_n$ are chosen among $i_0,\ldots,i_p$.  This is a finite filtration, so it suffices to show that the associated graded is acyclic.  The associated graded is a direct sum of complexes of the form $[\F^\bullet(K)\to\F^{\bullet-1}(K)]$ (which is obviously acyclic) and $[\F^\bullet(K\cap(K_1\cup K_2))\to\F^{\bullet-1}(K\cap K_1)\oplus\F^{\bullet-1}(K\cap K_2)\to\F^{\bullet-2}(K\cap K_1\cap K_2)]$ (which is acyclic by \eqref{homotopysheafproperty}).

The above argument works when $X\ne\varnothing$; if $X=\varnothing$, use \eqref{homotopysheafacyclic}.
\end{proof}

\begin{lemma}[A $\K$-sheaf calculates its own $\cH^0$]\label{sheafcechHzero}
If $\F$ satisfies \eqref{sheafvanishing} and \eqref{sheafproperty}, then the canonical map $\F(X)\to\cH^0(X;\F)$ is an isomorphism.
\end{lemma}

\begin{proof}
Use induction as in the proof of Proposition \ref{homotopysheafcechquasiisomorphism}.
\end{proof}

\begin{lemma}[$\cH^\bullet$ is determined by stalks I]\label{localvanishingcechzero}
Let $\F$ satisfy \eqref{sheafcontinuity}.  If the stalks of $\F$ vanish, then $\cH^\bullet(X;\F)=0$.
\end{lemma}

\begin{proof}
It suffices to show that for all $\alpha\in\F(K)$, there exists a finite compact cover $X=\bigcup_{i=1}^nK_i$ such that $\F(K)\to\bigoplus_{i=1}^n\F(K\cap K_i)$ annihilates $\alpha$.  We consider the commutative diagram:
\begin{equation}
\begin{CD}
\varinjlim_{K\subseteq U}\F(\overline U)@>>>\prod_{p\in K}\varinjlim_{p\in U}\F(\overline U)\cr
@VVV@VVV\cr
\F(K)@>>>\prod_{p\in K}\F_p
\end{CD}
\end{equation}
where the vertical maps are both isomorphisms.  Now the vanishing of $\F_p$ and a compactness argument shows that there are finitely many open sets $U_i\subseteq X$ covering $K$ such that $\alpha$ vanishes in $\F(K\cap\overline{U_i})$ for all $i$.  Thus the compact cover $X=(X\setminus\bigcup_{i=1}^nU_i)\cup\overline{U_1}\cup\cdots\cup\overline{U_n}$ has the desired properties.
\end{proof}

\begin{lemma}[$\cH^\bullet$ is determined by stalks II]\label{cechlocalglobalvanishingcomplex}
Let $\F^\bullet$ satisfy \eqref{homotopysheafcontinuity} and $H^i\F^\bullet=0$ for $i<<0$.  If $\F^\bullet$ has acyclic stalks, then $\cH^\bullet(X;\F^\bullet)=0$.
\end{lemma}

\begin{proof}
We show that for all $i$, the map $\cH^\bullet(X;\F^\bullet)\to\cH^\bullet(X;\tau_{\geq i}\F^\bullet)$ is an isomorphism (this is sufficient since $\cH^j(X;\tau_{\geq i}\F^\bullet)=0$ for $j<i$).  We proceed by induction on $i$.

Since $H^i\F^\bullet=0$ for $i<<0$, we have that $\F^\bullet\to\tau_{\geq i}\F^\bullet$ is a quasi-isomorphism, so $\cH^\bullet(X;\F^\bullet)\to\cH^\bullet(X;\tau_{\geq i}\F^\bullet)$ is an isomorphism for $i<<0$ (Lemma \ref{qicech}).  Thus we have the base case of the induction.  For the inductive step, it suffices to show that $\cH^\bullet(X;H^{i-1}\F^\bullet)=0$.  This follows from Lemma \ref{localvanishingcechzero}.
\end{proof}

\begin{proposition}[$\cH^\bullet$ is determined by stalks III]\label{cechdeterminedbystalks}
Let $\F^\bullet,\G^\bullet$ satisfy \eqref{homotopysheafcontinuity} and $H^i\F^\bullet=H^i\G^\bullet=0$ for $i<<0$.  If $\F^\bullet\to\G^\bullet$ induces a quasi-isomorphism on stalks, then it induces an isomorphism $\cH^\bullet(X;\F^\bullet)\to\cH^\bullet(X;\G^\bullet)$.
\end{proposition}

\begin{proof}
Recall the long exact sequence \eqref{cechmappingconeles} and apply Lemma \ref{cechlocalglobalvanishingcomplex} to the mapping cone $[\F^\bullet\to\G^{\bullet-1}]$.
\end{proof}

\begin{corollary}[A map of homotopy $\K$-sheaves being a quasi-isomorphism can be checked on stalks]\label{homotopysheafcheckqionstalks}
Let $\F^\bullet\to\G^\bullet$ be a map of homotopy $\K$-sheaves which satisfy $H^i\F^\bullet=H^i\G^\bullet=0$ for $i<<0$.  Then $\F^\bullet\to\G^\bullet$ is a quasi-isomorphism iff $\F^\bullet_p\to\G^\bullet_p$ is a quasi-isomorphism for all $p\in X$.
\end{corollary}

\begin{proof}
For any $K\subseteq X$, we have a commutative diagram:
\begin{equation}
\begin{CD}
H^\bullet\F^\bullet(K)@>\sim>>\cH^\bullet(K;\F^\bullet)\cr
@VVV@VVV\cr
H^\bullet\G^\bullet(K)@>\sim>>\cH^\bullet(K;\G^\bullet)
\end{CD}
\end{equation}
The rows are isomorphisms by Proposition \ref{homotopysheafcechquasiisomorphism}.  If $\F^\bullet_p\to\G^\bullet_p$ is a quasi-isomorphism, then the right vertical map is a quasi-isomorphism by Proposition \ref{cechdeterminedbystalks}.
\end{proof}

\begin{lemma}[Long exact sequence for $\cH^\bullet$]\label{cechles}
Let $\F,\G,\HHH$ satisfy \eqref{sheafcontinuity}.  If $0\to\F\to\G\to\HHH\to 0$ is exact on stalks, then it induces a long exact sequence on $\cH^\bullet$.
\end{lemma}

\begin{proof}
By Lemma \ref{cechlocalglobalvanishingcomplex}, we have $\cH^\bullet(X;[\F\to\G[-1]\to\HHH[-2]])=0$.  Now inspection of the hypercohomology spectral sequence (Lemma \ref{hypercohomologyss}) gives the desired long exact sequence.
\end{proof}

\begin{definition}[Pullback on $\cH^\bullet$]\label{pullbackcHKprshv}
Let $f:X\to Y$ be a map of compact spaces.  A finite compact cover of $Y$ pulls back to give a finite compact cover of $X$, and this gives an identification of the \v Cech complex for the cover of $Y$ with coefficients in $f_\ast\F^\bullet$ with the \v Cech complex for the cover of $X$ with coefficients in $\F^\bullet$.  Hence we get a natural map:
\begin{rlist}
\item$f^\ast:\cH^\bullet(Y;f_\ast\F^\bullet)\to\cH^\bullet(X;\F^\bullet)$ for $\F^i\in\Prshv_\K X$.
\end{rlist}
\end{definition}

\subsection{Pure homotopy \texorpdfstring{$\K$}{K}-sheaves}\label{cechsimplesection}

By Proposition \ref{homotopysheafcechquasiisomorphism}, a homotopy $\K$-sheaf can be thought of as a \emph{resolution} (that is, its global sections computes the cohomology of some complex of sheaves).  In this section, we introduce the notion of a \emph{pure} homotopy $\K$-sheaf, which may be thought of as a resolution of a sheaf (as opposed to a complex of sheaves).  More specifically, a pure homotopy $\K$-sheaf $\F^\bullet$ ``is'' a resolution of $H^0\F^\bullet$ (which by Lemma \ref{cechsimplebasic} is always a sheaf).

\begin{definition}\label{purehomotopyKsheaf}
We say that a homotopy $\K$-sheaf $\F^\bullet$ on $X$ is \emph{pure} iff:
\begin{rlist}
\item(Stalk cohomology) $H^i\F^\bullet_p=0$ for $i\ne 0$ and all $p\in X$.
\item(Weak vanishing) $H^i\F^\bullet=0$ for $i<<0$ locally on $X$ (meaning that for all $p\in X$, there exists an open set $U\subseteq X$ containing $p$ and an integer $N>-\infty$ such that $H^i\F^\bullet(K)=0$ for all $K\subseteq U$ and $i\leq N$).
\end{rlist}
\end{definition}

\begin{remark}
It would be nice to know whether the stalk cohomology condition implies the weak vanishing condition in general (it would be much easier to check purity).
\end{remark}

\begin{lemma}\label{cechsimplebasic}
Let $\F^\bullet$ be a pure homotopy $\K$-sheaf.  Then:
\begin{rlist}
\item(Strong vanishing) $H^i\F^\bullet=0$ for $i<0$.
\item$H^0\F^\bullet$ is a $\K$-sheaf.
\end{rlist}
\end{lemma}

\begin{proof}
By Lemma \ref{lowervanishingimpliessheaf}, strong vanishing implies that $H^0\F^\bullet$ is a $\K$-sheaf.  Now let us prove strong vanishing.  By restricting to a compact subset, it suffices to treat the case when the underlying space is compact.  Now from \eqref{homotopysheafproperty}, compactness, and weak vanishing, it follows that $H^i\F^\bullet=0$ for $i<<0$.  Now let us prove strong vanishing by induction on $i<0$ (we have just proven the base case).  For the inductive step, observe that $H^i\F^\bullet$ is a sheaf by the induction hypothesis ($H^{i-1}\F^\bullet=0$) and Lemma \ref{lowervanishingimpliessheaf}, and thus $H^i\F^\bullet_p=0\implies H^i\F^\bullet=0$.
\end{proof}

\begin{proposition}[$\cH^\bullet$ of a pure homotopy $\K$-sheaf]\label{cechcohomologypurehomotopysheaf}
Let $\F^\bullet$ be a pure homotopy $\K$-sheaf.  Then there is a canonical isomorphism:
\begin{equation}
H^\bullet\F^\bullet(X)=\cH^\bullet(X;H^0\F^\bullet)
\end{equation}
More generally, let $[\F^\bullet_0\to\F^{\bullet-1}_1\to\cdots\to\F^{\bullet-n}_n]$ be a complex of $\K$-presheaves where each $\F^\bullet_i$ is a pure homotopy $\K$-sheaf.  Then there is a canonical isomorphism:
\begin{equation}\label{cechsimplecanonicaliso}
H^\bullet[\F^\bullet_0(X)\to\cdots\to\F^{\bullet-n}_n(X)]=\cH^\bullet(X;[H^0\F^\bullet_0\to\cdots\to(H^0\F^\bullet_n)[-n]])
\end{equation}
\end{proposition}

\begin{proof}
The isomorphism \eqref{cechsimplecanonicaliso} is defined as the composition of the following isomorphisms:
\begin{equation}\label{isosforpurehtpysheafcech}
\begin{matrix}
H^\bullet[\F^\bullet_0(X)\to\cdots\to\F^{\bullet-n}_n(X)]\cr
\downarrow\cr
\cH^\bullet(X;[\F^\bullet_0\to\cdots\to\F^{\bullet-n}_n])\cr
\downarrow\cr
\cH^\bullet(X;[\tau_{\geq 0}\F^\bullet_0\to\cdots\to\tau_{\geq n}\F^{\bullet-n}_n])\cr
\uparrow\cr
\cH^\bullet(X;[H^0\F^\bullet_0\to\cdots\to H^n\F^{\bullet-n}_n])
\end{matrix}
\end{equation}
The maps are isomorphisms for the following reasons.  Map one: $[\F^\bullet_0\to\cdots\to\F^{\bullet-n}_n]$ is a homotopy $\K$-sheaf (Lemma \ref{associatedgradedhomotopysheaf}), and a homotopy $\K$-sheaf calculates its own $\cH^\bullet$ (Proposition \ref{homotopysheafcechquasiisomorphism}).  Map two: the map of coefficient $\K$-presheaves is a quasi-isomorphism (Lemma \ref{cechsimplebasic}) and thus induces an isomorphism on $\cH^\bullet$ (Lemma \ref{qicech}).  Map three: the map of coefficient $\K$-presheaves is a quasi-isomorphism on stalks (by purity), and thus induces an isomorphism on $\cH^\bullet$ (Proposition \ref{cechdeterminedbystalks}).
\end{proof}

\begin{lemma}[Checking purity on a cover]\label{openclosedpurity}
Let $\F^\bullet$ be a homotopy $\K$-sheaf.  Write $X=U\cup Z$ with $U$ open and $Z$ closed, and suppose that $i^\ast\F^\bullet$ and $j^\ast\F^\bullet$ are both pure ($i:Z\hookrightarrow X$ and $j:U\hookrightarrow X$).  Then $\F^\bullet$ is pure.
\end{lemma}

\begin{proof}
It suffices to show the weak vanishing property for $\F^\bullet$.  Let $\alpha\in H^i\F^\bullet(K)$ be arbitrary with $i<0$.  By Remark \ref{nbhdintersectiongood} and Lemma \ref{anycompactintersectionworks}, the following is a quasi-isomorphism:
\begin{equation}
\varinjlim_{\begin{smallmatrix}Z\subseteq V\cr V\textrm{ open}\end{smallmatrix}}\F^\bullet(K\cap\overline V)\to\F^\bullet(K\cap Z)
\end{equation}
Since $H^i\F^\bullet(K\cap Z)=0$ by strong vanishing for $i^\ast\F^\bullet$, we see that the image of $\alpha$ in $H^i\F^\bullet(K\cap\overline V)$ vanishes for some open $V\supseteq Z$.  Now applying \eqref{homotopysheafproperty} to $K=(K\cap\overline V)\cup(K\setminus V)$, we see that the vanishing of the image of $\alpha$ in $H^i\F^\bullet(K\cap\overline V)$ implies that $\alpha$ ``comes from'' the cohomology of $[\F^\bullet(K\setminus V)\to\F^{\bullet-1}(K\cap\overline V\cap(K\setminus V))$].  On the other hand, this latter group vanishes in degrees $i<0$ by strong vanishing for $j^\ast\F^\bullet$.  Thus $\alpha=0$, so we have even shown strong vanishing for $\F^\bullet$.
\end{proof}

\subsection{Poincar\'e--Lefschetz duality}\label{poincarelefschetzsection}

We prove a version of Poincar\'e duality for arbitrary closed subsets of a topological manifold.  This proof is a good illustration of the tools we have developed concerning pure homotopy $\K$-sheaves (which arise naturally in the proof).

We observed in Example \ref{flasquehomotopysheafresolution} that $U\mapsto C^\bullet(U)$ is a homotopy sheaf on any space (and it should be thought of as a resolution of the constant sheaf).  To prove Poincar\'e duality for a topological manifold $M$ of dimension $n$, we will show that $K\mapsto C_{\dim M-\bullet}(M,M\setminus K)$ is a pure homotopy $\K$-sheaf and calculate its $H^0$ as $\oo_M$ (in other words, it should be thought of as a resolution of the orientation sheaf of $M$).

\begin{convention}
Throughout this paper, we make no second countability or paracompactness assumptions on manifolds (topological or smooth).
\end{convention}

\begin{definition}[Orientation sheaf of manifold]\label{manordef}
Let $M$ be a topological manifold.  We let $\oo_M$ denote the \emph{orientation sheaf}\footnote{Note that the fundamental class lies in homology twisted by $\oo_M^\vee$.} of $M$ (the corresponding $\K$-sheaf is defined by $\oo_M(K):=H_{\dim M}(M,M\setminus K)$); it is locally isomorphic to the constant sheaf $\underline\ZZ$ (and following Convention \ref{signconventions}, it has parity $\dim M\in\ZZ/2$).  Now let $M$ be a topological manifold with boundary, and define orientation sheaves on $M$:
\begin{align}
\oo_M&:=j_*\oo_{M\setminus\partial M}\\
\oo_{M\rel\partial}&:=j_!\oo_{M\setminus\partial M}
\end{align}
where $j:M\setminus\partial M\hookrightarrow M$.  Then $\oo_M$ is locally isomorphic to the constant sheaf $\underline\ZZ$, and there is a sequence of sheaves which is exact on stalks:
\begin{equation}
0\to\oo_{M\rel\partial}\to\oo_M\to i_\ast\oo_{\partial M}\to 0
\end{equation}
where $i:\partial M\hookrightarrow M$, and the second map comes from the boundary map in the long exact sequence of the pair $(M,\partial M)$.
\end{definition}

\begin{lemma}[Homotopy $\K$-sheaf axioms for singular chains]\label{singularchainshomotopyKsheaf}
Let $X$ be a topological space.  Then we have:
\begin{rlist}
\item\label{scacyc}$C_\bullet(X,X)$ is acyclic.
\item\label{scpatch}Let $A,B\subseteq X$ be closed.  Then the complex:
\begin{equation*}
\bigl[C_\bullet(X,X\setminus(A\cup B))\to C_{\bullet+1}(X,X\setminus A)\oplus C_{\bullet+1}(X,X\setminus B)\to C_{\bullet+2}(X,X\setminus(A\cap B))\bigr]
\end{equation*}
is acyclic.
\item\label{sccont}Let $K=\bigcap_{\alpha\in A}K_\alpha$ where $\{K_\alpha\}_{\alpha\in A}$ is a family of closed subsets of $X$ which is \emph{filtered} in the sense that for all $\alpha_1,\alpha_2\in A$, there exists $\beta\in A$ with $K_\beta\subseteq K_{\alpha_1}\cap K_{\alpha_2}$.  Then $\varinjlim_{\alpha\in A}C_\bullet(X,X\setminus K_\alpha)\to C_\bullet(X,X\setminus K)$ is a quasi-isomorphism.
\end{rlist}
\end{lemma}

\begin{proof}
Statement (\ref{scacyc}) is obvious.

Statement (\ref{scpatch}) can be deduced from Mayer--Vietoris using a form of the nine lemma as we now explain.  Let us write $U:=X\setminus A$ and $V:=X\setminus B$.  Now consider the following total complex:
\begin{equation}
\begin{CD}
C_\bullet(U\cap V)@>>>C_{\bullet+1}(U)\oplus C_{\bullet+1}(V)@>>>C_{\bullet+2}(U\cup V)\cr
@VVV@VVV@VVV\cr
C_{\bullet+1}(X)@>>>C_{\bullet+2}(X)\oplus C_{\bullet+2}(X)@>>>C_{\bullet+3}(X)\cr
@VVV@VVV@VVV\cr
C_{\bullet+2}(X,U\cap V)@>>>C_{\bullet+3}(X,U)\oplus C_{\bullet+3}(X,V)@>>>C_{\bullet+4}(X,U\cup V)\cr
\end{CD}
\end{equation}
The columns are acyclic (by definition of relative chains), and hence the total complex is acyclic as well.  The first row is acyclic by Mayer--Vietoris, and the second row is obviously acyclic.  Thus the third row is acyclic, as needed.

Statement (\ref{sccont}) is true because the map is in fact an isomorphism on the chain level.  It is clearly surjective; to show injectivity we must show that a singular chain on $X$ which is disjoint from $K$ is in fact disjoint from $K_\alpha$ for some $\alpha\in A$.  This follows because the standard $n$-simplex is compact.
\end{proof}

\begin{lemma}[Poincar\'e--Lefschetz duality]\label{poincareduality}
Let $M$ be a topological manifold of dimension $n$ with boundary.  Let $i:X\hookrightarrow M$ be a closed subset.  Let $N\subseteq\partial M$ (closed) be a tamely embedded codimension zero submanifold with boundary.\footnote{In other words, $N\subseteq\partial M$ is a closed subset which locally looks like either $\varnothing\subseteq\RR^{n-1}$, $\RR_{\geq 0}\times\RR^{n-2}\subseteq\RR^{n-1}$, or $\RR^{n-1}\subseteq\RR^{n-1}$.}  Then there is a canonical isomorphism:
\begin{equation}\label{poincaredualitygoal}
H^\bullet[C_{n-1-\bullet}(N,N\setminus X)\to C_{n-\bullet}(M,M\setminus X)]=\cH^\bullet_c(X;i^\ast j_!j^\ast\oo_M)
\end{equation}
where $j:M^\circ\cup N^\circ\hookrightarrow M$.  In particular, specializing to $N=\varnothing$, there is a canonical isomorphism:
\begin{equation}
H_{n-\bullet}(M,M\setminus X)=\cH^\bullet_c(X;i^\ast\oo_{M\rel\partial M})
\end{equation}
\end{lemma}

\begin{proof}
Let $X^+$ be the one-point compactification of $X$.  Define a complex of $\K$-presheaves $\F^\bullet$ on $X^+$:
\begin{equation}\label{purehomotopyKsheafinPD}
\F^\bullet(K):=[C_{n-1-\bullet}(N,N\setminus K)\to C_{n-\bullet}(M,M\setminus K)]
\end{equation}
(where on the right hand side by $K$ we really mean $K\cap X$).  Applying Lemmas \ref{singularchainshomotopyKsheaf} and \ref{associatedgradedhomotopysheaf}, we see that $\F^\bullet$ is a homotopy $\K$-sheaf.

We claim that $\F^\bullet$ is a pure homotopy $\K$-sheaf.  Certainly the homology of $\F^\bullet$ is bounded below, because the singular chain complex of a topological manifold has homology bounded above and $\F^\bullet$ is built out of these.  Now, it is easy to calculate:
\begin{equation}\label{stalkhomologyofPDpure}
H^\bullet\F^\bullet_p=\begin{cases}\ZZ&p\in M^\circ\cup N^\circ\cr 0&p\notin M^\circ\cup N^\circ\end{cases}
\end{equation}
concentrated in degree zero (see \cite[p231 \S3.3]{hatcher} for the special case $N=\partial M=\varnothing$).  Hence $\F^\bullet$ is a pure homotopy $\K$-sheaf.  In fact, it is not hard to see (using the adjunction $j_!\dashv j^\ast$) that \eqref{stalkhomologyofPDpure} lifts to an isomorphism of sheaves $f_!i^\ast j_!j^\ast\oo_M\to H^0\F^\bullet$ where $f:X\hookrightarrow X^+$.

Now we conclude:
\begin{equation*}
H^\bullet\F^\bullet(X^+)\overset{\text{Prop \ref{cechcohomologypurehomotopysheaf}}}=\cH^\bullet(X^+;f_!i^\ast j_!j^\ast\oo_M)\overset{\text{Lem \ref{cptcechembeddingproperty}}}=\cH^\bullet_c(X;i^\ast j_!j^\ast\oo_M)
\end{equation*}
Now observe that $H^\bullet\F^\bullet(X^+)$ is the left hand side of \eqref{poincaredualitygoal}.
\end{proof}

\begin{remark}[Why homotopy $\K$-sheaves instead of homotopy sheaves?]
The naive modification of \eqref{purehomotopyKsheafinPD} substituting open $U\subseteq X$ in place of compact $K\subseteq X$ does not yield a homotopy sheaf.  To get a homotopy sheaf, one could apply the proposed functor $R\alpha_\ast:\hShv_\K X\to\hShv X$ from Remark \ref{homotopyKnotsheavesequivalencermk} to the homotopy $\K$-sheaf \eqref{purehomotopyKsheafinPD}.  It is somewhat easier, though, to just work directly in the setting of homotopy $\K$-sheaves (which has some advantages, for example stalks of $\K$-presheaves are easier to define/understand).  It is for this reason that throughout this paper we work with homotopy $\K$-sheaves instead of homotopy sheaves.
\end{remark}

\subsection{Homotopy colimits}\label{hocolimintro}

We make common use of the following type of ``homotopy diagram'' and its corresponding ``homotopy colimit''.  

\begin{definition}[Homotopy diagram]
Let $\SSS$ be a finite poset.  A \emph{homotopy diagram over $\SSS$} is a collection of complexes $\{A^\bullet_{\s,\ttt}\}_{\s\preceq\ttt\in\SSS}$ equipped with compatible maps $A^\bullet_{\s,\ttt}\to A^\bullet_{\s',\ttt'}$ for $\s\preceq\s'\preceq\ttt'\preceq\ttt$ (meaning that $A^\bullet_{\s,\ttt}\to A^\bullet_{\s',\ttt'}\to A^\bullet_{\s'',\ttt''}$ and $A^\bullet_{\s,\ttt}\to A^\bullet_{\s'',\ttt''}$ agree).
\end{definition}

\begin{definition}[Homotopy colimit]\label{homotopycolimit}
Let $\SSS$ be a finite poset and let $\{A^\bullet_{\s,\ttt}\}_{\s\preceq\ttt\in\SSS}$ be a homotopy diagram over $\SSS$.  We define the \emph{homotopy colimit}:
\begin{equation}
\hocolim_{\s\preceq\ttt\in\SSS}A^\bullet_{\s,\ttt}:=\bigoplus_{p\geq 0}\bigoplus_{\s_0\prec\cdots\prec\s_p\in\SSS}A^{\bullet+p}_{\s_0,\s_p}
\end{equation}
with differential (decreasing $p$) given by the alternating sum over forgetting one of the $\s_i$ (plus the internal differential).\footnote{This can be interpreted as the complex of simplicial chains on the nerve of $\SSS$ using a coefficient system determined by $\{A^\bullet_{\s,\ttt}\}$.}  Loosely speaking, we are ``gluing together'' the $\{A^\bullet_{\s,\s}\}_{\s\in\SSS}$ along the ``morphisms'' $A^\bullet_{\s,\s}\leftarrow A^\bullet_{\s,\ttt}\to A^\bullet_{\ttt,\ttt}$ for $\s\preceq\ttt$.
\end{definition}

\begin{lemma}[Terminal object for $\hocolim$]\label{comboretraction}
Let $\SSS$ be a finite poset with unique maximal element $\s^\ttop$, and let $\{A^\bullet_{\s,\ttt}\}_{\s\preceq\ttt\in\SSS}$ be a homotopy diagram with the property that every map $A^\bullet_{\s,\ttt}\to A^\bullet_{\s,\ttt'}$ ($\s\preceq\ttt'\preceq\ttt$) is a quasi-isomorphism.  Then the natural inclusion $A^\bullet_{\s^\ttop,\s^\ttop}\to\hocolim_{\s\preceq\ttt\in\SSS}A^\bullet_{\s,\ttt}$ is a quasi-isomorphism.
\end{lemma}

\begin{proof}
We filter $\hocolim_{\mathfrak s\preceq\mathfrak t\in\SSS}A^\bullet_{\mathfrak s,\mathfrak t}$ by the number of $\mathfrak s_0,\ldots,\mathfrak s_p$ which are not equal to $\mathfrak s^\ttop$.  The zeroth associated graded piece is the subcomplex $A^\bullet_{\mathfrak s^\ttop,\mathfrak s^\ttop}$, so it suffices to show that all the other associated graded pieces are acyclic.  Each of these is a direct sum of mapping cones $[A^{\bullet+1}_{\mathfrak s,\mathfrak s^\ttop}\to A^\bullet_{\mathfrak s,\mathfrak t}]$, which are acyclic by assumption.
\end{proof}

\begin{lemma}[$\hocolim$ preserves quasi-isomorphisms]\label{simplicialchainsquasiisomorphism}
Fix a finite poset $\SSS$ and let $\{A^\bullet_{\mathfrak s,\mathfrak t}\}_{\mathfrak s\preceq\mathfrak t\in\SSS}$ and $\{B^\bullet_{\mathfrak s,\mathfrak t}\}_{\mathfrak s\preceq\mathfrak t\in\SSS}$ be homotopy diagrams over $\SSS$.  Suppose that there are compatible quasi-isomorphisms $A^\bullet_{\mathfrak s,\mathfrak t}\to B^\bullet_{\mathfrak s,\mathfrak t}$.  Then the induced map $\hocolim_{\mathfrak s\preceq\mathfrak t\in\SSS}A^\bullet_{\mathfrak s,\mathfrak t}\to\hocolim_{\mathfrak s\preceq\mathfrak t\in\SSS}B^\bullet_{\mathfrak s,\mathfrak t}$ is a quasi-isomorphism.
\end{lemma}

\begin{proof}
Since the functor $\hocolim_{\mathfrak s\preceq\mathfrak t\in\SSS}$ commutes with the formation of mapping cones, it suffices to show that if each $A^\bullet_{\mathfrak s,\mathfrak t}$ is acyclic, then so is $\hocolim_{\mathfrak s\preceq\mathfrak t\in\SSS}A^\bullet_{\mathfrak s,\mathfrak t}$.  This holds since in this case it has a finite filtration whose associated graded is acyclic.
\end{proof}

\begin{definition}[Tensor product of homotopy diagrams]\label{simplicialproduct}
Let $\SSS$ and $\T$ be two finite posets, and let $\{A^\bullet_{\mathfrak s,\mathfrak s'}\}$ and $\{B^\bullet_{\mathfrak t,\mathfrak t'}\}$ be homotopy diagrams over $\SSS$ and $\T$ respectively.  Their tensor product $\{A^\bullet_{\mathfrak s,\mathfrak s'}\otimes B^\bullet_{\mathfrak t,\mathfrak t'}\}$ is naturally a homotopy diagram over $\SSS\times\T$.  Now there is a natural morphism:
\begin{equation}\label{tensorproductdirectlimitmorphism}
\hocolim_{\mathfrak s\preceq\mathfrak s'\in\SSS}A^\bullet_{\mathfrak s,\mathfrak s'}\otimes\hocolim_{\mathfrak t\preceq\mathfrak t'\in\T}B^\bullet_{\mathfrak t,\mathfrak t'}\to\hocolim_{\mathfrak s\times\mathfrak t\preceq\mathfrak s'\times\mathfrak t'\in\SSS\times\T}A^\bullet_{\mathfrak s,\mathfrak s'}\otimes B^\bullet_{\mathfrak t,\mathfrak t'}
\end{equation}
To define this morphism, we simply observe that the nerve of $\SSS\times\T$ is the standard simplicial subdivision of the product of the nerves of $\SSS$ and $\T$, and this is covered by a morphism of coefficient systems.\footnote{In fact, from this perspective one easily sees that \eqref{tensorproductdirectlimitmorphism} is always a quasi-isomorphism.}
\end{definition}

\subsection{Homotopy colimits of pure homotopy \texorpdfstring{$\K$}{K}-sheaves}\label{cechsimplegluingsection}

We introduce a gluing construction for pure homotopy $\K$-sheaves.  The relevance of this Lemma \ref{gluingcechsimple} is best understood in its (only) intended application, namely Lemma \ref{Cpure} (which is best understood in context).

\begin{lemma}[$\hocolim$ preserves homotopy $\K$-sheaves]\label{pihomotopysheaf}
Let $\SSS$ be a finite poset and let $\{\F^\bullet_{\mathfrak s,\mathfrak t}\}_{\mathfrak s\preceq\mathfrak t\in\SSS}$ be a homotopy diagram of homotopy $\K$-sheaves.  Then $\hocolim_{\mathfrak s\preceq\mathfrak t\in\SSS}\F^\bullet_{\mathfrak s,\mathfrak t}$ is a homotopy $\K$-sheaf.
\end{lemma}

\begin{proof}
The associated graded of the $p$-filtration on $\hocolim_{\mathfrak s\preceq\mathfrak t\in\SSS}\F^\bullet_{\mathfrak s,\mathfrak t}$ is a homotopy $\K$-sheaf by assumption; now use Lemma \ref{associatedgradedhomotopysheaf}.
\end{proof}

\begin{lemma}[Gluing pure homotopy $\K$-sheaves]\label{gluingcechsimple}
Let $A$ be a finite set.  Let $\{U_I\}_{I\subseteq A}$ be an open cover of a space $X$ satisfying:
\begin{rlist}
\item\label{nestedintersectionOK}$U_I\cap U_K\subseteq U_J$ for $I\subseteq J\subseteq K$.
\item\label{nonnestedintersectionbig}$U_I\cap U_{I'}\subseteq U_{I\cup I'}$ for $I,I'\subseteq A$.
\end{rlist}
Let $\G$ be a sheaf on $X$ and let $\G_{IJ}:=(j_{IJ})_!(j_{IJ})^\ast\G$ where $j_{IJ}:U_I\cap U_J\hookrightarrow X$.  By property (\ref{nestedintersectionOK}), this gives rise to a homotopy diagram $\{\G_{IJ}\}_{I\subseteq A}$ over $2^A$ of sheaves on $X$.

Let $\{\F_{IJ}^\bullet\}_{I\subseteq J\subseteq A}$ be a homotopy diagram over $2^A$ of pure homotopy $\K$-sheaves on $X$, and suppose we give a compatible system of isomorphisms $\G_{IJ}\xrightarrow\sim H^0\F^\bullet_{IJ}$.  Then $\F^\bullet:=\hocolim_{I\subseteq J\subseteq A}\F^\bullet_{IJ}$ is a pure homotopy $\K$-sheaf and there is a canonical induced isomorphism $\G\xrightarrow\sim H^0\F^\bullet$.
\end{lemma}

\begin{proof}
Certainly $\F^\bullet$ is a homotopy $\K$-sheaf by Lemma \ref{pihomotopysheaf}.  Since $A$ is finite, it is easy to see that $H^i\F^\bullet_{IJ}$ being bounded below implies that $H^i\F^\bullet$ is bounded below.

Now let us calculate $\F^\bullet_p$ using the spectral sequence associated to the $p$-filtration.  The $E_1$ term is concentrated along the $q=0$ row since $(H^q\F^\bullet_{IJ})_p$ is concentrated in degree zero; thus there are no further differentials after the $E_1$ page.  On the $E_1$ page the differentials coincide precisely with the differentials in the definition of $\hocolim_{I\subseteq J\subseteq A}(\G_{IJ})_p$ (regarding $\G_{IJ}$ as complexes concentrated in degree zero).  Hence we have an isomorphism:
\begin{equation}\label{FGiso}
H^\bullet\F^\bullet_p=H^\bullet\hocolim_{I\subseteq J\subseteq A}(\G_{IJ})_p
\end{equation}
To calculate the right hand side of \eqref{FGiso}, let us start with the trivial observation that:
\begin{equation}
(\G_{IJ})_p=\begin{cases}\G_p&p\in U_I\cap U_J\cr 0&p\notin U_I\cap U_J\end{cases}
\end{equation}
Now consider $(2^A)_p:=\{I\subseteq A:p\in U_I\}$, which satisfies the following properties:
\begin{rlist}
\item$(2^A)_p$ has a maximal element (restatement of (\ref{nonnestedintersectionbig})).
\item$I,K\in(2^A)_p$ implies $J\in(2^A)_p$ for $I\subseteq J\subseteq K$ (restatement of (\ref{nestedintersectionOK})).
\end{rlist}
These two properties imply that $\hocolim_{I\subseteq J\subseteq A}(\G_{IJ})_p$ is simply $\G_p$ tensored with the simplicial chain complex of the nerve of $(2^A)_p$ (which is contractible).  Hence \eqref{FGiso} gives a canonical isomorphism of stalks $H^0\F^\bullet_p=\G_p$ (and thus in particular $\F^\bullet$ is pure).

Now it remains to construct a canonical isomorphism of sheaves $\G\xrightarrow\sim H^0\F^\bullet$.  Over the open set $U_I$, we define this isomorphism to be the composition $\G_{II}\xrightarrow\sim H^0\F_{II}^\bullet\to H^0\F^\bullet$ (the second map being the inclusion of a $p=0$ subcomplex in the homotopy colimit).  At any point $p\in U_I$, this specializes to the isomorphism of stalks $\G_p=H^0\F^\bullet_p$ defined earlier.  Hence since $\{U_I\}_{I\subseteq A}$ is an open cover of $X$, these isomorphisms patch together to give the desired global isomorphism.
\end{proof}

\subsection{Steenrod homology}\label{steenrodhomologysec}

Steenrod homology $\sH_\bullet$ is a homology theory for compact Hausdorff spaces.  It is characterized uniquely by a certain set of axioms due to Berikashvili \cite{berikashviliI,berikashviliII} (see also Goldfarb \cite[p355, \S7.4]{goldfarb} and Inassaridze \cite{inassaridze}); a simpler axiomatic characterization of Steenrod homology on compact metrizable spaces is due to Milnor \cite{milnor}.  Note that it follows as usual from these axioms that Steenrod homology coincides with singular homology on finite CW-complexes.  Steenrod homology is due to Steenrod for compact metrizable spaces \cite{steenrodreg} and was later generalized to compact Hausdorff spaces (some sources include Edwards--Hastings \cite{edwardshastings}, Hastings \cite{hastings}, Carlsson--Pedersen \cite{carlssonpedersen}, and Goldfarb \cite{goldfarb}).  Another reference is Marde\v si\'c's book \cite{mardesic} (note that Marde\v si\'c studies the more general theory of strong homology, which coincides with Steenrod homology on compact spaces).

\subsubsection{\v Cech cochains}

The following definition is due to Carlsson--Pedersen \cite{carlssonpedersen}.  The idea of using open covers indexed by the points of the space being covered goes back at least to Godement \cite[II \S 5.8]{godement}, and was also used by Friedlander \cite{friedlander} in the context of \'etale homotopy theory.

\begin{definition}[Rigid open cover]
A \emph{rigid open cover} of a compact space $X$ consists of open sets $\{U_x\subseteq X\}_{x\in X}$ such that $x\in U_x$, $\overline{\{x:U_x=U\}}\subseteq U$, and $\#\{U:U=U_x\text{ for some }x\}<\infty$.  Warning: even if $U_x=U_y$ for some $x\ne y$, they are still different elements of the cover; in particular, the nerve of a rigid cover is always infinite unless $X$ is finite.
\end{definition}

For doing \v Cech theory, the category of rigid open covers is technically more convenient than the usual cateory of open covers (c.f.\ Remark \ref{coveringscategoryrmk}).  Specifically, the collection of rigid covers forms a set, and there is at most one morphism between any pair of rigid covers (we only consider refinements which act as the identity on the index set $X$).

\begin{definition}[$\cC^\bullet$ and $\cC^\bullet_c$ of sheaves; c.f.\ Definition \ref{cechdefnI}]
Let $\F$ be a sheaf.  We define the \v Cech cochains:
\begin{equation}\label{cechcomplexIderived}
\cC^\bullet(X;\F):=\varinjlim_{\begin{smallmatrix}\{U_x\subseteq X\}_{x\in X}\cr\text{rigid open cover}\end{smallmatrix}}\bigoplus_{p\geq 0}\prod_{\begin{smallmatrix}S\subseteq X\cr\left|S\right|=p+1\end{smallmatrix}}\F\Bigl(\bigcap_{x\in S}U_x\Bigr)[-p]
\end{equation}
with the standard \v Cech differential.  For any compact $K\subseteq X$, define $\cC^\bullet_K(X;\F)$ (\v Cech cochains with supports in $K$) via \eqref{cechcomplexIderived} except replacing every instance of $\F(U)$ with $\ker[\F(U)\to\F(U\setminus K)]$.  We let $\cC^\bullet_c(X;\F):=\varinjlim_{K\subseteq X}\cC^\bullet_K(X;\F)$ (\v Cech cochains with compact supports).

Clearly the homology of $\cC^\bullet(X;\F)$ (resp.\ $\cC^\bullet_K(X;\F)$, $\cC^\bullet_c(X;\F)$) is $\cH^\bullet(X;\F)$ (resp.\ $\cH^\bullet_K(X;\F)$, $\cH^\bullet_c(X;\F)$).
\end{definition}

\begin{definition}[$\cC^\bullet$ of complexes of $\K$-presheaves; c.f.\ Definition \ref{cechdefnII}]
Let $\F^\bullet$ be a complex of $\K$-presheaves on a compact space $X$.  We define:
\begin{equation}\label{cechcomplexIIderived}
\cC^\bullet(X;\F^\bullet):=\varinjlim_{\begin{smallmatrix}\{U_x\subseteq X\}_{x\in X}\cr\text{rigid open cover}\end{smallmatrix}}\bigoplus_{p\geq 0}\prod_{\begin{smallmatrix}S\subseteq X\cr\left|S\right|=p+1\end{smallmatrix}}\F^{\bullet-p}\Bigl(\bigcap_{x\in S}\overline{U_x}\Bigr)
\end{equation}
with the standard \v Cech differential (plus the internal differential of $\F^\bullet$).  We also define $\cC^\bullet(X;\F)$ for any $\K$-presheaf $\F$ by viewing it as a complex concentrated in degree zero.

The homology of $\cC^\bullet(X;\F^\bullet)$ is $\cH^\bullet(X;\F^\bullet)$, as can be seen by applying the arguments from the proof of Lemma \ref{cHsameforopencompactcovers}.  Note that for a sheaf $\F$, the complexes \eqref{cechcomplexIderived} and \eqref{cechcomplexIIderived} are canonically isomorphic.
\end{definition}

\begin{definition}[Pullback on $\cC^\bullet$ and $\cC^\bullet_c$; c.f.\ Definitions \ref{pullbackdefn}, \ref{pullbackcHKprshv}]\label{cechcochainspullback}
Let $f:X\to Y$ be a map of spaces.  A rigid open cover $\{U_y\}_{y\in X}$ pulls back to a rigid open cover $\{f^{-1}(U_{f(x)})\}_{x\in X}$, and this gives natural maps:
\begin{rlist}
\item$f^\ast:\cC^\bullet(Y;f_\ast\F)\to\cC^\bullet(X;\F)$ for $\F\in\Prshv X$.
\item$f^\ast:\cC^\bullet_c(Y;f_\ast\F)\to\cC^\bullet_c(X;\F)$ for $\F\in\Shv X$ (if $f$ is proper).
\item$f^\ast:\cC^\bullet(Y;f_\ast\F^\bullet)\to\cC^\bullet(X;\F^\bullet)$ for $\F^i\in\Prshv_\K X$.
\end{rlist}
whose action on homology coincide with the maps $f^\ast$ defined earlier.
\end{definition}

\subsubsection{Derived inverse limits}

\begin{definition}[Derived inverse limit $\Rlim$]
Let $\{C^\bullet_\lambda\}_{\lambda\in\Lambda}$ be an inverse system of complexes.  We define:
\begin{equation}
\Rlim_{\lambda\in\Lambda}C^\bullet_\lambda:=\prod_{q\geq 0}\prod_{\lambda_0\leq\ldots\leq\lambda_q}C^{\bullet-q}_{\lambda_0}
\end{equation}
with differential obtained by viewing this as cochains on the nerve of $\Lambda$ with a particular coefficient system.  See Marde\v si\'c \cite[\S 17]{mardesic} for more details on and basic properties of $\Rlim$.
\end{definition}

\begin{definition}[Derived functors $\varprojlim^i$]
Let $\{A_\lambda\}_{\lambda\in\Lambda}$ be an inverse system of abelian groups.  We define:
\begin{equation}\label{limidef}
\mathop{\textstyle\varprojlim^i}\limits_{\lambda\in\Lambda}A_\lambda:=H^i\Bigl[\Rlim_{\lambda\in\Lambda}A_\lambda\Bigr]
\end{equation}
(viewing $A_\lambda$ as an inverse system of complex concentrated in degree zero).  The inverse limit functor $\varprojlim$ from inverse systems of abelian groups indexed by $\Lambda$ to abelian groups is left exact, and $\varprojlim^i$ are its right derived functors (see Marde\v si\'c \cite[Corollary 11.47]{mardesic}).  See Marde\v si\'c \cite[\S\S 11--15]{mardesic} for more details on and basic properties of $\varprojlim^i$.
\end{definition}

\begin{lemma}[Cofinality for $\varprojlim^i$ and $\Rlim$]\label{cofinalityforRlim}
Let $f:\Lambda'\to\Lambda$ be weakly increasing ($\lambda_1\leq\lambda_2\implies f(\lambda_1)\leq f(\lambda_2)$) and cofinal ($f(\Lambda')\subseteq\Lambda$ cofinal).  Then the following natural map is an isomorphism:
\begin{equation}
\mathop{\textstyle\varprojlim^i}\limits_{\lambda\in\Lambda}A_\lambda\xrightarrow\sim\mathop{\textstyle\varprojlim^i}\limits_{\lambda'\in\Lambda'}A_{f(\lambda')}
\end{equation}
More generally, the following natural map is a quasi-isomorphism:
\begin{equation}
\Rlim_{\lambda\in\Lambda}C^\bullet_\lambda\xrightarrow\sim\Rlim_{\lambda'\in\Lambda'}C^\bullet_{f(\lambda')}
\end{equation}
\end{lemma}

\begin{proof}
For $\textstyle\varprojlim^i$, see \cite[p291, Theorem 14.9]{mardesic}, and for $\Rlim$, see \cite[p349, Corollary 17.18]{mardesic}.
\end{proof}

\begin{lemma}\label{quasiisomorphismforRlim}
If $C_\lambda^\bullet\xrightarrow\sim D_\lambda^\bullet$ is a morphism of inverse systems which is a level-wise quasi-isomorphism (i.e.\ for every $\lambda$), then the natural map:
\begin{equation}
\Rlim_{\lambda\in\Lambda}C^\bullet_\lambda\xrightarrow\sim\Rlim_{\lambda\in\Lambda}D^\bullet_\lambda
\end{equation}
is a quasi-isomorphism.
\end{lemma}

\begin{proof}
See \cite[p348, Theorem 17.16]{mardesic}.
\end{proof}

\subsubsection{Steenrod chains}

\begin{definition}[due to Chogoshvili \cite{chogoshvili}]
A partition $X=\bigsqcup_{i=1}^nE_i$ shall mean an \emph{unordered} partition into finitely many disjoint nonempty subsets $E_1,\ldots,E_n\subseteq X$; there is an associated finite closed covering $X=\bigcup_{i=1}^n\overline{E_i}$, whose nerve is denoted $N(X;\{\overline{E_i}\}_{i=1}^n)$ (the simplicial complex with vertices $\{1,\ldots,n\}$ where $S\subseteq\{1,\ldots,n\}$ spans a simplex iff $\bigcap_{i\in S}\overline{E_i}\ne\varnothing$).  A partition $\{F_j\}_{j=1}^m$ refines $\{E_i\}_{i=1}^n$ iff for all $j$, there exists a (necessarily unique) $i$ with $F_j\subseteq E_i$.  For a refinement $\{E_i\}_{i=1}^n\to\{F_j\}_{j=1}^m$, there is an associated map of nerves $N(X;\{\overline{F_j}\})\to N(X;\{\overline{E_i}\})$.
\end{definition}

For doing \v Cech theory, Chogoshvili's construction has exceptionally nice properties.  The collection of partitions is a set, and there is at most one morphism between any pair of partitions.  The poset of partitions is cofinite (a given partition refines only finitely many other partitions).  Each nerve $N(X;\{\overline{E_i}\})$ is a finite simplicial complex, and the transition maps $N(X;\{\overline{F_j}\})\to N(X;\{\overline{E_i}\})$ are all surjective.  Also, any pair of partitions has a minimal common refinement.  Given a map $f:X\to Y$, a partition $\{E_i\}$ of $Y$ pulls back to a partition $\{f^{-1}(E_i)\}$ of $X$, and there is an associated map on nerves $N(X;\{\overline{f^{-1}(E_i)}\})\to N(Y,\{\overline{E_i}\})$.

\begin{lemma}\label{inverselimithassamecech}
Let $X$ be compact.  There is a natural map:
\begin{equation}\label{chogoshvilinervecomparison}
\varprojlim_{X=\bigsqcup_{i=1}^nE_i}N(X;\{\overline{E_i}\})\to X
\end{equation}
which induces an isomorphism on \v Cech cohomology.
\end{lemma}

\begin{proof}
There is a natural correspondence $C\subseteq X\times N(X;\{\overline E_i\})$, whose fiber over $x\in X$ is the complete simplex on $\{i:x\in\overline{E_i}\}$.  Now it is not hard to check that the inverse limit of these correspondences maps bijectively (and thus homeomorphically) to $\varprojlim_{X=\bigsqcup_{i=1}^nE_i}N(X;\{\overline{E_i}\})$, thus giving rise to the desired map \eqref{chogoshvilinervecomparison}.

Now the inverse image of $p\in X$ under \eqref{chogoshvilinervecomparison} is an inverse limit of complete simplices with surjective transition maps (``the complete simplex on a profinite set''), and this has the \v Cech cohomology of a point.  Indeed, for any inverse system of compact spaces $\{X_\alpha\}$, the natural map $\varinjlim_\alpha\cH^\bullet(X_\alpha)\to\cH^\bullet(\varprojlim_\alpha X_\alpha)$ is an isomorphism (since open covers of $\varprojlim_\alpha X_\alpha$ pulled back from some $X_\alpha$ are cofinal among all open covers).  Thus it follows from the Leray spectral sequence that \eqref{chogoshvilinervecomparison} induces an isomorphism on \v Cech cohomology.
\end{proof}

\begin{definition}[Steenrod chains and homology]\label{steenrodhomdef}
Let $X$ be compact.  We define:
\begin{equation}\label{steenrodchainsRlimeqn}
\sC_\bullet(X):=\Rlim_{X=\bigsqcup_{i=1}^nE_i}\bigoplus_{p\geq 0}\bigoplus_{\begin{smallmatrix}S\subseteq\{1,\ldots,n\}\cr\left|S\right|=p+1\cr\bigcap_{i\in S}\overline{E_i}\ne\varnothing\end{smallmatrix}}\ZZ[-p]
\end{equation}
For a map $f:X\to Y$, there is an induced pushforward map $f_\ast:\sC_\bullet(X)\to\sC_\bullet(Y)$.  We denote by $\sH_\bullet(X;\F)$ the homology of $\sC_\bullet(X;\F)$, and we define relative homology $\sH_\bullet(X,Y;\F)$ as the homology of the relevant mapping cone.

More generally, let $\F$ be a locally constant sheaf of abelian groups on $X$.  Fix a partition $X=\bigsqcup_{i=1}^nE_i^0$ and trivializations of $\F|_{\overline{E_i^0}}$.  This gives rise to a local system $\F$ on $N(\{\overline{E_i}\}_{i=1}^n)$ for any partition $\{E_i\}$ refining $\{E_i^0\}$.  We now define $\sC_\bullet(X;\F)$ as in \eqref{steenrodchainsRlimeqn}, restricted to partitions refining $\{E_i^0\}$ and using the coefficient system on the nerve induced by $\F$ and the fixed trivializations of $\F|_{\overline{E_i^0}}$.  Since any two choices of trivializations of $\F|_{\overline{E_i}}$ become isomorphic (in the sense that that their sets of ``constant sections'' coincide) after pulling back to some common refinement, it follows that $\sC_\bullet(X;\F)$ is well-defined up to essentially unique quasi-isomorphism (by Lemma \ref{cofinalityforRlim}).  For a map $f:X\to Y$ and $\F\to f^\ast\G$, there is a pushforward map $\sC_\bullet(X;\F)\to\sC_\bullet(Y;\G)$.
\end{definition}

Although not logically necessary for our purposes, we now argue that Steenrod homology as given in Definition \ref{steenrodhomdef} coincides with the definition from Marde\v si\'c \cite{mardesic} (at least for constant coefficient systems).  First of all, note that (a special case of) Marde\v si\'c's definition of strong homology is that if a compact space $X$ is an inverse limit of compact polyhedra $X_\alpha$ over a cofinite index set, then:
\begin{equation}
\sH_\bullet(X)=H_\bullet\Rlim_\alpha C_\bullet(X_\alpha)
\end{equation}
(see \cite[p379, \S19.1]{mardesic}, and note that such an inverse system is a ``cofinite polyhedral resolution'' of its inverse limit \cite[p103, \S6]{mardesic}).  In particular, $\sC_\bullet(X)$ as defined in \eqref{steenrodchainsRlimeqn} computes the Steenrod homology of $\varprojlim_{X=\bigsqcup_{i=1}^nE_i}N(X;\{\overline{E_i}\})$ (using the morphism of inverse systems $C_\bullet^\mathrm{simp}(N(X;\{\overline{E_i}\}))\to C_\bullet^\mathrm{sing}(N(X;\{\overline{E_i}\}))$ given by barycentric subdivision, which is a level-wise quasi-isomorphism and thus induces a quasi-isomorphism on derived inverse limits by Lemma \ref{quasiisomorphismforRlim}).  Finally, note that the map \eqref{chogoshvilinervecomparison} induces an isomorphism on strong homology by Lemma \ref{inverselimithassamecech} and \cite[p446, Theorem 21.15]{mardesic}.

\begin{lemma}[Steenrod homology is the derived dual of \v Cech cohomology]\label{steenrodisderiveddualofcech}
Let $\F$ be a locally constant sheaf whose stalks are finitely generated free $R$-modules, and let $M$ be an $R$-module.  Then there is a natural isomorphism:
\begin{equation}\label{strongcechdualiso}
\sC_\bullet(X;\Hom(\F,M))=R\Hom_{D(R)}(\cC^\bullet(X;\F),M)
\end{equation}
in the derived category $D(R)$.  The same holds for relative (co)chains of a closed subspace $Y\subseteq X$, and the isomorphisms are compatible with the relevant exact triangles relating $Y$, $X$ and $(X,Y)$.
\end{lemma}

This result is very similar to Milnor \cite[p92, Definition, Lemma 5]{milnor} and Marde\v si\'c \cite[p446, Theorem 21.15]{mardesic}.  Note also that morally, the above result should be thought of as following trivially from some derived universal property of the form $\Rlim R\Hom(A_\alpha^\bullet,B^\bullet)=R\Hom(\varinjlim A_\alpha^\bullet,B^\bullet)$ (hopefully this provides some motivation for the proof below).

\begin{proof}
Let us relate $\cC^\bullet(X;\F)$ as defined using rigid open covers to a version using partitions.  Specifically, we consider the following quasi-isomorphisms:
\begin{align}
\cC^\bullet(X;\F)=\varinjlim_{\begin{smallmatrix}\{U_x\subseteq X\}_{x\in X}\cr\text{rigid open cover}\end{smallmatrix}}&\bigoplus_{p\geq 0}\prod_{\begin{smallmatrix}S\subseteq X\cr\left|S\right|=p+1\end{smallmatrix}}\F\Bigl(\bigcap_{x\in S}U_x\Bigr)[-p]\\
&\quad\uparrow\cr
\varinjlim_{\begin{smallmatrix}X=\bigsqcup_{i=1}^nE_i\cr\overline{E_i}\subseteq U_i\end{smallmatrix}}&\bigoplus_{p\geq 0}\prod_{\begin{smallmatrix}S\subseteq\{1,\ldots,n\}\cr\left|S\right|=p+1\end{smallmatrix}}\F\Bigl(\bigcap_{i\in S}U_i\Bigr)[-p]\\
&\quad\downarrow\cr
\label{cechbychogoshvili}\varinjlim_{X=\bigsqcup_{i=1}^nE_i}&\bigoplus_{p\geq 0}\prod_{\begin{smallmatrix}S\subseteq\{1,\ldots,n\}\cr\left|S\right|=p+1\end{smallmatrix}}\F\Bigl(\bigcap_{i\in S}\overline{E_i}\Bigr)[-p]\\
&\quad\uparrow\cr
\label{homotopycolimitofchogoshvili}\bigoplus_{q\geq 0}\bigoplus_{\{E_i^{(0)}\}_{i=1}^{n_0}<\cdots<\{E_i^{(q)}\}_{i=1}^{n_q}}&\bigoplus_{p\geq 0}\prod_{\begin{smallmatrix}S\subseteq\{1,\ldots,n_0\}\cr\left|S\right|=p+1\end{smallmatrix}}\F\Bigl(\bigcap_{i\in S}\overline{E_i^{(0)}}\Bigr)[q-p]\\
&\quad\uparrow\downarrow\cr
\label{almosthomotopycolimitofchogoshvilicoverstodualize}\bigoplus_{q\geq 0}\bigoplus_{\{E_i^{(0)}\}_{i=1}^{n_0}\leq\cdots\leq\{E_i^{(q)}\}_{i=1}^{n_q}}&\bigoplus_{p\geq 0}\prod_{\begin{smallmatrix}S\subseteq\{1,\ldots,n_0\}\cr\left|S\right|=p+1\end{smallmatrix}}\F\Bigl(\bigcap_{i\in S}\overline{E_i^{(0)}}\Bigr)[q-p]
\end{align}
The first map is induced by associating to a partition $\{E_i\}$ and $\overline{E_i}\subseteq U_i$ the rigid open cover assigning $U_i$ to $x\in E_i$, and pulling back along the induced map on nerves; it can be seen to be a quasi-isomorphism by a filtration argument (basically the Leray spectral sequence).  The second map is a quasi-isomorphism by Lemma \ref{intersectioncofinal} and the definition of $\alpha^\ast$.  The third map is defined as the tautological map on $q=0$ direct summands and zero for $q>0$; it can be seen to be a quasi-isomorphism by expressing the domain (resp.\ codomain) as the direct limit over partitions $\{E_i\}$ of the corresponding complex (resp.\ direct limit) restricted to partitions refined by $\{E_i\}$ and observing that for any fixed $\{E_i\}$ the map is a quasi-isomorphism by Lemma \ref{comboretraction}.  Finally, the last pair of maps consists of the natural inclusion ($\downarrow$) and the retraction annihilating any component with $\{E_i^{(j)}\}_{i=1}^{n_j}=\{E_i^{(j+1)}\}_{i=1}^{n_{j+1}}$ for some $j$ ($\uparrow$); one can see that both are quasi-isomorphisms by filtering by $\{E_i\}$ (as before) and then by the number of distinct $\{E_i^{(j)}\}_{i=1}^{n_j}$.

For any partition $\{E_i\}$ of $X$ and collection of trivializations of $\F|_{\overline E_i}$, there is a sub-$\K$-presheaf $\F^\pre\subseteq\F$ whose sections over $K$ are those sections of $\F$ whose restrictions to $K\cap\overline{E_i}$ are constant with respect to the specified trivializations.  The complex \eqref{cechbychogoshvili} calculates \v Cech cohomology \eqref{cechcomplexII} by a cofinality argument, and the inclusion $\F^\pre\to\F$ induces an isomorphism on stalks and thus on \v Cech cohomology by Lemma \ref{cechdeterminedbystalks}.  The third map above remains a quasi-isomorphism for $\F^\pre$ (for the same reason), and furthermore we may restrict \eqref{almosthomotopycolimitofchogoshvilicoverstodualize} to partitions refining the fixed $\{E_i\}$ by Lemma \ref{cofinalityforRlim}.

The conclusion of the above discussion is that we have thus constructed a canonical (and functorial) quasi-isomorphism between $\cC^\bullet(X;\F)$ and:
\begin{equation}
\label{homotopycolimitofchogoshvilicoverstodualize}\bigoplus_{q\geq 0}\bigoplus_{\{E_i\}\leq\{E_i^{(0)}\}_{i=1}^{n_0}\leq\cdots\leq\{E_i^{(q)}\}_{i=1}^{n_q}}\bigoplus_{p\geq 0}\prod_{\begin{smallmatrix}S\subseteq\{1,\ldots,n_0\}\cr\left|S\right|=p+1\end{smallmatrix}}\F^\pre\Bigl(\bigcap_{i\in S}\overline{E_i^{(0)}}\Bigr)[q-p]
\end{equation}
(for any fixed partition $\{E_i\}$ and trivializations of $\F|_{\overline{E_i}}$ giving rise to $\F^\pre\subseteq\F$).  Note that $\sC_\bullet(X,\Hom(\F,M))$ is precisely $\Hom(\eqref{homotopycolimitofchogoshvilicoverstodualize},M)$.  Thus it suffices to show that this particular $\Hom$ is in fact the $R\Hom$.

Recall that $R\Hom$ in $D(R)$ may be computed using a K-projective resolution of the first argument (see Spaltenstein \cite{spaltenstein}), where a complex $P^\bullet$ is called K-projective iff $\Hom^\bullet(P^\bullet,M^\bullet)$ is acyclic for every acyclic complex $M^\bullet$ (note that a bounded above complex of projective modules is K-projective).  Thus it suffices to show that \eqref{homotopycolimitofchogoshvilicoverstodualize} is K-projective.

We consider subsets $\wp$ of the poset of partitions of $X$ with the property that any partition refined by a partition in $\wp$ also lies in $\wp$.  Since the poset of partitions is cofinite, we may use Zorn's lemma to choose a directed system $\{\wp_i\}_{i\in I}$ of such collections, indexed by a well-ordered set $I$, with the following properties:
\begin{rlist}
\item$\wp_0=\varnothing$, where $0\in I$ is the least element.
\item$\left|\wp_i\setminus\wp_{i-1}\right|=1$ if $i\in I$ has a predecessor $i-1$.
\item$\wp_i=\bigcup_{i'<i}\wp_{i'}$ if $i\in I$ has no predecessor.
\end{rlist}
Now for $i\in I$, let $\eqref{homotopycolimitofchogoshvilicoverstodualize}_i\subseteq\eqref{homotopycolimitofchogoshvilicoverstodualize}$ denote the subcomplex generated by restricting to partitions in $\wp_i$.  It follows that:
\begin{rlist}
\item$\eqref{homotopycolimitofchogoshvilicoverstodualize}_0=0$, where $0\in I$ is the least element.
\item$\eqref{homotopycolimitofchogoshvilicoverstodualize}_{i-1}\to\eqref{homotopycolimitofchogoshvilicoverstodualize}_i$ is injective and component-wise split with K-projective cokernel if $i\in I$ has a predecessor $i-1$ (this holds since the cokernel is a bounded above complex of free modules).
\item$\varinjlim_{i'<i}\eqref{homotopycolimitofchogoshvilicoverstodualize}_{i'}\to\eqref{homotopycolimitofchogoshvilicoverstodualize}_i$ is an isomorphism if $i\in I$ has no predecessor.
\end{rlist}
It follows that $\eqref{homotopycolimitofchogoshvilicoverstodualize}=\varinjlim_i\eqref{homotopycolimitofchogoshvilicoverstodualize}_i$ is K-projective by Spaltenstein \cite[p131, 2.8 Corollary]{spaltenstein}.

Since the construction of the isomorphism \eqref{strongcechdualiso} was given by a functorial chain-level construction, it can be checked that it is natural, applies in the relative case, and is compatibile with exact triangles.
\end{proof}

\section{Gluing for implicit atlases on Gromov--Witten moduli spaces}\label{gluingappendix}

In this appendix, we provide the gluing analysis used in \S\ref{gromovwittensection} to verify that the implicit atlases constructed there satisfy the openness and submersion axioms (specifically, we prove Proposition \ref{GWgluingneeded}).

The gluing theorem we prove here is of a very standard sort which has been treated (in different related settings) many times over in the literature, and so we make no claim of originality in this appendix.  The methods used here are based on our partial understanding of the treatments of gluing in Abouzaid \cite{abouzaid} and McDuff--Salamon \cite{mcduffsalamonJholsymp}, as well as conversations with Abouzaid, Ekholm, Hofer, and Mazzeo (we also thank the anonymous referee for their comments).  We understand that there is work in progress of McDuff--Wehrheim proving a similar result (keeping track of more smoothness) in their (very similar) setting of Kuranishi atlases \cite{mcduffwehrheim,mcduffnotes}.  We also refer the reader to Ekholm--Smith \cite{ekholmsmithrevisited}, Fukaya--Oh--Ohta--Ono \cite{foootechnicaldetails}, and Hofer--Wysocki--Zehnder \cite{polyfoldGW} for related gluing results.

The essential content of our result is that, for a certain moduli space of holomorphic curves $\Mbar(X)$, the regular locus $\Mbar(X)^\reg\subseteq\Mbar(X)$ (i.e.\ the locus where a certain linearized $\delbar$-operator is surjective) is a topological manifold.  We prove this by constructing local manifold charts covering $\Mbar(X)^\reg$ (since being a topological manifold is a property rather than extra structure, we need not address any question of compatibility between different local charts or of their compatibility with any auxiliary group action).

\subsection{Setup and main result}\label{setupforgluing}

We use new notation, unrelated to (and simpler than) that from \S\ref{gromovwittensection}.

Fix a smooth almost complex manifold:
\begin{equation}
(X,J)
\end{equation}
Fix codimension two submanifolds with boundary:
\begin{equation}
D,D_1,\ldots,D_r\subseteq X
\end{equation}
(let $D^\circ:=D\setminus\partial D$, and similarly for $D_i$).  Fix a smooth manifold $\Mbar$ equipped with a smooth \'etale map $\Mbar\to\Mbar_{g,n+\ell}$\footnote{We give $\Mbar_{g,n+\ell}$ the standard smooth structure from its structure as a complex analytic orbifold.  Every point in a smooth orbifold $M$ is in the image of a smooth \'etale map $M'\to M$.  Indeed, $M$ is covered by open sets each diffeomorphic to $\RR^n/G$ for some finite group $G\to GL_n(\RR)$, and the map $\RR^n\to\RR^n/G\hookrightarrow M$ is smooth \'etale.} (in the orbifold sense).  Denote by $\Cbar_{g,n+\ell}\to\Mbar_{g,n+\ell}$ the universal family, and define a family of curves $\Cbar\to\Mbar$ via the following pullback square:
\begin{equation}
\begin{CD}
\Cbar@>>>\Cbar_{g,n+\ell}\cr
@VVV@VVV\cr
\Mbar@>>>\Mbar_{g,n+\ell}
\end{CD}
\end{equation}
Let $\Cbar^\circ\subseteq\Cbar$ denote the open subset where the fiber is smooth.  Fix a finite dimensional vector space $E$ and a linear map:
\begin{equation}\label{lambdafix}
\lambda:E\to C^\infty(\overbrace{\Cbar^\circ\underset\Mbar\times\cdots\underset\Mbar\times\Cbar^\circ}^{r\text{ times}}{}\underset\Mbar\times\Cbar^\circ\times X,\Omega^{0,1}_{\Cbar^\circ/\Mbar}\otimes_\CC TX)
\end{equation}
Here $\Omega^{0,1}_{\Cbar^\circ/\Mbar}$ is the $(0,1)$-part of the complexified vertical cotangent bundle $T^\ast_{\Cbar^\circ/\Mbar}\otimes_\RR\CC$ of (the last factor of) $\Cbar^\circ\to\Mbar$, and $C^\infty$ means smooth sections.  We require that $\lambda$ be zero in a neighborhood of the nodes $\Cbar\setminus\Cbar^\circ$ of the last $\Cbar^\circ$ factor.

For a nodal curve $C$, we denote by $\tilde C$ its normalization, i.e.\ the unique smooth compact Riemann surface equipped with a map $\tilde C\to C$ which identifies points in pairs to form the nodes of $C$.  A function on $C$ being smooth means (by definition) that its pullback to $\tilde C$ is smooth.

Fix a homology class $\beta\in H_2(X;\ZZ)$.  We are interested in the following moduli space:
\begin{equation}\label{Mbarfirstdef}
\Mbar(X):=\left\{\begin{matrix}s\in\Mbar\hfill\cr u:C_s\to X\hfill\cr x_i\in C_s\quad 1\leq i\leq r\hfill\cr e\in E\hfill\end{matrix}\;\middle|\;\begin{matrix}u\text{ smooth and }u_\ast[C_s]=\beta\hfill\cr u(p_{n+i})\in D^\circ\quad\hfill(1\leq i\leq\ell)\cr u(x_i)\in D_i^\circ\text{ and }u\pitchfork D_i\text{ at }x_i\quad\hfill(1\leq i\leq r)\cr\delbar u+\lambda(e)(x_1,\ldots,x_r,\cdot,u(\cdot))=0\hfill\end{matrix}\right\}
\end{equation}
We let $C_s$ denote the fiber of $\Cbar\to\Mbar$ over $s$, with marked points $p_1,\ldots,p_{n+\ell}\in C_s$.  Here $u\pitchfork D_i$ at $x_i$ means (by definition) that $x_i$ is not a node of $C_s$ and the map $T_{x_i}C_s\oplus T_{u(x_i)}D_i\xrightarrow{du\oplus\id}T_{u(x_i)}X$ is surjective.  Now $\lambda(e)(x_1,\ldots,x_r,\cdot,u(\cdot))$ is a smooth section of $\Omega^{0,1}_{\tilde C_s}\otimes_\CC u^\ast TX$ over $C_s$ supported away from the nodes; hence the last equation makes sense.  We give $\Mbar(X)$ the topology of uniform convergence (i.e.\ using the Hausdorff distance between graphs $\subseteq\Cbar\times X$ to compare maps $u$).

We spend this appendix studying $\Mbar(X)$, though the only reason to care about $\Mbar(X)$ itself is as an intermediate tool for proving the desired result for $\Mbar^\beta_{g,n}(X)_I$.  Note that $\Mbar(X)$ is not necessarily compact, however this will be irrelevant since we are only interested in its local properties.

Fix a subspace $E'\subseteq E$.

Let us now define the ``$E'$-regular locus'' $\Mbar(X)^\reg\subseteq\Mbar(X)$ (which for simplicity we will just call the ``regular locus'').  Fix $(s_0,u_0,\{x_i^0\},e_0)\in\Mbar(X)$; we will describe when $(s_0,u_0,\{x_i^0\},e_0)\in\Mbar(X)^\reg$.  We consider the smooth Banach manifold:
\begin{equation}
\B:=\Bigl\{(u,e)\in W^{k,p}(C_{s_0},X)\times E\Bigm|u(p_{n+i})\in D^\circ\;(1\leq i\leq\ell)\Bigr\}
\end{equation}
Over $\B$, we consider the smooth Banach bundle $\E$ whose fiber over a map $u:C_{s_0}\to X$ is $W^{k-1,p}(\tilde C_{s_0},\Omega^{0,1}_{\tilde C_{s_0}}\otimes_\CC u^\ast TX)\oplus E/E'$.  Now suppose $k$ is large; then there are unique continuous functions:
\begin{equation}\label{intersectionsfunction}
x_i:\B\to C_{s_0}\quad(1\leq i\leq r)
\end{equation}
defined in a neighborhood of $(u_0,e_0)\in\B$, which coincide with $x_i^0$ at $(u_0,e_0)$, and for which $u(x_i(u))\in D_i^\circ$.  Moreover, \eqref{intersectionsfunction} are ``highly differentiable'', by which we mean that for all $\ell<\infty$, the function \eqref{intersectionsfunction} is of class $C^\ell$ provided $k\geq k_0(\ell)$.  In the present situation, we can of course be more precise: $x_i$ is $C^\ell$ as a function of $u\in C^\ell$ (by the implicit function theorem), so \eqref{intersectionsfunction} is $C^\ell$ as long as $W^{k,p}\hookrightarrow C^\ell$ (which holds iff $(k-\ell)p>2$).  It follows that:
\begin{equation}\label{firstdelbarsection}
(u,e)\mapsto\left[\delbar u+\lambda(e)(x_1,\ldots,x_r,\cdot,u(\cdot))\right]\oplus e
\end{equation}
is a highly differentiable section of $\E$ over $\B$ (the only nonsmoothness comes from how the $x_i$'s depend on $u$).  Assume $1<p<\infty$.  We say $(s_0,u_0,\{x_i^0\},e_0)\in\Mbar(X)^\reg$ iff \eqref{firstdelbarsection} is transverse to the zero section at $(u_0,e_0)$.  It is an easy exercise using elliptic regularity to see that this notion is independent of $(k,p)$ (as long as $k$ is sufficiently large so that the condition makes sense).

For a topological manifold $M$, let $\oo_M$ denote its orientation sheaf (whose fiber at a point $p\in M$ is canonically $H_{\dim M}(M,M\setminus p;\ZZ)$).  For a vector space $E$, let $\oo_E$ denote its orientation module (canonically $H_{\dim E}(E,E\setminus 0;\ZZ)$).  The main result of this appendix is:

\begin{theorem}\label{mainresult}
In the above setup (from the beginning of \S\ref{setupforgluing} until here), we have:
\begin{rlist}
\item\label{openness}$\Mbar(X)^\reg\subseteq\Mbar(X)$ is open.
\item\label{manifold}$\Mbar(X)^\reg$ is a topological manifold of dimension $\dim\Mbar_{g,n+\ell}+(1-g)\dim X+2\langle c_1(X),\beta\rangle+\dim E-2\ell$.
\item\label{submersion}The projection $\Mbar(X)^\reg\to E/E'$ is a topological submersion, i.e.\ locally modeled on a projection $\RR^n\times\RR^m\to\RR^n$.
\item\label{orientation}The orientation sheaf $\oo_{\Mbar(X)^\reg}$ is canonically isomorphic to $\oo_E\otimes\bigotimes_{i=1}^\ell\oo_{u(p_{n+i})^\ast N_{D/X}}^\vee$.
\end{rlist}
\end{theorem}

Let us now explain how Theorem \ref{mainresult} implies Proposition \ref{GWgluingneeded}.  Proposition \ref{GWgluingneeded} is a local statement, so let us prove it in a neighborhood of a specific point $(C,u,\{\phi_\alpha\}_{\alpha\in I},\{e_\alpha\}_{\alpha\in I})\in\Mbar^\beta_{g,n}(X)_I^\reg$.  We will construct data $(X,J,D,D_1,\ldots,D_r,\Mbar,E,\lambda,E')$ as in the above setup and a point $(s,u,\{x_i\},e)\in\Mbar(X)$.  It will be clear from the construction that there is a natural homeomorphism between a small neighborhood of this point in $\Mbar(X)$ and a small neighborhood of the given point in $\Mbar_{g,n}^\beta(X)_I$, and using this we will infer Proposition \ref{GWgluingneeded} from Theorem \ref{mainresult}.

We claim that every unstable irreducible component of $C$ has a point where $du$ is injective.  If $I=\varnothing$, this follows from the domain stabilization step Lemma \ref{everythingcanbethickened}, and if $I\ne\varnothing$, then picking any $\alpha\in I$, this follows from the fact that $u\pitchfork D_\alpha$ and adding the intersections as marked points makes $C$ stable.  It follows from the claim that we may pick $D\subseteq X$ with $u\pitchfork D$ such that adding these intersections as extra marked points makes $C$ stable, and none of these points are nodes or marked points.  We take $\ell$ to be the \emph{minimum} number of points in $u^{-1}(D)$ necessary to stabilize $C$, and we fix an ordered $\ell$-tuple of such points, adding them as new marked points $p_{n+1},\ldots,p_{n+\ell}\in C$, so now $C$ is a curve of type $(g,n+\ell)$.  Now we let $r=\sum_{\alpha\in I}r_\alpha$, we let $D_1,\ldots,D_r\subseteq X$ consist of $r_\alpha$ copies of $D_\alpha$ (union over $\alpha\in I$), and we let the $x_i$'s be the intersection points $u^{-1}(D_\alpha)$ (union over $\alpha\in I$).  For the purposes of the present argument, let us reindex $\{x_i\}_{1\leq i\leq r}$ as $\{x^\alpha_i\}_{\alpha\in I,1\leq i\leq r_\alpha}$.  Now since $C$ is a curve of type $(g,n+\ell)$, it corresponds to a point in $\Mbar_{g,n+\ell}$.  Choose an \'etale map $\Mbar\to\Mbar_{g,n+\ell}$ covering this point, and choose a point $s\in\Mbar$ and an isomorphism $\iota:C\to C_s$.  Now for every $\alpha\in I$, there is a unique map:
\begin{equation}
\tilde\phi_\alpha:\overbrace{\Cbar^\circ\underset\Mbar\times\cdots\underset\Mbar\times\Cbar^\circ}^{r_\alpha\text{ times}}{}\underset\Mbar\times\Cbar^\circ\to\Cbar_\alpha
\end{equation}
defined in a neighborhood of $\{x_1^\alpha\}\times\cdots\times\{x_{r_\alpha}^\alpha\}\times C_s$ so that $\tilde\phi_\alpha(x^\alpha_1,\ldots,x^\alpha_{r_\alpha},\iota(\cdot))=\phi_\alpha(\cdot)$ and $\tilde\phi_\alpha(y_1,\ldots,y_{r_\alpha},\cdot)$ classifies the curve in the last factor after forgetting the last $\ell$ marked points and adding $y_1,\ldots,y_{r_\alpha}$ as marked points (this follows from the fact that $\Cbar_\alpha\to\Mbar_{g,n+r_\alpha}/S_{r_\alpha}$ is \'etale).  We let $E:=E_I=\bigoplus_{\alpha\in I}E_\alpha$, and we define:
\begin{equation*}
\lambda\biggl(\bigoplus_{\alpha\in I}e_\alpha\biggr)(\{y_i^\alpha\}_{\alpha\in I,1\leq i\leq r_\alpha},\cdot,\cdot)=\sum_{\alpha\in I}\lambda_\alpha(e_\alpha)(\tilde\phi_\alpha(y_1^\alpha,\ldots,y_{r_\alpha}^\alpha,\cdot),\cdot)
\end{equation*}
in a neighborhood of $\{x_i^\alpha\}_{\alpha\in I,1\leq i\leq r_\alpha}\times C_s\times X$ (and then we simply cut it off to be zero elsewhere).  Finally, we observe that with this definition, a neighborhood of $(s,u\circ\iota^{-1},\{x^\alpha_i\}_{\alpha\in I,1\leq i\leq r_\alpha},\bigoplus_{\alpha\in I}e_\alpha)\in\Mbar(X)$ coincides with a neighborhood of $(C,u,\{\phi_\alpha\}_{\alpha\in I},\{e_\alpha\}_{\alpha\in I})\in\Mbar^\beta_{g,n}(X)_I$ (note that the point $(C,u,\{\phi_\alpha\}_{\alpha\in I},\{e_\alpha\}_{\alpha\in I})\in\Mbar^\beta_{g,n}(X)_I^\reg$ has trivial automorphism group by definition of $\Mbar^\beta_{g,n}(X)_I^\reg$).

For $E'=E$, under this identification of a small open set in $\Mbar^\beta_{g,n}(X)_I$ and a small open set in $\Mbar(X)$, we have $\Mbar(X)^\reg\subseteq\Mbar^\beta_{g,n}(X)_I^\reg$ (for this it is important that we chose $\ell$ as the \emph{minimum} number of points necessary to stabilize $C$).  It follows that Theorem \ref{mainresult}(\ref{openness},\ref{orientation}) implies Proposition \ref{GWgluingneeded}(\ref{GWgluingopen},\ref{GWgluingor}) (note that over this small open set, we have an identification $\oo_{u(p_{n+i})^\ast N_{D/X}}=\oo_{T_{p_{n+i}}C_s}$ and that the latter is canonically trivial using the complex structure on $C_s$).  Taking the above construction with $J$ in place of $I$ and setting $E':=E_I\subseteq E_J=E$, we get Proposition \ref{GWgluingneeded}(\ref{GWgluingsubmersion}) from Theorem \ref{mainresult}(\ref{manifold},\ref{submersion}) as well.

\subsection{Local model for resolution of a node}

The rest of this appendix is devoted to the proof of Theorem \ref{mainresult}.  We now fix $(s_0,u_0,\{x_i^0\},e_0)\in\Mbar(X)^\reg$, and we prove the desired statements (\ref{openness})--(\ref{orientation}) in a neighborhood of this point.

Our first task is to fix a nice local coordinate system on $\Mbar$ near $s_0$.  Let $d$ be the number of nodes of $C_0:=C_{s_0}$.

On each side of each node of $C_0$, fix a ``cylindrical end'', that is, a map:
\begin{equation}\label{cylindricalends}
[0,\infty)\times S^1\to C_0
\end{equation}
which is a biholomorphism onto some small neighborhood $D^2\setminus 0$ of the node.  We use coordinates $(s,t)\in[0,\infty)\times S^1$, which is given the standard holomorphic structure $z=e^{s+it}$.

Let $\Mbar^d\subseteq\Mbar$ denote the locus of curves with exactly $d$ nodes.  Pick a smooth family of smooth almost complex structures $j_y$ on $C_0$ parameterized by $y\in\RR^{\dim\Mbar^d}$, where $j_0$ is the given almost complex structure on $C_0$, which is constant over the cylindrical ends \eqref{cylindricalends}, and such that the induced map $\RR^{\dim\Mbar^d}\to\Mbar^d$ is a diffeomorphism onto its image.

Now consider the following procedure, which takes a ``gluing parameter'' $\alpha=e^{-6S+i\theta}\in\CC$ and two copies of the standard end $[0,\infty)\times S^1\sqcup[0,\infty)\times S^1$.  We first truncate both ends, leaving just the subset $[0,6S]\times S^1\sqcup[0,6S]\times S^1$.  We then identify $(s,t)$ in the first piece with $(6S-s,\theta-t)=(s',t')$ in the second piece.  We call the resulting cylinder a ``neck''.

Now given gluing parameters\footnote{In the present construction, and in many other constructions/definitions to come later, certain expressions, equalities, etc. only make sense or only hold if the gluing parameters $\alpha_i\in\CC$ are sufficiently close to zero (i.e.\ $\left|\alpha_i\right|\leq\epsilon$ for some $\epsilon>0$ depending only on data which we have previously fixed).  We will often leave this assumption implicit, since we only care about what happens for $\alpha$ in a neighborhood of $0\in\CC^d$ anyway.  The same goes for $y\in\RR^{\dim\Mbar^d}$.} $\alpha=(\alpha_1,\ldots,\alpha_d)\in\CC^d$, we may perform the gluing operation above on $C_0$, using the chosen cylindrical ends \eqref{cylindricalends}.  We call the resulting curve $C_\alpha$, and it is equipped with cylindrical ends (corresponding to those $\alpha_i=0$) and necks (corresponding to those $\alpha_i\ne 0$):
\begin{align}
\label{Gend}[0,\infty)\times S^1\to C_\alpha\\
\label{Gneck}[0,6S_i]\times S^1\to C_\alpha
\end{align}
In each neck, we have coordinates $(s,t)\in[0,6S_i]\times S^1$ and coordinates $(s',t')\in[0,6S_i]\times S^1$, which satisfy $s+s'=6S_i$ and $t+t'=\theta_i$.  The curve $C_\alpha$ is also equipped with $n+\ell$ marked points $p_1,\ldots,p_{n+\ell}\in C_\alpha$ coming from the given $p_1,\ldots,p_{n+\ell}\in C_0$.  Since $j_y$ is constant over the cylindrical ends, it descends to give an almost complex structure on $C_\alpha$.

Now since $\Mbar\to\Mbar_{g,n+\ell}$ is \'etale, there is an induced map:
\begin{align}
\CC^d\times\RR^{\dim\Mbar^d}&\to\Mbar\\
(\alpha,y)&\mapsto(C_\alpha,j_y,p_1,\ldots,p_{n+\ell})
\end{align}
which is a local diffeomorphism near zero.  Using these local coordinates, we may alternatively define $\Mbar(X)$ as:
\begin{equation}\label{modulispace}
\Mbar(X):=\left\{\begin{matrix}\alpha\in\CC^d\hfill\cr y\in\RR^{\dim\Mbar^d}\hfill\cr u:C_\alpha\to X\hfill\cr x_i\in C_\alpha\quad 1\leq i\leq r\hfill\cr e\in E\hfill\end{matrix}\;\middle|\;\begin{matrix}u\text{ smooth and }u_\ast[C_\alpha]=\beta\hfill\cr u(p_{n+i})\in D^\circ\quad\hfill(1\leq i\leq\ell)\cr u(x_i)\in D_i^\circ\text{ and }u\pitchfork D_i\text{ at }x_i\quad\hfill(1\leq i\leq r)\cr\delbar_yu+\lambda(e)(\alpha,y,x_1,\ldots,x_r,\cdot,u(\cdot))=0\hfill\end{matrix}\right\}
\end{equation}
and this coincides with the definition \eqref{Mbarfirstdef} in a neighborhood of $(s_0,u_0,\{x_i^0\},e_0)=(0,0,u_0,\{x_i^0\},e_0)$.  For the purposes of \eqref{modulispace}, $\lambda$ denotes the function:
\begin{equation*}
\lambda:E\to C^\infty(\CC^d\times\RR^{\dim\Mbar^d}\times\overbrace{C_0\times\cdots\times C_0}^{r\text{ times}}{}\times C_0\times X,\Hom_\RR(T\tilde C_0,TX))
\end{equation*}
for which $\lambda(\cdot)(\cdot,y,\ldots)$ lands in $\Omega^{0,1}_{\tilde C_0,j_y}\otimes_\CC TX\subseteq\Hom_\RR(T\tilde C_0,TX)$, defined in terms of the old $\lambda$ via:
\begin{multline*}
\lambda^\mathrm{new}(e)(\alpha,y,x_1,\ldots,x_r,p,x):=\\
\lambda^\mathrm{old}(e)(x_1\in(C_\alpha,j_y)\subseteq\Cbar,\ldots,x_r\in (C_\alpha,j_y)\subseteq\Cbar,p\in(C_\alpha,j_y)\subseteq\Cbar,x)
\end{multline*}
We assume that the cylindrical ends were chosen disjoint from the support of $\lambda^\mathrm{old}$, so we can make sense of $\lambda^\mathrm{new}$ as giving sections on $C_\alpha$.

Now to prove the main result, it suffices to study the local structure of $\Mbar(X)$ (defined as in \eqref{modulispace}) near the basepoint:
\begin{equation}
(0,0,u_0:C_0\to X,\{x_i^0\},e_0)
\end{equation}
which we are assuming lies in $\Mbar(X)^\reg$.

\subsection{Pregluing}

Let $\exp:TX\to X$ denote the exponential map of some Riemannian metric on $X$ for which $D$ is totally geodesic.  Let $\nabla$ denote any $J$-linear\footnote{Given any connection $\nabla^0$, the connection $\nabla_XY:=\frac 12(\nabla^0_XY-J(\nabla^0_X(JY)))$ is $J$-linear.} connection on $TX$ (equivalently, a connection for which $\nabla J=0$).  Let $\PT_{x\to y}:T_xX\to T_yX$ denote parallel transport via $\nabla$ along the shortest geodesic between $x$ and $y$ (we will only use this notation when it may be assumed that $x$ and $y$ are very close in $X$); note that $\PT_{x\to y}$ is $J$-linear.

Fix a smooth function $\chi:\RR\to[0,1]$ satisfying:
\begin{equation}
\chi(x)=\begin{cases}0&x\leq 0\cr 1&x\geq 1\end{cases}
\end{equation}

\begin{definition}[Flattening]
For $\alpha\in\CC^d$, we define the ``flattened'' map $u_{0|\alpha}:C_0\to X$ as follows.  Away from the ends, $u_{0|\alpha}$ coincides with $u_0$.  Over an end $[0,\infty)\times S^1$, we define it as follows:
\begin{equation}
u_{0|\alpha}(s,t):=\begin{cases}
u_0(s,t)&s\leq S-1\cr
\exp_{u_0(n)}\left[\chi(S-s)\cdot\exp_{u_0(n)}^{-1}u_0(s,t)\right]&S-1\leq s\leq S\cr
u_0(n)&S\leq s
\end{cases}
\end{equation}
where $n\in C_0$ denotes the corresponding node.
\end{definition}

\begin{definition}[Pregluing]
For $\alpha\in\CC^d$, we define the ``preglued'' map $u_\alpha:C_\alpha\to X$ as follows.  Away from the necks, $u_\alpha$ coincides with $u_0$.  Over a neck $[0,6S]\times S^1$, we define it as follows:
\begin{equation}
u_\alpha(s,t):=\begin{cases}
u_0(s,t)&s\leq S-1\cr
\exp_{u_0(n)}\left[\chi(S-s)\cdot\exp_{u_0(n)}^{-1}u_0(s,t)\right]&S-1\leq s\leq S\cr
u_0(n)&S\leq s\leq 5S\cr
\exp_{u_0(n)}\left[\chi(S-s')\cdot\exp_{u_0(n)}^{-1}u_0(s',t')\right]&5S\leq s\leq 5S+1\cr
u_0(s',t')&5S+1\leq s
\end{cases}
\end{equation}
($u_\alpha$ should be thought of as the ``descent'' of $u_{0|\alpha}$ from $C_0$ to $C_\alpha$).
\end{definition}

\begin{definition}[Pregluing vector fields]
For $\xi\in C^\infty(C_0,u_0^\ast TX)$, we define:
\begin{equation}
\xi_\alpha\in C^\infty(C_\alpha,u_\alpha^\ast TX)
\end{equation}
as follows.  Away from the necks, $\xi_\alpha$ coincides with $\xi$.  Over a neck $[0,6S]\times S^1$, we define it as follows:
\begin{equation}\label{kernelpregluingformula}
\xi_\alpha(s,t):=\begin{cases}
\xi(s,t)&s\leq S-1\cr
\PT_{u_0(s,t)\to u_\alpha(s,t)}\left[\xi(s,t)\right]&S-1\leq s\leq S\cr
\PT_{u_0(s,t)\to u_\alpha(s,t)}\left[\xi(s,t)\right]\cdot(1-\chi(s-S))+\chi(s-S)\cdot \xi(n)&S\leq s\leq S+1\cr
\xi(n)&S+1\leq s\leq 5S-1\cr
\PT_{u_0(s',t')\to u_\alpha(s',t')}\left[\xi(s',t')\right]\cdot(1-\chi(s'-S))+\chi(s'-S)\cdot \xi(n)&5S-1\leq s\leq 5S\cr
\PT_{u_0(s',t')\to u_\alpha(s',t')}\left[\xi(s',t')\right]&5S\leq s\leq 5S+1\cr
\xi(s',t')&5S+1\leq s
\end{cases}
\end{equation}
\end{definition}

\subsection{Weighted Sobolev norms}

Fix $k\in\ZZ_{\geq 6}$ and fix $\delta\in(0,1)$.  Fix a metric on $TX$ for the purposes of defining Sobolev norms of sections; this could be the same as the metric used to define $\exp$, though there is no need for it to be.

We denote by $W^{k,2}(C_\alpha,u_\alpha^\ast TX)$ the space of sections of $u_\alpha^\ast TX$ over $C_\alpha$ whose pullback to $\tilde C_\alpha$ is $W^{k,2}$ (since $W^{k,2}\hookrightarrow C^0$ as $k\geq 6$, this is the same as $W^{k,2}$ sections over $\tilde C_\alpha$ which coincide on each pair of points identified in $\tilde C_\alpha\to C_\alpha$, where $\tilde C_\alpha$ denotes the normalization of $C_\alpha$).  Similarly we denote by $W^{k-1,2}(\tilde C_\alpha,\Omega^{0,1}_{\tilde C_\alpha,j_y}\otimes_\CC u_\alpha^\ast TX)$ the space of sections of class $W^{k-1,2}$.  Note that both $W^{k,2}$ and $W^{k-1,2}$ refer to spaces of functions on the compact Riemann surface $\tilde C_\alpha$.  The specific choice of metric on $C_\alpha$ used to define the norms $\left\|\cdot\right\|_{k,2}$ and $\left\|\cdot\right\|_{k-1,2}$ will be of little importance (they will matter only up to commensurability).

We now define certain weighted Sobolev spaces $W^{k,2,\delta}$ and $W^{k-1,2,\delta}$ on the possibly non-compact Riemann surface $C_\alpha$ minus the nodes.  The specific choice of norms $\left\|\cdot\right\|_{k,2,\delta}$ and $\left\|\cdot\right\|_{k-1,2,\delta}$ (not just their commensurability classes) will be of great importance.

To make calculations in ends/necks easier, for each node $n\in C_0$ we fix the local trivialization of $TX$ given by parallel transport:
\begin{equation}\label{localPTcoords}
\PT_{u_0(n)\to p}:T_{u_0(n)}X\to T_pX
\end{equation}
over $p$ contained in a small ball centered at $u_0(n)$.  We assume that the ends \eqref{cylindricalends} were chosen small enough so that $u_0$ maps each end into a compact subset of the small ball associated to the corresponding node.

\begin{definition}
For $\alpha\in\CC^d$, we define a weighted Sobolev space:
\begin{equation}
W^{k,2,\delta}(C_0,u_{0|\alpha}^\ast TX)
\end{equation}
using the following norm.  We use the usual $(k,2)$ norm away from the ends.  Over an end $[0,\infty)\times S^1\subseteq C_0$, we consider the local trivialization \eqref{localPTcoords}, in which a section of $u_{0|\alpha}^\ast TX$ simply becomes a function $\xi:[0,\infty)\times S^1\to T_{u_0(n)}X$.  In terms of this function, the contribution from the end to the norm squared is:
\begin{equation}\label{endnorm}
\left|\xi(n)\right|^2+\int_{[0,\infty)\times S^1}\biggl[\left|\xi(s,t)-\xi(n)\right|^2+\sum_{j=1}^k\left|D^j\xi(s,t)\right|^2\biggr]e^{2\delta s}\;ds\,dt
\end{equation}
The derivatives of $\xi$ are measured with respect to the standard metric on $[0,\infty)\times S^1$.

By Sobolev embedding $W^{2,2}\hookrightarrow C^0$ in two dimensions, in any end we have a uniform bound on $\left|\xi(s,t)-\xi(n)\right|e^{\delta s}$ and $\left|D^j\xi(s,t)\right|e^{\delta s}$ ($1\leq j\leq k-2$) linear in $\left\|\xi\right\|_{k,2,\delta}$.
\end{definition}

\begin{definition}
For $\alpha\in\CC^d$, we define a weighted Sobolev space:
\begin{equation}
W^{k,2,\delta}(C_\alpha,u_\alpha^\ast TX)
\end{equation}
using the following norm.  We use the usual $(k,2)$ norm away from the ends/necks.  Over an end, the contribution to the norm squared is \eqref{endnorm}.  Over a neck $[0,6S]\times S^1\subseteq C_\alpha$, we again think of a section as a function $\xi:[0,6S]\times S^1\to T_{u_0(n)}X$, and the contribution to the norm squared is:
\begin{equation*}
\biggl|\int_{S^1}\xi(3S,t)\,dt\biggr|^2+\int_{[0,6S]\times S^1}\biggl[\Bigl|\xi(s,t)-\int_{S^1}\xi(3S,t)\,dt\Bigr|^2+\sum_{j=1}^k\left|D^j\xi(s,t)\right|^2\biggr]e^{2\delta\min(s,6S-s)}\;ds\,dt
\end{equation*}
Of course, for fixed $\alpha\in\CC^d$, this norm is equivalent to not weighting the necks, though the two norms are not uniformly equivalent as $\alpha\to 0$.  The weight in the necks is important since the key point is to establish certain estimates which are uniform as $\alpha\to 0$.

By Sobolev embedding $W^{2,2}\hookrightarrow C^0$ in two dimensions, in any neck we have a uniform bound on $\left|\xi(s,t)-\int_{S^1}\xi(3S,t)\,dt\right|e^{\delta\min(s,6S-s)}$ and $\left|D^j\xi(s,t)\right|e^{\delta\min(s,6S-s)}$ ($1\leq j\leq k-2$) linear in $\left\|\xi\right\|_{k,2,\delta}$.
\end{definition}

\begin{definition}
For $\alpha\in\CC^d$, we define a weighted Sobolev space:
\begin{equation}
W^{k-1,2,\delta}(\tilde C_\alpha,\Omega^{0,1}_{\tilde C_\alpha,j_y}\otimes_\CC u_\alpha^\ast TX)
\end{equation}
as follows.  We use the usual $(k-1,2)$ norm away from the ends/necks.  Over an end or neck, we trivialize $(T\tilde C_\alpha,j_y)$ over $\CC$ by the basis vector $\frac\partial{\partial s}$, and we trivialize $TX$ via \eqref{localPTcoords}, and hence the section is simply a function $\eta$ from the end/neck to $T_{u_0(n)}X$.  In terms of this function, the contribution to the norm squared is (for end/neck respectively):
\begin{align}
\int_{[0,\infty)\times S^1}&\sum_{j=0}^{k-1}\left|D^j\eta(s,t)\right|^2e^{2\delta s}\;ds\,dt\\
\int_{[0,6S]\times S^1}&\sum_{j=0}^{k-1}\left|D^j\eta(s,t)\right|^2e^{2\delta\min(s,6S-s)}\;ds\,dt
\end{align}
By Sobolev embedding $W^{2,2}\hookrightarrow C^0$ in two dimensions, in any end (resp.\ neck) we have a uniform bound on $\left|D^j\eta(s,t)\right|e^{\delta s}$ (resp.\ $\left|D^j\eta(s,t)\right|e^{\delta\min(s,6S-s)}$) ($1\leq j\leq k-3$) linear in $\left\|\eta\right\|_{k,2,\delta}$.
\end{definition}

\subsection{Based \texorpdfstring{$\delbar$-section}{del bar section} \texorpdfstring{$\F_{\alpha,y}$}{F\_alpha,y} and linearized operator \texorpdfstring{$D_{\alpha,y}$}{D\_alpha,y}}

Fix a norm on $E$.  On direct sums of normed spaces we use the direct sum norm $\left\|a\oplus b\right\|:=\left\|a\right\|+\left\|b\right\|$.

We consider the following partially defined function:
\begin{gather*}
\mathcal F_{\alpha,y}:C^\infty(C_\alpha,u_\alpha^\ast TX)_D\oplus E\to C^\infty(\tilde C_\alpha,\Omega^{0,1}_{\tilde C_\alpha,j_y}\otimes_\CC u_\alpha^\ast TX)\\
\xi\mapsto\PT_{\exp_{u_\alpha}\xi\to u_\alpha}\left[\delbar_y\exp_{u_\alpha}\xi+\lambda(e_0+\proj_E\xi)(\alpha,y,x_1,\ldots,x_r,\cdot,(\exp_{u_\alpha}\xi)(\cdot))\right]
\end{gather*}
This function $\F_{\alpha,y}$ is defined for $\xi$ in a $C^1$-neighborhood of zero; for these $\xi$ we define $x_i=x_i(\xi)$ as in \eqref{intersectionsfunction}.  The subscript $D$ indicates restriction to sections which are tangent to $D$ at $p_{n+1},\ldots,p_{n+\ell}$.  It follows that for $\xi$ contained in a $C^0$-neighborhood of zero, $\exp_{u_\alpha}\xi$ sends $p_{n+1},\ldots,p_{n+\ell}$ to $D^\circ$.

Now we observe that $\mathcal F_{\alpha,y}$ extends continuously to a map:
\begin{equation}
\mathcal F_{\alpha,y}:W^{k,2,\delta}(C_\alpha,u_\alpha^\ast TX)_D\oplus E\to W^{k-1,2,\delta}(\tilde C_\alpha,\Omega^{0,1}_{\tilde C_\alpha,j_y}\otimes_\CC u_\alpha^\ast TX)
\end{equation}
which is defined for $\left\|\xi\right\|_{k,2,\delta}\leq c'$ (some $c'>0$).  Moreover, this map is highly differentiable (see \eqref{intersectionsfunction}--\eqref{firstdelbarsection} and the surrounding discussion; recall we have fixed $k\geq 6$).  We denote by:
\begin{equation}
D_{\alpha,y}:W^{k,2,\delta}(C_\alpha,u_\alpha^\ast TX)_D\oplus E\to W^{k-1,2,\delta}(\tilde C_\alpha,\Omega^{0,1}_{\tilde C_\alpha,j_y}\otimes_\CC u_\alpha^\ast TX)
\end{equation}
the derivative of $\F_{\alpha,y}$ at zero.

Let $T_\nabla(X,Y):=\nabla_XY-\nabla_YX-[X,Y]$ denote the torsion of $\nabla$.  Let $(\cdot)^{0,1}_y$ denote the projection $\Hom_\RR(T\tilde C_\alpha,u_\alpha^\ast TX)\to\Omega^{0,1}_{\tilde C_\alpha,j_y}\otimes_\CC u_\alpha^\ast TX$, so $(V)^{0,1}_y:=\frac 12(V+J\circ V\circ j_y)$.

\begin{lemma}\label{linearizedformula}
The linearized operator $D_{\alpha,y}$ is given by:
\begin{align}\label{formulaforlinearizedoperator}
D_{\alpha,y}\xi&=\Bigl(\nabla\xi+T_\nabla(\xi,du_\alpha)\Bigr)^{0,1}_y\cr
&\quad+\sum_{i=1}^r\frac{d[\lambda(e_0)]}{dx_i}(\alpha,y,x_1,\ldots,x_r,\cdot,u_\alpha(\cdot))(-\proj_{TC_\alpha}\xi(x_i))\cr
&\quad+\nabla_\xi[\lambda(e_0)](\alpha,y,x_1,\ldots,x_r,\cdot,u_\alpha(\cdot))\cr
&\quad+\lambda(\proj_E\xi)(\alpha,y,x_1,\ldots,x_r,\cdot,u_\alpha(\cdot))
\end{align}
where $\proj:u_\alpha^\ast TX\to TC_\alpha$ denotes the projection associated to the splitting $T_{u_\alpha(x_i)}X=T_{u_\alpha(x_i)}D_i\oplus T_{x_i}C_\alpha$, and $\nabla_\xi[\lambda(e_0)]$ means covariant derivative in the direction of $\xi$ along the $X$ factor.
\end{lemma}

\begin{proof}
The first term $\bigl(\nabla\xi+T_\nabla(\xi,du_\alpha)\bigr)^{0,1}_y$ comes from differentiating:
\begin{equation}
\PT_{\exp_{u_\alpha}\xi\to u_\alpha}\left[\delbar_y\exp_{u_\alpha}\xi\right]
\end{equation}
The second two terms come from differentiating:
\begin{equation}
\PT_{\exp_{u_\alpha}\xi\to u_\alpha}\left[\lambda(e_0)(\alpha,y,x_1,\ldots,x_r,\cdot,(\exp_{u_\alpha}\xi)(\cdot))\right]
\end{equation}
where we use the fact that $\frac d{dt}x_i(\exp_{u_\alpha}(t\xi))=-\proj_{TC_\alpha}\xi(x_i)$.  The last term comes from differentiating:
\begin{equation}
\PT_{\exp_{u_\alpha}\xi\to u_\alpha}\left[\lambda(\proj_E\xi)(\alpha,y,x_1,\ldots,x_r,\cdot,(\exp_{u_\alpha}\xi)(\cdot))\right]
\end{equation}
We leave the calculations to the reader.
\end{proof}

\begin{lemma}\label{fredholmandsamekernelcokernel}
Consider the following commutative square:
\begin{equation}\label{fredholmsquare}
\begin{tikzcd}
W^{k,2,\delta}(C_\alpha,u_\alpha^\ast TX)_D\oplus E\ar{r}{D_{\alpha,y}}&W^{k-1,2,\delta}(\tilde C_\alpha,\Omega^{0,1}_{\tilde C_\alpha,j_y}\otimes_\CC u_\alpha^\ast TX)\\
W^{k,2}(C_\alpha,u_\alpha^\ast TX)_D\oplus E\ar[hook]{u}\ar{r}{D_{\alpha,y}}&W^{k-1,2}(\tilde C_\alpha,\Omega^{0,1}_{\tilde C_\alpha,j_y}\otimes_\CC u_\alpha^\ast TX)\ar[hook]{u}
\end{tikzcd}
\end{equation}
Both horizontal operators are Fredholm, and the induced map from the kernel (resp.\ cokernel) of the bottom horizontal map to the kernel (resp.\ cokernel) of the top horizontal map is an isomorphism.
\end{lemma}

Let us remark on the reason for the vertical inclusions in \eqref{fredholmsquare}.  For the leftmost vertical inclusion, one just needs to check that for a function $f:D^2\to\RR$, the $W^{k,2,\delta}$ norm of its pullback to $[0,\infty)\times S^1$ is bounded linearly by its $W^{k,2}$ norm on $D^2$.  For the rightmost vertical inclusion, one checks the same property for $1$-forms.  In both of these calculations we use the fact that $\delta<1$.

\begin{proof}
The operator $D_{\alpha,y}$ equals $(\nabla\xi)^{0,1}_y$ plus compact terms (by Lemma \ref{linearizedformula}), so it is equivalent to show that $\xi\mapsto(\nabla\xi)^{0,1}_y$ is Fredholm (and we may forget about $E$).  Here $\nabla$ is really the pullback of $\nabla$ via $u_\alpha$.  Thus it suffices to show that if $V$ is a smooth complex vector bundle over a nodal Riemann surface $C$ equipped with a smooth connection $\nabla$, then the horizontal maps:
\begin{equation}
\begin{tikzcd}[column sep = large]
W^{k,2,\delta}(C,V)\ar{r}{\xi\mapsto(\nabla\xi)^{0,1}}&W^{k-1,2,\delta}(\tilde C,\Omega^{0,1}_{\tilde C}\otimes_\CC V)\\
W^{k,2}(C,V)\ar[hook]{u}\ar{r}{\xi\mapsto(\nabla\xi)^{0,1}}&W^{k-1,2}(\tilde C,\Omega^{0,1}_{\tilde C}\otimes_\CC V)\ar[hook]{u}
\end{tikzcd}
\end{equation}
are Fredholm.  For the bottom map, the Fredholm property is standard.  Indeed, $W^{k,2}(C,V)$ lies inside $W^{k,2}(\tilde C,V)$ as a closed subspace of finite codimension, and $\nabla^{0,1}:W^{k,2}(\tilde C,V)\to W^{k-1,2}(\tilde C,\Omega^{0,1}_{\tilde C}\otimes_\CC V)$ is Fredholm by elliptic regularity on the compact Riemann surface $\tilde C$.  For the top map, the Fredholm property follows from results of Lockhart--McOwen \cite{lockhartmcowen} concerning the Fredholmness (in weighted Sobolev spaces) of elliptic operators on manifolds with cylindrical ends.  Specifically, $W^{k,2,\delta}(C,V)$ lies inside $W^{k,2,\delta}(\tilde C,V)$ as a closed subspace of finite codimension, and the map $\nabla^{0,1}:W^{k,2,\delta}(\tilde C,V)\to W^{k-1,2,\delta}(\tilde C,\Omega^{0,1}_{\tilde C}\otimes_\CC V)$ is Fredholm by the theory of Lockhart--McOwen \cite{lockhartmcowen}.  This uses the fact that the weight $\delta\in(0,1)$ is not an eigenvalue of the asymptotic linearized operator in any end (the set of eigenvalues in any end is precisely $\ZZ$, corresponding to the powers $\{z^n\}_{n\in\ZZ}$ on $D^2\setminus 0$).  Let us also remark that both of these Fredholmness statements are easier at $p=2$ than for general $p\in(1,\infty)$, since for $p=2$ one can use Fourier analysis (after localizing the problem with a partition of unity).

Now it remains to show that the induced maps on kernel and cokernel are isomorphisms.  The map on kernels is obviously injective, and the map on cokernels is easily seen to be surjective (use the fact that the image of the top horizontal map is closed and the fact that the rightmost vertical inclusion is dense).

For the surjectivity of the map on kernels and the injectivity of the map on cokernels, it suffices to show that if $D_{\alpha,y}\xi=\eta$ for $\xi\in W^{k,2,\delta}$ and $\eta\in W^{k-1,2}$, then $\xi\in W^{k,2}$.  To show $\xi\in W^{k,2}$, it suffices to argue locally on one side of a given node.

Certainly (the vector field part of) $\xi$ extends continuously to $C_\alpha$ (by definition of $W^{k,2,\delta}$), and this continuous extension satisfies $D_{\alpha,y}\xi=\eta+\epsilon$ where $\epsilon$ is a distribution on $\tilde C_\alpha$ valued in $\Omega^{0,1}_{\tilde C_\alpha,j_y}\otimes_\CC u_\alpha^\ast TX$ and supported inside the (inverse images of the) nodes.  By elliptic regularity, it suffices to show that $\epsilon=0$.

\newlength\arrowwidthawesome
\settowidth{\arrowwidthawesome}{$\nrightarrow$}
Now for smooth test functions $\varphi$ supported inside a small neighborhood of the support of $\epsilon$, we have:
\begin{equation}
\langle\epsilon,\varphi\rangle=\langle D_{\alpha,y}\xi-\eta,\varphi\rangle=\langle\xi,D_{\alpha,y}^\ast\varphi\rangle-\langle\eta,\varphi\rangle
\end{equation}
We thus obtain the bound $\left|\langle\epsilon,\varphi\rangle\right|\leq c\left\|\varphi\right\|_{1,1}$ since $D_{\alpha,y}^\ast$ is a first order operator and $\xi,\eta$ are bounded (as they are continuous).  On the other hand, $\epsilon$ is supported at a finite set of points, so it is a linear combination of $\delta$-functions and their derivatives.  Hence $\epsilon$ does not satisfy the bound $\left|\langle\epsilon,\varphi\rangle\right|\leq c\left\|\varphi\right\|_{1,1}$ (recall that $W^{1,1}\hookrightarrow\hspace{-\arrowwidthawesome}\nrightarrow C^0$ since we are in two dimensions) unless $\epsilon=0$, as desired.
\end{proof}

We denote the kernel of $D_{0,0}$ (whose meaning is unambiguous by Lemma \ref{fredholmandsamekernelcokernel}) by:
\begin{equation}
K:=\ker D_{0,0}\subseteq C^\infty(C_0,u_0^\ast TX)_D\oplus E
\end{equation}
Note that our assumption $(0,0,u_0,\{x_i^0\},e_0)\in\Mbar(X)^\reg$ is equivalent to saying that $D_{0,0}$ is surjective and $K\twoheadrightarrow E/E'$ is surjective.

\subsection{Pregluing estimates}

Fix norms on $K$, $\CC^d$, and $\RR^{\dim\Mbar^d}$.

\begin{lemma}[Estimate for map pregluing]\label{mappregluingissmall}
We have:
\begin{equation}
\left\|\delbar_yu_\alpha+\lambda(e_0)(\alpha,y,x_1^0,\ldots,x_r^0,\cdot,u_\alpha(\cdot))\right\|_{k-1,2,\delta}\leq c\cdot\left[\left|y\right|+\sum_{i=1}^de^{-(1-\delta)S_i}\right]
\end{equation}
uniformly over $(\alpha,y)$ in a neighborhood of zero, for $c<\infty$ depending on data which has been previously fixed.
\end{lemma}

Data which has been previously fixed includes $(s_0,u_0,\{x_i^0\},e_0)$, the cylindrical ends, $j_y$, the metrics used to define $\exp$ and Sobolev norms, $k$, $\delta$, the cutoff function $\chi$, the norms on $E$, $K$, $\CC^d$, $\RR^{\dim\Mbar^d}$, etc.

\begin{proof}
Away from the ends/necks of $C_\alpha$, the expression is zero when $y=0$ (since $(0,0,u_0,\{x_i^0\},e_0)\in\Mbar(X)$) and more generally it follows easily that its $j$th derivative is bounded pointwise by $c_j\left|y\right|$.  In the ends of $C_\alpha$, the expression is zero.

In a given neck $[0,6S]\times S^1\subseteq C_\alpha$, argue as follows.  Recall that $\lambda$ vanishes in the ends/necks, so we just have to bound $\delbar_yu_\alpha$.  Now $\delbar_yu_\alpha$ is only nonzero over $[S-1,S]\cup[5S,5S+1]$; by symmetry we may just work over $[S-1,S]$.  Now since $u_0$ is smooth, it and all its derivatives decay exponentially in any cylindrical end $[0,\infty)\times S^1$ (since every end is holomorphically conjugate to the map $[0,\infty)\times S^1\to D^2$ given by $z=e^{-s-it}$).  Precisely, we have the following pointwise bound over $[S-1,S]$:
\begin{equation}\label{uzeroexponentialdecay}
\left|D^ju_0(s,t)\right|\leq c_je^{-S}
\end{equation}
(derivatives with respect to $s$ and $t$).  It follows from the definition of $u_\alpha$ that the same bound holds for $u_\alpha$.  Thus the contribution of $[S-1,S]$ to the norm of $\delbar_yu_\alpha$ is bounded by a constant times $e^{S\delta}e^{-S}$.
\end{proof}

\begin{lemma}[Estimate for kernel pregluing]\label{pregluingKestimate}
For all $\kappa\in K$, we have:
\begin{equation}
\left\|D_{\alpha,y}\kappa_\alpha\right\|_{k-1,2,\delta}\leq c\cdot\left[\left|y\right|+\sum_{i=1}^de^{-(1-\delta)S_i}\right]\left\|\kappa\right\|
\end{equation}
uniformly over $(\alpha,y)$ in a neighborhood of zero, for $c<\infty$ depending on data which has been previously fixed.
\end{lemma}

\begin{proof}
Away from the ends/necks, the norm is bounded by $c\cdot\left|y\right|\left\|\kappa\right\|$.  In the ends, the expression is zero.  In the necks, argue as follows.

The expression is only nonzero over $[S-1,S+1]\cup[5S-1,5S+1]$; by symmetry we may just work over $[S-1,S+1]$.  Then since $u_0$ and $\kappa$ are smooth, we have (by the same reasoning used to justify \eqref{uzeroexponentialdecay}) the following pointwise bound over $[S-1,S+1]$:
\begin{equation}
\left|D^j\kappa_\alpha(s,t)\right|\leq c_je^{-S}\left\|\kappa\right\|
\end{equation}
(derivatives with respect to $s$ and $t$, in local coordinates \eqref{localPTcoords}).  Thus the overall contribution of $[S-1,S+1]$ to the total norm of the expression is bounded by a constant times $e^{S\delta}e^{-S}\left\|\kappa\right\|$.
\end{proof}

\subsection{Approximate right inverse}

Recall that by assumption, the linearized operator:
\begin{equation}
D_{0,0}:W^{k,2,\delta}(C_0,u_0^\ast TX)_D\oplus E\to W^{k-1,2,\delta}(\tilde C_0,\Omega^{0,1}_{\tilde C_0,j_0}\otimes_\CC u_0^\ast TX)
\end{equation}
is surjective (even if we replace $E$ with $E'$).  We now proceed to fix a bounded right inverse:
\begin{equation}\label{chosenrightinverse}
Q_{0,0}:W^{k-1,2,\delta}(\tilde C_0,\Omega^{0,1}_{\tilde C_0,j_0}\otimes_\CC u_0^\ast TX)\to W^{k,2,\delta}(C_0,u_0^\ast TX)_D\oplus E'
\end{equation}
whose image admits a simple description (the description of the image will be important for the proof that the gluing map is continuous (Proposition \ref{continuityofgluing})).  Fix a collection of points $q_i\in C_0$ ($1\leq i\leq h$) not contained in any of the cylindrical ends, subspaces $V_i\subseteq T_{u_0(q_i)}X$, and a subspace $E''\subseteq E'$ so that the natural evaluation map:
\begin{equation}
L_0:K\xrightarrow\sim\biggl(\bigoplus_{i=1}^hT_{u_0(q_i)}X/V_i\biggr)\oplus E/E''
\end{equation}
is an isomorphism (such choices exist since $K\twoheadrightarrow E/E'$ is surjective and elements of $K$ satisfy unique continuation).  Now we can consider the same evaluation map on the larger space:
\begin{equation}\label{lineartofixinverse}
L_0:W^{k,2,\delta}(C_0,u_0^\ast TX)_D\oplus E\to W:=\biggl(\bigoplus_{i=1}^hT_{u_0(q_i)}X/V_i\biggr)\oplus E/E''
\end{equation}
Since $L_0$ sends $K=\ker D_{0,0}$ isomorphically to $W$, it follows that the restriction of $D_{0,0}$ to $\ker L_0$ is an isomorphism of Banach spaces.  Hence there is a unique right inverse:
\begin{equation*}
Q_{0,0}:W^{k-1,2,\delta}(\tilde C_0,\Omega^{0,1}_{\tilde C_0,j_0}\otimes_\CC u_0^\ast TX)\to\ker L_0\subseteq W^{k,2,\delta}(C_0,u_0^\ast TX)_D\oplus E
\end{equation*}
and it is bounded.  Since $E''\subseteq E'$, $\ker L_0$ is in fact contained in the right hand side of \eqref{chosenrightinverse}.  We fix once and for all this $Q_{0,0}$.

\begin{definition}[Approximate right inverse $T_{\alpha,y}$]
We define a map:
\begin{equation}
T_{\alpha,y}:W^{k-1,2,\delta}(\tilde C_\alpha,\Omega^{0,1}_{\tilde C_\alpha,j_y}\otimes_\CC u_\alpha^\ast TX)\to W^{k,2,\delta}(C_\alpha,u_\alpha^\ast TX)_D\oplus E
\end{equation}
as the composition:
\begin{equation}
T_{\alpha,y}:={\operatorname{glue}}\circ\PT\circ\id\circ Q_{0,0}\circ I_y\circ\PT\circ\operatorname{break}
\end{equation}
of maps in the following diagram, to be defined below:
\begin{equation}\label{biginversediagram}
\begin{tikzcd}
W^{k,2,\delta}(C_\alpha,u_\alpha^\ast TX)_D\oplus E\ar{r}{D_{\alpha,y}}&W^{k-1,2,\delta}(\tilde C_\alpha,\Omega^{0,1}_{\tilde C_\alpha,j_y}\otimes_\CC u_\alpha^\ast TX)\ar{d}{\mathrm{break}}\\
W^{k,2,\delta}(C_0,u_{0|\alpha}^\ast TX)_D\oplus E\ar{u}{\mathrm{glue}}\ar{r}{D_{0|\alpha,y}}&W^{k-1,2,\delta}(\tilde C_0,\Omega^{0,1}_{\tilde C_0,j_y}\otimes_\CC u_{0|\alpha}^\ast TX)\ar[leftrightarrow]{d}{\PT}\\
W^{k,2,\delta}(C_0,u_0^\ast TX)_D\oplus E\ar[leftrightarrow]{u}{\PT}\ar{r}{D_{0,y}}&W^{k-1,2,\delta}(\tilde C_0,\Omega^{0,1}_{\tilde C_0,j_y}\otimes_\CC u_0^\ast TX)\ar[leftrightarrow]{d}{I_y}\\
W^{k,2,\delta}(C_0,u_0^\ast TX)_D\oplus E\ar[equals]{u}{\id}\ar[yshift=0.5ex]{r}{D_{0,0}}&\ar[yshift=-0.5ex]{l}{Q_{0,0}}W^{k-1,2,\delta}(\tilde C_0,\Omega^{0,1}_{\tilde C_0,j_0}\otimes_\CC u_0^\ast TX)
\end{tikzcd}
\end{equation}
($D_{0|\alpha,y}$ denotes the linearized operator at $u_{0|\alpha}$).

We fix once and for all a smooth family of $(j_0,j_y)$-linear identifications $I_y:(T\tilde C_0,j_0)\to(T\tilde C_0,j_y)$ which are the identity over the ends/necks.  This gives rise to the bottom right vertical map in \eqref{biginversediagram}.

We define the middle vertical maps in \eqref{biginversediagram} using parallel transport (equivalently, using the local trivializations defined by \eqref{localPTcoords}).

We define the map:
\begin{equation}
W^{k-1,2,\delta}(\tilde C_\alpha,\Omega^{0,1}_{\tilde C_\alpha,j_y}\otimes_\CC u_\alpha^\ast TX)\xrightarrow{\mathrm{break}}W^{k-1,2,\delta}(\tilde C_0,\Omega^{0,1}_{\tilde C_0,j_0}\otimes_\CC u_{0|\alpha}^\ast TX)
\end{equation}
as follows.  Fix a smooth function $\bar\chi:\RR\to[0,1]$ such that:
\begin{equation}
\bar\chi(x)=\begin{cases}1&x\leq-1\cr0&x\geq+1\end{cases}\qquad\bar\chi(x)+\bar\chi(-x)=1
\end{equation}
Let $\eta\in W^{k-1,2,\delta}(\tilde C_\alpha,\Omega^{0,1}_{\tilde C_\alpha,j_y}\otimes_\CC u_\alpha^\ast TX)$.  Away from the ends with $\alpha\ne 0$, $\operatorname{break}(\eta)$ is simply $\eta$.  In any particular end $[0,\infty)\times S^1\subseteq C_0$ with $\alpha\ne 0$, we define:
\begin{equation}
\operatorname{break}(\eta)(s,t):=\begin{cases}
\eta(s,t)&s\leq 3S-1\cr
\bar\chi(s-3S)\cdot\eta(s,t)&3S-1\leq s\leq 3S+1\cr
0&3S+1\leq s
\end{cases}
\end{equation}
(noting the corresponding neck $[0,6S]\times S^1\subseteq C_\alpha$).

We define the map:
\begin{equation}
W^{k,2,\delta}(C_0,u_{0|\alpha}^\ast TX)_D\xrightarrow{\mathrm{glue}}W^{k,2,\delta}(C_\alpha,u_\alpha^\ast TX)_D
\end{equation}
Let $\xi\in W^{k,2,\delta}(C_0,u_{0|\alpha}^\ast TX)_D$.  Away from the necks of $C_\alpha$, $\operatorname{glue}(\xi)$ is simply $\xi$.  In any particular neck $[0,6S]\times S^1\subseteq C_\alpha$, we define:
\begin{equation}
\operatorname{glue}(\xi)(s,t):=\begin{cases}
\xi(s,t)&s\leq 2S\cr
\xi(n)+\chi(4S-s)\cdot[\xi(s,t)-\xi(n)]+\chi(4S-s')\cdot[\xi(s',t')-\xi(n)]&2S\leq s\leq 4S\cr
\xi(s',t')&4S\leq s
\end{cases}
\end{equation}
(noting the corresponding ends $(s,t)\in[0,\infty)\times S^1\subseteq C_0$ and $(s',t')\in[0,\infty)\times S^1\subseteq C_0$).  Note that the cutoff of $\xi(s,t)$ occurs around $4S\in[0,6S]$, where the weight $e^{2\delta\min(s,6S-s)}$ is much smaller than the weight $e^{2\delta s}$ at $4S\in[0,\infty)$.  We will see in the proof of Lemma \ref{approxtop} that this makes it easy to derive the desired estimates on $\delbar\mathrm{glue}(\xi)$.  This trick was explained to us by Abouzaid and attributed to Fukaya--Oh--Ohta--Ono \cite{foootechnicaldetails}.
\end{definition}

Let us make the elementary observation that the definition of $L_0$ extends perfectly well to give an analogous bounded linear map:
\begin{equation}\label{newLonpreglued}
L_\alpha:W^{k,2,\delta}(C_\alpha,u_\alpha^\ast TX)_D\oplus E\to W
\end{equation}
Since $\im Q_{0,0}\subseteq\ker L_0$, it follows from the definition of $T_{\alpha,y}$ that $\im T_{\alpha,y}\subseteq\ker L_\alpha$ as well.  This is a key ingredient in the proof that the gluing map is continuous (Proposition \ref{continuityofgluing}).

\begin{lemma}\label{approxbottom}
Let:
\begin{equation}
\begin{CD}
X@>D>>Y\cr
@AGAA@VVBV\cr
X'@>D'>>Y'
\end{CD}
\end{equation}
denote the bottom square in \eqref{biginversediagram}.  Then for $\xi\in X'$ and $\eta\in Y$ with $D'\xi=B\eta$, we have:
\begin{equation}
\left\|DG\xi-\eta\right\|\leq c\cdot\left\|\xi\right\|\left|y\right|
\end{equation}
uniformly over $(\alpha,y)$ in a neighborhood of zero, for $c<\infty$ depending on data which has been previously fixed.
\end{lemma}

\begin{proof}
In simpler terms, we have $\left\|D_{0,y}-I_y\circ D_{0,0}\right\|\leq c\cdot\left|y\right|$ (calculation left to the reader) and this trivially implies the claimed statement.
\end{proof}

\begin{lemma}\label{approxmiddle}
Let:
\begin{equation}
\begin{CD}
X@>D>>Y\cr
@AGAA@VVBV\cr
X'@>D'>>Y'
\end{CD}
\end{equation}
denote the middle square in \eqref{biginversediagram}.  Then for $\xi\in X'$ and $\eta\in Y$ with $D'\xi=B\eta$, we have:
\begin{equation}
\left\|DG\xi-\eta\right\|\leq c\cdot\left\|\xi\right\|\sum_{i=1}^de^{-S_i}
\end{equation}
uniformly over $(\alpha,y)$ in a neighborhood of zero, for $c<\infty$ depending on data which has been previously fixed.
\end{lemma}

\begin{proof}
In simpler terms, we bound the operator norm of the difference between the two diagonal compositions:
\begin{equation}\label{betterPTcommutingoperatornormbound}
\left\|\PT\circ D_{0,y}-D_{0|\alpha,y}\circ\PT\right\|\leq c\cdot\sum_{i=1}^de^{-S_i}
\end{equation}
(this trivially implies the claimed statement).  To show \eqref{betterPTcommutingoperatornormbound}, observe that the two operators only differ over the $[S-1,\infty)$ subset of each end $[0,\infty)\times S^1\subseteq C_0$.  Using estimates \eqref{uzeroexponentialdecay} arising from the fact that $u_0$ is smooth, we obtain the desired bound.
\end{proof}

\begin{lemma}\label{approxtop}
Let:
\begin{equation}
\begin{CD}
X@>D>>Y\cr
@AGAA@VVBV\cr
X'@>D'>>Y'
\end{CD}
\end{equation}
denote the top square in \eqref{biginversediagram}.  Then for $\xi\in X'$ and $\eta\in Y$ with $D'\xi=B\eta$, we have:
\begin{equation}
\left\|DG\xi-\eta\right\|\leq c\cdot\left\|\xi\right\|\sum_{i=1}^de^{-2\delta S_i}
\end{equation}
uniformly over $(\alpha,y)$ in a neighborhood of zero, for $c<\infty$ depending on data which has been previously fixed.
\end{lemma}

\begin{proof}
The difference $DG\xi-\eta$ is only nonzero in the necks over $s\in[2S,2S+1]$ and $s\in[4S-1,4S]$.  By symmetry, we may just do the bound over $s\in[4S-1,4S]$.  Over this region, one calculates that the norm of the difference is bounded as claimed.  The factor of $e^{-2\delta S}$ comes as the ratio between the $e^{2\delta S}$ weight given to $[4S-1,4S]\times S^1\subseteq[0,6S]\times S^1\subseteq C_\alpha$ and the $e^{4\delta S}$ weight given to $[4S-1,4S]\times S^1\subseteq[0,\infty)\times S^1\subseteq C_0$.
\end{proof}

\begin{lemma}\label{bootstrapinverse}
Let $X$ and $Y$ be Banach spaces, let $D:X\to Y$ and $T:Y\to X$ be bounded, and suppose $\epsilon:=\left\|DT-\mathbf 1_Y\right\|<1$.  Then:
\begin{equation}
Q:=T\cdot\sum_{n\geq 0}(\mathbf 1_Y-DT)^n
\end{equation}
converges and satisfies $DQ=\mathbf 1_Y$.  In addition, we have the following estimates:
\begin{equation}
\left\|Q\right\|\leq\frac 1{1-\epsilon}\left\|T\right\|\qquad
\left\|Q-T\right\|\leq\frac\epsilon{1-\epsilon}\left\|T\right\|
\end{equation}
\end{lemma}

\begin{proof}
Use the telescoping sum $(\mathbf 1-A)\sum_{n\geq 0}A^n=\mathbf 1$ for $\left\|A\right\|<1$.
\end{proof}

\begin{lemma}\label{bootstrappingapprox}
Suppose we have a diagram of Banach spaces as follows:
\begin{equation}
\begin{tikzcd}
X_1\ar{r}{D_1}&Y_1\ar{d}{B_1}\\
\vdots\ar{u}{G_1}&\vdots\ar{d}{B_{n-2}}\\
X_{n-1}\ar{u}{G_{n-2}}\ar{r}{D_{n-1}}&Y_{n-1}\ar{d}{B_{n-1}}\\
X_n\ar{u}{G_{n-1}}\ar[yshift=0.5ex]{r}{D_n}&Y_n\ar[yshift=-0.5ex]{l}{Q_n}
\end{tikzcd}
\end{equation}
where $\left\|D_i\right\|,\left\|G_i\right\|,\left\|B_i\right\|,\left\|Q_n\right\|\leq c$ and $D_nQ_n=\mathbf 1$.  Then for all $\delta>0$ there exists $\epsilon=\epsilon(n,c,\delta)>0$ such that if for all $1<i\leq n$:
\begin{equation}\label{keytocheckforARI}
D_i\xi=B_{i-1}\eta\implies\left\|D_{i-1}G_{i-1}\xi-\eta\right\|\leq\epsilon\cdot\left\|\xi\right\|
\end{equation}
then we have:
\begin{equation}\label{approxbootstrappingworks}
\left\|D_1G_1\cdots G_{n-1}Q_nB_{n-1}\cdots B_1-\mathbf 1_{Y_1}\right\|\leq\delta
\end{equation}
\end{lemma}

\begin{proof}
We work by induction on $n$.  The case $n=1$ is clear: $\epsilon(1,c,\delta)=\infty$.

Now assume $n\geq 2$.  Let us also assume without loss of generality that $c\geq 1$.  Applying \eqref{keytocheckforARI} to $i=n$ and $\xi=Q_nB_{n-1}\eta$, we see that:
\begin{equation}
\left\|D_{n-1}G_{n-1}Q_nB_{n-1}-\mathbf 1_{Y_{n-1}}\right\|\leq\epsilon\cdot\left\|Q_n\right\|\left\|B_{n-1}\right\|\leq\epsilon c^2
\end{equation}
Let us require that $\epsilon\leq\frac 12c^{-2}$, so the above bound is $\leq\frac 12$.  Then by Lemma \ref{bootstrapinverse} applied to $T=G_{n-1}Q_nB_{n-1}$, we see that there exists $Q_{n-1}$ with $D_{n-1}Q_{n-1}=\mathbf 1_{Y_{n-1}}$ and:
\begin{equation}
\left\|Q_{n-1}\right\|\leq 2c^3\qquad\left\|Q_{n-1}-G_{n-1}Q_nB_{n-1}\right\|\leq 2\epsilon c^5
\end{equation}
Now we see that:
\begin{multline}\label{approxclosetogiven}
\left\|D_1G_1\cdots G_{n-2}G_{n-1}Q_nB_{n-1}B_{n-2}\cdots B_1-D_1G_1\cdots G_{n-2}Q_{n-1}B_{n-2}\cdots B_1\right\|\\
\leq c^{2n-3}\left\|G_{n-1}Q_nB_{n-1}-Q_{n-1}\right\|\leq 2\epsilon c^{2n+2}
\end{multline}
Let us require that $\epsilon\leq\frac 14\delta c^{-2n-2}$, so the above bound is $\leq\frac 12\delta$.  Let us also require that $\epsilon\leq\epsilon(n-1,2c^3,\frac 12\delta)$ (which exists by the induction hypothesis), so that:
\begin{equation}\label{inductionbound}
\left\|D_1G_1\cdots G_{n-2}Q_{n-1}B_{n-2}\cdots B_1-\mathbf 1_{Y_1}\right\|\leq\frac 12\delta
\end{equation}
Combining \eqref{approxclosetogiven} and \eqref{inductionbound}, we get the desired bound \eqref{approxbootstrappingworks}.
\end{proof}

\begin{proposition}[Approximate right inverse $T_{\alpha,y}$]\label{approxworks}
We have:
\begin{align}
\left\|T_{\alpha,y}\right\|&\leq c\\
\label{Tapproxeqn}\left\|D_{\alpha,y} T_{\alpha,y}-\mathbf 1\right\|&\to 0\\
\im T_{\alpha,y}&\subseteq\ker L_\alpha
\end{align}
as $(\alpha,y)\to 0$, for $c<\infty$ depending on data which has been previously fixed.
\end{proposition}

\begin{proof}
It is easy to see that all the maps in \eqref{biginversediagram} are uniformly bounded.  Hence $\left\|T_{\alpha,y}\right\|\leq c$ as $(\alpha,y)\to 0$.  Now Lemma \ref{bootstrappingapprox} combined with Lemmas \ref{approxbottom}, \ref{approxmiddle}, \ref{approxtop} show that for $(\alpha,y)\to 0$, we have $\left\|D_{\alpha,y} T_{\alpha,y}-\mathbf 1\right\|\to 0$.  We observed earlier that $\im T_{\alpha,y}\subseteq\ker L_\alpha$.
\end{proof}

\begin{definition}[Right inverse $Q_{\alpha,y}$]
We define a map:
\begin{equation}
Q_{\alpha,y}:W^{k-1,2,\delta}(\tilde C_\alpha,\Omega^{0,1}_{\tilde C_\alpha,j_y}\otimes_\CC u_\alpha^\ast TX)\to W^{k,2,\delta}(C_\alpha,u_\alpha^\ast TX)_D\oplus E
\end{equation}
as the sum:
\begin{equation}
Q_{\alpha,y}:=T_{\alpha,y}\sum_{k=0}^\infty(\mathbf 1-D_{\alpha,y} T_{\alpha,y})^k
\end{equation}
\end{definition}

\begin{proposition}\label{Qestimates}
We have:
\begin{align}
\left\|Q_{\alpha,y}\right\|&\leq c\\
D_{\alpha,y}Q_{\alpha,y}&=\mathbf 1\\
\im Q_{\alpha,y}&\subseteq\ker L_\alpha
\end{align}
uniformly over $(\alpha,y)$ in a neighborhood of zero, for $c<\infty$ depending on data which has been previously fixed.
\end{proposition}

\begin{proof}
Apply Lemma \ref{bootstrapinverse} and Proposition \ref{approxworks}.
\end{proof}

\subsection{Quadratic estimates}

\begin{proposition}[Quadratic estimate]\label{quadestimate}
There exist $c'>0$ and $c<\infty$ (depending on data which has been previously fixed) such that for $\left\|\xi_1\right\|_{k,2,\delta},\left\|\xi_2\right\|_{k,2,\delta}\leq c'$, we have:
\begin{equation}\label{quadestgoal}
\bigl\|D_{\alpha,y}(\xi_1-\xi_2)-(\F_{\alpha,y}\xi_1-\F_{\alpha,y}\xi_2)\bigr\|_{k-1,2,\delta}\leq c\cdot\left\|\xi_1-\xi_2\right\|_{k,2,\delta}(\left\|\xi_1\right\|_{k,2,\delta}+\left\|\xi_2\right\|_{k,2,\delta})
\end{equation}
(and $\F_{\alpha,y}\xi_1$ and $\F_{\alpha,y}\xi_2$ are both defined), uniformly over $(\alpha,y)$ in a neighborhood of zero.
\end{proposition}

\begin{proof}
This is similar to McDuff--Salamon \cite[p68, Proposition 3.5.3]{mcduffsalamonJholsymp}.

We have already remarked that $\F_{\alpha,y}\xi$ is defined for $\left\|\xi\right\|_{k,2,\delta}\leq c'$.

Let $\F_{\alpha,y}'(\zeta,\xi)$ denote the derivative of $\F_{\alpha,y}$ at $\zeta$ applied to $\xi$.  So, for instance, $D_{\alpha,y}(\xi):=\F_{\alpha,y}'(0,\xi)$.  It suffices to show that:
\begin{equation}\label{derivislip}
\left\|\F_{\alpha,y}'(0,\xi)-\F_{\alpha,y}'(\zeta,\xi)\right\|_{k-1,2,\delta}\leq c\cdot\left\|\zeta\right\|_{k,2,\delta}\left\|\xi\right\|_{k,2,\delta}
\end{equation}
for $\left\|\zeta\right\|_{k,2,\delta}\leq c'$ uniformly as $(\alpha,y)\to 0$ (one recovers \eqref{quadestgoal} by integrating $\int_{\xi_1}^{\xi_2}\F_{\alpha,y}'(0,d\zeta)-\F_{\alpha,y}'(\zeta,d\zeta)$).

For $\zeta\in W^{k,2,\delta}(C_\alpha,u_\alpha^\ast TX)_D\oplus E$, let:
\begin{equation}
\F_{\alpha,y,\zeta}:W^{k,2,\delta}(C_\alpha,(\exp_{u_\alpha}\zeta)^\ast TX)_D\oplus E\to W^{k-1,2,\delta}(\tilde C_\alpha,\Omega^{0,1}_{\tilde C_\alpha,j_y}\otimes_\CC(\exp_{u_\alpha}\zeta)^\ast TX)
\end{equation}
denote the $\delbar$-section based at $\exp_{u_\alpha}\zeta:C_\alpha\to X$ and $e_0+\proj_E\zeta$ (so, for example, $\F_{\alpha,y}:=\F_{\alpha,y,0}$).  Let:
\begin{equation}
D_{\alpha,y,\zeta}:W^{k,2,\delta}(C_\alpha,(\exp_{u_\alpha}\zeta)^\ast TX)_D\oplus E\to W^{k-1,2,\delta}(\tilde C_\alpha,\Omega^{0,1}_{\tilde C_\alpha,j_y}\otimes_\CC(\exp_{u_\alpha}\zeta)^\ast TX)
\end{equation}
denote the derivative of $\F_{\alpha,y,\zeta}$ at zero.  Of course, $D_{\alpha,y,\zeta}$ may be calculated as in Lemma \ref{linearizedformula}, and the result is the same (i.e.\ we just substitute $(\exp_{u_\alpha}\zeta,e_0+\proj_E\zeta)$ in place of $(u_\alpha,e_0)$).

Now the first step in proving \eqref{derivislip} is to express $\F_{\alpha,y}'(\zeta,\xi)$ in terms of $D_{\alpha,y,\zeta}$.  To do this, we observe that:
\begin{equation}
\F_{\alpha,y}(a)=\left[\PT_{\exp_{u_\alpha}a\to u_\alpha}\circ\PT_{\exp_{u_\alpha}\zeta\to\exp_{u_\alpha}a}\right]\left[\F_{\alpha,y,\zeta}\Bigl((\exp_{\exp_{u_\alpha}\zeta}^{-1}\exp_{u_\alpha}a)\oplus(\proj_Ea-\proj_E\zeta)\Bigr)\right]
\end{equation}
We now differentiate with respect to $a$ and evaluate at $a=\zeta$ and $\dot a=\xi$.  We find:
\begin{multline}\label{changebasisderivative}
\F_{\alpha,y}'(\zeta,\xi)=\Biggl[\frac d{da}\biggr|_{\begin{smallmatrix}a=\zeta\cr\dot a=\xi\end{smallmatrix}}\Bigl(\PT_{\exp_{u_\alpha}a\to u_\alpha}\circ\PT_{\exp_{u_\alpha}\zeta\to\exp_{u_\alpha}a}\Bigr)\Biggr]\bigl(\F_{\alpha,y,\zeta}(0)\bigr)\\
+\PT_{\exp_{u_\alpha}\zeta\to u_\alpha}\Biggl[D_{\alpha,y,\zeta}\Biggl(\frac d{da}\biggr|_{\begin{smallmatrix}a=\zeta\cr\dot a=\xi\end{smallmatrix}}\Bigl(\exp_{\exp_{u_\alpha}\zeta}^{-1}\exp_{u_\alpha}a\Bigr)\oplus\proj_E\xi\Biggr)\Biggr]
\end{multline}
We rewrite the first term:
\begin{equation}
\Biggl[\frac d{da}\biggr|_{\begin{smallmatrix}a=\zeta\cr\dot a=\xi\end{smallmatrix}}\Bigl(\PT_{\exp_{u_\alpha}a\to u_\alpha}\circ\PT_{\exp_{u_\alpha}\zeta\to\exp_{u_\alpha}a}\Bigr)\Biggr]\PT_{u_\alpha\to\exp_{u_\alpha}\zeta}\bigl(\F_{\alpha,y}(\zeta)\bigr)
\end{equation}
We know that $\left\|\F_{\alpha,y}(\zeta)\right\|_{k-1,2,\delta}$ is bounded uniformly for $(\alpha,y)\to 0$ and $\left\|\zeta\right\|_{k,2,\delta}\leq c'$ (Lemma \ref{mappregluingissmall} implies that $\left\|\F_{\alpha,y}(0)\right\|_{k-1,2,\delta}$ is bounded as $(\alpha,y)\to 0$, and from this one may derive a bound on $\left\|\F_{\alpha,y}(\zeta)\right\|_{k-1,2,\delta}$ in terms of $\left\|\zeta\right\|_{k,2,\delta}$).  The operator $[\frac d{da}(\PT\circ\PT)]\PT$ in front is of the form $H(\zeta,\xi)$ for a smooth (non-linear) bundle map $H:TX\oplus TX\to\operatorname{End}(TX)$ (defined in a neighborhood of $\zeta=\xi=0$).  Since $H$ satisfies $H(0,\cdot)=H(\cdot,0)=0$, it follows that $\left\|H(\zeta,\xi)\right\|_{k,2,\delta}$ is bounded by $c\cdot\left\|\zeta\right\|_{k,2,\delta}\left\|\xi\right\|_{k,2,\delta}$ for $\left\|\zeta\right\|_{k,2,\delta},\left\|\xi\right\|_{k,2,\delta}\leq c'$.  Hence the $\left\|\cdot\right\|_{k-1,2,\delta}$-norm of the first term in \eqref{changebasisderivative} is bounded by $c\cdot\left\|\zeta\right\|_{k,2,\delta}\left\|\xi\right\|_{k,2,\delta}$, so for the purposes of proving \eqref{derivislip} it may be ignored.

The second term in \eqref{changebasisderivative} is approximated by $\PT_{\exp_{u_\alpha}\zeta\to u_\alpha}\left[D_{\alpha,y,\zeta}\left(\PT_{u_\alpha\to\exp_{u_\alpha}\zeta}\xi\right)\right]$ with error:
\begin{equation}\label{errorexpression}
\PT_{\exp_{u_\alpha}\zeta\to u_\alpha}\Biggl[D_{\alpha,y,\zeta}\Biggl(\biggl[\PT_{u_\alpha\to\exp_{u_\alpha}\zeta}\xi-\frac d{da}\biggr|_{\begin{smallmatrix}a=\zeta\cr\dot a=\xi\end{smallmatrix}}\Bigl(\exp_{\exp_{u_\alpha}\zeta}^{-1}\exp_{u_\alpha}a\Bigr)\biggr]\oplus 0\Biggr)\Biggr]
\end{equation}
which we may write as:
\begin{equation}
\PT_{\exp_{u_\alpha}\zeta\to u_\alpha}\Biggl[D_{\alpha,y,\zeta}\Biggl(\PT_{u_\alpha\to\exp_{u_\alpha}\zeta}\biggl[\xi-\PT_{\exp_{u_\alpha}\zeta\to u_\alpha}\frac d{da}\biggr|_{\begin{smallmatrix}a=\zeta\cr\dot a=\xi\end{smallmatrix}}\Bigl(\exp_{\exp_{u_\alpha}\zeta}^{-1}\exp_{u_\alpha}a\Bigr)\biggr]\oplus 0\Biggr)\Biggr]
\end{equation}
Now the $\left\|\cdot\right\|_{(k-1,2,\delta)\to(k-1,2,\delta)}$ norm of the outer $\PT$ is bounded uniformly for $(\alpha,y)\to 0$ and $\left\|\zeta\right\|_{k,2,\delta}\leq c'$, as is the $\left\|\cdot\right\|_{(k,2,\delta)\to(k-1,2,\delta)}$ norm of $D_{\alpha,y,\zeta}$ and the $\left\|\cdot\right\|_{(k,2,\delta)\to(k,2,\delta)}$ norm of the following $\PT$.  The difference $\xi-\PT\frac d{da}()$ is of the form $H(\zeta,\xi)$ for a smooth (non-linear) bundle map $H:TX\oplus TX\to TX$ (defined in a neighborhood of $\zeta=\xi=0$).  Since $H$ satisfies $H(0,\cdot)=H(\cdot,0)=0$, it follows that $\left\|H(\zeta,\xi)\right\|_{k,2,\delta}$ is bounded by $c\cdot\left\|\zeta\right\|_{k,2,\delta}\left\|\xi\right\|_{k,2,\delta}$ for $\left\|\zeta\right\|_{k,2,\delta},\left\|\xi\right\|_{k,2,\delta}\leq c'$.  Hence the error \eqref{errorexpression} has $\left\|\cdot\right\|_{k-1,2,\delta}$ bounded by $c\cdot\left\|\zeta\right\|_{k,2,\delta}\left\|\xi\right\|_{k,2,\delta}$.

Thus we have reduced the estimate \eqref{derivislip} to proving:
\begin{equation}\label{quadreducedtoPTcomp}
\left\|D_{\alpha,y}-\PT_{\exp_{u_\alpha}\zeta\to u}\circ D_{\alpha,y,\zeta}\circ\PT_{u_\alpha\to\exp_{u_\alpha}\zeta}\right\|_{(k,2,\delta)\to(k-1,2,\delta)}\leq c\cdot\left\|\zeta\right\|_{k,2,\delta}
\end{equation}
We calculated $D_{\alpha,y}$ in Lemma \ref{linearizedformula}, and $D_{\alpha,y,\zeta}$ may be expressed in exactly the same way (specifically, it is obtained by taking the expression for $D_{\alpha,y}$ and replacing every occurence of $u_\alpha$ with $\exp_{u_\alpha}\zeta$ and $e_0$ by $e_0+\proj_E\zeta$).  Now we compare term by term to prove \eqref{quadreducedtoPTcomp}.  We omit the details of this calculation.
\end{proof}

\subsection{Newton--Picard iteration}

\begin{lemma}\label{contractionprop}
There exists $c'>0$ (depending on data which has been previously fixed) such that for sufficiently small $(\alpha,y)$:
\begin{rlist}
\item The map $\F_{\alpha,y}$ is defined for $\left\|\xi\right\|_{k,2,\delta}\leq c'$.
\item For $\xi_1-\xi_2\in\im Q_{\alpha,y}$ and $\left\|\xi_1\right\|_{k,2,\delta},\left\|\xi_2\right\|_{k,2,\delta}\leq c'$, we have:
\begin{equation}
\left\|(\xi_1-\xi_2)-(Q_{\alpha,y}\F_{\alpha,y}\xi_1-Q_{\alpha,y}\F_{\alpha,y}\xi_2)\right\|_{k,2,\delta}\leq\frac 12\left\|\xi_1-\xi_2\right\|_{k,2,\delta}
\end{equation}
\end{rlist}
\end{lemma}

\begin{proof}
The first assertion has been shown earlier.  For the second, write:
\begin{multline}
\left\|(\xi_1-\xi_2)-(Q_{\alpha,y}\F_{\alpha,y}\xi_1-Q_{\alpha,y}\F_{\alpha,y}\xi_2)\right\|_{k,2,\delta}\\
=\left\|Q_{\alpha,y}D_{\alpha,y}(\xi_1-\xi_2)-(Q_{\alpha,y}\F_{\alpha,y}\xi_1-Q_{\alpha,y}\F_{\alpha,y}\xi_2)\right\|_{k,2,\delta}\\
\leq\left\|Q_{\alpha,y}\right\|\left\|D_{\alpha,y}(\xi_1-\xi_2)-(\F_{\alpha,y}\xi_1-\F_{\alpha,y}\xi_2)\right\|_{k-1,2,\delta}\\
\leq c\cdot\left\|Q_{\alpha,y}\right\|\left\|\xi_1-\xi_2\right\|_{k,2,\delta}(\left\|\xi_1\right\|_{k,2,\delta}+\left\|\xi_2\right\|_{k,2,\delta})
\end{multline}
by Proposition \ref{quadestimate}.  Since $\left\|Q_{\alpha,y}\right\|$ is uniformly bounded, this is enough.
\end{proof}

\begin{proposition}[Newton--Picard iteration]\label{newtoniteration}
There exists $c'>0$ (depending on data which has been previously fixed) so that for $(\alpha,y,\kappa\in K)$ sufficiently small, there exists a unique $\kappa_{\alpha,y}\in W^{k,2,\delta}(C_\alpha,u_\alpha^\ast TX)_D\oplus E$ satisfying:
\begin{align}
\kappa_{\alpha,y}&\in\kappa_\alpha+\im Q_{\alpha,y}\\
\left\|\kappa_{\alpha,y}\right\|_{k,2,\delta}&\leq c'\\
\F_{\alpha,y}\kappa_{\alpha,y}&=0
\end{align}
\end{proposition}

\begin{proof}
In fact, we will show that $\kappa_{\alpha,y}$ is given explicitly as the limit of the Newton iteration:
\begin{align}
\xi_0&:=\kappa_\alpha\\
\xi_n&:=\xi_{n-1}-Q_{\alpha,y}\F_{\alpha,y}\xi_{n-1}
\end{align}
By Lemma \ref{contractionprop}, the map $\xi\mapsto\xi-Q_{\alpha,y}\F_{\alpha,y}\xi$ is a $\frac 12$-contraction mapping when restricted to:
\begin{equation}\label{domainofcontraction}
\{\xi\in\kappa_\alpha+\im Q_{\alpha,y}:\left\|\xi\right\|_{k,2,\delta}\leq c'\}
\end{equation}
To finish the proof, it suffices to show that (for sufficiently small $(\alpha,y,\kappa)$) \eqref{domainofcontraction} is nonempty and is mapped to itself by $\xi\mapsto\xi-Q_{\alpha,y}\F_{\alpha,y}\xi$.

We know that $\left\|\kappa_\alpha\right\|_{k,2,\delta}\to 0$ as $\kappa\to 0$ (uniformly in $(\alpha,y)$), so \eqref{domainofcontraction} is nonempty.  By using Proposition \ref{quadestimate} with $(\xi_1,\xi_2)=(0,\kappa_\alpha)$ and Lemmas \ref{mappregluingissmall} and \ref{pregluingKestimate}, we conclude that:
\begin{equation}\label{kernelclosetoholomorphic}
\left\|\F_{\alpha,y}\kappa_\alpha\right\|_{k-1,2,\delta}\to 0
\end{equation}
as $(\alpha,y,\kappa)\to 0$.  Since the operator norm of $Q_{\alpha,y}$ is bounded uniformly as $(\alpha,y)\to 0$, we see that $\kappa_\alpha$ is almost fixed by $\xi\mapsto\xi-Q_{\alpha,y}\F_{\alpha,y}\xi$ as $(\alpha,y,\kappa)\to 0$.  It then follows from the contraction property that $\xi\mapsto\xi-Q_{\alpha,y}\F_{\alpha,y}\xi$ maps \eqref{domainofcontraction} to itself.
\end{proof}

\subsection{Gluing}

\begin{definition}[Gluing map]
We define:
\begin{align}
u_{\alpha,y,\kappa}&:=\exp_{u_\alpha}\kappa_{\alpha,y}\\
e_{\alpha,y,\kappa}&:=e_0+\proj_E\kappa_{\alpha,y}
\end{align}
where $\kappa_{\alpha,y}$ is as in Proposition \ref{newtoniteration}, and we consider the map:
\begin{align}\label{gluingmap}
\CC^d\times\RR^{\dim\Mbar^d}\times K&\to\Mbar(X)\\
(\alpha,y,\kappa)&\mapsto(\alpha,y,u_{\alpha,y,\kappa},e_{\alpha,y,\kappa})
\end{align}
(with $\{x_i\}$ understood).  It follows from the definition that \eqref{gluingmap} commutes with the projection from both sides to $\Mbar\times E/E'$.
\end{definition}

\begin{lemma}
The gluing map \eqref{gluingmap} maps sufficiently small $(\alpha,y,\kappa)$ to $\Mbar(X)^\reg$.
\end{lemma}

\begin{proof}
This is true since $Q_{\alpha,y}$ gives an approximate right inverse at $(u_{\alpha,y,\kappa},e_{\alpha,y,\kappa})$ (use \eqref{derivislip} with $\zeta=\kappa_{\alpha,y}$).
\end{proof}

Let $K_\alpha\subseteq C^\infty(C_\alpha,u_\alpha^\ast TX)_D\oplus E$ denote the image of $\kappa\mapsto\kappa_\alpha$.  It is clear by definition that $K\to K_\alpha$ is an isomorphism, and the respective $W^{k,2,\delta}$ norms are uniformly commensurable.  It is also clear that the following commutes:
\begin{equation}\label{Lonkernelisgood}
\begin{tikzcd}[column sep = tiny]
K\ar{dr}[swap]{L_0}\ar{rr}{\kappa\mapsto\kappa_\alpha}&&K_\alpha\ar{dl}{L_\alpha}\\
&W
\end{tikzcd}
\end{equation}
(all maps being isomorphisms).  Since $\im Q_{\alpha,y}\subseteq\ker L_\alpha$, it follows in particular that $\im Q_{\alpha,y}\cap K_\alpha=0$.  On the other hand, an index calculation shows that $\ind D_{\alpha,y}=\ind D_{0,0}$ (note that by Lemma \ref{fredholmandsamekernelcokernel}, it suffices to calculate their indices as operators $W^{k,2}\to W^{k-1,2}$ on $\tilde C_\alpha$ and $\tilde C_0$ respectively, and this is a standard calculation as in McDuff--Salamon \cite{mcduffsalamonJholsymp}).  Both are surjective, and hence we have $\dim\coker Q_{\alpha,y}=\dim\ker D_{\alpha,y}=\dim\ker D_{0,0}=\dim K=\dim K_\alpha$.  It follows that $\im Q_{\alpha,y}=\ker L_\alpha$ and that:
\begin{equation}\label{QKtoall}
\im Q_{\alpha,y}\oplus K_\alpha\xrightarrow\sim W^{k,2,\delta}(C_\alpha,u_\alpha^\ast TX)_D\oplus E
\end{equation}
is an isomorphism of Banach spaces.  We claim that in fact the two norms are uniformly commensurable as $(\alpha,y)\to 0$.  The map written is clearly uniformly bounded, so we just need to show the same for its inverse.  It suffices to show that the projection from the right hand side to $K_\alpha$ is uniformly bounded, but this is nothing other than $L_\alpha$ (clearly uniformly bounded) composed with the inverse of the isomorphism in \eqref{Lonkernelisgood} (also uniformly bounded).

\begin{lemma}\label{gluingisinjective}
The map \eqref{gluingmap} is injective in a neighborhood of zero.
\end{lemma}

\begin{proof}
Suppose that:
\begin{equation}
(\alpha,y,u_{\alpha,y,\kappa},e_{\alpha,y,\kappa})=(\alpha',y',u_{\alpha',y',\kappa'},e_{\alpha',y',\kappa'})
\end{equation}
We see immediately that $(\alpha,y)=(\alpha',y')$.  Now we see that:
\begin{align}
\exp_{u_\alpha}\kappa_{\alpha,y}&=\exp_{u_\alpha}(\kappa')_{\alpha,y}\\
\proj_E\kappa_{\alpha,y}&=\proj_E(\kappa')_{\alpha,y}
\end{align}
Since the norms of $\kappa_{\alpha,y}$ and $\kappa_{\alpha,y}'$ go to zero as $(\alpha,y,\kappa,\kappa')\to 0$ (in $W^{k,2,\delta}$, and hence in $C^0$) and the injectivity radius of the exponential map is fixed, we see that $\kappa_{\alpha,y}=(\kappa')_{\alpha,y}$.  It follows that $\kappa_\alpha-(\kappa')_\alpha\in\im Q_{\alpha,y}$, but since $K_\alpha\cap\im Q_{\alpha,y}=0$ we conclude $\kappa_\alpha=(\kappa')_\alpha$ and hence $\kappa=\kappa'$.
\end{proof}

\begin{proposition}\label{continuityofgluing}
The map \eqref{gluingmap} is continuous in a neighborhood of zero.
\end{proposition}

\begin{proof}
The key ingredient in this proof is our precise control of the image of the right inverse $Q_{\alpha,y}$ (specifically, that $\im Q_{\alpha,y}=\ker L_\alpha$).

Suppose $(\alpha_i,y_i,\kappa_i)\to(\alpha,y,\kappa)$ is a convergent net.\footnote{We could restrict to sequences rather than nets since $\CC^d\times\RR^{\dim\Mbar^d}\times K$ is first countable.  However, this would not make the argument any simpler.}  We will show that:
\begin{equation}
(u_{\alpha_i,y_i,\kappa_i},e_{\alpha_i,y_i,\kappa_i})\to(u_{\alpha,y,\kappa},e_{\alpha,y,\kappa})
\end{equation}

First, we claim that $\left\|(\kappa_i)_{\alpha_i,y_i}-\kappa_{\alpha_i,y_i}\right\|_\infty\to 0$.  In fact, we will show the stronger statement that $\left\|(\kappa_i)_{\alpha_i,y_i}-\kappa_{\alpha_i,y_i}\right\|_{k,2,\delta}\to 0$.  Since the Newton iteration converges uniformly, it suffices to show that $\bigl\|(\kappa_i)_{\alpha_i,y_i}^{n}-\kappa_{\alpha_i,y_i}^{n}\bigr\|_{k,2,\delta}\to 0$ for all $n\geq 0$, where $\kappa_{\alpha,y}^{n}$ denotes the $n$th step of the Newton iteration converging to $\kappa_{\alpha,y}$.  It is easy to verify this for $n=0$:
\begin{equation}
\left\|(\kappa_i)_{\alpha_i,y_i}^{0}-\kappa_{\alpha_i,y_i}^{0}\right\|_{k,2,\delta}=\left\|(\kappa_i-\kappa)_{\alpha_i}\right\|_{k,2,\delta}\to 0
\end{equation}
The inductive step follows from the fact that $Q_{\alpha_i,y_i}$ is uniformly bounded and $\F_{\alpha_i,y_i}$ is uniformly Lipschitz (e.g.\ as a consequence of Proposition \ref{quadestimate}).  Now from the claim, we see that it suffices to show that:
\begin{equation}
(u_{\alpha_i,y_i,\kappa},e_{\alpha_i,y_i,\kappa})\to(u_{\alpha,y,\kappa},e_{\alpha,y,\kappa})
\end{equation}

Recall that by definition:
\begin{align}
u_{\alpha,y,\kappa}&=\exp_{u_\alpha}\kappa_{\alpha,y}&\kappa_{\alpha,y}&=\kappa_\alpha+\xi&\text{for some }\xi&\in\im Q_{\alpha,y}\\
u_{\alpha_i,y_i,\kappa}&=\exp_{u_{\alpha_i}}\kappa_{\alpha_i,y_i}&\kappa_{\alpha_i,y_i}&=\kappa_{\alpha_i}+\xi_i&\text{for some }\xi_i&\in\im Q_{\alpha_i,y_i}
\end{align}
Now we define $\xi_{\alpha_i}\in W^{k,2,\delta}(C_{\alpha_i},u_{\alpha_i}^\ast TX)_D\oplus E$ by ``pregluing'' $\xi$ from $C_\alpha$ to $C_{\alpha_i}$ as follows.  Note that we may assume without loss of generality that at the nodes where $\alpha\ne 0$, we also have $\alpha_i\ne 0$.  Away from the ends/necks of $C_{\alpha_i}$, we set $\xi_{\alpha_i}=\xi$.  Over the ends of $C_{\alpha_i}$, note there is a corresponding end of $C_\alpha$, so we may also simply set $\xi_{\alpha_i}=\xi$ over the ends of $C_{\alpha_i}$.  Over the necks of $C_{\alpha_i}$ for which $\alpha=0$, we define $\xi_{\alpha_i}$ via \eqref{kernelpregluingformula} (note that this is reasonable since $\xi$ is smooth on $C_{\alpha}$).  Over the necks of $C_{\alpha_i}$ for which $\alpha\ne 0$ we define $\xi_{\alpha_i}$ as:
\begin{equation}\label{xitoIpregluingforcontinuous}
\xi_{\alpha_i}(s,t):=\PT_{u_{\alpha}(f_i(s),t)\to u_{\alpha_i}(s,t)}[\xi(f_i(s),t))]
\end{equation}
where $f_i:[0,6S_i]\to[0,6S]$ is defined as follows:
\begin{align}
f_i(s):=&\begin{cases}
s&s\leq S-2\cr
s-3S_i+3S&3S_i-2S+2\leq s\leq 3S_i+2S-2\cr
s-6S_i+6S&6S_i-S+2\leq s\cr
\end{cases}\\
f_i([S-2,3S_i-2S+2])&\subseteq[S-2,S+2]\\
f_i([3S_i+2S-2,6S_i-S+2])&\subseteq[5S-2,5S+2]
\end{align}
and is smooth with absolutely bounded derivatives of all orders.  More informally, $f_i$ is smooth and matches up $[0,S-2],[S+2,5S-2],[5S+2,6S]\subseteq[0,6S]$ with corresponding intervals of the same length inside $[0,6S_i]$, symmetrically.

Now the $C^0$-distance between:
\begin{align}
u_{\alpha,y,\kappa}=\exp_{u_\alpha}(\kappa_\alpha+\xi)&:C_\alpha\to X\quad\text{and}\cr
\exp_{u_{\alpha_i}}(\kappa_{\alpha_i}+\xi_{\alpha_i})&:C_{\alpha_i}\to X
\end{align}
goes to zero (this is easy to see from \eqref{kernelpregluingformula} and \eqref{xitoIpregluingforcontinuous}).  Hence it suffices to show that $\left\|\xi_i-\xi_{\alpha_i}\right\|_\infty\to 0$.  In fact, we will prove the stronger statement that:
\begin{equation}\label{correctedgetsbetter}
\left\|\xi_i-\xi_{\alpha_i}\right\|_{k,2,\delta}\to 0
\end{equation}

First, we claim that we may assume that:
\begin{equation}\label{pregluedisntbig}
\limsup_i\left\|\xi_{\alpha_i}\right\|_{k,2,\delta}
\end{equation}
is arbitrarily small by taking $(\alpha,y,\kappa)$ sufficiently close to zero.  We know $\|F_{\alpha,y}\kappa_\alpha\|_{k-1,2,\delta}\to 0$ as $(\alpha,y,\kappa)\to 0$ by \eqref{kernelclosetoholomorphic}.  Now since the Newton iteration is uniformly convergent and $Q_{\alpha,y}$ is uniformly bounded, it follows that $\left\|\xi\right\|_{k,2,\delta}\to 0$ as $(\alpha,y,\kappa)\to 0$.  We may easily bound $\|\xi_{\alpha_i}\|_{k,2,\delta}$ in terms of $\|\xi\|_{k,2,\delta}$ and the desired claim follows.

Now since the $\limsup$ \eqref{pregluedisntbig} can be assumed to be arbitrarily small, we may, in particular, assume that for sufficiently large $i$, $\|\kappa_{\alpha_i}+\xi_{\alpha_i}\|_{k,2,\delta}\leq c'$ for the constant $c'>0$ from Proposition \ref{contractionprop} for which \eqref{domainofcontraction} is a domain of contraction.  By construction, the fact that $\xi\in\im Q_{\alpha,y}=\ker L_\alpha$ implies that $\xi_{\alpha_i}\in\ker L_{\alpha_i}=\im Q_{\alpha_i,y_i}$.  Hence $\kappa_{\alpha_i}+\xi_{\alpha_i}$ lies in the domain of contraction \eqref{domainofcontraction} for $(\alpha_i,y_i,\kappa)$ where the Newton iteration applies.  By definition, we have $\F_{\alpha_i,y_i}(\kappa_{\alpha_i}+\xi_i)=0$ and $\kappa_{\alpha_i}+\xi_i$ is in the same domain of contraction.  Since the Newton iteration is a $\frac 12$-contraction on this domain and $Q_{\alpha_i,y_i}$ is uniformly bounded, to show \eqref{correctedgetsbetter}, it suffices to show that:
\begin{equation}\label{pregluedcorrectionalmosthol}
\left\|\F_{\alpha_i,y_i}(\kappa_{\alpha_i}+\xi_{\alpha_i})\right\|_{k-1,2,\delta}\to 0
\end{equation}

To prove \eqref{pregluedcorrectionalmosthol}, first recall that $\F_{\alpha,y}(\kappa_\alpha+\xi)=0$.

Away from the ends/necks of $C_{\alpha_i}$, the expression $\F_{\alpha_i,y_i}(\kappa_{\alpha_i}+\xi_{\alpha_i})$ differs from $\F_{\alpha,y}(\kappa_\alpha+\xi)=0$ only in terms of the complex structure $j_{y_i}$ in place of $j_y$.  Clearly this difference goes to zero as $y_i\to y$.

Over the ends of $C_{\alpha_i}$, the expression $\F_{\alpha_i,y_i}(\kappa_{\alpha_i}+\xi_{\alpha_i})$ coincides with $\F_{\alpha,y}(\kappa_\alpha+\xi)=0$.

Over the necks of $C_{\alpha_i}$, we bound $\F_{\alpha_i,y_i}(\kappa_{\alpha_i}+\xi_{\alpha_i})$ as follows.  Fix a neck $[0,6S_i]\subseteq C_{\alpha_i}$.  If $\alpha\ne 0$ for this neck, then the desired estimate follows easily from the definition of $\xi_{\alpha_i}$, $\kappa_{\alpha_i}$, and $u_{\alpha_i}$, and the fact that $\F_{\alpha,y}(\kappa_\alpha+\xi)=0$.  If $\alpha=0$ for this neck, then we argue as follows.  The expression $\F_{\alpha_i,y_i}(\kappa_{\alpha_i}+\xi_{\alpha_i})$ is only nonzero over $[S_i-1,S_i+1]\cup[5S_i-1,5S_i+1]$; by symmetry we may just bound it over $[S_i-1,S_i+1]$.  As in the proof of Lemma \ref{mappregluingissmall}, we win because of the weights: the expression in question has derivatives bounded by $e^{-S}$ and is weighted by $e^{\delta S}$.
\end{proof}

\subsection{Surjectivity of gluing}

We identify $\RR^{2n}=\CC^n$ via $z_j=x_j+iy_j$.  Let $g_\std$ denote the standard metric on $\RR^{2n}$, and let $J_\std$ denote the standard almost complex structure on $\CC^n$.

The following proposition is well-known, and is essentially contained in McDuff--Salamon \cite{mcduffsalamonJholsymp}.

\begin{proposition}[A priori estimate on long pseudo-holomorphic necks]\label{neckestimate}
For all $\mu<1$, there exists $\epsilon>0$ with the following property.  Let $J$ be a smooth almost complex structure on $B^{2n}(1)\subseteq\RR^{2n}$ satisfying $\left\|J-J_\std\right\|_{C^2}\leq\epsilon$ (measured with respect to $g_\std$).  There exists $c_k<\infty$ depending only on $\left\|J\right\|_{C^k}$ with the following property.  Let $u:[-R,R]\times S^1\to B^{2n}(1)$ be $J$-holomorphic.  If:
\begin{equation}
\int_{\partial[-R,R]\times S^1}\Bigl|\frac{du}{dt}\Bigr|^2\,dt<\epsilon
\end{equation}
Then:
\begin{equation}
\left|(D^ku)(s,t)\right|\leq c_k\cdot e^{\mu\cdot(\left|s\right|-R)}\left(\int_{\partial[-R,R]\times S^1}\Bigl|\frac{du}{dt}\Bigr|^2\,dt\right)^{1/2}
\end{equation}
for $\left|s\right|\leq R-1$.
\end{proposition}

\begin{proof}
First, apply Lemma \ref{energydecayexponential} to get an exponential decay bound on the $W^{1,2}$ norm of $u$ restricted to $[-r,r]\times S^1$.  Second, apply Lemma \ref{energycontrolsWoneinfty} to conclude that the bound on the $W^{1,2}$ norm implies a similar bound on the $W^{1,\infty}$ norm.  Finally, apply elliptic bootstrapping Lemma \ref{Woneinftycontrolsall} to conclude that the bound on the $W^{1,\infty}$ norm implies a similar bound on the $C^k$ norm for all $k<\infty$.
\end{proof}

\begin{lemma}[$W^{1,2}$ norm decays exponentially]\label{energydecayexponential}
For all $\mu<1$, there exists $\epsilon>0$ with the following property.  Let $u:[-R,R]\times S^1\to B^{2n}(1)$ be $J$-holomorphic for an almost complex structure $J$ on $B^{2n}(1)\subseteq\RR^{2n}$ satisfying $\left\|J-J_\std\right\|\leq\epsilon$ (measured with respect to $g_\std$).  Then:
\begin{equation*}
\left(\int_{[-r,r]\times S^1}\biggl[\Bigl|\frac{du}{ds}\Bigr|^2+\Bigl|\frac{du}{dt}\Bigr|^2\biggr]\,ds\,dt\right)^{1/2}\leq c\cdot e^{\mu\cdot(r-R)}\left(\int_{\partial[-R,R]\times S^1}\Bigl|\frac{du}{dt}\Bigr|^2\,dt\right)^{1/2}
\end{equation*}
for $0\leq r\leq R$.
\end{lemma}

\begin{proof}
This proof is essentially lifted from McDuff--Salamon \cite[p99, Lemma 4.7.3]{mcduffsalamonJholsymp}.

Let $\lambda_\std:=\frac 12\sum_jx_j\,dy_j-y_j\,dx_j$.  Note the identity:
\begin{equation}\label{zdzbaridentity}
\sum_jz_j\,d\bar z_j=\frac 12\sum_jd\left|z_j\right|^2-2i\lambda_\std
\end{equation}
Let $\gamma:S^1\to\mathbb R^{2n}$ be a smooth loop.  Write the Fourier series $\gamma(t)=\sum_ka_ke^{ikt}\in\CC^n$ where $a_k\in\CC^n$.  Now using \eqref{zdzbaridentity}, we see that:
\begin{equation*}
\int_{S^1}\gamma^\ast\lambda_\std=\frac i2\int_{S^1}\sum_j\gamma^\ast(z_j\,d\bar z_j)=\frac i2\int_{S^1}\sum_{k,\ell}a_k\bar a_\ell(-i\ell)e^{i(k-\ell)t}\,dt=\pi\sum_kk\left|a_k\right|^2
\end{equation*}
We may also calculate:
\begin{equation}
\int_{S^1}\Bigl|\frac{d\gamma}{dt}\Bigr|^2\,dt=\int_{S^1}\sum_{k,\ell}a_k\bar a_\ell(ik)(-i\ell)e^{i(k-\ell)t}\,dt=2\pi\sum_kk^2\left|a_k\right|^2
\end{equation}
Hence we conclude that:
\begin{equation}\label{sympisoperimetric}
\left|\int_{S^1}\gamma^\ast\lambda_\std\right|\leq\frac 12\int_{S^1}\Bigl|\frac{d\gamma}{dt}\Bigr|^2\,dt
\end{equation}

Let:
\begin{equation}
E(r):=\int_{[-r,r]\times S^1}\Bigl|\frac{du}{dt}\Bigr|^2\,ds\,dt\approx\int_{[-r,r]\times S^1}u^\ast\omega_\std=\int_{\partial[-r,r]\times S^1}u^\ast\lambda_\std
\end{equation}
($\approx$ means equality up to a factor which can be made arbitrarily close to $1$ by taking $\epsilon>0$ sufficiently small).  Applying \eqref{sympisoperimetric}, we conclude that:
\begin{equation}\label{boundonenergyfromboundary}
E(r)\lesssim\frac 12\int_{\partial[-r,r]\times S^1}\Bigl|\frac{du}{dt}\Bigr|^2\,dt
\end{equation}
($\lesssim$ means inequality up to a factor which can be made arbitrarily close to $1$ by taking $\epsilon>0$ sufficiently small).  The right hand side above equals $\frac 12E'(r)$, so we have $E'(r)\gtrsim 2E(r)$, and hence $E'(r)\geq 2\mu E(r)$ (for $\epsilon>0$ sufficiently small), from which we conclude that $E(r)\leq e^{2\mu\cdot(r-R)}E(R)$.  Using \eqref{boundonenergyfromboundary} to bound $E(R)$, we see that:
\begin{equation}
E(r)\lesssim e^{2\mu\cdot(r-R)}\frac 12\int_{\partial[-R,R]\times S^1}\Bigl|\frac{du}{dt}\Bigr|^2\,dt
\end{equation}
We have:
\begin{equation}
E(r)\approx\frac 12\int_{[-r,r]\times S^1}\biggl[\Bigl|\frac{du}{ds}\Bigr|^2+\Bigl|\frac{du}{dt}\Bigr|^2\biggr]\,ds\,dt
\end{equation}
Hence we conclude that:
\begin{equation}
\int_{[-r,r]\times S^1}\biggl[\Bigl|\frac{du}{ds}\Bigr|^2+\Bigl|\frac{du}{dt}\Bigr|^2\biggr]\,ds\,dt\lesssim e^{ 2\mu\cdot(r-R)}\int_{\partial[-R,R]\times S^1}\Bigl|\frac{du}{dt}\Bigr|^2\,dt
\end{equation}
which is the desired estimate.
\end{proof}

\begin{lemma}[$W^{1,2}$ controls $W^{1,\infty}$]\label{energycontrolsWoneinfty}
Let $J$ be an almost complex structure on $B^{2n}(1)\subseteq\RR^{2n}$; there exist $\epsilon>0$ and $c<\infty$, depending only on $\left\|J\right\|_{C^2}$ (measured with respect to $g_\std$) with the following property.  Let $u:[0,1]\times[0,1]\to B^{2n}(1)$ be $J$-holomorphic, and suppose that:
\begin{equation}
\int_{[0,1]^2}\biggl[\Bigl|\frac{du}{dx}\Bigr|^2+\Bigl|\frac{du}{dy}\Bigr|^2\biggr]\,dx\,dy<\epsilon
\end{equation}
Then we have:
\begin{equation}
\Bigl|\frac{du}{dx}({\textstyle\frac 12},{\textstyle\frac 12})\Bigr|+\Bigl|\frac{du}{dy}({\textstyle\frac 12},{\textstyle\frac 12})\Bigr|\leq c\cdot\left(\int_{[0,1]^2}\biggl[\Bigl|\frac{du}{dx}\Bigr|^2+\Bigl|\frac{du}{dy}\Bigr|^2\biggr]\,dx\,dy\right)^{1/2}
\end{equation}
\end{lemma}

\begin{proof}
This proof is essentially lifted from McDuff--Salamon \cite[p80, Lemma 4.3.1]{mcduffsalamonJholsymp}.

Let $w:=\frac 12u_x^2+\frac 12u_y^2$ (using the standard inner product $g_\std$), so $w:[0,1]^2\to\RR_{\geq 0}$.  Now we calculate:
\begin{equation}\label{formulaforDeltaw}
w_{xx}+w_{yy}=\Bigl[u_{xx}^2+2u_{xy}^2+u_{yy}^2\Bigr]+\Bigl[u_x\cdot(u_{xxx}+u_{xyy})+u_y\cdot(u_{xxy}+u_{yyy})\Bigr]
\end{equation}
Now differentiating \eqref{formulaforDeltau} yields:
\begin{align*}
u_{xxx}+u_{xyy}=&\ddot J(u,u_x,u_y)\circ u_x+\dot J(u,u_{xy})\circ u_x+\dot J(u,u_y)\circ u_{xx}\\
&-\ddot J(u,u_x,u_x)\circ u_y-\dot J(u,u_{xx})\circ u_y-\dot J(u,u_x)\circ u_{xy}\\
u_{xxy}+u_{yyy}=&\ddot J(u,u_y,u_y)\circ u_x+\dot J(u,u_{yy})\circ u_x+\dot J(u,u_y)\circ u_{xy}\\
&-\ddot J(u,u_x,u_y)\circ u_y-\dot J(u,u_{xy})\circ u_y-\dot J(u,u_x)\circ u_{yy}
\end{align*}
Now we conclude that:
\begin{align*}
\bigl|u_x\cdot(u_{xxx}&+u_{xyy})+u_y\cdot(u_{xxy}+u_{yyy})\bigr|\\
&\leq 2\left\|J\right\|_{C^2}\cdot(\left|u_x\right|^3\left|u_y\right|+\left|u_x\right|\left|u_y\right|^3)\\
&\qquad+2\left\|J\right\|_{C^1}\cdot(\left|u_x\right|^2\left|u_{xy}\right|+\left|u_x\right|\left|u_y\right|\left|u_{xx}\right|+\left|u_y\right|^2\left|u_{xy}\right|+\left|u_x\right|\left|u_y\right|\left|u_{yy}\right|)\\
&\leq c\left\|J\right\|_{C^2}w^2+c\left\|J\right\|_{C^1}w\cdot(\left|u_{xx}\right|+\left|u_{xy}\right|+\left|u_{yy}\right|)\\
&\leq c\left\|J\right\|_{C^2}w^2+c\left\|J\right\|_{C^1}^2w^2+\frac 1{100}(u_{xx}^2+2u_{xy}^2+u_{yy}^2)
\end{align*}
Plugging this into \eqref{formulaforDeltaw}, we conclude that:
\begin{equation}
\Delta w=w_{xx}+w_{yy}\geq -c\cdot (\left\|J\right\|_{C^2}+\left\|J\right\|_{C^1}^2)w^2
\end{equation}
Now we may apply a mean value inequality McDuff--Salamon \cite[p81, Lemma 4.3.2]{mcduffsalamonJholsymp} or Wehrheim \cite[p306, Theorem 1.1]{wehrheimmeanvalue} to see that there exist $\epsilon>0$ and $c<\infty$ such that if $\int_{[0,1]^2}w\,dx\,dy<\epsilon$, then $w(\frac 12,\frac 12)\leq c\cdot\int_{[0,1]^2}w\,dx\,dy$.  Thus we are done.
\end{proof}

\begin{lemma}[$W^{1,\infty}$ controls $W^{k,p}$]\label{Woneinftycontrolsall}
Let $u:[0,1]\times[0,1]\to B^{2n}(1)$ be $J$-holomorphic for an almost complex structure $J$ on $B^{2n}(1)\subseteq\RR^{2n}$.  For all $k\geq 1$, there exists $c_k<\infty$ depending only on $\left\|J\right\|_{C^k}$ such that if:
\begin{equation}
\sup_{(x,y)\in[0,1]^2}\biggl[\Bigl|\frac{du}{dx}\Bigr|+\Bigl|\frac{du}{dy}\Bigr|\biggr]\leq 1
\end{equation}
then:
\begin{equation}
\left|(D^ku)({\textstyle\frac 12},{\textstyle\frac 12})\right|\leq c_k\cdot\sup_{(x,y)\in[0,1]^2}\biggl[\Bigl|\frac{du}{dx}\Bigr|+\Bigl|\frac{du}{dy}\Bigr|\biggr]
\end{equation}
\end{lemma}

\begin{proof}
For more details see McDuff--Salamon \cite[p533, \S B.4]{mcduffsalamonJholsymp}.

The $\delbar$-equation for $u$ may be written as:
\begin{equation}
u_y=J(u)\circ u_x
\end{equation}
Differentiating with respect to $x$ and to $y$, we conclude that:
\begin{align}
u_{xy}&=\dot J(u,u_x)\circ u_x+J(u)\circ u_{xx}\\
u_{xx}&=\dot J(u,u_y)\circ u_x+J(u)\circ u_{xy}
\end{align}
Combining these and using the fact that $\dot J$ and $J$ anticommute, we conclude that:
\begin{equation}\label{formulaforDeltau}
\Delta u=u_{xx}+u_{yy}=\dot J(u,u_y)\circ u_x-\dot J(u,u_x)\circ u_y
\end{equation}
Now a standard elliptic bootstrapping argument based on \eqref{formulaforDeltau} gives the desired result.  By hypothesis, we have an $L^\infty$ bound in terms of $\left\|J\right\|_{C^1}$ on the right hand side of \eqref{formulaforDeltau}, which gives a $W^{2,p}$ bound ($2<p<\infty$) on $u$ in terms of $\left\|J\right\|_{C^1}$.  Now we have a $W^{1,p}$ bound on the right hand side of \eqref{formulaforDeltau} in terms of $\left\|J\right\|_{C^2}$, which gives a $W^{3,p}$ bound on $u$ in terms of $\left\|J\right\|_{C^2}$.  Iterating, we get a $W^{k,p}$ bound on $u$ in terms of $\left\|J\right\|_{C^{k-1}}$, which is enough.  There is no need to worry about elliptic regularity estimates near the boundary since we can shrink the domain slightly after each iteration.
\end{proof}

\begin{proposition}\label{gluingsurjective}
The restriction of \eqref{gluingmap} to any neighborhood of zero is surjective onto a neighborhood of $(0,0,u_0,e_0)\in\Mbar(X)$.
\end{proposition}

\begin{proof}
Let $(\alpha_i,y_i,u_i,\{x^i_j\},e_i)\in\Mbar(X)$ be a sequence converging to $(0,0,u_0,\{x^0_j\},e_0)$.  We must show that for $i$ sufficiently large, $(\alpha_i,y_i,u_i,\{x^i_j\},e_i)$ is contained in the image of the map \eqref{gluingmap}.  We may restrict to sequences rather than nets since the topology on $\Mbar(X)$ is first countable (recall the definition of the topology following \eqref{Mbarfirstdef}; it is even a metric topology).  This is convenient when we apply Arzel\`a--Ascoli.

Let us define $\xi_i\in C^\infty(C_{\alpha_i},u_{\alpha_i}^\ast TX)_D\oplus E$ by the property:
\begin{align}
u_i&=\exp_{u_{\alpha_i}}\xi_i\\
e_i&=e_0+\proj_E\xi_i
\end{align}
(and the exponential follows the shortest geodesic).  Obviously $\proj_E\xi_i\to 0$ and $\left\|\xi_i\right\|_\infty\to 0$.

Now we claim that $u_i\to u$ in the $C^\infty$ topology away from the nodes, or equivalently, that $\xi_i\to 0$ in the $C^\infty$ topology away from the nodes of $C_0$.  To see this argue as follows.  The thickened holomorphic curve equation from \eqref{modulispace} is equivalent to an honest holomorphic curve equation for the graph $u:C\to C\times X$ where the almost complex structure on $C\times X$ is defined in terms of $\lambda(e)$.  Hence the Gromov--Schwarz Lemma (see \cite[p316, 1.3.A]{gromov} or \cite[p223, Corollary 4.1.4]{muller}) applies and we conclude that $\left\|du_i\right\|_\infty$ is uniformly bounded, on compact sets away from the nodes (equivalently, the same for $\left\|d\xi_i\right\|_\infty$).  Now elliptic bootstrapping (as in the proof of Lemma \ref{Woneinftycontrolsall}) implies that all derivatives are bounded uniformly on compact sets away from the nodes.  Using Arzel\`a--Ascoli and diagonalization, we conclude that there exists a subsequence of $u_i$ which is convergent in the $C^\infty$ topology away from the nodes.  Since we know that $u_i\to u$ in the $C^0$ topology, the limit of this $C^\infty$ convergent subsequence must be $u$.  This argument can be applied to any subsequence of $u_i$, so we conclude that every subsequence of $u_i$ has a subsequence which converges to $u$ in $C^\infty$ away from the nodes.  It follows that in fact $u_i$ itself converges to $u$ in $C^\infty$ away from the nodes.

In the ends/necks of $C_{\alpha_i}$, note that $u_i$ is genuinely $J$-holomorphic (the $\lambda$ term vanishes), so we may apply Proposition \ref{neckestimate} to conclude that $\left\|\xi_i\right\|_{k,2,\delta}\to 0$ (here using the fact that $\delta<1$).  We should remark that to apply Proposition \ref{neckestimate}, we need to be able to choose some $M<\infty$ such that for any neck $[0,6S]\times S^2\subseteq C_{\alpha_i}$, the derivatives of $u_i$ restricted to $\partial[M,6S-M]\subseteq C_{\alpha_i}$ are bounded by $\delta$ (the constant in Proposition \ref{neckestimate}).  Such an $M$ exists because of our earlier observation that $u_i\to u$ in $C^\infty$ on compact sets away from the nodes, and we can certainly choose a large $M$ such that the derivatives of $u$ over $\{M\}\times S^1\subseteq[0,\infty)\times S^1\subseteq C_0$ (any end) are arbitrarily small.

Let us define $\kappa_i\in K$ by the property that $\xi_i\in(\kappa_i)_{\alpha_i}+\im Q_{\alpha_i,y_i}$.  Now since $\left\|\xi_i\right\|_{k,2,\delta}\to 0$ and the norms on the left and right of \eqref{QKtoall} are uniformly commensurable, we see that $\kappa_i\to 0$.  Now the uniqueness in Proposition \ref{newtoniteration} shows that $\xi_i=(\kappa_i)_{\alpha_i,y_i}$ for $i$ sufficiently large.  Thus we are done.
\end{proof}

\subsection{Conclusion of the proof}

\begin{lemma}\label{injectiveHdorffhomeo}
Let $X$ be a Hausdorff topological space and let $f:\RR^n\to X$ be continuous.  Suppose that:
\begin{rlist}
\item$f$ is injective.
\item The restriction of $f$ to any neighborhood of zero is surjective onto a neighborhood of $f(0)$.
\end{rlist}
Then there is an open set $0\in U\subseteq\RR^n$ such that $f(U)\subseteq X$ is open and $f:U\xrightarrow\sim f(U)$ is a homeomorphism.
\end{lemma}

\begin{proof}
Let $B(1)\subseteq\RR^n$ denote the closed unit ball and $B(1)^\circ$ its interior.  We know that $f(B(1))$ is compact (since $f$ is continuous) and Hausdorff (since $X$ is Hausdorff).  We know that $f:B(1)\to f(B(1))$ is bijective (since $f$ is injective), so it is in fact a homeomorphism.

Now choose an open subset $V\subseteq X$ with $f(0)\in V\subseteq f(B(1)^\circ)$ (which exists by assumption).  Set $U=f^{-1}(V)\cap B(1)^\circ$, which is clearly open.  Now $f:U\to f(U)$ is a homeomorphism onto its image (since $U\subseteq B(1)$), and $f(U)=V$ is open.
\end{proof}

\begin{proof}[Proof of Theorem \ref{mainresult}(\ref{openness}),(\ref{manifold}),(\ref{submersion})]
We have shown that the map $g:=$\eqref{gluingmap} is continuous, injective, and that its restriction to any neighborhood of zero is surjective onto a neighborhood of the image of zero.  The target $\Mbar(X)$ is Hausdorff, and thus it follows from Lemma \ref{injectiveHdorffhomeo} that for some open neighborhood of zero $U\subseteq\CC^d\times\RR^{\dim\Mbar^d}\times K$, we have $g(U)$ is open and $g:U\xrightarrow\sim g(U)$ is a homeomorphism.  Since $g$ respects the projection map from both sides to $\Mbar\times E/E'$, we obtain the desired conclusions (\ref{openness}), (\ref{manifold}), and (\ref{submersion}).
\end{proof}

\subsection{Gluing orientations}

To show Theorem \ref{mainresult}(\ref{orientation}) (the statement about orientations), observe that since $\Mbar(X)^\reg\to\Mbar$ is a submersion, and $\Mbar$ is canonically oriented as a complex manifold, it follows that $\oo_{\Mbar(X)^\reg}$ is canonically identified with the orientation sheaf of the fibers of $\Mbar(X)^\reg\to\Mbar$.  Now orienting a fiber is the same as orienting the kernel $K$.  It is standard to see that $\oo_K=\oo_E\otimes\bigotimes_{i=1}^\ell\oo_{u(p_{n+i})^\ast N_{D/X}}^\vee$.

This argument gives an identification $\oo_{\Mbar(X)^\reg}$ with $\oo_E\otimes\bigotimes_{i=1}^\ell\oo_{u(p_{n+i})^\ast N_{D/X}}^\vee$ at every point in $\Mbar(X)^\reg$.  Via the gluing map, we get an identification of $\oo_{\Mbar(X)^\reg}$ with $\oo_E\otimes\bigotimes_{i=1}^\ell\oo_{u(p_{n+i})^\ast N_{D/X}}^\vee$ over a small neighborhood of every point in $\Mbar(X)^\reg$.  It is probably straightforward to check that these identifications agree on overlaps (this is a concrete question about whether the projection map $K\to K_\alpha\to\ker D_{\alpha,y}$ is orientation preserving; a similar question is dealt with in Floer--Hofer \cite{floerhoferorientations}).

This is the best way to prove the desired result, but we can actually get away with a less technical argument.  Namely, it is easy to see that the identifications induced by points with smooth domain curve agree on overlaps (since then there is no gluing and we have a nice smooth Banach bundle picture).  Since the nodal locus is codimension two inside $\Mbar(X)^\reg$ (this follows from the gluing map constructed previously), the identification $\oo_{\Mbar(X)^\reg}=\oo_E\otimes\bigotimes_{i=1}^\ell\oo_{u(p_{n+i})^\ast N_{D/X}}^\vee$ away from the nodal locus extends uniquely to all of $\Mbar(X)^\reg$.

\section{Gluing for implicit atlases on Hamiltonian Floer moduli spaces}\label{gluingHF}

In this appendix, we supply the gluing analysis which was quoted in \S\ref{hamiltonianfloersec} to justify our assertions that the implicit atlases constructed there satisfy the openness and submersion axioms and to identify their orientation local systems (specifically, we prove Propositions \ref{HFgluingneeded} and \ref{HFgluingSS}).  Our arguments here are very similar to those used in Appendix \ref{gluingappendix} to prove the analogous results for the implicit atlases on Gromov--Witten moduli spaces, and Appendix \ref{gluingappendix} should be read first.  As in Appendix \ref{gluingappendix}, our work here is little more than an appropriate combination of already existing techniques.

\subsection{Setup for gluing}\label{hfgluingsetup}

The goal of this subsection is to reduce Propositions \ref{HFgluingneeded} and \ref{HFgluingSS}\footnote{The moduli spaces in Proposition \ref{HFgluingSS} carry a natural $S^1$-action, and following our arguments carefully may lead to the construction of an $S^1$-equivariant gluing map.  However we do not need such an $S^1$-equivariant gluing map, so henceforth we will ignore this $S^1$-action.} to a single concrete gluing statement concerning a new moduli space $\Mbar(M)$.  We then spend the rest of the appendix proving this statement.

Fix a symplectic manifold $(M,\omega)$, an integer $n\geq 0$, and a simplex $\sigma\in\JH_n(M)$, i.e.\ maps $H:\Delta^n\to C^\infty(M\times S^1)$ and $J:\Delta^n\to J(M)$ as in Definition \ref{JHdef}.  Also fix $\gamma^-\in\PPP_{H(0)}$ and $\gamma^+\in\PPP_{H(n)}$.

Since Propositions \ref{HFgluingneeded} and \ref{HFgluingSS} are local statements, it suffices to prove them in a neighborhood of a given point.  Thus let us fix a basepoint $(C_0,\ell_0,u_0,\{\phi^\beta_0\}_{\beta\in I},\{e^\beta_0\}_{\beta\in I})\in\Mbar(\sigma,\gamma^-,\gamma^+)^{\leq\s}_I$ in the setup of either Proposition \ref{HFgluingneeded} or \ref{HFgluingSS}.  This point consists (in particular) of the following data:
\begin{rlist}
\item $C_0$, a Riemann surface of genus zero with marked points $q^-,q^+\in C_0$.  Note that $C_0$ may be written uniquely as ${\coprod_{i=1}^kC_0^{(i)}}/{\sim}$, where each $C_0^{(i)}$ is a nodal Riemann surface with marked points $q_{i-1},q_i$ (contained in the same irreducible component), the equivalence relation $\sim$ identifies $q_i\in C_0^{(i)}$ with $q_i\in C_0^{(i+1)}$, and $q^-=q_0$, $q^+=q_k$.  We call any irreducible component containing some $q_i$ a \emph{main component}, and we call all other irreducible components \emph{bubble components}.  The nodes other than $q_1,\ldots,q_{k-1}\in C_0$ will be called \emph{bubble nodes}.
\item $\ell_0=\{\ell_0^i:\RR\to\Delta^n\}_{1\leq i\leq k}$, where $\ell_0^i$ is a (possibly constant) Morse flow line (for the flow from Definition \ref{simplexflows}) from vertex $v_{i-1}$ to vertex $v_i$, for vertices $0=v_0\leq\cdots\leq v_k=n$.
\item $A_0=\{A_0^i:C_0^{(i)}\setminus\{q_{i-1},q_i\}\to S^1\times\RR\}_{1\leq i\leq k}$, where $A_0^i$ is holomorphic, sends $q_{i-1}$ (resp.\ $q_i$) to $-\infty$ (resp.\ $\infty$), and restricts to a biholomorphism on the main component of $C_0^{(i)}$.  Recall that we always use $(t,s)\in S^1\times\RR$ as coordinates, and that we equip $S^1\times\RR$ (and any subset thereof) with the standard complex structure $z=e^{s+it}$.  We may identify each main component of $C_0$ with $S^1\times\RR$ via $A_0$, and thus we have coordinates $(t,s)$ on each such component.
\item $u_0=\{u_0^i:C_0^{(i)}\setminus\{q_{i-1},q_i\}\to M\}_{1\leq i\leq k}$, where $u_0^i$ is a smooth map with finite energy, converging to $u_0^i(t,-\infty)=\gamma_{i-1}(t)$ and $u_0^i(t,\infty)=\gamma_i(t)$, for periodic orbits $\gamma_i\in\PPP_{H(v_i)}$ with $\gamma_0=\gamma^-$ and $\gamma_k=\gamma^+$.
\item $e_0:=\bigoplus_{\beta\in I}e^\beta_0\in\bigoplus_{\beta\in I}E_\beta=E_I=:E$.
\end{rlist}
such that:
\begin{equation}
\Bigl(du_0+2d(\proj_{S^1}A_0)\otimes X_{H((\ell_0\times\id_{S^1})(A_0(\cdot)))}(u_0)+\sum_{\beta\in I}\lambda_\beta(e^\beta_0)(\phi^\beta_0(\cdot),u_0(\cdot))\Bigr)^{0,1}_{J(\ell_0(A_0(\cdot)))}=0
\end{equation}
There are a few differences (of an essentially self-explanatory nature) between our notation here and the notation from \S\ref{hamiltonianfloersec}.  Note that $A_0$ is holomorphic even when $e_0\ne 0$, since $\lambda_\beta$ lands in $\Omega^{0,1}_{\Cbar_{0,2+r_\beta}/\Mbar_{0,2+r_\beta}}\otimes_\RR TM$ (as opposed to $\Omega^{0,1}_{\Cbar_{0,2+r_\beta}/\Mbar_{0,2+r_\beta}}\otimes_\RR T(M\times S^1\times\RR)$).

We now proceed to formulate an alternative description of the moduli space $\Mbar(\sigma,\gamma^-,\gamma^+)^{\leq\s}_I$ in a neighborhood of $(C_0,\ell_0,u_0,\{\phi^\beta_0\}_{\beta\in I},\{e^\beta_0\}_{\beta\in I})$.  This description will be tailored specifically for the present goal of proving a gluing theorem, and we will give a more manageable repackaging of the set of thickening datums $\beta\in I$.

\subsubsection{Points $\{x^i_0\}_{1\leq i\leq r}$ and submanifolds $\{D_i\}_{1\leq i\leq r}\subseteq M\times S^1\times\Delta^n$}\label{reindexingsec}

Let us consider the intersections $\{x_0^{\beta,i}\}_{\beta\in I,1\leq i\leq r_\beta}$ of $(\id_{M\times S^1}\times\ell_0)(u_0\times A_0)|_{(C_0)_\beta}$ with $D_\beta$ (recall from Definition \ref{IAonfloerthickeningsdef}(\ref{hfiacalphapart}) that $(C_0)_\beta\subseteq C_0$ denotes the union of $C_0^{(i)}$'s corresponding to $\beta$, and that there are exactly $r_\beta$ such intersections, all of which are transverse).  We immediately reindex these intersection points as $\{x_0^i\}_{1\leq i\leq r}$ (defining $r:=\sum_{\beta\in I}r_\beta$).  Now $D_\beta\subseteq M\times S^1\times\Delta^{[i_0\ldots i_m]}$ (for $[i_0\ldots i_m]$ dictated by $\beta$), and a neighborhood of $\Delta^{[i_0\ldots i_m]}\subseteq\Delta^n$ may be naturally parameterized by $\Delta^{[i_0\ldots i_m]}\times[0,\epsilon)^{\{1,\ldots,n\}\setminus\{i_0,\ldots,i_m\}}$.  Let $\tilde D_\beta:=(D_\beta\setminus\partial^\ess D_\beta)\times[0,\epsilon)^{\{1,\ldots,n\}\setminus\{i_0,\ldots,i_m\}}\subseteq M\times S^1\times\Delta^n$, and define $D_1,\ldots,D_r\subseteq M\times S^1\times\Delta^n$ as the reindexing of $\{\tilde D_\beta\}_{\beta\in I,1\leq i\leq r_\beta}$ corresponding to our earlier reindexing of $\{x_0^{\beta,i}\}_{\beta\in I,1\leq i\leq r_\beta}$ to $\{x_0^i\}_{1\leq i\leq r}$ (thus $D_1,\ldots,D_r$ contains $r_\beta$ copies of $\tilde D_\beta$).  In particular, $(\id_{M\times S^1}\times\ell_0)((u_0\times A_0)(x_0^i))\in D_i$ with transverse intersection (this intersection is transverse, rather than merely ``transverse when viewed inside $M\times S^1\times\Delta^{[i_0\ldots i_m]}$'', since we replaced $D_\beta$ with $\tilde D_\beta$; we made this replacement purely for the sake of this linguistic convenience).

\subsubsection{Points $\{p_i\}_{1\leq i\leq L},\{p_i'\}_{1\leq i\leq L'}\in C_0$ and submanifolds $D,H\subseteq M$}\label{choosingDH}

We claim that every unstable bubble component of $C_0$ has a point where $du_0$ is injective.  Indeed, on any unstable bubble component which is not contained in $(C_0)_\beta$ for any $\beta\in I$, the map $u_0$ is a (nonconstant!) $J(\ell_0(A_0(\cdot)))$-holomorphic sphere in $M$, and thus has such a point of injectivity.  Any unstable bubble component contained in $(C_0)_\beta$ for some $\beta\in I$ is stabilized by its intersections with the corresponding $D_\beta$.  Since these intersections are transverse (and $A_0$ is constant on the bubble), it follows that $du_0$ has a point of injectivity on such bubble components as well.  Hence the claim is valid.  Let us now choose points $p_1,\ldots,p_L\in C_0$ (a minimal set of distinct non-nodal points stabilizing all bubble components of $C_0$) and a codimension two submanifold with boundary $D\subseteq M$ such that $u_0(p_i)\in D$ with transverse intersection for $1\leq i\leq L$.

We claim that every unstable main component of $C_0$ over which $\ell_0$ is constant has a point where $\frac\partial{\partial s}u_0$ is nonzero.  Indeed, this holds because such an unstable main component cannot be a trivial cylinder (otherwise the trajectory would be unstable).  Let us now choose points $p_1',\ldots,p_{L'}'\in C_0$ (consisting of exactly one non-nodal point in every unstable main component over which $\ell_0$ is constant) and a codimension one submanifold with boundary $H\subseteq M$ such that $u_0(p_i')\in H$ with $\frac\partial{\partial s}u_0$ transverse to $H$ for $1\leq i\leq L'$.  We assume that $D$ and $H$ are disjoint: this can be achieved by first perturbing $D$ (and correspondingly $p_1,\ldots,p_L$) so that on each of the main components under consideration here, there exists a point where $\frac\partial{\partial s}u_0$ is nonzero and which is not mapped to $D$.

\subsubsection{Gluing $C_0,\coprod_{i=1}^k\RR$ and varying $j_0,A_0,\ell_0$}\label{gluingandvaryingjAl}

A subset of the Riemann sphere $\CC\cup\{\infty\}$ homeomorphic to $S^1$ is called a \emph{circle} iff its intersection with $\CC$ is either a straight line or a circle (in the usual sense).  It is well-known that this notion is invariant under biholomorphisms of the Riemann sphere.  For any Riemann surface $C$ biholomorphic to the Riemann sphere minus finitely many points, a subset of $C$ homeomorphic to $S^1$ is called a \emph{circle} iff its image under such a biholomorphism is a circle.

On both sides of every bubble node of $C_0$, fix cylindrical ends:
\begin{equation}\label{bubbleends}
S^1\times[0,\infty)\to C_0
\end{equation}
which are \emph{circular}, in the sense that every (equivalently, some) cross section $S^1\times\{s\}$ is a circle (inside the corresponding irreducible component of $C_0$).  We will call the ends \eqref{bubbleends} the \emph{bubble ends}.  We also fix some large $N<\infty$, and we call the subsets $S^1\times[N,\infty)$ and $S^1\times(-\infty,-N]$ of the main components of $C_0$ the \emph{main ends} (which come in two types: \emph{positive} and \emph{negative}).  At a few later points in the gluing argument, we will need to assume that the bubble ends were chosen sufficiently small and that $N$ was chosen sufficiently large.

We now fix a smooth family of (necessarily integrable) almost complex structures $j_y$ on $C_0$ and a smooth family of $j_y$-holomorphic maps:
\begin{equation}\label{Aymaps}
A_y^i:C^{(i)}_0\setminus\{q_{i-1},q_i\}\to S^1\times\RR
\end{equation}
parameterized by $y\in\RR^\ast$ (for an integer $\ast\geq 0$ to be specified shortly), specializing to $(j_0,A_0)$ at $y=0$ (where $j_0$ denotes the given almost complex structure on $C_0$).  We require that this family $(j_y,A_y)_{y\in\RR^\ast}$ satisfy the following conditions.  The family $j_y$ must be constant over the bubble ends, and $A_y$ (and hence $j_y$) must be constant over the main ends.  The bubble ends \eqref{bubbleends} must also be circular with respect to every $j_y$.  Now, for each stable bubble component of $C_0$ (say, with $\nu\geq 3$ special points), the family $j_y$ induces a map $\RR^\ast\to\mathcal M_{0,\nu}$.  For each main component of $C_0$ (say, with $\nu\geq 0$ bubble nodes), the family $A_y$ induces a map $\RR^\ast\to(S^1\times\RR)^\nu$ (namely taking the images of the bubble nodes under $A_y$).  We require that the resulting map:
\begin{equation*}
\RR^\ast\to\prod_{\begin{smallmatrix}\text{bubble components}\cr\text{with $\nu\geq 3$}\end{smallmatrix}}\mathcal M_{0,\nu}\times\prod_{\begin{smallmatrix}\text{main components}\cr\text{with $\ell_0$ non-constant}\end{smallmatrix}}(S^1\times\RR)^\nu\times\prod_{\begin{smallmatrix}\text{main components}\cr\text{with $\ell_0$ constant}\cr\text{and $\nu\geq 1$}\end{smallmatrix}}(S^1\times\RR)^\nu\!/\RR
\end{equation*}
be a diffeomorphism onto its image (in particular, this determines $\ast\in\ZZ_{\geq 0}$).  We fix, once and for all, a family $(j_y,A_y)_{y\in\RR^\ast}$ satisfying the above properties (such a family always exists).

Given a set of gluing parameters\footnote{We only care about what happens for gluing parameters $\alpha$ in a neighborhood of zero, and many constructions will only make sense (and many statements will only be true) for sufficiently small gluing parameters $\alpha$, even though this assumption is not always explicitly stated.  The same goes for the parameters $w$ (introduced below) and $y$.} $\alpha\in\CC^d\times\RR_{\geq 0}^{k-1}$ (where $d$ is the number of bubble nodes of $C_0$), we may glue $C_0$ to obtain a curve $C_\alpha$ as follows.  At a bubble node with gluing parameter $\alpha=e^{-6S+i\theta}\in\CC$, we truncate both bubble ends \eqref{bubbleends} from $S^1\times[0,\infty)$ to $S^1\times[0,6S]$, and we identify them via the map $s'=6S-s$ and $t'=\theta-t$ (if $\alpha=0$, then we do nothing).  At each main node $q_i$ with gluing parameter $\alpha=e^{-6S}\in\RR_{\geq 0}$, we truncate the main ends $S^1\times[N,\infty)\subseteq C_0^{(i)}$ and $S^1\times(-\infty,-N]\subseteq C_0^{(i+1)}$ incident at $q_i$ to $S^1\times[N,6S-N]$ and $S^1\times[-6S+N,-N]$, and we identify them via $s'=s-6S$ and $t'=t$ (if $\alpha=0$, then we do nothing).  The curve $C_\alpha$ now has \emph{main ends} and \emph{bubble ends} (those where no gluing was performed), as well as \emph{bubble necks} $S^1\times[0,6S]$ and \emph{main necks} $S^1\times[N,6S-N]$.

We perform a similar gluing operation to $\coprod_{i=1}^k\RR$, and we denote the result by $(\coprod_{i=1}^k\RR)_\alpha$.  Namely, for the gluing parameter $\alpha=e^{-6S}\in\RR_{\geq 0}$ associated to $q_i$, we truncate the $i$th copy of $\RR$ to $(-\infty,6S]$, we truncate the $(i+1)$st copy of $\RR$ to $[-6S,\infty)$, and we identify $[0,6S]$ in the $i$th copy with $[-6S,0]$ in the $(i+1)$st copy via $s'=s-6S$ (if $\alpha=0$, then we do nothing).

The almost complex structure $j_y$ clearly descends to $C_\alpha$, since it is constant over the ends of $C_0$.  Furthermore, the maps \eqref{Aymaps} induce $j_y$-holomorphic maps:
\begin{equation}
A_y:C_\alpha\setminus\{q_0,\ldots,q_k\}\to S^1\times\Bigl(\coprod_{i=1}^k\RR\Bigr)_\alpha
\end{equation}
characterized uniquely by the property that they agree with \eqref{Aymaps} over the images of the truncated main components of $C_0$ (the existence of such a map $A_y$ follows from our assumption that the bubble ends \eqref{bubbleends} are circular with respect to every $j_y$).  The points $p_1,\ldots,p_L,p_1',\ldots,p_{L'}'\in C_0$ clearly descend to points $p_1,\ldots,p_L,p_1',\ldots,p_{L'}'\in C_\alpha$, and there is a node $q_i\in C_\alpha$ whenever the corresponding gluing parameter is zero.

Eventually, only a subset of $\CC^d\times\RR_{\geq 0}^{k-1}$ will be relevant for us, namely the subset $\CC^d\times(\RR_{\geq 0}^{k-1})^{\leq\s}\subseteq\CC^d\times\RR_{\geq 0}^{k-1}$ cut out by the requirement that if the break in the trajectory at $(v_i,\gamma_i)$ is dictated by the stratum $\s$, then the gluing parameter at $q_i$ is zero.

For every nonconstant $\ell_0^i$, choose a local hypersurface $H^i\subseteq\Delta^{[v_{i-1}\ldots v_i]}$ transverse to $\ell_0^i$ at $\ell_0^i(0)$.  For $w\in\prod_{i=1}^{k\prime}H^i$ (where $\prod_{i=1}^{k\prime}$ denotes the product over those $i$ for which $\ell_0^i$ is nonconstant), let $\ell_w:\coprod_{i=1}^k\RR\to\Delta^n$ denote the unique trajectory satisfying $\prod_{i=1}^{k\prime}\ell_w^i(0)=w$ and broken at the same sequence of vertices $0=v_0\leq\cdots\leq v_k=n$.  We let $0:=\prod_{i=1}^{k\prime}\ell_0^i(0)\in\prod_{i=1}^{k\prime}H^i$, so in this notation $\ell_w=\ell_0$ for $w=0$.

Now choose a neighborhood parameterization $\Delta^{[v_{i-1}\ldots v_i]}\times[0,\epsilon)^{\{1,\ldots,n\}\setminus\{v_{i-1},\ldots,v_i\}}\to\Delta^n$, and define a local hypersurface $\bar H^i:=H^i\times[0,\epsilon)^{\{1,\ldots,n\}\setminus\{v_{i-1},\ldots,v_i\}}\subseteq\Delta^n$ and a projection map $\bar H^i\to H^i$.  For $\alpha\in\RR_{\geq 0}^{k-1}$, let $\ell_{\alpha,w}:(\coprod_{i=1}^k\RR)_\alpha\to\Delta^n$ denote the trajectory characterized uniquely by the property that $\prod_{i=1}^{k\prime}\ell_{\alpha,w}(0_i)\in\prod_{i=1}^{k\prime}\bar H^i$ projects to $w\in\prod_{i=1}^{k\prime}H^i$, where $0_i\in(\coprod_{i=1}^k\RR)_\alpha$ denotes the image of $0$ in the $i$th copy of $\RR$.  In this notation, $\ell_{\alpha,w}=\ell_0$ for $\alpha=0$ and $w=0$.  The existence and uniqueness of $\ell_{\alpha,w}$ for $(\alpha,w)$ close to $(0,0)$ follows simply by explicitly integrating the Morse flow on $\Delta^n$ from Definition \ref{simplexflows}.

Eventually, only a subset of $\prod_{i=1}^{k\prime}H^i$ will be relevant for us, namely the subset $\prod_{i=1}^{k\prime}(H^i)^{\leq\s}\subseteq\prod_{i=1}^{k\prime}H^i$, where $(H^i)^{\leq\s}\subseteq H^i$ is the intersection of $H^i$ with $\Delta^{[i_0\ldots i_m]}\cap\Delta^{[v_{i-1}\ldots v_i]}=\Delta^{[i_0\ldots i_m]\cap[v_{i-1}\ldots v_i]}$, where $(\sigma|[i_0\ldots i_m],\gamma_a,\gamma_b)$ is the part of the stratum $\s$ containing $\ell^i_0$.

\subsubsection{Linear map $\lambda$}

Recall that for all $\beta\in I$, we have a linear map:
\begin{equation}
\lambda_\beta:E_\beta\to C^\infty(\Cbar_{0,2+r_\beta}\times M,\Omega^{0,1}_{\Cbar_{0,2+r_\beta}/\Mbar_{0,2+r_\beta}}\otimes_\RR TM)
\end{equation}
We now repackage these $\lambda_\beta$ into a single linear map:
\begin{equation}
\lambda:E\to C^\infty(\CC^d\times(\RR_{\geq 0}^{k-1})^{\leq\s}\times\RR^\ast\times C_0^r\times C_0\times M,\Omega^1_{C_0}\otimes_\RR TM)
\end{equation}
which vanishes over the ends of (the last factor of) $C_0$.  We will actually only define $\lambda$ in a small neighborhood of $\{0\}\times\{0\}\times\{x_0^1\}\times\cdots\times\{x_0^r\}\times C_0\times M$, as this is the only part of the domain which will be relevant.

Suppose we are given $e\in E=\bigoplus_{\beta\in I}E_\beta$, $\alpha\in\CC^d\times(\RR_{\geq 0}^{k-1})^{\leq\s}$, and $y\in\RR^\ast$, along with points $x^1,\ldots,x^r\in C_0$, each in a small neighborhood of the corresponding $x_0^i$; let us define $\lambda(e)(\alpha,y,x^1,\ldots,x^r,\cdot,\cdot)$ as a section $C_0\times M\to\Omega^{0,1}_{C_0,j_y}\otimes_\RR TM\subseteq\Omega^1_{C_0}\otimes_\RR TM$.  Consider the glued curve $C_\alpha$ equipped with the almost complex structure $j_y$ and the marked points $x^1,\ldots,x^r\in C_\alpha$ (descended from $C_0$).  Since $\alpha\in\CC^d\times(\RR_{\geq 0}^{k-1})^{\leq\s}$, there is a subcurve $(C_\alpha)_\beta\subseteq C_\alpha$ corresponding to any given $\beta\in I$.  If we equip it with the $r_\beta$ marked points from $x^1,\ldots,x^r$ corresponding to $\beta$ (with respect to the reindexing from \S\ref{reindexingsec}), then this induces a unique map $\phi^\beta:(C_\alpha)_\beta\to\Cbar_{0,2+r_\beta}$ (isomorphism onto a fiber) which is close to the given map $\phi_0^\beta:(C_0)_\beta\to\Cbar_{0,2+r_\beta}$ (this assumes we are in the setting of Proposition \ref{HFgluingneeded}; in the setting of Proposition \ref{HFgluingSS}, we instead have a map $\phi^\beta:((C_\alpha)_\beta)^\st\to\Cbar_{0,2+r_\beta}$ close to the given map $\phi_0^\beta:((C_0)_\beta)^\st\to\Cbar_{0,2+r_\beta}$, where $(C_\alpha)_\beta\to((C_\alpha)_\beta)^\st$ contracts those components all of whose constituent components of $(C_0)_\beta$ were contracted by $(C_0)_\beta\to((C_0)_\beta)^\st$).  Now the pullback of $\lambda_\beta(\proj_{E_\beta}e)$ under $\phi^\beta$ gives us a section $C_\alpha\times M\to\Omega^{0,1}_{C_\alpha,j_y}\otimes_\RR TM$ (defined to be zero at those points not contained in $(C_\alpha)_\beta$).  We may assume without loss of generality that the ends of $C_0$ were chosen small enough so that this section vanishes over the ends/necks of $C_\alpha$, and hence gives rise to a well-defined lift to a section $C_0\times M\to\Omega^{0,1}_{C_0,j_y}\otimes_\RR TM$.  We declare $\lambda(e)(\alpha,y,x^1,\ldots,x^r,\cdot,\cdot)$ to be the sum of these sections over all $\beta\in I$.  This defines the function $\lambda$, which is indeed of class $C^\infty$ since the partially defined composition $C_0\to C_\alpha\to(C_\alpha)_\beta\to\Cbar_{0,2+r_\beta}$ depends smoothly on $\alpha,y,x^1,\ldots,x^r$.

\subsubsection{A new moduli space $\Mbar(M)$}

Let $\Mbar(M)$ denote the moduli space of tuples $(\alpha,w,y,u,e,\{x^i\}_{1\leq i\leq r})$, where:
\begin{rlist}
\item$\alpha\in\CC^d\times(\RR_{\geq 0}^{k-1})^{\leq\s}$.
\item$w\in\prod_{i=1}^{k\prime}(H^i)^{\leq\s}$.
\item$y\in\RR^\ast$.
\item$u:C_\alpha\setminus\{q_0,\ldots,q_k\}\to M$ is smooth with finite energy, is asymptotic to $\gamma_i$ at $q_i\in C_\alpha$ (whenever the gluing parameter $\alpha$ at $q_i\in C_0$ is zero), and $u(p_i)\in D^\circ$ ($1\leq i\leq L$) and $u(p_i')\in H^\circ$ ($1\leq i\leq L'$).
\item$e\in E$.
\item$x^i\in C_\alpha$ (not nodes) satisfy $(\id_{M\times S^1}\times\ell_{\alpha,w})((u\times A_y)(x^i))\in D_i$ with transverse intersection ($1\leq i\leq r$).
\item We require that:
\begin{multline}\label{finalHFdelbareqn}
\Bigl(du+2d(\proj_{S^1}A_y)\otimes X_{H((\ell_{\alpha,w}\times\id_{S^1})(A_y(\cdot)))}\\
+\lambda(e)(\alpha,y,x^1,\ldots,x^r,\cdot,u(\cdot))\Bigr)^{0,1}_{j_y,J(\ell_{\alpha,w}(A_y(\cdot)))}=0
\end{multline}
\end{rlist}
We equip $\Mbar(M)$ with the topology of uniform convergence.  More precisely, let $\Mbar:=\CC^d\times(\RR_{\geq 0}^{k-1})^{\leq\s}$, and let $\Cbar{}^\$\to\Mbar$ be the bundle whose fiber over $\alpha$ is $C_\alpha^\$$, obtained from $C_\alpha$ by replacing each $q_i$ ($0\leq i\leq k$) with a copy of $S^1$ (thus $u$ as above extends continuously to $C_\alpha^\$$ and equals $\gamma_i(t)$ on the $S^1$ over $q_i$ whenever the corresponding $\alpha$ is zero).  We equip $\Mbar(M)$ with the obvious topology on $w\in\prod_{i=1}^{k\prime}(H^i)^{\leq\s}$, $y\in\RR^\ast$, $e\in E$, $\{x^i\}_{1\leq i\leq r}\in\Cbar{}^r$, and the Hausdorff topology on the graph of $u$ inside $\Cbar{}^\$\times M$.

Note that there is a distinguished basepoint $(0,0,0,u_0,e_0,\{x_0^i\}_{1\leq i\leq r})\in\Mbar(M)$.  Furthermore, a neighborhood of this basepoint in $\Mbar(M)$ is canonically homeomorphic to a neighborhood of the given basepoint in $\Mbar(\sigma,\gamma^-,\gamma^+)_I^{\leq\s}$ (as long as the given basepoint in $\Mbar(\sigma,\gamma^-,\gamma^+)_I^{\leq\s}$ has trivial automorphism group).  Note that to justify this statement in the setting of Proposition \ref{HFgluingSS}, we must appeal to Lemma \ref{morseistransverse}.

\subsubsection{The regular locus $\Mbar(M)^\reg$}\label{regularlocusdef}

We now define a subset $\Mbar(M)^\reg\subseteq\Mbar(M)$, depending on a choice of subspace $E'\subseteq E$.  Fix a point $(\alpha,w,y,u_1,e_1,\{x^i_1\}_{1\leq i\leq r})\in\Mbar(M)$, and we will describe when it lies in $\Mbar(M)^\reg$.

Fix an integer $k\in\ZZ_{\geq 6}$ (not to be confused with the number $k$ of main components of $C_0$) and a small real number $\delta\in(0,1)$ (we will be precise about how small $\delta$ must be shortly).  Let $W^{k,2,\delta}(C_\alpha,M)$ denote the smooth Banach manifold consisting of maps $u:C_\alpha\setminus\{q_0,\ldots,q_k\}\to M$ which are of locally of class $W^{k,2}$ such that for all positive main ends of $C_\alpha$ with corresponding node $q_i$, we have:
\begin{equation}\label{deltaweightedend}
\int_{S^1\times[N',\infty)}\sum_{j=0}^k\left|D^j\bigl[\exp^{-1}_{\gamma_i(t)}u(t,s)\bigr]\right|^2e^{2s\delta}\;dt\,ds<\infty
\end{equation}
for sufficiently large $N'<\infty$ (along with the analogous condition over negative main ends); this definition is independent of the choice of metric and connection used in \eqref{deltaweightedend}.  Over $W^{k,2,\delta}(C_\alpha,M)$, we consider the smooth Banach bundle $\E$ whose fiber over $u:C_\alpha\to M$ is $W^{k-1,2,\delta}(\tilde C_\alpha,\Omega^{0,1}_{\tilde C_\alpha,j_y}\otimes_\CC u^\ast TM_{J(\ell_{\alpha,w}(A_y(\cdot)))})$, namely the space of sections $\eta:\tilde C_\alpha\to\Omega^{0,1}_{\tilde C_\alpha,j_y}\otimes_\CC u^\ast TM_{J(\ell_{\alpha,w}(A_y(\cdot)))}$ which are locally of class $W^{k-1,2}$ and which satisfy:
\begin{equation}\label{deltaweightedendII}
\int_{S^1\times[N',\infty)}\sum_{j=0}^{k-1}\Bigl|D^j\bigl[\PT_{u(t,s)\to\gamma_i(t)}\eta(t,s)\bigr]\Bigr|^2e^{2s\delta}\;dt\,ds<\infty
\end{equation}
over positive main ends (along with the analogous condition over negative main ends); this definition is independent of the choice of metric and connection used in \eqref{deltaweightedendII}.

It is well-known (see Proposition \ref{neckestimateHF} below) that for every non-degenerate periodic orbit $\gamma$ of a smooth Hamiltonian $H:M\times S^1\to\RR$ and every smooth $\omega$-compatible almost complex structure $J$ on $M$, there exists $\delta_{H,J,\gamma}>0$ such that for every $u:S^1\times[0,\infty)\to M$ (resp.\ $u:S^1\times(-\infty,0]\to M$) of finite energy satisfying $(du+2dt\otimes X_H)^{0,1}_J=0$ and asymptotic to $u(t,\infty)=\gamma(t)$ (resp.\ $u(t,-\infty)=\gamma(t)$), all derivatives of $u$ decay like $e^{-\delta s}$ for any $\delta<\delta_{H,J,\gamma}$ (precisely, $\delta_{H,J,\gamma}$ is the smallest magnitude of any eigenvalue of the asymptotic linearized operator, which depends only on $\omega$, $H$, $J$, and $\gamma$).  In particular, for $0\leq i\leq k$, there is a corresponding $\delta_i:=\delta_{H(v_i),J(v_i),\gamma_i}>0$.  We fix $\delta\in(0,1)$ with $\delta<\delta_i$ for all $i$; thus $u_1\in W^{k,2,\delta}(C_\alpha,M)$.

Now since $k\geq 6$, there are unique continuous functions:
\begin{equation}\label{intersectionsfunctionHF}
x^i:W^{k,2,\delta}(C_\alpha,M)\to C_\alpha\quad(1\leq i\leq r)
\end{equation}
defined for $u$ in a neighborhood of $u_1\in W^{k,2,\delta}(C_\alpha,M)$, coinciding with the given $x^i_1\in C_\alpha$ at $u_1$, and which satisfy $(\id_{M\times S^1}\times\ell_{\alpha,y})((u\times A_y)(x^i(u)))\in D^i$ (the intersection is automatically transverse for $u$ close to $u_1$).  Moreover, \eqref{intersectionsfunctionHF} are of class $C^{k-2}$ (see the discussion following \eqref{intersectionsfunction}).  It follows that the left hand side of \eqref{finalHFdelbareqn} is a $C^{k-2}$ section of $\E$ over $W^{k,2,\delta}(C_\alpha,M)\times E$.  By results of Lockhart--McOwen \cite{lockhartmcowen}, this section is Fredholm for $\delta>0$ as above (note that over each end, the almost complex structure on $M$ is constant and $\omega$-compatible).  Let $W^{k,2,\delta}(C_\alpha,M)_{D,H}$ denote the subspace cut out by the requirements $u(p_i)\in D$ and $u(p_i')\in H$.  We say that the given point $(\alpha,w,y,u_1,e_1,\{x^i_1\}_{1\leq i\leq r})\in\Mbar(M)$ lies in $\Mbar(M)^\reg$ iff the section $(u,e)\mapsto\eqref{finalHFdelbareqn}\oplus e$ of $\E\oplus E/E'$ over $W^{k,2,\delta}(C_\alpha,M)_{D,H}\times E$ is transverse to the zero section at $(u_1,e_1)$ (it follows from elliptic regularity theory that this condition is independent of the choice of $k,\delta$ as above).

Now suppose that we take $E':=E_{I'}$ for $I'\subseteq I$ and that the given basepoint in $\Mbar(\sigma,\gamma^-,\gamma^+)_I^{\leq\s}$ was chosen inside the inverse image of $(\Mbar(\sigma,\gamma^-,\gamma^+)_{I'}^{\leq\s})^\reg$.  Then, the basepoint of $\Mbar(M)$ lies in $\Mbar(M)^\reg$ (this uses the fact that $L,L'$ were chosen to be minimal), and furthermore $\Mbar(M)^\reg\cap\proj_{E/E'}^{-1}(0)$ is contained inside the inverse image of $(\Mbar(\sigma,\gamma^-,\gamma^+)_{I'}^{\leq\s})^\reg$.

\subsection{Our goal: the gluing map}

We henceforth assume that the basepoint in $\Mbar(M)$ lies in $\Mbar(M)^\reg$.  The remainder of this appendix is devoted to the construction of a germ of a local chart:
\begin{equation}\label{HFgluinggoal}
\Bigl(\CC^d\times(\RR_{\geq 0}^{k-1})^{\leq\s}\times\RR^\ast\times\prod_{i=1}^{k\prime}(H^i)^{\leq\s}\times K,(0,0,0,0)\Bigr)\to\Bigl(\Mbar(M),(0,0,0,u_0,e_0,\{x_0^i\}_{1\leq i\leq r})\Bigr)
\end{equation}
which respects the natural projection from both sides to $\CC^d\times(\RR_{\geq 0}^{k-1})^{\leq\s}\times\RR^\ast\times\prod_{i=1}^{k\prime}(H^i)^{\leq\s}\times E/E'$ and whose image is contained in $\Mbar(M)^\reg$ ($K$ denotes the kernel of the linearization of \eqref{finalHFdelbareqn} at the distinguished basepoint of $\Mbar(M)$; see \eqref{definitionofkernel}).  We will also discuss the compatibility of gluing with orientations, and more generally we will discuss how to define coherent orientations on the moduli spaces of Floer trajectories.  Propositions \ref{HFgluingneeded} and \ref{HFgluingSS} follow from the existence of such a chart \eqref{HFgluinggoal} and our earlier observations relating $\Mbar(M)$ to $\Mbar(\sigma,\gamma^-,\gamma^+)_I^{\leq\s}$.

\subsection{Pregluing}

Let $\exp:TM\to M$ denote the exponential map of some Riemannian metric on $M$ for which $D$ and $H$ are totally geodesic (such a metric exists since $D$ and $H$ are disjoint).  Recall that $J$ is a smooth family of almost complex structures on $M$ parameterized by $\Delta^n$ which is constant near the vertices.  Let $\nabla$ denote a smooth family of $J$-linear connections on $M$ parameterized by $\Delta^n$ which is constant near the vertices (for instance, we could take $\nabla_XY:=\frac 12(\nabla^0_XY-J(\nabla^0_X(JY)))$ for any fixed connection $\nabla^0$).  Let $\PT_{x\to y}:T_xM\to T_yM$ denote parallel transport via $\nabla$ along the shortest geodesic between $x$ and $y$ (we will only use this notation when it may be assumed that $x$ and $y$ are very close in $M$); note that $\PT_{x\to y}$ is $J$-linear, and that $\PT_{x\to y}$ is a \emph{family} of maps parameterized by $\Delta^n$.

Fix a smooth function $\chi:\RR\to[0,1]$ satisfying:
\begin{equation}
\chi(x)=\begin{cases}0&x\leq 0\cr 1&x\geq 1\end{cases}
\end{equation}

\begin{definition}[Flattening]
For $\alpha\in\CC^d\times\RR_{\geq 0}^{k-1}$, we define the ``flattened'' map $u_{0|\alpha}:C_0\to M$ as follows.  Away from the ends, $u_{0|\alpha}$ coincides with $u_0$.  Over a bubble end $S^1\times[0,\infty)$, we define $u_{0|\alpha}$ as follows:
\begin{equation}
u_{0|\alpha}(t,s):=\begin{cases}
u_0(t,s)&s\leq S-1\cr
\exp_{u_0(n)}\left[\chi(S-s)\cdot\exp_{u_0(n)}^{-1}u_0(t,s)\right]&S-1\leq s\leq S\cr
u_0(n)&S\leq s
\end{cases}
\end{equation}
where $n\in C_0$ denotes the corresponding node.  Over a positive main end $S^1\times[N,\infty)$, we define $u_{0|\alpha}$ as:
\begin{equation}
u_{0|\alpha}(t,s):=\begin{cases}
u_0(t,s)&s\leq S-1\cr
\exp_{\gamma(t)}\left[\chi(S-s)\cdot\exp_{\gamma(t)}^{-1}u_0(t,s)\right]&S-1\leq s\leq S\cr
\gamma(t)&S\leq s
\end{cases}
\end{equation}
where $\gamma(t)=u_0(t,\infty)$ denotes the corresponding periodic orbit; an analogous definition applies over the negative main ends.
\end{definition}

\begin{definition}[Pregluing]
For $\alpha\in\CC^d\times\RR_{\geq 0}^{k-1}$, we define the ``preglued'' map $u_\alpha:C_\alpha\to M$ as follows.  Away from the necks, $u_\alpha$ coincides with $u_0$.  Over a bubble neck $S^1\times[0,6S]$, we define $u_\alpha$ as:
\begin{equation}
u_\alpha(t,s):=\begin{cases}
u_0(t,s)&s\leq S-1\cr
\exp_{u_0(n)}\left[\chi(S-s)\cdot\exp_{u_0(n)}^{-1}u_0(t,s)\right]&S-1\leq s\leq S\cr
u_0(n)&S\leq s\leq 5S\cr
\exp_{u_0(n)}\left[\chi(S-s')\cdot\exp_{u_0(n)}^{-1}u_0(t',s')\right]&5S\leq s\leq 5S+1\cr
u_0(t',s')&5S+1\leq s
\end{cases}
\end{equation}
and over a main neck $S^1\times[N,6S-N]$, we define $u_\alpha$ as:
\begin{equation}
u_\alpha(t,s):=\begin{cases}
u_0(t,s)&s\leq S-1\cr
\exp_{\gamma(t)}\left[\chi(S-s)\cdot\exp_{\gamma(t)}^{-1}u_0(t,s)\right]&S-1\leq s\leq S\cr
\gamma(t)&S\leq s\leq 5S\cr
\exp_{\gamma(t)}\left[\chi(S+s')\cdot\exp_{\gamma(t)}^{-1}u_0(t',s')\right]&5S\leq s\leq 5S+1\cr
u_0(t',s')&5S+1\leq s
\end{cases}
\end{equation}
($u_\alpha$ should be thought of as the ``descent'' of $u_{0|\alpha}$ from $C_0$ to $C_\alpha$).
\end{definition}

\subsection{Weighted Sobolev norms}

Recall that we have fixed $k\in\ZZ_{\geq 6}$ and $\delta\in(0,1)$ smaller than $\delta_i>0$ for $0\leq i\leq k$.

We now introduce new weighted Sobolev spaces $W^{k,2,\delta,\delta}$ (with weights over all ends and necks) which we will work with from now on.  The specific choice of norms (not just their commensurability classes) on these $W^{k,2,\delta,\delta}$ spaces is crucial.

\begin{definition}\label{weightedonespace}
We define the weighted Sobolev space $W^{k,2,\delta,\delta}(C_\alpha,u_\alpha^\ast TM)$ using the usual $(k,2)$-norm away from the ends/necks, and using the following weighted $(k,2)$-norms over the bubble ends/necks and main ends/necks respectively (we will write the contribution to the norm squared):
\begin{align}
\label{endnormHFbubble}\left|\xi(n)\right|^2+&\int_{S^1\times[0,\infty)}\biggl[\left|\xi(t,s)-\xi(n)\right|^2+\sum_{j=1}^k\left|D^j\xi(t,s)\right|^2\biggr]e^{2\delta s}\;dt\,ds\\
\label{necknormHFbubble}\biggl|\int_{S^1}\xi(t,3S)\,dt\biggr|^2+&\int_{S^1\times[0,6S]}\biggl[\Bigl|\xi(t,s)-\int_{S^1}\xi(t,3S)\,dt\Bigr|^2+\sum_{j=1}^k\left|D^j\xi(t,s)\right|^2\biggr]e^{2\delta\min(s,6S-s)}\;dt\,ds\\
\label{endnormHFmain}&\int_{S^1\times[N,\infty)}\sum_{j=0}^k\left|D^j\xi(t,s)\right|^2e^{2\delta(s-N)}\;dt\,ds\\
\label{necknormHFmain}&\int_{S^1\times[N,6S-N]}\sum_{j=0}^k\left|D^j\xi(t,s)\right|^2e^{2\delta\min(s-N,6S-N-s)}\;dt\,ds
\end{align}
These are to be interpreted as follows.  For each bubble node $n\in C_0$, we fix, once and for all, a small trivialization of $TM$ in a neighborhood of $u_0(n)$; this allows us to view $\xi$ as a function (rather than a section) for the purposes of the integrals over the bubble ends/necks.  For each $i=0,\ldots,k$, we fix, once and for all, a smooth family of trivializations of $TM$ near $\gamma_i(t)$ (parameterized by $t\in S^1$); this allows us to view $\xi$ as a section of $\gamma_i^\ast TM$ for the purposes of the integrals over the main ends/necks.  We also fix, once and for all, a connection on each such bundle $\gamma_i^\ast TM$ over $S^1$.  The derivatives in the integrals above are measured with respect to the standard metric on $S^1\times\RR$.  The case of negative main ends is completely analogous to that of positive main ends.  

By Sobolev embedding $W^{2,2}\hookrightarrow C^0$ in two dimensions, we get uniform bounds linear in $\left\|\xi\right\|_{k,2,\delta,\delta}$ on $\left|\xi(t,s)-\xi(n)\right|e^{\delta s}$ in the bubble ends, $\left|\xi(t,s)\right|e^{\delta s}$ in the main ends, and $\left|D^j\xi(t,s)\right|e^{\delta s}$ ($1\leq j\leq k-2$) in all ends (as well as similar estimates in the necks).

We will occasionally use other very similar weighted Sobolev spaces (e.g.\ $W^{k,2,\delta,\delta}(C_0,u_{0|\alpha}^\ast TM)$), and we leave it to the reader to make the necessary adjustments to the definition (which is essentially identical to the above).
\end{definition}

\begin{remark}
The particular choice of trivializations and connections in the definition above is not crucial: any other (fixed) choice would lead to a uniformly commensurable norm (this holds because $u_0$ satisfies the exponential decay estimates \eqref{uzerobubble}, \eqref{uzeromain}, and because $\delta<1$ and $\delta<\delta_i$).
\end{remark}

\begin{definition}
We define the weighted Sobolev space $W^{k-1,2,\delta,\delta}(\tilde C_\alpha,T^\ast\tilde C_\alpha\otimes_\RR u_\alpha^\ast TM)$ using the usual $(k-1,2)$-norm away from the ends/necks, and using the following weighted $(k-1,2)$-norms over the bubble ends/necks (we will write the contribution to the norm squared):
\begin{align}
\int_{S^1\times[0,\infty)}&\sum_{j=0}^{k-1}\left|D^j\eta(t,s)\right|^2e^{2\delta s}\;dt\,ds\\
\int_{S^1\times[0,6S]}&\sum_{j=0}^{k-1}\left|D^j\eta(t,s)\right|^2e^{2\delta\min(s,6S-s)}\;dt\,ds
\end{align}
(for the main ends/necks, simply make the obvious replacement of $[0,\infty)$ with $[N,\infty)$ and $e^{2\delta s}$ with $e^{2\delta(s-N)}$, etc.).  These are to be interpreted as follows.  We trivialize $T\tilde C_\alpha$ over any end/neck with the basis vectors $\frac\partial{\partial t},\frac\partial{\partial s}$.  We trivialize $TM$ as in Definition \ref{weightedonespace}, and hence the section $\eta$ is simply a pair of functions $\eta=(\eta_1,\eta_2)$.

By Sobolev embedding $W^{2,2}\hookrightarrow C^0$ in two dimensions, we get uniform exponential decay bounds on $\eta$ up to $k-3$ derivatives in any end/neck, linear in $\left\|\eta\right\|_{k-1,2,\delta,\delta}$.

We are actually interested in certain closed subspaces of $W^{k-1,2,\delta,\delta}(\tilde C_\alpha,T^\ast\tilde C_\alpha\otimes_\RR u_\alpha^\ast TM)$, e.g.\ $W^{k-1,2,\delta,\delta}(\tilde C_\alpha,\Omega^{0,1}_{\tilde C_\alpha,j}\otimes_\CC u_\alpha^\ast TM_J)$ for certain almost complex structures $j,J$ on $C_\alpha,M$ respectively, which we equip with the restriction of the norm defined above.  We will occasionally use other very similar weighted Sobolev spaces, and we leave it to the reader to make the necessary adjustments to the definition (which is essentially identical to the above).
\end{definition}

Henceforth, we will work exclusively with the weighted Sobolev spaces defined above, rather than those from \S\ref{regularlocusdef}.  The Fredholm index and the kernel/cokernel of the relevant linearized operators are unchanged by the placement of weights in the bubble ends/necks (the argument from Lemma \ref{fredholmandsamekernelcokernel} applies without modification).

\subsection{Based section \texorpdfstring{$\F_{\alpha,w,y}$}{F\_alpha,w,y} and linearized operator \texorpdfstring{$D_{\alpha,w,y}$}{D\_alpha,w,y}}

We consider the following partially defined function:
\begin{equation*}
\F_{\alpha,w,y}:C^\infty(C_\alpha,u_\alpha^\ast TM)_{D,H}\oplus E\to C^\infty(\tilde C_\alpha,\Omega^{0,1}_{\tilde C_\alpha,j_y}\otimes_\CC u_\alpha^\ast TM_{J(\ell_{\alpha,w}(A_y(\cdot)))})
\end{equation*}
\vspace{-0.2in}
\begin{multline}\label{FformulaHF}
\F_{\alpha,w,y}(\xi):=\PT_{\exp_{u_\alpha}\xi\to u_\alpha}^{\ell_{\alpha,w}(A_y(\cdot))}\Bigl(d\exp_{u_\alpha}\xi+2d(\proj_{S^1}A_y)\otimes X_{H((\ell_{\alpha,w}\times\id_{S^1})(A_y(\cdot)))}(\exp_{u_\alpha}\xi)\\
{}+\lambda(e_0+\proj_E\xi)(\alpha,y,x^1,\ldots,x^r,\cdot,(\exp_{u_\alpha}\xi)(\cdot))\Bigr)^{0,1}_{j_y,J(\ell_{\alpha,w}(A_y(\cdot)))}
\end{multline}
(recall that $\PT$ and $(\cdot)^{0,1}$ commute).  This function $\F_{\alpha,w,y}$ is defined for $\xi$ in a $C^1$-neighborhood of zero; for these $\xi$ we define $x^i=x^i(\xi)$ as the unique intersection of $(\id_{M\times S^1}\times\ell_{\alpha,w})\circ(\exp_{u_\alpha}\xi\times A_y)$ with $D_i$ close to the image of $x^i_0\in C_0$ in $C_\alpha$ (note, however, that even $x^i(0)$ may not coincide exactly with the image of $x^i_0\in C_0$ in $C_\alpha$); as before, these functions $x^i$ are of class $C^{k-2}$.  The subscript $_{D,H}$ indicates restriction to sections which are tangent to $D$ at $p_1,\ldots,p_L$ and tangent to $H$ at $p_1',\ldots,p_{L'}'$.  Thus for $\xi$ contained in a $C^0$-neighborhood of zero, $\exp_{u_\alpha}\xi$ sends $p_1,\ldots,p_L$ to $D^\circ$ and sends $p_1',\ldots,p_{L'}'$ to $H^\circ$.

Now we observe that $\F_{\alpha,w,y}$ induces a continuous map:
\begin{equation}
\F_{\alpha,w,y}:W^{k,2,\delta,\delta}(C_\alpha,u_\alpha^\ast TM)_{D,H}\oplus E\to W^{k-1,2,\delta,\delta}(\tilde C_\alpha,\Omega^{0,1}_{\tilde C_\alpha,j_y}\otimes_\CC u_\alpha^\ast TM_{J(\ell_{\alpha,w}(A_y(\cdot)))})
\end{equation}
which is defined for $\left\|\xi\right\|_{k,2,\delta,\delta}\leq c'$ (some $c'>0$) and small $\alpha,w$.  Moreover, this map is of class $C^{k-2}$.  We denote by:
\begin{equation}
D_{\alpha,w,y}:W^{k,2,\delta,\delta}(C_\alpha,u_\alpha^\ast TM)_{D,H}\oplus E\to W^{k-1,2,\delta,\delta}(\tilde C_\alpha,\Omega^{0,1}_{\tilde C_\alpha,j_y}\otimes_\CC u_\alpha^\ast TM_{J(\ell_{\alpha,w}(A_y(\cdot)))})
\end{equation}
the derivative of $\F_{\alpha,w,y}$ at zero.

Let $T_\nabla(X,Y):=\nabla_XY-\nabla_YX-[X,Y]$ denote the torsion of $\nabla$.

\begin{lemma}\label{linearizedformulaHF}
The linearized operator $D_{\alpha,w,y}$ is given by:
\begin{align}\label{formulaforlinearizedoperatorHF}
D_{\alpha,w,y}\xi&=\biggl(\nabla^{\ell_{\alpha,w}(A_y(\cdot))}\xi+T_\nabla^{\ell_{\alpha,w}(A_y(\cdot))}(\xi,du_\alpha)\\
&\quad\quad+2d(\proj_{S^1}A_y)\otimes\nabla_\xi^{\ell_{\alpha,w}(A_y(\cdot))}X_{H((\ell_{\alpha,w}\times\id_{S^1})(A_y(\cdot)))}\\
\label{firstLterm}&\quad\quad+\sum_{i=1}^r\frac{d[\lambda(e_0)]}{dx^i}(\alpha,y,x^1,\ldots,x^r,\cdot,u_\alpha(\cdot))(-\proj_{TC_\alpha}\xi(x^i))\\
&\quad\quad+\nabla_\xi^{\ell_{\alpha,w}(A_y(\cdot))}[\lambda(e_0)](\alpha,y,x^1,\ldots,x^r,\cdot,u_\alpha(\cdot))\\
\label{lastLterm}&\quad\quad+\lambda(\proj_E\xi)(\alpha,y,x^1,\ldots,x^r,\cdot,u_\alpha(\cdot))\biggr)^{0,1}_{j_y,J(\ell_{\alpha,w}(A_y(\cdot)))}
\end{align}
where $\proj_{TC_\alpha}:T(M\times S^1\times\Delta^n)\to TC_\alpha$ denotes the projection associated to the splitting $T(M\times S^1\times\Delta^n)=TD_i\oplus TC_\alpha$ at the point $x^i\in C_\alpha$ and $(\id_{M\times S^1}\times\ell_{\alpha,w})((u_\alpha\times A_y)(x^i))\in M\times S^1\times\Delta^n$, and $\nabla_\xi[\lambda(e_0)]$ denotes the derivative in the direction of $\xi$ along the $M$ factor with respect to the connection $\nabla$.
\end{lemma}

\begin{proof}
Calculation as in Lemma \ref{linearizedformula}.  Note that $\PT$ and $(\cdot)^{0,1}$ in \eqref{FformulaHF} commute, since $\PT$ is $J$-linear.
\end{proof}

We denote the kernel of $D_{0,0,0}$ by:
\begin{equation}\label{definitionofkernel}
K:=\ker D_{0,0,0}\subseteq C^\infty(C_0,u_0^\ast TM)_{D,H}\oplus E
\end{equation}
Note that our assumption that $(0,0,0,u_0,e_0,\{x_0^i\})\in\Mbar(M)^\reg$ is equivalent to the statement that $D_{0,0,0}$ is surjective and $K\twoheadrightarrow E/E'$ is surjective.

\begin{definition}[Kernel pregluing]\label{kernelpregluingHF}
For $\kappa\in K\subseteq C^\infty(C_0,u_0^\ast TM)$, we define $\kappa_\alpha\in C^\infty(C_\alpha,u_\alpha^\ast TM)$ as follows.  Away from the necks, $\kappa_\alpha$ coincides with $\kappa$.  Over a bubble neck $S^1\times[0,6S]$, we define $\kappa_\alpha$ as:
\begin{equation}\label{kernelpregluingformulaHFbubble}
\kappa_\alpha(t,s):=\begin{cases}
\kappa(t,s)&s\leq S-1\cr
\PT_{u_0(t,s)\to u_\alpha(t,s)}\left[\kappa(t,s)\right]&S-1\leq s\leq S\cr
\PT_{u_0(t,s)\to u_\alpha(t,s)}\left[\kappa(t,s)\right]\cdot(1-\chi(s-S))+\chi(s-S)\cdot \kappa(n)&S\leq s\leq S+1\cr
\kappa(n)&S+1\leq s\leq 5S-1\cr
\PT_{u_0(t',s')\to u_\alpha(t',s')}\left[\kappa(t',s')\right]\cdot(1-\chi(s'-S))+\chi(s'-S)\cdot \kappa(n)&5S-1\leq s\leq 5S\cr
\PT_{u_0(t',s')\to u_\alpha(t',s')}\left[\kappa(t',s')\right]&5S\leq s\leq 5S+1\cr
\kappa(t',s')&5S+1\leq s
\end{cases}
\end{equation}
and over a main neck $S^1\times[N,6S-N]$, we define $\kappa_\alpha$ as:
\begin{equation}\label{kernelpregluingformulaHFmain}
\kappa_\alpha(t,s):=\begin{cases}
\kappa(t,s)&s\leq S-1\cr
\PT_{u_0(t,s)\to u_\alpha(t,s)}\left[\kappa(t,s)\right]&S-1\leq s\leq S\cr
\PT_{u_0(t,s)\to u_\alpha(t,s)}\left[\kappa(t,s)\right]\cdot(1-\chi(s-S))&S\leq s\leq S+1\cr
0&S+1\leq s\leq 5S-1\cr
\PT_{u_0(t',s')\to u_\alpha(t',s')}\left[\kappa(t',s')\right]\cdot(1-\chi(-s'-S))&5S-1\leq s\leq 5S\cr
\PT_{u_0(t',s')\to u_\alpha(t',s')}\left[\kappa(t',s')\right]&5S\leq s\leq 5S+1\cr
\kappa(t',s')&5S+1\leq s
\end{cases}
\end{equation}
\end{definition}

\subsection{Pregluing estimates}

Fix norms on $E$ and $K$, and equip $\CC^d\times\RR^{k-1}$ and $\RR^\ast$ with their standard norms.  Equip each $H^i$ (and hence $\prod_{i=1}^{k\prime}H^i$) with the pullback of the standard norm under a fixed choice of local diffeomorphism to $\RR^n\times\RR^m_{\geq 0}$ sending $0$ to $0$.

Note that we have the following estimates on $u_0$ and $\kappa\in K$:
\begin{align}
\label{uzerobubble}\text{in bubble ends:}&&\left|D^j\exp_{u_0(n)}^{-1}u_0(t,s)\right|&\leq c_je^{-s}\\
\label{kappabubble}\text{in bubble ends:}&&\left|D^j[\kappa(t,s)-\kappa(n)]\right|&\leq c_je^{-s}\left\|\kappa\right\|\\
\label{uzeromain}\text{in main ends:}&&\left|D^j\exp_{\gamma_i(t)}^{-1}u_0(t,s)\right|&\leq c_je^{-\delta's}&&\forall\delta'<\delta_i\\
\label{kappamain}\text{in main ends:}&&\left|D^j\kappa(t,s)\right|&\leq c_je^{-\delta's}\left\|\kappa\right\|&&\forall\delta'<\delta_i
\end{align}
The estimates in the bubble ends hold simply because $u_0$ and $\kappa$ are smooth on $C_0\setminus\{q_0,\ldots,q_k\}$.  The estimates in the main ends hold for $u_0$ by Proposition \ref{neckestimateHF} and for $\kappa$ since $K=\ker D_{0,0,0}$ remains the same for any choice of $k\geq 6$ and any collection of end weights $\delta$, each of which is less than the corresponding $\delta_i$.

\begin{lemma}[Estimate for map pregluing]\label{mappregluingissmallHF}
We have the following estimate on $\F_{\alpha,w,y}(0)$:
\begin{multline}\label{mappregluingdelbar}
\biggl\|\Bigl(du_\alpha+2d(\proj_{S^1}A_y)\otimes X_{H((\ell_{\alpha,w}\times\id_{S^1})(A_y(\cdot)))}(u_\alpha)\\
{}+\lambda(e_0)(\alpha,y,x^1,\ldots,x^r,\cdot,u_\alpha(\cdot))\Bigr)^{0,1}_{j_y,J(\ell_{\alpha,w}(A_y(\cdot)))}\biggr\|_{k-1,2,\delta,\delta}
\leq c\cdot\bigl(\left|\alpha\right|^\epsilon+\left|w\right|+\left|y\right|\bigr)
\end{multline}
uniformly over $(\alpha,w,y)$ in a neighborhood of zero, for $c<\infty$ and $\epsilon>0$ depending on data which has been previously fixed.
\end{lemma}

\begin{proof}
Recall that:
\begin{equation}\label{delbareqnforuzeroHF}
\Bigl(du_0+2d(\proj_{S^1}A_0)\otimes X_{H((\ell_{0,0}\times\id_{S^1})(A_0(\cdot)))}(u_0)
+\lambda(e_0)(0,0,x_0^1,\ldots,x_0^r,\cdot,u_0(\cdot))\Bigr)^{0,1}_{j_0,J(\ell_{0,0}(A_0(\cdot)))}=0
\end{equation}

We estimate $\F_{\alpha,w,y}(0)$ over the main ends/necks.  Note that over this region, the $\lambda$ term vanishes, $j_y=j_0$, $J(\ell_{\alpha,w}(A_y(\cdot)))=J(\ell_{0,0}(A_0(\cdot)))$, and $H((\ell_{\alpha,w}\times\id_{S^1})(A_y(\cdot)))=H((\ell_{0,0}\times\id_{S^1})(A_0(\cdot)))$.  Now due to \eqref{delbareqnforuzeroHF}, it follows that in this region, $\F_{\alpha,w,y}(0)$ is supported inside the subsets $S^1\times([S-1,S]\cup[5S,5S+1])$ of the main necks $S^1\times[N,6S-N]$.  The contribution of such a region to the norm of $\F_{\alpha,w,y}(0)$ is bounded above by a constant times $\left|\alpha\right|^\epsilon$, as follows from the estimate \eqref{uzeromain} on the derivatives of $u_0$ and the fact that $\delta<\delta_i$.

We estimate $\F_{\alpha,w,y}(0)$ away from the ends/necks.  Note that $\ell_{\alpha,w}$ is close to $\ell_{0,0}$, in the sense that they differ over $\coprod_{i=1}^k[-N,N]$ (which may be viewed as a subset of $\coprod_{i=1}^k\RR$ and of $(\coprod_{i=1}^k)_\alpha$) by a constant times $\left|\alpha\right|^\epsilon+\left|w\right|$ in $C^\ell$ for some fixed $\epsilon>0$ (depending on the flow on $\Delta^n$ from Definition \ref{simplexflows}) and any $\ell<\infty$.  It follows that away from the ends/necks, the $C^\ell$ distance between $J(\ell_{\alpha,w}(A_y(\cdot)))$ on $C_\alpha$ and $J(\ell_{0,0}(A_0(\cdot)))$ on $C_0$ is bounded by a constant (depending on $\ell$) times $\left|\alpha\right|^\epsilon+\left|w\right|+\left|y\right|$, where we identify $C_0$ and $C_\alpha$ away from the ends/necks in the canonical way (any $\ell<\infty$).  The same holds for $H((\ell_{\alpha,w}\times\id_{S^1})(A_y(\cdot)))$ and $H((\ell_{0,0}\times\id_{S^1})(A_0(\cdot)))$.  It also follows that the distance between $x^i\in C_\alpha$ and (the image in $C_\alpha$ of) $x^i_0\in C_0$ is bounded by a constant times $|\alpha|^\epsilon+\left|w\right|$.  Hence we conclude that away from the ends/necks, $\F_{\alpha,w,y}(0)$ differs from the left hand side of \eqref{delbareqnforuzeroHF} in $C^\ell$ by a constant times $\left|\alpha\right|^\epsilon+\left|w\right|+\left|y\right|$ for any $\ell<\infty$.  It follows that the contribution of this region to the norm of $\F_{\alpha,w,y}(0)$ is bounded by a constant times $\left|\alpha\right|^\epsilon+\left|w\right|+\left|y\right|$.

We estimate $\F_{\alpha,w,y}(0)$ over the bubble ends.  The reasoning above applies as written to imply that $\F_{\alpha,w,y}(0)$ is bounded in $C^\ell$ (with respect to the usual metric on $\tilde C_\alpha$) by a constant times $\left|\alpha\right|^\epsilon+\left|w\right|+\left|y\right|$ for any $\ell<\infty$.  Since $\delta<1$, the weighted Sobolev norm over the bubble ends is also bounded by a constant times $\left|\alpha\right|^\epsilon+\left|w\right|+\left|y\right|$.

We estimate $\F_{\alpha,w,y}(0)$ over the bubble necks $S^1\times[0,6S]$.  The argument for the bubble ends applies to show that the contribution outside $S^1\times[S-1,5S+1]$ is bounded by a constant times $\left|\alpha\right|^\epsilon+\left|w\right|+\left|y\right|$.  Over $S^1\times([S-1,S]\cup[5S,5S+1])$, we bound the expression termwise: both $u_\alpha$ and $A_y$ are $O(e^{-s})$ in all derivatives, so the contribution of this region is bounded by $\left|\alpha\right|^\epsilon$.  Over $S^1\times[S,5S]$, only the term involving $X_H$ is nonzero, and since $A_y=O(e^{-s})$ in all derivatives, its contribution to the norm is bounded by a constant times $\left|\alpha\right|^\epsilon$ since $\delta<1$.
\end{proof}

\begin{lemma}[Estimate for kernel pregluing]\label{pregluingKestimateHF}
For all $\kappa\in K$, we have:
\begin{equation}
\left\|D_{\alpha,w,y}\kappa_\alpha\right\|_{k-1,2,\delta,\delta}\leq c\cdot\bigl(\left|\alpha\right|^\epsilon+\left|w\right|+\left|y\right|\bigr)\left\|\kappa\right\|
\end{equation}
uniformly over $(\alpha,w,y)$ in a neighborhood of zero, for $c<\infty$ and $\epsilon>0$ depending on data which has been previously fixed.
\end{lemma}

\begin{proof}
Recall that $D_{0,0,0}\kappa=0$; we will use this to estimate $D_{\alpha,w,y}\kappa_\alpha$ via the explicit expression for $D_{\alpha,w,y}$ from Lemma \ref{linearizedformulaHF}.

We estimate $D_{\alpha,w,y}\kappa_\alpha$ over the main ends/necks.  Over this region, the $\lambda$ terms vanish, and $J,H$ are the same for $(\alpha,w,y)$ as they are for $(0,0,0)$.  Thus $D_{\alpha,w,y}\kappa_\alpha$ is supported inside $S^1\times([S-1,S+1]\cup[5S-1,5S+1])$ in the main necks, and vanishes in the main ends.  From the exponential decay estimates \eqref{uzeromain}--\eqref{kappamain}, we obtain that the contribution to the norm of $D_{\alpha,w,y}\kappa_\alpha$ over the main ends/necks is bounded by a constant times $\left|\alpha\right|^\epsilon\left\|\kappa\right\|$.

We estimate $D_{\alpha,w,y}\kappa_\alpha$ away from the ends/necks.  As in the proof of Lemma \ref{mappregluingissmallHF}, the difference between $J,H$ for $(\alpha,w,y)$ and for $(0,0,0)$ is bounded in $C^\ell$ by a constant times $\left|\alpha\right|^\epsilon+\left|w\right|+\left|y\right|$ for any $\ell<\infty$; similarly for the distance between $x^i$ and $x^i_0$.  It thus follows from the explicit form in Lemma \ref{linearizedformulaHF} that the contribution to the norm of $D_{\alpha,w,y}\kappa_\alpha$ over this region is bounded by a constant times $(\left|\alpha\right|^\epsilon+\left|w\right|+\left|y\right|)\left\|\kappa\right\|$.

Over the bubble ends, the same reasoning applies, and we obtain the desired bound since $\delta<1$.

We estimate $D_{\alpha,w,y}\kappa_\alpha$ over the bubble necks.  Outside $S^1\times[S-1,5S+1]$, the reasoning for the bubble ends applies to show that the contribution is bounded as desired.  Over $S^1\times([S-1,S+1]\cup[5S-1,5S+1])$, the exponential decay estimates on $u_0$ and $\kappa$, along with similar estimates on $A_y$, show that the contribution of this region is bounded by $\left|\alpha\right|^\epsilon\left\|\kappa\right\|$ for some $\epsilon>0$ since $\delta<1$.  Over $S^1\times[S+1,5S-1]$, only the term involving $X_H$ is nonzero, and its contribution is bounded by a constant times $\left|\alpha\right|^\epsilon\left\|\kappa\right\|$ since $\delta<1$ and $A_y$ decays like $O(e^{-s})$ in all derivatives.
\end{proof}

\subsection{Approximate right inverse}

Recall that by assumption, the linearized operator:
\begin{equation}
D_{0,0,0}:W^{k,2,\delta,\delta}(C_0,u_0^\ast TM)_{D,H}\oplus E\to W^{k-1,2,\delta,\delta}(\tilde C_0,\Omega^{0,1}_{\tilde C_0,j_0}\otimes_\CC u_0^\ast TM_{J(\ell_{0,0}(A_0(\cdot)))})
\end{equation}
is surjective (even if we replace $E$ with $E'$).  We now proceed to fix a bounded right inverse:
\begin{equation}\label{Qzerozerozero}
Q_{0,0,0}:W^{k-1,2,\delta,\delta}(\tilde C_0,\Omega^{0,1}_{\tilde C_0,j_0}\otimes_\CC u_0^\ast TM_{J(\ell_{0,0}(A_0(\cdot)))})\to W^{k,2,\delta,\delta}(C_0,u_0^\ast TM)_{D,H}\oplus E'
\end{equation}
whose image admits a simple description.  Fix a collection of points $z_i\in C_0$ ($1\leq i\leq h$) not nodes and not contained in any of the ends, subspaces $V_i\subseteq T_{u_0(z_i)}M$, and a subspace $E''\subseteq E'$ so that the natural evaluation map:
\begin{equation}
L_0:K\xrightarrow\sim\biggl(\bigoplus_{i=1}^hT_{u_0(z_i)}M/V_i\biggr)\oplus E/E''
\end{equation}
is an isomorphism (such choices exist since $K\twoheadrightarrow E/E'$ is surjective and we may shrink the ends without loss of generality).  Now we can consider the same evaluation map on the larger space:
\begin{equation}\label{lineartofixinverseHF}
L_0:W^{k,2,\delta,\delta}(C_0,u_0^\ast TM)_{D,H}\oplus E\to W:=\biggl(\bigoplus_{i=1}^hT_{u_0(z_i)}M/V_i\biggr)\oplus E/E''
\end{equation}
Since $L_0$ sends $K=\ker D_{0,0,0}$ isomorphically to $W$, it follows that the restriction of $D_{0,0,0}$ to $\ker L_0$ is an isomorphism of Banach spaces.  Hence there is a unique right inverse:
\begin{equation}
Q_{0,0,0}:W^{k-1,2,\delta,\delta}(\tilde C_0,\Omega^{0,1}_{\tilde C_0,j_0}\otimes_\CC u_0^\ast TM_{J(\ell_{0,0}(A_0(\cdot)))})\to W^{k,2,\delta,\delta}(C_0,u_0^\ast TM)_{D,H}\oplus E
\end{equation}
with image $\ker L_0$, and it is bounded.  Since $E''\subseteq E'$, $\im Q_{0,0,0}=\ker L_0$ is in fact contained in the right hand side of \eqref{Qzerozerozero}.  We fix once and for all this $Q_{0,0,0}$.

\begin{definition}[Approximate right inverse $T_{\alpha,w,y}$]
We define a map:
\begin{equation}
T_{\alpha,w,y}:W^{k-1,2,\delta,\delta}(\tilde C_\alpha,\Omega^{0,1}_{\tilde C_\alpha,j_y}\otimes_\CC u_\alpha^\ast TM_{J(\ell_{\alpha,w}(A_y(\cdot)))})\to W^{k,2,\delta,\delta}(C_\alpha,u_\alpha^\ast TM)_{D,H}\oplus E
\end{equation}
as the composition:
\begin{equation}
T_{\alpha,w,y}:={\operatorname{glue}}\circ\PT\circ Q_{0,0,0}\circ\PT\circ(\id^{1,0}_\ast\otimes\id^{1,0})^{-1}\circ\operatorname{break}
\end{equation}
of maps in the following diagram, to be defined below:
\begin{equation}\label{biginversediagramHF}
\begin{tikzcd}
W^{k,2,\delta,\delta}(C_\alpha,u_\alpha^\ast TM)_{D,H}\oplus E\ar{r}{D_{\alpha,w,y}}&W^{k-1,2,\delta,\delta}(\tilde C_\alpha,\Omega^{0,1}_{\tilde C_\alpha,j_y}\otimes_\CC u_\alpha^\ast TM_{J(\ell_{\alpha,w}(A_y(\cdot)))})\ar{d}{\mathrm{break}}\\
W^{k,2,\delta,\delta}(C_0,u_{0|\alpha}^\ast TM)_{D,H}\oplus E\ar{u}{\mathrm{glue}}\ar{r}{D_{0|\alpha,0,y}}&W^{k-1,2,\delta,\delta}(\tilde C_0,\Omega^{0,1}_{\tilde C_0,j_y}\otimes_\CC u_{0|\alpha}^\ast TM_{J(\ell_{0,0}(A_y(\cdot)))})\\
W^{k,2,\delta,\delta}(C_0,u_{0|\alpha}^\ast TM)_{D,H}\oplus E\ar[equals]{u}\ar{r}{D_{0|\alpha,0,0}}&W^{k-1,2,\delta,\delta}(\tilde C_0,\Omega^{0,1}_{\tilde C_0,j_0}\otimes_\CC u_{0|\alpha}^\ast TM_{J(\ell_{0,0}(A_0(\cdot)))})\ar[leftrightarrow]{d}{\PT}\ar[swap]{u}{{\id^{1,0}_\ast}\otimes{\id^{1,0}}}\\
W^{k,2,\delta,\delta}(C_0,u_0^\ast TM)_{D,H}\oplus E\ar[leftrightarrow]{u}{\PT}\ar[yshift=0.5ex]{r}{D_{0,0,0}}&\ar[yshift=-0.5ex]{l}{Q_{0,0,0}}W^{k-1,2,\delta,\delta}(\tilde C_0,\Omega^{0,1}_{\tilde C_0,j_0}\otimes_\CC u_0^\ast TM_{J(\ell_{0,0}(A_0(\cdot)))})
\end{tikzcd}
\end{equation}
The top and bottom horizontal maps $D_{\alpha,w,y}$ and $D_{0,0,0}$ are the linearized operators defined earlier.  The third horizontal map $D_{0|\alpha,0,0}$ is the linearized operator at the flattened map $u_{0|\alpha}$ (its definition is identical to that of $D_{0,0,0}$ except for using $u_{0|\alpha}$ in place of $u_0$).  Similarly, the second horizontal map $D_{0|\alpha,0,y}$ is the linearized operator at the flattened map $u_{0|\alpha}$, using $(j_y,A_y)$ in place of $(j_0,A_0)$.

The vertical maps $\PT$ are simply parallel transport $\PT^{\ell_{0,0}(A_0(\cdot))}$; this is valid since $\PT$ is $J$-linear.

The vertical map $\id^{1,0}_\ast\otimes\id^{1,0}$ denotes the tensor product of $\id^{1,0}:TM_{J(a)}\to TM_{J(b)}$ (the $(1,0)$-component of the identity map) and $\id^{1,0}_\ast:\Omega^{0,1}_{C_0,j_0}\to\Omega^{0,1}_{C_0,j_y}$ (the map induced by $\id^{1,0}:(TC_0,j_0)\to(TC_0,j_y)$).

We define the map:
\begin{equation}
W^{k,2,\delta,\delta}(C_0,u_{0|\alpha}^\ast TM)_{D,H}\xrightarrow{\mathrm{glue}}W^{k,2,\delta,\delta}(C_\alpha,u_\alpha^\ast TM)_{D,H}
\end{equation}
Let $\xi\in W^{k,2,\delta,\delta}(C_0,u_{0|\alpha}^\ast TM)_{D,H}$.  Away from the necks of $C_\alpha$, $\operatorname{glue}(\xi)$ is simply $\xi$.  In any particular bubble neck $S^1\times[0,6S]\subseteq C_\alpha$, we define:
\begin{equation}
\operatorname{glue}(\xi)(s,t):=\begin{cases}
\xi(s,t)&s\leq 2S\cr
\xi(n)+\chi(4S-s)\cdot[\xi(s,t)-\xi(n)]+\chi(4S-s')\cdot[\xi(s',t')-\xi(n)]&2S\leq s\leq 4S\cr
\xi(s',t')&4S\leq s
\end{cases}
\end{equation}
(noting the corresponding ends $(t,s)\in S^1\times[0,\infty)\subseteq C_0$ and $(t',s')\in S^1\times[0,\infty)\subseteq C_0$); this definition also applies over the main necks with the obvious adjustment (and no $\xi(n)$ terms).

We define the map:
\begin{equation}
W^{k-1,2,\delta,\delta}(\tilde C_\alpha,\Omega^{0,1}_{\tilde C_\alpha,j_y}\otimes_\CC u_\alpha^\ast TM_{J(\ell_{\alpha,w}(A_y(\cdot)))})\xrightarrow{\mathrm{break}}W^{k-1,2,\delta,\delta}(\tilde C_0,\Omega^{0,1}_{\tilde C_0,j_y}\otimes_\CC u_{0|\alpha}^\ast TM_{J(\ell_{0,0}(A_y(\cdot)))})
\end{equation}
as follows.  Fix a smooth function $\bar\chi:\RR\to[0,1]$ such that:
\begin{equation}
\bar\chi(x)=\begin{cases}1&x\leq-1\cr0&x\geq+1\end{cases}\qquad\bar\chi(x)+\bar\chi(-x)=1
\end{equation}
Let $\eta\in W^{k-1,2,\delta,\delta}(\tilde C_\alpha,\Omega^{0,1}_{\tilde C_\alpha,j_y}\otimes_\CC u_\alpha^\ast TM_{J(\ell_{\alpha,w}(A_y(\cdot)))})$.  Away from the ends with $\alpha\ne 0$, $\operatorname{break}(\eta)$ is given by $(\id^{1,0})^{-1}(\eta)$, where $\id^{1,0}:TM_{J(\ell_{0,0}(A_y(\cdot)))}\to TM_{J(\ell_{\alpha,w}(A_y(\cdot)))}$ denotes the $(1,0)$-component of the identity map.  In any particular bubble end $[0,\infty)\times S^1\subseteq C_0$ with $\alpha\ne 0$, we define:
\begin{equation}
\operatorname{break}(\eta)(t,s):=\begin{cases}
(\id^{1,0})^{-1}(\eta(t,s))&s\leq 3S-1\cr
\bar\chi(s-3S)\cdot(\id^{1,0})^{-1}(\eta(t,s))&3S-1\leq s\leq 3S+1\cr
0&3S+1\leq s
\end{cases}
\end{equation}
(noting the corresponding neck $[0,6S]\times S^1\subseteq C_\alpha$); this definition also applies over the main ends, with the obvious adjustment for negative main ends.
\end{definition}

Let us make the elementary observation that the definition of $L_0$ extends perfectly well to give an analogous bounded linear map:
\begin{equation}\label{newLonpregluedHF}
L_\alpha:W^{k,2,\delta,\delta}(C_\alpha,u_\alpha^\ast TM)_{D,H}\oplus E\to W
\end{equation}
Since $\im Q_{0,0,0}\subseteq\ker L_0$, it follows from the definition of $T_{\alpha,w,y}$ that $\im T_{\alpha,w,y}\subseteq\ker L_\alpha$ as well.

\begin{lemma}\label{approxbottomHF}
Let:
\begin{equation}
\begin{CD}
X@>D>>Y\cr
@AGAA@VVBV\cr
X'@>D'>>Y'
\end{CD}
\end{equation}
denote the bottom square in \eqref{biginversediagramHF}.  Then for $\xi\in X'$ and $\eta\in Y$ with $D'\xi=B\eta$, we have:
\begin{equation}
\left\|DG\xi-\eta\right\|\leq c\cdot\left|\alpha\right|^\epsilon\left\|\xi\right\|
\end{equation}
uniformly over $(\alpha,w,y)$ in a neighborhood of zero, for $c<\infty$ and $\epsilon>0$ depending on data which has been previously fixed.
\end{lemma}

\begin{proof}
In simpler terms, we bound the operator norm of the difference between the two diagonal compositions:
\begin{equation}\label{betterPTcommutingoperatornormboundHF}
\left\|\PT\circ D_{0,0,0}-D_{0|\alpha,0,0}\circ\PT\right\|\leq c\left|\alpha\right|^\epsilon
\end{equation}
(this trivially implies the claimed statement).  To show \eqref{betterPTcommutingoperatornormboundHF}, observe that the two operators only differ over the $S^1\times[S-1,\infty)$ subset of each end.  Now it follows from Lemma \ref{linearizedformulaHF} that the contribution to the operator norm over $S^1\times[S-1,\infty)$ is bounded by a constant times the $C^k$-distance between $u_0$ and $u_{0|\alpha}$ over $S^1\times[0,\infty)$.  The estimates \eqref{uzerobubble}, \eqref{uzeromain} imply that this distance is bounded by a constant times $\left|\alpha\right|^\epsilon$ (to be precise, the bubble ends contribute $\left|\alpha\right|^{1/6-\rho}$ and each main end contributes $\left|\alpha\right|^{\delta_i/6-\rho}$ for any $\rho>0$).
\end{proof}

\begin{lemma}\label{approxmiddleHF}
Let:
\begin{equation}
\begin{CD}
X@>D>>Y\cr
@AGAA@VVBV\cr
X'@>D'>>Y'
\end{CD}
\end{equation}
denote the middle square in \eqref{biginversediagramHF}.  Then for $\xi\in X'$ and $\eta\in Y$ with $D'\xi=B\eta$, we have:
\begin{equation}
\left\|DG\xi-\eta\right\|\leq c\cdot\left|y\right|\left\|\xi\right\|
\end{equation}
uniformly over $(\alpha,w,y)$ in a neighborhood of zero, for $c<\infty$ and $\epsilon>0$ depending on data which has been previously fixed.
\end{lemma}

\begin{proof}
In simpler terms, we have $\left\|D_{0|\alpha,0,y}-(\id^{1,0}_\ast\otimes\id^{1,0})\circ D_{0|\alpha,0,0}\right\|\leq c\cdot\left|y\right|$, which trivially implies the claimed statement.  To prove this, argue as follows.

We are comparing two first-order differential operators.  Appealing to their explicit form from Lemma \ref{linearizedformulaHF}, we see that their coefficients differ by a constant times $\left|y\right|$ in $C^\ell$ for any $\ell<\infty$ (measuring with respect to the cylindrical coordinates $S^1\times[0,\infty)$ in the ends).  It follows that we have the desired estimate.
\end{proof}

\begin{lemma}\label{approxtopHF}
Let:
\begin{equation}
\begin{CD}
X@>D>>Y\cr
@AGAA@VVBV\cr
X'@>D'>>Y'
\end{CD}
\end{equation}
denote the top square in \eqref{biginversediagramHF}.  Then for $\xi\in X'$ and $\eta\in Y$ with $D'\xi=B\eta$, we have:
\begin{equation}
\left\|DG\xi-\eta\right\|\leq c\cdot\bigl(\left|\alpha\right|^\epsilon+\left|w\right|\bigr)\left\|\xi\right\|
\end{equation}
uniformly over $(\alpha,w,y)$ in a neighborhood of zero, for $c<\infty$ and $\epsilon>0$ depending on data which has been previously fixed.
\end{lemma}

\begin{proof}
We consider the following diagram, which is a copy of the top square in \eqref{biginversediagramHF} with one extra vector space and some extra maps (to be defined below).
\begin{equation}\label{topdiagrambigger}
\begin{tikzcd}[row sep = tiny]
{}&W^{k-1,2,\delta,\delta}(\tilde C_\alpha,\Omega^{0,1}_{\tilde C_\alpha,j_y}\otimes_\CC u_\alpha^\ast TM_{J(\ell_{\alpha,w}(A_y(\cdot)))})\ar[shift left=0.5ex]{dd}{\mathrm{break}'}\ar[rounded corners,to path={(\tikztostart) -- ([xshift=2ex]\tikztostart.east) -- ([xshift=2ex]\tikztotarget.east) \tikztonodes -- (\tikztotarget)}]{dddd}{\mathrm{break}}\\
W^{k,2,\delta,\delta}(C_\alpha,u_\alpha^\ast TM)_{D,H}\oplus E\ar{ru}{D_{\alpha,w,y}}\\
{}&W^{k-1,2,\delta,\delta}(\tilde C_0,\Omega^{0,1}_{\tilde C_0,j_y}\otimes_\CC u_{0|\alpha}^\ast TM_{J(\ell_{\alpha,w}(A_y(\cdot)))})\ar[shift right=-0.5ex]{uu}{\mathrm{glue}}\ar[shift left=0.5ex]{dd}{(\id^{1,0})^{-1}}\\
W^{k,2,\delta,\delta}(C_0,u_{0|\alpha}^\ast TM)_{D,H}\oplus E\ar{uu}{\mathrm{glue}}\ar{rd}{D_{0|\alpha,0,y}}\ar{ru}{D_{\alpha,w,y}}\\
{}&W^{k-1,2,\delta,\delta}(\tilde C_0,\Omega^{0,1}_{\tilde C_0,j_y}\otimes_\CC u_{0|\alpha}^\ast TM_{J(\ell_{0,0}(A_y(\cdot)))})\ar[shift right=-0.5ex]{uu}{\id^{1,0}}
\end{tikzcd}
\end{equation}
Since the almost complex structure $J(\ell_{\alpha,w}(A_y(\cdot)))$ is a function of a point in $C_\alpha$, we should remark immediately on what we mean by $W^{k-1,2,\delta,\delta}(\tilde C_0,\Omega^{0,1}_{\tilde C_0,j_y}\otimes_\CC u_{0|\alpha}^\ast TM_{J(\ell_{\alpha,w}(A_y(\cdot)))})$.  Away from the ends, the curve $C_0$ is identified canonically with $C_\alpha$, and this identification extends holomorphically to a (non-injective) map from $C_0$ to $C_\alpha$ defined outside the $S^1\times[4S,\infty)$ subsets of the ends.  Thus there is a well-defined such $W^{k-1,2,\delta,\delta}$ space of sections on $C_0$ defined outside the $S^1\times[4S,\infty)$ subsets of the ends, valued in $\Omega^{0,1}_{\tilde C_0,j_y}\otimes_\CC u_{0|\alpha}^\ast TM_{J(\ell_{\alpha,w}(A_y(\cdot)))}$.  Now, for instance, the middle horizontal map $D_{\alpha,w,y}$ can be seen as giving a section in this $W^{k-1,2,\delta,\delta}$ space defined outside $S^1\times[4S,\infty)$.  In the proof below, it is convenient to use this $W^{k-1,2,\delta,\delta}$ space of sections defined outside $S^1\times[4S,\infty)$, though we must be careful that the expressions we write are well-defined.

Let us define the rest of the maps in \eqref{topdiagrambigger}.  The vertical map $\operatorname{break}$ has been factored as $(\id^{1,0})^{-1}\circ\operatorname{break}'$ in the obvious way.  Finally, let us define the map:
\begin{equation}
W^{k-1,2,\delta,\delta}(\tilde C_0,\Omega^{0,1}_{\tilde C_0,j_y}\otimes_\CC u_{0|\alpha}^\ast TM_{J(\ell_{\alpha,w}(A_y(\cdot)))})\xrightarrow{\mathrm{glue}}W^{k-1,2,\delta,\delta}(\tilde C_\alpha,\Omega^{0,1}_{\tilde C_\alpha,j_y}\otimes_\CC u_\alpha^\ast TM_{J(\ell_{\alpha,w}(A_y(\cdot)))})
\end{equation}
Let $\eta\in W^{k-1,2,\delta,\delta}(\tilde C_0,\Omega^{0,1}_{\tilde C_0,j_y}\otimes_\CC u_{0|\alpha}^\ast TM_{J(\ell_{\alpha,w}(A_y(\cdot)))})$.  Away from the necks of $C_\alpha$, $\operatorname{glue}(\eta)$ is simply $\eta$.  In any particular bubble neck $S^1\times[0,6S]\subseteq C_\alpha$, we define:
\begin{equation}
\operatorname{glue}(\eta)(t,s):=\begin{cases}
\eta(t,s)&s\leq 2S\cr
\chi(4S-s)\eta(t,s)+\chi(4S-s')\eta(t',s')&2S\leq s\leq 4S\cr
\eta(t',s')&4S\leq s
\end{cases}
\end{equation}
(this definition also applies over the main necks with the obvious adjustment).  Note that $\operatorname{glue}\circ\operatorname{break}'$ is the identity map.

Now suppose that $D_{0|\alpha,0,y}\xi=\operatorname{break}(\eta)$; we must show that:
\begin{equation}
\left\|D_{\alpha,w,y}(\operatorname{glue}(\xi))-\eta\right\|_{k-1,2,\delta,\delta}\leq c\cdot\bigl(\left|\alpha\right|^\epsilon+\left|w\right|\bigr)\left\|\xi\right\|_{k,2,\delta,\delta}
\end{equation}
Using the triangle inequality and the fact that $\eta=\operatorname{glue}(\operatorname{break}'(\eta))=\operatorname{glue}(\id^{1,0}(D_{0|\alpha,0,y}\xi))$, we conclude that $\left\|D_{\alpha,w,y}(\operatorname{glue}(\xi))-\eta\right\|$ is bounded above by:
\begin{equation*}
\left\|D_{\alpha,w,y}(\operatorname{glue}(\xi))-\operatorname{glue}(D_{\alpha,w,y}(\xi))\right\|+\left\|\operatorname{glue}\left[D_{\alpha,w,y}(\xi)-\id^{1,0}(D_{0|\alpha,0,y}(\xi))\right]\right\|
\end{equation*}
We now estimate each term separately.

To estimate $\left\|D_{\alpha,w,y}(\operatorname{glue}(\xi))-\operatorname{glue}(D_{\alpha,w,y}(\xi))\right\|$, argue as follows.  The difference is only nonzero over the regions $S^1\times([2S,2S+1]\cup[4S-1,4S])$ of each neck.  Now the norm is bounded by $\left\|\xi\right\|e^{-2S\delta}$ (calculation left to the reader), where the factor of $e^{-2S\delta}$ comes as the ratio of the weight given to $S^1\times[4S-1,4S]$ inside a neck vs inside an end; this gives the desired bound since $\delta>0$.

To estimate $\left\|\operatorname{glue}\left[D_{\alpha,w,y}(\xi)-\id^{1,0}(D_{0|\alpha,0,y}\xi)\right]\right\|$, argue as follows.  This is bounded by a constant times the $(k-1,2,\delta,\delta)$-norm of $D_{\alpha,w,y}(\xi)-\id^{1,0}(D_{0|\alpha,0,y}\xi)$ over the complement of the subsets $S^1\times[4S,\infty)$ of the ends.  Now this is bounded by $(\left|\alpha\right|^\epsilon+\left|w\right|)\left\|\xi\right\|_{k,2,\delta,\delta}$ using the reasoning from Lemma \ref{approxmiddleHF}.
\end{proof}

\begin{proposition}[Approximate right inverse $T_{\alpha,w,y}$]\label{approxworksHF}
We have:
\begin{align}
\left\|T_{\alpha,w,y}\right\|&\leq c\\
\label{TapproxeqnHF}\left\|D_{\alpha,w,y} T_{\alpha,w,y}-\mathbf 1\right\|&\to 0\\
\im T_{\alpha,w,y}&\subseteq\ker L_\alpha
\end{align}
as $(\alpha,w,y)\to 0$, for $c<\infty$ depending on data which has been previously fixed.
\end{proposition}

\begin{proof}
It is easy to see that all the maps in \eqref{biginversediagramHF} are uniformly bounded.  Hence $\left\|T_{\alpha,w,y}\right\|\leq c$ as $(\alpha,w,y)\to 0$.  Now Lemma \ref{bootstrappingapprox} combined with Lemmas \ref{approxbottomHF}, \ref{approxmiddleHF}, \ref{approxtopHF} show that for $(\alpha,w,y)\to 0$, we have $\left\|D_{\alpha,w,y} T_{\alpha,w,y}-\mathbf 1\right\|\to 0$.  We observed earlier that $\im T_{\alpha,w,y}\subseteq\ker L_\alpha$.
\end{proof}

\begin{definition}[Right inverse $Q_{\alpha,w,y}$]
We define a map:
\begin{equation}
Q_{\alpha,w,y}:W^{k-1,2,\delta,\delta}(\tilde C_\alpha,\Omega^{0,1}_{\tilde C_\alpha,j_y}\otimes_\CC u_\alpha^\ast TM_{J(\ell_{\alpha,w}(A_y(\cdot)))})\to W^{k,2,\delta,\delta}(C_\alpha,u_\alpha^\ast TM)_{D,H}\oplus E
\end{equation}
as the sum:
\begin{equation}
Q_{\alpha,w,y}:=T_{\alpha,w,y}\sum_{k=0}^\infty(\mathbf 1-D_{\alpha,w,y}T_{\alpha,w,y})^k
\end{equation}
\end{definition}

\begin{proposition}\label{QestimatesHF}
We have:
\begin{align}
\left\|Q_{\alpha,w,y}\right\|&\leq c\\
D_{\alpha,w,y}Q_{\alpha,w,y}&=\mathbf 1\\
\im Q_{\alpha,w,y}&\subseteq\ker L_\alpha
\end{align}
uniformly over $(\alpha,w,y)$ in a neighborhood of zero, for $c<\infty$ depending on data which has been previously fixed.
\end{proposition}

\begin{proof}
Apply Lemma \ref{bootstrapinverse} and Proposition \ref{approxworksHF}.
\end{proof}

\subsection{Quadratic estimates}

\begin{proposition}[Quadratic estimate]\label{quadestimateHF}
There exist $c'>0$ and $c<\infty$ (depending on data which has been previously fixed) such that for $\left\|\xi_1\right\|_{k,2,\delta,\delta},\left\|\xi_2\right\|_{k,2,\delta,\delta}\leq c'$, we have:
\begin{equation}\label{quadestgoalHF}
\bigl\|D_{\alpha,w,y}(\xi_1-\xi_2)-(\F_{\alpha,w,y}\xi_1-\F_{\alpha,w,y}\xi_2)\bigr\|_{k-1,2,\delta,\delta}\leq c\cdot\left\|\xi_1-\xi_2\right\|_{k,2,\delta,\delta}(\left\|\xi_1\right\|_{k,2,\delta,\delta}+\left\|\xi_2\right\|_{k,2,\delta,\delta})
\end{equation}
(and $\F_{\alpha,w,y}\xi_1$ and $\F_{\alpha,w,y}\xi_2$ are both defined), uniformly over $(\alpha,w,y)$ in a neighborhood of zero.
\end{proposition}

\begin{proof}
The proof is identical to the proof of Lemma \ref{quadestimate}, with the evident notational differences, $\F_{\alpha,w,y}$ for $\F_{\alpha,y}$, and Lemma \ref{linearizedformulaHF} for Lemma \ref{linearizedformula}.  As with Lemma \ref{quadestimate}, we actually prove the stronger statement that:
\begin{equation}\label{derivislipHF}
\left\|\F_{\alpha,w,y}'(0,\xi)-\F_{\alpha,w,y}'(\zeta,\xi)\right\|_{k-1,2,\delta,\delta}\leq c\cdot\left\|\zeta\right\|_{k,2,\delta,\delta}\left\|\xi\right\|_{k,2,\delta,\delta}
\end{equation}
\end{proof}

\subsection{Newton--Picard iteration}

\begin{lemma}\label{contractionpropHF}
There exists $c'>0$ (depending on data which has been previously fixed) such that for sufficiently small $(\alpha,w,y)$:
\begin{rlist}
\item The map $\F_{\alpha,w,y}$ is defined for $\left\|\xi\right\|_{k,2,\delta,\delta}\leq c'$.
\item For $\xi_1-\xi_2\in\im Q_{\alpha,w,y}$ and $\left\|\xi_1\right\|_{k,2,\delta,\delta},\left\|\xi_2\right\|_{k,2,\delta,\delta}\leq c'$, we have:
\begin{equation}
\left\|(\xi_1-\xi_2)-(Q_{\alpha,w,y}\F_{\alpha,w,y}\xi_1-Q_{\alpha,w,y}\F_{\alpha,w,y}\xi_2)\right\|_{k,2,\delta,\delta}\leq\frac 12\left\|\xi_1-\xi_2\right\|_{k,2,\delta,\delta}
\end{equation}
\end{rlist}
\end{lemma}

\begin{proof}
The proof is identical to the proof of Lemma \ref{contractionprop}, using Proposition \ref{quadestimateHF} in place of Proposition \ref{quadestimate}.
\end{proof}

\begin{proposition}[Newton--Picard iteration]\label{newtoniterationHF}
There exists $c'>0$ (depending on data which has been previously fixed) so that for $(\alpha,w,y,\kappa\in K)$ sufficiently small, there exists a unique $\kappa_{\alpha,w,y}\in W^{k,2,\delta,\delta}(C_\alpha,u_\alpha^\ast TM)_{D,H}\oplus E$ satisfying:
\begin{align}
\kappa_{\alpha,w,y}&\in\kappa_\alpha+\im Q_{\alpha,w,y}\\
\left\|\kappa_{\alpha,w,y}\right\|_{k,2,\delta,\delta}&\leq c'\\
\F_{\alpha,w,y}\kappa_{\alpha,w,y}&=0
\end{align}
\end{proposition}

\begin{proof}
The proof is essentially identical to the proof of Proposition \ref{newtoniteration}; we write it out anyway.

In fact, we will show that $\kappa_{\alpha,w,y}$ is given explicitly as the limit of the Newton iteration:
\begin{align}
\xi_0&:=\kappa_\alpha\\
\xi_n&:=\xi_{n-1}-Q_{\alpha,w,y}\F_{\alpha,w,y}\xi_{n-1}
\end{align}
By Lemma \ref{contractionpropHF}, the map $\xi\mapsto\xi-Q_{\alpha,w,y}\F_{\alpha,w,y}\xi$ is a $\frac 12$-contraction mapping when restricted to:
\begin{equation}\label{domainofcontractionHF}
\{\xi\in\kappa_\alpha+\im Q_{\alpha,w,y}:\left\|\xi\right\|_{k,2,\delta,\delta}\leq c'\}
\end{equation}
To finish the proof, it suffices to show that (for sufficiently small $(\alpha,w,y,\kappa)$) \eqref{domainofcontractionHF} is nonempty and is mapped to itself by $\xi\mapsto\xi-Q_{\alpha,w,y}\F_{\alpha,w,y}\xi$.

We know that $\left\|\kappa_\alpha\right\|_{k,2,\delta,\delta}\to 0$ as $\kappa\to 0$ (uniformly in $(\alpha,w,y)$), so \eqref{domainofcontractionHF} is nonempty.  By using Proposition \ref{quadestimateHF} with $(\xi_1,\xi_2)=(0,\kappa_\alpha)$ and Lemmas \ref{mappregluingissmallHF} and \ref{pregluingKestimateHF}, we conclude that:
\begin{equation}\label{kernelclosetoholomorphicHF}
\left\|\F_{\alpha,w,y}\kappa_\alpha\right\|_{k-1,2,\delta,\delta}\to 0
\end{equation}
as $(\alpha,w,y,\kappa)\to 0$.  Since the operator norm of $Q_{\alpha,w,y}$ is bounded uniformly as $(\alpha,w,y)\to 0$, we see that $\kappa_\alpha$ is almost fixed by $\xi\mapsto\xi-Q_{\alpha,w,y}\F_{\alpha,w,y}\xi$ as $(\alpha,w,y,\kappa)\to 0$.  It then follows from the contraction property that $\xi\mapsto\xi-Q_{\alpha,w,y}\F_{\alpha,w,y}\xi$ maps \eqref{domainofcontractionHF} to itself.
\end{proof}

\subsection{Gluing}

\begin{definition}[Gluing map]
We define:
\begin{align}
u_{\alpha,w,y,\kappa}&:=\exp_{u_\alpha}\kappa_{\alpha,w,y}\\
e_{\alpha,w,y,\kappa}&:=e_0+\proj_E\kappa_{\alpha,w,y}
\end{align}
where $\kappa_{\alpha,w,y}$ is as in Proposition \ref{newtoniterationHF}, and we consider the following \emph{gluing map}:
\begin{align}\label{gluingmapHF}
\CC^d\times(\RR_{\geq 0}^{k-1})^{\leq\s}\times\RR^\ast\times\prod_{i=1}^{k\prime}(H^i)^{\leq\s}\times K&\to\Mbar(M)\\
(\alpha,w,y,\kappa)&\mapsto(\alpha,w,y,u_{\alpha,w,y,\kappa},e_{\alpha,w,y,\kappa})
\end{align}
($\{x^i\}_{1\leq i\leq r}$ are determined uniquely by $\alpha,w,y,u_{\alpha,w,y,\kappa}$, so we omit them from the notation).  It follows from the definition that the gluing map commutes with the projection from both sides to $\CC^d\times(\RR_{\geq 0}^{k-1})^{\leq\s}\times\RR^\ast\times\prod_{i=1}^{k\prime}(H^i)^{\leq\s}\times E/E'$.
\end{definition}

\begin{lemma}\label{mapsstufftoregHF}
The gluing map \eqref{gluingmapHF} maps sufficiently small $(\alpha,w,y,\kappa)$ to $\Mbar(M)^\reg$.
\end{lemma}

\begin{proof}
This is true since $Q_{\alpha,w,y}$ gives an approximate right inverse at $(u_{\alpha,w,y,\kappa},e_{\alpha,w,y,\kappa})$ (use \eqref{derivislipHF} with $\zeta=\kappa_{\alpha,w,y}$).
\end{proof}

Let $K_\alpha\subseteq C^\infty(C_\alpha,u_\alpha^\ast TM)_{D,H}\oplus E$ denote the image of $\kappa\mapsto\kappa_\alpha$.  It is clear by definition that $K\to K_\alpha$ is an isomorphism and that the respective $W^{k,2,\delta,\delta}$ norms are uniformly commensurable.  It is also clear that the following commutes:
\begin{equation}\label{LonkernelisgoodHF}
\begin{tikzcd}[column sep = tiny]
K\ar{dr}[swap]{L_0}\ar{rr}{\kappa\mapsto\kappa_\alpha}&&K_\alpha\ar{dl}{L_\alpha}\\
&W
\end{tikzcd}
\end{equation}
(all maps being isomorphisms).  Since $\im Q_{\alpha,w,y}\subseteq\ker L_\alpha$, it follows in particular that $\im Q_{\alpha,w,y}\cap K_\alpha=0$.  On the other hand, an index calculation shows that $\ind D_{\alpha,w,y}=\ind D_{0,0,0}$ (note that by the argument for Lemma \ref{fredholmandsamekernelcokernel}, it suffices to calculate their indices as operators $W^{k,2,\delta}\to W^{k-1,2,\delta}$, i.e.\ with no weights at the bubble nodes, on $\tilde C_\alpha$ and $\tilde C_0$ respectively; the calculation of such indices is originally due to Floer \cite{floersfhs,floerindex} and is by now standard).  Both are surjective, and hence we have $\dim\coker Q_{\alpha,w,y}=\dim\ker D_{\alpha,w,y}=\dim\ker D_{0,0,0}=\dim K=\dim K_\alpha$.  It follows that $\im Q_{\alpha,w,y}=\ker L_\alpha$ and that:
\begin{equation}
\im Q_{\alpha,w,y}\oplus K_\alpha\xrightarrow\sim W^{k,2,\delta,\delta}(C_\alpha,u_\alpha^\ast TX)_D\oplus E
\end{equation}
is an isomorphism of Banach spaces (an alternative justification for this is to use an argument similar to the proof of Propositions \ref{gluingsurjective}/\ref{gluingsurjectiveHF} to show that the natural projection $K_\alpha\xrightarrow{\kappa_\alpha\mapsto\kappa_\alpha-Q_{\alpha,w,y}D_{\alpha,w,y}\kappa_\alpha}\ker D_{\alpha,w,y}$ is surjective, and thus bijective).  We claim that in fact the two norms are uniformly commensurable as $(\alpha,w,y)\to 0$.  The map written is clearly uniformly bounded, so we just need to show the same for its inverse.  It suffices to show that the projection from the right hand side to $K_\alpha$ is uniformly bounded, but this is nothing other than $L_\alpha$ (clearly uniformly bounded) composed with the inverse of the isomorphism in \eqref{LonkernelisgoodHF} (also uniformly bounded).

\begin{lemma}
The gluing map \eqref{gluingmapHF} is injective in a neighborhood of zero.
\end{lemma}

\begin{proof}
The proof is identical to the proof of Lemma \ref{gluingisinjective}, with the obvious notational differences.
\end{proof}

\begin{proposition}\label{continuityofgluingHF}
The gluing map \eqref{gluingmapHF} is continuous in a neighborhood of zero.
\end{proposition}

\begin{proof}
The proof follows the same basic outline as the proof of Proposition \ref{continuityofgluing}; some parts of the proof are identical, and we will omit these.  The key ingredient is our precise control of the image of the right inverse $Q_{\alpha,w,y}$ (specifically, that $\im Q_{\alpha,w,y}=\ker L_\alpha$).

Suppose $(\alpha_i,w_i,y_i,\kappa_i)\to(\alpha,w,y,\kappa)$ is a convergent net.\footnote{We could restrict to sequences rather than nets since $\CC^d\times(\RR_{\geq 0}^{k-1})^{\leq\s}\times\RR^\ast\times\prod_{i=1}^{k\prime}(H^i)^{\leq\s}\times K$ is first countable.  However, this would not make the argument any simpler.}  We will show that:
\begin{equation}
(u_{\alpha_i,w_i,y_i,\kappa_i},e_{\alpha_i,w_i,y_i,\kappa_i})\to(u_{\alpha,w,y,\kappa},e_{\alpha,w,y,\kappa})
\end{equation}
First, observe that the argument from the proof of Proposition \ref{continuityofgluing} applies verbatim to give that in fact it suffices to show that:
\begin{equation}
(u_{\alpha_i,w_i,y_i,\kappa},e_{\alpha_i,w_i,y_i,\kappa})\to(u_{\alpha,w,y,\kappa},e_{\alpha,w,y,\kappa})
\end{equation}
Now recall that by definition:
\begin{align}
u_{\alpha,w,y,\kappa}&=\exp_{u_\alpha}\kappa_{\alpha,w,y}&\kappa_{\alpha,w,y}&=\kappa_\alpha+\xi&\text{for some }\xi&\in\im Q_{\alpha,w,y}
\end{align}
Now we define $\xi_{\alpha_i}\in W^{k,2,\delta,\delta}(C_{\alpha_i},u_{\alpha_i}^\ast TX)_D\oplus E$ by ``pregluing'' $\xi$ from $C_\alpha$ to $C_{\alpha_i}$ as follows.  Note that we may assume without loss of generality that at the nodes where $\alpha\ne 0$, we also have $\alpha_i\ne 0$.  Away from the ends/necks of $C_{\alpha_i}$, we set $\xi_{\alpha_i}=\xi$.  Note that for every end of $C_{\alpha_i}$, there is a corresponding end of $C_\alpha$, so we may also simply set $\xi_{\alpha_i}=\xi$ over the ends of $C_{\alpha_i}$.  Over the necks of $C_{\alpha_i}$ for which $\alpha=0$, we define $\xi_{\alpha_i}$ via \eqref{kernelpregluingformulaHFbubble}--\eqref{kernelpregluingformulaHFmain} (note that this is reasonable since $\xi$ satisfies the estimates \eqref{kappabubble}, \eqref{kappamain} as a consequence of Proposition \ref{neckestimateHF}).  Over the necks of $C_{\alpha_i}$ for which $\alpha\ne 0$, we define $\xi_{\alpha_i}$ as:
\begin{equation}
\xi_{\alpha_i}(s,t):=\PT_{u_{\alpha}(t,f_i(s))\to u_{\alpha_i}(t,s)}[\xi(t,f_i(s)))]
\end{equation}
where $f_i:[0,6S_i]\to[0,6S]$ is smooth and satisfies:
\begin{align}
f_i(s):=&\begin{cases}
s&s\leq S-2\cr
s-3S_i+3S&3S_i-2S+2\leq s\leq 3S_i+2S-2\cr
s-6S_i+6S&6S_i-S+2\leq s\cr
\end{cases}\\
f_i([S-2,3S_i-2S+2])&\subseteq[S-2,S+2]\\
f_i([3S_i+2S-2,6S_i-S+2])&\subseteq[5S-2,5S+2]
\end{align}
so that $f_i\to\id$ uniformly in all derivatives as $S_i\to S$.  More informally, $f_i$ is smooth and matches up $[0,S-2],[S+2,5S-2],[5S+2,6S]\subseteq[0,6S]$ with corresponding intervals of the same length inside $[0,6S_i]$, symmetrically.

Now at this point we appeal to the corresponding arguments from the proof of Proposition \ref{continuityofgluing}, which apply as written (using the fact that $\im Q_{\alpha,w,y}=\ker L_\alpha$) and imply that it suffices to show that:
\begin{equation}\label{pregluedcorrectionalmostholHF}
\left\|\F_{\alpha_i,w_i,y_i}(\kappa_{\alpha_i}+\xi_{\alpha_i})\right\|_{k-1,2,\delta,\delta}\to 0
\end{equation}
Recall that $\F_{\alpha,w,y}(\kappa_\alpha+\xi)=0$; we will use this to estimate $\F_{\alpha_i,w_i,y_i}(\kappa_{\alpha_i}+\xi_{\alpha_i})$.

Away from the ends/necks of $C_{\alpha_i}$, it is clear by definition that $\F_{\alpha_i,w_i,y_i}(\kappa_{\alpha_i}+\xi_{\alpha_i})\to\F_{\alpha,w,y}(\kappa_\alpha+\xi)$ uniformly in all derivatives.  Thus the contribution to the norm of $\F_{\alpha_i,w_i,y_i}(\kappa_{\alpha_i}+\xi_{\alpha_i})$ over this region approaches zero.

Over the necks of $C_{\alpha_i}$ which correspond to necks of $C_\alpha$, we again have uniform convergence in all derivatives, so the contribution of these regions approaches zero as well.

We estimate the contribution over the ends of $C_{\alpha_i}$ (recall that these necessarily correspond to ends of $C_\alpha$).  Over main ends, we have $\F_{\alpha_i,w_i,y_i}(\kappa_{\alpha_i}+\xi_{\alpha_i})=\F_{\alpha,w,y}(\kappa_\alpha+\xi)$.  Now the convergence $\F_{\alpha_i,w_i,y_i}(\kappa_{\alpha_i}+\xi_{\alpha_i})\to\F_{\alpha,w,y}(\kappa_\alpha+\xi)$ uniformly in all derivatives is valid near the bubble nodes in the usual metric on $C_{\alpha_i}=C_\alpha$.  It follows that we also have convergence in the relevant $\delta$-weighted Sobolev norm since $\delta<1$.  Thus the contribution of this region to the norm approaches zero.

Finally, let us estimate the contribution to the norm over the necks of $C_{\alpha_i}$ which correspond to ends of $C_\alpha$.  We treat main necks and bubble necks separately.  Over main necks, $\F_{\alpha_i,w_i,y_i}(\kappa_{\alpha_i}+\xi_{\alpha_i})$ is supported inside $S^1\times([S-1,S+1],[5S-1,5S+1])$.  The contribution of this region to its norm goes to zero, as follows using the exponential decrease on $\kappa$ and $\xi$ from Proposition \ref{neckestimateHF} and the fact that $\delta<\delta_i$.  Over bubble ends, $\F_{\alpha_i,w_i,y_i}(\kappa_{\alpha_i}+\xi_{\alpha_i})$ converges to $\F_{\alpha,w,y}(\kappa_\alpha+\xi)$ uniformly in all derivatives over the complement of $S^1\times[S-1,5S+1]$ (measured with respect to the usual metric on $C_0$), and hence the contribution of this region to the weighted norm goes to zero since $\delta<1$.  Over $S^1\times([S-1,S+1]\cup[5S-1,5S+1])$, the smoothness of $\kappa$ and $\xi$ over $C_\alpha$ and the fact that $\delta<1$ shows that the contribution of this region to the norm goes to zero (since $S\to\infty$).  Finally, over $S^1\times[S+1,5S-1]$, both $u_{\alpha_i}$ and $\kappa_{\alpha_i}+\xi_{\alpha_i}$ are constant, so $\F_{\alpha_i,w_i,y_i}(\kappa_{\alpha_i}+\xi_{\alpha_i})$ is simply:
\begin{equation*}
\PT_{\exp_{u_0(n)}(\kappa(n)+\xi(n))\to u_0(n)}^{\ell_{\alpha_i,w_i}(A_{y_i}(\cdot))}\Bigl(2d(\proj_{S^1}A_{y_i})\otimes X_{H((\ell_{\alpha_i,w_i}\times\id_{S^1})(A_{y_i}(\cdot)))}\Bigr)^{0,1}_{j,J(\ell_{\alpha_i,w_i}(A_{y_i}(\cdot)))}
\end{equation*}
The norm of this expression over $S^1\times[S+1,5S-1]$ approaches zero since $A_{y_i}=O(e^{-s})$ in all derivatives and $\delta<1$.
\end{proof}

\subsection{Surjectivity of gluing}

We recall a well-known principle of exponential decay for Floer trajectories converging to non-degenerate periodic orbits; this has appeared in many forms in the literature, for example in Floer \cite[pp801--802]{floer88}.

\begin{proposition}[A priori estimate on decay of connecting cylinders]\label{neckestimateHF}
Let $(M,\omega)$ be a symplectic manifold, $J$ an $\omega$-compatible almost complex structure on $M$, $H:M\times S^1\to\RR$ a Hamiltonian, and $\gamma:S^1\to M$ a non-degenerate periodic orbit of $H$.  We consider the (unbounded) asymptotic linearized operator:
\begin{align}
L^2(S^1,\gamma^\ast TM)&\to L^2(S^1,\gamma^\ast TM)\\
\xi&\mapsto J\mathcal L_{X_H}\xi
\end{align}
where $\mathcal L_{X_H}$ denotes the symplectic connection on $\gamma^\ast TM$ induced by $X_H$, and we use the standard inner product $g(\cdot,\cdot):=\omega(\cdot,J\cdot)$ on $\gamma^\ast TM$ for the inner product on $L^2(S^1,\gamma^\ast TM)$.  This operator is self-adjoint; we denote by $\delta>0$ the smallest magnitude of any of its eigenvalues ($\delta$ is positive since the orbit is non-degenerate).

We consider (partially defined) sections $\xi:S^1\times\RR\to\gamma^\ast TM$ with $\left|\xi(t,s)\right|<\epsilon$ such that $u:=\exp_\gamma\xi:S^1\times\RR\to M$ satisfies $(du+2dt\otimes X_H(u))^{0,1}_J=0$.  Now for all $\mu<1$, there exists $\epsilon>0$ such that we have the following estimates.

Suppose that $\xi$ as above is defined on $S^1\times[0,\infty)$.  Then:
\begin{equation}
\left|D^k\xi\right|\leq c_k\cdot e^{-\mu\delta s}\left(\int_{S^1}\bigl|\xi(t,0)\bigr|^2\,dt\right)^{1/2}\qquad s\geq 1
\end{equation}
for all $k\geq 0$.  A symmetric statement holds for $u$ and $\xi$ defined over $S^1\times(-\infty,0]$.

Suppose that $\xi$ as above is defined on $S^1\times[0,N]$.  Then:
\begin{equation*}
\left|D^k\xi\right|\leq c_k\cdot\left[e^{-\mu\delta s}\left(\int_{S^1}\bigl|\xi(t,0)\bigr|^2\,dt\right)^{1/2}+e^{-\mu\delta(N-s)}\left(\int_{S^1}\bigl|\xi(t,N)\bigr|^2\,dt\right)^{1/2}\right]\qquad 1\leq s\leq N-1
\end{equation*}
for all $k\geq 0$.
\end{proposition}

\begin{proof}
This proof is adapted from Salamon \cite[p170, Lemma 2.11]{salamonnotes}.

We have by assumption that $\partial_su+J(u)(\partial_tu-X_H(u))=0$, where $u(t,s)=\exp_{\gamma(t)}\xi(t,s)$.  Now we may rewrite this equation in exponential coordinates in terms of $\xi$ as follows.  Denote by $\partial_t$ the connection $\mathcal L_{X_H}$ on $\gamma^\ast TM$, which is given by the Lie derivative with respect to $X_H$.  Now the equation for $u$ is equivalent to:
\begin{equation}
\partial_s\xi+J\partial_t\xi+A(\xi)\partial_t\xi+Q(\xi)=0
\end{equation}
for certain smooth (non-linear) bundle maps $A:\gamma^\ast TM\to\operatorname{End}(\gamma^\ast TM)$ and $Q:\gamma^\ast TM\to\gamma^\ast TM$.  We have $A(0)=0$, $Q(0)=0$, and $Q'(0,\cdot)=0$ (we denote by $A'$ and $Q'$ their ``vertical'' derivatives).

Now we let:
\begin{equation}
f(s):=\int_{S^1}\left|\xi(t,s)\right|^2dt=\left\|\xi\right\|_2^2
\end{equation}
and we will show that $f''(s)\geq(1-o(1))4\delta^2f(s)$, where $o(1)$ denotes a quantity which can be made arbitrarily small by choosing $\epsilon>0$ sufficiently small.  We have:
\begin{equation}\label{fppexpr}
f''(s)=2\int_{S^1}\left|\partial_s\xi\right|^2dt+2\int_{S^1}\langle\xi,\partial_s\partial_s\xi\rangle\,dt=2\left\|\partial_s\xi\right\|_2^2+2\langle\xi,\partial_s\partial_s\xi\rangle
\end{equation}
and we will now prove estimates $\left\|\partial_s\xi\right\|_2^2\geq(1-o(1))\delta^2f(s)$ and $\langle\xi,\partial_s\partial_s\xi\rangle\geq(1-o(1))\delta^2f(s)$.  Recall that $J\partial_t$ is self-adjoint and that $\left\|J\partial_t\xi\right\|_2\geq\delta\left\|\xi\right\|_2$.

To bound $\left\|\partial_s\xi\right\|_2^2$, we write:
\begin{align}
\left\|\partial_s\xi\right\|_2&=\left\|J\partial_t\xi+A(\xi)\partial_t\xi+Q(\xi)\right\|_2\cr
&\geq\left\|J\partial_t\xi\right\|_2-\left\|A(\xi)\partial_t\xi\right\|_2-\left\|Q(\xi)\right\|_2\cr
&\geq\left\|J\partial_t\xi\right\|_2-\left\|A(\xi)\right\|_\infty\left\|\partial_t\xi\right\|_2-c\left\|\xi\right\|_\infty\left\|\xi\right\|_2\cr
&=\left\|J\partial_t\xi\right\|_2-o(1)\left\|J\partial_t\xi\right\|_2-o(1)\left\|\xi\right\|_2\geq(1-o(1))\delta\left\|\xi\right\|_2
\end{align}
and thus $\left\|\partial_s\xi\right\|_2^2\geq(1-o(1))\delta^2f(s)$.

To bound $\langle\xi,\partial_s\partial_s\xi\rangle$, we write:
\begin{align}
\langle\xi,\partial_s\partial_s\xi\rangle&=-\langle\xi,\partial_s[J\partial_t\xi+A(\xi)\partial_t\xi+Q(\xi)]\rangle\cr
&=-\langle J\partial_t\xi,\partial_s\xi\rangle-\langle\xi,A'(\xi,\partial_s\xi)\partial_t\xi\rangle-\langle\xi,A(\xi)\partial_t\partial_s\xi\rangle-\langle\xi,Q'(\xi,\partial_s\xi)\rangle
\end{align}
Now substituting in $\partial_s\xi=-J\partial_t\xi-A(\xi)\partial_t\xi-Q(\xi)$, we obtain the following:
\begin{align}
\langle\xi,\partial_s\partial_s\xi\rangle=\left\|J\partial_t\xi\right\|_2^2
&+\langle J\partial_t\xi,A(\xi)\partial_t\xi\rangle\cr
&+\langle J\partial_t\xi,Q(\xi)\rangle\cr
&+\langle\xi,A'(\xi,[J+A(\xi)]\partial_t\xi)\partial_t\xi\rangle\cr
&+\langle\xi,A'(\xi,Q(\xi))\partial_t\xi\rangle\cr
&+\langle\xi,A(\xi)\partial_t([J+A(\xi)]\partial_t\xi)\rangle\cr
&+\langle\xi,A(\xi)\partial_t(Q(\xi))\rangle\cr
&+\langle\xi,Q'(\xi,[J+A(\xi)]\partial_t\xi)\rangle\cr
&+\langle\xi,Q'(\xi,Q(\xi))\rangle
\end{align}
The first term $\left\|J\partial_t\xi\right\|_2^2$ is the main term; let us estimate the remaining error terms.  The first error term is bounded by $\left\|\xi\right\|_\infty\left\|J\partial_t\xi\right\|_2^2$, the second by $\left\|\xi\right\|_\infty\left\|\xi\right\|_2\left\|J\partial_t\xi\right\|_2$, the third by $\left\|\xi\right\|_\infty\left\|J\partial_t\xi\right\|_2^2$, the fourth by $\left\|\xi\right\|_\infty^2\left\|\xi\right\|_2\left\|J\partial_t\xi\right\|_2$, the fifth by $\left\|\xi\right\|_\infty\left\|J\partial_t\xi\right\|^2+\left\|\xi\right\|_\infty\left\|\xi\right\|_2\left\|J\partial_t\xi\right\|_2$ (integrate by parts to move the outermost $\partial_t$ onto $\xi$, $A(\xi)$, and $\langle\cdot,\cdot\rangle$), the sixth by $\left\|\xi\right\|_\infty^2\left\|\xi\right\|_2\left\|J\partial_t\xi\right\|_2$, the seventh by $\left\|\xi\right\|_\infty\left\|\xi\right\|_2\left\|J\partial_t\xi\right\|_2$, and the eighth by $\left\|\xi\right\|_\infty^2\left\|\xi\right\|_2^2$.  All of these are $o(1)\left\|J\partial_t\xi\right\|_2^2$, so we conclude that $\langle\xi,\partial_s\partial_s\xi\rangle\geq(1-o(1))\delta^2f(s)$.  Combining the above estimates, we obtain the desired inequality $f''(s)\geq(1-o(1))4\delta^2f(s)$.

Now we have the following maximum principle for the differential inequality $g''(s)\geq r^2g(s)$.  Namely, suppose that $g:[a,b]\to\RR_{\geq 0}$ satisfies $g''(s)\geq r^2g(s)$ and that $G:[a,b]\to\RR_{\geq 0}$ satisfies $G''(s)=r^2G(s)$; if $g\leq G$ at the endpoints $a,b$, then it follows that $g\leq G$ over the whole interval $[a,b]$ (one may easily derive a contradiction by assuming that, on the contrary, $g=G$ at $a_1,b_1$ and $g>G$ on $(a_1,b_1)$ for $a\leq a_1<b_1\leq b$).  Specializing to the case at hand, we have:
\begin{align*}
f(s)&\leq e^{-(1-o(1))2\delta s}f(0)&&\text{in the case of }S^1\times[0,\infty)\\
f(s)&\leq e^{-(1-o(1))2\delta s}f(0)+e^{-(1-o(1))2\delta(N-s)}f(N)&&\text{in the case of }S^1\times[0,N]
\end{align*}
Now Lemma \ref{Ltwocontrolsinnbhd} below asserts that for a $J$-holomorphic curve $u_0$ and a sufficiently small perturbation thereof $u=\exp_{u_0}\xi$ which is also $J$-holomorphic, we have uniform bounds on all derivatives of $\xi$ (away from the boundary) in terms of the $L^2$-norm of $\xi$.  Applying this where $u_0$ is the trivial cylinder over $\gamma(t)$, we conclude that our bounds on $f(s)$ imply the desired result.
\end{proof}

\begin{lemma}\label{Ltwocontrolsinnbhd}
Let $u:[0,1]\times[0,1]\to B^{2n}(1)$ be $J$-holomorphic with respect to some smooth almost complex structure $J$ on the unit ball $B^{2n}(1)\subseteq\RR^{2n}$.  There exists $\epsilon>0$ (depending only on upper bounds on $\left\|J\right\|_{C^\ell}$ and $\left\|u\right\|_{C^\ell}$ for some absolute $\ell<\infty$) with the following property.  Suppose that $\xi:[0,1]\times[0,1]\to\RR^{2n}$ with $|\xi|<\epsilon$ pointwise is such that $u+\xi$ has image contained in $B^{2n}(1)$ and is $J$-holomorphic.  Then:
\begin{equation}
\left|D^k\xi(0,0)\right|\leq c_k\biggl(\int_{[0,1]\times[0,1]}\left|\xi(x,y)\right|^2\,dx\,dy\biggr)^{1/2}
\end{equation}
for all $k<\infty$ for constants $c_k<\infty$ depending only on upper bounds on $\left\|J\right\|_{C^\ell}$ and $\left\|u\right\|_{C^\ell}$ for some $\ell=\ell(k)<\infty$.
\end{lemma}

\begin{proof}
By assumption, we have:
\begin{align}
u_x+J(u)u_y&=0\\
(u+\xi)_x+J(u+\xi)(u+\xi)_y&=0
\end{align}
which together imply that:
\begin{equation}
\xi_x+J(u+\xi)\xi_y=[J(u)-J(u+\xi)]u_y
\end{equation}
which we prefer to write as:
\begin{equation}\label{ellipticforxi}
(\partial_x+B\partial_y)\xi=A(\xi)
\end{equation}
where $B:[0,1]^2\to\operatorname{End}(\RR^{2n})$ denotes $J(u+\xi)$ and $A:[0,1]^2\times\RR^{2n}\to\RR^{2n}$ denotes $A(\xi)=[J(u)-J(u+\xi)]u_y$.  By definition, we have $A(0)=0$ and:
\begin{equation}\label{boundonA}
\left|D^kA\right|\leq c_k
\end{equation}
for some $c_k<\infty$ depending only on upper bounds on $\left\|J\right\|_{C^\ell}$ and $\left\|u\right\|_{C^\ell}$ for some $\ell=\ell(k)<\infty$.  By definition, we have $B^2=-1$; let us now argue that we also have:
\begin{equation}\label{boundonB}
\left|D^kB\right|\leq c_k
\end{equation}
for sufficiently small $\epsilon>0$ and $c_k<\infty$ as before.  Indeed, for sufficiently small $\epsilon>0$, we may apply the Gromov--Schwarz lemma to $u+\xi$ and conclude that $|D\xi|\leq c<\infty$ depending only on upper bounds on $\left\|J\right\|_{C^\ell}$ and $\left\|u\right\|_{C^\ell}$ for some $\ell<\infty$.  Now Lemma \ref{Woneinftycontrolsall} applied to $u+\xi$ gives $|D^k\xi|\leq c_k<\infty$, which implies the desired bound on the derivatives of $B$.

Now we apply $(\partial_x-B\partial_y)$ to both sides of \eqref{ellipticforxi} to obtain:
\begin{equation}
\Delta\xi=(\partial_x-B\partial_y)A(\xi)-((\partial_x-B\partial_y)B)(\xi_y)
\end{equation}
For any smooth function $\phi:[0,1]^2\to\RR_{\geq 0}$ supported in the interior of $[0,1]^2$, we thus have:
\begin{equation*}
\Delta(\phi\xi)=\xi\Delta\phi+2\phi_x\xi_x+2\phi_y\xi_y+\phi\cdot(\partial_x-B\partial_y)A(\xi)-\phi\cdot((\partial_x-B\partial_y)B)(\xi_y)
\end{equation*}
Now the desired result follows from the usual bootstrapping of the elliptic regularity estimates for the Laplacian on $\RR^2$, namely $\left\|f\right\|_{H^{s+2}}\leq c_s\left\|f+\Delta f\right\|_{H^s}$ (the first step being the case $s=-1$).  We may shrink the support of $\phi$ slightly at each step, so there is no need to worry about regularity near the boundary.
\end{proof}

\begin{proposition}\label{gluingsurjectiveHF}
The restriction of the gluing map \eqref{gluingmapHF} to any neighborhood of zero is surjective onto a neighborhood of $(0,0,0,u_0,e_0,\{x^i_0\}_{1\leq i\leq r})\in\Mbar(M)$.
\end{proposition}

\begin{proof}
The proof is identical to the proof of Proposition \ref{gluingsurjective}, except for appealing to Proposition \ref{neckestimateHF} in addition to Proposition \ref{neckestimate} in the appropriate places.
\end{proof}

\subsection{Conclusion of the proof}

We have shown that the map $g:=$\eqref{gluingmapHF} is continuous, injective, and that its restriction to any neighborhood of zero is surjective onto a neighborhood of the image of zero.  The target $\Mbar(M)$ is Hausdorff, and thus it follows from Lemma \ref{injectiveHdorffhomeo} that for some open neighborhood of zero $U\subseteq\CC^d\times(\RR_{\geq 0}^{k-1})^{\leq\s}\times\RR^\ast\times\prod_{i=1}^{k\prime}(H^i)^{\leq\s}\times K$, we have $g(U)$ is open and $g:U\xrightarrow\sim g(U)$ is a homeomorphism.  Thus the gluing map \eqref{gluingmapHF} satisfies the properties desired for the map \eqref{HFgluinggoal}.

\subsection{Gluing orientations}

We now show the how to endow the moduli spaces $\Mbar(\sigma,p,q)$ with coherent orientations using the results of Floer--Hofer \cite{floerhoferorientations}.

\subsubsection{Orientations on spaces of flow lines on $\Delta^n$}\label{simplexflowsorsec}

For any simplex $\sigma$, define the following orientation line:\footnote{An orientation line is a $\ZZ/2$-graded free $\ZZ$-module of rank one.}
\begin{equation}
\oo_\sigma:=\begin{cases}\bigotimes_{i=1}^{\dim\sigma-1}\oo_\RR&\dim\sigma>0\cr\oo_\RR^\vee&\dim\sigma=0\end{cases}
\end{equation}
Let us construct an identification between $\oo_\sigma$ and the orientation sheaf of the space of flow lines on $\sigma$ (from Definition \ref{simplexflows}) for $n=\dim\sigma>0$.  Let $f:\RR\to[0,1]$ be the unique solution to the initial value problem $f(0)=\frac 12$ and $f'(x)=\pi\sin(\pi f(x))$.  Now every flow line $\ell:\RR\to\Delta^n$ is of the form $\ell(t)=(f(t+a_1),\ldots,f(t+a_n))$ for some $a_1\leq\cdots\leq a_n$ which are unique up to the addition of an overall constant.  Thus the space of flow lines is parameterized by $(b_1,\ldots,b_{n-1})\in[0,\infty)^{n-1}$, where $b_i=a_{i+1}-a_i$; moreover, this parameterization extends continuously to a homeomorphism between $[0,\infty]^{n-1}$ and the space of broken flow lines on $\Delta^n$.  In these coordinates, $b_k=\infty$ iff the flow line is broken at vertex $k$, and $b_k=0$ iff the flow line factors through $\Delta^{[0\ldots\hat k\ldots n]}\subseteq\Delta^n$.  Now the coordinates $b_1,\ldots,b_{n-1}$ determine an identification between $\oo_\sigma=\bigotimes_{i=1}^{n-1}\oo_\RR$ and the orientation sheaf of the space of flow lines on $\sigma$, when $\dim\sigma>0$.

For $\dim\sigma=0$, let us simply remark that tensoring with $\oo_\sigma=\oo_\RR^\vee$ is the effect on orientation sheaves of quotienting by an $\RR$-action, and it will follow from this fact that this is the correct definition of $\oo_\sigma$ when $\dim\sigma=0$ (morally speaking, we may think of the space of flow lines on $\sigma$ as the stacky quotient $\pt/\RR$).

Now let us observe that there are natural (odd) ``boundary'' maps:
\begin{align}
\label{simplexflowproductor}\oo_\sigma&\to\oo_{\sigma|[0\ldots k]}\otimes\oo_{\sigma|[k\ldots n]}\\
\label{simplexflowfaceor}\oo_\sigma&\to\oo_{\sigma|[0\ldots\hat k\ldots n]}
\end{align}
induced by the geometric inclusions of boundary strata \eqref{Fproduct}--\eqref{Fface}.  Specifically, the first map is induced by the inclusion of the space of pairs of flow lines on $\sigma|[0\ldots k]$ and $\sigma|[k\ldots n]$ into the space of flow lines on $\sigma$ ($0<k<n$), and the second by the inclusion of flow lines on $\sigma|[0\ldots\hat k\ldots n]$ into flow lines on $\sigma$ ($0<k<n$).  In fact, the first map \eqref{simplexflowproductor} admits the following alternative description, which shows that it is in fact defined for $0\leq k\leq n$.  Let $\F^{(i)}(\sigma)$ denote the moduli space of stable broken Morse flow lines on a simplex $\sigma$ with $i$ ordered marked points appearing in order (this is defined in the expected way, allowing constant flow lines as long as they are stabilized by the presence of at least one marked point).  Denote the orientation module of $\F^{(i)}(\sigma)$ by $\oo_\sigma^{(i)}$, and note that there is a natural isomorphism $\oo_\sigma^{(i)}=\oo_\sigma\otimes\oo_\RR^{\otimes i}$ (even for $\dim\sigma=0$; this provides another justification of our definition of $\oo_\sigma$ in this case).  Now there is a natural concatenation map $\F^{(i)}(\sigma|[0\ldots k])\times\F^{(j)}(\sigma|[k\ldots n])\to\F^{(i+j)}(\sigma)$ which is the inclusion of a codimension one boundary stratum, thus giving rise to a boundary map $\oo_\sigma^{(i+j)}\to\oo_{\sigma|[0\ldots k]}^{(i)}\otimes\oo_{\sigma|[k\ldots n]}^{(j)}$.  After factoring out $\oo_\RR^{\otimes(i+j)}$, this map coincides with \eqref{simplexflowproductor} for $0<k<n$, and thus may be used to define \eqref{simplexflowproductor} for $0\leq k\leq n$.

\begin{remark}
The standard orientation of $\RR$ gives a standard generator $[\RR]_1\otimes\cdots\otimes[\RR]_{n-1}\in\oo_\sigma$ for $\dim\sigma>0$, and $[\RR]^\vee\in\oo_\sigma$ for $\dim\sigma=0$, where $[\RR]\cdot[\RR]^\vee=1$ (i.e.\ $[\RR]^\vee\in\oo_\RR^\vee$ is the ``right dual'' of $[\RR]\in\oo_\RR$).  Using these generators to trivialize $\oo_\sigma$ determines a sign convention in which \eqref{simplexflowproductor} is given by $(-1)^{k+1}$ and \eqref{simplexflowfaceor} is given by $(-1)^k$.
\end{remark}

\subsubsection{Orientations of linearized operators of Floer equations}

Fix two Hamiltonians $H_0,H_1:M\times S^1\to\RR$ and two non-degenerate periodic orbits $p,q:S^1\to M$ of $H_0,H_1$ respectively.  Now for any path of $\omega$-compatible almost complex structures $J:\RR\to J(M)$ (constant near $s=\pm\infty$), any path of Hamiltonians $H:M\times S^1\times\RR\to\RR$ (constant near $s=\pm\infty$) with $H(s=-\infty)=H_0$ and $H(s=\infty)=H_1$, and any map $u:S^1\times\RR\to M$ in $W^{k,2,\delta}(S^1\times\RR,M)$ (converging to $q$ at $-\infty$ and to $p$ at $\infty$), there is a natural linearized operator:
\begin{equation}\label{Dfloer}
D_\text{Floer}:W^{k,2,\delta}(S^1\times\RR,u^\ast TM)\to W^{k-1,2,\delta}(S^1\times\RR,\Omega^{0,1}_{S^1\times\RR}\otimes_\CC u^\ast TM_J)
\end{equation}
(we assume that $\delta>0$ is less than the smallest positive eigenvalues of the asymptotic linearized operators).  Let us agree to define this linearized operator using the $\RR$-family of $J$-linear connections on $TM$ given by $\nabla_XY:=\frac 12(\nabla^0_XY-J(\nabla^0_X(JY)))$, where $\nabla^0$ denotes the Levi-Civita connection of the metric associated to the compatible pair $(\omega,J)$.  The results of Floer--Hofer \cite{floerhoferorientations} (see the remark after Theorem 2) imply that the Fredholm orientation line of $D_\text{Floer}$ is a trivial bundle over the space of paths $H$, $J$, and maps $u$.  Thus if we fix $p$, $q$, and a homotopy class of maps $u$, there is a well-defined orientation line $\oo_{p,q}$, canonically isomorphic to $\oo_{D_\text{Floer}}$ for any choice of $H$, $J$, and $u$.

Floer--Hofer \cite[Theorem 10]{floerhoferorientations} also construct natural associative (even) isomorphisms:
\begin{equation}\label{floerhoferoogluing}
\oo_{p,q}\otimes\oo_{q,r}\to\oo_{p,r}
\end{equation}
by a certain kernel gluing procedure.  We refer the reader to \cite{floerhoferorientations} for the details of this construction.  Suffice it to say here that, after adding a finite-dimensional piece to the domains of each of two linearized operators \eqref{Dfloer} (from $p$ to $q$ and from $q$ to $r$ respectively) so that they become surjective with kernels $K_1$ and $K_2$, there is a natural kernel pregluing map from $K_1\oplus K_2$ to the domain of certain glued operator (from $p$ to $r$); the $L^2$-orthogonal projection onto the kernel $K_3$ of the glued operator gives an isomorphism $K_1\oplus K_2\xrightarrow\sim K_3$, and this induces the map \eqref{floerhoferoogluing}, which may be shown to be independent of the choices used to define it.

The coherent trivializations of $\oo_{p,q}$ resulting from the maps \eqref{floerhoferoogluing} are known to coincide with the usual coherent orientations of Morse theory, when restricted to $S^1$-invariant Hamiltonians and their $S^1$-invariant Floer trajectories, see e.g.\ Floer \cite{floersfhs}.

\subsubsection{Orientations on thickened moduli spaces $\Mbar(\sigma,p,q)_I^\reg$}

The existence of the desired coherent orientations for $\Mbar(\sigma,p,q)$ follows easily from the following result, which we spend the rest of this apendix proving.

\begin{proposition}\label{gluingHForientations}
The orientation sheaf (in the sense of Definition \ref{IAorientationsheaf}) of every moduli space $\Mbar(\sigma,p,q)$ can be canonically identified with $\oo_\sigma\otimes\oo_{p,q}$.  Moreover, under these identifications, the boundary maps on orientation sheaves induced (as in \eqref{sesofO}) by the structure maps:
\begin{align}
\label{hforproduct}\Mbar(\sigma|[0\ldots k],p,q)\times\Mbar(\sigma|[k\ldots n],q,r)&\to\partial\Mbar(\sigma,p,r)\\
\label{hforface}\Mbar(\sigma|[0\ldots\hat k\ldots n],p,q)&\to\partial\Mbar(\sigma,p,q)
\end{align}
coincide with the maps:
\begin{align}
\label{hforproductor}\oo_\sigma\otimes\oo_{p,r}&\to\oo_{\sigma|[0\ldots k]}\otimes\oo_{p,q}\otimes\oo_{\sigma|[k\ldots n]}\otimes\oo_{q,r}\\
\label{hforfaceor}\oo_\sigma\otimes\oo_{p,q}&\to\oo_{\sigma|[0\ldots\hat k\ldots n]}\otimes\oo_{p,q}
\end{align}
induced by \eqref{simplexflowproductor}--\eqref{simplexflowfaceor} and \eqref{floerhoferoogluing}.
\end{proposition}

Let us introduce various linearized operators which will play a role in the proof of Proposition \ref{gluingHForientations} below.  At any point in a thickened moduli space $\Mbar(\sigma,p,q)_I$, we have a linearized operator:
\begin{equation}
\label{linearizedstd}D:E_I\oplus W^{k,2,\delta}(C,u^\ast TM)\to W^{k-1,2,\delta}(C,\Omega^{0,1}_C\otimes_\CC u^\ast TM_{J(\ell(\cdot))})
\end{equation}
More precisely, the operator $D$ denotes the usual linearization of the $I$-thickened holomorphic equation, e.g.\ as calculated in Lemma \ref{linearizedformulaHF} (corresponding to variations of the map $u$ and of the element $e\in E_I$; we always keep $\ell$ fixed when defining linearized operators).  We will need to make use of this and other linearized operators at maps $u$ which do not necessarily satisfy the relevant pseudo-holomorphic curve equation, and for such $u$, the lineraized operator $D$ depends on a choice of (a family of) $J$-linear connections on $M$.  Let us fix the convention of always using the $J$-linear connection $\nabla_XY:=\frac 12(\nabla^0_XY-J(\nabla^0_X(JY)))$ where $\nabla^0$ is the Levi-Civita connection of the metric associated to the compatible pair $(\omega,J)$ (more precisely, this is a family of connections parameterized by $\sigma=\Delta^n$).

We will make use of another linearized operator:
\begin{equation}
\label{linearizedhol}D_\text{hol}:E_I\oplus W^{k,2,\delta}(C,u^\ast TM)\to W^{k-1,2,\delta}(C,\Omega^{0,1}_C\otimes_\CC u^\ast TM_{J(\ell(\cdot))})
\end{equation}
The operator $D_\text{hol}$ denotes the linearization of the usual holomorphic curve equation (i.e.\ without the thickening terms $\lambda_\alpha(e_\alpha)$); thus it is given by the expression in Lemma \ref{linearizedformulaHF} without the terms \eqref{firstLterm}--\eqref{lastLterm}.  Clearly there is a natural isomorphism $\oo_D=\oo_{D_\text{hol}}$ since the terms \eqref{firstLterm}--\eqref{lastLterm} are compact.

\begin{proof}[Proof of Proposition \ref{gluingHForientations}]
To identify the orientation sheaf of $\Mbar(\sigma,p,q)$ with $\oo_\sigma\otimes\oo_{p,q}$, it suffices to identify the orientation sheaf of every $\Mbar(\sigma,p,q)_I^\reg$ with $\oo_{E_I}\otimes\oo_\sigma\otimes\oo_{p,q}$ in a compatible way (for all finite subsets $I\subseteq\A(\sigma,p,q)^{\geq\s^\ttop}$).  Moreover, it suffices to make this identification over the open subset $\Mbar(\sigma,p,q)_I^{\reg\circ}\subseteq\Mbar(\sigma,p,q)_I^\reg$ where the domain curve is smooth (it then automatically extends uniquely to all of $\Mbar(\sigma,p,q)_I^\reg$, by virtue of the local topological description of $\Mbar(\sigma,p,q)_I^\reg$ given by the gluing map); we may also check compatibility with the inclusions $I\subseteq J$ over $\Mbar(\sigma,p,q)_I^{\reg\circ}$.  Now the kernel $\tilde K=\ker D$ of the linearization \eqref{linearizedstd} forms a vector bundle over $\Mbar(\sigma,p,q)_I^{\reg\circ}$, and the orientation sheaf of $\Mbar(\sigma,p,q)_I^{\reg\circ}$ is isomorphic to $\oo_\sigma\otimes\oo_{\tilde K}$ (to see this, one must distinguish the two cases $\dim\sigma>0$ and $\dim\sigma=0$).  Now we have $\oo_{\tilde K}=\oo_D=\oo_{D_\text{hol}}=\oo_{E_I}\otimes\oo_{p,q}$; this defines a fiberwise isomorphism of $\oo_{\tilde K}$ with $\oo_{E_I}\otimes\oo_{p,q}$, and since the operators $D$ and $D_\text{hol}$ vary nicely over the base $\Mbar(\sigma,p,q)_I^{\reg\circ}$, it is easy to see that this is in fact an isomorphism of sheaves.  Thus we have the desired identification.

Now let us show that the boundary map induced by \eqref{hforface} coincides with the tautological map \eqref{hforfaceor} (this just amounts to chasing definitions).  Let $\s\in\SSS_{\Mbar}(\sigma,p,q)$ denote the stratum $(\sigma|[0\ldots\hat k\ldots n],p,q)$ (i.e.\ the stratum consisting of trajectories over $\sigma$ which factor through $\sigma|[0\ldots\hat k\ldots n]$).  Now on the space $\Mbar(\sigma|[0\ldots\hat k\ldots n],p,q)=\Mbar(\sigma,p,q)^{\leq\s}$, we have three implicit atlases:
\begin{equation}
\A(\sigma|[0\ldots\hat k\ldots n],p,q)^{\geq\s^\ttop}\subseteq\A(\sigma,p,q)^{\geq\s}\supseteq\A(\sigma,p,q)^{\geq\s^\ttop}
\end{equation}
(note that the two occurences of $\s^\ttop$ refer to the top elements of two different strata posets); in fact (the index set of) the middle atlas is the disjoint union of (the index sets of) the atlases on the right and on the left.  The reasoning used above to identify the orientation sheaf of $\Mbar(\sigma|[0\ldots\hat k\ldots n],p,q)=\Mbar(\sigma,p,q)^{\leq\s}$ under the leftmost implicit atlas applies equally well for all three atlases (and, in particular, the resulting identifications coincide).  Hence, for the purpose of identifying the boundary map on orientation sheaves, it suffices to consider the single atlas $\A(\sigma,p,q)^{\geq\s^\ttop}$ on $\Mbar(\sigma,p,q)$ and the inclusion of the boundary stratum $\Mbar(\sigma,p,q)^{\leq\s}\hookrightarrow\Mbar(\sigma,p,q)$.  Now it suffices to check that the desired compatibility holds for the inclusion $(\Mbar(\sigma,p,q)^{\leq\s}_I)^{\reg\circ}\hookrightarrow\Mbar(\sigma,p,q)_I^{\reg\circ}$ (for all finite subsets $I\subseteq\A(\sigma,p,q)^{\geq\s^\ttop}$), and this follows by definition.

Now let us show that the boundary map induced by \eqref{hforproduct} coincides with the tautological map \eqref{hforproductor}.  First, we chase definitions as above.  Let $\s\in\SSS_{\Mbar}(\sigma,p,r)$ denote the stratum $(\sigma|[0\ldots k],p,q)$--$(\sigma|[k\ldots n],q,r)$ (i.e.\ the stratum consisting of Floer trajectories over $\sigma$ which are broken at vertex $k$ and periodic orbit $q\in\PPP_{H(k)}$).  Now on the space $\Mbar(\sigma|[0\ldots k],p,q)\times\Mbar(\sigma|[k\ldots n],q,r)=\Mbar(\sigma,p,r)^{\leq\s}$, we have three implicit atlases:
\begin{equation*}
\A(\sigma|[0\ldots k],p,q)^{\geq\s^\ttop}\sqcup\A(\sigma|[k\ldots n],q,r)^{\geq\s^\ttop}\subseteq\A(\sigma,p,r)^{\geq\s}\supseteq\A(\sigma,p,r)^{\geq\s^\ttop}
\end{equation*}
and in fact (the index set of) the middle atlas is the disjoint union of (the index sets of) the atlases on the right and on the left.  The reasoning used above to identify the orientation sheaves of $\Mbar(\sigma|[0\ldots k],p,q)$ and $\Mbar(\sigma|[k\ldots n],q,r)$ applies equally well to identify the orientation sheaf of their product $\Mbar(\sigma|[0\ldots k],p,q)\times\Mbar(\sigma|[k\ldots n],q,r)=\Mbar(\sigma,p,r)^{\leq\s}$ under all three atlases above (the fact that the thickening datums in the latter two atlases do not treat the portions of the trajectories from $p$ to $q$ and from $q$ to $r$ ``independently'' does not cause any problems).  Hence for the purposes of identifying the boundary map on orientation sheaves, it suffices to consider the single atlas $\A(\sigma,p,r)^{\geq\s^\ttop}$ on $\Mbar(\sigma,p,r)$ and the inclusion of the boundary stratum $\Mbar(\sigma,p,r)^{\leq\s}\hookrightarrow\Mbar(\sigma,p,r)$.  As before, it suffices to check that the desired compatibility holds for the inclusion $(\Mbar(\sigma,p,r)^{\leq\s}_I)^\reg\hookrightarrow\Mbar(\sigma,p,r)_I^\reg$ (for all finite subsets $I\subseteq\A(\sigma,p,r)^{\geq\s^\ttop}$), and moreover, this may be checked on the locus where the domain curve is smooth except for the required node over $k$ asymptotic to $q\in\PPP_{H(k)}$.

It will be convenient to assume that $0<k<n$, so let us first deduce the case of general $0\leq k\leq n$ from the case $0<k<n$.  Fix $\sigma$ and $0\leq k\leq n=\dim\sigma$.  We consider the degenerate simplex $\sigma'\to\sigma$ given by ``doubling vertex $0$'' if $k=0$ and ``doubling vertex $n$'' if $k=n$ (i.e.\ if $0=k<n$, we consider $\Delta^{n+1}\to\Delta^n$ given by superimposing the first two vertices of $\Delta^{n+1}$, if $0<k=n$, we consider $\Delta^{n+1}\to\Delta^n$ given by superimposing the last two vertices of $\Delta^{n+1}$, and if $0=k=n$ we consider $\Delta^2\to\Delta^0$); let $k'=1$ if $k=0$ and let $k'=n'-1$ if $k=n$, where $n'=\dim\sigma'$, and note that $0<k'<n'$.  Now there is a natural map $\A(\sigma,p,q)^{\geq\s^\ttop}\to\A(\sigma',p,q)^{\geq\s^\ttop}$ given by pulling back along $\sigma'\to\sigma$, and $\sigma'\to\sigma$ maps flow lines to flow lines (for the flow from Definition \ref{simplexflows}).  Hence, given any point $x\in\Mbar(\sigma,p,r)_I^\reg$ with a single node over $k$ asymptotic to $q\in\PPP_{H(k)}$, we may lift it (canonically, up to translating the part(s) of the trajectory over $0=k$ and/or over $k=n$) to a point $x'\in\Mbar(\sigma',p,r)_I^\reg$ with a single node over $k'$ asymptotic to $q\in\PPP_{H'(k')}=\PPP_{H(k)}$.  Moreover, there is a germ of homeomorphism $\Mbar(\sigma',p,r)_I^\reg=\Mbar(\sigma,p,r)_I^\reg\times\RR^{\mathbf 1(k=0)+\mathbf 1(k=n)}$ between neighborhoods of $x'$ and $x\times 0$.  Now the desired compatibility of orientations for $(\Mbar(\sigma,p,r)_I^{\leq\s})^\reg\hookrightarrow\Mbar(\sigma,p,r)_I^\reg$ follows from the compatibility for $(\Mbar(\sigma',p,r)_I^{\leq\s'})^\reg\hookrightarrow\Mbar(\sigma',p,r)_I^\reg$ (which has from $0<k'<n'$).  Hence we may assume without loss of generality that $0<k<n$.

Now we have come to the heart of the matter, where we will need to analyze the gluing map.  To review: we have identified the orientation sheaf of $\Mbar(\sigma,p,r)_I^\reg$ with $\oo_{E_I}\otimes\oo_\sigma\otimes\oo_{p,r}$, and we have identified the orientation sheaf of $(\Mbar(\sigma,p,r)^{\leq\s}_I)^\reg$  with $\oo_{E_I}\otimes\oo_{\sigma|[0\ldots k]}\otimes\oo_{p,q}\otimes\oo_{\sigma|[k\ldots n]}\otimes\oo_{q,r}$ (recall that $\s$ denotes the stratum consisting of trajectories broken at vertex $k$ and periodic orbit $q\in\PPP_{H(k)}$).  We must show that the boundary map on orientation sheaves induced by the inclusion $(\Mbar(\sigma,p,r)^{\leq\s}_I)^\reg\hookrightarrow\Mbar(\sigma,p,r)_I^\reg$ is the expected map \eqref{hforproductor}, and it suffices to check this over the locus where the domain curve is smooth except for the required node over $k$ asymptotic to $q\in\PPP_{H(k)}$.  We may further assume that $0<k<n$.

We consider the gluing setup for $\Mbar(\sigma,p,r)_I^\reg$ at a point in $(\Mbar(\sigma,p,r)^{\leq\s}_I)^\reg$ where $C_0$ is smooth except for the required node over vertex $k$ asymptotic to $q\in\PPP_{H(k)}$.  In other words, $C_0\setminus\{q_0,q_1,q_2\}=S^1\times\RR\sqcup S^1\times\RR$, and $u_0:S^1\times\RR\sqcup S^1\times\RR\to M$ is a trajectory from $p$ to $q$ and from $q$ to $r$.  From \S\ref{choosingDH}, we take $L=L'=0$ (i.e.\ no points $p_i$ or $p_i'$), and we may take $D=H=\varnothing$.  From \S\ref{gluingandvaryingjAl}, there is one gluing parameter $\alpha\in\RR_{\geq 0}$, there is no nontrivial variation in $(j_0,A_0)$ (i.e.\ $y\in\RR^\ast=\RR^0$), and there is some variation in $\ell_0$ parameterized by $w\in\prod_{i=1}^{2\prime}H^i$.  As usual, we denote by $K\subseteq W^{k,2,\delta}(C_0,u_0^\ast TM)$ the kernel of the linearized operator (we shall omit the subscript $D,H$ since $L=L'=0$).  Now the gluing construction gives rise to a gluing map:
\begin{align}
K\times\RR_{\geq 0}\times\prod_{i=1}^{2\prime}H^i&\to\Mbar(\sigma,p,r)_I^\reg\\
(\kappa,\alpha,w)&\mapsto(u_{\alpha,w,y,\kappa},e_{\alpha,w,y,\kappa})
\end{align}
($y=0$).  Now if we restrict to $\alpha=0$, this map realizes the identification of the orientation sheaf of $(\Mbar(\sigma,p,r)^{\leq\s}_I)^\reg$ with $\oo_K\otimes\oo_{\sigma|[0\ldots k]}\otimes\oo_{\sigma|[k\ldots n]}$.  Thus it suffices to show that for fixed (sufficiently small) $\alpha>0$ and fixed $w=0$, the gluing map is differentiable in the $K$ direction, and that its derivative (a map from $K$ to the kernel of the linearized operator at the glued map $u_{\alpha,w,y,\kappa}$) agrees (on orientation lines) with the Floer--Hofer map \eqref{floerhoferoogluing} (recall that $\oo_K$ is identified with $\oo_{E_I}\otimes\oo_{p,q}\otimes\oo_{q,r}$ and that the orientation line of the kernel at the glued map is identified with $\oo_{E_I}\otimes\oo_{p,r}$).

For fixed $(\alpha,w,y)$, the gluing map is given by:
\begin{equation}\label{restrictedgluingmap}
K\xrightarrow{\kappa\mapsto\kappa_\alpha}K_\alpha\xrightarrow{\kappa_\alpha\mapsto\kappa_{\alpha,w,y}}\F_{\alpha,w,y}^{-1}(0)
\end{equation}
The second map is defined \emph{a priori} by a Newton iteration, however a more natural description \emph{a posteriori} is that $\{\kappa_{\alpha,w,y}\}=\{\kappa_\alpha+\im Q_{\alpha,w,y}\}\pitchfork\F_{\alpha,w,y}^{-1}(0)$ is the unique (necessarily transverse) intersection in a $(k,2,\delta,\delta)$-neighborhood of zero of fixed size.  Note that in this neighborhood of fixed size, $\F_{\alpha,w,y}^{-1}(0)$ is a (highly differentiable) submanifold, since $\F_{\alpha,w,y}$ is highly differentiable and $D_{\kappa_{\alpha,w,y}}\F_{\alpha,w,y}$ is surjective as observed in the proof of Lemma \ref{mapsstufftoregHF}.  From this description of the second map, it is clearly differentiable.  Thus the derivative of the restricted gluing map \eqref{restrictedgluingmap} at a given $\kappa$ is given by:
\begin{equation}
T_\kappa K\xrightarrow{\dot\kappa\mapsto\dot\kappa_\alpha}T_{\kappa_\alpha}K_\alpha\xrightarrow{\proj_{Q_{\alpha,w,y}}}T_{\kappa_{\alpha,w,y}}\F_{\alpha,w,y}^{-1}(0)=\ker D_{\kappa_{\alpha,w,y}}\F_{\alpha,w,y}
\end{equation}
The second map is ``projection with respect to $Q_{\alpha,w,y}$'', i.e.\ the map induced by identifying both the domain and codomain with $W^{k,2,\delta,\delta}(C_\alpha,u_\alpha^\ast TM)/\im Q_{\alpha,w,y}$.  It suffices to study the derivative at $\kappa=0$, namely the first line of the following commuting diagram:
\begin{equation}
\begin{tikzcd}[column sep = small]
T_0K\ar{rr}{\dot\kappa\mapsto\dot\kappa_\alpha}&&T_0K_\alpha\ar{rr}{\proj_{Q_{\alpha,w,y}}}\ar{dr}[swap]{\proj_{Q_{\alpha,w,y}}}&&\ker D_{(0)_{\alpha,w,y}}\F_{\alpha,w,y}\ar{dl}{\proj_{Q_{\alpha,w,y}}}\\
&&&\ker D_{\alpha,w,y}
\end{tikzcd}
\end{equation}
The rightmost diagonal map is orientation preserving (i.e.\ it commutes with the identification of the orientation lines of the domain and codomain with $\oo_{E_I}\otimes\oo_{p,r}$), because $\proj_{Q_{\alpha,w,y}}$ gives a local trivialization of the bundle $\ker D\F_{\alpha,w,y}$ with respect to which the orientation is constant by definition.  Now the leftmost diagonal map is just $\dot\kappa_\alpha\mapsto\dot\kappa_\alpha-Q_{\alpha,w,y}D_{\alpha,w,y}\dot\kappa_\alpha$.  Hence it suffices to show that the following composition (note that we have changed notation from $\dot\kappa$ back to $\kappa$):
\begin{equation}\label{concretestufftocheckorientations}
\ker D_{0,0,0}=K\xrightarrow{\kappa\mapsto\kappa_\alpha}K_\alpha\xrightarrow{\kappa_\alpha\mapsto\kappa_\alpha-Q_{\alpha,w,y}D_{\alpha,w,y}\kappa_\alpha}\ker D_{\alpha,w,y}
\end{equation}
acts as \eqref{floerhoferoogluing} on orientations (with respect to the previously defined isomorphisms $\oo_K=\oo_{E_I}\otimes\oo_{p,q}\otimes\oo_{q,r}$ and $\oo_{\ker D_{\alpha,w,y}}=\oo_{E_I}\otimes\oo_{p,r}$).

We analyze \eqref{concretestufftocheckorientations} as follows.  For convenience, we may as well assume that $w=0$, and recall that necessarily $y=0$.  Now let us consider the following diagram:
\begin{equation}
\begin{tikzcd}\label{orientationsinversediagram}
K_{\alpha,0,0}\ar[hook,yshift=-0.5ex]{r}\ar[leftarrow,yshift=0.5ex]{r}{\proj_{T_{\alpha,0,0}}}&W^{k,2,\delta,\delta}(C_\alpha,u_\alpha^\ast TM)\oplus E_I\ar{r}{D_{\alpha,0,0}}&W^{k-1,2,\delta,\delta}(\tilde C_\alpha,\Omega^{0,1}_{\tilde C_\alpha,j_0}\otimes_\CC u_\alpha^\ast TM_{J(\ell_{\alpha,0}(A_0(\cdot)))})\ar{d}{\mathrm{break}}\\
K_{0|\alpha,0,0}\ar[hook,yshift=-0.5ex]{r}\ar[leftarrow,yshift=0.5ex]{r}{\proj_{T_{0|\alpha,0,0}}}&W^{k,2,\delta,\delta}(C_0,u_{0|\alpha}^\ast TM)\oplus E_I\ar{u}{\mathrm{glue}}\ar{r}{D_{0|\alpha,0,0}}&W^{k-1,2,\delta,\delta}(\tilde C_0,\Omega^{0,1}_{\tilde C_0,j_0}\otimes_\CC u_{0|\alpha}^\ast TM_{J(\ell_{0,0}(A_0(\cdot)))})\ar[leftrightarrow]{d}{\PT}\\
K_{0,0,0}\ar[hook,yshift=0.5ex]{r}\ar[leftarrow,yshift=-0.5ex]{r}[swap]{\proj_{Q_{0,0,0}}}\ar{ru}{\kappa\mapsto\kappa_{0|\alpha}}&W^{k,2,\delta,\delta}(C_0,u_0^\ast TM)\oplus E_I\ar[leftrightarrow]{u}{\PT}\ar[yshift=0.5ex]{r}{D_{0,0,0}}&\ar[yshift=-0.5ex]{l}{Q_{0,0,0}}W^{k-1,2,\delta,\delta}(\tilde C_0,\Omega^{0,1}_{\tilde C_0,j_0}\otimes_\CC u_0^\ast TM_{J(\ell_{0,0}(A_0(\cdot)))})
\end{tikzcd}
\end{equation}
The right half is just a copy of \eqref{biginversediagramHF} (with the middle square collapsed since $y=0$).  The left half consists of the inclusions of the kernels of the operators on the right half, as well as the projection maps associated to the images of the (approximate) right inverses $Q_{0,0,0}$, $T_{0|\alpha,0,0}:=\PT\circ Q_{0,0,0}\circ\PT$, and $T_{\alpha,0,0}=\mathrm{glue}\circ\PT\circ Q_{0,0,0}\circ\PT\circ\mathrm{break}$ (recall that $\im T_{\alpha,0,0}=\im Q_{\alpha,0,0}$).  The diagonal map $\kappa\mapsto\kappa_{0|\alpha}$ is defined by cutting off as in Definition \ref{kernelpregluingHF}; thus $\kappa_\alpha=\mathrm{glue}(\kappa_{0|\alpha})$.

Now the map \eqref{concretestufftocheckorientations} which we would like to analyze may be written as the composition $\kappa\mapsto\proj_{T_{\alpha,0,0}}(\mathrm{glue}(\kappa_{0|\alpha}))$ from \eqref{orientationsinversediagram}.  Now, we know that $\|D_{0|\alpha,0,0}\kappa_{0|\alpha}\|_{k-1,2,\delta,\delta}=\|D_{\alpha,0,0}\kappa_\alpha\|_{k-1,2,\delta,\delta}$, which is small by Lemma \ref{pregluingKestimateHF}.  It follows that the map we would like to analyze is well-approximated (as $\alpha\to 0$) by the map $\kappa\mapsto\proj_{T_{\alpha,0,0}}(\mathrm{glue}(\proj_{T_{0|\alpha,0,0}}\kappa_{0|\alpha}))$.  In particular, since we are only interested in its action on orientations, it suffices to consider the latter map $\kappa\mapsto\proj_{T_{\alpha,0,0}}(\mathrm{glue}(\proj_{T_{0|\alpha,0,0}}\kappa_{0|\alpha}))$.

Now we claim that the map $K_{0,0,0}\to K_{0|\alpha,0,0}$ given by $\kappa\mapsto\proj_{T_{0|\alpha,0,0}}\kappa_{0|\alpha}$ preserves orientation (the orientation lines of the domain and codomain are both identified with $\oo_{E_I}\otimes\oo_{p,q}\otimes\oo_{q,r}$).  To see this, simply observe the map (as well as its domain and codomain and their orientations) vary continuously over $\alpha\in[0,\epsilon)$, and that the statement is true for $\alpha=0$ because then the map is the identity map.  Hence we have reduced the problem to showing that the map:
\begin{equation}\label{reducedmaporientations}
W^{k,2,\delta,\delta}(C_0,u_{0|\alpha}^\ast TM)\supseteq K_{0|\alpha,0,0}\xrightarrow{\mathrm{glue}}W^{k,2,\delta,\delta}(C_\alpha,u_\alpha^\ast TM)\oplus E_I\xrightarrow{\proj_{T_{\alpha,0,0}}}K_{\alpha,0,0}
\end{equation}
acts by \eqref{floerhoferoogluing}.

To analyze the map ${\proj_{T_{\alpha,0,0}}}\circ{\mathrm{glue}}:K_{0|\alpha,0,0}\to K_{\alpha,0,0}$, argue as follows.  For any finite-dimensional vector space $F$ and a map $\Lambda:F\to W^{k-1,2,\delta,\delta}(\tilde C_0,\Omega^{0,1}_{\tilde C_0,j_0}\otimes_\CC u_0^\ast TM_{J(\ell_{0,0}(A_0(\cdot)))})$ supported away from the ends, we may consider the following modified version of \eqref{orientationsinversediagram}:
\begin{equation}
\begin{tikzcd}\label{orientationsinversediagramcompl}
K_{\alpha,0,0}^{\Lambda,t}\ar[hook,yshift=-0.5ex]{r}\ar[leftarrow,yshift=0.5ex]{r}{\proj_{T_{\alpha,0,0}^{\Lambda,t}}}&F\oplus W^{k,2,\delta,\delta}(C_\alpha,u_\alpha^\ast TM)\oplus E_I\ar{r}{\Lambda\oplus D_{\alpha,0,0}^t}&W^{k-1,2,\delta,\delta}(\tilde C_\alpha,\Omega^{0,1}_{\tilde C_\alpha,j_0}\otimes_\CC u_\alpha^\ast TM_{J(\ell_{\alpha,0}(A_0(\cdot)))})\ar{d}{\mathrm{break}}\\
K_{0|\alpha,0,0}^{\Lambda,t}\ar[hook,yshift=-0.5ex]{r}\ar[leftarrow,yshift=0.5ex]{r}{\proj_{T_{0|\alpha,0,0}^{\Lambda,t}}}&F\oplus W^{k,2,\delta,\delta}(C_0,u_{0|\alpha}^\ast TM)\oplus E_I\ar{u}{\id\oplus\mathrm{glue}}\ar{r}{\Lambda\oplus D_{0|\alpha,0,0}^t}&W^{k-1,2,\delta,\delta}(\tilde C_0,\Omega^{0,1}_{\tilde C_0,j_0}\otimes_\CC u_{0|\alpha}^\ast TM_{J(\ell_{0,0}(A_0(\cdot)))})\ar[leftrightarrow]{d}{\PT}\\
K_{0,0,0}^{\Lambda,t}\ar[hook,yshift=0.5ex]{r}\ar[leftarrow,yshift=-0.5ex]{r}[swap]{\proj_{Q_{0,0,0}^{\Lambda,t}}}\ar{ru}{\kappa\mapsto\kappa_{0|\alpha}}&F\oplus W^{k,2,\delta,\delta}(C_0,u_0^\ast TM)\oplus E_I\ar[leftrightarrow]{u}{\id\oplus\PT}\ar[yshift=0.5ex]{r}{\Lambda\oplus D_{0,0,0}^t}&\ar[yshift=-0.5ex]{l}{Q_{0,0,0}^{\Lambda,t}}W^{k-1,2,\delta,\delta}(\tilde C_0,\Omega^{0,1}_{\tilde C_0,j_0}\otimes_\CC u_0^\ast TM_{J(\ell_{0,0}(A_0(\cdot)))})
\end{tikzcd}
\end{equation}
Here $t\in[0,1]$ indicates that the terms \eqref{firstLterm}--\eqref{lastLterm} carry a factor of $t$.  The estimates from Lemmas \ref{approxbottomHF} and \ref{approxtopHF} apply to this modified diagram as well.  Thus as long as we fix a bounded right inverse $Q_{0,0,0}^{\Lambda,t}$ of $D_{0,0,0}^{\Lambda,t}$, the rest of the diagram makes sense (and, in particular, $T_{0|\alpha,0,0}^{\Lambda,t}$ and $T_{\alpha,0,0}^{\Lambda,t}$ are approximate right inverses) for sufficiently small $\alpha>0$.  Furthermore, in any family of $(\Lambda,t,Q^{\Lambda,t})$, the kernels $K_{0,0,0}$, $K_{0|\alpha,0,0}$, $K_{\alpha,0,0}$ form vector bundles, and the identifications of their orientation lines with $\oo_F\otimes\oo_{p,q}\otimes\oo_{q,r}\otimes\oo_{E_I}$ and $\oo_F\otimes\oo_{p,r}\otimes\oo_{E_I}$ (respectively) vary continuously.

Now the map:
\begin{equation}\label{KmapLt}
K_{0|\alpha,0,0}^{\Lambda,t}\xrightarrow{{\proj_{T_{\alpha,0,0}^{\Lambda,t}}}\circ{\mathrm{glue}}}K_{\alpha,0,0}^{\Lambda,t}
\end{equation}
is exactly the map we wish to analyze when $F=0$, $\Lambda=0$, $t=1$, and $Q_{0,0,0}^{\Lambda,t}=0\oplus Q_{0,0,0}$.  More generally, if we allow $F$ nonzero (but still $\Lambda=0$), this map is simply our desired map plus the identity map on $F$.  Since the space of acceptable maps $\Lambda$ is contractible, it suffices to show that \eqref{KmapLt} has the desired action on orientations for any single pair $(F,\Lambda)$, $t=1$, and $Q_{0,0,0}^{\Lambda,t}=0\oplus Q_{0,0,0}$.

Now by compactness of $[0,1]$, there exists a pair $(F,\Lambda)$ so that $\Lambda\oplus D_{0,0,0}^t$ is surjective for all $t\in[0,1]$.  Fix such a pair $(F,\Lambda)$, and also fix a continuously varying family of bounded right inverses $Q_{0,0,0}^{\Lambda,t}$ with $Q_{0,0,0}^{\Lambda,1}=0\oplus Q_{0,0,0}$.  Since the kernels form a bundle over the base $[0,1]$ and their orientations vary continuously, it suffices to analyze \eqref{KmapLt} for this $(F,\Lambda)$, $t=0$, and this $Q_{0,0,0}^{\Lambda,0}$.

Now the Floer--Hofer map \eqref{floerhoferoogluing} is \emph{defined} (see \cite[Proposition 9]{floerhoferorientations}) by the property that it is induced by a certain map $K_{0|\alpha,0,0}^{\Lambda,0}\to K_{\alpha,0,0}^{\Lambda,0}$ closely related to \eqref{KmapLt}; clearly it suffices to show that the difference between the Floer--Hofer map and \eqref{KmapLt} is very small.  Consider the map:
\begin{equation}\label{KmapLtLtwo}
K_{0|\alpha,0,0}^{\Lambda,t}\xrightarrow{{\proj_{L^2}}\circ{\mathrm{glue}}}K_{\alpha,0,0}^{\Lambda,t}
\end{equation}
in the setting of the above choices of $(F,\Lambda)$, $t=0$, and $Q_{0,0,0}^{\Lambda,t}$, where instead of projecting off of $\im T_{\alpha,0,0}^{\Lambda,0}$, we use the $L^2$-orthogonal projection.  The difference between \eqref{KmapLtLtwo} and \eqref{KmapLt} is exactly ${\proj_{L^2}}\circ Q^{\Lambda,0}_{\alpha,0,0}\circ(\Lambda\oplus D^0_{\alpha,0,0})\circ{\mathrm{glue}}$, which has small norm by the estimate in Lemma \ref{approxtopHF}.  Thus it suffices to compare \eqref{KmapLtLtwo} to the Floer--Hofer map.  Now the Floer--Hofer map is given by:
\begin{equation}\label{KmapLtFH}
K_{0|\alpha,0,0}^{\Lambda,0}\xrightarrow{{\proj_{L^2}}\circ{\mathrm{glue}'}}K_{\alpha,0,0}^{\Lambda,0}
\end{equation}
for a certain map $\mathrm{glue}'$ (see \cite[Proposition 9]{floerhoferorientations}).  However, the norm of the difference $\mathrm{glue}-\mathrm{glue}'$ is very small due to the exponential decay of elements of the kernel (here, we may use the explicit description of $K_{0|\alpha,0,0}$ as the image of $\kappa\mapsto\kappa_{0|\alpha}-Q_{0|\alpha,0,0}D_{0|\alpha,0,0}\kappa_{0|\alpha}$).  Thus we are done.
\end{proof}

\addcontentsline{toc}{section}{References}
\bibliographystyle{amsalpha}
\bibliography{implicitatlas}

\end{document}